\documentclass{amsart}
\usepackage{setspace}
\usepackage{a4}
\usepackage{amssymb,amsmath,amsthm,latexsym}
\usepackage{amsfonts}
\usepackage{amsfonts}
\usepackage{graphicx}
\usepackage{textcomp}
\usepackage{cite}
\usepackage{enumerate}
\usepackage[mathscr]{euscript}
\usepackage{mathtools}
\newtheorem{theorem}{Theorem}[section]

\newtheorem{corollary}[theorem] {Corollary}
\newtheorem{definition}[theorem]{Definition}
\newtheorem{example}[theorem]{Example}
\newtheorem{lemma} [theorem]{Lemma}

\newtheorem{proposition}[theorem]{Proposition}
\newtheorem{remark}[theorem]{Remark}
\newtheorem{caution}[theorem]{Caution}
\newtheorem{statement}[theorem]{Statement}
\title{This is the title}
\usepackage{amssymb}
\usepackage{amsmath}
\usepackage{tikz}
\usepackage{hyperref}

\begin{document}

\begin{center}
{\bf{\large}EXTENSION OF FRAMES AND BASES - I}\\
K. MAHESH KRISHNA, AND P. SAM JOHNSON \\
Department of Mathematical and Computational Sciences\\ 
National Institute of Technology Karnataka (NITK), Surathkal\\
Mangaluru 575 025, India  \\
Emails: kmaheshak@gmail.com, kmaheshakma16f02@nitk.edu.in, \\
       nitksam@gmail.com, sam@nitk.ac.in \\
       
Date: \today
\end{center}

\hrule
\vspace{0.5cm}
%--------------------------------------
\textbf{Abstract}: We extend the theory  of operator-valued frames (resp. bases), hence the theory of frames (resp. bases),  for Hilbert spaces and Hilbert C*-modules, in two folds. This extension leads  us to develop the theory of operator-valued frames (resp. bases) for Banach spaces. We give a characterization for the operator-valued frames indexed by a group-like unitary system. This answers an open question asked in the paper titled ``Operator-valued frames" by Kaftal, Larson, and Zhang  in \textit{Trans. Amer. Math. Soc.} (2009). We study stability of  the  extension.  We also extend Riesz-Fischer theorem, Bessel's inequality, variation formula, dimension formula, and trace formula. Further, notions of p-orthogonality, p-orthonormality and Riesz p-bases have been developed in Banach spaces and Paley-Wiener theorem has also been generalized. We derive  `4-inequality,'   `4-parallelogram law,' and  `4-projection theorem.'


\tableofcontents
\section{Introduction}\label{INTRODUCTION}
Frame for a Hilbert space is a relaxed orthonormal  basis. This was first introduced and studied by Duffin, and Schaeffer, in 1952, in the study of sequences of type $ \{e^{i\lambda_nx}\}_{n\in \mathbb{Z}}, \lambda_n \in \mathbb{C}, x\in \left(-r, r \right), r>0$ \cite{DUFFIN1}. Paper of Daubechies, Grossmann, and Meyer \cite{MEYER1} triggered the  area. To give further introduction, we set notations first.

Fonts $\mathcal{X},\mathcal{X}_0,\mathcal{Y}$ denote Banach spaces. Dual of  $\mathcal{X}$ is denoted by $\mathcal{X}^*$. For a set $Y$ in  $\mathcal{X}$, we denote  the span closure of $Y$ by $ \overline{\operatorname{span}}Y$. The identity operator on  $ \mathcal{X}$ is denoted by $ I_\mathcal{X}.$ Field of scalars ($ \mathbb{R}$ or $ \mathbb{C}$) is denoted by $\mathbb{K}$. The unit circle group is denoted by $ \mathbb{T}.$ Notations $ \mathbb{J}, \mathbb{L},  \mathbb{M},\mathbb{L}_j$ are used for  indexing sets. Our sequences are also indexed by these notations. Cardinality of  $\mathbb{J}$ is denoted by $\operatorname{Card}(\mathbb{J})$. If a sequence $ \{x_j\}_{j\in \mathbb{J}}$ is in $\mathcal{X}$, by the convergence of $ \sum_{j\in \mathbb{J}}x_j$ (in $\mathcal{X}$) we mean the convergence of the net  obtained by the set inclusion, on the collection of all finite subsets of $ \mathbb{J}$.
Letter $ \mathcal{H}$ always denotes a Hilbert space, so is any of its `integer' subscripts. Banach space of all bounded linear operators from $ \mathcal{H}$ to $ \mathcal{H}_0 $ is denoted by $ \mathcal{B}(\mathcal{H}, \mathcal{H}_0)$. We write $ \mathcal{B}(\mathcal{H}, \mathcal{H})$ as  $ \mathcal{B}(\mathcal{H})$.  If $ Y \subseteq\mathcal{B}(\mathcal{H})$, we denote the commutant of $Y$ by $Y'$. A  unital C*-algebra is denoted by $ \mathscr{A}$ and all of our C*-algebras are unital.  Whenever we want to use the  identity of $ \mathscr{A}$, we write $ \mathscr{A}$ as $(\mathscr{A},e)$. Standard Hilbert C*-module indexed with $\mathbb{J} $ is denoted simply by $ \mathscr{H}_\mathscr{A}$  and  the same indexed by $\mathbb{L} $ is denoted by  $ \mathscr{H}_\mathscr{A}(\mathbb{L}).$   Letters $\mathscr{E},\mathscr{E}_0,\mathscr{E}_1,\mathscr{E}_2 $ denote Hilbert C*-modules. Whenever  both $\mathscr{E},\mathscr{E}_0 $ are  over the same  C*-algebra $ \mathscr{A}$,  the collection of all  maps from $\mathscr{E}$ to $\mathscr{E}_0 $ which are linear, $ \mathscr{A}$-linear (we call such maps as homomorphisms)  bounded adjointable is denoted by $ \operatorname{Hom}^*_{\mathscr{A}}(\mathscr{E},\mathscr{E}_0)$. We declare $\operatorname{End}^*_{\mathscr{A}}(\mathscr{E}) \coloneqq\operatorname{Hom}^*_{\mathscr{A}}(\mathscr{E},\mathscr{E})$.  We use ``in", for words  such as   `frames', `orthogonal', `orthonormal', ..., whenever we are handeling collections in $\mathcal{B}(\mathcal{H}, \mathcal{H}_0) $ or $\operatorname{Hom}^*_{\mathscr{A}}(\mathscr{E},\mathscr{E}_0) $ or $ \mathcal{B}(\mathcal{X}, \mathcal{X}_0)$, and we use   ``for'', whenever we are dealing with sequences in  $ \mathcal{H}$ or  $\mathscr{E} $ or $\mathcal{X} $. For instance, ``$ \{A_j\}_{j\in \mathbb{J}} $ is orthonormal in $\mathcal{B}(\mathcal{H}, \mathcal{H}_0)$"  and ``$\{x_j\}_{j\in \mathbb{J}} $ is orthonormal for $ \mathcal{H}$".

Frames played a prominent role in both  pure and applied mathematics.   As  illustrations,  in the pure, 
\begin{enumerate}[\upshape(i)]
\item In 2005-2006 it was shown that \cite{CASAZZACHRISTENSENLINDNERVERSHYNIN, CASAZZAFICKUSTREMAINERIC, CASAZZATREMAIN}, the famous (operator algebra question) Kadison-Singer conjecture  (i.e., every pure state on the abelian von Neumann algebra of all bounded diagonal operators on $\ell^2(\mathbb{N})$ extends uniquely as a pure state on $\mathcal{B}(\ell^2(\mathbb{N}))$\cite{KADISONSINGER}) is true if and only if  Feichtinger conjecture  (i.e., every bounded  frame in a Hilbert space can be written as a   finite union of Riesz sequences (cf. \cite{CASAZZACHRISTENSENLINDNERVERSHYNIN})) is true if and only if Bourgain-Tzafriri conjecture  (i.e., for every $b>0$, there exists $ m \in \mathbb{N}$ and an $a>0$ such that for every linear operator  $T: \mathbb{C}^n\rightarrow \mathbb{C}^n $ with $ \|Te_j\|=1, \forall j=1,...,n$, and $\|T\|\leq \sqrt{b},$ we can find a partition $\{Q_j\}_{j=1}^m $ of $ \{1,...,n\}$ such that $\|\sum_{k\in Q_j}a_kTe_k\|^2\geq a \sum_{k\in Q_j}|c_k|^2, \forall j=1,...,m, \forall c_k \in \mathbb{C}, \forall k \in Q_j$, where $\{e_j\}_{j=1}^n $ is the standard orthonormal basis  for $\mathbb{C}^n$ (cf. \cite{CASAZZACHRISTENSENLINDNERVERSHYNIN})) is true if and only if the Paving conjecture  (i.e., given $ \epsilon > 0$, there exists $ m \in \mathbb{N}$ such that for every $ n \in \mathbb{N}$ and every linear operator $T: \mathbb{C}^n\rightarrow \mathbb{C}^n $ whose matrix w.r.t. the standard orthonormal basis $\{e_j\}_{j=1}^n $ for $\mathbb{C}^n $ has zero diagonal, we can find a partition $\{Q_j\}_{j=1}^m $ of $ \{1,...,n\}$ such that $ \|P_{Q_j}TP_{Q_j}\|\leq \epsilon\|T\|, \forall j=1,...,m,$ where $P_{Q_j} $ is the orthogonal projection from $\mathbb{C}^n $ onto $ \operatorname{span}\{e_k\}_{k\in Q_j}$ (cf. \cite{CASAZZACHRISTENSENLINDNERVERSHYNIN})) is true (Marcus, Spielman, and Srivastava solved Kadison-Singer conjecture  \cite{MARCUSSPIELMANSRIVASTAVA}). 
\item Define $T_a:\mathcal{L}^2(\mathbb{R}) \ni f\mapsto T_af\in \mathcal{L}^2(\mathbb{R}),(T_af)(x)=f(x-a),\forall x \in \mathbb{R}$, for each $a \in \mathbb{R}$, and  $E_b:\mathcal{L}^2(\mathbb{R}) \ni f\mapsto E_bf\in \mathcal{L}^2(\mathbb{R}),(E_bf)(x)=e^{2\pi ibx}f(x),\forall x \in \mathbb{R}$, for each $b \in \mathbb{R}$. Let $ \chi_{[0,c)}$ be the characteristic function on $[0,c)$. The $abc$-problem demands to characterize  $a,b,$  and $c$   such that $ \{E_{mb}T_{na}\chi_{[0,c)}\}_{m,n\in \mathbb{Z}}$ is a (Gabor or Weyl-Heisenberg) frame for $ \mathcal{L}^2(\mathbb{R})$ (Dai, and Sun solved $abc$-problem  \cite{DAISUNMEMOIRS}).
\end{enumerate}
In applied, frames are  in regular use in sampling theory \cite{HANKORNELSON}, filter banks \cite{CASAZZAGITTAFF}, signal and image  processing \cite{DONOHOELAD}, quantum measurement \cite{ELDARFORNEY}, and  wireless communication \cite{HEATHPAULRAJ}.
\begin{definition}\cite{DUFFIN1}\label{OLE}
A collection $ \{x_j\}_{j \in \mathbb{J}}$ in  a Hilbert space $ \mathcal{H}$ is said to be a frame for $\mathcal{H}$ if there exist $ a, b >0$ such that
\begin{equation}\label{SEQUENTIALEQUATION1}
a\|h\|^2\leq\sum_{j \in \mathbb{J}}|\langle h, x_j \rangle|^2 \leq b\|h\|^2  ,\quad \forall h \in \mathcal{H}.
\end{equation}
Constants $ a$ and $ b$ are called as lower and upper frame bounds, respectively. Supremum (resp. infimum) of the set of all lower (resp. upper) frame bounds is called optimal lower (resp. upper) frame bound. If the optimal frame bounds are equal, then the frame is called as tight frame. A tight frame whose optimal bound is one is termed as Parseval frame.
\end{definition}
For  good references on the theory of frames, we refer  \cite{PETERGCASAZZA}, \cite{OLE1}, \cite{HANLARSON1}, \cite{YOUNG1}, \cite{CASAZZAGITTAFF}, \cite{HANKORNELSON}, \cite{WALDRONFINITE} (those last three are dedicated to finite frames. Further,  last one is devoted to finite tight frames).
 
Another way of writing Inequality (\ref{SEQUENTIALEQUATION1}) (i.e., Definition \ref{OLE}) is 
 \begin{equation}\label{SEQUENTIALEQUATION2}
 \text{the map}~ \mathcal{H} \ni h \mapsto \sum_{j\in \mathbb{J}}\langle h, x_j\rangle x_j \in \mathcal{H} ~\text{is well-defined bounded positive invertible operator}.
 \end{equation}
 Let us now define $ A_j : \mathcal{H} \ni  h \mapsto \langle h, x_j\rangle \in \mathbb{K} $, for each $ j \in \mathbb{J}$.
 Hence one more way for Statement (\ref{SEQUENTIALEQUATION2}) is 
 \begin{equation}\label{SEQUENTIALEQUATION3}
 \sum_{j\in \mathbb{J}} A_j^*A_j ~\text{converges in the strong-operator topology  on } \mathcal{B}(\mathcal{H}) \text{ to a bounded positive invertible operator.}
 \end{equation}
 Now Statement (\ref{SEQUENTIALEQUATION3}) leads to 
 \begin{definition}\cite{KAFTALLARSONZHANG1}\label{KAFTAL}
 A collection  $ \{A_j\}_{j \in \mathbb{J}} $  in $ \mathcal{B}(\mathcal{H}, \mathcal{H}_0)$ is said to be an operator-valued frame on $ \mathcal{H}$ with range in $ \mathcal{H}_0$ if the series $ \sum_{j\in \mathbb{J}} A_j^*A_j$  converges in the strong-operator topology on $ \mathcal{B}(\mathcal{H})$ to a  bounded invertible operator.
 \end{definition}
 Sun \cite{SUN1} used the term G-frame for operator-valued frame. We call Definition \ref{OLE} as \textit{sequential version}  and  Definition \ref{KAFTAL} as \textit{operator version} of frames. Sun  showed that these two notions are equivalent, under certain circumstances (Theorem 3.1 in \cite{SUN1}). We will show this in both folds of our extensions (Theorem \ref{SEQUENTIAL CHARACTERIZATION}, and Theorem \ref{WEAKSEQUENTIALCHARACTERIZATION}).

In Section \ref{MK} we enlarge the theory  of operator-valued frames  for Hilbert spaces  by Kaftal, Larson, and Zhang \cite{ KAFTALLARSONZHANG1}, and by Sun \cite{SUN1} (this extension leads  us to develop the theory of operator-valued frames for Banach spaces, which we handle in Section \ref{OPERATOR-VALUEDFRAMESFRAMESFORBANACHSPACES}). Definition \ref{ONBRIESZDEFINITION}  introduces relative  orthonormal sets and Riesz basis for operators. This section contains extension (operator version) of (i) Riesz-Fischer theorem (Theorem \ref{ERFOV}), and  (ii) Bessel's inequality (Theorem \ref{GBSOV}). Notions of Parseval, Riesz, orthonormal, dual, orthogonal, disjoint  operator-valued frames are  introduced in this section. A dilation result (Theorem \ref{OPERATORDILATION}) also appears  in this section.

Section \ref{CHARACTERIZATIONSOF THE EXTENSION} contains various characterizations (Theorem \ref{OPERATORVALUEDCHARACTERIZATIONSOFTHEEXTENSION}, Theorem \ref{OPERATORCHARACTERIZATIONHILBERT2}, Theorem \ref{SEQUENTIAL CHARACTERIZATION}) of orthonormal bases, Riesz bases, operator-valued frames, Bessel sequences, Riesz operator-valued frames and orthonormal operator-valued frames. Theorem \ref{SEQUENTIAL CHARACTERIZATION} connects operators to elements and this gives an idea to define extension in sequential form (we launch this in Section \ref{SEQUENTIAL}).

Section \ref{SIMILARITYCOMPOSITIONANDTENSORPRODUCT} contains four types of similarities between two  operator-valued frames. It also contains composition and tensor product of operator-valued frames.

Section \ref{FRAMESANDDISCRETEGROUPREPRESENTATIONS} studies operator-valued frames indexed by groups. Advantage of indexing a frame with group is - we can generate frames by starting with two fixed operators. Theorem \ref{gc1} characterizes unitary representations of discrete groups which generate Parseval operator-valued frames.

In Section \ref{FRAMESANDGROUP-LIKEUNITARYSYSTEMS} we study operator-valued frames indexed by group-like unitary systems.  Theorem  \ref{CHARACTERIZATIONGROUPLIKE} characterizes  unitary representations of group-like unitary system   which generate Parseval operator-valued frames. Corollary \ref{ANSWER} answers the   question raised by Kaftal, Larson, and Zhang  in \textit{Trans. Amer. Math. Soc.} \cite{KAFTALLARSONZHANG1}, namely ``We do not know if there is a similar necessary and sufficient condition for frames indexed by a unitary system, or at least by some structured unitary system, such as a Gabor system'' (Page No. 6371, (iii)  of Remark 6.8 in \cite{KAFTALLARSONZHANG1}).

Section \ref{PERTURBATIONS} shows that Definition \ref{1} is stable under perturbations. Theorem \ref{PERTURBATION RESULT 1}, Theorem \ref{PERTURBATION RESULT 2}, and  Theorem \ref{OVFQUADRATICPERTURBATION} are main results here.

Section \ref{SEQUENTIAL} starts by abstracting the clues from Theorem \ref{SEQUENTIAL CHARACTERIZATION}. All the notions, results and characterizations which appeared in Section \ref{MK} are in this section, in sequential form. As an eagle view on results, Theorem \ref{SEQUENTIALDILATION}  (resp. \ref{SEQUENTIALCHARACTERIZATIONHILBERT1}, \ref{SEQUENTIALCHARACTERIZATIONHILBERT2}, \ref{SEQUENTIALSIMILARITYCHARACTERIZATION}, \ref{GROUPC1}, \ref{UNITARY1}, \ref{PERTURBATION1SV}, \ref{PERTURBATION2SV}, \ref{PERTURBATION3SV}) is sequential version of Theorem \ref{OPERATORDILATION}    (resp. \ref{OPERATORVALUEDCHARACTERIZATIONSOFTHEEXTENSION}, \ref{OPERATORCHARACTERIZATIONHILBERT2}, \ref{RIGHTSIMILARITY}, \ref{gc1}, \ref{CHARACTERIZATIONGROUPLIKE}, \ref{PERTURBATION RESULT 1}, \ref{PERTURBATION RESULT 2}, \ref{OVFQUADRATICPERTURBATION}). This section contains extension (sequential version) of   Bessel's inequality (Theorem \ref{GBESV}). Corollary \ref{CORRECTEXTENSION} tells that our extension is correct.

Section \ref{THEFINITEDIMENSIONALCASE} handles frames in finite dimensional Hilbert spaces. Theorem \ref{FINITEDIMENSIONALCHARATERIZATIONHILBERT} characterizes finite frames for finite dimensional spaces.  Properties of the space, frames, and operators are pasted together in Theorem \ref{FINITEDIMENSIONALPASTINGTHEOREM}, which also has extended variation formula,  extended dimension formula, and extended trace formula. Theorem \ref{REALTOCOMPLEX} (resp. Theorem \ref{COMPLEXTOREAL}) tells when can we get a frame for $\mathbb{C}^m $ (resp. $\mathbb{R}^m $) from a frame for $\mathbb{R}^m $ (resp. $\mathbb{C}^m $). An interesting example of tight frame for $\mathbb{R}^2 $, using circular functions, is in Proposition \ref{INTERESTINGEXAMPLEPROPOSITION}.

Section \ref{FURTHEREXTENSION} further extends Section \ref{MK}, and Section \ref{SEQUENTIAL}. We achieve this by removing condition (ii) in  Definition \ref{1}. Theorem \ref{WEAKSEQUENTIALCHARACTERIZATION} gives a link between operator, and sequential versions, whereas Theorem \ref{WEAK CHARACTERIZATION} provides a characterization of `further extension' for collections inside $ \mathcal{B}(\mathcal{H})$.

In Section \ref{EXTENSIONOFHOMOMORPHISM-VALUEDFRAMESFORHILBERTC*-MODULES}
we try to do the results of Section \ref{MK} in Hilbert C*-module settings. This enlarges the notion of operator-valued frames (what we call as homomorphism-valued frames) for Hilbert C*-modules by Kaftal, Larson, and Zhang  \cite{KAFTALLARSONZHANGMODULEOVF}.
This section contains extension (homomorphism version) of (i) Riesz-Fischer theorem (Theorem \ref{RIESZFISCHERHOMOMORPHISMVERSION}), and (ii) Bessel's inequality (Theorem \ref{GBEOVHM}). Here we only state results  whose proofs  are similar to the proofs of corresponding results in Section  \ref{MK} (we  briefly sketch the proof of some of the statements for which certain arguments in the corresponding proof in Section \ref{MK} need additional support to validate, unlike Hilbert spaces).

Section \ref{SEQUENTIALVERSIONOFHOMOMORPHISM-VALUEDFRAMES} has sequential version of Section \ref{EXTENSIONOFHOMOMORPHISM-VALUEDFRAMESFORHILBERTC*-MODULES} and this is  comparable with Section \ref{SEQUENTIAL}. Again, for proofs we follow the same strategy  done for the proofs in Section \ref{EXTENSIONOFHOMOMORPHISM-VALUEDFRAMESFORHILBERTC*-MODULES}. This section contains extension (sequential version in Hilbert C*-modules) of (i) Riesz-Fischer theorem (Theorem \ref{RIESZFISCHERGENERAL}), and (ii) Bessel's inequality (Theorem \ref{GBISVHM}).

Further extension of Section \ref{EXTENSIONOFHOMOMORPHISM-VALUEDFRAMESFORHILBERTC*-MODULES}, and Section  \ref{SEQUENTIALVERSIONOFHOMOMORPHISM-VALUEDFRAMES} are done in 
Section \ref{FURTHEREXTENSIONINMODULES}.

Section \ref{OPERATOR-VALUEDFRAMESFRAMESFORBANACHSPACES} starts with the fundamental idea of writing frame inequality (Inequality (\ref{FRAMEINEQUALITYFUNDAMENTAL})) independent from the inner product (Inequality (\ref{FUNDAMENTALIDEA})). This results in the precise formulation of the notion  of p-operator-valued frames (Definition \ref{HM1}). Notions of  p-orthogonality and  p-orthonormality for collection of operators between Banach spaces are in  Definition \ref{BANACHOPERATORORTHOGONALSET} and Definition \ref{BANACHOPERATORORTHONORMALSET}, respectively. Definition \ref{RELATIVEPORTHONORMALANDRIESZBANACH} gives notions of relative p-orthonormality and Riesz p-basis for operators. A characterization of all these notions is in Theorem \ref{POVFCHARACTERIZATIONBANACH}.

 Section \ref{SVSECTION} takes a different root than its operator-version (Section \ref{OPERATOR-VALUEDFRAMESFRAMESFORBANACHSPACES}), unlike Section \ref{SEQUENTIAL} or Section \ref{SEQUENTIALVERSIONOFHOMOMORPHISM-VALUEDFRAMES}, which were born from Section \ref{MK} or Section \ref{EXTENSIONOFHOMOMORPHISM-VALUEDFRAMESFORHILBERTC*-MODULES}, respectively. This section contains sequential definitions  of  p-frame, Riesz p-basis, Riesz p-frame, orthonormal p-frame. Theorem \ref{PSEQUENTIALCHARACTERIZATIONBANACH} characterizes all these notions. Definition \ref{ORTHOGONALITYINBANACH} (resp. Definition \ref{ORTHONORMALITYINBANACH}) gives the notion of  p-orthogonal sequence (resp. p-orthogonal basis) in Banach spaces. Theorem \ref{2-OGSIMPLIESOGS},  and  Theorem \ref{2-ONBIMPLIESONB}  show that these definitions  give back traditional orthogonality and orthonormality definitions in Hilbert spaces (whenever $p=2$).
 4-inequality (Theorem \ref{4INEQUALITY}), 4-parallelogram law (Theorem \ref{SPECIALPL}) and 4-projection theorem (Theorem \ref{4PROJECTIONTHEOREM}) are derived in this section. Extended trace formula, and extended dimension formula for Banach spaces are in Theorem \ref{EXTENDEDDIMENSINANDTRACEFORMULABANACH}.

The appendix, Section \ref{APPENDIX} has the notion of Riesz p-basis and the relative Riesz p-basis for Banach spaces. Definition \ref{SINGLEPRIESZ} coincides with the usual Riesz basis definition in Hilbert spaces, when $p=2$ (Remark \ref{PRIESZTOUSUALRIESZ}). Theorem \ref{GENERALIZEDPALEYWEINER} extends Paley-Wiener theorem (from Hilbert space to Banach space). Relative Riesz p-bases are characterized in Theorem \ref{RIESZPBASISSEQUENTIALCHARACTERIZATIONAPPENDIX}.

Section \ref{CONJECTURE}, the last section contains some conjectures about the path-connectedness of frames indexed by groups, and group-like unitary systems (they are inspired from Theorem 8.1 in \cite{KAFTALLARSONZHANG1}).

\section{Extension of operator-valued frames and bases}\label{MK}
\begin{definition}\label{1}
Let $ \mathbb{J}$ be an index set, $\mathcal{H}, \mathcal{H}_0 $  be Hilbert spaces. Define 
$ L_j : \mathcal{H}_0 \ni h \mapsto e_j\otimes h \in  \ell^2(\mathbb{J}) \otimes \mathcal{H}_0$,  where $\{e_j\}_{j \in \mathbb{J}} $ is  the standard orthonormal basis for $\ell^2(\mathbb{J})$,  for each $ j \in \mathbb{J}$. A collection $ \{A_j\}_{j \in \mathbb{J}} $  in $ \mathcal{B}(\mathcal{H}, \mathcal{H}_0)$ is said to be an \textit{operator-valued frame} (we write (ovf)) in $ \mathcal{B}(\mathcal{H}, \mathcal{H}_0) $  with respect to a collection  $ \{\Psi_j\}_{j \in \mathbb{J}}  $ in $ \mathcal{B}(\mathcal{H}, \mathcal{H}_0) $ if 
\begin{enumerate}[\upshape(i)]
\item the series $ \sum_{j\in \mathbb{J}} \Psi_j^*A_j$  converges in the strong-operator topology (SOT) on $ \mathcal{B}(\mathcal{H})$ to a  bounded positive invertible operator,
\item both $ \sum_{j\in \mathbb{J}} L_jA_j$, $ \sum_{j\in \mathbb{J}} L_j\Psi_j$ converge in the strong-operator topology on $ \mathcal{B}(\mathcal{H},\ell^2(\mathbb{J}) \otimes \mathcal{H}_0 )$ to  bounded operators.
\end{enumerate}
We denote the limit of  $ \sum_{j\in \mathbb{J}} \Psi_j^*A_j$  by $ S_{A, \Psi}$ (which we call as frame operator for   $ \{A_j\}_{j \in \mathbb{J}} $  w.r.t.  $ \{\Psi_j\}_{j \in \mathbb{J}} $), and use notations $\theta_A $, $ \theta_\Psi$ (which we call as analysis operators for  $ \{A_j\}_{j \in \mathbb{J}} $  w.r.t.  $ \{\Psi_j\}_{j \in \mathbb{J}} $ and their adjoints as synthesis operators) to denote the  limits of $ \sum_{j\in \mathbb{J}} L_jA_j$, $ \sum_{j\in \mathbb{J}} L_j\Psi_j$, respectively.
Real $ \alpha, \beta  >0 $ satisfying $\alpha I_\mathcal{H} \leq S_{A, \Psi}\ \leq \beta I_\mathcal{H} $ are called as lower and upper frame bounds, taken in order.
Let $ a= \sup\{\alpha : \alpha I_\mathcal{H} \leq S_{A, \Psi}\}$, $ b= \inf\{\beta :  S_{A, \Psi}\leq \beta I_\mathcal{H} \}$. One sees that $ a=\|S_{A,\Psi}^{-1}\|^{-1}$ and $ b = \|S_{A,\Psi}\|$.  We call $ a$ as optimal lower frame bound, $ b$ as optimal upper frame bound for the frame $ \{A_j \}_{j \in \mathbb{J}}$ w.r.t.  $ \{\Psi_j\}_{j \in \mathbb{J}}.$  If $ a=b$, then the frame is called as tight frame (or exact frame) w.r.t.  $ \{\Psi_j \}_{j \in \mathbb{J}}$. A tight frame is said to be Parseval w.r.t.  $ \{\Psi_j \}_{j \in \mathbb{J}}$ if optimal frame bound is one. 

 Whenever $ \{A_j \}_{j \in \mathbb{J}}$  is an operator-valued frame w.r.t. $ \{\Psi_j \}_{j \in \mathbb{J}}$ we write $(\{A_j\}_{j\in \mathbb{J}}, \{\Psi_j\}_{j\in \mathbb{J}}) $ is (ovf).
 
For fixed $ \mathbb{J}$, $\mathcal{H}, \mathcal{H}_0 $ and $ \{\Psi_j \}_{j \in \mathbb{J}}$, the set of all operator-valued frames in $ \mathcal{B}(\mathcal{H}, \mathcal{H}_0) $ with respect to collection  $ \{\Psi_j \}_{j \in \mathbb{J}}$ is denoted by $ \mathscr{F}_\Psi.$
\end{definition}
 If the condition \text{\upshape(i)} in Definition \ref{1} is replaced by ``the series $ \sum_{j\in \mathbb{J}} \Psi_j^*A_j$  converges  in the strong-operator topology on $ \mathcal{B}(\mathcal{H})$ to a positive bounded  operator," then we say $ \{A_j\}_{j\in\mathbb{J}}$   w.r.t. $ \{\Psi_j\}_{j\in\mathbb{J}}$ is Bessel.

We mention here that the  sayings   ``$ \{A_j \}_{j \in \mathbb{J}}$  is an operator-valued frame w.r.t. $ \{\Psi_j \}_{j \in \mathbb{J}}$ in $ \mathcal{B}(\mathcal{H}, \mathcal{H}_0)$'' and ``$ \{A_j \}_{j \in \mathbb{J}}$    w.r.t. $ \{\Psi_j \}_{j \in \mathbb{J}}$  is an operator-valued frame in  $ \mathcal{B}(\mathcal{H}, \mathcal{H}_0)$'' are the same.

Definition \ref{1} is more flexible. As  examples, (i) any invertible  operator is a frame w.r.t. a positive multiple of it, (ii) Let $ U,V$ be two commuting positive invertible operators on a Hilbert space. Then $ U$ is (ovf)  w.r.t. $ V.$

Both the conditions in Definition \ref{1} are independent as the following example reveals.
\begin{example}
On $ \mathbb{C},$ define $ A_nx\coloneqq\frac{x}{\sqrt{n}},  \forall x \in \mathbb{C}, \forall n \in \mathbb{N}$, and $\Psi_1x\coloneqq x, \Psi_nx\coloneqq0, \forall x \in \mathbb{C}, \forall n \in \mathbb{N}\setminus\{1\} $. Then $ \sum_{n=1}^\infty\Psi_n^*A_nx$ converges to a positive invertible operator but  $ \sum_{n=1}^\infty L_nA_nx$ doesn't. On the other hand both $\sum_{n=1}^{2}L_nI_\mathbb{C},\sum_{n=1}^{2}L_n((-1)^nI_\mathbb{C}) $ are bounded operators but $ \sum_{n=1}^{2}((-1)^nI_\mathbb{C})^*I_\mathbb{C}$ is  not invertible (it is zero operator).
\end{example}
We note that (i) in Definition \ref{1} implies that  there are real $ a,b >0$ such that 
$$ a\|h\|^2\leq\sum\limits_{j\in \mathbb{J}}\langle A_jh, \Psi_jh\rangle =\sum\limits_{j\in \mathbb{J}}\langle \Psi_jh, A_jh\rangle  \leq b\|h\|^2,~ \forall h \in \mathcal{H},$$
and another way for (ii) is 
$$\{A_j\}_{j\in \mathbb{J}} \text{ is Bessel w.r.t.}~ \{A_j\}_{j\in \mathbb{J}}, ~  \text{and}~ \{\Psi_j\}_{j\in \mathbb{J}} ~\text{ is Bessel w.r.t.}~ \{\Psi_j\}_{j\in \mathbb{J}}.$$
 Whenever $ \Psi_j=A_j,\forall j \in \mathbb{J}$ condition (i) in Definition \ref{1} implies condition (ii) (Proposition 2.3 in \cite{KAFTALLARSONZHANG1}). Thus  Definition \ref{KAFTAL} is a particular case of Definition \ref{1}.

An operator-valued frame (resp. a Bessel sequence) with respect to itself is an operator-valued frame (resp. a Bessel sequence).

We note the following. 
\begin{enumerate}[(i)]
\item Definition \ref{1} is symmetric, i.e., if  $\{A_j\}_{j \in \mathbb{J}} $ is an (ovf) w.r.t. $ \{\Psi_j\}_{j \in \mathbb{J}} $,   then $ \{\Psi_j\}_{j \in \mathbb{J}} $ is  an  (ovf) w.r.t. $ \{A_j\}_{j \in \mathbb{J}} $. 
\item $\{h \in \mathcal{H}: A_jh=0, \forall j \in \mathbb{J}\} =\{0\}= \{h \in \mathcal{H}: \Psi_jh=0, \forall j \in \mathbb{J}\} $, and $\overline{\operatorname{span}}\cup_{j\in \mathbb{J}} A^*_j(\mathcal{H}_0)=\mathcal{H}= \overline{\operatorname{span}}\cup_{j\in \mathbb{J}} \Psi^*_j(\mathcal{H}_0). $
\item $ \theta_Ah =\sum_{j\in \mathbb{J}}e_j\otimes A_jh , \forall  h \in \mathcal{H}.$
\item $ S_{A, \Psi}=  S_{\Psi, A}$. 
\item The operators $ L_j$'s defined in Definition \ref{1} are  isometries from $\mathcal{H}_0 $ to $ \ell^2(\mathbb{J}) \otimes \mathcal{H}_0$, and  for
 $  j,k \in \mathbb{J }$ we have  
 \begin{align}\label{LEQUATION}
  L_j^*L_k =
\left\{
\begin{array}{ll}
I_{\mathcal{H}_0 } & \mbox{if } j=k \\
0 & \mbox{if } j\neq k
\end{array}
\right.
~\text{and} \quad 
 \sum\limits_{j\in \mathbb{J}} L_jL_j^*=I_{\ell^2(\mathbb{J})}\otimes I_{\mathcal{H}_0}
  \end{align}
where  the convergence is in the strong-operator topology.
\item $L_k^*(\{a_j\}_{j \in \mathbb{J}}\otimes y) =a_ky, \forall  \{a_j\}_{j \in \mathbb{J}} \in \ell^2(\mathbb{J}), \forall y \in \mathcal{H}_0$, for each $ k $ in $ \mathbb{J}.$ 
\item If  $ \{A_j\}_{j \in \mathbb{J}}, $ $ \{B_j\}_{j \in \mathbb{J}}  \in \mathscr{F}_\Psi $, then $\{A_j+B_j\}_{j \in \mathbb{J}} \in \mathscr{F}_\Psi $, and $\{\alpha A_j\}_{j \in \mathbb{J}} \in \mathscr{F}_\Psi, \forall \alpha >0. $
\item If $(\{A_j\}_{j\in \mathbb{J}}, \{\Psi_j\}_{j\in \mathbb{J}}) $ is tight (ovf) with bound $ a, $ then $ S_{A,\Psi}=aI_\mathcal{H}.$
\end{enumerate}
 It is interesting to note that the collection of all operator-valued frames  (Definition \ref{KAFTAL}) is  closed by nonzero scalar multiplication but  need not be closed with addition. From the very Definition \ref{1} we get 
\begin{proposition}
Let $(\{A_j\}_{j\in \mathbb{J}}, \{\Psi_j\}_{j\in \mathbb{J}}) $ be  (ovf)   in $ \mathcal{B}(\mathcal{H}, \mathcal{H}_0)$  with an upper frame  bound $b$. If $\Psi_j^*A_j\geq 0, \forall j \in \mathbb{J} ,$ then $ \|\Psi_j^*A_j\|\leq b, \forall j \in \mathbb{J}.$
\end{proposition}
 \begin{proof}
For each $ h \in \mathcal{H}, j \in \mathbb{J}$ we get $ \langle\Psi_j^*A_jh,h\rangle \leq \sum_{j\in\mathbb{J}}\langle\Psi_j^*A_jh,h\rangle \leq b \langle h,h \rangle$ and therefore $ \|\Psi_j^*A_j\|=\sup_{h\in \mathcal{H},\|h\|=1}\langle\Psi_j^*A_jh,h\rangle \leq b, \forall j \in \mathbb{J}.$
 \end{proof}
 Following is an extension of `expansion result' for Bessel sequences to operator-valued frames due to Li, and Sun \cite{LISUN1}.
 \begin{proposition}
 If  $(\{A_j\}_{j\in \mathbb{J}}, \{\Psi_j\}_{j\in \mathbb{J}}) $ is  Bessel   in $ \mathcal{B}(\mathcal{H}, \mathcal{H}_0)$, then there exists a $ B \in \mathcal{B}(\mathcal{H}, \mathcal{H}_0)$ such that $(\{A_j\}_{j\in \mathbb{J}}\cup\{B\}, \{\Psi_j\}_{j\in \mathbb{J}}\cup\{B\}) $ is a tight (ovf). In particular, if  $(\{A_j\}_{j\in \mathbb{J}}, \{\Psi_j\}_{j\in \mathbb{J}}) $ is   (ovf)   in $ \mathcal{B}(\mathcal{H}, \mathcal{H}_0)$, then there exists a $ B \in \mathcal{B}(\mathcal{H}, \mathcal{H}_0)$ such that $(\{A_j\}_{j\in \mathbb{J}}\cup\{B\}, \{\Psi_j\}_{j\in \mathbb{J}}\cup\{B\}) $ is a tight (ovf).
 \end{proposition}
\begin{proof}
Let $ \lambda> \|S_{A, \Psi}\|.$ Define $ B \coloneqq(\lambda I_\mathcal{H}-S_{A, \Psi})^{1/2}.$ Then $ \sum_{j\in\mathbb{J}}\Psi_j^*A_j+B^*B=S_{A, \Psi}+(\lambda I_\mathcal{H}-S_{A, \Psi})=\lambda I_\mathcal{H}.$ 
\end{proof}
 
\begin{definition}\cite{SUN1}\label{ONBDEFINITIONOVHS}
A collection  $ \{A_j\}_{j \in \mathbb{J}}$ in $ \mathcal{B}(\mathcal{H}, \mathcal{H}_0)$ is said to be an orthonormal basis in  $ \mathcal{B}(\mathcal{H},\mathcal{H}_0)$ if 
$$\langle A_j^*y, A_k^*z\rangle=\delta_{j,k}\langle y, z\rangle  , ~\forall y, z \in \mathcal{H}_0, ~\forall j, k \in \mathbb{J} ~ \text{and } ~ \sum\limits_{j \in \mathbb{J}}\|A_jh\|^2=\|h\|^2, ~ \forall h \in \mathcal{H}.$$
\end{definition}
We observe $\langle A_j^*y, A_k^*z\rangle=\delta_{j,k}\langle y, z\rangle  , \forall y, z \in \mathcal{H}_0, \forall j, k \in \mathbb{J}$ if and only if $A_jA_k^*=\delta_{j,k}I_{\mathcal{H}_0} , \forall j, k \in \mathbb{J}$. Hence if $ \{A_j\}_{j \in \mathbb{J}}$ is orthonormal, then $\|A_j\|^2=\|A_jA_j^*\|=1, \forall j \in \mathbb{J} $. Also if  the spaces are complex, and if  $ \{A_j\}_{j \in \mathbb{J}}$ is orthonormal, then  $ \sum_{j\in \mathbb{J}}A_j^*A_j=I_\mathcal{H}$ in SOT.
\begin{example}
\begin{enumerate}[\upshape(i)]
\item If $ U: \mathcal{H}\rightarrow \mathcal{H}_0$ is unitary, then $\{U\}$ is an orthonormal basis in  $ \mathcal{B}(\mathcal{H}, \mathcal{H}_0)$.
\item From Equation (\ref{LEQUATION}), we deduce that $ \{L^*_j\}_{j \in \mathbb{J}}$  is an orthonormal basis in $ \mathcal{B}(\ell^2(\mathbb{J})\otimes\mathcal{H}_0, \mathcal{H}_0)$.
\end{enumerate}
\end{example}
Inspired from the  Definition \ref{ONBDEFINITIONOVHS}, we define 
 \begin{definition}\label{OGSDEFINITIONOVHS}
 A collection  $ \{A_j\}_{j \in \mathbb{J}}$  in $ \mathcal{B}(\mathcal{H}, \mathcal{H}_0)$ is said to be 	an 
 orthogonal set in $ \mathcal{B}(\mathcal{H},\mathcal{H}_0)$ if 
  $$\langle A_j^*y, A_k^*z\rangle=0 , ~\forall y, z \in \mathcal{H}_0, ~\forall j, k \in \mathbb{J}, j\neq k.$$
\end{definition}
 \begin{definition}\label{ONSDEFINITIONOVHS}
 A collection  $ \{A_j\}_{j \in \mathbb{J}}$  in $ \mathcal{B}(\mathcal{H}, \mathcal{H}_0)$ is said to be  an orthonormal set in $ \mathcal{B}(\mathcal{H},\mathcal{H}_0)$ if 
 $$\langle A_j^*y, A_k^*z\rangle=\delta_{j,k}\langle y, z\rangle  , ~\forall y, z \in \mathcal{H}_0, ~\forall j, k \in \mathbb{J} ~ \text{and } ~ \sum\limits_{j \in \mathbb{J}}\|A_jh\|^2\leq\|h\|^2, ~ \forall h \in \mathcal{H}.$$
 
 \end{definition}
\begin{proposition}
 Let $ \{A_j\}_{j \in \mathbb{J}}$ be in  $ \mathcal{B}(\mathcal{H}, \mathbb{K}).$ Choose $ x_j \in \mathcal{H}$  such that $ A_jh=\langle h, x_j\rangle, \forall h \in \mathcal{H}, \forall j \in \mathbb{J}.$ Then $ \{A_j\}_{j \in \mathbb{J}}$	 is an orthonormal set (resp. basis)  in  $ \mathcal{B}(\mathcal{H},\mathbb{K})$ if and only if    $ \{x_j\}_{j \in \mathbb{J}}$ is an orthonormal set (resp. basis)   for $ \mathcal{H}.$
 \end{proposition}
 \begin{proof}
$\langle A_j^*y, A_k^*z\rangle=\langle yx_j, zx_k\rangle =y\bar{z}\langle x_j, x_k\rangle  , \forall y, z \in \mathbb{K}, \forall j, k \in \mathbb{J} ~\text{and}~ \sum_{j \in \mathbb{J}}\|A_jh\|^2=\sum_{j \in \mathbb{J}}|\langle h , x_j \rangle |^2,  \forall h \in \mathcal{H}.$
 \end{proof}
 \begin{theorem}
 \begin{enumerate}[\upshape(i)]
\item Every orthonormal set $Y$ in  $ \mathcal{B}(\mathcal{H}, \mathcal{H}_0)$ is contained in a maximal orthonormal set.
\item If $ \mathcal{B}(\mathcal{H}, \mathcal{H}_0)$ has an operator   $T$ such that $TT^*$ is bounded invertible, then  $ \mathcal{B}(\mathcal{H}, \mathcal{H}_0)$ has a maximal orthonormal set.
\end{enumerate}
\end{theorem}
\begin{proof}
\begin{enumerate}[\upshape(i)]
\item This is an application of Zorn's lemma to the poset $(\mathscr{P},\preceq)$, where 
$$\mathscr{P}=\{Z : Z ~\text{is an orthonormal set in}~  \mathcal{B}(\mathcal{H}, \mathcal{H}_0)~ \text{such that}~ Z \supseteq Y\} $$
and for $ Z_1,Z_2 \in \mathscr{P}, Z_1\preceq Z_2 $ if $ Z_1\subseteq Z_2.$
\item We apply (i) to the orthonormal set $ Y=\{(TT^*) ^{-1/2}T\}$.
\end{enumerate}
 \end{proof}
 It is interesting to observe that second condition in the definition  of orthonormal set comes from first. In fact,
 \begin{lemma}\label{FIRSTIMPLIESSECONDLEMMA}
 If $ \{A_j\}_{j \in \mathbb{J}}$ in     $ \mathcal{B}(\mathcal{H}, \mathcal{H}_0)$  satisfies $\langle A_j^*y, A_k^*z\rangle=\delta_{j,k}\langle y, z\rangle , \forall y, z \in \mathcal{H}_0, \forall j, k \in \mathbb{J}$, then $\sum_{j \in \mathbb{J}}\|A_jh\|^2\leq \|h\|^2, \forall h \in \mathcal{H}$.
 \end{lemma}
 \begin{proof}
 For every $ h \in \mathcal{H}$ and every finite subset $\mathbb{S} \subseteq \mathbb{J}$ we get 
 \begin{align*}
 0&\leq \left\| h-\sum_{j \in \mathbb{S}}A_j^*A_jh\right\|^2=\left\langle h-\sum_{j \in \mathbb{S}}A_j^*A_jh , h-\sum_{k \in \mathbb{S}}A_k^*A_kh \right\rangle \\
 &=\|h\|^2-2\sum_{j \in \mathbb{S}}\|A_jh\|^2+\sum_{j \in \mathbb{S}}\left\langle A_jh,A_j\left(\sum_{k \in \mathbb{S}}A_k^*A_kh\right) \right\rangle \\
 &=\|h\|^2-2\sum_{j \in \mathbb{S}}\|A_jh\|^2+\sum_{j \in \mathbb{S}}\|A_jh\|^2=\|h\|^2-\sum_{j \in \mathbb{S}}\|A_jh\|^2
\end{align*}
 $\Rightarrow\sum_{j \in \mathbb{S}}\|A_jh\|^2\leq \|h\|^2 $. Therefore $\sum_{j \in \mathbb{J}}\|A_jh\|^2\leq \|h\|^2$.
 \end{proof}
 \begin{remark}
 \begin{enumerate}[\upshape(i)]
 \item Lemma \ref{FIRSTIMPLIESSECONDLEMMA} is `an' extension of Bessel's inequality.
 \item Let $ Y$ be an orthogonal set of `vectors'  in a Hilbert space which excludes 0. The simple procedure of  dividing every element of $ Y$ by its norm gives an orthonormal set. This procedure no longer holds for `operators'. As an example, consider $ \ell^2(\mathbb{N})$. For each $n \in \mathbb{N}$, let $ P_n$ be the projection onto the $ n^{\text{th}}$ coordinate of $ \ell^2(\mathbb{N})$. Then $\langle P_j^*(\{a_n\}_{n\in\mathbb{N}}), P_k^*(\{b_n\}_{n\in\mathbb{N}})\rangle=0 , \forall \{a_n\}_{n\in\mathbb{N}}, \{b_n\}_{n\in\mathbb{N}}\in \ell^2(\mathbb{N}), \forall j, k \in \mathbb{N}, j\neq k ~ \text{and}~  \sum_{k=1}^{\infty}\|P_k(\{a_n\}_{n\in\mathbb{N}})\|^2=\sum_{k=1}^{\infty}|a_k|^2=\|\{a_n\}_{n\in\mathbb{N}}\|^2,  \forall \{a_n\}_{n\in\mathbb{N}} \in  \ell^2(\mathbb{N}).$
 Now $\|P_n\|=1, \forall n \in \mathbb{N}$ and $\langle P_1^*(\{1/n\}_{n\in\mathbb{N}}), P_1^*(\{1/n\}_{n\in\mathbb{N}})\rangle=1\neq \sum_{n=1}^{\infty}1/n^2=\langle \{1/n\}_{n\in\mathbb{N}}, \{1/n\}_{n\in\mathbb{N}}\rangle$. Therefore $ \left\{\frac{P_n}{\|P_n\|}=P_n\right\}_{n\in\mathbb{N}}$ is not an orthonormal set for $\mathcal{B}(\ell^2(\mathbb{N}))$.  However,  Proposition \ref{ORTHOGONALTOORTHONORMALOV} gives a procedure for converting certain types  of  orthogonal sets of operators into orthonormal sets. 
 \end{enumerate}
 \end{remark}
 \begin{proposition}\label{ORTHOGONALTOORTHONORMALOV}
 Let $ \{A_j\}_{j \in \mathbb{J}}$  be  orthogonal   in $ \mathcal{B}(\mathcal{H}, \mathcal{H}_0)$. If  $ A_jA_j^*$'s are bounded invertible for all $ j \in \mathbb{J}$, then $ \{U_j \coloneqq (A_jA_j^*)^{-1/2}A_j\}_{j \in \mathbb{J}}$ is  orthonormal   in $ \mathcal{B}(\mathcal{H}, \mathcal{H}_0)$ and $\overline{\operatorname{span}}_{\mathcal{B}(\mathcal{H}_0)}\{U_j\}_{j \in \mathbb{J}}=\overline{\operatorname{span}}_{\mathcal{B}(\mathcal{H}_0)}\{A_j\}_{j \in \mathbb{J}}$.
 \end{proposition}
 \begin{proof}
 $U_jU_k^*=((A_jA_j^*)^{-1/2}A_j)(A_k^*(A_kA_k^*)^{-1/2})=(A_jA_j^*)^{-1/2}0(A_kA_k^*)^{-1/2}=0, \forall j\neq k, \forall j,k \in \mathbb{J} $,  and $U_jU_j^*=(A_jA_j^*)^{-1/2}A_jA_j^*(A_jA_j^*)^{-1/2}=I_{\mathcal{H}_0} ,\forall j \in \mathbb{J}$. Inequality condition in the definition of orthonormal set comes from Lemma \ref{FIRSTIMPLIESSECONDLEMMA}.
\end{proof}
 
 \begin{theorem}
 \begin{enumerate}[\upshape(i)]
 \item If  $ \{A_n\}_{n=1}^m $ is  orthogonal  in $ \mathcal{B}(\mathcal{H}, \mathcal{H}_0)$, then $$\left\|\sum_{n=1}^{m}A_n^*y_n\right\|^2 =\sum_{n=1}^{m}\|A_n^*y_n\|^2, ~\forall y_1,...,y_n \in \mathcal{H}_0.$$ 
 In particular, $ \|A_1+\cdots+A_m\|^2\leq\|A_1\|^2+\cdots+\|A_m\|^2.$
 \item If  $ \{A_n\}_{n=1}^m $ is  orthonormal  in  $ \mathcal{B}(\mathcal{H}, \mathcal{H}_0)$, then $\|\sum_{n=1}^{m}A_n^*y_n\|^2 =\sum_{n=1}^{m}\|y_n\|^2, \forall y_1,...,y_n \in \mathcal{H}_0$. In particular, $ \|A_1+\cdots+A_m\|^2= m.$
 \item If $ \{A_j\}_{j \in \mathbb{J}} $ is  orthogonal   in  $ \mathcal{B}(\mathcal{H}, \mathcal{H}_0)$ such that $A_jA_j^* $ is invertible for all $ j \in \mathbb{J}$, then $ \{A_j\}_{j \in \mathbb{J}} $ is linearly independent over $ \mathbb{K}$ as well as over $\mathcal{B}(\mathcal{H}_0)$. In particular, if $ \{A_j\}_{j \in \mathbb{J}} $ is  orthonormal   in  $ \mathcal{B}(\mathcal{H}, \mathcal{H}_0)$, then it is linearly independent over $ \mathbb{K}$ as well as over   $\mathcal{B}(\mathcal{H}_0)$.
 \end{enumerate}
 \end{theorem}
 \begin{proof}
  \begin{enumerate}[\upshape(i)]
 \item $\|\sum_{n=1}^{m}A_n^*y_n\|^2 = \langle \sum_{n=1}^{m}A_n^*y_n, \sum_{k=1}^{m}A_k^*y_k \rangle=\sum_{n=1}^{m}\langle A_n^*y_n,  A_n^*y_n \rangle  =\sum_{n=1}^{m}\|A_n^*y_n\|^2$, $ \forall y_1,...,y_n \in \mathcal{H}_0$ and  $\|A_1+\cdots+A_m\|=\|A^*_1+\cdots+A_m^*\|=\sup_{g \in \mathcal{H}_0, \|g\|=1}\|(A^*_1+\cdots+A_m^*)g\|  \leq( \sum_{n=1}^{m}\|A_n^*\|^2)^{1/2}=(\sum_{n=1}^{m}\|A_n\|^2)^{1/2}$.
 \item follows from (i), since $ A_nA_n^*=I_{\mathcal{H}_0},n=1,...,m$.
 \item Let $ \mathbb{S}$ be a finite subset of $ \mathbb{J}$ and $ c_j \in \mathbb{K}$  (resp. $T_j \in \mathcal{B}(\mathcal{H}_0)$), $ j \in \mathbb{S}$  be such that $ \sum_{j\in\mathbb{S}}c_jA_j=0$ (resp. $\sum_{j\in\mathbb{S}}T_jA_j=0$). Then for each  fixed $ k \in  \mathbb{S} $, we get $c_kA_kA_k^*=\sum_{j\in\mathbb{S}}c_jA_jA_k^*=0$ (resp. $T_kA_kA_k^*=\sum_{j\in\mathbb{S}}T_jA_jA_k^*=0$) which implies $ c_k=0$ (resp. $ T_k=0$). When  $ \{A_j\}_{j \in \mathbb{J}} $ is orthonormal, $ A_jA_j^*=I_{\mathcal{H}_0}$,  hence it is linearly  independent over $ \mathbb{K}$ as well as over $\mathcal{B}(\mathcal{H}_0)$.
\end{enumerate}
 \end{proof}
 Following is an extension of Riesz-Fischer theorem.
\begin{theorem}\label{ERFOV}
Let $ \{A_j\}_{j \in \mathbb{J}} $ be  orthonormal  in $ \mathcal{B}(\mathcal{H}, \mathcal{H}_0)$,  $ \{U_j\}_{j \in \mathbb{J}} $ be  in $ \mathcal{B}(\mathcal{H}_0)$ and $ y \in \mathcal{H}_0$. Then
$$\sum_{j\in\mathbb{J}}A_j^*U_jy ~\text{converges in} ~  \mathcal{H} ~ \text{if and only if}~ \sum_{j\in\mathbb{J}}\|U_jy\|^2 ~\text{converges} . $$ 
\end{theorem}
\begin{proof}
For every finite subset $ \mathbb{S}$ of  $\mathbb{J}$ and $ y \in \mathcal{H}_0$ we get
\begin{equation*}
 \left\|\sum_{j\in \mathbb{S}}A_j^*U_jy \right\|^2=\left\langle\sum_{j\in \mathbb{S}}U_jy,A_j\left(\sum_{k\in \mathbb{S}}A_k^*U_ky\right) \right \rangle =\sum_{j\in \mathbb{S}}\|U_jy\|^2.
\end{equation*} 
Thus $\{\sum_{j\in \mathbb{S}}A_j^*U_jy:\mathbb{S}~\text{ is a finite subset of }~ \mathbb{J}\}$ is a Cauchy net if and only if $\{\sum_{j\in \mathbb{S}}\|U_jy\|^2:\mathbb{S}~\text{ is a finite subset of }~ \mathbb{J}\} $ is a Cauchy net. Since $\mathcal{H}$ and $\mathbb{K}$ are complete, result follows. 
\end{proof}
\begin{corollary}
Let $ \{A_j\}_{j \in \mathbb{J}} $ be  orthonormal in $ \mathcal{B}(\mathcal{H}, \mathcal{H}_0)$,  $ \{c_j\}_{j \in \mathbb{J}} $ be a sequence of scalars  and $ y \in \mathcal{H}_0$. Then
$$\sum_{j\in\mathbb{J}}c_jA_j^*y ~\text{converges in} ~  \mathcal{H} ~ \text{if and only if}~\{c_j\|y\|\}_{j \in \mathbb{J}} \in \ell^2(\mathbb{J}).$$ 	
In particular, if $y \in \mathcal{H}_0 $ is nonzero, then  $\sum_{j\in\mathbb{J}}c_jA_j^*y $ converges in $ \mathcal{H} $ if and only if $\{c_j\}_{j \in \mathbb{J}} \in \ell^2(\mathbb{J}).$ 	
\end{corollary}
 We set up the orthonormal and Riesz basis definitions in extended set up as 
\begin{definition}\label{ONBRIESZDEFINITION}
Let  $ \{A_j\}_{j \in \mathbb{J}},\{\Psi_j\}_{j \in \mathbb{J}} $ be in $ \mathcal{B}(\mathcal{H}, \mathcal{H}_0).$ 
We say 
\begin{enumerate}[\upshape(i)]
\item $ \{A_j\}_{j \in \mathbb{J}}$ is an orthonormal set (resp. basis)  w.r.t. $\{\Psi_j\}_{j \in \mathbb{J}} $ if  $ \{A_j\}_{j \in \mathbb{J}}$ or   $\{\Psi_j\}_{j \in \mathbb{J}} $ is an orthonormal set (resp. basis) in $ \mathcal{B}(\mathcal{H}, \mathcal{H}_0)$, say $ \{A_j\}_{j \in \mathbb{J}}$ is an orthonormal set  (resp. basis) in $ \mathcal{B}(\mathcal{H}, \mathcal{H}_0)$,   and there exists a sequence  $\{c_j\}_{j \in \mathbb{J}} $  of   reals such that  $ 0<\inf\{c_j\}_{j \in \mathbb{J}}\leq \sup\{c_j\}_{j \in \mathbb{J}}<\infty$   and $ \Psi_j=c_jA_j, \forall j \in \mathbb{J}.$ We write $ (\{A_j\}_{j \in \mathbb{J}}, \{\Psi_j\}_{j \in \mathbb{J}})$ is an orthonormal set (resp. basis).
\item $ \{A_j\}_{j \in \mathbb{J}}$ is a Riesz basis  w.r.t. $\{\Psi_j\}_{j \in \mathbb{J}} $ if there exists an orthonormal basis $ \{F_j\}_{j \in \mathbb{J}}$ in $ \mathcal{B}(\mathcal{H}, \mathcal{H}_0)$ and   invertible  $ U, V \in \mathcal{B}(\mathcal{H})$   with $ V^*U$ is positive  such that $ A_j=F_jU, \Psi_j=F_jV,  \forall j \in \mathbb{J}.$ We write $ (\{A_j\}_{j \in \mathbb{J}}, \{\Psi_j\}_{j \in \mathbb{J}})$ is Riesz a basis.
\end{enumerate}
\end{definition}
\begin{remark}
In \text{\upshape(ii)} of Definition \ref{ONBRIESZDEFINITION}, since $ U$ and $V$ are invertible, $ V^*U$ is not just positive, it is also invertible.
\end{remark}
 From the definition we see that whenever $( \{A_j\}_{j \in \mathbb{J}},\{\Psi_j\}_{j \in \mathbb{J}}) $  is an orthonormal (resp. Riesz) set (resp. basis), then $( \{\Psi_j\}_{j \in \mathbb{J}},\{A_j\}_{j \in \mathbb{J}}) $  is an orthonormal (resp. Riesz) set (resp. basis).
 
Whenever, $ \Psi_j=A_j,  \forall j \in \mathbb{J}$ we are forced to take $ c_j=1, \forall j \in \mathbb{J}$ (a right multiplication by $ A_j^*$ to $ A_j=c_jA_j$ gives this).  Thus Definition \ref{ONBRIESZDEFINITION} reduces to the definition of orthonormal basis for  `one collection case'.
  
\begin{theorem}\label{GBSOV}
 Let $( \{A_j\}_{j \in \mathbb{J}},\{\Psi_j=c_jA_j\}_{j \in \mathbb{J}}) $  be  orthonormal  in $ \mathcal{B}(\mathcal{H}, \mathcal{H}_0).$ Then 
 \begin{enumerate}[\upshape(i)]
\item Generalized Bessel's inequality - operator version: 
$$ \sum\limits_{j\in\mathbb{J}}(2-c_j)\langle A_jh, \Psi_jh\rangle\leq \|h\|^2 , ~\forall h \in \mathcal{H}.$$
If $ c_j\leq 2, \forall j \in \mathbb{J},$ then $ \sum_{j\in \mathbb{J}}(2-c_j)\Psi_j^*A_j\leq I_\mathcal{H}. $
\item For $ h \in \mathcal{H},$ 
 \begin{align*}
 h= \sum \limits_{j\in \mathbb{J}}\Psi_j^*A_jh \iff \sum\limits_{j\in\mathbb{J}}(2-c_j)\langle A_jh, \Psi_jh\rangle = \|h\|^2 \iff  \sum\limits_{j\in\mathbb{J}}c_j^2\| A_jh\|^2  = \|h\|^2.
 \end{align*}
 If $ c_j\leq 1, \forall j \in \mathbb{J},$ then 
 $ h= \sum_{j\in \mathbb{J}}\Psi_j^*A_jh\iff (1-c_j)A_jh=0 , \forall j \in \mathbb{J} \iff (1-c_j)A_j^*A_jh\perp h, \forall j \in \mathbb{J}.$

 \end{enumerate}
 \end{theorem}
 \begin{proof}
\begin{enumerate}[\upshape(i)]
 \item For  $ h \in \mathcal{H}$ and  each finite subset $ \mathbb{S} \subseteq \mathbb{J},$
 \begin{align*}
  \left\| \sum_{j\in \mathbb{S}} \Psi_j^*A_jh\right\|^2
  &= \left\langle  \sum_{j\in \mathbb{S}} c_jA_j^*A_jh, \sum_{k\in \mathbb{S}} c_kA_k^*A_kh\right\rangle=\sum_{j\in \mathbb{S}} c_j^2 \|A_jh\|^2\\
  &\leq \left(\sup\{c_j^2\}_{j \in \mathbb{J}}\right)\sum\limits_{j \in \mathbb{S}}\| A_jh\| ^2, ~ \text{which is convergent.}
  \end{align*}
  Therefore $ \sum_{j\in \mathbb{J}} \Psi_j^*A_jh$ exists and similarly $\sum_{j\in \mathbb{J}}(2-c_j)\Psi_j^*A_jh $ also exists. Then 
  \begin{align*}
  0&\leq \left\|h-\sum_{j\in \mathbb{J}} \Psi_j^*A_jh\right\|^2 = \left\langle h-\sum_{j\in \mathbb{J}}c_jA_j^*A_jh, h-\sum_{k\in \mathbb{J}}c_kA_k^*A_kh\right\rangle \\
  &=\|h\|^2-2\sum_{j\in \mathbb{J}}c_j\|A_jh\|^2+\sum_{j\in \mathbb{J}}c^2_j\|A_jh \|^2=\|h\|^2-\sum_{j\in \mathbb{J}}(2c_j-c_j^2)\|A_jh \|^2,
  \end{align*}
  $ \Rightarrow\sum_{j\in \mathbb{J}}(2c_j-c_j^2)\|A_jh \|^2= \sum_{j\in\mathbb{J}}(2-c_j)\langle A_jh,\Psi_jh\rangle  \leq \|h\|^2.$ 
  \item First 'if and only if' part  comes from (i). We next use -  $\{A_j\}_{j\in \mathbb{J}} $ is an orthonormal set, which gives other  if's and only if's.
  \end{enumerate}
 \end{proof}
 \begin{remark}
 Since $\{c_j\}_{j\in \mathbb{J}} $ is bounded, one can simply get $ \sum_{j\in\mathbb{J}}(2-c_j)\langle A_jh, \Psi_jh\rangle\leq M\|h\|^2 , \forall h \in \mathcal{H} ,$ for some real $ M>0$. In fact,  one choice for $ M$ is $ \operatorname{sup}\{c_j|2-c_j|\}_{j\in \mathbb{J}}$.  What is interesting in the Theorem \ref{GBSOV} is that we can reduce $ M$  upto one.
 \end{remark}
 \begin{corollary}\label{GFEOV}
(Generalized Fourier expansion) Let $( \{A_j\}_{j \in \mathbb{J}},\{\Psi_j=c_jA_j\}_{j \in \mathbb{J}}) $  be an orthonormal basis in $ \mathcal{B}(\mathcal{H}, \mathcal{H}_0).$ Then
$$ \frac{1}{\sup\{c_j\}_{j \in \mathbb{J}}}\sum_{j\in \mathbb{J}}\Psi_j^*A_j\leq I_\mathcal{H}\leq \frac{1}{\inf\{c_j\}_{j \in \mathbb{J}}}\sum_{j\in \mathbb{J}}\Psi_j^*A_j.$$
 \end{corollary}
 \begin{proof}
From the proof of Theorem \ref{GBSOV}, $\sum_{j\in \mathbb{J}}\Psi_j^*A_j $ exists, in SOT. Then $ \frac{1}{\sup\{c_j\}_{j \in \mathbb{J}}}\sum_{j\in \mathbb{J}}\Psi_j^*A_j=\frac{1}{\sup\{c_j\}_{j \in \mathbb{J}}}\sum_{j\in \mathbb{J}}c_jA_j^*A_j\leq \sum_{j\in \mathbb{J}}A_j^*A_j= I_\mathcal{H}\leq\frac{1}{\inf\{c_j\}_{j \in \mathbb{J}}}\sum_{j\in \mathbb{J}}c_jA_j^*A_j= \frac{1}{\inf\{c_j\}_{j \in \mathbb{J}}}\sum_{j\in \mathbb{J}}\Psi_j^*A_j.$
\end{proof}
\begin{remark}
Whenever $ \Psi_j=A_j ,\forall j \in \mathbb{J}, \mathcal{H}_0=\mathbb{C}$,  last corollary gives Fourier expansion.
\end{remark}
\begin{corollary}
 If  $( \{A_j\}_{j \in \mathbb{J}},\{\Psi_j=c_jA_j\}_{j \in \mathbb{J}}) $  is  an orthonormal basis in $ \mathcal{B}(\mathcal{H}, \mathcal{H}_0)$, then $\inf\{c_j\}_{j \in \mathbb{J}} \leq \|\sum_{j\in \mathbb{J}}\Psi_j^*A_j\| \leq \sup\{c_j\}_{j \in \mathbb{J}}.$
 \end{corollary}
 \begin{proof}
 Take norm in generalized Fourier expansion.
 \end{proof}
 \begin{corollary}
 If $( \{A_j\}_{j \in \mathbb{J}},\{\Psi_j=c_jA_j\}_{j \in \mathbb{J}}) $  is orthonormal  in $ \mathcal{B}(\mathcal{H}, \mathcal{H}_0)$ such that $c_j\leq 2, \forall j \in \mathbb{J},$ then  $\| \sum_{j\in \mathbb{J}}(2-c_j)\Psi_j^*A_j\|\leq 1. $	
 \end{corollary}
Following theorem is a generalization of  ``If $ \{e_j\}_{j \in \mathbb{J}}$ is  orthonormal for $ \mathcal{H}$, then for each $ h \in \mathcal{H}$, the set $ Y_h=\{e_j: \langle h, e_j\rangle \neq0,j \in \mathbb{J} \}$ is either finite or countable'', operator version.
\begin{theorem}
If $ (\{A_j\}_{j \in \mathbb{J}},\{\Psi_j=c_jA_j\}_{j \in \mathbb{J}}) $ is   orthonormal  in  $ \mathcal{B}(\mathcal{H}, \mathcal{H}_0)$ with $ c_j\leq2,\forall j \in \mathbb{J} $, then	for each $ h \in \mathcal{H}$, the set $ Y_h=\{A_j: (2-c_j)\langle  A_jh, \Psi_jh \rangle \neq0,j \in \mathbb{J} \}$ is either  finite or countable.
\end{theorem}
\begin{proof}
For $ n \in \mathbb{N}$, define 
 $$ Y_{n,h}\coloneqq\left\{A_j: (2-c_j)\langle  A_jh, \Psi_jh \rangle > \frac{1}{n}\|h\|^2, j\in \mathbb{J} \right\}.$$	
   
Suppose, for some $n$, $ Y_{n,h}$ has more than $n-1$ elements, say $A_1,...,A_n$. Then $ \sum_{j=1 }^n(2-c_j)\langle  A_jh, \Psi_jh \rangle > n\frac{1}{n}\|h\|^2=\|h\|^2$. From Theorem \ref{GBSOV},  $ \sum_{j\in\mathbb{J}}(2-c_j)\langle  A_jh, \Psi_jh \rangle \leq \|h\|^2 $. This gives $  \|h\|^2< \|h\|^2$ which is impossible. Hence $ \operatorname{Card}(Y_{n,h})\leq n-1$. Thus the  countable union  $\cup_{n=1}^\infty Y_{n,h}=Y_h$ is  countable.	
 \end{proof}
 \begin{theorem}\label{ONBIMPLIESOVF}
 \begin{enumerate}[\upshape(i)]
 \item If $ (\{A_j\}_{j \in \mathbb{J}},\{\Psi_j\}_{j \in \mathbb{J}}) $ is an orthonormal basis in  $ \mathcal{B}(\mathcal{H}, \mathcal{H}_0)$, then it is a Riesz basis.
 \item If $ (\{A_j=F_jU\}_{j \in \mathbb{J}},\{\Psi_j=F_jV\}_{j \in \mathbb{J}}) $ is a Riesz  basis in  $ \mathcal{B}(\mathcal{H}, \mathcal{H}_0)$, then it is an (ovf) with optimal frame bounds $ \|(V^*U)^{-1}\|^{-1}$ and  $\|V^*U\| $.
 \end{enumerate}
 \end{theorem} 
 \begin{proof}
 \begin{enumerate}[\upshape(i)]
 \item We may assume  $ \{A_j\}_{j \in \mathbb{J}}$ is an orthonormal basis.  Then there exists a sequence  $\{c_j\}_{j \in \mathbb{J}} $  of   reals such that  $ 0<\inf\{c_j\}_{j \in \mathbb{J}}\leq \sup\{c_j\}_{j \in \mathbb{J}}<\infty$   and $ \Psi_j=c_jA_j, \forall j \in \mathbb{J}.$ Define $ F_j\coloneqq A_j, \forall j \in \mathbb{J},  U\coloneqq I_\mathcal{H}$ and  $ V \coloneqq \sum_{j\in \mathbb{J}}c_jA_j^*A_j.$ From the proof of Theorem \ref{GBSOV}, $ V$ is a well-defined bounded operator. Then $ F_jU=F_jI_\mathcal{H}=F_j=A_j, F_jV=\sum_{k\in \mathbb{J}}c_kF_jA_k^*A_k=\sum_{k\in \mathbb{J}}c_kA_jA_k^*A_k=c_jA_j=\Psi_j, \forall j \in \mathbb{J}.$  Since all $c_j $'s are positive,  $ V$ is positive invertible, whose inverse is $\sum_{j\in \mathbb{J}}c_j^{-1}A_j^*A_j. $ We finally note $ V^*U=V\geq 0.$
 \item For every finite subset $\mathbb{S} $ of $\mathbb{J}$ and $ h \in \mathcal{H},$ we get  $ \|\sum_{j\in \mathbb{S}}L_jF_jh\|^2=\sum_{j\in \mathbb{S}}\|F_jh\|^2$ and this converges to $ \|h\|^2.$ Therefore $ \theta_A=\theta_FU$ exists as bounded operator. Also, $ \|\sum_{j\in \mathbb{S}}L_j\Psi_jh\|^2=\sum_{j\in \mathbb{S}}\|c_jA_jh\|^2 \leq (\sup\{c^2_j\}_{j \in \mathbb{J}})\sum_{j\in \mathbb{S}}\|A_jh\|^2  $ and this converges. Therefore $ \theta_\Psi=\theta_FV$ also exists as bounded operator. Now $ S_{A,\Psi}= \sum_{j\in\mathbb{J}}V^*F_j^*F_jU=V^*I_\mathcal{H}U=V^*U$ which is positive invertible. This also gives optimal frame bounds.
 \end{enumerate}
 \end{proof} 
 \begin{remark}
 Whenever  $ (\{A_j\}_{j \in \mathbb{J}},\{\Psi_j=c_jA_j\}_{j \in \mathbb{J}}) $ is an orthonormal basis in  $ \mathcal{B}(\mathcal{H}, \mathcal{H}_0),$ we can explicitly write the inverse of $ S_{A, \Psi}$ using $ c_j$'s and $ A_j$'s, and it is $\sum_{j\in\mathbb{J}} c_j^{-1}A_j^*A_j.$ In fact,  $ S_{A, \Psi}(\sum_{j\in\mathbb{J}} c_j^{-1}A_j^*A_j)=\sum_{k \in \mathbb{J}}c_kA_k^*(\sum_{j\in\mathbb{J}} c_j^{-1}A_kA_j^*A_j)= I_\mathcal{H}, (\sum_{j\in\mathbb{J}} c_j^{-1}A_j^*A_j)S_{A,\Psi}= \sum_{j\in\mathbb{J}} c_j^{-1}A_j^*(\sum_{k\in\mathbb{J}} c_k^{-1}A_jA_k^*A_k)=I_\mathcal{H}.$
 \end{remark}
 
 \begin{theorem}
 Let $ (\{A_j=F_jU\}_{j \in \mathbb{J}},\{\Psi_j=F_jV\}_{j \in \mathbb{J}}) $ be a Riesz  basis in  $ \mathcal{B}(\mathcal{H}, \mathcal{H}_0)$. Then
\begin{enumerate}[\upshape(i)]
 \item There exists unique  $ \{B_j\}_{j\in\mathbb{J}}, \{\Phi_j\}_{j\in\mathbb{J}}$ in $ \mathcal{B}(\mathcal{H}, \mathcal{H}_0)$ such that 
 $$ I_\mathcal{H}= \sum\limits_{j\in\mathbb{J}}B_j^*A_j = \sum\limits_{j\in\mathbb{J}}\Phi_j^*\Psi_j, \text{in SOT}$$ and $( \{B_j\}_{j\in\mathbb{J}}, \{\Phi_j\}_{j\in\mathbb{J}})$ is Riesz (ovf).
 \item $\{h \in \mathcal{H}: A_jh=0, \forall j \in \mathbb{J}\} =\{0\}= \{h \in \mathcal{H}: \Psi_jh=0, \forall j \in \mathbb{J}\} $, and $\overline{\operatorname{span}}\cup_{j\in \mathbb{J}} A^*_j(\mathcal{H}_0)=\mathcal{H}= \overline{\operatorname{span}}\cup_{j\in \mathbb{J}} \Psi^*_j(\mathcal{H}_0).$  If $ VU^*\geq0,$ then there are real $ a, b >0$ such that for every finite subset $ \mathbb{S}$ of $ \mathbb{J}$,
 $$ a\sum\limits_{j\in \mathbb{S}}|c_j|^2\|h\|^2\leq \left\langle \sum\limits_{j\in\mathbb{S}}c_jA^*_jh ,\sum\limits_{k\in\mathbb{S}}c_k\Psi_k^*h  \right \rangle \leq b\sum\limits_{j\in \mathbb{S}}|c_j|^2\|h\|^2, ~\forall h \in \mathcal{H},\forall c_j\in \mathbb{K}, \forall j\in \mathbb{S}.$$ 
 \end{enumerate}
 \end{theorem}
 \begin{proof}
 \begin{enumerate}[\upshape(i)]
 \item Define $ B_j\coloneqq F_j(U^{-1})^*, \Phi_j\coloneqq F_j(V^{-1})^*, \forall j \in \mathbb{J}.$ Then $\sum_{j\in\mathbb{J}}B_j^*A_j =\sum_{j\in\mathbb{J}}U^{-1}F^*_jF_jU =I_\mathcal{H}=\sum_{j\in\mathbb{J}}V^{-1}F^*_jF_jV= \sum_{j\in\mathbb{J}}\Phi_j^*\Psi_j$. Suppose there are $ \{C_j\}_{j\in\mathbb{J}}, \{\Xi_j\}_{j\in\mathbb{J}}$ in $ \mathcal{B}(\mathcal{H}, \mathcal{H}_0)$ such that $ I_\mathcal{H}= \sum_{j\in\mathbb{J}}C_j^*A_j = \sum_{j\in\mathbb{J}}\Xi_j^*\Psi_j$, then  $ I_\mathcal{H}= \sum_{j\in\mathbb{J}}C_j^*F_jU = \sum_{j\in\mathbb{J}}\Xi_j^*F_jV \Rightarrow U^{-1}=\sum_{j\in\mathbb{J}}C_j^*F_j ~\text{and}~ V^{-1}=\sum_{j\in\mathbb{J}}\Xi_j^*F_j\Rightarrow B_k^*=U^{-1}F_k^*=(\sum_{j\in\mathbb{J}}C_j^*F_j)F^*_k=C_k^* ~\text{and}~ \Phi_k^*=V^{-1}F_k^*=(\sum_{j\in\mathbb{J}}\Xi_j^*F_j)F^*_k=\Xi_k^*,\forall k \in \mathbb{J}$. Next, $((V^{-1})^*)^*(U^{-1})^*=V^{-1}(U^{-1})^*=(U^*V)^{-1}=((V^*U)^*)^{-1}\geq 0 $. Hence  $( \{B_j\}_{j\in\mathbb{J}}, \{\Phi_j\}_{j\in\mathbb{J}})$ is Riesz (ovf).
 \item Since $ U,V$ are bounded invertible, we get the first part. Whenever $ VU^*\geq0,$
 \begin{align*}
 \frac{1}{\|VU^*\|^{-1}}\sum\limits_{j\in \mathbb{S}}|c_j|^2\|h\|^2&= \frac{1}{\|VU^*\|^{-1}}\left\langle \sum\limits_{j\in\mathbb{S}}c_jF^*_jh ,\sum\limits_{k\in\mathbb{S}}c_kF_k^*h\right\rangle
 \leq \left\langle VU^*\left( \sum\limits_{j\in\mathbb{S}}c_jF^*_jh\right) ,\sum\limits_{k\in\mathbb{S}}c_kF_k^*h\right\rangle\\
 &=\left\langle \sum\limits_{j\in\mathbb{S}}c_jU^*F^*_jh ,\sum\limits_{k\in\mathbb{S}}c_kV^*F_k^*h\right\rangle
 =\left\langle \sum\limits_{j\in\mathbb{S}}c_jA_j^*h ,\sum\limits_{k\in\mathbb{S}}c_k\Psi_k^*h\right\rangle
 \\
 &\leq \|VU^*\|\left\langle \sum\limits_{j\in\mathbb{S}}c_j^*F^*_jh ,\sum\limits_{k\in\mathbb{S}}c_kF_k^*h\right\rangle
 = \|VU^*\|\sum\limits_{j\in \mathbb{S}}|c_j|^2\|h\|^2,\forall h \in \mathcal{H},\forall c_j\in \mathbb{K}, \forall j\in \mathbb{S}.
 \end{align*}
 \end{enumerate}
 \end{proof}
Using frames  we can describe the left-inverses of  analysis operators completely, as next proposition shows.
\begin{proposition}
Let $ (\{A_j\}_{j\in \mathbb{J}}, \{\Psi_j\}_{j\in \mathbb{J}}) $  be an  (ovf) in    $\mathcal{B}(\mathcal{H}, \mathcal{H}_0)$. Then the bounded left-inverses of 
\begin{enumerate}[\upshape(i)]
\item $ \theta_A$ are precisely  $S_{A,\Psi}^{-1}\theta_\Psi^*+U(I_{\ell^2(\mathbb{J})\otimes\mathcal{H}_0}-\theta_AS_{A,\Psi}^{-1}\theta_\Psi^*)$, where $U \in \mathcal{B}(\ell^2(\mathbb{J})\otimes\mathcal{H}_0, \mathcal{H})$.
\item $ \theta_\Psi$ are precisely  $S_{A,\Psi}^{-1}\theta_A^*+V(I_{\ell^2(\mathbb{J})\otimes\mathcal{H}_0}-\theta_\Psi S_{A,\Psi}^{-1}\theta_A^*)$, where $V\in \mathcal{B} (\ell^2(\mathbb{J})\otimes\mathcal{H}_0, \mathcal{H})$.
\end{enumerate}	
\end{proposition}
\begin{proof}
We prove (ii).  $(\Leftarrow)$ Let $V: \ell^2(\mathbb{J})\otimes\mathcal{H}_0\rightarrow \mathcal{H}$ be a bounded operator. Then $(S_{A,\Psi}^{-1}\theta_A^*+V(I_{\ell^2(\mathbb{J})\otimes\mathcal{H}_0}-\theta_\Psi S_{A,\Psi}^{-1}\theta_A^*))\theta_\Psi=I_\mathcal{H}+V\theta_\Psi-V\theta_\Psi I_\mathcal{H}=I_\mathcal{H}$. Therefore  $S_{A,\Psi}^{-1}\theta_A^*+V(I_{\ell^2(\mathbb{J})\otimes\mathcal{H}_0}-\theta_\Psi S_{A,\Psi}^{-1}\theta_A^*)$ is a bounded left-inverse of $\theta_\Psi$.
   
$(\Rightarrow)$ Let $ L:\ell^2(\mathbb{J})\otimes\mathcal{H}_0\rightarrow \mathcal{H}$ be a bounded left-inverse of $ \theta_\Psi$. Define $V\coloneqq L$. Then $S_{A,\Psi}^{-1}\theta_A^*+V(I_{\ell^2(\mathbb{J})\otimes\mathcal{H}_0}-\theta_\Psi S_{A,\Psi}^{-1}\theta_A^*) =S_{A,\Psi}^{-1}\theta_A^*+L(I_{\ell^2(\mathbb{J})\otimes\mathcal{H}_0}-\theta_\Psi S_{A,\Psi}^{-1}\theta_A^*)=S_{A,\Psi}^{-1}\theta_A^*+L-I_{\mathcal{H}}S_{A,\Psi}^{-1}\theta_A^*= L$. 	
 \end{proof} 
Following proposition is fundamental in the development of theory.
\begin{proposition}\label{2.2}
For every $ \{A_j\}_{j \in \mathbb{J}}  \in \mathscr{F}_\Psi$,
 \begin{enumerate}[ \upshape (i)]
\item $ \theta_A^* \theta_A =  \sum_{j\in \mathbb{J}} A_j^*A_j $ in SOT.
\item $ S_{A, \Psi} = \theta_\Psi^*\theta_A=\theta_A^*\theta_\Psi =S_{\Psi,A}.$
\item $(\{A_j\}_{j\in \mathbb{J}}, \{\Psi_j\}_{j\in \mathbb{J}}) $  is Parseval  if and only if $  \theta_\Psi^*\theta_A =I_\mathcal{H}.$ 
\item $(\{A_j\}_{j\in \mathbb{J}}, \{\Psi_j\}_{j\in \mathbb{J}}) $  is Parseval  if and only if $ \theta_A\theta_\Psi^* $ is idempotent.
\item $  A_j=L_j^*\theta_A, \forall j\in \mathbb{J}.$
\item $\theta_AS_{A,\Psi}^{-1}\theta_\Psi^* $ is idempotent.
\item $\theta_A $ and $ \theta_\Psi$ are injective and  their ranges are closed.
\item  $\theta_A^* $ and $ \theta_\Psi^*$ are surjective.
\end{enumerate}
\end{proposition}
\begin{proof} Let $ h, g \in \mathcal{H}.$ We observe 
$$  \left\langle  \theta_A^* \theta_Ah, g \right\rangle= \left\langle  \sum\limits_{j\in \mathbb{J}} L_jA_jh,  \sum\limits_{k\in \mathbb{J}} L_kA_kg\right\rangle= \sum\limits_{j\in \mathbb{J}}\sum\limits_{k\in \mathbb{J}} \langle  L_k^*L_jA_jh,   A_kg\rangle=  \sum\limits_{j\in \mathbb{J}} \langle  A_jh,   A_jg\rangle=\left\langle\sum\limits_{j\in \mathbb{J}}A_j^*A_jh, g \right\rangle ,$$ 
and 
$$ 
\langle S_{A, \Psi}h, g \rangle = \sum\limits_{j\in \mathbb{J}} \langle  A_jh,   \Psi_jg\rangle= \sum\limits_{j\in \mathbb{J}}\sum\limits_{k\in \mathbb{J}} \langle  L_k^*L_jA_jh,   \Psi_kg\rangle= \langle  \theta_Ah, \theta_\Psi g\rangle=\langle\theta_\Psi^*\theta_Ah ,g \rangle ,$$
 thus proving (i) and (ii); (iii) follows from (ii). For (iv),  if $(\{A_j\}_{j\in \mathbb{J}}, \{\Psi_j\}_{j\in \mathbb{J}}) $  is Parseval then $ (\theta_A\theta_\Psi^*)^2=\theta_A\theta_\Psi^*\theta_A\theta_\Psi^*=\theta_AI_\mathcal{H}\theta_\Psi^*=\theta_A\theta_\Psi^* $, and if $ \theta_A\theta_\Psi^*$ is idempotent then multiplying from the left by $\theta_\Psi^* $ and from the right by $ \theta_A$ to  $ \theta_A\theta_\Psi^*\theta_A\theta_\Psi^*=\theta_A\theta_\Psi^*$ gives $ S_{A, \Psi}^3=S_{A, \Psi}^2$.  Since $ L_j^* $ is linear, we get (v). For (vi), $(\theta_AS_{A,\Psi}^{-1}\theta_\Psi^*)^2=\theta_AS_{A,\Psi}^{-1}(\theta_\Psi^*\theta_AS_{A,\Psi}^{-1})\theta_\Psi^*=\theta_AS_{A,\Psi}^{-1}\theta_\Psi^* $; (ii) implies  first part of (vii), and (viii). For the second part of (vii), let $ \{x_n\}_{n=1}^\infty$ in $\mathcal{H} $  be such that  $ \{\theta _Ax_n\}_{n=1}^\infty$ converges to $ y \in \mathcal{H}_0$. This gives $ S_{A,\Psi}x_n \rightarrow \theta_\Psi^*y$  as $ n \rightarrow \infty$ and this in turn gives $x_n \rightarrow S_{A,\Psi}^{-1} \theta_\Psi^*y $  as $ n \rightarrow \infty.$ An application of $ \theta_A$ gives $\theta_Ax_n \rightarrow \theta_AS_{A,\Psi}^{-1} \theta_\Psi^*y $ as $ n \rightarrow \infty.$ Therefore $ y=\theta_A(S_{A,\Psi}^{-1} \theta_\Psi^*y).$
\end{proof}

Because of (ii), i.e., $ \sum_{j\in\mathbb{J}}\Psi_j^*A_j=\theta_\Psi^*\theta_A$ in the last theorem, Definition \ref{1} is equivalent to 
\begin{definition}
A collection $ \{A_j\}_{j \in \mathbb{J}} $   in $ \mathcal{B}(\mathcal{H}, \mathcal{H}_0)$ is said to be an (ovf)  w.r.t. $ \{\Psi_j\}_{j \in \mathbb{J}}  $ in $ \mathcal{B}(\mathcal{H}, \mathcal{H}_0) $ if there exist $ a, b, c , d>0$ such that 
\begin{enumerate}[\upshape(i)]
\item $\sum_{j\in \mathbb{J}}\Psi_j^*A_j=\sum_{j\in \mathbb{J}}A_j^*\Psi_j,$
\item $a\|h\|^2\leq\sum_{j\in \mathbb{J}}\langle A_jh, \Psi_jh\rangle  \leq b\|h\|^2, \forall h \in \mathcal{H},$	
\item $\sum_{j\in \mathbb{J}}\| A_jh\|^2 \leq c\|h\|^2, \forall h \in \mathcal{H}; \sum_{j\in \mathbb{J}} \|\Psi_jh\|^2  \leq d\|h\|^2, \forall h \in \mathcal{H}.$
\end{enumerate}
\end{definition}
Condition (iii) in the previous definition says both $\theta_A, \theta_\Psi$ exist and hence $ S_{A,\Psi}$ exists and is bounded. Also   (i) and  (ii) say that $S_{A,\Psi}$ is positive invertible.

Following the last theorem, one sees that  $ \theta_\Psi^* \theta_\Psi =  \sum_{j\in \mathbb{J}} \Psi_j^*\Psi_j $ in SOT,  and $ \Psi_j=L_j^*\theta_\Psi $ for every $ j\in \mathbb{J}$. We call the idempotent operator  $ P_{A, \Psi}\coloneqq\theta_AS_{A,\Psi}^{-1}\theta_\Psi^*$ as the \textit{`frame idempotent'} for $(\{A_j\}_{j\in \mathbb{J}}, \{\Psi_j\}_{j\in \mathbb{J}}) $  whose range is closed. Moreover, $ P_{\Psi, A}=P_{A, \Psi}^*.$
 \begin{definition}\label{RIESZOVF}
 An (ovf)  $(\{A_j\}_{j\in \mathbb{J}}, \{\Psi_j\}_{j\in \mathbb{J}})$  in $\mathcal{B}(\mathcal{H}, \mathcal{H}_0)$ is said to be a Riesz  (ovf)  if $ P_{A,\Psi}= I_{\ell^2(\mathbb{J})}\otimes I_{\mathcal{H}_0}$. A Parseval and  Riesz (ovf) (i.e., $\theta_\Psi^*\theta_A=I_\mathcal{H} $ and  $\theta_A\theta_\Psi^*=I_{\ell^2(\mathbb{J})}\otimes I_{\mathcal{H}_0} $) is called as an orthonormal (ovf). In other words,   $(\{A_j\}_{j\in \mathbb{J}}, \{\Psi_j\}_{j\in \mathbb{J}})$ is Riesz (ovf) if $ \theta_\Psi^*$ (resp. $\theta^*_A$) is right inverse of $ \theta_A$ (resp. $ \theta_\Psi$), and $(\{A_j\}_{j\in \mathbb{J}}, \{\Psi_j\}_{j\in \mathbb{J}})$ is orthonormal (ovf) if $ \theta_\Psi^*$ (resp. $\theta^*_A$) is  inverse (both right and left) of $ \theta_A$ (resp. $ \theta_\Psi$).
 \end{definition}
 An adjoint operation in the Definition \ref{RIESZOVF} reveals that it is symmetric.
 \begin{caution}
 Definition of Riesz basis (resp. orthonormal basis)  and definition of Riesz (ovf) (resp. orthonormal (ovf)) are different. Theorem \ref{OPERATORVALUEDCHARACTERIZATIONSOFTHEEXTENSION} further illustrates this.
 \end{caution}
 \begin{proposition}\label{RBIMPLIESRFOV}
 \begin{enumerate}[\upshape(i)]
 \item If $ (\{A_j\}_{j \in \mathbb{J}},\{\Psi_j\}_{j \in \mathbb{J}}) $ is a Riesz basis  in  $ \mathcal{B}(\mathcal{H}, \mathcal{H}_0)$, then it is a Riesz  (ovf).
\item If $ (\{A_j\}_{j \in \mathbb{J}},\{\Psi_j=c_jA_j\}_{j \in \mathbb{J}}) $ is an orthonormal  basis in  $ \mathcal{B}(\mathcal{H}, \mathcal{H}_0)$, then it is a Riesz (ovf).
\end{enumerate}
\end{proposition} 
\begin{proof}
 \begin{enumerate}[\upshape(i)]
 \item Let $ \{F_j\}_{j \in \mathbb{J}}$  be an orthonormal basis in  $ \mathcal{B}(\mathcal{H}, \mathcal{H}_0)$ and  $ U, V : \mathcal{H} \rightarrow \mathcal{H}$ be  bounded invertible with $ V^*U$ is positive  such that $ A_j=F_jU, \Psi_j=F_jV,  \forall j \in \mathbb{J}$. We proved in Theorem \ref{ONBIMPLIESOVF} that a Riesz basis is an (ovf). For Riesz, 
 \begin{align*}
  P_{A,\Psi} &=\left(\sum_{j\in\mathbb{J}}L_jA_j\right)(V^*U)^{-1}\left(\sum_{k\in\mathbb{J}}\Psi_k^*L_k^*\right)=\left(\sum_{j\in\mathbb{J}}L_jF_jU\right)U^{-1}{V^*}^{-1}\left(\sum_{k\in\mathbb{J}}V^*F_k^*L_k^*\right)\\
  &=\sum_{j\in\mathbb{J}}L_jF_j\left(\sum_{k\in\mathbb{J}}F_k^*L_k^*\right)=\sum_{j\in\mathbb{J}}L_jL^*_j=I_{\ell^2(\mathbb{J})}\otimes I_{\mathcal{H}_0}.
  \end{align*}
 \item $S_{A,\Psi}=\sum_{j\in \mathbb{J}}c_jA_j^*A_j$; $S_{A,\Psi}^{-1}=\sum_{j\in \mathbb{J}}c_j^{-1}A_j^*A_j$,\\
  $P_{A,\Psi}=\theta_AS_{A,\Psi}^{-1}\theta_\Psi^*= (\sum_{k\in \mathbb{J}}L_kA_k)(\sum_{j\in \mathbb{J}}c_j^{-1}A_j^*A_j)(\sum_{l\in \mathbb{J}}c_lA^*_lL_l^*)=(\sum_{k\in \mathbb{J}}L_kA_k)(\sum_{j\in \mathbb{J}}A^*_jL_j^*)\\ =\sum_{k\in \mathbb{J}}L_kL_k^* = I_{\ell^2(\mathbb{J})}\otimes I_{\mathcal{H}_0}.$
 \end{enumerate}
 \end{proof}
 \begin{remark}
 Statement  \text{\upshape(ii)} in  Proposition \ref{RBIMPLIESRFOV} can also be derived from \text{\upshape(i)} of Theorem \ref{ONBIMPLIESOVF} and \text{\upshape(i)} of Proposition \ref{RBIMPLIESRFOV}.
 \end{remark}
 \begin{proposition}\label{RIESZOVFCHARACTERIZATIONPROPOSITION}
 An (ovf) $ (\{A_j\}_{j\in \mathbb{J}}, \{\Psi_j\}_{j\in \mathbb{J}}) $ in $ \mathcal{B}(\mathcal{H}, \mathcal{H}_0)$ is Riesz (ovf)  if and only if  $\theta_A(\mathcal{H})=\ell^2(\mathbb{J})\otimes \mathcal{H}_0  $ if and only if $\theta_\Psi(\mathcal{H})=\ell^2(\mathbb{J})\otimes \mathcal{H}_0.$
 \end{proposition}
 \begin{proof}
 Let $ P_{A,\Psi}= I_{\ell^2(\mathbb{J})}\otimes I_{\mathcal{H}_0}.$ Now $\ell^2(\mathbb{J})\otimes \mathcal{H}_0=P_{A,\Psi}(\mathcal{H})=\theta_AS^{-1}_{A,\Psi}\theta^*_\Psi(\mathcal{H})\subseteq\theta_A(\mathcal{H})\subseteq\ell^2(\mathbb{J})\otimes \mathcal{H}_0.$
 If $ \theta_A(\mathcal{H})=\ell^2(\mathbb{J})\otimes \mathcal{H}_0 $ and $ y \in\ell^2(\mathbb{J})\otimes \mathcal{H}_0, $ then there exists an $ x \in \mathcal{H}$ such that $y=\theta_Ax .$ Now $ y=\theta_AS_{A,\Psi}^{-1}S_{A,\Psi}x=(\theta_AS_{A,\Psi}^{-1}\theta_\Psi^*)\theta_Ax =P_{A,\Psi}y.$ The symmetry of definition of (ovf)  proves the second part of `if and only if.' 
 \end{proof} 
\begin{proposition}
An (ovf) $ (\{A_j\}_{j\in \mathbb{J}}, \{\Psi_j\}_{j\in \mathbb{J}}) $ in $ \mathcal{B}(\mathcal{H}, \mathcal{H}_0)$ is orthonormal (ovf)  if and only if it is Parseval (ovf) and $ A_j\Psi_k^*=\delta_{j,k}I_{\mathcal{H}_0},\forall j, k \in \mathbb{J}$. 
\end{proposition}
\begin{proof}
$(\Rightarrow)$	
Definition of  orthonormal (ovf) includes Parseval. Now $\theta_A\theta_\Psi^*=I_{\ell^2(\mathbb{J})}\otimes I_{\mathcal{H}_0}.$ Hence $ e_m\otimes y=\theta_A\theta_\Psi^*(e_m\otimes y)=\sum_{j\in \mathbb{J}}L_jA_j(\sum_{k\in \mathbb{J}}\Psi^*_kL_k^*(e_m\otimes y))=\sum_{j\in \mathbb{J}}L_jA_j\Psi^*_my=\sum_{j\in \mathbb{J}}(e_j\otimes A_j\Psi^*_my)= e_m\otimes( A_j\Psi^*_m y+\sum_{j\in \mathbb{J},j\neq m}(e_j\otimes A_j\Psi^*_my)),\forall m \in \mathbb{J}, \forall y \in\mathcal{H}_0 $. Thus $A_j\Psi^*_my=\delta_{j,m}y,\forall y \in \mathcal{H}_0$, since $\{e_j\}_{j\in \mathbb{J}}$ is orthonormal.

$(\Leftarrow)$ We have to show that $ (\{A_j\}_{j\in \mathbb{J}}, \{\Psi_j\}_{j\in \mathbb{J}}) $ is Riesz (ovf). Consider $ \theta_A\theta_\Psi^*=\sum_{j\in \mathbb{J}}L_jA_j(\sum_{k\in \mathbb{J}}\Psi_k^*L_k^*)=\sum_{j\in \mathbb{J}}L_jL_j^*=I_{\ell^2(\mathbb{J})}\otimes I_{\mathcal{H}_0}.$

\end{proof}
 Next, we generalize general Naimark/Han/Larson dilation theorem \cite{HANLI1}.
 \begin{theorem}\label{OPERATORDILATION}
 Let $(\{A_j\}_{j\in \mathbb{J}},\{\Psi_j\}_{j\in \mathbb{J}} )$ be  a Parseval (ovf) in $ \mathcal{B}(\mathcal{H}, \mathcal{H}_0)$ such that $ \theta_A(\mathcal{H})=\theta_\Psi(\mathcal{H})$ and $ P_{A,\Psi}$ is projection. Then there exist a Hilbert space $ \mathcal{H}_1 $ which contains $ \mathcal{H}$ isometrically and  bounded linear operators $B_j,\Phi_j:\mathcal{H}_1\rightarrow \mathcal{H}_0, \forall j \in \mathbb{J} $ such that $(\{B_j\}_{j\in \mathbb{J}} ,\{\Phi_j\}_{j\in \mathbb{J}})$ is an orthonormal (ovf) in $ \mathcal{B}(\mathcal{H}_1, \mathcal{H}_0)$ and $B_j|_{\mathcal{H}}=  A_j,\Phi_j|_{\mathcal{H}}=\Psi_j, \forall j \in \mathbb{J}$. 
 \end{theorem}
  \begin{proof}
We first note that  $P_{A,\Psi}$ is the  orthogonal projection from $ \ell^2(\mathbb{J})\otimes \mathcal{H}_0$ onto $\theta_A(\mathcal{H})=\theta_\Psi(\mathcal{H})$. Define $ \mathcal{H}_1\coloneqq\mathcal{H}\oplus \theta_A(\mathcal{H})^\perp$. Then $\mathcal{H} \ni h \mapsto h\oplus 0 \in \mathcal{H}_1 $ is an isometry.  Define $ B_j:\mathcal{H}_1\ni h\oplus g\mapsto A_jh+L_j^*P_{A,\Psi}^\perp g \in \mathcal{H}_0,  \Phi_j:\mathcal{H}_1\ni h\oplus g\mapsto \Psi_jh+L_j^*P_{A,\Psi}^\perp g \in \mathcal{H}_0 , \forall j \in \mathbb{J}$. Clearly $B_j|_{\mathcal{H}}=  A_j,\Phi_j|_{\mathcal{H}}=\Psi_j, \forall j \in \mathbb{J}$. Now $ \theta_B(h\oplus g)=\sum_{j\in\mathbb{J}}L_jA_jh+\sum_{j\in\mathbb{J}}L_jL_j^*P_{A,\Psi}^\perp g=\theta_Ah+P_{A,\Psi}^\perp g, \forall  h\oplus g \in \mathcal{H}_1$. Similarly $\theta_\Phi(h\oplus g)=\theta_\Psi h+P_{A,\Psi}^\perp g, \forall h\oplus g\in \mathcal{H}_1 $.  Now $\langle \theta_B^*z,h\oplus g \rangle= \langle z,  \theta_B(h\oplus g) \rangle = \langle \theta_A^*z,  h\rangle+\langle P_{A,\Psi}^\perp z,  g\rangle = \langle \theta_A^*z\oplus P_{A,\Psi}^\perp z, h\oplus g\rangle , \forall z \in \ell^2(\mathbb{J})\otimes \mathcal{H}_0 , \forall  h\oplus g \in \mathcal{H}_1$. Hence $\theta_B^*z=\theta_A^*z\oplus P_{A,\Psi}^\perp z, \forall z \in \ell^2(\mathbb{J})\otimes \mathcal{H}_0 $ and similarly $\theta_\Phi^*z=\theta_\Psi^*z\oplus P_{A,\Psi}^\perp z, \forall z \in \ell^2(\mathbb{J})\otimes \mathcal{H}_0$.
 By  using $\theta_A(\mathcal{H})=\theta_\Psi(\mathcal{H}) $ and $\theta_\Psi^*P_{A,\Psi}^\perp=0=P_{A,\Psi}^\perp\theta_A ,$ we get  $ S_{B,\Phi}(h\oplus g)= \theta_\Phi^*(\theta_Ah+ P_{A,\Psi}^\perp g)=\theta_\Psi^*(\theta_Ah+P_{A,\Psi}^\perp g)\oplus P_{A,\Psi}^\perp(\theta_Ah+P_{A,\Psi}^\perp g)=(S_{A,\Psi}h+0)\oplus(0+P_{A,\Psi}^\perp g)=S_{A,\Psi}h\oplus P_{A,\Psi}^\perp g=I_\mathcal{H}h\oplus I_{\theta_A(\mathcal{H})^\perp}g, \forall h\oplus g\in \mathcal{H}_1$. Hence $(\{B_j\}_{j\in \mathbb{J}},\{\Phi_j\}_{j\in \mathbb{J}} )$ is  Parseval (ovf) in $ \mathcal{B}(\mathcal{H}_1, \mathcal{H}_0)$. We find $ P_{B,\Phi}z=\theta_BS_{B,\Phi}^{-1}\theta_\Phi^*z=\theta_B\theta_\Phi^*z=\theta_B(\theta_\Psi^*z\oplus P_{A,\Psi}^\perp z)=\theta_A(\theta_\Psi^*z)+ P_{A,\Psi}^\perp(P_{A,\Psi}^\perp z)=P_{A,\Psi} z+P_{A,\Psi}^\perp z=P_{A,\Psi} z+(I_{\ell^2(\mathbb{J})}\otimes I_{\mathcal{H}_0})-P_{A,\Psi})z =I_{\ell^2(\mathbb{J})}\otimes I_{\mathcal{H}_0} z, \forall z \in  \ell^2(\mathbb{J})\otimes \mathcal{H}_0$. 
 Hence $(\{B_j\}_{j\in \mathbb{J}},\{\Phi_j\}_{j\in \mathbb{J}})$ is  Riesz-(ovf) in $ \mathcal{B}(\mathcal{H}_1, \mathcal{H}_0)$. Thus $(\{B_j\}_{j\in \mathbb{J}},\{\Phi_j\}_{j\in \mathbb{J}})$ is  orthonormal (ovf) in $ \mathcal{B}(\mathcal{H}_1, \mathcal{H}_0)$.
  \end{proof}
 \begin{definition}
 An (ovf)  $(\{B_j\}_{j\in \mathbb{J}}, \{\Phi_j\}_{j\in \mathbb{J}})$  in $\mathcal{B}(\mathcal{H}, \mathcal{H}_0)$ is said to be a dual of  (ovf) $ (\{A_j\}_{j\in \mathbb{J}}, \{\Psi_j\}_{j\in \mathbb{J}})$ in $\mathcal{B}(\mathcal{H}, \mathcal{H}_0)$  if $ \theta_\Phi^*\theta_A= \theta_B^*\theta_\Psi=I_{\mathcal{H}}$. The `operator-valued frame' $(\{\widetilde{A}_j\coloneqq A_jS_{A,\Psi}^{-1}\}_{j\in \mathbb{J}}, \{\widetilde{\Psi}_j\coloneqq\Psi_jS_{A,\Psi}^{-1}\}_{j \in \mathbb{J}} )$, which is a `dual' of $ (\{A_j\}_{j\in \mathbb{J}}, \{\Psi_j\}_{j\in \mathbb{J}})$ is called the canonical dual of $ (\{A_j\}_{j\in \mathbb{J}}, \{\Psi_j\}_{j\in \mathbb{J}})$.
 \end{definition}
 From the duality definition it is clear that  $ (\{B_j\}_{j\in \mathbb{J}}, \{\Phi_j\}_{j\in \mathbb{J}})$ is dual of $ (\{A_j\}_{j\in \mathbb{J}}, \{\Psi_j\}_{j\in \mathbb{J}})$ if and only if both 
 $ (\{A_j\}_{j\in \mathbb{J}}, \{\Phi_j\}_{j\in \mathbb{J}})$ and $ (\{\Psi_j\}_{j\in \mathbb{J}}, \{B_j\}_{j\in \mathbb{J}})$ are Parseval operator-valued frames.

 By taking adjoint in the duality condition, we see that $(\{B_j\}_{j\in \mathbb{J}}, \{\Phi_j\}_{j\in \mathbb{J}})$ is dual of $( \{A_j\}_{j\in \mathbb{J}}, \{\Psi_j\}_{j\in \mathbb{J}})$ if and only if $(\{A_j\}_{j\in \mathbb{J}}, \{\Psi_j\}_{j\in \mathbb{J}})$ is dual of $(\{B_j\}_{j\in \mathbb{J}}, \{\Phi_j\}_{j\in \mathbb{J}})$. Further, if  $\{B_j\}_{j\in \mathbb{J}}, \{C_j\}_{j\in \mathbb{J}}\in \mathscr{F}_\Psi $ are duals of $ \{A_j\}_{j\in \mathbb{J}} \in \mathscr{F}_\Psi$, then the `operator-valued frame' $(\left\{(B_j+C_j)/2\right\}_{j\in \mathbb{J}},\{\Psi_j\}_{j\in \mathbb{J}}) $ is also a dual of $( \{A_j\}_{j\in \mathbb{J}}, \{\Psi_j\}_{j\in \mathbb{J}}).$
 
 Sun  showed that canonical dual frame gives rise to expansion coefficients with minimum norm (Lemma 2.1 in \cite{SUN1}). Following is the result in the extended setting.
 \begin{proposition}
 Let $( \{A_j\}_{j\in \mathbb{J}}, \{\Psi_j\}_{j\in \mathbb{J}} )$ be an (ovf) in $ \mathcal{B}(\mathcal{H}, \mathcal{H}_0).$  If $ h \in \mathcal{H}$ has representation  $ h=\sum_{j\in\mathbb{J}}A_j^*y_j= \sum_{j\in\mathbb{J}}\Psi_j^*z_j, $ for some $ \{y_j\}_{j\in \mathbb{J}},\{z_j\}_{j\in \mathbb{J}}$ in $ \mathcal{H}_0$, then 
 $$ \sum\limits_{j\in \mathbb{J}}\langle y_j,z_j\rangle =\sum\limits_{j\in \mathbb{J}}\langle \widetilde{\Psi}_jh,\widetilde{A}_jh\rangle+\sum\limits_{j\in \mathbb{J}}\langle y_j-\widetilde{\Psi}_jh,z_j-\widetilde{A}_jh\rangle. $$
 \end{proposition}
 \begin{proof}
 \begin{align*}
 \text{Right side}&=\sum\limits_{j\in \mathbb{J}}\langle \widetilde{\Psi}_jh,\widetilde{A}_jh\rangle+\sum\limits_{j\in \mathbb{J}}\langle y_j, z_j\rangle -\sum\limits_{j\in \mathbb{J}}\langle y_j, \widetilde{A}_jh\rangle-\sum\limits_{j\in \mathbb{J}}\langle \widetilde{\Psi}_jh, z_j\rangle +\sum\limits_{j\in \mathbb{J}}\langle \widetilde{\Psi}_jh,\widetilde{A}_jh\rangle\\
 &=2\sum\limits_{j\in \mathbb{J}}\langle \widetilde{\Psi}_jh,\widetilde{A}_jh\rangle+ \sum\limits_{j\in \mathbb{J}}\langle y_j, z_j\rangle-\sum\limits_{j\in \mathbb{J}}\langle y_j,A_jS_{A,\Psi}^{-1}h\rangle-\sum\limits_{j\in \mathbb{J}}\langle \Psi_jS_{A,\Psi}^{-1}h, z_j\rangle\\
 &= 2\left\langle\sum\limits_{j\in \mathbb{J}}S_{A,\Psi}^{-1}A_j^*\Psi_jS_{A,\Psi}^{-1}h, h \right\rangle+ \sum\limits_{j\in \mathbb{J}}\langle y_j, z_j\rangle-\left\langle \sum\limits_{j\in \mathbb{J}}A_j^*y_j,S_{A,\Psi}^{-1}h\right \rangle -\left\langle S_{A,\Psi}^{-1}h , \sum\limits_{j\in \mathbb{J}}\Psi_j^*z_j\right \rangle\\
 &=2 \langle S_{A,\Psi}^{-1}h,h \rangle + \sum\limits_{j\in \mathbb{J}}\langle y_j, z_j\rangle -\langle h, S_{A,\Psi}^{-1}h\rangle-\langle S_{A,\Psi}^{-1}h, h\rangle=\text{Left side.}
 \end{align*}
 \end{proof} 
 \begin{theorem}\label{CANONICALDUALFRAMEPROPERTYOPERATORVERSION}
  Let $(\{A_j\}_{j\in \mathbb{J}},\{\Psi_j\}_{j\in \mathbb{J}})$ be an (ovf) with frame bounds $ a$ and $ b.$ Then
  \begin{enumerate}[\upshape(i)]
  \item The canonical dual (ovf) of the canonical dual (ovf)  of $ (\{A_j\}_{j\in \mathbb{J}} ,\{\Psi_j\}_{j\in \mathbb{J}} )$ is itself.
  \item$ \frac{1}{b}, \frac{1}{a}$ are frame bounds for the canonical dual of $ (\{A_j\}_{j\in \mathbb{J}},\{\Psi_j\}_{j\in \mathbb{J}}).$
  \item If $ a, b $ are optimal frame bounds for $( \{A_j\}_{j\in \mathbb{J}} , \{\Psi_j\}_{j\in \mathbb{J}}),$ then $ \frac{1}{b}, \frac{1}{a}$ are optimal  frame bounds for its canonical dual.
  \end{enumerate} 
 \end{theorem} 
 \begin{proof}
 \begin{enumerate}[\upshape(i)]
 \item Frame operator for the canonical dual  $(\{A_jS_{A,\Psi}^{-1}\}_{j \in \mathbb{J}}, \{\Psi_jS_{A,\Psi}^{-1}\}_{j\in \mathbb{J}} )$ is 
 $$ \sum\limits_{j\in \mathbb{J}}(\Psi_jS_{A,\Psi}^{-1})^* (A_jS_{A,\Psi}^{-1}) =S_{A,\Psi}^{-1}\left(\sum\limits_{j\in \mathbb{J}}\Psi_j ^*A_j\right)S_{A,\Psi}^{-1} =S_{A,\Psi}^{-1}S_{A,\Psi}S_{A,\Psi}^{-1}= S_{A,\Psi}^{-1}.$$
  Therefore, its canonical dual is $(\{(A_jS_{A,\Psi}^{-1})S_{A,\Psi}\}_{j \in \mathbb{J}} , \{(\Psi_jS_{A,\Psi}^{-1})S_{A,\Psi}\}_{j\in \mathbb{J}})$ which is original frame.
 \item  Multiplying the inequality $a\leq S_{A,\Psi}\leq b $  by $ S_{A,\Psi}^{-1}$ gives (ii).
 \item Let $ c$ be the optimal upper frame bound for  the canonical dual. Since $a$ is already a lower frame bound for $(\{A_j\}_{j\in \mathbb{J}}, \{\Psi_j\}_{j\in \mathbb{J}})$, from (ii), $c\leq \frac{1}{a}. $ Again from (ii), $ \frac{1}{c}$ is a lower frame bound for $(\{(A_jS_{A,\Psi}^{-1})S_{A,\Psi}=A_j\}_{j\in \mathbb{J}}, \{(\Psi_jS_{A,\Psi}^{-1})S_{A,\Psi}=\Psi_j\}_{j\in \mathbb{J}})$. But then $ \frac{1}{c}\leq a. $ Therefore $ c=\frac{1}{a}.$ Similarly one can prove that $\frac{1}{b}$ is the optimal lower frame bound for the canonical dual.  
 \end{enumerate}
 \end{proof}
 \begin{proposition}\label{DUALOVFCHARACTERIZATION}
 Let  $ (\{A_j\}_{j\in \mathbb{J}}, \{\Psi_j\}_{j\in \mathbb{J}}) $ and $ (\{B_j\}_{j\in \mathbb{J}}, \{\Phi_j\}_{j\in \mathbb{J}}) $ be operator-valued frames in  $\mathcal{B}(\mathcal{H}, \mathcal{H}_0)$. Then the following are equivalent.
 \begin{enumerate}[\upshape(i)]
 \item $ (\{B_j\}_{j\in \mathbb{J}},\{\Phi_j\}_{j\in \mathbb{J}}) $ is dual of $( \{A_j\}_{j\in \mathbb{J}}, \{\Psi_j\}_{j\in \mathbb{J}}) $. 
 \item $ \sum_{j\in \mathbb{J}}\Phi_j^*A_j = \sum_{j\in \mathbb{J}}B_j^*\Psi_j=I_\mathcal{H}$ where the convergence is in the SOT.
 \end{enumerate}
 \end{proposition}
 \begin{proof}
 $\theta_\Phi^*\theta_A= \sum_{j\in \mathbb{J}}\Phi_j^*A_j, \theta_B^*\theta_\Psi= \sum
 _{j\in \mathbb{J}}B_j^*\Psi_j.$
 \end{proof}
 \begin{theorem}
 Let $ (\{A_j\}_{j\in \mathbb{J}}, \{\Psi_j\}_{j\in \mathbb{J}})$  be an  (ovf) in   $\mathcal{B}(\mathcal{H}, \mathcal{H}_0)$. If $ (\{A_j\}_{j\in \mathbb{J}}, \{\Psi_j\}_{j\in \mathbb{J}})$ is a Riesz  basis  in  $\mathcal{B}(\mathcal{H}, \mathcal{H}_0)$, then $ (\{A_j\}_{j\in \mathbb{J}}, \{\Psi_j\}_{j\in \mathbb{J}}) $ has unique dual. Converse holds if $ \theta_A(\mathcal{H})=\theta_\Psi(\mathcal{H})$.
 \end{theorem}
 \begin{proof}
$(\Rightarrow)$ Let  $ (\{B_j\}_{j\in \mathbb{J}}, \{\Phi_j\}_{j\in \mathbb{J}})$ and  $ (\{C_j\}_{j\in \mathbb{J}}, \{\Xi_j\}_{j\in \mathbb{J}})$ be  operator-valued frames such that  both are duals  of  $ (\{A_j\}_{j\in \mathbb{J}}, \{\Psi_j\}_{j\in \mathbb{J}})$. Then $ \theta_\Psi^*\theta_B=I_\mathcal{H}=\theta_A^*\theta_\Phi=\theta_\Psi^*\theta_C=\theta_A^*\theta_\Xi$ $\Rightarrow $ $\theta_\Psi^*(\theta_B-\theta_C)=0=\theta_A^*(\theta_\Phi-\theta_\Xi)$ $\Rightarrow $ $ \theta_A\theta_\Psi^*(\theta_B-\theta_C)=I_\mathcal{H}(\theta_B-\theta_C)=0=\theta_\Psi\theta_A^*(\theta_\Phi-\theta_\Xi)=I_\mathcal{H}(\theta_\Phi-\theta_\Xi)$. An action of $L_j^*$ from the left gives $B_j=C_j,\Phi_j=\Xi_j, \forall j \in \mathbb{J}$.

$(\Leftarrow)$ 	We are also given $ \theta_A(\mathcal{H})=\theta_\Psi(\mathcal{H})$. Suppose $ (\{A_j\}_{j\in \mathbb{J}}, \{\Psi_j\}_{j\in \mathbb{J}})$ is not a  Riesz  basis  in  $\mathcal{B}(\mathcal{H}, \mathcal{H}_0)$. Then, from Proposition \ref{RIESZOVFCHARACTERIZATIONPROPOSITION}, $ \theta_A(\mathcal{H}) \subsetneq \ell^2(\mathbb{J})\otimes \mathcal{H}_0$. Let $ P:\ell^2(\mathbb{J})\otimes \mathcal{H}_0 \rightarrow \theta_A(\mathcal{H})^\perp=\theta_\Psi(\mathcal{H})^\perp$ be the orthogonal projection, $T:\theta_A(\mathcal{H})^\perp\rightarrow \mathcal{H} $ be a nonzero bounded linear operator. Define $B_j\coloneqq A_jS_{A,\Psi}^{-1}+L_j^*PT^*$, $\Phi_j\coloneqq \Psi_jS_{A,\Psi}^{-1}+L_j^*PT^*, \forall j \in \mathbb{J}$. Since $T$ is nonzero and $P$ is onto, there exists a $k \in \mathbb{J}$ such that $B_k\neq A_kS_{A,\Psi}^{-1}$. Now $ \theta_B=\theta_AS_{A,\Psi}^{-1}+PT^*$, $ \theta_\Phi=\theta_\Psi S_{A,\Psi}^{-1}+PT^*$, $S_{B,\Phi}=\theta_\Phi^*\theta_B=(S_{A,\Psi}^{-1}\theta_\Psi^*+TP)(\theta_AS_{A,\Psi}^{-1}+PT^*) =S_{A,\Psi}^{-1}+S_{A,\Psi}^{-1}0T^*+T0S_{A,\Psi}^{-1}+TPT^*=S_{A,\Psi}^{-1}+TPT^*$, which is positive invertible. Thus $(\{B_j\}_{j\in \mathbb{J}}, \{\Phi_j\}_{j\in \mathbb{J}})$ is an (ovf) which is different from the canonical dual of $ (\{A_j\}_{j\in \mathbb{J}}, \{\Psi_j\}_{j\in \mathbb{J}})$. We show this (ovf) is also a dual of $ (\{A_j\}_{j\in \mathbb{J}}, \{\Psi_j\}_{j\in \mathbb{J}})$ (and thus getting a contradiction): $\theta_A^*\theta_\Phi=\theta_A^*(\theta_\Psi S_{A,\Psi}^{-1}+PT^*)=I_\mathcal{H}+0$, $\theta_\Psi^*\theta_B=\theta_\Psi^*(\theta_A S_{A,\Psi}^{-1}+PT^*)=I_\mathcal{H}+0$.

\end{proof}
 \begin{proposition}
Let $ (\{A_j\}_{j\in \mathbb{J}}, \{\Psi_j\}_{j\in \mathbb{J}}) $  be an (ovf) in    $\mathcal{B}(\mathcal{H}, \mathcal{H}_0)$. If  $ (\{B_j\}_{j\in \mathbb{J}}, \{\Phi_j\}_{j\in \mathbb{J}}) $ is dual of $ (\{A_j\}_{j\in \mathbb{J}}, \{\Psi_j\}_{j\in \mathbb{J}}) $, then there exist Bessel sequences $ \{C_j\}_{j\in \mathbb{J}}$ and $ \{\Xi_j\}_{j\in \mathbb{J}} $ (w.r.t. themselves) in  $\mathcal{B}(\mathcal{H}, \mathcal{H}_0)$ such that $ B_j=A_jS_{A,\Psi}^{-1}+C_j, \Phi_j=\Psi_jS_{A,\Psi}^{-1}+\Xi_j,\forall j \in \mathbb{J}$,  and $\theta_C(\mathcal{H})\perp \theta_\Psi(\mathcal{H}),\theta_\Xi(\mathcal{H})\perp \theta_A(\mathcal{H})$. Converse holds if $\theta_\Xi^*\theta_C \geq 0$.
 \end{proposition}
 \begin{proof}
 $(\Rightarrow)$ Define $ C_j\coloneqq B_j-A_jS_{A,\Psi}^{-1}, \Xi_j\coloneqq\Phi_j-\Psi_jS_{A,\Psi}^{-1},\forall j \in \mathbb{J}$. Since $ \{A_j\}_{j\in \mathbb{J}}$, $ \{\Psi_j\}_{j\in \mathbb{J}} $, $ \{B_j\}_{j\in \mathbb{J}}$, $ \{\Phi_j\}_{j\in \mathbb{J}} $ are Bessel (w.r.t. themselves),
 $ \{C_j\}_{j\in \mathbb{J}}$ and $ \{\Xi_j\}_{j\in \mathbb{J}} $ are Bessel (w.r.t. themselves). Further $\theta_C=\theta_B-\theta_AS_{A,\Psi}^{-1}$, $\theta_\Xi=\theta_\Phi-\theta_\Psi S_{A,\Psi}^{-1}$.  Consider $ \theta_\Psi^*\theta_C=\theta_\Psi^*(\theta_B-\theta_AS_{A,\Psi}^{-1})=I_\mathcal{H}-I_\mathcal{H}=0$, $ \theta_A^*\theta_\Xi=\theta_A^*(\theta_\Phi-\theta_\Psi S_{A,\Psi}^{-1})=I_\mathcal{H}-I_\mathcal{H}=0$.
 
 $(\Leftarrow)$	Clearly $\theta_B$ and $\theta_\Phi$ exist,  bounded and $ \theta_B=\theta_AS_{A,\Psi}^{-1}+\theta_C$, $\theta_\Phi=\theta_\Psi S_{A,\Psi}^{-1}+\theta_\Xi$. Using  
 $\theta_\Xi^*\theta_C \geq 0$, $S_{B,\Phi}=\theta_\Phi^*\theta_B=(S_{A,\Psi}^{-1}\theta_\Psi^*+\theta_\Xi^*)(\theta_AS_{A,\Psi}^{-1}+\theta_C)=S_{A,\Psi}^{-1}+S_{A,\Psi}^{-1}0+0S_{A,\Psi}^{-1}+\theta_\Xi^*\theta_C =S_{A,\Psi}^{-1}+\theta_\Xi^*\theta_C\geq S_{A,\Psi}^{-1} $, hence $S_{B,\Phi} $ is positive invertible. To show the dual, $\theta_A^*\theta_\Phi=\theta_A^*(\theta_\Psi S_{A,\Psi}^{-1}+\theta_\Xi)=I_\mathcal{H}+0$, $\theta_\Psi^*\theta_B=\theta_\Psi^*(\theta_AS_{A,\Psi}^{-1}+\theta_C)=I_\mathcal{H}+0$.
 \end{proof}
 \begin{definition}
 An (ovf)  $(\{B_j\}_{j\in \mathbb{J}},  \{\Phi_j\}_{j\in \mathbb{J}})$  in $\mathcal{B}(\mathcal{H}, \mathcal{H}_0)$ is said to be orthogonal to an (ovf)  $( \{A_j\}_{j\in \mathbb{J}}, \{\Psi_j\}_{j\in \mathbb{J}})$ in $\mathcal{B}(\mathcal{H}, \mathcal{H}_0)$ if $ \theta_\Phi^*\theta_A= \theta_B^*\theta_\Psi=0.$
 \end{definition}
 We note that the definition is symmetric. Also, dual frames cannot be orthogonal to each other and orthogonal frames can not be dual to each other. Further, unlike dual frames, if $ (\{B_j\}_{j\in \mathbb{J}}, \{\Phi_j\}_{j\in \mathbb{J}})$ is orthogonal to $ (\{A_j\}_{j\in \mathbb{J}}, \{\Psi_j\}_{j\in \mathbb{J}})$, then  both $ (\{A_j\}_{j\in \mathbb{J}}, \{\Phi_j\}_{j\in \mathbb{J}})$ and $ (\{B_j\}_{j\in \mathbb{J}}, \{\Psi_j\}_{j\in \mathbb{J}})$ are not frames.
 \begin{proposition}\label{ORTHOGONALOVFCHARACTERIZATION}
 Let  $ (\{A_j\}_{j\in \mathbb{J}}, \{\Psi_j\}_{j\in \mathbb{J}}) $ and $ (\{B_j\}_{j\in \mathbb{J}}, \{\Phi_j\}_{j\in \mathbb{J}}) $ be operator-valued frames in  $\mathcal{B}(\mathcal{H}, \mathcal{H}_0)$. Then the following are equivalent.
 \begin{enumerate}[\upshape(i)]
 \item $ (\{B_j\}_{j\in \mathbb{J}},\{\Phi_j\}_{j\in \mathbb{J}}) $ is orthogonal to  $( \{A_j\}_{j\in \mathbb{J}},  \{\Psi_j\}_{j\in \mathbb{J}}) $. 
 \item $ \sum_{j\in \mathbb{J}}\Phi_j^*A_j = \sum_{j\in \mathbb{J}}B_j^*\Psi_j=0$ where the convergence is in the SOT.
 \end{enumerate}
\end{proposition}
\begin{remark}
Proposition \ref{DUALOVFCHARACTERIZATION} and Proposition \ref{ORTHOGONALOVFCHARACTERIZATION} are very simple, but they tell how we can define duality and orthogonality without using analysis operators (we do this in Section \ref{FURTHEREXTENSION}).	
\end{remark}
\begin{proposition}
Two orthogonal operator-valued frames  have common dual (ovf).	
\end{proposition} 
\begin{proof}
Let  $ (\{A_j\}_{j\in \mathbb{J}}, \{\Psi_j\}_{j\in \mathbb{J}}) $ and $ (\{B_j\}_{j\in \mathbb{J}}, \{\Phi_j\}_{j\in \mathbb{J}}) $ be  two orthogonal operator-valued frames in $\mathcal{B}(\mathcal{H}, \mathcal{H}_0)$. Define $ C_j\coloneqq A_jS_{A,\Psi}^{-1}+B_jS_{B,\Phi}^{-1},\Xi_j\coloneqq \Psi _jS_{A,\Psi}^{-1}+\Phi_jS_{B,\Phi}^{-1}, \forall j \in \mathbb{J}$. Then $ \theta_C=\theta_AS_{A,\Psi}^{-1}+\theta_BS_{B,\Phi}^{-1}, \theta_\Xi=\theta_\Psi S_{A,\Psi}^{-1}+\theta_\Phi S_{B,\Phi}^{-1}$,  $ S_{C,\Xi}=\theta_\Xi^*\theta_C=(S_{A,\Psi}^{-1}\theta_\Psi^* +S_{B,\Phi}^{-1}\theta_\Phi^* )(\theta_AS_{A,\Psi}^{-1}+\theta_BS_{B,\Phi}^{-1})=S_{A,\Psi}^{-1}+S_{B,\Phi}^{-1} $ which is positive and $ \langle S_{C,\Xi}h, h\rangle =\langle S_{A,\Psi}^{-1}h,h \rangle +\langle S_{B,\Phi}^{-1}h, h\rangle \geq \min\left\{ \|S_{A,\Psi}\|^{-1},\|S_{B,\Phi}\|^{-1}\right\}\|h\|^2, \forall h \in \mathcal{H}$, hence $S_{C,\Xi}$ is invertible. Therefore $(\{C_j\}_{j\in \mathbb{J}}, \{\Xi_j\}_{j\in \mathbb{J}})$ is an (ovf) in $\mathcal{B}(\mathcal{H}, \mathcal{H}_0)$. This is a common dual of  $ (\{A_j\}_{j\in \mathbb{J}}, \{\Psi_j\}_{j\in \mathbb{J}}) $ and $ (\{B_j\}_{j\in \mathbb{J}}, \{\Phi_j\}_{j\in \mathbb{J}}).$ In fact, $\theta_C^*\theta_\Psi=(S_{A,\Psi}^{-1}\theta_A^*+S_{B,\Phi}^{-1}\theta_B^*)\theta_\Psi=I_\mathcal{H}+0, \theta_\Xi^*\theta_A=(S_{A,\Psi}^{-1}\theta_\Psi^* +S_{B,\Phi}^{-1}\theta_\Phi^*)\theta_A=I_\mathcal{H}+0$, and $\theta_C^*\theta_\Phi= (S_{A,\Psi}^{-1}\theta_A^*+S_{B,\Phi}^{-1}\theta_B^*)\theta_\Phi=0+I_\mathcal{H}, \theta_\Xi^*\theta_B=(S_{A,\Psi}^{-1}\theta_\Psi^* +S_{B,\Phi}^{-1}\theta_\Phi^*)\theta_B=0+I_\mathcal{H} $.	
\end{proof}
 Next we have an interpolation result.
\begin{proposition}
Let $ (\{A_j\}_{j\in \mathbb{J}}, \{\Psi_j\}_{j\in \mathbb{J}}) $ and $ (\{B_j\}_{j\in \mathbb{J}}, \{\Phi_j\}_{j\in \mathbb{J}}) $ be  two Parseval operator-valued frames in   $\mathcal{B}(\mathcal{H}, \mathcal{H}_0)$ which are  orthogonal. If $C,D,E,F \in \mathcal{B}(\mathcal{H})$ are such that $ C^*E+D^*F=I_\mathcal{H}$, then  $ (\{A_jC+B_jD\}_{j\in \mathbb{J}}, \{\Psi_jE+\Phi_jF\}_{j\in \mathbb{J}}) $ is a  Parseval (ovf) in  $\mathcal{B}(\mathcal{H}, \mathcal{H}_0)$. In particular,  if scalars $ c,d,e,f$ satisfy $\bar{c}e+\bar{d}f =1$, then $ (\{cA_j+dB_j\}_{j\in \mathbb{J}}, \{e\Psi_j+f\Phi_j\}_{j\in \mathbb{J}}) $ is  a  Parseval (ovf).
\end{proposition}   
\begin{proof}
We see $ \theta_{AC+BD} =\sum_{j\in \mathbb{J}}L_j(A_jC+B_jD)=\theta_AC+\theta_BD, \theta_{\Psi E+\Phi F}=\sum_{j\in \mathbb{J}}L_j(\Psi_jE+\Phi_jF)=\theta_\Psi E+\theta_\Phi F$ and hence $S_{AC+BD,\Psi E+\Phi F}= \theta_{\Psi E+\Phi F}^* \theta_{AC+BD}=(\theta_\Psi E+\theta_\Phi F)^*(\theta_AC+\theta_BD)=E^*\theta_\Psi^*\theta_AC+E^*\theta_\Psi^*\theta_BD+F^*\theta_\Phi^*\theta_AC+F^*\theta_\Phi^*\theta_BD=E^*S_{A,\Psi}C+0+0+F^*S_{B,\Phi }D=E^*I_\mathcal{H}C+F^*I_\mathcal{H}D=I_\mathcal{H}.$
\end{proof}
\begin{definition}
Two operator-valued frames $(\{A_j\}_{j\in \mathbb{J}},\{\Psi_j\}_{j\in \mathbb{J}} )$  and $ (\{B_j\}_{j\in \mathbb{J}}, \{\Phi_j\}_{j\in \mathbb{J}} )$   in $ \mathcal{B}(\mathcal{H}, \mathcal{H}_0)$  are called 
disjoint if $(\{A_j\oplus B_j\}_{j \in \mathbb{J}},\{\Psi_j\oplus \Phi_j\}_{j \in \mathbb{J}})$ is (ovf) in $ \mathcal{B}(\mathcal{H}\oplus \mathcal{H}, \mathcal{H}_0).$   
\end{definition}
The important thing to remember in disjoint frame definition is that only the domain of collection of operators will change not the  codomain.
\begin{proposition}
If $(\{A_j\}_{j\in \mathbb{J}},\{\Psi_j\}_{j\in \mathbb{J}} )$  and $ (\{B_j\}_{j\in \mathbb{J}}, \{\Phi_j\}_{j\in \mathbb{J}} )$  are  orthogonal operator-valued frames  in $ \mathcal{B}(\mathcal{H}, \mathcal{H}_0)$, then  they  are disjoint. Further, if both $(\{A_j\}_{j\in \mathbb{J}},\{\Psi_j\}_{j\in \mathbb{J}} )$  and $ (\{B_j\}_{j\in \mathbb{J}}, \{\Phi_j\}_{j\in \mathbb{J}} )$ are  Parseval, then $(\{A_j\oplus B_j\}_{j \in \mathbb{J}},\{\Psi_j\oplus \Phi_j\}_{j \in \mathbb{J}})$ is Parseval.
\end{proposition}
\begin{proof}
$\theta_{A\oplus B}(h\oplus g)=\sum_{j\in \mathbb{J}}L_j(A_j\oplus B_j)(h\oplus g)=\sum_{j\in \mathbb{J}}L_j(A_jh+ B_jg)=\theta_Ah+\theta_Bg,\forall h\oplus g \in \mathcal{H}\oplus\mathcal{H}$, $\theta_{\Psi\oplus \Phi}(h\oplus g)=\theta_\Psi h+\theta_\Phi g,\forall h\oplus g \in \mathcal{H}\oplus\mathcal{H}$, $\langle \theta_{\Psi\oplus \Phi}^*y, h\oplus g\rangle =\langle y,\theta_{\Psi\oplus \Phi}(h\oplus g)\rangle=\langle \theta_\Psi^*y, h \rangle+\langle\theta_\Phi^*y,  g \rangle=\langle \theta_\Psi^*y\oplus\theta_\Phi^*y, h\oplus g\rangle ,\forall y \in \ell^2(\mathbb{J})\otimes \mathcal{H}_0, \forall h\oplus g \in \mathcal{H}\oplus\mathcal{H} $. Thus  $S_{A\oplus B,\Psi\oplus\Phi}(h\oplus g)=\theta^*_{\Psi\oplus\Phi}\theta_{A\oplus B}(h\oplus g)=\theta^*_{\Psi\oplus\Phi}(\theta_Ah+\theta_Bg)=\theta_\Psi^*(\theta_Ah+\theta_Bg)\oplus\theta_\Phi^*(\theta_Ah+\theta_Bg)=(S_{A ,\Psi}+0)\oplus (0+S_{ B,\Phi})=S_{A ,\Psi}\oplus S_{ B,\Phi}$, which is bounded positive invertible with $S_{A\oplus B,\Psi\oplus\Phi}^{-1}=S_{A ,\Psi}^{-1}\oplus S_{ B,\Phi}^{-1}$. That last conclusion comes from the expression for $S_{A\oplus B,\Psi\oplus\Phi}$.
\end{proof}

\section{Characterizations of the extension}\label{CHARACTERIZATIONSOF THE EXTENSION}
\begin{theorem}\label{OPERATORVALUEDCHARACTERIZATIONSOFTHEEXTENSION}
Let $ \{F_j\}_{j \in \mathbb{J}}$ be an arbitrary orthonormal basis in $ \mathcal{B}(\mathcal{H},\mathcal{H}_0).$ Then 
\begin{enumerate}[\upshape(i)]
\item The orthonormal  bases   $ (\{A_j\}_{j \in \mathbb{J}},\{\Psi_j\}_{j \in \mathbb{J}})$ in  $ \mathcal{B}(\mathcal{H},\mathcal{H}_0)$ are precisely $( \{F_jU\}_{j \in \mathbb{J}},\{c_jF_jU\}_{j \in \mathbb{J}}) $,  where $ U \in \mathcal{B}(\mathcal{H}) $ is unitary and $ c_j \in \mathbb{R}, \forall j \in \mathbb{J} $ such that $ 0<\inf\{c_j\}_{j \in \mathbb{J}}\leq \sup\{c_j\}_{j \in \mathbb{J}}< \infty.$
\item The Riesz bases   $ (\{A_j\}_{j \in \mathbb{J}},\{\Psi_j\}_{j \in \mathbb{J}})$ in  $ \mathcal{B}(\mathcal{H},\mathcal{H}_0)$  are precisely $( \{F_jU\}_{j \in \mathbb{J}},\{F_jV\}_{j \in \mathbb{J}}) $, where $ U,V \in \mathcal{B}(\mathcal{H}) $ are invertible  such that  $ V^*U$ is positive.
\item The operator-valued frames $ (\{A_j\}_{j \in \mathbb{J}},\{\Psi_j\}_{j \in \mathbb{J}})$ in  $ \mathcal{B}(\mathcal{H},\mathcal{H}_0)$  are precisely $( \{F_jU\}_{j \in \mathbb{J}},\{F_jV\}_{j \in \mathbb{J}}) $, where $ U,V \in \mathcal{B}(\mathcal{H}) $ are such that  $ V^*U$ is positive invertible.
\item The Bessel sequences  $(\{A_j\}_{j \in \mathbb{J}},\{\Psi_j\}_{j \in \mathbb{J}})$ in  $ \mathcal{B}(\mathcal{H},\mathcal{H}_0)$  are precisely $( \{F_jU\}_{j \in \mathbb{J}},\{F_jV\}_{j \in \mathbb{J}}) $, where $ U,V \in \mathcal{B}(\mathcal{H}) $ are  such that  $ V^*U$ is positive. 
\item The Riesz operator-valued frames $ (\{A_j\}_{j \in \mathbb{J}},\{\Psi_j\}_{j \in \mathbb{J}})$ in $ \mathcal{B}(\mathcal{H},\mathcal{H}_0)$  are precisely $( \{F_jU\}_{j \in \mathbb{J}},\{F_jV\}_{j \in \mathbb{J}}) $, where $ U,V \in \mathcal{B}(\mathcal{H}) $ are  such that  $ V^*U$ is positive invertible and $ U(V^*U)^{-1}V^* =I_{\mathcal{H}}$.  
\item The orthonormal operator-valued frames $ ( \{A_j\}_{j \in \mathbb{J}},\{\Psi_j\}_{j \in \mathbb{J}})$ in $ \mathcal{B}(\mathcal{H},\mathcal{H}_0)$  are precisely $( \{F_jU\}_{j \in \mathbb{J}},\{F_jV\}_{j \in \mathbb{J}}) $, where $ U,V \in \mathcal{B}(\mathcal{H}) $ are such that $ V^*U=I_\mathcal{H}= UV^*$. 
\end{enumerate}
\end{theorem}  
\begin{proof}
\begin{enumerate}[\upshape(i)]
 \item We first show that every collection of the form  $( \{F_jU\}_{j \in \mathbb{J}},\{c_jF_jU\}_{j \in \mathbb{J}}) $,  $ U : \mathcal{H} \rightarrow  \mathcal{H} $  unitary,  $ c_j \in \mathbb{R}, \forall j \in \mathbb{J}$ with $ 0<\inf\{c_j\}_{j \in \mathbb{J}}\leq \sup\{c_j\}_{j \in \mathbb{J}}< \infty$ is an orthonormal basis in  $ \mathcal{B}(\mathcal{H},\mathcal{H}_0).$ For this,  it is sufficient to prove $ \{F_jU\}_{j \in \mathbb{J}}$ is an orthonormal basis. This we achieve using orthonormality of $F_j $'s and unitariness of $ U.$ Indeed,  $ \langle (F_jU)^*y,(F_kU)^*z \rangle=\langle U^*F^*_jy,U^*F^*_kz \rangle = \langle F^*_jy,F^*_kz \rangle =\delta_{j,k}\langle y, z\rangle, \forall j, k \in \mathbb{J}, \forall y,z \in \mathcal{H}_0 $ and $ \sum_{j\in \mathbb{J}}\|F_jUh\|^2=\langle\sum_{j\in \mathbb{J}}F_j^*F_j Uh , Uh\rangle =\|Uh\|^2=\|h\|^2, \forall h\in \mathcal{H}. $  For the other side, let $ (\{A_j\}_{j \in \mathbb{J}},\{\Psi_j\}_{j \in \mathbb{J}})$  be an orthonormal basis in  $ \mathcal{B}(\mathcal{H},\mathcal{H}_0).$ We may assume $\{A_j\}_{j \in \mathbb{J}}$ is an orthonormal basis and there exist $ c_j \in \mathbb{R}, \forall j \in \mathbb{J}$ with $ 0<\inf\{c_j\}_{j \in \mathbb{J}}\leq \sup\{c_j\}_{j \in \mathbb{J}}< \infty$ and $ \Psi_j=c_jA_j, \forall j \in \mathbb{J}.$ Define $ U\coloneqq \sum_{j\in \mathbb{J}}F_j^*A_j.$ This exists in SOT, since for every finite subset $\mathbb{S}$ of  $\mathbb{J}$ and $ h \in \mathcal{H},$ 
 \begin{align*}
 \left\|\sum_{j\in \mathbb{S}}F_j^*A_jh\right\|^2=\left\langle\sum_{j\in \mathbb{S}}F_j^*A_jh, \sum_{k\in \mathbb{S}}F_k^*A_kh\right\rangle= \sum_{j\in \mathbb{S}}\left\langle A_jh, F_j\left(\sum_{k\in \mathbb{S}}F_k^*A_kh\right) \right\rangle=\sum_{j\in \mathbb{S}}\|A_jh\|^2.
 \end{align*}
 Now, $ F_jU=F_j(\sum_{k\in \mathbb{J}}F_k^*A_k)=A_j, c_jF_jU=c_jA_j=\Psi_j,  \forall j \in \mathbb{J}.$ We are done if we show $ U$ is unitary, and this is so: $ UU^*=(\sum_{j\in \mathbb{J}}F_j^*A_j)(\sum_{k\in \mathbb{J}}A^*_kF_k)=\sum_{j\in \mathbb{J}}F_j^*(\sum_{k\in \mathbb{J}}A_jA^*_kF_k)=\sum_{j\in \mathbb{J}}F_j^*F_j=I_\mathcal{H}, U^*U=(\sum_{j\in \mathbb{J}}A_j^*F_j)(\sum_{k\in \mathbb{J}}F^*_kA_k)=\sum_{j\in \mathbb{J}}A_j^*(\sum_{k\in \mathbb{J}}F_jF^*_kA_k)=\sum_{j\in \mathbb{J}}A_j^*A_j=I_\mathcal{H}.$
 \item From the very definition of Riesz basis,  we conclude whenever $U,V : \mathcal{H} \rightarrow  \mathcal{H} $ are invertible  with  $ V^*U$ is positive, then  $( \{F_jU\}_{j \in \mathbb{J}},\{F_jV\}_{j \in \mathbb{J}}) $ is a Riesz basis. We turn for other way. Let $ (\{A_j\}_{j \in \mathbb{J}},\{\Psi_j\}_{j \in \mathbb{J}})$ be Riesz basis. Then there exists an orthonormal basis $\{G_j\}_{j \in \mathbb{J}}$   in  $ \mathcal{B}(\mathcal{H},\mathcal{H}_0)$ and invertible $ R, S:\mathcal{H}\rightarrow \mathcal{H} $ with $S^*R$ is positive and $ A_j=G_jR, \Psi_j=G_jS,  \forall j \in \mathbb{J} $. Define $ U\coloneqq\sum_{j\in \mathbb{J}}F_j^*G_jR, V\coloneqq \sum_{j\in \mathbb{J}}F_j^*G_jS.$ Since $ \{F_j\}_{j \in \mathbb{J}}$ and  $\{G_j\}_{j \in \mathbb{J}}$ are orthonormal bases, as in the proof of (i), $ U,V$ are well-defined. Now $ F_jU=G_jR=A_j, F_jV=G_jS=\Psi_j, \forall j \in \mathbb{J}.$ It remains to show $ U$ and $ V$ are invertible and $ V^*U $ is positive. For this we consider $ U (R^{-1}(\sum_{k\in \mathbb{J}}G_k^*F_k))= (\sum_{j\in \mathbb{J}}F_j^*G_jR) (R^{-1}(\sum_{k\in \mathbb{J}}G_k^*F_k))=\sum_{j\in \mathbb{J}}F_j^* (\sum_{k\in \mathbb{J}}G_jG_k^*F_k)=\sum_{j\in \mathbb{J}}F_j^*F_j=I_{\mathcal{H}},$ $ (R^{-1}(\sum_{k\in \mathbb{J}}G_k^*F_k))U =R^{-1}(\sum_{j\in \mathbb{J}}G_j^*F_j)(\sum_{k\in \mathbb{J}}F_k^*G_kR)=R^{-1}(\sum_{j\in \mathbb{J}}G_j^*(\sum_{k\in \mathbb{J}}F_jF_k^*G_kR))=R^{-1}(\sum_{j\in \mathbb{J}}G_j^*G_j)R =I_{\mathcal{H}},$ and similarly $ V^{-1}=S^{-1}(\sum_{j\in \mathbb{J}}G_j^*F_j), 
 V^*U=S^* (\sum_{j\in \mathbb{J}}G_j^*F_j)(\sum_{k\in \mathbb{J}}F_k^*G_kR)=S^*R \geq 0.$
 \item $(\Leftarrow)$ $ \sum_{j\in\mathbb{J}}L_j(F_jU)= (\sum_{j\in\mathbb{J}}L_jF_j)U, \sum_{j\in\mathbb{J}}L_j(F_jV)= (\sum_{j\in\mathbb{J}}L_jF_j)V.$ These show analysis operators for $ (\{F_jU\}_{j \in \mathbb{J}},\{F_jV\}_{j \in \mathbb{J}})$ are well-defined and the equality $\sum_{j\in\mathbb{J}}(F_jV)^*(F_jU)=V^*U$ shows that it is an (ovf).
 
 $(\Rightarrow)$ Let $ (\{A_j\}_{j \in \mathbb{J}},\{\Psi_j\}_{j \in \mathbb{J}})$ be (ovf). Since  analysis operators for this frame exists,  as in the proof of (i), $  \sum_{j\in \mathbb{J}}F_j^*A_j,  \sum_{j\in \mathbb{J}}F_j^*\Psi_j$ exist (note that,  following  the proof of (i), we get $\|\sum_{j\in \mathbb{S}}F_j^*A_jh\|^2=\sum_{j\in \mathbb{S}}\|A_jh\|^2,$ which  converges to $\| \theta_Ah\|^2=\|\sum_{j\in \mathbb{J}}L_jA_jh\|^2=\sum_{j\in \mathbb{J}}\|A_jh\|^2$) as bounded operators, call them as $ U,V$ respectively. But then $F_jU=A_j, F_jV=\Psi_j ,  \forall j \in \mathbb{J} $ and $ V^*U=(\sum_{j\in\mathbb{J}}\Psi_j^*F_j)(\sum_{k\in\mathbb{J}}F_k^*A_k)=\sum_{j\in\mathbb{J}}\Psi_j^*A_j=S_{A,\Psi}$ which is positive invertible.
 \item Similar to (iii).
 \item We view (iii).  $ (\Leftarrow)$ $ P_{A,\Psi}=\theta_AS_{A,\Psi}^{-1}\theta_\Psi^*=(\sum_{j\in\mathbb{J}}L_jF_jU)(V^*U)^{-1}(\sum_{k\in\mathbb{J}}V^*F^*_kL_k^*)=\theta_FU(V^*U)^{-1}V^*\theta_F^*=\theta_FI_\mathcal{H}\theta_F^* =\sum_{j\in\mathbb{J}}L_jF_j(\sum_{k\in\mathbb{J}}F_k^*L^*_k)=\sum_{j\in\mathbb{J}}L_jL_j^*=I_{\ell^2(\mathbb{J})}\otimes I_{\mathcal{H}_0}$. 
 
 $ (\Rightarrow)$
 \begin{align*}
 U(V^*U)^{-1}V^*&=\left(\sum_{k\in\mathbb{J}}F_k^*A_k\right)S_{A,\Psi}^{-1}\left(\sum_{j\in\mathbb{J}}\Psi_j^*F_j\right)\\
 &=\left(\sum_{l\in\mathbb{J}}F_l^*L_l^*\right)\left(\sum_{k\in\mathbb{J}}L_kA_k\right)S_{A,\Psi}^{-1}\left(\sum_{j\in\mathbb{J}}\Psi_j^*L_j^*\right)\left(\sum_{m\in\mathbb{J}}L_mF_m\right)\\
 &=\left(\sum_{l\in\mathbb{J}}F_l^*L_l^*\right)\theta_AS_{A,\Psi}^{-1}\theta_\Psi^*\left(\sum_{m\in\mathbb{J}}L_mF_m\right)=\left(\sum_{l\in\mathbb{J}}F_l^*L_l^*\right)P_{A,\Psi}\left(\sum_{m\in\mathbb{J}}L_mF_m\right) \\
  &=\left(\sum_{l\in\mathbb{J}}F_l^*L_l^*\right)(I_{\ell^2(\mathbb{J})}\otimes I_{\mathcal{H}_0})\left(\sum_{m\in\mathbb{J}}L_mF_m\right) =\left(\sum_{l\in\mathbb{J}}F_l^*L_l^*\right)\left(\sum_{m\in\mathbb{J}}L_mF_m\right) \\
  &=\sum_{l\in\mathbb{J}}F_l^*F_l=I_{\mathcal{H}}.
 \end{align*}
 \item From (v). $ (\Leftarrow)$ $S_{A,\Psi}=V^*U=I_\mathcal{H}, P_{A,\Psi}=\theta_AS_{A,\Psi}^{-1}\theta_\Psi^*=\theta_A\theta_\Psi^*= \theta_FUV^*\theta_F^*=\theta_FI_\mathcal{H}\theta_F^*=I_{\ell^2(\mathbb{J})}\otimes I_{\mathcal{H}_0}.$
 
 $(\Rightarrow)$ $V^*U=S_{A,\Psi}=I_\mathcal{H},$ 
 \begin{align*}
 UV^*&= \left(\sum_{j\in \mathbb{J}}F_j^*A_j\right)\left( \sum_{k\in \mathbb{J}}\Psi_k^*F_k\right) =
  \left(\sum_{j\in \mathbb{J}}F_j^*L^*_j \left(\sum_{l\in \mathbb{J}}L_lA_l\right)\right) \left( \sum_{k\in \mathbb{J}}\Psi_k^*L^*_k \left( \sum_{m\in \mathbb{J}}L_mF_m\right)\right) \\ &=\theta_F^*\theta_A\theta_\Psi^*\theta_F=\theta_F^*P_{A,\Psi}\theta_F =\theta_F^*(I_{\ell^2(\mathbb{J})}\otimes I_{\mathcal{H}_0})\theta_F= \sum_{j\in \mathbb{J}}F_j^*F_j=I_\mathcal{H}. 
 \end{align*}
 \end{enumerate}
 \end{proof}
 \begin{remark}
 Once $ (\{A_j\}_{j \in \mathbb{J}},\{\Psi_j\}_{j \in \mathbb{J}})$ is given (may be an orthonormal basis or a Riesz basis or an (ovf) or a Bessel sequence or a Riesz (ovf) or an orthonormal (ovf)), the operators $ U$ as well as $V$ are uniquely determined. In fact, if $W \in \mathcal{B}(\mathcal{H})$ also satisfies $F_jU=F_jW=A_j, \forall j \in \mathbb{J}$, then $U=I_\mathcal{H}U=\sum_{j\in \mathbb{J}}F_j^*(F_jU)=\sum_{j\in \mathbb{J}}F_j^*(F_jW)=W.$
 \end{remark}
\begin{corollary}
\begin{enumerate}[\upshape(i)]
 \item If   $ (\{A_j\}_{j \in \mathbb{J}},\{\Psi_j=c_jA_j\}_{j \in \mathbb{J}})$ is an orthonormal basis  in  $ \mathcal{B}(\mathcal{H},\mathcal{H}_0)$, then 
 $$ \sup\{\|A_j\|\}_{j\in\mathbb{J}}\leq1, ~\sup\{\|\Psi_j\|\}_{j\in\mathbb{J}}\leq \sup\{c_j\}_{j\in\mathbb{J}}, ~ A_j\Psi_j^*=c_jI_{\mathcal{H}_0}, ~ \forall j \in \mathbb{J}.$$
 \item If   $ (\{A_j\}_{j \in \mathbb{J}},\{\Psi_j\}_{j \in \mathbb{J}})$ is  Bessel  in $ \mathcal{B}(\mathcal{H},\mathcal{H}_0)$, then 
  $$ \sup\{\|A_j\|\}_{j\in\mathbb{J}}\leq\|U\|, ~\sup\{\|\Psi_j\|\}_{j\in\mathbb{J}}\leq\|V\| , ~ \sup\{\|A_j\Psi_j^*\|\}_{j\in\mathbb{J}}\leq\|UV^*\|.$$
\end{enumerate}	
\end{corollary}
\begin{corollary}\label{OVCOROLLARYHILBERT}
Let $ \{F_j\}_{j \in \mathbb{J}}$ be an arbitrary orthonormal basis in $ \mathcal{B}(\mathcal{H},\mathcal{H}_0).$ Then 
\begin{enumerate}[\upshape(i)]
\item The orthonormal  bases   $ (\{A_j\}_{j \in \mathbb{J}},\{A_j\}_{j \in \mathbb{J}})$ in  $ \mathcal{B}(\mathcal{H},\mathcal{H}_0)$ are precisely $( \{F_jU\}_{j \in \mathbb{J}},\{F_jU\}_{j \in \mathbb{J}}) $,  where $ U \in \mathcal{B}(\mathcal{H}) $ is unitary.
\item The Riesz bases   $ (\{A_j\}_{j \in \mathbb{J}},\{A_j\}_{j \in \mathbb{J}})$ in  $ \mathcal{B}(\mathcal{H},\mathcal{H}_0)$  are precisely $( \{F_jU\}_{j \in \mathbb{J}},\{F_jU\}_{j \in \mathbb{J}}) $, where $ U \in \mathcal{B}(\mathcal{H}) $ is  invertible.  
\item The operator-valued frames $ (\{A_j\}_{j \in \mathbb{J}},\{A_j\}_{j \in \mathbb{J}})$ in  $ \mathcal{B}(\mathcal{H},\mathcal{H}_0)$  are precisely $( \{F_jU\}_{j \in \mathbb{J}},\{F_jU\}_{j \in \mathbb{J}}) $, where $ U \in \mathcal{B}(\mathcal{H}) $ is  such that  $ U^*U$  is invertible.
\item The Bessel sequences  $(\{A_j\}_{j \in \mathbb{J}},\{A_j\}_{j \in \mathbb{J}})$ in  $ \mathcal{B}(\mathcal{H},\mathcal{H}_0)$  are precisely $( \{F_jU\}_{j \in \mathbb{J}},\{F_jU\}_{j \in \mathbb{J}}) $, where $ U\in \mathcal{B}(\mathcal{H}) $.
\item The Riesz operator-valued frames $ (\{A_j\}_{j \in \mathbb{J}},\{A_j\}_{j \in \mathbb{J}})$ in $ \mathcal{B}(\mathcal{H},\mathcal{H}_0)$  are precisely $( \{F_jU\}_{j \in \mathbb{J}},\{F_jU\}_{j \in \mathbb{J}}) $, where $ U \in \mathcal{B}(\mathcal{H}) $ is  such that  $ U^*U$ is  invertible and $ U(U^*U)^{-1}U^* =I_{\mathcal{H}}$.  
\item The orthonormal operator-valued frames $ (\{A_j\}_{j \in \mathbb{J}},\{A_j\}_{j \in \mathbb{J}})$ in $ \mathcal{B}(\mathcal{H},\mathcal{H}_0)$  are precisely $( \{F_jU\}_{j \in \mathbb{J}},\{F_jU\}_{j \in \mathbb{J}}) $, where $ U \in \mathcal{B}(\mathcal{H}) $ is such that $ U^*U=I_\mathcal{H}= UU^*$. 
\item The operator-valued frames $ (\{A_j\}_{j \in \mathbb{J}},\{A_j\}_{j \in \mathbb{J}})$ in  $ \mathcal{B}(\mathcal{H},\mathcal{H}_0)$  are precisely $( \{F_jU\}_{j \in \mathbb{J}},\{F_jU\}_{j \in \mathbb{J}}) $, where $ U \in \mathcal{B}(\mathcal{H}) $ is  such that  $ U^*$  is surjective.
\item $ (\{A_j\}_{j \in \mathbb{J}},\{A_j\}_{j \in \mathbb{J}})$ is an orthonormal basis in 	 $ \mathcal{B}(\mathcal{H},\mathcal{H}_0)$  if and only if it is an orthonormal (ovf).
\end{enumerate}
\end{corollary}
\begin{proof}
 For (i), $c_j=1, \forall j \in \mathbb{J}$, and for (ii)-(vi), $T^*T$ is positive, $\forall T \in \mathcal{B}(\mathcal{H})$. That (ii) and (vi) give (viii). Now, for $T \in \mathcal{B}(\mathcal{H})$, we have   $T^*T$ is invertible in $\mathcal{B}(\mathcal{H}) $ if and only if $ T^*$ is surjective, hence (vii).

\end{proof}

\begin{caution}
Theorem \ref{OPERATORVALUEDCHARACTERIZATIONSOFTHEEXTENSION} holds whenever indexing set of orthonormal basis $ \{F_j\}_{j \in \mathbb{J}}$  is same as that of $ (\{A_j\}_{j \in \mathbb{J}},\{\Psi_j\}_{j \in \mathbb{J}})$, which may be    orthonormal basis/Riesz basis/operator-valued frame/Bessel/Riesz (ovf)/orthonormal (ovf). Interestingly, Theorem \ref{OPERATORVALUEDCHARACTERIZATIONSOFTHEEXTENSION} fails if this is not so.
\end{caution}
\begin{theorem}\label{OPERATORCHARACTERIZATIONHILBERT2}
Let $\{A_j\}_{j\in\mathbb{J}},\{\Psi_j\}_{j\in\mathbb{J}}$ be in $ \mathcal{B}(\mathcal{H}, \mathcal{H}_0).$ Then  $ (\{A_j\}_{j\in \mathbb{J}}, \{\Psi_j\}_{j\in \mathbb{J}})$  is an (ovf) with bounds  $a $ and  $ b$ (resp. Bessel with bound $ b$)
\begin{enumerate}[\upshape(i)]
\item   if and only if $$U:\ell^2(\mathbb{J})\otimes \mathcal{H}_0 \ni y\mapsto\sum\limits_{j\in\mathbb{J}}A_j^*L_j^*y \in \mathcal{H}, ~\text{and} ~ V:\ell^2(\mathbb{J})\otimes \mathcal{H}_0 \ni z\mapsto\sum\limits_{j\in\mathbb{J}}\Psi_j^*L^*_jz \in \mathcal{H} $$ 
are well-defined, $ U,V \in \mathcal{B}(\ell^2(\mathbb{J})\otimes \mathcal{H}_0,\mathcal{H})$  such that  $ aI_\mathcal{H}\leq VU^*\leq bI_\mathcal{H}$ (resp. $ 0\leq VU^*\leq bI_\mathcal{H}$).
\item    if and only if $$U:\ell^2(\mathbb{J})\otimes \mathcal{H}_0 \ni y\mapsto\sum\limits_{j\in\mathbb{J}}A_j^*L_j^*y \in \mathcal{H}, ~\text{and} ~ S: \mathcal{H} \ni g\mapsto \sum\limits_{j\in \mathbb{J}}L_j\Psi_jg \in \ell^2(\mathbb{J})\otimes \mathcal{H}_0 $$ 
are well-defined, $ U \in \mathcal{B}(\ell^2(\mathbb{J})\otimes \mathcal{H}_0,\mathcal{H})$, $ S \in \mathcal{B}(\mathcal{H}, \ell^2(\mathbb{J})\otimes \mathcal{H}_0)$ such that  $ aI_\mathcal{H}\leq S^*U^*\leq bI_\mathcal{H}$ (resp. $ 0\leq S^*U^*\leq bI_\mathcal{H}$).
\item  if and only if  $$R:   \mathcal{H} \ni h\mapsto \sum\limits_{j\in \mathbb{J}}L_jA_jh \in \ell^2(\mathbb{J})\otimes \mathcal{H}_0, ~\text{and} ~ V: \ell^2(\mathbb{J})\otimes \mathcal{H}_0 \ni z\mapsto\sum\limits_{j\in\mathbb{J}}\Psi_j^*L_j^*z \in \mathcal{H} $$ 
are well-defined, $ R \in \mathcal{B}(\mathcal{H}, \ell^2(\mathbb{J})\otimes \mathcal{H}_0)$, $ V \in \mathcal{B}(\ell^2(\mathbb{J})\otimes \mathcal{H}_0,\mathcal{H})$ such that  $ aI_\mathcal{H}\leq VR\leq bI_\mathcal{H}$ (resp. $ 0\leq VR\leq bI_\mathcal{H}$).
\item  if and only if  $$ R:   \mathcal{H} \ni h\mapsto \sum\limits_{j\in \mathbb{J}}L_jA_jh \in \ell^2(\mathbb{J})\otimes \mathcal{H}_0, ~\text{and} ~  S:   \mathcal{H} \ni g\mapsto \sum\limits_{j\in \mathbb{J}}L_j\Psi_jg \in \ell^2(\mathbb{J})\otimes \mathcal{H}_0 $$
are well-defined, $ R,S \in \mathcal{B}(\mathcal{H}, \ell^2(\mathbb{J})\otimes \mathcal{H}_0)$  such that  $ aI_\mathcal{H}\leq S^*R\leq bI_\mathcal{H}$ (resp. $ 0\leq S^*R\leq bI_\mathcal{H}$). 
\end{enumerate}

\end{theorem} 
\begin{proof}
We argue only for (i), in frame situation.  $(\Rightarrow)$ Now $U=\theta_A^*$, $V=\theta_\Psi^*$ and $VU^*=\theta_\Psi^*\theta_A=S_{A,\Psi}$.

$(\Leftarrow)$ Now $\theta_A=U^*$, $\theta_\Psi=V^*$ and $S_{A,\Psi}=\theta_\Psi^*\theta_A=VU^*$.
\end{proof}

Let $ \{A_j\}_{j\in\mathbb{J}}, \{\Psi_j\}_{j\in\mathbb{J}}$ be in $ \mathcal{B}(\mathcal{H},\mathcal{H}_0).$ For each fixed $j\in \mathbb{J}$, suppose $ \{e_{j,k}\}_{k\in\mathbb{L}_j}$ is an orthonormal basis for $ \mathcal{H}_0.$ Riesz representation theorem,  when applied to continuous linear functionals $ \mathcal{H}\ni h \mapsto\langle A_jh, e_{j,k}\rangle\in \mathbb{K}, \mathcal{H}\ni h \mapsto\langle \Psi_jh, e_{j,k}\rangle\in \mathbb{K} $ gives  unique $u_{j,k},v_{j,k}  \in \mathcal{H}$ such that $ \langle A_jh, e_{j,k}\rangle=\langle h, u_{j,k}\rangle, \forall h \in \mathcal{H}, \forall k \in \mathbb{L}_j$ and $ \langle \Psi_jh, e_{j,k}\rangle=\langle h, v_{j,k}\rangle, \forall h \in \mathcal{H}, \forall k \in \mathbb{L}_j.$ Since $ j \in \mathbb{J}$ was  arbitrary, we get $ \langle A_jh, e_{j,k}\rangle=\langle h, u_{j,k}\rangle, \forall h \in \mathcal{H}, \forall k \in \mathbb{L}_j , \forall j \in \mathbb{J} $ and $ \langle \Psi_jh, e_{j,k}\rangle=\langle h, v_{j,k}\rangle, \forall h \in \mathcal{H},  \forall k \in \mathbb{L}_j, \forall j\in \mathbb{J}.$ Now $A_jh=\sum_{k\in \mathbb{L}_j}\langle  A_jh,e_{j,k}\rangle e_{j,k}=\sum_{k\in \mathbb{L}_j}\langle  h,u_{j,k}\rangle e_{j,k}, \forall h \in \mathcal{H},\forall j\in \mathbb{J} $ and similarly $ \Psi_jh= \sum_{k\in \mathbb{L}_j}\langle  h,v_{j,k}\rangle e_{j,k}, \forall h \in \mathcal{H},\forall j\in \mathbb{J}.$ We now show,   for each $j\in \mathbb{J},$ both $\{u_{j,k}\}_{k\in \mathbb{L}_j} $ and $\{v_{j,k}\}_{k\in \mathbb{L}_j} $ are Bessel sequences, and find the adjoints of $ A_j$'s and $\Psi_j $'s in terms of these:  $ \sum_{{k\in \mathbb{L}_j}}|\langle h , u_{j,k}\rangle |^2=\|A_jh\|^2\leq \|A_j\|^2\|h\|^2, \sum_{{k\in \mathbb{L}_j}}|\langle h , v_{j,k}\rangle |^2=\|\Psi_jh\|^2\leq \|\Psi_j\|^2\|h\|^2, \forall h \in \mathcal{H}, \langle h, A_j^*y\rangle=\langle A_jh, y\rangle =\sum_{k\in \mathbb{L}_j}\langle  h,u_{j,k}\rangle \langle e_{j,k}, y \rangle=\langle h , \sum_{k\in \mathbb{L}_j} \langle y,e_{j,k} \rangle u_{j,k}\rangle,
\langle h, \Psi_j^*y\rangle=\langle \Psi_jh, y\rangle =\sum_{k\in \mathbb{L}_j}\langle  h,v_{j,k}\rangle \langle e_{j,k},y  \rangle=\langle h , \sum_{k\in \mathbb{L}_j} \langle y,e_{j,k} \rangle v_{j,k}\rangle ,\forall h\in \mathcal{H},\forall y\in \mathcal{H}_0.$ Therefore $ A_j^*y=\sum_{k\in \mathbb{L}_j} \langle y,e_{j,k} \rangle u_{j,k}$, $ \Psi_j^*z=\sum_{k\in \mathbb{L}_j} \langle z,e_{j,k} \rangle v_{j,k},  \forall y,z \in \mathcal{H}_0, \forall  j\in \mathbb{J}.$ Evaluation of these at $e_{j,k_0}$ gives $ u_{j,k_0}=A_j^*e_{j,k_0}, v_{j,k_0}=\Psi_j^*e_{j,k_0}, \forall k_0 \in  \mathbb{L}_j, j\in \mathbb{J}.$ This produces the following  important theorem.
\begin{theorem}\label{SEQUENTIAL CHARACTERIZATION}
Let $ \{A_j\}_{j\in\mathbb{J}}, \{\Psi_j\}_{j\in\mathbb{J}}$ be in $ \mathcal{B}(\mathcal{H},\mathcal{H}_0).$ Suppose $ \{e_{j,k}\}_{k\in\mathbb{L}_j}$ is an orthonormal basis for $ \mathcal{H}_0,$ for each $j \in \mathbb{J}.$ Let  $ u_{j,k}=A_j^*e_{j,k}, v_{j,k}=\Psi_j^*e_{j,k}, \forall k \in  \mathbb{L}_j, j\in \mathbb{J}.$ Then $(\{A_j\}_{j\in\mathbb{J}}, \{\Psi_j\}_{j\in\mathbb{J}})$ is 
\begin{enumerate}[\upshape(i)]
\item   an orthonormal set (resp. basis) in $ \mathcal{B}(\mathcal{H},\mathcal{H}_0)$ if and only if  $ \{u_{j,k}: k\in \mathbb{L}_j, j \in \mathbb{J}\} $ or $ \{v_{j,k}: k\in \mathbb{L}_j, j \in \mathbb{J}\} $ is an orthonormal set (resp. basis), say $ \{u_{j,k}: k\in \mathbb{L}_j, j \in \mathbb{J}\} $ is an orthonormal set (resp. basis) for $\mathcal{H}$ and there is a sequence  $ \{c_j\}_{j \in \mathbb{J}} $ of reals such that $ 0<\inf\{c_j\}_{j\in \mathbb{J}}\leq\sup\{c_j\}_{j\in \mathbb{J}}<\infty$ and for each $ j \in \mathbb{J},$ $v_{j,k}=c_ju_{j,k} , \forall k \in  \mathbb{L}_j.$
\item  a Riesz basis in $ \mathcal{B}(\mathcal{H},\mathcal{H}_0)$ if and only if  there are bounded invertible operators $ U,V :  \mathcal{H}\rightarrow \mathcal{H} $  and an orthonormal basis $ \{f_{j,k}: k\in \mathbb{L}_j, j \in \mathbb{J}\} $ for $\mathcal{H}$ such that $ VU^*\geq0$ and $u_{j,k}=Uf_{j,k}, v_{j,k}=Vf_{j,k},  \forall k \in  \mathbb{L}_j, \forall j\in \mathbb{J}. $
\item   an (ovf) in $ \mathcal{B}(\mathcal{H},\mathcal{H}_0)$  with bounds $a $ and $ b$  if and only if there exist $ c,d >0$ such that the map 
$$ T: \mathcal{H} \ni h \mapsto\sum\limits_{j\in \mathbb{J}}\sum\limits_{k\in \mathbb{L}_j}\langle h, u_{j,k}\rangle v_{j,k} \in  \mathcal{H} $$
is well-defined bounded positive invertible operator such that $ a\|h\|^2 \leq \langle Th,h \rangle \leq b\|h\|^2, \forall h \in \mathcal{H} $, and 
$$  \sum\limits_{j\in \mathbb{J}}\sum\limits_{k\in \mathbb{L}_j}|\langle h, u_{j,k}\rangle |^2 \leq c\|h\|^2 ,~ \forall h \in \mathcal{H}; ~  \sum\limits_{j\in \mathbb{J}}\sum\limits_{k\in \mathbb{L}_j} |\langle h, v_{j,k}\rangle |^2\leq d\|h\|^2 ,~ \forall h \in \mathcal{H}.$$
\item  Bessel  in $ \mathcal{B}(\mathcal{H},\mathcal{H}_0)$  with bound  $ b$  if and only if there exist $ c,d >0$ such that the map 
$$ T: \mathcal{H} \ni h \mapsto\sum\limits_{j\in \mathbb{J}}\sum\limits_{k\in \mathbb{L}_j}\langle h, u_{j,k}\rangle v_{j,k} \in  \mathcal{H} $$
is well-defined bounded positive  operator such that $ \langle Th,h \rangle \leq b\|h\|^2, \forall h \in \mathcal{H} $, and 
$$  \sum\limits_{j\in \mathbb{J}}\sum\limits_{k\in \mathbb{L}_j}|\langle h, u_{j,k}\rangle |^2 \leq c\|h\|^2 ,~ \forall h \in \mathcal{H}; ~  \sum\limits_{j\in \mathbb{J}}\sum\limits_{k\in \mathbb{L}_j} |\langle h, v_{j,k}\rangle |^2\leq d\|h\|^2 ,~ \forall h \in \mathcal{H}.$$ 
\item  an (ovf)  in $ \mathcal{B}(\mathcal{H},\mathcal{H}_0)$  with bounds $a $ and $ b$  if and only if there exist $ c,d, r >0$ such that 
 $$\left \|\sum\limits_{j\in \mathbb{J}}\sum\limits_{k\in \mathbb{L}_j}\langle h, u_{j,k}\rangle v_{j,k}\right\|\leq r\|h\|,~\forall h \in \mathcal{H}   ; ~\sum\limits_{j\in \mathbb{J}}\sum\limits_{k\in \mathbb{L}_j}\langle h, u_{j,k}\rangle v_{j,k} =\sum\limits_{j\in \mathbb{J}}\sum\limits_{k\in \mathbb{L}_j}\langle h, v_{j,k}\rangle u_{j,k} ,~\forall h \in \mathcal{H} ;$$
 $$a\|h\|^2\leq \sum\limits_{j\in \mathbb{J}}\sum\limits_{k\in \mathbb{L}_j}\langle h, u_{j,k}\rangle \langle  v_{j,k} , h\rangle  \leq b\|h\|^2 ,~ \forall h \in \mathcal{H}, ~\text{and} $$
 $$  \sum\limits_{j\in \mathbb{J}}\sum\limits_{k\in \mathbb{L}_j}|\langle h, u_{j,k}\rangle |^2 \leq c\|h\|^2 ,~ \forall h \in \mathcal{H}; ~ \sum\limits_{j\in \mathbb{J}}\sum\limits_{k\in \mathbb{L}_j} |\langle h, v_{j,k}\rangle |^2\leq d\|h\|^2 ,~ \forall h \in \mathcal{H}.$$
\item Bessel in $ \mathcal{B}(\mathcal{H},\mathcal{H}_0)$  with bound  $ b$ if and only  if there exist $ c,d, r >0$ such that 
$$\left \|\sum\limits_{j\in \mathbb{J}}\sum\limits_{k\in \mathbb{L}_j}\langle h, u_{j,k}\rangle v_{j,k}\right\|\leq r\|h\|,~\forall h \in \mathcal{H} ;~ \sum\limits_{j\in \mathbb{J}}\sum\limits_{k\in \mathbb{L}_j}\langle h, u_{j,k}\rangle v_{j,k} =\sum\limits_{j\in \mathbb{J}}\sum\limits_{k\in \mathbb{L}_j}\langle h, v_{j,k}\rangle u_{j,k} ,~\forall h \in \mathcal{H} ;$$
 $$ 0 \leq \sum\limits_{j\in \mathbb{J}}\sum\limits_{k\in \mathbb{L}_j}\langle h, u_{j,k}\rangle \langle  v_{j,k} , h\rangle \leq b\|h\|^2 ,~ \forall h \in \mathcal{H}, ~\text{and} $$
 $$  \sum\limits_{j\in \mathbb{J}}\sum\limits_{k\in \mathbb{L}_j}|\langle h, u_{j,k}\rangle |^2 \leq c\|h\|^2 ,~ \forall h \in \mathcal{H}; ~  \sum\limits_{j\in \mathbb{J}}\sum\limits_{k\in \mathbb{L}_j} |\langle h, v_{j,k}\rangle|^2\leq d\|h\|^2 ,~ \forall h \in \mathcal{H}.$$ 
\end{enumerate}
\end{theorem}
\begin{proof}
\begin{enumerate}[\upshape(i)]
\item   $( \Rightarrow)$  We may assume $\{A_j\}_{j\in\mathbb{J}} $ is an orthonormal set/basis. 
 Then there exists a sequence $ \{c_j\}_{j\in \mathbb{J}}$ of reals such that $0<\inf\{c_j\}_{j\in \mathbb{J}} \leq\sup\{c_j\}_{j\in \mathbb{J}}<\infty $ and $ \Psi_j=c_jA_j, \forall j \in \mathbb{J}.$ Then for each $ j \in \mathbb{J},$  $ v_{j,k}=\Psi_j^*e_{j,k}=(c_jA_j)^*e_{j,k}=c_jA_j^*e_{j,k}=c_ju_{j,k}, \forall k \in \mathbb{L}_j.$ To show orthonormality of $u_{j,k} $'s, consider the cases $ j\neq l$ and $ j=l$,  for $j,l \in \mathbb{J}$. In the first case, $ \langle u_{j,k},u_{l,m} \rangle=\langle A^*_je_{j,k},A^*_le_{l,m} \rangle=\delta_{j,l}\langle e_{j,k}, e_{l,m} \rangle=0$, and in the second case $ \langle u_{j,k},u_{j,m} \rangle = \langle A^*_je_{j,k},A^*_je_{j,m} \rangle =\langle e_{j,k},e_{j,m}\rangle=\delta_{k,m}.$

 Basis case: Now $\{A_j\}_{j\in\mathbb{J}} $  satisfies $ \sum_{j\in\mathbb{J}} \|A_jh\|^2=\|h\|^2, \forall h \in \mathcal{H}$ and this implies $\|h\|^2= \sum_{j\in\mathbb{J}} \|\sum_{k\in \mathbb{L}_j}\langle h, u_{j,k}\rangle e_{j,k}\|^2=\sum_{j\in\mathbb{J}} \sum_{k\in \mathbb{L}_j} | \langle h, u_{j,k}\rangle |^2, \forall h \in \mathcal{H} .$  Thus, due to the orthonormality of $ \{u_{j,k}: k\in \mathbb{L}_j, j \in \mathbb{J}\} ,$ it is a basis.

$ (\Leftarrow)$  We may assume $\{u_{j,k}: k\in \mathbb{L}_j, j\in \mathbb{J}\} $ is an orthonormal set/basis.  Then there exists a sequence $ \{c_j\}_{j\in \mathbb{J}}$ of reals such that $0<\inf\{c_j\}_{j\in \mathbb{J}} \leq\sup\{c_j\}_{j\in \mathbb{J}}<\infty $ and  for each $ j \in \mathbb{J},$ $ v_{j,k}=c_{j,k}u_{j,k}, \forall k \in \mathbb{L}_j.$  Let $ j, l \in \mathbb{J}.$ Whenever $j\neq l, $ 
$$ \langle A_j^*y ,A_l^*z \rangle= \left\langle \sum_{k\in\mathbb{L}_j}\langle y, e_{j,k}\rangle u_{j,k}, \sum_{m\in\mathbb{L}_l}\langle z, e_{l,m}\rangle u_{l,m} \right \rangle  =0, ~\forall y, z \in \mathcal{H}_0 ,$$ 
 and whenever $ j=l,$ 
$$ \langle A_j^*y ,A_j^*z  \rangle=\left\langle \sum_{k\in\mathbb{L}_j}\langle y, e_{j,k}\rangle u_{j,k}, \sum_{m\in\mathbb{L}_j}\langle z, e_{j,m}\rangle u_{j,m} \right \rangle  = \sum_{k\in \mathbb{L}_j}\langle y,e_{j,k} \rangle \langle e_{j,k},z\rangle=\langle y, z \rangle  ,~\forall y, z \in \mathcal{H}_0. $$ For  all $ h \in \mathcal{H},$
\begin{align*}
\sum_{j\in \mathbb{J}}\|A_jh\|^2&=\left\langle \sum_{j\in\mathbb{J}}A_j^*A_jh, h\right\rangle=\left\langle \sum_{j\in\mathbb{J}}\sum_{k\in \mathbb{L}_j}\langle A_jh,e_{j,k} \rangle A_j^*e_{j,k}, h\right\rangle\\
&=\left\langle \sum_{j\in\mathbb{J}}\sum_{k\in \mathbb{L}_j}\langle h,u_{j,k} \rangle u_{j,k}, h\right\rangle= \sum_{j\in\mathbb{J}}\sum_{k\in \mathbb{L}_j}|\langle h,u_{j,k} \rangle|^2 \leq \| h\|^2,
\end{align*}
and 
\begin{align*}
 \Psi_jh&=\sum_{k\in\mathbb{L}_j}\langle \Psi_jh,e_{j,k}\rangle e_{j,k}=\sum_{k\in\mathbb{L}_j}\langle h,v_{j,k}\rangle e_{j,k}=\sum_{k\in\mathbb{L}_j}\langle h,c_ju_{j,k}\rangle e_{j,k}\\
 &=c_j\sum_{k\in\mathbb{L}_j}\langle h,u_{j,k}\rangle e_{j,k}=c_j\sum_{k\in\mathbb{L}_j}\langle h,A_j^*e_{j,k}\rangle e_{j,k}=c_jA_jh, ~  \forall j \in \mathbb{J}.
\end{align*}
Basis case:  
$\sum_{j\in \mathbb{J}}\|A_jh\|^2= \sum_{j\in\mathbb{J}}\sum_{k\in \mathbb{L}_j}|\langle h,u_{j,k} \rangle|^2 = \| h\|^2, \forall h \in \mathcal{H}.$
\item  $ (\Rightarrow)$ There exist $ \{F_j\}_{j \in \mathbb{J}}$, an orthonormal basis in  $ \mathcal{B}(\mathcal{H}, \mathcal{H}_0)$ and  bounded invertible operators $ U, V : \mathcal{H} \rightarrow \mathcal{H}$   with $ V^*U$  positive  such that $ A_j=F_jU, \Psi_j=F_jV,  \forall j \in \mathbb{J}.$ Then $u_{j,k}= A_j^*e_{j,k}=U^*F_j^*e_{j,k}, v_{j,k}=\Psi_j^*e_{j,k} =V^*F_j^*e_{j,k},  \forall k \in  \mathbb{L}_j, \forall j\in \mathbb{J},$ and $ V^*(U^*)^*=V^*U\geq 0$. Now we are through if we show $ \{F^*_je_{j,k} :k \in \mathbb{L}_j, j \in \mathbb{J}\}$ is an orthonormal basis for $ \mathcal{H}.$ Let $ j, k \in \mathbb{J}.$ For $j\neq k,$ $ \langle F_j^*e_{j,l},F_k^*e_{k,m} \rangle =\delta_{j,k}\langle e_{j,l}, e_{k,m}\rangle =0;$ for $ j=k,$  $ \langle F_j^*e_{j,l},F_j^*e_{j,m} \rangle =\langle e_{j,l}, e_{j,m}\rangle=\delta_{l,m}$ and 
$$\sum_{k\in \mathbb{L}_j, j \in \mathbb{J}}|\langle h,F_j^*e_{j,k} \rangle|^2=\sum_{j\in \mathbb{J}}\sum_{k\in \mathbb{L}_j} |\langle F_jh, e_{j,k} \rangle|^2= \sum_{j\in \mathbb{J}}\|F_jh\|^2= \|h\|^2, ~\forall h \in \mathcal{H} .$$

 $ ( \Leftarrow)$ Let $ U,V :  \mathcal{H}\rightarrow \mathcal{H} $ be bounded invertible,  $ \{f_{j,k}: k\in \mathbb{L}_j, j \in \mathbb{J}\} $ be  an orthonormal basis  for $\mathcal{H}$ such that $ VU^*\geq0$ and $u_{j,k}=Uf_{j,k}, v_{j,k}=Vf_{j,k},  \forall k \in  \mathbb{L}_j, \forall j\in \mathbb{J}. $  Define $ F_j: \mathcal{H} \ni \sum_{l \in \mathbb{J},k \in \mathbb{L}_l} a_{l,k} f_{l,k}\mapsto \sum_{k \in \mathbb{L}_j} a_{j,k}e_{j,k}\in  \mathcal{H}_0 .$ Claim: $ A_j=F_jU^*, \Psi_j=F_jV^*, \forall j \in \mathbb{J}$. To show this, it is enough if we get $ A^*_j=UF_j^*, \Psi_j^*=VF_j^*, \forall j \in \mathbb{J}.$ For this, we find $ F_j^*, \forall j \in \mathbb{J}.$ For each fixed $ j  \in \mathbb{J}, $ and for $ y \in \mathcal{H}_0$, since  $ \{e_{j,k}\}_{k\in \mathbb{L}_j}$ is an orthonormal basis for $ \mathcal{H}_0$,  we can write $ y= \sum_{k \in \mathbb{L}_j}a_{j,k}e_{j,k},$ uniquely. Also,  each $ h \in \mathcal{H}$ has unique expansion $h= \sum_{l \in \mathbb{J},m \in \mathbb{L}_l} b_{l,m} f_{l,m}$. Then $\langle F_j^*y, h\rangle =\langle F_j^* (\sum_{k \in \mathbb{L}_j}a_{j,k}e_{j,k}),\sum_{m \in \mathbb{L}_l, l \in \mathbb{J} } b_{l,m} f_{l,m} \rangle =\langle  \sum_{k \in \mathbb{L}_j}a_{j,k}e_{j,k},\sum_{m \in \mathbb{L}_j} b_{j,m} e_{j,m} \rangle =\sum_{k\in \mathbb{L}_j}a_{j,k}\overline{b_{j,k}}$ and $\langle \sum_{k \in \mathbb{L}_j}a_{j,k}f_{j,k} , h\rangle =\langle  \sum_{k \in \mathbb{L}_j}a_{j,k}f_{j,k},\sum_{m \in \mathbb{L}_l, l \in \mathbb{J} } b_{l,m} f_{l,m} \rangle =  \sum_{k \in \mathbb{L}_j}a_{j,k}\langle f_{j,k},\sum_{ l \in \mathbb{J} } \sum_{ m \in \mathbb{L}_l}b_{l,m} f_{l,m} \rangle =\sum_{k\in \mathbb{L}_j}a_{j,k}\overline{b_{j,k}} .$ Therefore $ F_j^*y=\sum_{k \in \mathbb{L}_j}a_{j,k}f_{j,k}$. This gives $ UF_j^*y= \sum_{k \in \mathbb{L}_j}a_{j,k}Uf_{j,k}=\sum_{k \in \mathbb{L}_j}a_{j,k}u_{j,k}=\sum_{k \in \mathbb{L}_j}a_{j,k}A_j^*e_{j,k}=A_j^*y,  VF_j^*y= \sum_{k \in \mathbb{L}_j}a_{j,k}Vf_{j,k}=\sum_{k \in \mathbb{L}_j}a_{j,k}v_{j,k}=\sum_{k \in \mathbb{L}_j}a_{j,k}\Psi_j^*e_{j,k}=\Psi_j^*y,   \forall j \in \mathbb{J}.$ Using $ \{F_j^*\}_{ j \in \mathbb{J}}$ we show $\{F_j\}_{ j \in \mathbb{J}} $ is an orthonormal basis in $\mathcal{B}(\mathcal{H},\mathcal{H}_0)$. For $ j\neq k $, given $ y, z \in \mathcal{H}_0,$ we write  $y= \sum_{l \in \mathbb{L}_j}c_{j,l}e_{j,l},  z= \sum_{m \in \mathbb{L}_k}d_{k,m}e_{k,m}$,  uniquely. Then $\langle F_j^*y, F_k^*z\rangle= \langle \sum_{l \in \mathbb{L}_j}c_{j,l}f_{j,l}, \sum_{m \in \mathbb{L}_k}d_{k,m}f_{k,m}\rangle= 0 .$ For $ j=k$, given $ w,x \in \mathcal{H}_0,$ we write  $w= \sum_{l \in \mathbb{L}_j}r_{j,l}e_{j,l},  x= \sum_{m \in \mathbb{L}_j}s_{j,m}e_{j,m}$,  uniquely. Then $ \langle F_j^*w, F_j^*x\rangle= \langle\sum_{l \in \mathbb{L}_j}r_{j,l}f_{j,l}, \sum_{m \in \mathbb{L}_j}s_{j,m}f_{j,m}\rangle =\sum_{l \in \mathbb{L}_j}r_{j,l}\overline{s_{j,l}} =\langle w, x \rangle.$ Now for $h=\sum_{k \in \mathbb{L}_l, l \in \mathbb{J}}a_{l,k} f_{l,k}\in \mathcal{H}$, $ \sum_{j \in \mathbb{J}} \|F_jh\|^2=\sum_{j \in \mathbb{J}}\| \sum_{k \in \mathbb{L}_j} a_{j,k}e_{j,k}\|^2=\sum_{j \in \mathbb{J}} \sum_{k \in \mathbb{L}_j} |a_{j,k}|^2 =\|h\|^2.$ At last $ (V^*)^*U^*=VU^*\geq 0.$

\item For all $ h\in \mathcal{H},$
$$ \sum\limits_{j \in \mathbb{J}}\sum\limits_{k \in \mathbb{L}_j}\langle h, u_{j,k}\rangle  v_{j,k}= \sum\limits_{j \in \mathbb{J}}\sum\limits_{k \in \mathbb{L}_j}\langle A_jh, e_{j,k}\rangle  \Psi_j^*e_{j,k} =\sum\limits_{j \in \mathbb{J}}\Psi_j^*A_jh,$$
\begin{align*}
\sum\limits_{j \in \mathbb{J}}\langle A_jh, \Psi_jh\rangle &=\sum\limits_{j \in \mathbb{J}} \left \langle\sum\limits_{k \in \mathbb{L}_j}\langle h, A_j^*e_{j,k} \rangle e_{j,k}, \sum\limits_{l \in \mathbb{L}_j} \langle h, \Psi_j^*e_{j,l}\rangle e_{j,l} \right\rangle\\
&=\sum\limits_{j \in \mathbb{J}}\left \langle\sum\limits_{k \in \mathbb{L}_j}\langle h, u_{j,k}\rangle e_{j,k}, \sum\limits_{l \in \mathbb{L}_j} \langle h, v_{j,l}\rangle e_{j,l} \right\rangle=\sum\limits_{j \in \mathbb{J}}\sum\limits_{k \in \mathbb{L}_j}\langle h, u_{j,k}\rangle \langle v_{j,k}, h\rangle, 
\end{align*}
\begin{align*}
\left\|\sum\limits_{j \in \mathbb{J}}L_jA_jh\right\|^2=\sum\limits_{j \in \mathbb{J}}\|A_jh\|^2=\sum\limits_{j \in \mathbb{J}}\sum\limits_{k \in \mathbb{L}_j}|\langle h, u_{j,k}\rangle|^2; ~\left\|\sum\limits_{j \in \mathbb{J}}L_j\Psi_jh\right\|^2=\sum\limits_{j \in \mathbb{J}}\sum\limits_{k \in \mathbb{L}_j}|\langle h, v_{j,k}\rangle|^2.
\end{align*}
 \item Similar to (iii).
 \item $ \sum_{j \in \mathbb{J}} \Psi_j^*A_j$ exists and is bounded  positive  invertible if and only if  there exist $ c,d, r  >0$ such that 
 $  \|\sum_{j\in \mathbb{J}}\sum_{k\in \mathbb{L}_j}\langle h, u_{j,k}\rangle v_{j,k}\|\leq r\|h\|,\forall h \in \mathcal{H}, \sum_{j\in \mathbb{J}}\sum_{k\in \mathbb{L}_j}\langle h, u_{j,k}\rangle v_{j,k} =\sum_{j\in \mathbb{J}}\sum_{k\in \mathbb{L}_j}\langle h, v_{j,k}\rangle u_{j,k} ,\forall h \in \mathcal{H}$ and $
 a\|h\|^2\leq \sum_{j\in \mathbb{J}}\sum_{k\in \mathbb{L}_j}\langle h, u_{j,k}\rangle \langle  v_{j,k} , h\rangle  \leq b\|h\|^2 , \forall h \in \mathcal{H} $, and $ \sum_{j \in \mathbb{J}} L_jA_j$ (resp. $ \sum_{j \in \mathbb{J}} L_j\Psi_j$) exists and is bounded if and only if there exists $ c>0$ (resp. $ d>0$) such that $\sum_{j\in \mathbb{J}}\sum_{k\in \mathbb{L}_j}|\langle h, u_{j,k}\rangle |^2 \leq c\|h\|^2 , \forall h \in \mathcal{H} $ (resp. $ \sum_{j\in \mathbb{J}}\sum_{k\in \mathbb{L}_j}|\langle h, v_{j,k}\rangle |^2 \leq d\|h\|^2 , \forall h \in \mathcal{H}$).
 \item Similar to (v).
\end{enumerate}
\end{proof}
\begin{remark}
Theorem \ref{SEQUENTIAL CHARACTERIZATION} gives light to ``define" extension in sequential form, we will do this in Section \ref{SEQUENTIAL}.	
\end{remark}
\begin{caution}
Theorem \ref{SEQUENTIAL CHARACTERIZATION} does not say ``it is enough to study either operator version or sequential version",  because it holds under the assumption  ``$ \{e_{j,k}\}_{k\in\mathbb{L}_j}$ is an orthonormal basis for $ \mathcal{H}_0,$ for each $j \in \mathbb{J}$", which need not hold for all $ \mathcal{H}_0$ and for each $ j \in \mathbb{J}$. Thus it is necessary to study both	operator version and sequential version. So Section \ref{SEQUENTIAL} doesn't come from this section, we have to study it separately, with a separate definition (that should match with Theorem \ref{SEQUENTIAL CHARACTERIZATION}, whenever its assumptions are fulfilled).
\end{caution}

\section{Similarity, composition and tensor product}\label{SIMILARITYCOMPOSITIONANDTENSORPRODUCT}
\begin{definition}
 An (ovf)  $(\{B_j\}_{j\in \mathbb{J}},  \{\Phi_j\}_{j\in \mathbb{J}})$  in $ \mathcal{B}(\mathcal{H}, \mathcal{H}_0)$    is said to be right-similar (resp. left-similar) to an (ovf)  $(\{A_j\}_{j\in \mathbb{J}},   \{\Psi_j\}_{j\in \mathbb{J}})$ in $ \mathcal{B}(\mathcal{H}, \mathcal{H}_0)$  if there exist invertible  $ R_{A,B}, R_{\Psi, \Phi} \in \mathcal{B}(\mathcal{H})$  (resp. $ L_{A,B}, L_{\Psi, \Phi}$ $ \in \mathcal{B}(\mathcal{H}_0)$) such that $B_j=A_jR_{A,B} , \Phi_j=\Psi_jR_{\Psi, \Phi} $  (resp. $ B_j=L_{A,B}A_j , \Phi_j=L_{\Psi, \Phi}\Psi_j$), $ \forall j \in \mathbb{J}. $
\end{definition}
Since the operators giving similarity are invertible, definition  of similarity is symmetric.
\begin{proposition}\label{RIGHTSIMILARITYPROPOSITIONOPERATORVERSION}
Let $ \{A_j\}_{j\in \mathbb{J}}\in \mathscr{F}_\Psi$  with frame bounds $a, b,$  let $R_{A,B}, R_{\Psi, \Phi} \in \mathcal{B}(\mathcal{H})$ be positive, invertible, commute with each other, commute with $ S_{A, \Psi}$, and let $B_j=A_jR_{A,B} , \Phi_j=\Psi_jR_{\Psi, \Phi},  \forall j \in \mathbb{J}.$ Then 
\begin{enumerate}[\upshape(i)]
\item $ \{B_j\}_{j\in \mathbb{J}}\in \mathscr{F}_\Phi$ and $ \frac{a}{\|R_{A,B}^{-1}\|\|R_{\Psi,\Phi}^{-1}\|}\leq S_{B, \Phi} \leq b\|R_{A,B}R_{\Psi,\Phi}\|.$ Assuming that $ (\{A_j\}_{j\in \mathbb{J}},\{\Psi_j\}_{j\in \mathbb{J}})$ is Parseval, then $(\{B_j\}_{j\in \mathbb{J}},  \{\Phi_j\}_{j\in \mathbb{J}})$ is Parseval  if and only if   $ R_{\Psi, \Phi}R_{A,B}=I_\mathcal{H}.$  
\item $ \theta_B=\theta_A R_{A,B}, \theta_\Phi=\theta_\Psi R_{\Psi,\Phi}, S_{B,\Phi}=R_{\Psi,\Phi}S_{A, \Psi}R_{A,B},  P_{B,\Phi}=P_{A, \Psi}.$
\end{enumerate}
\end{proposition}
\begin{proof}
$ \sum_{j\in \mathbb{J}}\Phi_j^*B_j=\sum_{j\in \mathbb{J}}(\Psi_jR_{\Psi,\Phi})^*(A_jR_{A,B})=R_{\Psi, \Phi}\left (\sum_{j\in \mathbb{J}}\Psi_j^*A_j\right )R_{A,B}$ implies $S_{B, \Phi} $ exists, positive, invertible and equals $ R_{\Psi,\Phi}S_{A,\Psi}R_{A,B}.$ Now, $ \theta_B=\sum_{j\in \mathbb{J}}L_jB_j=\sum_{j\in \mathbb{J}}L_jA_jR_{A,B}=\theta_AR_{A,B} $, and similarly for $ \theta_\Phi;$ $ P_{B,\Phi}=\theta_BS_{B,\Phi}^{-1}\theta_\Phi^*=(\theta_AR_{A,B})(R_{\Psi,\Phi}S_{A, \Psi}R_{A,B})^{-1}(\theta_\Psi R_{\Psi,\Phi})^*=P_{A,\Psi}.$ Assumptions together with inequality $ \frac{1}{\|R_{A,B}^{-1}\|\|R_{\Psi,\Phi}^{-1}\|}\leq  R_{A,B}R_{\Psi,\Phi}\leq \|R_{A,B}R_{\Psi,\Phi}\| $ and (ii) gives the inequality in (i). Second part of (i) is clear. 
\end{proof}
\begin{lemma}\label{SIM}
 Let $ \{A_j\}_{j\in \mathbb{J}}\in \mathscr{F}_\Psi,$ $ \{B_j\}_{j\in \mathbb{J}}\in \mathscr{F}_\Phi$ and   $B_j=A_jR_{A,B} ,\Phi_j=\Psi_jR_{\Psi, \Phi},  \forall j \in \mathbb{J}$, for some invertible $ R_{A,B} ,R_{\Psi, \Phi} \in \mathcal{B}(\mathcal{H}).$ Then 
  $ \theta_B=\theta_A R_{A,B}, \theta_\Phi=\theta_\Psi R_{\Psi,\Phi}, S_{B,\Phi}=R_{\Psi,\Phi}^*S_{A, \Psi}R_{A,B},  P_{B,\Phi}=P_{A, \Psi}.$ Assuming that $ (\{A_j\}_{j\in \mathbb{J}},\{\Psi_j\}_{j\in \mathbb{J}})$ is Parseval, then $(\{B_j\}_{j\in \mathbb{J}},  \{\Phi_j\}_{j\in \mathbb{J}})$ is Parseval  if and only if   $ R_{\Psi, \Phi}^*R_{A,B}=I_\mathcal{H}.$ 
 \end{lemma}
 \begin{proof}
 From the definitions of $\theta_B $  and $\theta_\Phi $ we get first two conclusions; substituting these in definitions of $ S_{B,\Phi}$ and $ P_{A,\Psi}$ give the remaining two.
 \end{proof}
 
\begin{theorem}\label{RIGHTSIMILARITY}
Let $ \{A_j\}_{j\in \mathbb{J}}\in \mathscr{F}_\Psi,$ $ \{B_j\}_{j\in \mathbb{J}}\in \mathscr{F}_\Phi.$ The following are equivalent.
\begin{enumerate}[\upshape(i)]
\item $B_j=A_jR_{A,B} , \Phi_j=\Psi_jR_{\Psi, \Phi} ,  \forall j \in \mathbb{J},$ for some invertible  $ R_{A,B} ,R_{\Psi, \Phi} \in \mathcal{B}(\mathcal{H}). $
\item $\theta_B=\theta_AR_{A,B}' , \theta_\Phi=\theta_\Psi R_{\Psi, \Phi}' $ for some invertible  $ R_{A,B}' ,R_{\Psi, \Phi}' \in \mathcal{B}(\mathcal{H}). $
\item $P_{B,\Phi}=P_{A,\Psi}.$
\end{enumerate}
If one of the above conditions is satisfied, then  invertible operators in  $ \operatorname{(i)}$ and  $ \operatorname{(ii)}$ are unique and are given by $R_{A,B}=S_{A,\Psi}^{-1}\theta_\Psi^*\theta_B, R_{\Psi, \Phi}=S_{A,\Psi}^{-1}\theta_A^*\theta_\Phi.$
In the case that $(\{A_j\}_{j\in \mathbb{J}},  \{\Psi_j\}_{j\in \mathbb{J}})$ is Parseval, then $(\{B_j\}_{j\in \mathbb{J}},  \{\Phi_j\}_{j\in \mathbb{J}})$ is  Parseval if and only if $R_{\Psi, \Phi}^*R_{A,B} $  is the identity operator if and only if $R_{A,B}R_{\Psi, \Phi}^* $  is the identity operator. 
\end{theorem}
\begin{proof}
(i) $\Rightarrow$ (ii) is direct. For the reverse, let (ii) hold. From Proposition \ref{2.2}, $ B_j=L_j^*\theta_B=L_j^*\theta_AR_{A,B}'=A_jR_{A,B}'$; the same procedure gives $ \Phi_j$ also.  Lemma \ref{SIM} gives (ii)  $ \Rightarrow$ (iii). Assume (iii). We note the following $ \theta_B=P_{B,\Phi}\theta_B$ and $ \theta_\Phi=P_{B,\Phi}^*\theta_\Phi.$ Using these, $ \theta_B=P_{A,\Psi}\theta_B=\theta_A(S_{A,\Psi}^{-1}\theta_\Psi^*\theta_B)$ and $ \theta_\Phi=P_{A,\Psi}^*\theta_\Phi=(\theta_AS_{A,\Psi}^{-1}\theta_\Psi^*)^*\theta_\Phi=\theta_\Psi(S_{A,\Psi}^{-1}\theta_A^*\theta_\Phi).$ We now try to show that both $S_{A,\Psi}^{-1}\theta_\Psi^*\theta_B$  and $S_{A,\Psi}^{-1}\theta_A^*\theta_\Phi$ are invertible. This is achieved via,  $(S_{A,\Psi}^{-1}\theta_\Psi^*\theta_B)(S_{B,\Phi}^{-1}\theta_\Phi^*\theta_A)=S_{A,\Psi}^{-1}\theta_\Psi^*P_{B,\Phi}\theta_A= S_{A,\Psi}^{-1}\theta_\Psi^*P_{A,\Psi}\theta_A= S_{A,\Psi}^{-1}\theta_\Psi^*\theta_A=I_\mathcal{H}, ( S_{B,\Phi}^{-1}\theta_\Phi^*\theta_A)(S_{A,\Psi}^{-1}\theta_\Psi^*\theta_B)= S_{B,\Phi}^{-1}\theta_\Phi^*P_{A,\Psi}\theta_B=S_{B,\Phi}^{-1}\theta_\Phi^*P_{B,\Phi}\theta_B=S_{B,\Phi}^{-1}\theta_\Phi^*\theta_B=I_\mathcal{H}$ and $(S_{A,\Psi}^{-1}\theta_A^*\theta_\Phi)(S_{B,\Phi}^{-1}\theta_B^*\theta_\Psi)=S_{A,\Psi}^{-1}\theta_A^*P_{B,\Phi}^*\theta_\Psi=S_{A,\Psi}^{-1}\theta_A^*P_{A,\Psi}^*\theta_\Psi=S_{A,\Psi}^{-1}\theta_A^*\theta_\Psi=I_\mathcal{H},(S_{B,\Phi}^{-1}\theta_B^*\theta_\Psi)(S_{A,\Psi}^{-1}\theta_A^*\theta_\Phi)=S_{B,\Phi}^{-1}\theta_B^*P_{A,\Psi}^* \theta_\Phi=S_{B,\Phi}^{-1}\theta_B^*P_{B,\Phi}^* \theta_\Phi=S_{B,\Phi}^{-1}\theta_B^* \theta_\Phi= I_\mathcal{H}.$

Let $ R_{A,B}, R_{\Psi,\Phi} \in \mathcal{B}(\mathcal{H}) $ be invertible. From the previous arguments, $ R_{A,B}$ and $R_{\Psi,\Phi} $ satisfy (i) if and only if  they satisfy (ii). Let $B_j=A_jR_{A,B} , \Phi_j=\Psi_jR_{\Psi, \Phi} ,  \forall j \in \mathbb{J}.$ Using (ii), $\theta_B=\theta_AR_{A,B} , \theta_\Phi=\theta_\Psi R_{\Psi, \Phi}$ and these imply  $\theta_\Psi^*\theta_B=\theta_\Psi^*\theta_AR_{A,B}=S_{A,\Psi}R_{A,B} , \theta_A^*\theta_\Phi=\theta_A^*\theta_\Psi R_{\Psi, \Phi}=S_{A,\Psi}R_{\Psi, \Phi}$ these imply the formula for $R_{A,B}$ and $ R_{\Psi, \Phi}.$ For the last,  we recall $ S_{B,\Phi}=R_{\Psi,\Phi}^*S_{A, \Psi}R_{A,B}$  and the following fact: let  $a $ and $ b$ be  invertible elements  such that $ ab=e, $ the identity element. Then $ ba=(a^{-1}(ab)b^{-1})^{-1}=e.$
\end{proof}
\begin{corollary}
For any given (ovf) $ (\{A_j\}_{j \in \mathbb{J}} , \{\Psi_j\}_{j \in \mathbb{J}})$, the canonical dual of $ (\{A_j\}_{j \in \mathbb{J}} , \{\Psi_j\}_{j \in \mathbb{J}}  )$ is the only dual (ovf) that is right-similar to $ (\{A_j\}_{j \in \mathbb{J}} , \{\Psi_j\}_{j \in \mathbb{J}} )$.
\end{corollary}
\begin{proof}
Whenever $ (\{B_j\}_{j \in \mathbb{J}} , \{\Phi_j\}_{j \in \mathbb{J}} )$ is dual of $ (\{A_j\}_{j \in \mathbb{J}} , \{\Psi_j\}_{j \in \mathbb{J}})$ as well as right-similar to $ (\{A_j\}_{j \in \mathbb{J}} , \{\Psi_j\}_{j \in \mathbb{J}} )$, we have $ \theta_B^*\theta_\Psi=I_\mathcal{H}=\theta_\Phi^*\theta_A$ and  there exist invertible $ R_{A,B},R_{\Psi,\Phi}\in \mathcal{B}(\mathcal{H})$ such that  $B_j=A_jR_{A,B} , \Phi_j=\Psi_jR_{\Psi, \Phi} ,  \forall j \in \mathbb{J} $. Theorem \ref{RIGHTSIMILARITY} tells  $R_{A,B}=S_{A,\Psi}^{-1}\theta_\Psi^*\theta_B, R_{\Psi, \Phi}=S_{A,\Psi}^{-1}\theta_A^*\theta_\Phi.$ But then $R_{A,B}=S_{A,\Psi}^{-1}I_\mathcal{H}=S_{A,\Psi}^{-1}, R_{\Psi, \Phi}=S_{A,\Psi}^{-1}I_\mathcal{H}=S_{A,\Psi}^{-1}.$ So $ (\{B_j=A_jS_{A,\Psi}^{-1}\}_{j \in \mathbb{J}} , \{\Phi_j=\Psi_jS_{A,\Psi}^{-1}\}_{j \in \mathbb{J}} )$ is the  canonical  dual of $ (\{A_j\}_{j \in \mathbb{J}} , \{\Psi_j\}_{j \in \mathbb{J}})$.
\end{proof}
\begin{corollary}
Two right-similar operator-valued frames cannot be orthogonal.
\end{corollary}
\begin{proof}
Let $(\{B_j\}_{j \in \mathbb{J}}, \{\Phi_j\}_{j \in \mathbb{J}})$ be right-similar to $ (\{A_j\}_{j \in \mathbb{J}} , \{\Psi_j\}_{j \in \mathbb{J}})$. Choose invertible $ R_{A,B},R_{\Psi,\Phi}\in \mathcal{B}(\mathcal{H})$ such that  $B_j=A_jR_{A,B} , \Phi_j=\Psi_jR_{\Psi, \Phi} ,  \forall j \in \mathbb{J} $. Using 	Theorem \ref{RIGHTSIMILARITY} we get $ \theta_B^*\theta_\Psi=(\theta_AR_{A,B})^*\theta_\Psi=R_{A,B}^*\theta_A^*\theta_\Psi=R_{A,B}^*S_{A,\Psi}\neq 0$ because  $R_{A,B}^* $ and $S_{A,\Psi} $ are invertible.
\end{proof}
\begin{remark}
 For every (ovf) $(\{A_j\}_{j \in \mathbb{J}},\{\Psi_j\}_{j \in \mathbb{J}})$, each  of `operator-valued frames'  $( \{A_jS_{A, \Psi}^{-1}\}_{j \in \mathbb{J}}, \{\Psi_j\}_{j \in \mathbb{J}}),$   $( \{A_jS_{A, \Psi}^{-1/2}\}_{j \in \mathbb{J}}, \{\Psi_jS_{A,\Psi}^{-1/2}\}_{j \in \mathbb{J}}),$ and  $ (\{A_j \}_{j \in \mathbb{J}}, \{\Psi_jS_{A,\Psi}^{-1}\}_{j \in \mathbb{J}})$ is a Parseval (ovf) which is right-similar to  $ (\{A_j\}_{j \in \mathbb{J}} , \{\Psi_j\}_{j \in \mathbb{J}}  ).$  Thus every (ovf) is right-similar to  Parseval  operator-valued frames.
\end{remark}
\begin{proposition}\label{3.3}
 Let $ \{A_j\}_{j\in \mathbb{J}}\in \mathscr{F}_\Psi,$ $ \{B_j\}_{j\in \mathbb{J}}\in \mathscr{F}_\Phi$ and   $B_j=L_{A,B}A_j , \Phi_j=L_{\Psi, \Phi}\Psi_j,  \forall j \in \mathbb{J}$, for some invertible $ L_{A,B} ,L_{\Psi, \Phi} \in \mathcal{B}(\mathcal{H}_0).$ Then 
 \begin{enumerate}[\upshape(i)]
 \item $ \theta_B=(I_{\ell^2(\mathbb{J})}\otimes L_{A,B})\theta_A , \theta_\Phi=(I_{\ell^2(\mathbb{J})}\otimes L_{\Psi,\Phi})\theta_\Psi, S_{B,\Phi}=\theta_\Psi^*(I_{\ell^2(\mathbb{J})}\otimes L_{\Psi,\Phi}^*L_{A,B})\theta_A,  P_{B,\Phi}=(I_{\ell^2(\mathbb{J})}\otimes L_{A,B})\theta_A(\theta_\Psi^*(I_{\ell^2(\mathbb{J})}\otimes L_{\Psi,\Phi}^*L_{A,B})\theta_A)^{-1}\theta_\Psi^*(I_{\ell^2(\mathbb{J})}\otimes L_{\Psi,\Phi}^*).$ 
  \item Assuming $( \{A_j\}_{j \in \mathbb{J}}, \{\Psi_j\}_{j \in \mathbb{J}})$ is Parseval, then $( \{B_j\}_{j \in \mathbb{J}}, \{\Phi_j\}_{j \in \mathbb{J}})$ is Parseval if and only if $ P_{A, \Psi}(I_{\ell^2(\mathbb{J})}\otimes L_{\Psi,\Phi}^*L_{A,B})P_{A,\Psi}=P_{A,\Psi}$ if and only if $ P_{B,\Phi}=(I_{\ell^2(\mathbb{J})}\otimes L_{A,B})P_{A,\Psi}(I_{\ell^2(\mathbb{J})}\otimes L_{\Psi, \Phi}^*).$
  \end{enumerate}
 \end{proposition}
\begin{proof}
(i) From the definition of $L_j ,$ we see that $ L_jR=(I_{\ell^2(\mathbb{J})}\otimes R) L_j, \forall R \in \mathcal{B}(\mathcal{H}_0), \forall j \in \mathbb{J}.$ Using this we get  $\theta_B= \sum_{j \in \mathbb{J}} L_jB_j= \sum_{j \in \mathbb{J}} L_jL_{A,B}A_j=(I_{\ell^2(\mathbb{J})}\otimes L_{A,B})\sum_{j \in \mathbb{J}} L_jA_j= (I_{\ell^2(\mathbb{J})}\otimes L_{A,B})\theta_A $ and $ \theta_\Phi=\sum_{j \in \mathbb{J}} L_j\Phi_j= \sum_{j \in \mathbb{J}} L_jL_{\Psi,\Phi}\Psi_j=(I_{\ell^2(\mathbb{J})}\otimes L_{\Psi,\Phi})\sum_{j \in \mathbb{J}} L_j\Psi_j= (I_{\ell^2(\mathbb{J})}\otimes L_{\Psi,\Phi})\theta_\Psi.$ Now $S_{B,\Phi}=\theta_\Phi^*\theta_B=((I_{\ell^2(\mathbb{J})}\otimes L_{\Psi,\Phi})\theta_\Psi)^*(I_{\ell^2(\mathbb{J})}\otimes L_{A,B})\theta_A =\theta_\Psi^*(I_{\ell^2(\mathbb{J})}\otimes L_{\Psi,\Phi}^*L_{A,B})\theta_A$ and $ P_{B,\Phi}=\theta_BS_{B,\Phi}^{-1}\theta_\Phi^*=(I_{\ell^2(\mathbb{J})}\otimes L_{A,B})\theta_A(\theta_\Psi^*(I_{\ell^2(\mathbb{J})}\otimes L_{\Psi,\Phi}^*L_{A,B})\theta_A)^{-1}((I_{\ell^2(\mathbb{J})}\otimes L_{\Psi,\Phi})\theta_\Psi)^* =(I_{\ell^2(\mathbb{J})}\otimes L_{A,B})\theta_A(\theta_\Psi^*(I_{\ell^2(\mathbb{J})}\otimes L_{\Psi,\Phi}^*L_{A,B})\theta_A)^{-1}\theta_\Psi^*(I_{\ell^2(\mathbb{J})}\otimes L_{\Psi,\Phi}^*).$ 
 
 (ii) (First equivalence) For the direct part, let  $ I_\mathcal{H}=S_{B,\Phi}=\theta_\Psi^*(I_{\ell^2(\mathbb{J})}\otimes L_{\Psi,\Phi}^*L_{A,B})\theta_A$. Premultiplying this equality by $\theta_A $  and postmultiplying by $\theta_\Psi^* $ give the conclusion. For the converse part,  let $ \theta_A\theta_\Psi^*=\theta_A(\theta_\Psi^*(I_{\ell^2(\mathbb{J})}\otimes L_{\Psi, \Phi}^*L_{A,B})\theta_A)\theta_\Psi^*$. The bracketed term in this equality is the  expression for $ S_{B, \Phi}. $ Thus $   \theta_A\theta_\Psi^*= \theta_AS_{B,\Phi}\theta_\Psi^*.$  Multiplying this equality from the left by $ \theta_\Psi^*$ and from the right by $ \theta_A$ give  $ I_\mathcal{H}S_{B,\Phi}I_\mathcal{H}=I_\mathcal{H}I_\mathcal{H}.$ 
 
 (Second equivalence)  Let  $ P_{B,\Phi}=(I_{\ell^2(\mathbb{J})}\otimes L_{A,B})P_{A,\Psi}(I_{\ell^2(\mathbb{J})}\otimes L_{\Psi, \Phi}^*)$ which gives $\theta_\Phi^*\theta_BS_{B,\Phi}^{-1}\theta_\Phi^*\theta_B=\theta_\Phi^*(I_{\ell^2(\mathbb{J})}\otimes L_{A,B})\theta_A\theta_\Psi^*(I_{\ell^2(\mathbb{J})}\otimes L_{\Psi, \Phi}^*)\theta_B=\theta_\Phi^*\theta_B\theta_\Phi^*\theta_B=S_{B,\Phi}^2.$  On the other hand, let $( \{B_j\}_{j \in \mathbb{J}}, \{\Phi_j\}_{j \in \mathbb{J}})$ be Parseval. Substituting for $\theta_B $ and $\theta_\Phi $ in the definition of $P_{B, \Phi} $ gives the last equality in the statement.
\end{proof}
\begin{definition}
Let $(\{A_j\}_{j \in \mathbb{J}}, \{\Psi_j \}_ {j \in \mathbb{J}})$ and $(\{B_j\}_{j \in \mathbb{J}}, \{\Phi_j \}_ {j \in \mathbb{J}})$ be two operator-valued frames in  $ \mathcal{B}(\mathcal{H}, \mathcal{H}_0) $. We say that  $ (\{B_j\}_{j \in \mathbb{J}},  \{\Phi_j \}_ {j \in \mathbb{J}})$ is   
\begin{enumerate}[\upshape(i)]
\item RL-similar (right-left-similar) to $( \{A_j\}_{j \in \mathbb{J}}, \{\Psi_j \}_ {j \in \mathbb{J}})$ if there exist invertible $R_{A,B}, L_{\Psi, \Phi} $ in $  \mathcal{B}(\mathcal{H}), \mathcal{B}(\mathcal{H}_0)  $, respectively  such that $B_j=A_jR_{A,B} , \Phi_j=L_{\Psi, \Phi}\Psi_j, \forall j \in \mathbb{J}.$
\item LR-similar (left-right-similar) to $( \{A_j\}_{j \in \mathbb{J}},  \{\Psi_j \}_ {j \in \mathbb{J}})$ if there exist invertible $L_{A,B}, R_{\Psi, \Phi}  $ in $ \mathcal{B}(\mathcal{H}_0), \mathcal{B}(\mathcal{H})$, respectively  such that $B_j=L_{A,B}A_j , \Phi_j=\Psi_jR_{\Psi, \Phi} , \forall j \in \mathbb{J}.$
\end{enumerate}
\end{definition}
\begin{proposition}
Let $ \{A_j\}_{j\in \mathbb{J}}\in \mathscr{F}_\Psi,$ $ \{B_j\}_{j\in \mathbb{J}}\in \mathscr{F}_\Phi$ and   $B_j=A_jR_{A,B}, \Phi_j=L_{\Psi, \Phi}\Psi_j,  \forall j \in \mathbb{J}$, for some invertible $ R_{A,B} ,L_{\Psi, \Phi} $ in $ \mathcal{B}(\mathcal{H}),\mathcal{B}(\mathcal{H}_0) $, respectively. Then 
 $ \theta_B=\theta_A R_{A,B}, \theta_\Phi=(I_{\ell^2(\mathbb{J})}\otimes L_{\Psi, \Phi})\theta_\Psi , S_{B,\Phi}=\theta_\Psi^*(I_{\ell^2(\mathbb{J})}\otimes L^*_{\Psi, \Phi})\theta_AR_{A,B},  P_{B,\Phi}=\theta_A(\theta_\Psi^*(I_{\ell^2(\mathbb{J})}\otimes L^*_{\Psi, \Phi})\theta_A)^{-1}\theta_\Psi^*(I_{\ell^2(\mathbb{J})}\otimes L^*_{\Psi, \Phi}).$
 \end{proposition}
 \begin{proof}
$\theta_B=\sum_{j\in\mathbb{J}}L_jB_j=\sum_{j\in\mathbb{J}}L_jA_jR_{A,B}=\theta_A R_{A,B} $,  $\theta_\Phi=\sum_{j\in\mathbb{J}}L_j\Phi_j=\sum_{j\in\mathbb{J}}L_jL_{\Psi, \Phi}\Psi_j=(I_{\ell^2(\mathbb{J})}\otimes L_{\Psi, \Phi})\theta_\Psi $, $S_{B,\Phi}=\theta_\Phi^*\theta_B=((I_{\ell^2(\mathbb{J})}\otimes L_{\Psi, \Phi})\theta_\Psi)^*\theta_AR_{A,B}=\theta_\Psi^*(I_{\ell^2(\mathbb{J})}\otimes L^*_{\Psi, \Phi})\theta_AR_{A,B} $, $P_{B,\Phi}=\theta_BS_{B,\Phi}^{-1}\theta_\Phi^*=(\theta_AR_{A,B})(\theta_\Psi^*(I_{\ell^2(\mathbb{J})}\otimes L^*_{\Psi, \Phi})\theta_AR_{A,B})^{-1}((I_{\ell^2(\mathbb{J})}\otimes L_{\Psi, \Phi})\theta_\Psi)^*=\theta_A(\theta_\Psi^*(I_{\ell^2(\mathbb{J})}\otimes L^*_{\Psi, \Phi})\theta_A)^{-1}\theta_\Psi^*(I_{\ell^2(\mathbb{J})}\otimes L^*_{\Psi, \Phi}) .$
\end{proof}
\begin{remark}
 RL-similarity looks very innocent here, but it gives a characterization result (like Theorem \ref{RIGHTSIMILARITY})	 in Banach spaces (Theorem \ref{RIGHTLEFTSIMILARITYPOVF}), where right-similarity is innocent.
\end{remark}
\begin{proposition}
Let $ \{A_j\}_{j\in \mathbb{J}}\in \mathscr{F}_\Psi,$ $ \{B_j\}_{j\in \mathbb{J}}\in \mathscr{F}_\Phi$ and   $B_j=L_{A,B}A_j, \Phi_j=\Psi_jR_{\Psi, \Phi},  \forall j \in \mathbb{J}$, for some invertible $ L_{A,B} ,R_{\Psi, \Phi} $ in $ \mathcal{B}(\mathcal{H}_0),\mathcal{B}(\mathcal{H}) $, respectively. Then 
$ \theta_B=(I_{\ell^2(\mathbb{J})}\otimes L_{A,B})\theta_A , \theta_\Phi=  \theta_\Psi R_{\Psi,\Phi} , S_{B,\Phi}=R_{\Psi,\Phi}^*\theta_\Psi^*(I_{\ell^2(\mathbb{J})}\otimes L_{A,B})\theta_A ,  P_{B,\Phi}=(I_{\ell^2(\mathbb{J})}\otimes L_{A,B})\theta_A(\theta_\Psi^*(I_{\ell^2(\mathbb{J})}\otimes L_{A,B})\theta_A)^{-1}\theta_\Psi^*.$
\end{proposition}
\begin{proof}
 $\theta_B=\sum_{j\in\mathbb{J}}L_jB_j=\sum_{j\in\mathbb{J}}L_jL_{A,B}A_j=(I_{\ell^2(\mathbb{J})}\otimes L_{A,B})\theta_A $, $\theta_\Phi=\sum_{j\in\mathbb{J}}L_j\Phi_j=\sum_{j\in\mathbb{J}}L_j\Psi_jR_{\Psi,\Phi}=\theta_\Psi R_{\Psi,\Phi}  $, $S_{B,\Phi}=\theta_\Phi^*\theta_B=(\theta_\Psi R_{\Psi,\Phi})^*((I_{\ell^2(\mathbb{J})}\otimes L_{A,B})\theta_A) =R_{\Psi,\Phi}^*\theta_\Psi^*(I_{\ell^2(\mathbb{J})}\otimes L_{A,B})\theta_A $, $P_{B,\Phi}=\theta_BS_{B,\Phi}^{-1}\theta_\Phi^*=((I_{\ell^2(\mathbb{J})}\otimes L_{A,B})\theta_A)(R_{\Psi,\Phi}^*\theta_\Psi^*(I_{\ell^2(\mathbb{J})}\otimes L_{A,B})\theta_A)^{-1}(\theta_\Psi R_{\Psi,\Phi})^* =(I_{\ell^2(\mathbb{J})}\otimes L_{A,B})\theta_A(\theta_\Psi^*(I_{\ell^2(\mathbb{J})}\otimes L_{A,B})\theta_A)^{-1}\theta_\Psi^* .$
\end{proof}
\textbf{Composition of frames}: Let $ \{A_j\}_{j \in \mathbb{J}} $  be an (ovf) w.r.t. $ \{\Psi_j\}_{j \in \mathbb{J}} $  in  $ \mathcal{B}(\mathcal{H}, \mathcal{H}_0),$ and $ \{B_l\}_{l \in \mathbb{L}} $  be an (ovf) w.r.t. $ \{\Phi_l\}_{l \in \mathbb{L}} $  in  $ \mathcal{B}(\mathcal{H}_0, \mathcal{H}_1).$ Suppose  $\{C_{(l, j)}\coloneqq B_lA_j\}_{(l, j)\in \mathbb{L}\bigtimes  \mathbb{J}} $ is  an (ovf) w.r.t. $\{\Xi_{(l, j)}\coloneqq \Phi_l\Psi_j\}_{(l, j)\in \mathbb{L}\bigtimes  \mathbb{J}} $ in $ \mathcal{B}(\mathcal{H}, \mathcal{H}_1)$. Then the frame $( \{C_{(l, j)}\}_{(l, j)\in \mathbb{L}\bigtimes  \mathbb{J}}, \{\Xi_{(l, j)}\}_{(l, j)\in \mathbb{L}\bigtimes  \mathbb{J}} )$ is called  as composition of frames $ (\{A_j\}_{j \in \mathbb{J}}, \{\Psi_j\}_{j\in \mathbb{J}})$  and $ (\{B_l\}_{l \in \mathbb{L}},  \{\Phi_l\}_{l\in \mathbb{L}}).$ 
\begin{proposition}
Suppose $(\{C_{(l, j)}\coloneqq B_lA_j\}_{(l, j)\in \mathbb{L}\bigtimes  \mathbb{J}} ,\{\Xi_{(l, j)}\coloneqq \Phi_l\Psi_j\}_{(l, j)\in \mathbb{L}\bigtimes  \mathbb{J}} )$ is the composition of  operator-valued frames  $ (\{A_j\}_{j \in \mathbb{J}}, \{\Psi_j\}_{j \in \mathbb{J}}) $  in  $ \mathcal{B}(\mathcal{H}, \mathcal{H}_0),$ and $ (\{B_l\}_{l \in \mathbb{L}}, \{\Phi_l\}_{l \in \mathbb{L}} )$ in $ \mathcal{B}(\mathcal{H}_0, \mathcal{H}_1).$ Then 
\begin{enumerate}[\upshape(i)]
\item $ \theta_C=(I_{\ell^2(\mathbb{J})}\otimes\theta_B)\theta_A, \theta_\Xi=(I_{\ell^2(\mathbb{J})}\otimes\theta_\Phi)\theta_\Psi,  S_{C,\Xi}=\theta_\Psi^*(I_{\ell^2(\mathbb{J})}\otimes S_{B, \Phi})\theta_A,  P_{C, \Xi}=(I_{\ell^2(\mathbb{J})}\otimes\theta_B)\theta_A (\theta_\Psi^*(I_{\ell^2(\mathbb{J})}\otimes S_{B, \Phi})\theta_A)^{-1}\theta_\Psi^*(I_{\ell^2(\mathbb{J})}\otimes\theta_\Phi^*).$ If $ (\{A_j\}_{j \in \mathbb{J}}, \{\Psi_j\}_{j \in \mathbb{J}}) $ and $ (\{B_l\}_{l \in \mathbb{J}}, \{\Phi_l\}_{l \in \mathbb{J}}) $  are Parseval frames, then $(\{C_{(l, j)}\}_{(l, j)\in \mathbb{J}\bigtimes  \mathbb{J}} ,\{\Xi_{(l,j)}\}_{(l,j)\in \mathbb{J}\bigtimes \mathbb{J}})$ is Parseval frame.
\item If $P_{A, \Psi}$ commutes with  $I_{\ell^2(\mathbb{J})}\otimes S_{B, \Phi} $, then $ P_{C, \Xi}=(I_{\ell^2(\mathbb{J})}\otimes\theta_B)P_{A,\Psi} (I_{\ell^2(\mathbb{J})}\otimes S_{B, \Phi}^{-1})P_{A,\Psi}(I_{\ell^2(\mathbb{J})}\otimes\theta_\Phi^*) $.
\end{enumerate}
\end{proposition}
\begin{proof}
\begin{enumerate}[\upshape(i)]
\item Let $ \{e_j\}_{j\in\mathbb{J}}$ (resp. $ \{f_l\}_{l\in\mathbb{L}}$) be the standard orthonormal basis for  $ \ell^2(\mathbb{J})$ (resp. $ \ell^2(\mathbb{L})$).  For $ h \in \mathcal{H},$ we have 
\begin{align*}
 (I_{\ell^2(\mathbb{J})}\otimes\theta_B)\theta_Ah&=\sum\limits_{j\in \mathbb{J}}(I_{\ell^2(\mathbb{J})}\otimes\theta_B)(e_j\otimes A_jh)
 = \sum\limits_{j\in \mathbb{J}}(e_j\otimes\theta_B(A_jh))\\
 &=\sum\limits_{j\in \mathbb{J}}e_j\otimes \left(\sum \limits_{l\in\mathbb{L}}(f_l\otimes B_l(A_jh))\right)
 =\sum\limits_{j\in \mathbb{J}} \sum \limits_{l\in\mathbb{L}}(e_j\otimes f_l\otimes C_{(l,j)}h)\\
 &=\sum\limits_{(l,j)\in\mathbb{L}\bigtimes  \mathbb{J} }((e_j\otimes f_l)\otimes C_{(l,j)}h)=\sum\limits_{(l,j)\in\mathbb{L}\bigtimes  \mathbb{J} }L_{(l,j)} C_{(l,j)}h=\theta_Ch.
\end{align*}
Similarly $\theta_\Xi=(I_{\ell^2(\mathbb{J})}\otimes\theta_\Phi)\theta_\Psi.$ Thus $S_{C,\Xi}=\theta_\Xi^*\theta_C = \theta_\Psi^*(I_{\ell^2(\mathbb{J})}\otimes \theta_\Phi^*)(I_{\ell^2(\mathbb{J})}\otimes \theta_B)\theta_A= \theta_\Psi^*(I_{\ell^2(\mathbb{J})}\otimes \theta_\Phi^*\theta_B)\theta_A=\theta_\Psi^*(I_{\ell^2(\mathbb{J})}\otimes S_{B, \Phi})\theta_A$ which implies $ P_{C, \Xi}=\theta_CS_{C,\Xi}^{-1}\theta_\Xi^*=(I_{\ell^2(\mathbb{J})}\otimes\theta_B)\theta_A (\theta_\Psi^*(I_{\ell^2(\mathbb{J})}\otimes S_{B, \Phi})\theta_A)^{-1}\theta_\Psi^*(I_{\ell^2(\mathbb{J})}\otimes\theta_\Phi^*) $.  Whenever  $ S_{A, \Psi}$ and  $S_{B, \Phi}$  are identity,  the operator  $\theta_\Psi^*(I_{\ell^2(\mathbb{J})}\otimes S_{B, \Phi})\theta_A $ becomes  identity, hence $(\{C_{(l, j)}\}_{(l, j)\in \mathbb{J}\bigtimes  \mathbb{J}},\{\Xi_{(l, j)} \}_{(l, j)\in \mathbb{J}\bigtimes  \mathbb{J}})$ is Parseval.
\item Let  $P_{A, \Psi}$,   $I_{\ell^2(\mathbb{J})}\otimes S_{B, \Phi} $ commute with each other. Then $ (\theta_\Psi^*(I_{\ell^2(\mathbb{J})}\otimes S_{B, \Phi})\theta_A)  (S_{A,\Psi}^{-1}\theta_\Psi^*(I_{\ell^2(\mathbb{J})}\otimes S_{B, \Phi}^{-1})\theta_AS_{A,\Psi}^{-1})=\theta_\Psi^*(I_{\ell^2(\mathbb{J})}\otimes S_{B, \Phi})P_{A, \Psi}(I_{\ell^2(\mathbb{J})}\otimes S_{B, \Phi}^{-1})\theta_AS_{A,\Psi}^{-1}=\theta_\Psi^*P_{A,\Psi}(I_{\ell^2(\mathbb{J})}\otimes S_{B, \Phi}) (I_{\ell^2(\mathbb{J})}\otimes S_{B, \Phi}^{-1})\theta_AS_{A,\Psi}^{-1}=\theta_\Psi^*\theta_AS_{A, \Psi}^{-1}=I_\mathcal{H}$ and $( S_{A,\Psi}^{-1}\theta_\Psi^*(I_{\ell^2(\mathbb{J})}\otimes S_{B, \Phi}^{-1})\theta_AS_{A,\Psi}^{-1})(\theta_\Psi^*(I_{\ell^2(\mathbb{J})}\otimes S_{B, \Phi})\theta_A)=S_{A,\Psi}^{-1}\theta_\Psi^*(I_{\ell^2(\mathbb{J})}\otimes S_{B, \Phi}^{-1})P_{A,\Psi}(I_{\ell^2(\mathbb{J})}\otimes S_{B, \Phi})\theta_A=S_{A,\Psi}^{-1}\theta_\Psi^*(I_{\ell^2(\mathbb{J})}\otimes S_{B, \Phi}^{-1})(I_{\ell^2(\mathbb{J})}\otimes S_{B, \Phi})P_{A,\Psi}\theta_A=S_{A,\Psi}^{-1}\theta_\Psi^*\theta_A
 =I_\mathcal{H}.$ Thus $ S_{C, \Xi}^{-1}=S_{A,\Psi}^{-1}\theta_\Psi^*(I_{\ell^2(\mathbb{J})}\otimes S_{B, \Phi}^{-1})\theta_AS_{A,\Psi}^{-1}.$  So $ P_{C, \Xi}= (I_{\ell^2(\mathbb{J})}\otimes\theta_B)\theta_AS_{A,\Psi}^{-1}\theta_\Psi^*(I_{\ell^2(\mathbb{J})}\otimes S_{B, \Phi}^{-1})\theta_AS_{A,\Psi}^{-1}\theta_\Psi^* (I_{\ell^2(\mathbb{J})}\otimes\theta_\Phi^*)=(I_{\ell^2(\mathbb{J})}\otimes\theta_B)P_{A,\Psi} (I_{\ell^2(\mathbb{J})}\otimes S_{B, \Phi}^{-1})P_{A,\Psi}(I_{\ell^2(\mathbb{J})}\otimes\theta_\Phi^*).$
\end{enumerate}
\end{proof}

\textbf{Tensor product  of frames}: Let $ \{A_j\}_{j \in \mathbb{J}} $  be an (ovf) w.r.t. $ \{\Psi_j\}_{j \in \mathbb{J}} $  in  $ \mathcal{B}(\mathcal{H}, \mathcal{H}_0),$ and $ \{B_l\}_{l \in \mathbb{L}} $  be an (ovf) w.r.t. $ \{\Phi_l\}_{l \in \mathbb{L}} $  in  $ \mathcal{B}(\mathcal{H}_1, \mathcal{H}_2).$ The (ovf)  $(\{C_{(j, l)}\coloneqq A_j\otimes B_l\}_{(j, l)\in \mathbb{J}\bigtimes  \mathbb{L}},\{\Xi_{(j, l)}\coloneqq \Psi_j\otimes\Phi_l\}_{(j, l)\in \mathbb{J}\bigtimes  \mathbb{L}} )$ in $ \mathcal{B}(\mathcal{H}\otimes\mathcal{H}_1, \mathcal{H}_0\otimes\mathcal{H}_2)$ is called  as tensor product  of frames $( \{A_j\}_{j \in \mathbb{J}}, \{\Psi_j\}_{j\in \mathbb{J}})$ and $( \{B_l\}_{l \in \mathbb{L}},  \{\Phi_l\}_{l\in \mathbb{L}}).$
\begin{proposition}
Let $(\{C_{(j, l)}\coloneqq A_j\otimes B_l\}_{(j, l)\in \mathbb{J}\bigtimes  \mathbb{L}},\{\Xi _{(j, l)}\coloneqq \Psi_j\otimes \Phi_l\}_{(j, l)\in \mathbb{J}\bigtimes  \mathbb{L}}) $ be the  tensor product of  operator-valued frames  $( \{A_j\}_{j \in \mathbb{J}}, \{\Psi_j\}_{j \in \mathbb{J}}) $  in  $ \mathcal{B}(\mathcal{H},\mathcal{H}_0),$ and $( \{B_l\}_{l \in \mathbb{L}}, \{\Phi_l\}_{l \in \mathbb{L}} )$ in $ \mathcal{B}(\mathcal{H}_1, \mathcal{H}_2).$ Then  $\theta_C=\theta_A\otimes\theta_B, \theta_\Xi=\theta_\Psi\otimes\theta_\Phi, S_{C, \Xi}=S_{A, \Psi}\otimes S_{B, \Phi}, P_{C, \Xi}=P_{A, \Psi}\otimes P_{B, \Phi}.$ If  $( \{A_j\}_{j \in \mathbb{J}}, \{\Psi_j\}_{j \in \mathbb{J}}) $ and $ (\{B_l\}_{l \in \mathbb{L}},  \{\Phi_l\}_{l \in \mathbb{L}} )$ are Parseval, then $(\{C_{(j, l)}\}_{(j, l)\in \mathbb{J}\bigtimes  \mathbb{L}} ,\{\Xi_{(j,l)}\}_{(j,l)\in \mathbb{J}\bigtimes \mathbb{L}})$ is Parseval.
\end{proposition}
\begin{proof}
Since the operators  are linear, we verify the equalities at  elementary tensors. Let $ \{e_j\}_{j\in\mathbb{J}}$ (resp. $ \{f_l\}_{l\in\mathbb{L}}$) be the standard orthonormal basis for  $ \ell^2(\mathbb{J})$ (resp. $ \ell^2(\mathbb{L})$). For $ h \in \mathcal{H}, g \in \mathcal{H}_1,$ we have 
\begin{align*}
(\theta_A\otimes \theta_B)(h\otimes y)&=\theta_Ah\otimes\theta_Bg=\left(\sum\limits_{j\in \mathbb{J}}(e_j\otimes A_jh)\right)\otimes\left(\sum\limits_{l\in \mathbb{L}}(f_l\otimes B_lg)\right)\\
&=\sum\limits_{j\in \mathbb{J}}\sum\limits_{l\in \mathbb{L}}(e_j\otimes A_jh\otimes f_l\otimes B_lg)=\sum\limits_{(j,l)\in \mathbb{J}\bigtimes \mathbb{L}}(e_j\otimes f_l\otimes A_jh\otimes B_lg)\\
&=\sum\limits_{(j,l)\in \mathbb{J}\bigtimes \mathbb{L}}((e_j\otimes f_l)\otimes(A_j\otimes B_l)(h\otimes g))=\sum\limits_{(j,l)\in \mathbb{J}\bigtimes \mathbb{L}}L_{(j,l)}C_{(j,l)}(h\otimes g)=\theta_C(h\otimes g).
\end{align*}
 Similarly $\theta_\Xi=\theta_\Psi\otimes\theta_\Phi$. Further, $ S_{C, \Xi}=\theta_\Xi^*\theta_C = (\theta_\Psi^*\otimes\theta_\Phi^*)(\theta_A\otimes\theta_B)=S_{A, \Psi}\otimes S_{B, \Phi}$ and $ P_{C, \Xi}= \theta_CS_{C,\Xi}^{-1}\theta_\Xi^*=(\theta_A\otimes\theta_B)(S_{A, \Psi}^{-1}\otimes S_{B, \Phi}^{-1})(\theta_\Psi^*\otimes\theta_\Phi^*)=\theta_AS_{A,\Psi}^{-1}\theta_\Psi^*\otimes\theta_BS_{B,\Phi}^{-1}\theta_\Phi^* =P_{A, \Psi}\otimes P_{B, \Phi}.$ Last part of the statement now  follows directly.
\end{proof}

\section{Frames and discrete group representations} \label{FRAMESANDDISCRETEGROUPREPRESENTATIONS}
Let $ G$ be a discrete group (a topological group with discrete topology), $ \{\chi_g\}_{g\in G}$ be the  standard orthonormal  basis for $\ell^2(G) $.  Let $\lambda $ be the left regular representation of $ G$ defined by $ \lambda_g\chi_q(r)=\chi_q(g^{-1}r), \forall  g, q, r \in G$;  $\rho $ be the right regular representation of $ G$ defined by $ \rho_g\chi_q(r)=\chi_q(rg), \forall g, q, r \in G.$ We denote the von Neumann algebra generated by unitaries $\{\lambda_g\}_{g\in G} $ in $ \mathcal{B}(\ell^2(G))$ by $\mathscr{L}(G) $. Similarly $\mathscr{R}(G) $ denotes the von Neumann algebra generated by $\{\rho_g\}_{g\in G} $ in $ \mathcal{B}(\ell^2(G))$. We  recall  $\mathscr{L}(G)'=\mathscr{R}(G)$, $ \mathscr{R}(G)'=\mathscr{L}(G). $

\begin{definition}
Let $ \pi$ be a unitary representation of a discrete 
group $ G$ on  a Hilbert space $ \mathcal{H}.$ An operator $ A$ in $ \mathcal{B}(\mathcal{H}, \mathcal{H}_0)$ is called an  operator frame generator (resp. a  Parseval frame generator) w.r.t. an operator $ \Psi$ in $ \mathcal{B}(\mathcal{H}, \mathcal{H}_0)$ if $(\{A_g\coloneqq A \pi_{g^{-1}}\}_{g\in G}, \{\Psi_g\coloneqq \Psi \pi_{g^{-1}}\}_{g\in G})$ is an (ovf) in $ \mathcal{B}(\mathcal{H}, \mathcal{H}_0)$. In this case, we write $ (A,\Psi)$ is an operator  frame generator for $\pi$.
\end{definition}
We note that whenever $ A$ is a  generator w.r.t.  $ \Psi$, $ \Psi$ is a   generator w.r.t. $ A.$

\begin{proposition}\label{REPRESENATIONLEMMA}
Let $ (A,\Psi)$ and $ (B,\Phi)$ be   operator frame generators    in $\mathcal{B}(\mathcal{H},  \mathcal{H}_0)$ for a unitary representation $ \pi$ of  $G$ on $ \mathcal{H}.$ Then
\begin{enumerate}[\upshape(i)]
\item $ \theta_A\pi_g=(\lambda_g\otimes I_{\mathcal{H}_0})\theta_A,  \theta_\Psi \pi_g=(\lambda_g\otimes I_{\mathcal{H}_0})\theta_\Psi,  \forall g \in G.$
\item $ \theta_A^*\theta_B,   \theta_\Psi^*\theta_\Phi,\theta_A^*\theta_\Phi$ are in the commutant $ \pi(G)'$ of $ \pi(G)''.$ Further, $ S_{A,\Psi} \in \pi(G)'$ and $(AS_{A, \Psi}^{-{1/2}} , \Psi S_{A, \Psi}^{-{1/2}})$ is a Parseval frame generator. 
\item $ \theta_AT\theta_\Psi^*, \theta_AT\theta_B^*, \theta_\Psi T\theta_\Phi^* \in \mathscr{R}(G)\otimes \mathcal{B}(\mathcal{H}_0), \forall T \in \pi(G)'.$ In particular, $ P_{A, \Psi} \in \mathscr{R}(G)\otimes \mathcal{B}(\mathcal{H}_0). $
\end{enumerate}
\end{proposition}
\begin{proof} Let $ g,p,q, \in G $ and $ h \in \mathcal{H}_0.$
\begin{enumerate}[\upshape(i)]
\item  From the definition of $ \lambda_g $ and $ \chi_q$,  we get $ \lambda_g\chi_q=\chi_{gq}.$ Therefore $ L_{gq}h=\chi_{gq}\otimes h= \lambda_g\chi_q\otimes h= (\lambda_g\otimes I_{\mathcal{H}_0})(\chi_q\otimes h)=(\lambda_g\otimes I_{\mathcal{H}_0})L_qh.$  Using this, 
\begin{align*}
\theta_A\pi_g&=\sum\limits_{p\in G} L_pA_p\pi_g=\sum\limits_{p\in G} L_pA\pi_{p^{-1}}\pi_g=\sum\limits_{p\in G} L_pA\pi_{{p^{-1}}g}\\
&=\sum\limits_{q\in G} L_{gq}A\pi_{q^{-1}}=\sum\limits_{q\in G}(\lambda_g\otimes I_{\mathcal{H}_0}) L_{q}A\pi_{q^{-1}}=(\lambda_g\otimes I_{\mathcal{H}_0})\theta_A. 
\end{align*}
Similarly $ \theta_\Psi \pi_g=(\lambda_g\otimes I_{\mathcal{H}_0})\theta_\Psi.$
\item $ \theta_A^*\theta_B\pi_g=\theta_A^* (\lambda_g\otimes I_{\mathcal{H}_0})\theta_B=((\lambda_{g^{-1}}\otimes I_{\mathcal{H}_0})\theta_A)^*\theta_B=(\theta_A\pi_{g^{-1}})^*\theta_B=\pi_g\theta_A^*\theta_B.$ In the same way, $ \theta_\Psi^*\theta_\Phi, \theta_A^*\theta_\Phi\in \pi(G)'.$ By taking $ B=A$ and $ \Phi=\Psi$ we get  $ S_{A,\Psi} \in \pi(G)'.$ Since $ \pi(G)'$ is a C*-algebra (in fact, a von Neumann algebra) and $ S_{A, \Psi}$ is positive invertible, we have $S_{A, \Psi}^{-1/2} \in \pi(G)'.$ This gives 
\begin{align*}
\sum\limits_{g\in G} (\Psi S_{A,\Psi}^{-\frac{1}{2}}\pi_{g^{-1}})^*(A S_{A,\Psi}^{-\frac{1}{2}}\pi_{g^{-1}})&=\sum\limits_{g\in G}\pi_g S_{A,\Psi}^{-\frac{1}{2}}\Psi AS_{A,\Psi}^{-\frac{1}{2}}\pi_{g^{-1}}=\sum\limits_{g\in G} S_{A,\Psi}^{-\frac{1}{2}}\pi_g\Psi A\pi_{g^{-1}}S_{A,\Psi}^{-\frac{1}{2}}\\
&=S_{A,\Psi}^{-\frac{1}{2}}\left(\sum\limits_{g\in G}(\Psi\pi_{g^{-1}})^*(A\pi_{g^{-1}})\right)S_{A,\Psi}^{-\frac{1}{2}}=S_{A,\Psi}^{-\frac{1}{2}}\left(\sum\limits_{g\in G}\Psi_g^*A_g\right)S_{A,\Psi}^{-\frac{1}{2}}\\
&=S_{A,\Psi}^{-\frac{1}{2}}S_{A,\Psi}S_{A,\Psi}^{-\frac{1}{2}}=I_\mathcal{H},
\end{align*}
 and  hence the last part.
\item Let  $ T \in \pi(G)'.$ Then 
$$\theta_AT\theta_\Psi^*(\lambda_g\otimes I_{\mathcal{H}_0})= \theta_AT((\lambda_{g^{-1}}\otimes I_{\mathcal{H}_0})\theta_\Psi)^*=\theta_AT\pi_g\theta_\Psi^*=\theta_A\pi_gT\theta_\Psi^*=(\lambda_g\otimes I_{\mathcal{H}_0})\theta_AT\theta_\Psi^*.$$
Now, from the construction of $ \mathscr{L}(G),$ we get $\theta_AT\theta_\Psi^* \in (\mathscr{L}(G)\otimes \{I_{\mathcal{H}_0}\})'=\mathscr{L}(G)'\otimes \{I_{\mathcal{H}_0}\}'=\mathscr{R}(G)\otimes \mathcal{B}(\mathcal{H}_0).$ In the same fashion $\theta_AT\theta_B^*, \theta_\Psi S\theta_\Phi^* \in \mathscr{R}(G)\otimes \mathcal{B}(\mathcal{H}_0), \forall  S \in \pi (G)'.$ For the choice  $ T=S_{A,\Psi}^{-1}$ we get $ P_{A, \Psi} \in \mathscr{R}(G)\otimes \mathcal{B}(\mathcal{H}_0). $
\end{enumerate}
\end{proof}

\begin{theorem}\label{gc1}
Let $ G$ be a discrete group with identity $ e$ and $( \{A_g\}_{g\in G},  \{\Psi_g\}_{g\in G})$ be a Parseval  (ovf) in $ \mathcal{B}(\mathcal{H},\mathcal{H}_0).$ Then there is a  unitary representation $ \pi$  of $ G$ on  $ \mathcal{H}$  for which 
$$ A_g=A_e\pi_{g^{-1}}, ~\Psi_g=\Psi_e\pi_{g^{-1}}, ~\forall  g \in G$$
 if and only if 
$$A_{gp}A_{gq}^*=A_pA_q^* ,~ A_{gp}\Psi_{gq}^*=A_p\Psi_q^*,~ \Psi_{gp}\Psi_{gq}^*=\Psi_p\Psi_q^*, ~ \forall g,p,q \in G.$$
\end{theorem} 
\begin{proof}
 For `only if' part, we have
 $$A_{gp}\Psi_{gq}^*= A_e \pi_{(gp)^{-1}}(\Psi_e\pi_{(gq)^{-1}})^*=A_e\pi_{p^{-1}}\pi_{g^{-1}}\pi_g\pi_q\Psi_e^*=A_p\Psi_q^*, \quad\forall g,p,q \in G$$
 and similarly the others can be shown. We now prove the `if part'. Using  assumptions, we use the following three equalities in the proof, among them  we derive the second, remainings are similar.
 For all $ g \in G,$
 \begin{align*}
 (\lambda_g\otimes I_{\mathcal{H}_0})\theta_A\theta_A^*=\theta_A\theta_A^*(\lambda_g\otimes I_{\mathcal{H}_0}), ~ (\lambda_g\otimes I_{\mathcal{H}_0})\theta_A\theta_\Psi^*=\theta_A\theta_\Psi^*(\lambda_g\otimes I_{\mathcal{H}_0}),\\
 (\lambda_g\otimes I_{\mathcal{H}_0})\theta_\Psi\theta_\Psi^*=\theta_\Psi\theta_\Psi^*(\lambda_g\otimes I_{\mathcal{H}_0}).
 \end{align*}
 Noticing $ \lambda_g$ is unitary, we get  $(\lambda_g\otimes I_{\mathcal{H}_0})^{-1}=(\lambda_g\otimes I_{\mathcal{H}_0})^*$; also we observed in the proof of Proposition \ref{REPRESENATIONLEMMA} that  $(\lambda_g\otimes I_{\mathcal{H}_0})L_q=L_{gq}.$ So
 \begin{align*}
 (\lambda_g\otimes I_{\mathcal{H}_0})\theta_A\theta_\Psi^*(\lambda_g\otimes I_{\mathcal{H}_0})^*&=\left(\sum\limits_{p\in G}(\lambda_g\otimes I_{\mathcal{H}_0})L_pA_p\right)\left(\sum\limits_{q\in G}(\lambda_g\otimes I_{\mathcal{H}_0})L_q\Psi_q\right)^*\\
 &=\sum\limits_{p\in G} L_{gp}\left(\sum\limits_{q\in G}A_p\Psi_q^*L_{gq}^*\right)
 =\sum\limits_{r\in G} L_r\left(\sum\limits_{s\in G}A_{g^{-1}r}\Psi_{g^{-1}s}^*L_s^*\right)\\
 & =\sum\limits_{r\in G} L_r\left(\sum\limits_{s\in G}A_r\Psi_s^*L_s^*\right)=\theta_A\theta_\Psi^*.
 \end{align*}
 Define $ \pi : G \ni g  \mapsto \pi_g\coloneqq \theta_\Psi^*(\lambda_g\otimes I_{\mathcal{H}_0})\theta_A  \in \mathcal{B}(\mathcal{H}).$ Using the fact that frame is Parseval,  $ \pi_g\pi_h=\theta_\Psi^*(\lambda_g\otimes I_{\mathcal{H}_0})\theta_A \theta_\Psi^*(\lambda_h\otimes I_{\mathcal{H}_0})\theta_A =\theta_\Psi^*\theta_A \theta_\Psi^*(\lambda_g\otimes I_{\mathcal{H}_0}) (\lambda_h\otimes I_{\mathcal{H}_0})\theta_A = \theta_\Psi^*(\lambda_{gh}\otimes I_{\mathcal{H}_0})\theta_A =\pi_{gh}$ for all $g, h \in G,$ and $\pi_g\pi_g^*=\theta_\Psi^*(\lambda_g\otimes I_{\mathcal{H}_0})\theta_A\theta_A^*(\lambda_{g^{-1}}\otimes I_{\mathcal{H}_0})\theta_\Psi=\theta_\Psi^*\theta_A\theta_A^*(\lambda_g\otimes I_{\mathcal{H}_0})(\lambda_{g^{-1}}\otimes I_{\mathcal{H}_0})\theta_\Psi=I_\mathcal{H},  \pi_g^*\pi_g=\theta_A^*(\lambda_{g^{-1}}\otimes I_{\mathcal{H}_0})\theta_\Psi\theta_\Psi^*(\lambda_{g}\otimes I_{\mathcal{H}_0})\theta_A=\theta_A^*(\lambda_{g^{-1}}\otimes I_{\mathcal{H}_0})(\lambda_{g}\otimes I_{\mathcal{H}_0})\theta_\Psi\theta_\Psi^*\theta_A=I_\mathcal{H} $ for all $ g \in G$.  Since $ G $ has the discrete topology, this proves $ \pi$ is a unitary representation. It remains to prove  $ A_g=A_e\pi_{g^{-1}}, \Psi_g=\Psi_e\pi_{g^{-1}}  $ for all $ g \in G$. Indeed,
 $$A_e\pi_{g^{-1}}= L_e^*\theta_A\theta_\Psi^*(\lambda_{g^{-1}}\otimes I_{\mathcal{H}_0})\theta_A=L_e^*(\lambda_{g^{-1}}\otimes I_{\mathcal{H}_0})\theta_A\theta_\Psi^*\theta_A=((\lambda_g\otimes I_{\mathcal{H}_0})L_e)^*\theta_A=L_{ge}^*\theta_A=A_g,$$ 
 and
 $$\Psi_e\pi_{g^{-1}}=L_e^*\theta_\Psi \theta_\Psi^*(\lambda_{g^{-1}}\otimes I_{\mathcal{H}_0})\theta_A=L_e^*(\lambda_{g^{-1}}\otimes I_{\mathcal{H}_0})\theta_\Psi\theta_\Psi^*\theta_A=((\lambda_g\otimes I_{\mathcal{H}_0})L_e)^*\theta_\Psi=L_{ge}^*\theta_\Psi=\Psi_g.$$
\end{proof}
In the direct part of Theorem \ref{gc1},   we can drop `Parseval' since it has not been used in the proof;  same is true in the following corollary.

\begin{corollary}
Let $ G$ be a discrete group with identity $ e$ and $( \{A_g\}_{g\in G},  \{\Psi_g\}_{g\in G})$ be an (ovf) in $ \mathcal{B}(\mathcal{H},\mathcal{H}_0).$ Then there is a  unitary representation $ \pi$  of $ G$ on  $ \mathcal{H}$  for which
\begin{enumerate}[\upshape(i)]
\item  $ A_g=A_eS_{A,\Psi}^{-1}\pi_{g^{-1}}S_{A,\Psi}, \Psi_g=\Psi_e\pi_{g^{-1}}  $ for all $ g \in G$  if and only if $A_{gp}S_{A,\Psi}^{-2}A_{gq}^*=A_pS_{A,\Psi}^{-2}A_q^* , A_{gp}S_{A,\Psi}^{-1}\Psi_{gq}^*=A_pS_{A,\Psi}^{-1}\Psi_q^*, \Psi_{gp}\Psi_{gq}^*=\Psi_p\Psi_q^*$ for all $ g,p,q \in G.$
\item $ A_g=A_eS_{A,\Psi}^{-1/2}\pi_{g^{-1}}S_{A,\Psi}^{1/2}, \Psi_g=\Psi_eS_{A,\Psi}^{-1/2}\pi_{g^{-1}}S_{A,\Psi}^{1/2}  $ for all $ g \in G$  if and only if $A_{gp}S_{A,\Psi}^{-1}A_{gq}^*=A_pS_{A,\Psi}^{-1}A_q^* , A_{gp}S_{A,\Psi}^{-1}\Psi_{gq}^*=A_pS_{A,\Psi}^{-1}\Psi_q^*, \Psi_{gp}S_{A,\Psi}^{-1}\Psi_{gq}^*=\Psi_pS_{A,\Psi}^{-1}\Psi_q^*$ for all $ g,p,q \in G.$
\item  $ A_g=A_e\pi_{g^{-1}}, \Psi_g=\Psi_eS_{A,\Psi}^{-1}\pi_{g^{-1}}S_{A,\Psi}  $ for all $ g \in G$  if and only if $A_{gp}A_{gq}^*=A_pA_q^* , A_{gp}S_{A,\Psi}^{-1}\Psi_{gq}^*=A_pS_{A,\Psi}^{-1}\Psi_q^*, \Psi_{gp}S_{A,\Psi}^{-2}\Psi_{gq}^*=\Psi_pS_{A,\Psi}^{-2}\Psi_q^*$ for all $ g,p,q \in G.$
\end{enumerate}	
\end{corollary}
\begin{proof}
We apply Theorem  \ref{gc1} the Parseval (ovf) 
\begin{enumerate}[\upshape(i)]
\item  $(\{A_gS_{A,\Psi}^{-1}\}_{g\in G}, \{\Psi_g\}_{g\in G})$ to get: there is a  unitary representation $ \pi$  of $ G$ on  $ \mathcal{H}$  for which $ A_gS_{A,\Psi}^{-1}=(A_eS_{A,\Psi}^{-1})\pi_{g^{-1}}, \Psi_g=\Psi_e\pi_{g^{-1}}  $ for all $ g \in G$  if and only if $(A_{gp}S_{A,\Psi}^{-1})(A_{gq}S_{A,\Psi}^{-1})^*=(A_pS_{A,\Psi}^{-1})(A_qS_{A,\Psi}^{-1})^*$, $(A_{gp}S_{A,\Psi}^{-1})\Psi_{gq}^*= (A_pS_{A,\Psi}^{-1})\Psi_q^*$, $ \Psi_{gp}\Psi_{gq}^*=\Psi_p\Psi_q^*$ for all $ g,p,q \in G.$
\item $( \{A_gS_{A,\Psi}^{-1/2}\}_{g\in G}, \{\Psi_gS_{A,\Psi}^{-1/2}\}_{g\in G})$ to get: there is a  unitary representation $ \pi$  of $ G$ on  $ \mathcal{H}$  for which $ A_gS_{A,\Psi}^{-1/2}=(A_eS_{A,\Psi}^{-1/2})\pi_{g^{-1}}, \Psi_gS_{A,\Psi}^{-1/2}=(\Psi_eS_{A,\Psi}^{-1/2})\pi_{g^{-1}}  $ for all $ g \in G$  if and only if $(A_{gp}S_{A,\Psi}^{-1/2})(A_{gq}S_{A,\Psi}^{-1/2})^*=(A_pS_{A,\Psi}^{-1/2})(A_qS_{A,\Psi}^{-1/2})^*,(A_{gp}S_{A,\Psi}^{-1/2})(\Psi_{gq}S_{A,\Psi}^{-1/2})^*= (A_pS_{A,\Psi}^{-1/2})(\Psi_qS_{A,\Psi}^{-1/2})^*,$  $
(\Psi_{gp}S_{A,\Psi}^{-1/2})(\Psi_{gq}S_{A,\Psi}^{-1/2})^*=(\Psi_pS_{A,\Psi}^{-1/2})(\Psi_qS_{A,\Psi}^{-1/2})^*$ for all $ g,p,q \in G.$
\item $( \{A_g\}_{g\in G}, \{\Psi_gS_{A,\Psi}^{-1}\}_{g\in G})$ to get: there is a  unitary representation $ \pi$  of $ G$ on  $ \mathcal{H}$  for which $ A_g=A_e\pi_{g^{-1}}, \Psi_gS_{A,\Psi}^{-1}=(\Psi_eS_{A,\Psi}^{-1})\pi_{g^{-1}}  $ for all $ g \in G$  if and only if $A_{gp}A_{gq}^*=A_pA_q^* , A_{gp}(\Psi_{gq}S_{A,\Psi}^{-1})^*=A_p(\Psi_qS_{A,\Psi}^{-1})^*, (\Psi_{gp}S_{A,\Psi}^{-1})(\Psi_{gq}S_{A,\Psi}^{-1})^*=(\Psi_pS_{A,\Psi}^{-1})(\Psi_qS_{A,\Psi}^{-1})^*$ for all $ g,p,q \in G.$
\end{enumerate}		
\end{proof}

\section{Frames and group-like unitary systems} \label{FRAMESANDGROUP-LIKEUNITARYSYSTEMS}
\begin{definition}\cite{GABARDOHAN}
A collection   $ \mathcal{U}\subseteq \mathcal{B}(\mathcal{H})$  containing $I_\mathcal{H}$ is called as a unitary system.  If the group generated by  unitary system $ \mathcal{U}$, denoted by $ \operatorname{group}(\mathcal{U})$ is such that 
\begin{enumerate}[\upshape(i)]
\item $\operatorname{group}(\mathcal{U}) \subseteq \mathbb{T}\mathcal{U}\coloneqq \{\alpha U : \alpha \in \mathbb{T}, U\in \mathcal{U}  \}$, and 
\item $\mathcal{U}$ is linearly independent, meaning - $\mathbb{T}U\ne\mathbb{T}V $ whenever $ U, V \in \mathcal{U}$ are such that $ U\ne V,$ 
\end{enumerate}
then $\mathcal{U}$ is called as a group-like unitary system.
\end{definition}

Let $ \mathcal{U}$ be a group-like unitary system. As in \cite{GABARDOHAN}, we  define mappings $ f:\operatorname{group}(\mathcal{U})\rightarrow \mathbb{T}$ and $ \sigma:\operatorname{group}(\mathcal{U})\rightarrow \mathcal{U}$ in the following way. For each  $ U \in  \operatorname{group}(\mathcal{U}) $ there are unique $\alpha\in \mathbb{T}, V \in \mathcal{U} $  such that $ U=\alpha V$. Define $ f(U)=\alpha$ and $\sigma(U)=V $. These  $ f, \sigma $ are well-defined and satisfy $ U=f(U)\sigma(U), \forall U \in \operatorname{group}(\mathcal{U}).$ These mappings are called as \textit{corresponding mappings} associated to $ \mathcal{U}$. We use the following proposition repeatedly.
\begin{proposition}\cite{GABARDOHAN}\label{PER} 
For a group-like unitary system $\mathcal{U}$ and $ f, \sigma $ as above,
\begin{enumerate}[\upshape(i)]
\item $ f(U\sigma(VW))f(VW)=f(\sigma(UV)W)f(UV), \forall U,V,W \in \operatorname{group}(\mathcal{U}).$
\item $ \sigma(U\sigma(VW))=\sigma(\sigma(UV)W), \forall U,V,W \in \operatorname{group} (\mathcal{U}).$
\item $ \sigma(U)=U$ and $ f(U)=1$ for all $ U \in \mathcal{U}.$
\item If $  V, W \in \operatorname{group} (\mathcal{U}),$ then
\begin{align*}
\mathcal{U}&=\{\sigma(UV) : U \in \mathcal{U}\}=\{\sigma(VU^{-1}) : U \in \mathcal{U}\}\\
&=\{\sigma(VU^{-1}W) : U \in \mathcal{U}\}=\{\sigma(V^{-1}U) : U \in \mathcal{U}\}.
\end{align*}
\item For fixed  $  V, W \in \mathcal{U}$, the following mappings are  injective and continuous from  $ \mathcal{U} $ to itself:
$$ U\mapsto \sigma(VU) \quad (\text{resp.} ~ \sigma(UV), \sigma(UV^{-1}), \sigma(V^{-1}U), \sigma(VU^{-1}), \sigma(U^{-1}V), \sigma(VU^{-1}W)).$$
\end{enumerate}
\end{proposition}
Since $\operatorname{group} (\mathcal{U}) $ is a group, we note that, in (iv) of previous proposition,  we can replace $V$ by $V^{-1}$. Hence, whenever  $V \in \operatorname{group} (\mathcal{U})$, we have $\sum_{U \in \mathcal{U}}x_U=\sum_{U \in \mathcal{U}}x_{\sigma(VU)}$.

\begin{definition}\cite{GABARDOHAN}
A unitary representation $ \pi$ of a group-like unitary system $   \mathcal{U}$ on $ \mathcal{H}$ is an injective mapping from $  \mathcal{U}$ into the set of unitary operators on $ \mathcal{H}$ such that 
$$\pi(U)\pi(V)=f(UV)\pi(\sigma(UV)) , \quad {\pi(U)}^{-1}=f(U^{-1})\pi(\sigma(U^{-1})), ~ \forall U,V \in \mathcal{U}, $$
where $ f$ and $ \sigma $  are the corresponding mappings associated with $  \mathcal{U}.$ 
\end{definition}
Since $\pi $ is injective, once we have a unitary representation of a group-like unitary system $   \mathcal{U}$ on $\mathcal{H}$, then $ \pi(\mathcal{U})$ is also a group-like unitary system.

Let $ \mathcal{U}$ be a  group-like unitary system and  $ \{\chi_U\}_{U\in \mathcal{U}}$ be the  standard orthonormal  basis for $\ell^2(\mathcal{U}) $.  We define $\lambda $ on  $ \mathcal{U}$  by $ \lambda_U\chi_V=f(UV)\chi_{\sigma(UV)}, \forall   U,V \in \mathcal{U}.$ Then $ \lambda $ is a unitary  representation which we call as  left  regular representation of $ \mathcal{U}$. Similarly, we define right regular representation of $ \mathcal{U}$ by $ \rho_U\chi_V=f(VU^{-1})\chi_{\sigma(VU^{-1})}, \forall U,V \in \mathcal{U}$ \cite{GABARDOHAN}.  
\begin{definition}
Let $ \mathcal{U}$  be a group-like unitary system.  An operator $ A$ in $ \mathcal{B}(\mathcal{H}, \mathcal{H}_0)$ is called an  operator frame generator (resp. a  Parseval  frame generator) w.r.t. $ \Psi$ in $ \mathcal{B}(\mathcal{H}, \mathcal{H}_0)$  if $(\{A_U\coloneqq A\pi(U)^{-1}\}_{U\in \mathcal{U}},\{\Psi_U\coloneqq \Psi\pi(U)^{-1}\}_{U\in \mathcal{U}})$ is an (ovf)  (resp. a Parseval (ovf)) in $ \mathcal{B}(\mathcal{H}, \mathcal{H}_0)$.  We write $ (A,\Psi)$ is an operator frame generator for $\pi$.
\end{definition}

\begin{theorem}\label{CHARACTERIZATIONGROUPLIKE}
Let $ \mathcal{U}$ be a  group-like unitary system  with identity $ I$ and $(\{A_U\}_{U\in \mathcal{U}},\{\Psi_U\}_{U\in \mathcal{U}})$ be a Parseval (ovf) in $ \mathcal{B}(\mathcal{H},\mathcal{H}_0)$ with $ \theta_A^*$  or $ \theta_\Psi^*$ is injective. Then there is a unitary representation $ \pi$  of $ \mathcal{U}$ on  $ \mathcal{H}$  for which 
$$ A_U=A_I\pi(U)^{-1}, ~\Psi_U=\Psi_I\pi(U)^{-1},~\forall  U \in \mathcal{U}$$
  if and only if 
 \begin{align*}
A_{\sigma(UV)}A_{\sigma(UW)}^*&=f(UV)\overline{f(UW)} A_VA_W^* ,\\
 A_{\sigma(UV)}\Psi_{\sigma(UW)}^*&=f(UV)\overline{f(UW)} A_V\Psi_W^*,\\ \Psi_{\sigma(UV)}\Psi_{\sigma(UW)}^*&=f(UV)\overline{f(UW)} \Psi_V\Psi_W^*
\end{align*}
for all $ U,V,W \in \mathcal{U}.$
\end{theorem} 

\begin{proof}
$(\Rightarrow)$ For all $U,V,W \in \mathcal{U}$, we have
\begin{align*}
A_{\sigma(UV)}A_{\sigma(UW)}^*&= A_I\pi(\sigma(UV))^{-1}( A_I\pi(\sigma(UW))^{-1} )^*\\
&=A_I(\overline{f(UV)}\pi(U)\pi(V))^{-1} \overline{f(UW)}\pi(U)\pi(W)A^*_I\\
&=f(UV)\overline{f(UW)} A_I\pi(V)^{-1}(A_I\pi(W)^{-1})^*\\
&=f(UV)\overline{f(UW)} A_VA_W^*, 
\end{align*}
others can be shown similarly. 

$(\Leftarrow)$  We have to construct unitary representation which satisfies the stated conditions. Following observation plays an important role in this part. Let $ h\in \mathcal{H}.$ Then $ L_{\sigma(UV)}h=\chi_{\sigma(UV)}\otimes h=\overline{f(UV)}\lambda_U\chi_V\otimes h=\overline{f(UV)}(\lambda_U\chi_V\otimes h)=\overline{f(UV)}(\lambda_U\otimes I_{\mathcal{H}_0})(\chi_V\otimes h)=\overline{f(UV)}(\lambda_U\otimes I_{\mathcal{H}_0})L_V h.$

As in the proof of Theorem \ref{gc1},  we argue the following, for which now we prove the first.
For all $ U \in \mathcal{U},$
 \begin{align*}
 (\lambda_U\otimes I_{\mathcal{H}_0})\theta_A\theta_A^*=\theta_A\theta_A^*(\lambda_U\otimes I_{\mathcal{H}_0}), ~ (\lambda_U\otimes I_{\mathcal{H}_0})\theta_A\theta_\Psi^*=\theta_A\theta_\Psi^*(\lambda_U\otimes I_{\mathcal{H}_0}),\\
  (\lambda_U\otimes I_{\mathcal{H}_0})\theta_\Psi\theta_\Psi^*=\theta_\Psi\theta_\Psi^*(\lambda_U\otimes I_{\mathcal{H}_0}).
 \end{align*}
Consider 

\begin{align*}
 (\lambda_U\otimes I_{\mathcal{H}_0})\theta_A\theta_A^*(\lambda_U\otimes I_{\mathcal{H}_0})^*&=\left(\sum\limits_{V\in \mathcal{U}}(\lambda_U\otimes I_{\mathcal{H}_0})L_VA_V\right)\left(\sum\limits_{W\in \mathcal{U}}(\lambda_U\otimes I_{\mathcal{H}_0})L_WA_W\right)^*\\ &=\left(\sum\limits_{V\in \mathcal{U}} f(UV)L_{\sigma(UV)}A_V\right)\left(\sum\limits_{W\in \mathcal{U}}f(UW)L_{\sigma(UW)}A_W\right)^*\\
 &=\sum\limits_{V\in \mathcal{U}} L_{\sigma(UV)}\left(\sum\limits_{W\in \mathcal{U}}f(UV)\overline{f(UW)}A_VA_W^*L_{\sigma(UW)}^*\right)\\
  &= \sum\limits_{V\in \mathcal{U}} L_{\sigma(UV)}\left(\sum\limits_{W\in \mathcal{U}}A_{\sigma(UV)}A_{\sigma(UW)}^*L_{\sigma(UW)}^*\right)\\
  &=\left(\sum\limits_{V\in \mathcal{U}} L_{\sigma(UV)}A_{\sigma(UV)}\right)\left(\sum\limits_{W\in \mathcal{U}}L_{\sigma(UW)}A_{\sigma(UW)}\right)^*=\theta_A\theta_A^*
 \end{align*}
where last part of Proposition \ref{PER} is used in the last equality.

Define $ \pi : \mathcal{U} \ni U  \mapsto \pi(U)\coloneqq \theta_\Psi^*(\lambda_U\otimes I_{\mathcal{H}_0})\theta_A  \in \mathcal{B}(\mathcal{H}).$  Then $ \pi(U)\pi(V)=\theta_\Psi^*(\lambda_U\otimes I_{\mathcal{H}_0})\theta_A \theta_\Psi^*(\lambda_V\otimes I_{\mathcal{H}_0})\theta_A =\theta_\Psi^*\theta_A \theta_\Psi^*(\lambda_U\otimes I_{\mathcal{H}_0}) (\lambda_V\otimes I_{\mathcal{H}_0})\theta_A = \theta_\Psi^*(\lambda_U\lambda_V\otimes I_{\mathcal{H}_0})\theta_A =\theta_\Psi^*(f(UV)\lambda_{\sigma(UV)}\otimes I_{\mathcal{H}_0})\theta_A  =f(UV) \theta_\Psi^*(\lambda_{\sigma(UV)}\otimes I_{\mathcal{H}_0})\theta_A =f(UV)\pi({\sigma(UV)})$ for all $U, V \in \mathcal{U},$ and $\pi(U)\pi(U)^*=\theta_\Psi^*(\lambda_U\otimes I_{\mathcal{H}_0})\theta_A\theta_A^*(\lambda_U^*\otimes I_{\mathcal{H}_0})\theta_\Psi=\theta_\Psi^*\theta_A\theta_A^*(\lambda_U\otimes I_{\mathcal{H}_0})(\lambda_U^*\otimes I_{\mathcal{H}_0})\theta_\Psi=I_\mathcal{H},  \pi(U)^*\pi(U)=\theta_A^*(\lambda_U^*\otimes I_{\mathcal{H}_0})\theta_\Psi\theta_\Psi^*(\lambda_{U}\otimes I_{\mathcal{H}_0})\theta_A=\theta_A^*(\lambda_U^*\otimes I_{\mathcal{H}_0})(\lambda_U\otimes I_{\mathcal{H}_0})\theta_\Psi\theta_\Psi^*\theta_A=I_\mathcal{H} $ for all $ U \in \mathcal{U}$. Further, 
\begin{align*}
\pi(U)f(U^{-1})\pi(\sigma(U^{-1}))&=\theta_\Psi^*(\lambda_U\otimes I_{\mathcal{H}_0})\theta_Af(U^{-1})\theta_\Psi^*(\lambda_{\sigma(U^{-1})}\otimes I_{\mathcal{H}_0})\theta_A \\
&=f(U^{-1})\theta_\Psi^*\theta_A\theta_\Psi^*(\lambda_U\otimes I_{\mathcal{H}_0})(\lambda_{\sigma(U^{-1})}\otimes I_{\mathcal{H}_0})\theta_A \\
&=f(U^{-1})\theta_\Psi^*(\lambda_U\otimes I_{\mathcal{H}_0})(\lambda_{\sigma(U^{-1})}\otimes I_{\mathcal{H}_0})\theta_A\\
&=f(U^{-1})\theta_\Psi^*(\lambda_U\lambda_{\sigma(U^{-1})}\otimes I_{\mathcal{H}_0})\theta_A\\
&=f(U^{-1})\theta_\Psi^*(f(U\sigma(U^{-1}))\lambda_{\sigma(U\sigma (U^{-1}))}\otimes I_{\mathcal{H}_0})\theta_A\\
&=\theta_\Psi^*(f(U\sigma(U^{-1}I))f(U^{-1}I)\lambda_{\sigma(U\sigma (U^{-1}I))}\otimes I_{\mathcal{H}_0})\theta_A\\
&=\theta_\Psi^*(f(\sigma(UU^{-1})I)f(UU^{-1})\lambda_{\sigma({\sigma(UU^{-1})I})}\otimes I_{\mathcal{H}_0})\theta_A\\
&=\theta_\Psi^*(\lambda_I\otimes I_{\mathcal{H}_0})\theta_A=I_\mathcal{H}
\end{align*}
$\Rightarrow {\pi(U)}^{-1}=f(U^{-1})\pi(\sigma(U^{-1}))$ for all $ U \in \mathcal{U}$. We shall now use $ \theta_A^*$ is injective (or even $ \theta_\Psi^*$ is injective)  to show $ \pi$ is injective and thereby to  get $ \pi$ is a unitary representation. Let $ \pi(U)=\pi(V).$ Then $\theta_\Psi^*(\lambda_U\otimes I_{\mathcal{H}_0})\theta_A =\theta_\Psi^*(\lambda_V\otimes I_{\mathcal{H}_0})\theta_A  \Rightarrow \theta_\Psi^*(\lambda_U\otimes I_{\mathcal{H}_0})\theta_A \theta_A^*=\theta_\Psi^*(\lambda_V\otimes I_{\mathcal{H}_0})\theta_A  \theta_A^* \Rightarrow \theta_\Psi^*\theta_A  \theta_A^*(\lambda_U\otimes I_{\mathcal{H}_0}) =\theta_\Psi^*\theta_A  \theta_A^*(\lambda_V\otimes I_{\mathcal{H}_0}) \Rightarrow \lambda_U\otimes I_{\mathcal{H}_0}=\lambda_V\otimes I_{\mathcal{H}_0}.$ We show $ U$ and $ V$ are identical at elementary tensors. For $ h \in \ell^2(\mathcal{U}),  y \in \mathcal{H}_0,  $ we get, $(\lambda_U\otimes I_{\mathcal{H}_0})(h\otimes y)=(\lambda_V\otimes I_{\mathcal{H}_0})(h\otimes y)\Rightarrow \lambda_Uh\otimes y=\lambda_Vh\otimes y \Rightarrow (\lambda_U-\lambda_V)h\otimes y=0 \Rightarrow 0= \langle  (\lambda_U-\lambda_V)h\otimes y, (\lambda_U-\lambda_V)h\otimes y\rangle= \|(\lambda_U-\lambda_V)h\|^2 \|y\|^2 .$ We may assume $y\neq0$ (if  $y=0$, then  $h\otimes y=0$). But then $ (\lambda_U-\lambda_V)(h)=0,$ and $ \lambda$ is a unitary representation (it is injective) gives $ U=V.$ The pending part  $ A_U=A_I\pi(U)^{-1}, \Psi_U=\Psi_I\pi(U)^{-1}  $ for all $ U \in \mathcal{U}$ we show, now.
\begin{align*}
 A_I\pi(U)^{-1}&= L_I^*\theta_A(\theta_\Psi^*(\lambda_U\otimes I_{\mathcal{H}_0})\theta_A)^*
 = L_I^*(\theta_\Psi^*(\lambda_U\otimes I_{\mathcal{H}_0})\theta_A\theta_A^*)^*=L_I^*(\theta_\Psi^*\theta_A\theta_A^*(\lambda_U\otimes I_{\mathcal{H}_0}))^*\\
 & = L_I^*(\theta_A^*(\lambda_U\otimes I_{\mathcal{H}_0}))^*=(\theta_A^*(\lambda_U\otimes I_{\mathcal{H}_0})L_I)^*= (\theta_A^*\overline{f(UI)}(\lambda_U\otimes I_{\mathcal{H}_0})L_I)^*\\
 &= (\theta_A^*L_{\sigma({UI})})^*=L_U^*\theta_A=A_U
\end{align*}
and
 \begin{align*}
  \Psi_I\pi(U)^{-1}&= L_I^*\theta_\Psi(\theta_\Psi^*(\lambda_U\otimes I_{\mathcal{H}_0})\theta_A)^*
  = L_I^*(\theta_\Psi^*(\lambda_U\otimes I_{\mathcal{H}_0})\theta_A\theta_\Psi^*)^*=L_I^*(\theta_\Psi^*\theta_A\theta_\Psi^*(\lambda_U\otimes I_{\mathcal{H}_0}))^*\\
  & = L_I^*(\theta_\Psi^*(\lambda_U\otimes I_{\mathcal{H}_0}))^*=(\theta_\Psi^*(\lambda_U\otimes I_{\mathcal{H}_0})L_I)^*= (\theta_\Psi^*\overline{f(UI)}(\lambda_U\otimes I_{\mathcal{H}_0})L_I)^*\\
  &= (\theta_\Psi^*L_{\sigma({UI})})^*=L_U^*\theta_\Psi=\Psi_U.
 \end{align*}
\end{proof}
Neither Parsevalness of the frame nor $ \theta_A^*$ or $ \theta_\Psi^*$ is injective was used  in the direct part of previous theorem and the corollaries which come out of it.

Since $ \theta_A$ is acting between Hilbert spaces, we know that $ \overline{\theta_A(\mathcal{H})}=\operatorname{Ker}(\theta_A^*)^\perp$ and $ \operatorname{Ker}(\theta_A^*)=\theta_A(\mathcal{H})^\perp.$ From Proposition \ref{2.2}, range of $\theta_A$  is closed. Therefore $ \theta_A(\mathcal{H})=\operatorname{Ker}(\theta_A^*)^\perp.$ Thus the condition $ \theta_A^*$ is injective in the last theorem can be replaced by $ \theta_A$ is onto. Same is true  for $ \theta_\Psi^*.$

\begin{corollary}
Let $ \mathcal{U}$ be a  group-like unitary system  with identity $ I$ and $(\{A_U\}_{U\in \mathcal{U}},\{\Psi_U\}_{U\in \mathcal{U}})$ be an (ovf)  in $ \mathcal{B}(\mathcal{H},\mathcal{H}_0)$  with $ \theta_A^*$  or $ \theta_\Psi^*$ is injective. Then there is a unitary representation $ \pi$  of  $ \mathcal{U}$ on  $ \mathcal{H}$  for which
\begin{enumerate}[\upshape(i)]
\item    $ A_U=A_IS^{-1}_{A,\Psi}\pi(U)^{-1}S_{A, \Psi}, \Psi_U=\Psi_I\pi(U)^{-1}  $ for all $ U \in \mathcal{U}$  if and only if $A_{\sigma(UV)}S^{-2}_{A, \Psi}A_{\sigma(UW)}^*=f(UV)\overline{f(UW)} A_VS^{-2}_{A, \Psi}A_W^*$, $
 A_{\sigma(UV)}S_{A, \Psi}^{-1}\Psi_{\sigma(UW)}^*=f(UV)\overline{f(UW)} A_VS^{-1}_{A, \Psi}\Psi_W^*$, $ \Psi_{\sigma(UV)}\Psi_{\sigma(UW)}^*=f(UV)\overline{f(UW)} \Psi_V\Psi_W^*$ for all $ U,V,W \in \mathcal{U}.$
\item   $ A_U=A_IS_{A,\Psi}^{-1/2}\pi(U)^{-1}S_{A,\Psi}^{1/2}, \Psi_U=\Psi_IS_{A,\Psi}^{-1/2}\pi(U)^{-1}S_{A,\Psi}^{1/2}  $ for all $ U \in \mathcal{U}$  if and only if $A_{\sigma(UV)}S_{A,\Psi}^{-1}A_{\sigma(UW)}^*=
f(UV)\overline{f(UW)} A_VS_{A,\Psi}^{-1}A_W^* $, $
 A_{\sigma(UV)}S_{A,\Psi}^{-1} \Psi_{\sigma(UW)}^*=
f(UV)\overline{f(UW)} A_VS_{A,\Psi}^{-1}\Psi_W^*$, $
 \Psi_{\sigma(UV)}S_{A,\Psi}^{-1}\Psi_{\sigma(UW)}^*=f(UV)\overline{f(UW)} \Psi_VS_{A,\Psi}^{-1}\Psi_W^*$ for all $ U,V,W \in \mathcal{U}.$
\item   $ A_U=A_I\pi(U)^{-1}, \Psi_U =\Psi_IS^{-1}_{A, \Psi}\pi(U)^{-1}S_{A, \Psi}  $ for all $ U \in \mathcal{U}$  if and only if $A_{\sigma(UV)}A_{\sigma(UW)}^*=f(UV)\overline{f(UW)} A_VA_W^*$, $
A_{\sigma(UV)}S^{-1}_{A, \Psi}\Psi_{\sigma(UW)}^*=f(UV)\overline{f(UW)}A_VS^{-1}_{A, \Psi}\Psi_W^*$, $
 \Psi_{\sigma(UV)}S^{-2}_{A, \Psi}\Psi_{\sigma(UW)} ^*
 =f(UV)\overline{f(UW)} \Psi_VS^{-2}_{A, \Psi}\Psi_W^*$ for all $ U,V,W \in \mathcal{U}.$
\end{enumerate}
\end{corollary}
\begin{proof}
We apply Theorem \ref{CHARACTERIZATIONGROUPLIKE} to the  Parseval (ovf)
\begin{enumerate}[\upshape(i)]
\item  $(\{A_US_{A,\Psi}^{-1}\}_{U\in \mathcal{U}} , \{\Psi_U\}_{U\in \mathcal{U}})$,  there is a unitary representation $ \pi$  of  $ \mathcal{U}$ on  $ \mathcal{H}$  for which  $ A_US_{A,\Psi}^{-1}=(A_IS^{-1}_{A,\Psi})\pi(U)^{-1}, \Psi_U=\Psi_I\pi(U)^{-1}  $ for all $ U \in \mathcal{U}$  if and only if $(A_{\sigma(UV)}S^{-1}_{A, \Psi})(A_{\sigma(UW)}S_{A,\Psi}^{-1})^*=f(UV)\overline{f(UW)}( A_VS^{-1}_{A, \Psi})(A_WS_{A,\Psi}^{-1})^* $, $ (A_{\sigma(UV)}S_{A, \Psi}^{-1}) \Psi_{\sigma(UW)}^*=
f(UV)\overline{f(UW)}( A_VS^{-1}_{A, \Psi})\Psi_W^*$, $ \Psi_{\sigma(UV)}\Psi_{\sigma(UW)}^*=f(UV)\overline{f(UW)} \Psi_V\Psi_W^*$ for all $ U,V,W \in \mathcal{U}.$
\item  $(\{A_US_{A,\Psi}^{-1/2}\}_{U\in \mathcal{U}}, \{\Psi_US_{A,\Psi}^{-1/2}\}_{U\in \mathcal{U}})$, there is a unitary representation $ \pi$  of  $ \mathcal{U}$ on  $ \mathcal{H}$  for which  $ A_US_{A,\Psi}^{-1/2}=(A_IS_{A,\Psi}^{-1/2})\pi(U)^{-1}, \Psi_US_{A,\Psi}^{-1/2}=(\Psi_IS_{A,\Psi}^{-1/2})\pi(U)^{-1}  $ for all $ U \in \mathcal{U}$  if and only if  $(A_{\sigma(UV)}S_{A,\Psi}^{-1/2})(A_{\sigma(UW)}S_{A,\Psi}^{-1/2})^*=f(UV)\overline{f(UW)}( A_VS_{A,\Psi}^{-1/2})(A_WS_{A,\Psi}^{-1/2})^*$, $
 (A_{\sigma(UV)}S_{A,\Psi}^{-1/2}) (\Psi_{\sigma(UW)}S_{A,\Psi}^{-1/2})^*=
f(UV)\overline{f(UW)}( A_VS_{A,\Psi}^{-1/2})(\Psi_WS_{A,\Psi}^{-1/2})^*$, $ (\Psi_{\sigma(UV)}S_{A,\Psi}^{-1/2})(\Psi_{\sigma(UW)}S_{A,\Psi}^{-1/2})^*=f(UV)\overline{f(UW)} (\Psi_VS_{A,\Psi}^{-1/2})(\Psi_WS_{A,\Psi}^{-1/2})^*$ for all $ U,V,W \in \mathcal{U}.$
\item  $(\{A_U\}_{U\in \mathcal{U}} , \{\Psi_US_{A,\Psi}^{-1}\}_{U\in \mathcal{U}})$, there is a unitary representation $ \pi$  of  $ \mathcal{U}$ on  $ \mathcal{H}$  for which  $ A_U=A_I\pi(U)^{-1}, \Psi_US^{-1}_{A, \Psi}=(\Psi_IS^{-1}_{A, \Psi})\pi(U)^{-1}  $ for all $ U \in \mathcal{U}$  if and only if $A_{\sigma(UV)}A_{\sigma(UW)}^*=
f(UV)\overline{f(UW)} A_VA_W^*$, $
 A_{\sigma(UV)}(\Psi_{\sigma(UW)}S^{-1}_{A, \Psi})^*=f(UV)\overline{f(UW)}A_V(\Psi_WS^{-1}_{A, \Psi})^*$, $
 (\Psi_{\sigma(UV)}S^{-1}_{A, \Psi})(\Psi_{\sigma(UW)}S^{-1}_{A, \Psi})^*=f(UV)\overline{f(UW)} (\Psi_VS^{-1}_{A, \Psi})(\Psi_WS^{-1}_{A, \Psi})^*$ for all $ U,V,W \in \mathcal{U}.$
\end{enumerate}
\end{proof}

\begin{corollary}\label{ANSWER}
Let $ \mathcal{U}$ be a  group-like unitary system  with identity $ I$ and $ \{A_U\}_{U\in \mathcal{U}}$ be a Parseval operator-valued frame in $ \mathcal{B}(\mathcal{H},\mathcal{H}_0)$ (w.r.t. itself)  with $ \theta_A^*$  is injective. Then there is a unitary representation $ \pi$  of $ \mathcal{U}$ on  $ \mathcal{H}$  for which $ A_U=A_I\pi(U)^{-1}  $ for all $ U \in \mathcal{U}$  if and only if $A_{\sigma(UV)}A_{\sigma(UW)}^*=f(UV)\overline{f(UW)} A_VA_W^* $ for all $ U,V,W \in \mathcal{U}.$
\end{corollary}
\begin{remark}
Corollary \ref{ANSWER} is  an answer to the sentence in \text{\upshape (iii)} of Remark 6.8 in \cite{KAFTALLARSONZHANG1}.
\end{remark}
\section{Perturbations}\label{PERTURBATIONS}
First result on perturbation of a frame for a Hilbert space is due to  Christensen, in 1995, which states
\begin{theorem}\cite{OLE3}
Let $ \{x_n\}_{n=1}^\infty$ be a frame for  $\mathcal{H} $ with bounds $ a$ and $b$. If  $ \{y_n\}_{n=1}^\infty$  in $\mathcal{H} $ satisfies
 $$ c \coloneqq\sum_{n=1}^{\infty}\|x_n-y_n\|^2<a,$$
 then it is  a frame for $\mathcal{H} $  with bounds $a\left(1-\sqrt{\frac{c}{a}}\right)^2 $ and $b\left(1+\sqrt{\frac{c}{b}}\right)^2.$
\end{theorem}
Three months later, Christensen himself  generalized this and derived the following.

\begin{theorem}\cite{OLE2}
Let $ \{x_n\}_{n=1}^\infty$ be a frame for  $\mathcal{H} $ with bounds $ a$ and $b$.  If  $ \{y_n\}_{n=1}^\infty$  in $\mathcal{H} $ is  such that there exist $ \alpha, \gamma \geq0$ with $\alpha+\frac{\gamma}{\sqrt{a}}< 1 $ and
 $$\left\|\sum_{n=1}^{m}c_n(x_n-y_n) \right\|\leq \alpha\left\|\sum_{n=1}^{m}c_nx_n\right \|+\gamma \left(\sum_{n=1}^{m}|c_n|^2\right)^\frac{1}{2},  ~\forall c_1,  \dots, c_m \in \mathbb{K}, m=1, \dots, $$
then it is  a frame for $\mathcal{H} $  with bounds $a\left(1-(\alpha+\frac{\gamma}{\sqrt{a}})\right)^2 $ and $b\left(1+(\alpha+\frac{\gamma}{\sqrt{b}})\right)^2.$
\end{theorem}
Casazza, and Christensen  extended the previous result further in 1997, and obtained the next theorem.

\begin{theorem}\cite{OLECAZASSA}\label{OLECAZASSA}
Let $ \{x_n\}_{n=1}^\infty$ be a frame for  $\mathcal{H} $ with bounds $ a$ and $b$.  If  $ \{y_n\}_{n=1}^\infty$  in $\mathcal{H} $ is  such that there exist $ \alpha, \beta, \gamma \geq0$ with $ \max\{\alpha+\frac{\gamma}{\sqrt{a}}, \beta\}<1$ and
$$\left\|\sum_{n=1}^{m}c_n(x_n-y_n) \right\|\leq \alpha\left\|\sum_{n=1}^{m}c_nx_n\right \|+\gamma \left(\sum_{n=1}^{m}|c_n|^2\right)^\frac{1}{2}+\beta\left\|\sum_{n=1}^{m}c_ny_n\right \|,  ~\forall c_1,  \dots, c_m \in \mathbb{K}, m=1, \dots, $$
 then it is a frame for $\mathcal{H} $  with bounds $a\left(1-\frac{\alpha+\beta+\frac{\gamma}{\sqrt{a}}}{1+\beta}\right)^2 $ and $b\left(1+\frac{\alpha+\beta+\frac{\gamma}{\sqrt{b}}}{1-\beta}\right)^2.$
\end{theorem}
One important  result about the invertibility of an operator (which is an extension of result of Carl Neumann),  used in the derivation of  Theorem \ref{OLECAZASSA} (also due to Casazza, and Christensen) is 
 \begin{theorem}\cite{OLECAZASSA}\label{cc1}
Let $ \mathcal{X}, \mathcal{Y}$ be Banach spaces, $ U : \mathcal{X}\rightarrow \mathcal{Y}$ be a bounded invertible operator. If  a bounded  operator $ V : \mathcal{X}\rightarrow \mathcal{Y}$ is  such that there exist  $ \alpha, \beta \in \left [0, 1  \right )$ with 
$$ \|Ux-Vx\|\leq\alpha\|Ux\|+\beta\|Vx\|,\quad \forall x \in  \mathcal{X},$$
then $ V $ is invertible and 
$$ \frac{1-\alpha}{1+\beta}\|Ux\|\leq\|Vx\|\leq\frac{1+\alpha}{1-\beta} \|Ux\|, \quad\forall x \in  \mathcal{X};$$
$$ \frac{1-\beta}{1+\alpha}\frac{1}{\|U\|}\|y\|\leq\|V^{-1}y\|\leq\frac{1+\beta}{1-\alpha} \|U^{-1}\|\|y\|, \quad\forall y \in  \mathcal{Y}.$$
 \end{theorem}
 Study of perturbation of operator-valued frames is initiated by Sun \cite{SUN2}.
 \begin{theorem}\cite{SUN2}
 Let $ \{A_j\}_{j\in \mathbb{J}}$ be an (ovf) (w.r.t. itself) in $ \mathcal{B}(\mathcal{H}, \mathcal{H}_0)$ with frame bounds $ a$ and $b$. Suppose  $\{B_j\}_{j\in \mathbb{J}} $ in $ \mathcal{B}(\mathcal{H}, \mathcal{H}_0)$ is such that  there exist $\alpha, \beta, \gamma \geq 0  $ with $ \max\{\alpha+\frac{\gamma}{\sqrt{a}}, \beta\}<1$ and one of the following two conditions holds:
 $$\left\|\sum\limits_{j\in\mathbb{S}}(A^*_j-B^*_j)y_j\right\|\leq\alpha\left\|\sum\limits_{j\in\mathbb{S}}A^*_jy_j\right\|+\beta\left\|\sum\limits_{j\in\mathbb{S}}B^*_jy_j\right\|+\gamma\left(\sum\limits_{j\in\mathbb{S}}\|y_j\|^2 \right)^\frac{1}{2},~\forall y_j\in \mathcal{H}_0, \forall j \in \mathbb{S} $$
 for  every finite subset $ \mathbb{S}$ of $ \mathbb{J},$ or
 $$\left(\sum\limits_{j\in\mathbb{J}}\|(A_j-B_j)h\|^2 \right)^\frac{1}{2}\leq \alpha\left(\sum\limits_{j\in\mathbb{J}}\|A_jh\|^2 \right)^\frac{1}{2} +\beta\left(\sum\limits_{j\in\mathbb{J}}\|B_jh\|^2 \right)^\frac{1}{2}+\gamma\|h\|,~\forall h\in \mathcal{H}.$$
 Then $ \{B_j\}_{j\in \mathbb{J}}$ is an (ovf) (w.r.t. itself) with bounds $a\left(1-\frac{\alpha+\beta+\frac{\gamma}{\sqrt{a}}}{1+\beta}\right)^2 $ and $b\left(1+\frac{\alpha+\beta+\frac{\gamma}{\sqrt{b}}}{1-\beta}\right)^2.$ 
 \end{theorem}
 For the extension, we have  following results.
 \begin{theorem}\label{PERTURBATION RESULT 1}
 Let $ (\{A_j\}_{j\in \mathbb{J}}, \{\Psi_j\}_{j\in \mathbb{J}}) $  be  an (ovf) in $ \mathcal{B}(\mathcal{H}, \mathcal{H}_0)$. Suppose  $\{B_j\}_{j\in \mathbb{J}} $ in $ \mathcal{B}(\mathcal{H}, \mathcal{H}_0)$ is such that $ \Psi_j^*B_j\geq 0, \forall j \in \mathbb{J}$ and there exist $\alpha, \beta, \gamma \geq 0  $ with $ \max\{\alpha+\gamma\|\theta_\Psi S_{A,\Psi}^{-1}\|, \beta\}<1$ and for every finite subset $ \mathbb{S}$ of $ \mathbb{J}$
 \begin{equation}\label{p3}
 \left\|\sum\limits_{j\in \mathbb{S}}(A_j^*-B_j^*)L_j^*y\right\|\leq \alpha\left\|\sum\limits_{j\in \mathbb{S}}A_j^*L_j^*y\right\|+\beta\left\|\sum\limits_{j\in \mathbb{S}}B_j^*L_j^*y\right\|+\gamma \left(\sum\limits_{j\in \mathbb{S}}\|L_j^*y\|^2\right)^\frac{1}{2},\quad \forall y \in \ell^2(\mathbb{J})\otimes \mathcal{H}_0.
 \end{equation} 
  Then  $ (\{B_j\}_{j\in \mathbb{J}},\{\Psi_j\}_{j\in \mathbb{J}}) $ is an (ovf) with bounds $ \frac{1-(\alpha+\gamma\|\theta_\Psi S_{A,\Psi}^{-1}\|)}{(1+\beta)\|S_{A,\Psi}^{-1}\|}$ and $\frac{\|\theta_\Psi\|((1+\alpha)\|\theta_A\|+\gamma)}{1-\beta} $.
 \end{theorem}
 \begin{proof}
 For each finite subset $\mathbb{S} $ of $ \mathbb{J}$ and for every $ y$ in $ \ell^2(\mathbb{J})\otimes \mathcal{H}_0$, 
 \begin{align*}
 \left\| \sum\limits_{j\in \mathbb{S}}B_j^*L_j^*y\right\|&\leq \left\| \sum\limits_{j\in \mathbb{S}}(A_j^*-B_j^*)L_j^*y\right\|+\left\| \sum\limits_{j\in \mathbb{S}}A_j^*L_j^*y\right\|\\
  &\leq(1+\alpha)\left\| \sum\limits_{j\in \mathbb{S}}A_j^*L_j^*y\right\|+\beta\left\| \sum\limits_{j\in \mathbb{S}}B_j^*L_j^*y\right\|+\gamma\left( \sum\limits_{j\in \mathbb{S}}\|L_j^*y\|^2\right)^\frac{1}{2}
 \end{align*}
 which implies 
 \begin{equation}\label{p1}
 \left\| \sum\limits_{j\in \mathbb{S}}B_j^*L_j^*y\right\|\leq\frac{1+\alpha}{1-\beta}\left\| \sum\limits_{j\in \mathbb{S}}A_j^*L_j^*y\right\|+\frac{\gamma}{1-\beta}\left( \sum\limits_{j\in \mathbb{S}}\|L_j^*y\|^2\right)^\frac{1}{2}, \quad \forall y \in \ell^2(\mathbb{J})\otimes \mathcal{H}_0.
 \end{equation}
 We notice 
 $$ \langle y,y\rangle =\langle (I_{\ell^2(\mathbb{J})}\otimes I_{\mathcal{H}_0})y,y\rangle=\left\langle\sum\limits_{j\in \mathbb{J}}L_jL_j^* y,y\right\rangle=\sum\limits_{j\in \mathbb{J}}\|L_j^* y\|^2 , \quad \forall y \in \ell^2(\mathbb{J})\otimes \mathcal{H}_0.$$
 Let $ \mathbb{S}_1,\mathbb{S}_2$ be two finite subsets of $ \mathbb{J}$ with $ \mathbb{S}_1\subseteq\mathbb{S}_2.$ Following inequality shows that $\sum_{j\in \mathbb{J}}B_j^*L_j^*y $ exists for all  $   y \in \ell^2(\mathbb{J})\otimes \mathcal{H}_0.$ 
 \begin{align*}
 \left\|\sum\limits_{j\in \mathbb{S}_2}B_j^*L_j^*y-\sum\limits_{j\in \mathbb{S}_1}B_j^*L_j^*y\right\|&=\left\|\sum\limits_{j\in \mathbb{S}_2\setminus\mathbb{S}_1}B_j^*L_j^*y\right\|\\
 & \leq \frac{1+\alpha}{1-\beta}\left\| \sum\limits_{j\in \mathbb{S}_2\setminus\mathbb{S}_1}A_j^*L_j^*y\right\|+\frac{\gamma}{1-\beta}\left( \sum\limits_{j\in \mathbb{S}_2\setminus\mathbb{S}_1}\|L_j^*y\|^2\right)^\frac{1}{2}, \quad \forall y \in \ell^2(\mathbb{J})\otimes \mathcal{H}_0.
 \end{align*}
 From the continuity of norm, Inequality (\ref{p1}) gives
 \begin{align}\label{p2}
  \left\| \sum\limits_{j\in \mathbb{J}}B_j^*L_j^*y\right\|&\leq\frac{1+\alpha}{1-\beta}\left\| \sum\limits_{j\in \mathbb{J}}A_j^*L_j^*y\right\|+\frac{\gamma}{1-\beta}\left( \sum\limits_{j\in \mathbb{J}}\|L_j^*y\|^2\right)^\frac{1}{2}\nonumber \\
  &=\frac{1+\alpha}{1-\beta}\left\| \theta_A^*y\right\|+\frac{\gamma}{1-\beta}\|y\| , \quad \forall y \in \ell^2(\mathbb{J})\otimes \mathcal{H}_0
  \end{align}
  and this gives $ \sum_{j\in \mathbb{J}}B_j^*L_j^* $   is bounded; therefore its adjoint exists, which is $ \theta_B$; Inequality (\ref{p2}) now produces $\|\theta_B^*y\|\leq \frac{1+\alpha}{1-\beta}\left\| \theta_A^*y\right\|+\frac{\gamma}{1-\beta}\|y\| , \forall y \in \ell^2(\mathbb{J})\otimes \mathcal{H}_0 $ and from this $\|\theta_B\|=\|\theta_B^*\|\leq \frac{1+\alpha}{1-\beta}\left\| \theta_A^*\right\|+\frac{\gamma}{1-\beta} =\frac{1+\alpha}{1-\beta}\left\| \theta_A\right\|+\frac{\gamma}{1-\beta}.$
 Now using the hypothesis $\Psi^*_jB_j\geq0 $ for all $ j \in \mathbb{J} ,$ we get $ \theta_\Psi^*\theta_B=\sum_{j\in \mathbb{J}}\Psi_j^*B_j\geq 0.$ All in all,  we derived $ S_{B, \Psi}$ is a positive bounded linear operator. 
 Continuity of the norm, existence of frame  operators together with Inequality (\ref{p3}) give
 $$ \|\theta_A^*y-\theta_B^*y\|\leq \alpha\|\theta_A^*y\|+\beta\|\theta_B^*y\|+\gamma\|y\|, \quad \forall y \in \ell^2(\mathbb{J})\otimes \mathcal{H}_0$$

 which implies
 $$  \|\theta_A^*(\theta_\Psi S_{A,\Psi}^{-1} h)-\theta_B^*(\theta_\Psi S_{A,\Psi}^{-1}h)\|\leq \alpha\|\theta_A^*(\theta_\Psi S_{A,\Psi}^{-1} h)\|+\beta\|\theta_B^*(\theta_\Psi S_{A,\Psi}^{-1} h)\|+\gamma\|\theta_\Psi S_{A,\Psi}^{-1} h\|, \quad \forall h \in  \mathcal{H}.$$
 But $ \theta_A^*\theta_\Psi S_{A,\Psi}^{-1}=I_\mathcal{H}$ and $\theta_B^*\theta_\Psi S_{A,\Psi}^{-1}= S_{B,\Psi} S_{A,\Psi}^{-1}.$ Therefore 
 \begin{align*}
  \| h- S_{B,\Psi}S_{A,\Psi}^{-1}h\|
  &\leq \alpha\| h\|+\beta\|S_{B,\Psi} S_{A,\Psi}^{-1} h\|+\gamma\|\theta_\Psi S_{A,\Psi}^{-1} h\|\\
&\leq(\alpha+\gamma\|\theta_\Psi S_{A,\Psi}^{-1}\|)\|h\|+\beta\|S_{B,\Psi} S_{A,\Psi}^{-1} h\|, \quad \forall h \in  \mathcal{H}.
 \end{align*}
 Since $ \max\{\alpha+\gamma\|\theta_\Psi S_{A,\Psi}^{-1}\|, \beta\}<1$, Theorem \ref{cc1} tells that  $S_{B,\Psi} S_{A,\Psi}^{-1} $ is invertible and $\|(S_{B,\Psi} S_{A,\Psi}^{-1})^{-1}\| \leq \frac{1+\beta}{1-(\alpha+\gamma\|\theta_\Psi S_{A, \Psi}^{-1}\|)}.$ From these, we get $(S_{B,\Psi} S_{A,\Psi}^{-1})S_{A,\Psi}=S_{B,\Psi} $ is invertible and $ \| S_{B,\Psi}^{-1}\|\leq\|S_{A,\Psi}^{-1}\|\| S_{A,\Psi}S_{B,\Psi}^{-1}\| \leq \frac{\|S_{A,\Psi}^{-1}\|(1+\beta)}{1-(\alpha+\gamma\|\theta_\Psi S_{A, \Psi}^{-1}\|)}.$ Therefore $ (\{B_j\}_{j\in \mathbb{J}}, \{\Psi_j\}_{j\in \mathbb{J}})$ is an (ovf).  Observing that $\|S_{B,\Psi}\|\leq \|\theta_\Psi\|\|\theta_B\|\leq \frac{\|\theta_\Psi\|((1+\alpha)\|\theta_A\|+\gamma)}{1-\beta} $, and  $ \|S_{B,\Psi}^{-1}\|^{-1}$ and $\|S_{B,\Psi}\| $ are optimal lower and upper frame bounds for $ (\{B_j\}_{j\in \mathbb{J}}, \{\Psi_j\}_{j\in \mathbb{J}})$,  we get the frame bounds stated in the theorem.
\end{proof}

\begin{remark}
 Theorem \ref{PERTURBATION RESULT 1} is free from frame bounds for $ (\{A_j\}_{j\in \mathbb{J}}, \{\Psi_j\}_{j\in \mathbb{J}}) $.
\end{remark}

\begin{corollary}
 Let $ (\{A_j\}_{j\in \mathbb{J}}, \{\Psi_j\}_{j\in \mathbb{J}}) $  be  an (ovf) in $ \mathcal{B}(\mathcal{H}, \mathcal{H}_0)$. Suppose  $\{B_j\}_{j\in \mathbb{J}} $ in $ \mathcal{B}(\mathcal{H}, \mathcal{H}_0)$ is such that $ \Psi_j^*B_j\geq 0, \forall j \in \mathbb{J}$ and
 $$ r \coloneqq \sum_{j\in \mathbb{J}}\|A_j-B_j\|^2 <\frac{1}{\|\theta_\Psi S_{A,\Psi}^{-1}\|^2}.$$
 Then $ (\{B_j\}_{j\in \mathbb{J}}, \{\Psi_j\}_{j\in \mathbb{J}}) $ is   an (ovf) with bounds $ \frac{1-\sqrt{r}\|\theta_\Psi S_{A,\Psi}^{-1}\|}{\|S_{A,\Psi}^{-1}\|}$ and ${\|\theta_\Psi\|(\|\theta_A\|+\sqrt{r})} $.
\end{corollary}
\begin{proof}
Take $ \alpha =0, \beta=0, \gamma=\sqrt{r}$. Then $ \max\{\alpha+\gamma\|\theta_\Psi S_{A,\Psi}^{-1}\|, \beta\}<1$ and for every finite subset $ \mathbb{S}$ of $ \mathbb{J}$,
$$ \left\|\sum\limits_{j\in \mathbb{S}}(A_j^*-B_j^*)L_j^*y\right\|\leq \left(\sum\limits_{j\in \mathbb{S}}\|A_j^*-B_j^*\|^2 \right)^\frac{1}{2}\left(\sum\limits_{j\in \mathbb{S}}\|L_j^*y\|^2\right)^\frac{1}{2}\leq \gamma\left(\sum\limits_{j\in \mathbb{S}}\|L_j^*y\|^2\right)^\frac{1}{2}, ~\forall y \in \ell^2(\mathbb{J})\otimes \mathcal{H}_0.$$
Now use Theorem \ref{PERTURBATION RESULT 1}.
\end{proof}
 \begin{theorem}\label{PERTURBATION RESULT 2}
 Let $ (\{A_j\}_{j\in \mathbb{J}}, \{\Psi_j\}_{j\in \mathbb{J}} )$ be an (ovf) in $ \mathcal{B}(\mathcal{H}, \mathcal{H}_0)$ with bounds $ a$ and $b$. Suppose  $\{B_j\}_{j\in \mathbb{J}} $ is Bessel  (w.r.t. itself) in $ \mathcal{B}(\mathcal{H}, \mathcal{H}_0)$  such that $ \theta_\Psi^*\theta_B\geq0$  and 
 there exist $\alpha, \beta, \gamma \geq 0  $ with $ \max\{\alpha+\frac{\gamma}{\sqrt{a}}, \beta\}<1$ and 
\begin{equation} \label{pe1}
 \left|\sum\limits_{j\in \mathbb{J}}\langle(A_j-B_j)h,\Psi_jh \rangle\right|^\frac{1}{2}\leq \alpha \left(\sum\limits_{j\in \mathbb{J}}\langle A_jh,\Psi_jh \rangle\right)^\frac{1}{2}+\beta\left(\sum\limits_{j\in \mathbb{J}}\langle B_jh,\Psi_jh \rangle\right)^\frac{1}{2}+\gamma \|h\|,\quad  \forall h \in \mathcal{H}.
 \end{equation}
 Then  $ (\{B_j\}_{j\in \mathbb{J}}, \{\Psi_j\}_{j\in \mathbb{J}}) $ is an (ovf) with bounds $a\left(1-\frac{\alpha+\beta+\frac{\gamma}{\sqrt{a}}}{1+\beta}\right)^2 $ and $b\left(1+\frac{\alpha+\beta+\frac{\gamma}{\sqrt{b}}}{1-\beta}\right)^2 .$
 \end{theorem}
 \begin{proof}
	For all $ h $ in $\mathcal{H}$, 
	
 \begin{align*}
 \left(\sum\limits_{j\in \mathbb{J}}\langle B_jh,\Psi_jh \rangle\right)^\frac{1}{2}
 &\leq  \left|\sum\limits_{j\in \mathbb{J}}\langle (B_j-A_j)h,\Psi_jh \rangle\right|^\frac{1}{2}+\left(\sum\limits_{j\in \mathbb{J}}\langle A_jh,\Psi_jh \rangle\right)^\frac{1}{2} \\
 &\leq(1+ \alpha )\left(\sum\limits_{j\in \mathbb{J}}\langle A_jh,\Psi_jh \rangle\right)^\frac{1}{2}+\beta\left(\sum\limits_{j\in \mathbb{J}}\langle B_jh,\Psi_jh \rangle\right)^\frac{1}{2}+\gamma \|h\|
 \end{align*}
 
 which implies  
 \begin{align*}
 (1-\beta)\left(\sum\limits_{j\in \mathbb{J}}\langle B_jh,\Psi_jh \rangle\right)^\frac{1}{2}
 \leq(1+ \alpha )\left(\sum\limits_{j\in \mathbb{J}}\langle A_jh,\Psi_jh \rangle\right)^\frac{1}{2}+\gamma \|h\|
 \leq (1+\alpha )\sqrt{b}\|h\|+\gamma\|h\|, \quad \forall h \in \mathcal{H}.
 \end{align*}
  In a similar manner, from Inequality (\ref{pe1}),  
 \begin{align*}
 \left(\sum\limits_{j\in \mathbb{J}}\langle A_jh,\Psi_jh \rangle\right)^\frac{1}{2}
 &\leq  \left|\sum\limits_{j\in \mathbb{J}}\langle (A_j-B_j)h,\Psi_jh \rangle\right|^\frac{1}{2}+\left(\sum\limits_{j\in \mathbb{J}}\langle B_jh,\Psi_jh \rangle\right)^\frac{1}{2} \\
 &\leq \alpha \left(\sum\limits_{j\in \mathbb{J}}\langle A_jh,\Psi_jh \rangle\right)^\frac{1}{2}+(1+\beta)\left(\sum\limits_{j\in \mathbb{J}}\langle B_jh,\Psi_jh \rangle\right)^\frac{1}{2}+\gamma \|h\|\\
 &\leq \left(\alpha+\frac{\gamma}{\sqrt{a}}\right) \left(\sum\limits_{j\in \mathbb{J}}\langle A_jh,\Psi_jh \rangle\right)^\frac{1}{2}+(1+\beta)\left(\sum\limits_{j\in \mathbb{J}}\langle B_jh,\Psi_jh \rangle\right)^\frac{1}{2},\quad  \forall h \in \mathcal{H}
 \end{align*}
 implies 
 \begin{align*}
 \left(1-\left(\alpha +\frac{\gamma}{\sqrt{a}}\right) \right)\left(\sum\limits_{j\in \mathbb{J}}\langle A_jh,\Psi_jh \rangle\right)^\frac{1}{2}
 \leq  (1+\beta)\left(\sum\limits_{j\in \mathbb{J}}\langle B_jh,\Psi_jh \rangle\right)^\frac{1}{2},\quad  \forall h \in \mathcal{H}.
  \end{align*}
  But $\sqrt{a}\|h\|\leq  \left(\sum_{j\in \mathbb{J}}\langle A_jh,\Psi_jh \rangle\right)^\frac{1}{2},  \forall h \in \mathcal{H}.$
 Thus $ (\{B_j\}_{j\in \mathbb{J}}, \{\Psi_j\}_{j\in \mathbb{J}} )$ is an (ovf) with bounds $a\left(\frac{1-\left(\alpha+ \frac{\gamma}{\sqrt{a}}\right)}{1+\beta}\right)^2=a\left(1-\frac{\alpha+\beta+\frac{\gamma}{\sqrt{a}}}{1+\beta}\right)^2 $ and 
 $\left(\frac{(1+\alpha)\sqrt{b}+\gamma}{1-\beta}\right)^2=b\left(1+\frac{\alpha+\beta+\frac{\gamma}{\sqrt{b}}}{1-\beta}\right)^2 .$
 \end{proof}
 \begin{theorem}\label{OVFQUADRATICPERTURBATION}
 Let $ (\{A_j\}_{j\in \mathbb{J}}, \{\Psi_j\}_{j\in \mathbb{J}}) $  be an  (ovf) in $ \mathcal{B}(\mathcal{H}, \mathcal{H}_0)$. Suppose  $\{B_j\}_{j\in \mathbb{J}} $ in $ \mathcal{B}(\mathcal{H}, \mathcal{H}_0)$ is such that $ \Psi_j^*B_j\geq 0, \forall j \in \mathbb{J}$, $   \sum_{j\in \mathbb{J}}\|A_j-B_j\|^2$ converges, and 
$\sum_{j\in \mathbb{J}}\|A_j-B_j\|\|\Psi_jS_{A,\Psi}^{-1}\|<1.$
Then  $ (\{B_j\}_{j\in \mathbb{J}}, \{\Psi_j\}_{j\in \mathbb{J}}) $ is an (ovf) with bounds   $\frac{1-\sum_{j\in \mathbb{J}}\|A_j-B_j\|\|\Psi_jS_{A,\Psi}^{-1}\|}{\|S_{A,\Psi}^{-1}\|}$ and $\|\theta_\Psi\|(\sum_{j\in \mathbb{J}}\|A_j-B_j\|^2+\|\theta_A\|) $.
 \end{theorem}
 \begin{proof}
Let $ \alpha =\sum_{j\in \mathbb{J}}\|A_j-B_j\|^2 $ and $\beta =\sum_{j\in \mathbb{J}}\|A_j-B_j\|\|\Psi_jS_{A,\Psi}^{-1}\|$. For each finite subset $\mathbb{S} $ of $ \mathbb{J}$ and for every $ y$ in $ \ell^2(\mathbb{J})\otimes \mathcal{H}_0$, 

 \begin{align*}
\left\| \sum\limits_{j\in \mathbb{S}}B_j^*L_j^*y\right\|&\leq \left\| \sum\limits_{j\in \mathbb{S}}(A_j^*-B_j^*)L_j^*y\right\|+\left\| \sum\limits_{j\in \mathbb{S}}A_j^*L_j^*y\right\|\leq \sum\limits_{j\in \mathbb{S}}\|A_j-B_j\|\|L_j^*y\|+\left\| \sum\limits_{j\in \mathbb{S}}A_j^*L_j^*y\right\| \\
&\leq \left( \sum\limits_{j\in \mathbb{S}}\|A_j-B_j\|^2\right)^\frac{1}{2}\left( \sum\limits_{j\in \mathbb{S}}\|L_j^*y\|^2\right)^\frac{1}{2}+\left\| \sum\limits_{j\in \mathbb{S}}A_j^*L_j^*y\right\|\\
&\leq \alpha \left( \sum\limits_{j\in \mathbb{S}}\|L_j^*y\|^2\right)^\frac{1}{2}+\left\| \sum\limits_{j\in \mathbb{S}}A_j^*L_j^*y\right\|=\alpha \left\langle  \sum\limits_{j\in \mathbb{S}}L_jL_j^*y, y\right\rangle ^\frac{1}{2}+\left\| \sum\limits_{j\in \mathbb{S}}A_j^*L_j^*y\right\|,
\end{align*}

 which converges to $\alpha\|y\|+\|\theta_A^*y\|$. Hence 
$\theta_B$ exists and $\|\theta_B\|\leq \alpha+\|\theta_A\|$. Therefore  $S_{B,\Psi}=\theta_\Psi^*\theta_B=\sum_{j\in \mathbb{J}}\Psi^*_jB_j$ exists and is positive.
Now 
\begin{align*}
\|I_\mathcal{H}-S_{B,\Psi}S_{A,\Psi}^{-1}\|&=\left\|\sum_{j\in \mathbb{J}}A_j^*\Psi_j S_{A,\Psi}^{-1}-\sum_{j\in \mathbb{J}}B_j^*\Psi_j S_{A,\Psi}^{-1}\right\|=\left\|\sum_{j\in \mathbb{J}}(A_j^*-B_j^*)\Psi_j S_{A,\Psi}^{-1}\right\|\\
&\leq \sum_{j\in \mathbb{J}}\|A_j-B_j\|\|\Psi_j S_{A,\Psi}^{-1}\| =\beta<1.
\end{align*}
Therefore $S_{B,\Psi}S_{A,\Psi}^{-1}$ is invertible and $ \|(S_{B,\Psi}S_{A,\Psi}^{-1})^{-1}\|\leq 1/(1-\beta)$. Conclusion of frame bounds is similar to proof of Theorem \ref{PERTURBATION RESULT 1}.
\end{proof}

\section{Sequential version of the extension}\label{SEQUENTIAL} 
\begin{definition}\label{SEQUENTIAL2}
A set of vectors   $ \{x_j\}_{j\in \mathbb{J}}$  in a Hilbert space  $\mathcal{H}$ is said to be a frame    w.r.t.  set $ \{\tau_j\}_{j\in \mathbb{J}}$ in $\mathcal{H}$  if there are  $  c,  d  >0$ such that
\begin{enumerate}[\upshape(i)]
\item the map  $\mathcal{H} \ni  h \mapsto \sum_{j\in\mathbb{J}}\langle h,  x_j\rangle\tau_j \in \mathcal{H} $ is a well-defined  bounded positive  invertible operator.
\item $  \sum_{j \in \mathbb{J}}|\langle h,x_j\rangle|^2  \leq c\|h\|^2 , \forall h \in \mathcal{H}; 
\sum_{j \in \mathbb{J}}|\langle h,\tau_j\rangle|^2 \leq d\|h\|^2 , \forall h \in \mathcal{H}.$
\end{enumerate}
The operator in \text{\upshape(i)} is denoted by $S_{x,\tau}$ and is called as frame operator. Let $a, b >0$ be such that $ aI_\mathcal{H}\leq S_{x,\tau} \leq bI_\mathcal{H}. $
We call $ a$ and $ b$ as lower and upper frame bounds, respectively. Supremum of the set of all lower frame bounds is called optimal lower frame bound and infimum of the set of all upper frame bounds is called optimal upper frame bound.  With this,  optimal lower (resp. upper) frame bound is $ \|S_{x, \tau}^{-1}\|^{-1}$ (resp. $ \|S_{x,\tau}\|$). If optimal bounds are equal, we say the frame is exact. Whenever optimal bound of an exact frame is one,  we say it is Parseval. We term the bounded operators $\theta_x: \mathcal{H}\ni h\mapsto\{\langle h,x_j\rangle\}_{j\in\mathbb{J}}\in \ell^2(\mathbb{J}), \theta_\tau : \mathcal{H}\ni h\mapsto\{\langle h,\tau_j\rangle\}_{j\in\mathbb{J}}\in \ell^2(\mathbb{J})$ as analysis operators and adjoints of these as synthesis operators. 
If condition \text{\upshape(i)} is replaced by
$$\text{the map} ~ \mathcal{H} \ni  h \mapsto \sum_{j\in\mathbb{J}}\langle h,  x_j\rangle\tau_j \in \mathcal{H}  ~ \text{is well-defined  bounded positive  operator},$$
then we call $\{x_j\}_{j\in\mathbb{J}}$ is  Bessel w.r.t. $\{\tau_j\}_{j\in\mathbb{J}}$. If $ \{x_j\}_{j\in\mathbb{J}}$ is frame (resp. Bessel) w.r.t. $\{\tau_j\}_{j\in\mathbb{J}}$, then we write $(\{x_j\}_{j\in\mathbb{J}},\{\tau_j\}_{j\in\mathbb{J}})$ is frame (resp. Bessel).
     
For fixed $ \mathbb{J}, \mathcal{H},$  and $ \{\tau_j\}_{j\in \mathbb{J}}$   the set of all frames for $ \mathcal{H}$  w.r.t.  $ \{\tau_j\}_{j\in \mathbb{J}}$ is denoted by $ \mathscr{F}_\tau.$
\end{definition}
 We note, whenever $(\{x_j\}_{j\in\mathbb{J}},\{\tau_j\}_{j\in\mathbb{J}})$ is a frame for $ \mathcal{H}$, then $\overline{\operatorname{span}}\{x_j\}_{j\in \mathbb{J}}=\mathcal{H}=\overline{\operatorname{span}}\{\tau_j\}_{j\in \mathbb{J}}.$
\begin{theorem}\label{OVFTOSEQUENCEANDVICEVERSATHEOREM}
Let $\{x_j\}_{j\in \mathbb{J}}, \{\tau_j\}_{j\in \mathbb{J}}$ be in $\mathcal{H}$. Define $A_j: \mathcal{H} \ni h \mapsto \langle h, x_j \rangle \in \mathbb{K} $, $\Psi_j: \mathcal{H} \ni h \mapsto \langle h, \tau_j \rangle \in \mathbb{K}, \forall j \in \mathbb{J} $. Then   $(\{x_j\}_{j\in \mathbb{J}}, \{\tau_j\}_{j\in \mathbb{J}})$ is a frame for  $\mathcal{H}$ if and only if  $(\{A_j\}_{j\in \mathbb{J}}, \{\Psi_j\}_{j\in \mathbb{J}})$ is an  operator-valued frame  in $\mathcal{B}(\mathcal{H},\mathbb{K})$.
\end{theorem} 
\begin{proof}
$ \sum_{j \in \mathbb{J}}\Psi_j^*A_jh=\sum_{j \in \mathbb{J}}\langle h, x_j \rangle \tau_j,\forall h \in \mathcal{H}$.
\end{proof}
If we wish, we can avoid mapping given  in Definition \ref{SEQUENTIAL2}, and set up an alternate definition using only inequalities and an equality, precisely as shown below.
\begin{proposition}
Definition \ref{SEQUENTIAL2} holds if and only if there are  $ a, b, c,  d  >0$ such that
\begin{enumerate}[\upshape(i)]
\item $ a\|h\|^2\leq \sum_{j \in \mathbb{J}}\langle h,x_j\rangle\langle \tau_j, h \rangle \leq b\|h\|^2 , \forall h \in \mathcal{H},$
\item $  \sum_{j \in \mathbb{J}}|\langle h,x_j\rangle|^2  \leq c\|h\|^2 , \forall h \in \mathcal{H}; 
\sum_{j \in \mathbb{J}}|\langle h,\tau_j\rangle|^2 \leq d\|h\|^2 , \forall h \in \mathcal{H},$
\item $ \sum_{j \in \mathbb{J}}\langle h,x_j\rangle \tau_j=\sum_{j \in \mathbb{J}}\langle h,\tau_j\rangle x_j,  \forall h \in \mathcal{H}.$
 \end{enumerate}
If the space is over $ \mathbb{C},$ then  \text {\upshape{(iii)}} can be omitted.

\end{proposition}  
\begin{proof}
Direct part follows easily. We assume definition in  the statement of proposition. Condition (ii) gives the existence of $ S_{x, \tau}$ as bounded operator, since it is composition of $ \theta_x$ and $\theta_\tau^* $,  both are bounded. We now find the adjoint of $S_{x, \tau} .$ For all $ h, g \in \mathcal{H}, \langle S_{x, \tau}h , g \rangle =\sum_{j\in \mathbb{J}}\langle h , x_j\rangle\langle \tau_j , g\rangle= \langle h , \sum_{j\in \mathbb{J}}\langle g , \tau_j\rangle x_j\rangle=\langle h,  S_{\tau, x}g \rangle.$ Therefore $  S_{x, \tau}^*= S_{\tau, x}.$ But (iii) gives  $  S_{x, \tau}= S_{\tau, x}.$ Hence  $  S_{x, \tau}$ is self-adjoint. Positivity and invertibility of $  S_{x, \tau}$ are now given by (i).  Assume now that space is over the complex field and don't assume (iii). Now the middle term in (i) is $ \langle S_{x, \tau}h, h \rangle.$ We are now happy,  since the Hilbert space $\mathcal{H}$ is complex, and in this an operator $ T$ is positive if and only if $ \langle T h, h \rangle \geq 0, \forall h \in \mathcal{H}.$
\end{proof} 
From Definition \ref{SEQUENTIAL2} we see that if $( \{x_j\}_{j\in\mathbb{J}},   \{\tau_j\}_{j\in\mathbb{J}})$ is Bessel (or frame),  then both $ \{x_j\}_{j\in\mathbb{J}},   \{\tau_j\}_{j\in\mathbb{J}}$ are Bessel sequences (w.r.t. themselves).
\begin{proposition}
 Let  $( \{x_j\}_{j \in \mathbb{J}},\{\tau_j\}_{j \in \mathbb{J}} )$ be a frame for  $ \mathcal{H}$  with upper frame bound $b$. If for some $ j \in \mathbb{J} $ we have  $  \langle x_j, x_l \rangle\langle \tau_l, x_j \rangle \geq0, \forall l  \in \mathbb{J},$ then $ \langle x_j, \tau_j \rangle\leq b$ for that $j. $
 \end{proposition}
 \begin{proof}
 $\langle x_j, x_j \rangle\langle \tau_j, x_j \rangle \leq \sum_{l\in \mathbb{J}}\langle x_j, x_l \rangle\langle \tau_l, x_j \rangle\leq b\|x_j\|^2. $
 \end{proof}
\begin{proposition}
 Every Bessel sequence 	$(\{x_j\}_{j \in \mathbb{J}},\{\tau_j\}_{j \in \mathbb{J}} )$ for  $ \mathcal{H}$ can be extended to a tight frame for $ \mathcal{H}$. In particular,  every frame	$(\{x_j\}_{j \in \mathbb{J}},\{\tau_j\}_{j \in \mathbb{J}} )$ for $ \mathcal{H}$ can be extended to a tight frame for $ \mathcal{H}$.
\end{proposition}
\begin{proof}
Let $\{e_l\}_{l \in \mathbb{L}}$ be an orthonormal basis for $ \mathcal{H}$, and $ \lambda>\| S_{x,\tau}\|$. Define 	$ y_l \coloneqq (\lambda I_\mathcal{H}- S_{x,\tau})^{1/2}e_l, \forall l \in \mathbb{L}$. Then $\sum_{j\in \mathbb{J}}\langle h, x_j \rangle \tau_j+\sum_{l\in \mathbb{L}}\langle h, y_l \rangle y_l=S_{x,\tau}h+(\lambda I_\mathcal{H}- S_{x,\tau})h=h , \forall h \in \mathcal{H}$. Therefore $(\{x_j\}_{j \in \mathbb{J}}\cup\{y_l\}_{l \in \mathbb{L}},\{\tau_j\}_{j \in \mathbb{J}}\cup\{y_l\}_{l \in \mathbb{L}})$ is a tight frame for  $ \mathcal{H}$.
\end{proof}
 First frame algorithm is due to Duffin and Schaeffer \cite{DUFFIN1} and this continues to hold in our extension.
\begin{proposition}
Let $( \{x_n\}_{n=1}^\infty,\{\tau_n\}_{n=1}^\infty )$ be a frame for  $ \mathcal{H}$	with bounds $a$ and  $b$. For  $ h \in \mathcal{H}$ define 
$$ h_0\coloneqq0, ~h_n\coloneqq h_{n-1}+\frac{2}{a+b}S_{x,\tau}(h-h_{n-1}), ~\forall n \geq1.$$
Then 
$$ \|h_n-h\|\leq \left(\frac{b-a}{b+a}\right)^n\|h\|, ~\forall n \geq1.$$
\end{proposition} 
\begin{proof}
We first observe 
\begin{align*}
h-h_n&=h-h_{n-1}-\frac{2}{a+b}S_{x,\tau}(h-h_{n-1})=\left(I_\mathcal{H} -\frac{2}{b+a}S_{x,\tau}\right)(h-h_{n-1})\\
&=\cdots=\left(I_\mathcal{H} -\frac{2}{b+a}S_{x,\tau}\right)^nh, ~\forall h \in \mathcal{H} , ~\forall n \geq 1.
\end{align*}	
Now we try to bound the norm of $I_\mathcal{H} -\frac{2}{b+a}S_{x,\tau}$. For all $ h \in \mathcal{H}$,
\begin{align*}
-\frac{b-a}{b+a}\|h\|^2  =\|h\|^2- \frac{2b}{b+a}\|h\|^2 &\leq  \|h\|^2-  \frac{2}{b+a} \langle S_{x,\tau}h,h \rangle =\left\langle \left(I_\mathcal{H} -\frac{2}{b+a}S_{x,\tau}\right)h, h\right\rangle \\
&\leq \|h\|^2- \frac{2a}{b+a}\|h\|^2=\frac{b-a}{b+a}\|h\|^2. 
\end{align*} 	 
Therefore $ \|I_\mathcal{H} -\frac{2}{b+a}S_{x,\tau} \|=\sup_{h\in \mathcal{H}, \|h\|=1}|\langle(I_\mathcal{H} -\frac{2}{b+a}S_{x,\tau})h,h \rangle |\leq \frac{b-a}{b+a}.$ This gives $\|h_n-h\|\leq \|I_\mathcal{H} -\frac{2}{b+a}S_{x,\tau}\|^n\|h\| \leq \left(\frac{b-a}{b+a}\right)^n\|h\|, \forall n \geq1. $

\end{proof}
\begin{definition}\label{RELATIVERIESZSEQUENTIAL}
Let  $ \{x_j\}_{j \in \mathbb{J}},\{\tau_j\}_{j \in \mathbb{J}} $ be in $ \mathcal{H}.$ We say 
\begin{enumerate}[\upshape(i)]
\item $ \{x_j\}_{j \in \mathbb{J}}$ is an orthonormal set (resp. basis)  w.r.t. $\{\tau_j\}_{j \in \mathbb{J}} $ if  $ \{x_j\}_{j \in \mathbb{J}}$ or   $\{\tau_j\}_{j \in \mathbb{J}} $ is an orthonormal set (resp. basis) for $ \mathcal{H}$, say $ \{x_j\}_{j \in \mathbb{J}}$ is an orthonormal set (resp. basis) for $ \mathcal{H}$, and there exists  a sequence $\{c_j\}_{j \in \mathbb{J}} $ of  reals such that  $ 0<\inf\{c_j\}_{j\in \mathbb{J}}\leq \sup\{c_j\}_{j\in \mathbb{J}}<\infty$   and $ \tau_j=c_jx_j, \forall j \in \mathbb{J}.$ We write $ (\{x_j\}_{j\in \mathbb{J}}, \{\tau_j\}_{j\in \mathbb{J}})$ is orthonormal  set (resp. basis).
\item  $ \{x_j\}_{j \in \mathbb{J}}$ is a Riesz  basis  w.r.t. $\{\tau_j\}_{j \in \mathbb{J}} $ if there  are   invertible $ U,V\in \mathcal{B}(\mathcal{H})$ and an orthonormal basis $\{f_j\}_{j\in \mathbb{J}} $ for $\mathcal{H}$ such that $x_j=Uf_j, \tau_j=Vf_j, \forall j \in \mathbb{J}$ and $ VU^*\geq 0.$ We write $ (\{x_j\}_{j\in \mathbb{J}}, \{\tau_j\}_{j\in \mathbb{J}})$ is a Riesz  basis.
\end{enumerate}
\end{definition}
Like ``operator version", last definition is symmetric. Let $ \tau_j=x_j, \forall j \in \mathbb{J}.$ Then $c_j=1, \forall j \in \mathbb{J}.$ Therefore we have the definition   of orthonormality in the sequence case and  for the Riesz basis case, we must have $ U=V.$ Then $VU^*=UU^*\geq 0,$ comes automatically.
\begin{caution}
We have to remember, clearly, the ``star'' is in different positions (it is on different operators)  in \text{\upshape(ii)} of Definition \ref{ONBRIESZDEFINITION} and in \text{\upshape(ii)} of Definition \ref{RELATIVERIESZSEQUENTIAL}.
\end{caution} 
\begin{theorem}\label{GBESV}
Let $( \{x_j\}_{j \in \mathbb{J}},\{\tau_j=c_jx_j\}_{j \in \mathbb{J}}) $  be  orthonormal  for $\mathcal{H}.$  Then 
 \begin{enumerate}[\upshape(i)]
\item Generalized Bessel's inequality - sequential version: 
$$ \sum\limits_{j\in\mathbb{J}}(2-c_j)\langle h, x_j\rangle \langle \tau_j, h \rangle\leq \|h\|^2 , ~\forall h \in \mathcal{H}.$$
\item For $ h \in \mathcal{H},$ 
\begin{align*}
 h= \sum \limits_{j\in \mathbb{J}}\langle h, x_j\rangle \tau_j \iff \sum\limits_{j\in\mathbb{J}}(2-c_j)\langle h, x_j\rangle \langle \tau_j, h\rangle = \|h\|^2 \iff  \sum\limits_{j\in\mathbb{J}}c_j^2|\langle h, x_j\rangle |^2  = \|h\|^2.
 \end{align*}
 If $ c_j\leq 1, \forall j$, then $ h= \sum_{j\in \mathbb{J}}\langle h, x_j\rangle \tau_j  \iff (1-c_j)\langle h, x_j\rangle=0 , \forall j \in \mathbb{J} \iff (1-c_j)x_j\perp h, \forall j \in \mathbb{J}.$
\end{enumerate}
 \end{theorem}
\begin{proof}
\begin{enumerate}[\upshape(i)]
 \item For  $ h \in \mathcal{H}$ and  each finite subset $ \mathbb{S} \subseteq \mathbb{J},$
 \begin{align*}
 \left\| \sum\limits_{j\in \mathbb{S}}\langle h,  x_j \rangle \tau_j\right\|^2
 &= \left\langle \sum_{j\in \mathbb{S}}c_j\langle h, x_j\rangle x_j, \sum_{k\in \mathbb{S}}c_k\langle h, x_k\rangle x_k \right\rangle=\sum\limits_{j \in \mathbb{S}} c_j^2| \langle h,  x_j \rangle| ^2\\
 &\leq \left(\sup\{c^2_j\}_{j\in \mathbb{J}}\right)\sum\limits_{j \in \mathbb{S}}| \langle h,  x_j \rangle| ^2,
 \end{align*}
 the last sum converges. Hence $\sum_{j\in \mathbb{S}}\langle h,  x_j \rangle \tau_j$ exists. In a similar manner $\sum_{j\in\mathbb{J}}(2-c_j)\langle h, x_j\rangle \langle \tau_j, h \rangle $ also exists.  Then 
 \begin{align*}
  0&\leq \left\|h-\sum\limits_{j\in \mathbb{J}}\langle h,  x_j \rangle \tau_j\right\|^2 = \left\langle h-\sum_{j\in \mathbb{J}}c_j\langle h, x_j\rangle x_j, h-\sum_{k\in \mathbb{J}}c_k\langle h, x_k\rangle x_k \right\rangle \\
  &=\|h\|^2-2\sum_{j\in \mathbb{J}}c_j|\langle h, x_j\rangle |^2+\sum_{j\in \mathbb{J}}c^2_j|\langle h, x_j\rangle |^2
 =\|h\|^2-\sum_{j\in \mathbb{J}}(2c_j-c_j^2)|\langle h, x_j\rangle |^2,
  \end{align*}
 $ \Rightarrow\sum_{j\in \mathbb{J}}(2c_j-c_j^2)|\langle h, x_j\rangle |^2=\sum_{j\in\mathbb{J}}(2-c_j)\langle h, x_j\rangle \langle \tau_j, h \rangle\leq \|h\|^2.$ 
 \item First `if and only if' follows from (i). Others follow from orthonormality of $\{x_j\}_{j\in \mathbb{J}}.$
 \end{enumerate}
 \end{proof}
 \begin{corollary}
 (Generalized Parseval formula) Let $( \{x_j\}_{j \in \mathbb{J}},\{\tau_j=c_jx_j\}_{j \in \mathbb{J}}) $  be  an orthonormal basis for  $ \mathcal{H}.$ Then
 $$ \frac{1}{\sup\{c_j\}_{j\in \mathbb{J}}}\sum_{j\in \mathbb{J}}\langle h, x_j\rangle \langle\tau_j, h \rangle \leq \|h\|^2\leq \frac{1}{\inf\{c_j\}_{j\in \mathbb{J}} }\sum_{j\in \mathbb{J}}\langle h, x_j\rangle \langle\tau_j, h \rangle, ~\forall h \in \mathcal{H}.$$
 \end{corollary}
 \begin{proof}
 From the proof of Theorem \ref{GBESV}, $\sum_{j\in \mathbb{J}}\langle h, x_j\rangle \langle\tau_j, h\rangle $ exists, $ \forall h \in \mathcal{H}$. Then $ \frac{1}{\sup\{c_j\}_{j\in \mathbb{J}}}\sum_{j\in \mathbb{J}}\langle h, x_j\rangle \langle\tau_j, h\rangle   =\frac{1}{\sup\{c_j\}_{j\in \mathbb{J}}}\sum_{j\in \mathbb{J}}c_j|\langle h, x_j\rangle|^2  \leq \sum_{j\in \mathbb{J}}|\langle h, x_j\rangle|^2 = \|h\|^2\leq \frac{1}{\inf\{c_j\}_{j\in \mathbb{J}}}\sum_{j\in \mathbb{J}} c_j|\langle h, x_j\rangle |^2\leq  \frac{1}{\inf\{c_j\}_{j\in \mathbb{J}}}\sum_{j\in \mathbb{J}}\langle h, x_j\rangle \langle\tau_j, h \rangle$, $\forall h \in \mathcal{H}.$
 \end{proof}
 \begin{remark}
 Whenever $ \tau_j=x_j ,\forall j \in \mathbb{J},$ last corollary gives the Parseval formula. 
 \end{remark}
 \begin{theorem}
 If $ (\{x_j\}_{j\in \mathbb{J}}, \{\tau_j=c_jx_j\}_{j\in \mathbb{J}})$  is an orthonormal set for  $\mathcal{H}$ with $ c_j\leq 2 ,\forall j \in \mathbb{J}$, then for each $ h \in \mathcal{H}$, the set $ Y_h=\{x_j: (2-c_j)\langle h, x_j\rangle \langle \tau_j, h \rangle \neq 0, j\in \mathbb{J}\}$ is either  finite or countable.
 \end{theorem}
 \begin{proof}
 For $ n \in \mathbb{N}$, define 
 $$ Y_{n,h}\coloneqq\left\{x_j: (2-c_j)\langle h, x_j\rangle \langle \tau_j, h \rangle > \frac{1}{n}\|h\|^2, j\in \mathbb{J} \right\}.$$	
 
 Suppose, for some $n$, $ Y_{n,h}$ has more than $n-1$ elements, say $x_1,...,x_n$. Then $ \sum_{j=1 }^n(2-c_j)\langle h, x_j\rangle \langle \tau_j, h \rangle > n\frac{1}{n}\|h\|^2=\|h\|^2$. From Theorem \ref{GBESV},  $ \sum_{j\in\mathbb{J}}(2-c_j)\langle h, x_j\rangle \langle \tau_j, h \rangle\leq \|h\|^2 $ This gives $  \|h\|^2< \|h\|^2$ which is impossible. Hence $ \operatorname{Card}(Y_{n,h})\leq n-1$ and hence  
 $Y_h=\cup_{n=1}^\infty Y_{n,h}$ being a countable union of finite sets  is finite or countable.
 \end{proof}
 \begin{theorem}\label{RIESZBASISIMPLIESFRAMESEQUENTIAL}
 \begin{enumerate}[\upshape(i)]
 \item If $ (\{x_j\}_{j\in \mathbb{J}}, \{\tau_j\}_{j\in \mathbb{J}})$  is an orthonormal basis for $\mathcal{H}$, then it is a Riesz basis.
 \item If $(\{x_j=Uf_j\}_{j\in\mathbb{J}}, \{\tau_j=Vf_j\}_{j\in\mathbb{J}})$ is a Riesz basis for $\mathcal{H}$, then it is  a frame  with  optimal frame bounds $\|(VU^*)^{-1}\|^{-1}$ and $ \|VU^*\|.$
 \end{enumerate}
 \end{theorem}
 \begin{proof}
 \begin{enumerate}[\upshape(i)]
\item  We may assume  $ \{x_j\}_{j \in \mathbb{J}}$ is an orthonormal basis.  Then there exists a sequence  $\{c_j\}_{j \in \mathbb{J}} $  of   reals such that  $ 0<\inf\{c_j\}_{j\in \mathbb{J}}\leq \sup\{c_j\}_{j\in \mathbb{J}}<\infty$   and $ \tau_j=c_jx_j, \forall j \in \mathbb{J}.$ Define $ f_j\coloneqq x_j, \forall j \in \mathbb{J},  U\coloneqq I_\mathcal{H}$ and  $ V : \mathcal{H} \ni h \mapsto\sum_{j\in \mathbb{J}}c_j\langle h, x_j\rangle x_j \in \mathcal{H}.$ From the proof  Theorem \ref{GBESV}, $ V$ is a well-defined bounded operator. Then $ Uf_j=x_j, Vf_j=\sum_{k\in \mathbb{J}}c_k\langle f_j,x_k \rangle x_k= c_jx_j=\tau_j, \forall j \in \mathbb{J}.$  Since all $c_j $'s are positive,  $ V$ is positive invertible, whose inverse (at $ h \in \mathcal{H} $) is $\sum_{j\in \mathbb{J}}c_j^{-1}\langle h, x_j\rangle x_j . $ In fact, $ V^{-1}Vh=\sum_{k\in \mathbb{J}}c_k^{-1}\langle \sum_{j\in \mathbb{J}}c_j\langle h, x_j\rangle x_j, x_k\rangle x_k= \sum_{k\in \mathbb{J}}\langle h, x_k\rangle x_k=h, VV^{-1}h=\sum_{k\in \mathbb{J}}c_k\langle \sum_{j\in \mathbb{J}}c_j^{-1}\langle h, x_j\rangle x_j, x_k\rangle x_k= \sum_{k\in \mathbb{J}}\langle h, x_k\rangle x_k=h,  \forall h \in \mathcal{H}.$ We end by noting  that $ VU^*=VI_\mathcal{H}=V\geq 0.$
\item Following inequalities show that  $ \{x_j\}_{j\in\mathbb{J}}, \{ \tau_j\}_{j\in\mathbb{J}}$ are Bessel sequences (w.r.t. themselves): 
$ \sum_{j\in\mathbb{S}}|\langle h,x_j\rangle|^2 = \sum_{j\in\mathbb{S}}|\langle h,Uf_j\rangle|^2 = \sum_{j\in\mathbb{S}}|\langle U^*h,f_j\rangle|^2\leq \|U^*h\|^2\leq\|U^*\|^2\|h\|^2 ,
\sum_{j\in\mathbb{S}}|\langle g,\tau_j\rangle|^2 = \sum_{j\in\mathbb{S}}|\langle g,Vf_j\rangle|^2\leq \|V^*g\|^2\leq\|V^*\|^2\|g\|^2, \forall h, g \in \mathcal{H}, \text{for each finite}~ \mathbb{S} \subseteq\mathbb{J}.$
 Further,
 $$  \frac{1}{\|(VU^*)^{-1}\|}\langle h,h \rangle\leq \langle VU^*h,h \rangle \leq  \|VU^*\|\langle h,h \rangle,~ \forall h \in \mathcal{H};$$
\begin{align*}
\langle VU^*h,h \rangle& =\langle U^*h,V^*h \rangle =\sum\limits_{j\in\mathbb{J}}\langle U^*h,f_j\rangle\langle f_j,V^*h\rangle\\ &=\sum\limits_{j\in\mathbb{J}}\langle h,Uf_j\rangle\langle Vf_j,h\rangle=\sum\limits_{j\in\mathbb{J}}\langle h,x_j\rangle \langle \tau_j,h\rangle, ~\forall h \in \mathcal{H},
\end{align*}
and 
\begin{align*}
\sum_{j\in\mathbb{J}}\langle h,x_j\rangle \tau_j&=\sum_{j\in\mathbb{J}}\langle h,Uf_j\rangle Vf_j=V\left(\sum_{j\in\mathbb{J}}\langle U^*h,f_j\rangle f_j \right)=VU^*h=UV^*h\\
&=U\left(\sum_{j\in\mathbb{J}}\langle V^*h,f_j\rangle f_j \right)=\sum_{j\in\mathbb{J}}\langle h,Vf_j\rangle Uf_j=\sum_{j\in\mathbb{J}}\langle h,\tau_j\rangle x_j, ~\forall h \in \mathcal{H}.
\end{align*}
For the optimal bounds  we note that   if $ T$ is a positive invertible operator on a Hilbert space $\mathcal{H}$, then $ \sup\{\alpha :\alpha I_{\mathcal{H}}\leq T\}=\|T^{-1}\|^{-1}$ and $ \inf\{\beta :T\leq\beta I_{\mathcal{H}}\}=\|T\|.$
\end{enumerate}
\end{proof}
\begin{theorem}
 Let $ (\{x_j=Uf_j\}_{j\in\mathbb{J}}, \{\tau_j=Vf_j\}_{j\in\mathbb{J}})$ be a Riesz basis for $\mathcal{H}.$ Then
\begin{enumerate}[\upshape(i)]
 \item There exist unique  $ \{y_j\}_{j\in\mathbb{J}}, \{\omega_j\}_{j\in\mathbb{J}}$ in $\mathcal{H}$ such that 
$$ h= \sum\limits_{j\in\mathbb{J}}\langle h, y_j\rangle x_j = \sum\limits_{j\in\mathbb{J}}\langle h, \omega_j\rangle\tau_j, ~ \forall h \in \mathcal{H}$$ and $( \{y_j\}_{j\in\mathbb{J}}, \{\omega_j\}_{j\in\mathbb{J}})$ is Riesz. Moreover,  the convergence is unconditional.
\item $ \{x_j\}_{j\in\mathbb{J}}$, $ \{\tau_j\}_{j\in\mathbb{J}}$  are complete in $\mathcal{H} $.  If $ V^*U\geq0,$ then there are real $ a, b >0$ such that for every finite subset $ \mathbb{S}$ of $ \mathbb{J}$
$$ a\sum\limits_{j\in \mathbb{S}}|c_j|^2\leq \left\langle \sum\limits_{j\in\mathbb{S}}c_jx_j ,\sum\limits_{k\in\mathbb{S}}c_k\tau_k  \right \rangle \leq b\sum\limits_{j\in \mathbb{S}}|c_j|^2, ~\forall c_j\in \mathbb{K}, \forall j\in \mathbb{S}.$$ 
\end{enumerate}
\end{theorem}
\begin{proof}
\begin{enumerate}[\upshape(i)]
\item Since $ \{x_j\}_{j\in\mathbb{J}}, \{\tau_j\}_{j\in\mathbb{J}}$ are Riesz bases (w.r.t. themselves), we have $ y_j=(U^{-1})^*f_j, \omega_j=(V^{-1})^*f_j , \forall j \in \mathbb{J}$ and they are unique. The knowledge `inverse of a positive invertible element in a C*-algebra is again positive' gives $( \{y_j\}_{j\in\mathbb{J}}, \{\omega_j\}_{j\in\mathbb{J}})$ is Riesz.
\item Completeness is clear. Let $ V^*U\geq 0.$ Then 
\begin{align*}
\frac{1}{\|(V^*U)^{-1}\|}\sum\limits_{j\in \mathbb{S}}|c_j|^2&=\frac{1}{\|(V^*U)^{-1}\|}\left\langle \sum\limits_{j\in\mathbb{S}}c_jf_j ,\sum\limits_{k\in\mathbb{S}}c_kf_k  \right \rangle\leq \left\langle \sum\limits_{j\in\mathbb{S}}c_jUf_j ,\sum\limits_{k\in\mathbb{S}}c_kVf_k  \right \rangle\\
&=\left\langle \sum\limits_{j\in\mathbb{S}}c_jx_j ,\sum\limits_{k\in\mathbb{S}}c_k\tau_k  \right \rangle \leq \left\langle V^*U\left(\sum\limits_{j\in\mathbb{S}}c_jf_j \right),\sum\limits_{k\in\mathbb{S}}c_kf_k  \right \rangle \leq \|V^*U\| \sum\limits_{j\in \mathbb{S}}|c_j|^2
\end{align*}
for all $c_j\in \mathbb{K}, \forall j\in \mathbb{S}.$
\end{enumerate}
\end{proof}
\begin{lemma}(cf. \cite{OLE1})\label{RIESZCHARACTERIZATION}
Let $ \{x_j\}_{j\in\mathbb{J}}$ be complete in $ \mathcal{H}$, $ \{y_j\}_{j\in\mathbb{J}}$ be Bessel in $ \mathcal{H}_0$ with bound $b$  and assume that there exists $ a>0$ such that $ a \sum_{j\in \mathbb{S}}|c_j|^2\leq \|\sum_{j\in \mathbb{S}}c_jx_j\|^2$ for all finite $\mathbb{S}\subseteq\mathbb{J},c_j \in \mathbb{K}, j\in \mathbb{S}.$ Then $T : \operatorname{span} \{x_j\}_{j\in\mathbb{J}} \ni \sum_{\operatorname{finite}}c_jx_j \mapsto \sum_{\operatorname{finite}}c_jy_j \in  \operatorname{span} \{y_j\}_{j\in\mathbb{J}}$,  is a bounded linear operator and has unique extension from $\mathcal{H} $ into $ \mathcal{H}_0$; the norm of $ T$ and its extension is at most $\sqrt{b/a}$.
\end{lemma}
\begin{theorem}
Let $ \{x_j\}_{j\in\mathbb{J}},  \{\tau_j\}_{j\in\mathbb{J}}$ be in $\mathcal{H}$. Then the following are equivalent.
\begin{enumerate}[\upshape(i)]
\item $( \{x_j\}_{j\in\mathbb{J}},  \{\tau_j\}_{j\in\mathbb{J}})$ is a Riesz basis for  $\mathcal{H}$.
 \item $ \{x_j\}_{j\in\mathbb{J}},  \{\tau_j\}_{j\in\mathbb{J}}$ are complete  in $\mathcal{H}$, and there exist $ a, b, c, d >0$ such that for every finite subset $ \mathbb{S}\subseteq\mathbb{J}$, 
 \begin{align*}
 a\sum_{j\in \mathbb{S}}|c_j|^2\leq \left\|\sum_{j\in \mathbb{S}}c_jx_j\right\|^2\leq b\sum_{j\in \mathbb{S}}|c_j|^2, ~\forall c_j \in \mathbb{K},\forall j\in \mathbb{S}, \\
 c\sum_{j\in \mathbb{S}}|d_j|^2\leq \left\|\sum_{j\in \mathbb{S}}d_j\tau_j\right\|^2\leq d\sum_{j\in \mathbb{S}}|c_j|^2, ~\forall d_j \in \mathbb{K}, \forall j\in \mathbb{S},
 \end{align*}
 and 
 $$ \sum_{j\in \mathbb{J}}\langle h, x_j\rangle \langle \tau_j, h\rangle \geq 0, \forall h \in \mathcal{H}.$$
 \end{enumerate}
 \end{theorem}
 \begin{proof}
 (i) $\Rightarrow$ (ii) Let $ U,V \in \mathcal{B}(\mathcal{H})$ be bounded invertible with $ VU^*\geq0$ and $ \{f_j\}_{j\in\mathbb{J}}$ be an orthonormal basis for $\mathcal{H} $ such that $ x_j=Uf_j, \tau_j=Vf_j, \forall j \in \mathbb{J}.$ Since $ \{x_j\}_{j\in\mathbb{J}},\{y_j\}_{j\in\mathbb{J}} $ are Riesz, first two inequalities in (ii) hold. For that last, we find $ U^*$ and $ V^*$. For all $h, g \in \mathcal{H}, \langle U^*h, g\rangle =\langle h, \sum_{j\in \mathbb{J}}\langle g, f_j\rangle x_j\rangle= \langle \sum_{j\in \mathbb{J}}\langle h, x_j\rangle f_j, g\rangle$.  Therefore $ U^*h=\sum_{j\in \mathbb{J}}\langle h, x_j\rangle f_j$ and similarly $V^*h=\sum_{j\in \mathbb{J}}\langle h, \tau_j\rangle f_j, \forall h \in \mathcal{H}$. These give $ 0\leq \langle U^*h, V^*h\rangle=\sum_{j\in \mathbb{J}}\langle h, x_j\rangle \langle \tau_j, h\rangle,\forall h \in \mathcal{H} $.
  
(ii) $\Rightarrow$ (i) Let $ \{e_j\}_{j\in\mathbb{J}}$ be an orthonormal basis for $\mathcal{H} $. Lemma  \ref{RIESZCHARACTERIZATION} gives that the mappings $ \mathcal{H}\ni\sum_{j\in \mathbb{J}}a_je_j  \mapsto  \sum_{j\in \mathbb{J}}a_jx_j \in \mathcal{H} $, $ \mathcal{H}\ni\sum_{j\in \mathbb{J}}a_je_j  \mapsto \sum_{j\in \mathbb{J}}a_j\tau_j \in \mathcal{H} $, $ \mathcal{H}\ni\sum_{j\in \mathbb{J}}a_jx_j  \mapsto \sum_{j\in \mathbb{J}}a_je_j \in \mathcal{H} $, $ \mathcal{H}\ni\sum_{j\in \mathbb{J}}a_j\tau_j  \mapsto \sum_{j\in \mathbb{J}}a_je_j \in \mathcal{H} $  are well-defined  bounded linear operators on  $\mathcal{H}$, call them as $ U,V, E,F$, respectively. Then $ Ue_j=x_j, Ve_j=\tau_j, \forall j \in \mathbb{J}$,  and $ UE=EU=I_\mathcal{H}=VF=FE$. Further, $\langle VU^*h, h\rangle =\langle U^*h, U^*h\rangle = \sum_{j\in \mathbb{J}}\langle U^*h, e_j\rangle \langle e_j, V^*h\rangle=\sum_{j\in \mathbb{J}}\langle h, x_j\rangle \langle \tau_j, h\rangle \geq 0, \forall h \in \mathcal{H}. $
\end{proof}

\begin{proposition}
For every $ \{x_j\}_{j \in \mathbb{J}}  \in \mathscr{F}_\tau$,
\begin{enumerate}[\upshape (i)]
 \item $\theta_x^*(\{c_j\}_{j\in\mathbb{J}})=\sum_{j\in \mathbb{J}}c_jx_j, \theta_\tau^*(\{c_j\}_{j\in\mathbb{J}})=\sum_{j\in \mathbb{J}}c_j\tau_j,  \forall  \{c_j\}_{j\in\mathbb{J}} \in \ell^2(\mathbb{J}) ;$\\
 $\theta_x^* \theta_xh =  \sum_{j\in \mathbb{J}} \langle h,x_j\rangle x_j, \theta_\tau^* \theta_\tau h =  \sum_{j\in \mathbb{J}} \langle h,\tau_j\rangle \tau_j,  \forall h \in \mathcal{H}.$ 
\item $ S_{x, \tau} = \theta_\tau^*\theta_x=\theta_x^*\theta_\tau .$ In particular,
 $$ S_{x, \tau}h =\sum\limits_{j\in \mathbb{J}}\langle h,x_j\rangle\tau_j=\sum\limits_{j\in \mathbb{J}}\langle h,\tau_j\rangle x_j,~ \forall h \in \mathcal{H} ~\operatorname{and}~$$
 $$\langle S_{x, \tau}h, g\rangle =\sum\limits_{j\in \mathbb{J}}\langle h,x_j\rangle\langle\tau_j, g\rangle=\sum\limits_{j\in \mathbb{J}}\langle h,\tau_j\rangle \langle x_j,g\rangle, ~ \forall h,g  \in \mathcal{H}.$$
 \item Every $ h \in \mathcal{H}$ can be written as 
 $$h =\sum\limits_{j\in \mathbb{J}}\langle h, S^{-1}_{x, \tau}\tau_j\rangle  x_j=\sum\limits_{j\in \mathbb{J}}\langle h,\tau_j \rangle S^{-1}_{x, \tau}  x_j =\sum\limits_{j\in \mathbb{J}}\langle h, S^{-1}_{x, \tau}x_j\rangle  \tau_j=\sum\limits_{j\in \mathbb{J}}\langle h, x_j\rangle S^{-1}_{x, \tau}  \tau_j.$$
 \item $(\{x_j\}_{j \in \mathbb{J}}, \{\tau_j \}_ {j \in \mathbb{J}})$ is Parseval  if and only if $  \theta_\tau^*\theta_x=I_\mathcal{H}.$ 
 \item $(\{x_j\}_{j \in \mathbb{J}}, \{\tau_j \}_ {j \in \mathbb{J}})$ is Parseval  if and only if $ \theta_x\theta_\tau^* $ is idempotent.
 \item $\theta_xS_{x,\tau}^{-1}\theta_\tau^* $ is idempotent.
 \item $\theta_x$ and $ \theta_\tau$ are injective and  their  ranges are closed.
 \item  $\theta_x^* $ and $ \theta_\tau^*$ are surjective.
 \end{enumerate}
 \end{proposition} 
 \begin{proof}
 Similar to proof of Proposition \ref{2.2}.
\end{proof}
 We call $ P_{x,\tau}\coloneqq\theta_xS_{x,\tau}^{-1}\theta_\tau^* $ as the  \textit{frame idempotent}.  
\begin{definition}
 A frame  $( \{x_j\}_{j \in \mathbb{J}}, \{\tau_j\}_{j \in \mathbb{J}})$ for $ \mathcal{H}$ is called a Riesz frame if $P_{x,\tau} = I_{\ell^2(\mathbb{J})} $. A Parseval and Riesz frame (i.e., $\theta_\tau^*\theta_x=I_\mathcal{H} $ and $\theta_x\theta_\tau^*= I_{\ell^2(\mathbb{J})}$) is called as an orthonormal frame.
 \end{definition}
\begin{proposition}
\begin{enumerate}[\upshape(i)]
\item  If $( \{x_j\}_{j \in \mathbb{J}}, \{\tau_j\}_{j \in \mathbb{J}})$ is a Riesz basis for $ \mathcal{H}$, then it is a  Riesz frame.
\item  If $( \{x_j\}_{j \in \mathbb{J}}, \{\tau_j=c_jx_j\}_{j \in \mathbb{J}})$ is an  orthonormal basis for $ \mathcal{H}$, then it is  a Riesz frame.
 \end{enumerate}
\end{proposition} 
\begin{proof}
\begin{enumerate}[\upshape(i)]
\item Let $ \{f_j\}_{j \in \mathbb{J}}$  be an orthonormal basis for  $ \mathcal{H}$ and  $ U, V : \mathcal{H} \rightarrow \mathcal{H}$ be  bounded invertible with $ VU^*$ is positive  such that $ x_j=Uf_j, \tau_j=Vf_j,  \forall j \in \mathbb{J}$. From Theorem \ref{RIESZBASISIMPLIESFRAMESEQUENTIAL}, $( \{x_j\}_{j \in \mathbb{J}}, \{\tau_j\}_{j \in \mathbb{J}})$ is a frame for $ \mathcal{H}$. Now $ \theta_xh=\{\langle h, Uf_j\rangle \}_{j \in \mathbb{J}}=\theta_fU^*h,\forall h \in \mathcal{H}$. Similarly $\theta_\tau=\theta_f V^*$. Therefore $ S_{x,\tau}=V\theta_f^*\theta_fU^*=VU^*$. This gives $P_{x,\tau} = \theta_xS_{x,\tau}^{-1}\theta_\tau^*=\theta_fU^*(VU^*)^{-1}(\theta_f V^*)^*=\theta_f\theta_f^*= I_{\ell^2(\mathbb{J})}$.
\item $S_{x,\tau}h=\sum_{j\in \mathbb{J}}\langle h, x_j\rangle \tau_j=\sum_{j\in \mathbb{J}}\langle h, x_j\rangle c_jx_j,\forall h \in \mathcal{H}$; $S_{x,\tau}^{-1}h=\sum_{j\in \mathbb{J}}c_j^{-1}\langle h, x_j\rangle x_j,\forall h \in \mathcal{H}$, $P_{x,\tau}( \{a_j\}_{j \in \mathbb{J}}) = \theta_xS_{x,\tau}^{-1}\theta_\tau^*(\{a_j\}_{j \in \mathbb{J}})=\theta_xS_{x,\tau}^{-1}(\sum_{j\in \mathbb{J}}a_jc_jx_j) =\theta_x(\sum_{k\in \mathbb{J}}c_k^{-1}\langle \sum_{j\in \mathbb{J}}a_jc_jx_j, x_k\rangle x_k)= \theta_x(\sum_{k\in \mathbb{J}}a_kx_k)=\sum_{l\in \mathbb{J}}\langle\sum_{k\in \mathbb{J}}a_kx_k , x_l\rangle e_l=\{a_j\}_{j \in \mathbb{J}},\forall \{a_j\}_{j \in \mathbb{J}} \in \ell^2(\mathbb{J})$, where
$( \{e_l\}_{l \in \mathbb{J}}$ is the standard orthonormal basis for $\ell^2(\mathbb{J})$.

\end{enumerate}
\end{proof}
\begin{proposition}\label{RIESZFRAMECHARACTERIZATION}
A frame  $( \{x_j\}_{j \in \mathbb{J}}, \{\tau_j\}_{j \in \mathbb{J}})$ for $ \mathcal{H}$ is  a Riesz frame if  and only if $ \theta_x(\mathcal{H})=\ell^2(\mathbb{J})$ if  and only if $ \theta_\tau(\mathcal{H})=\ell^2(\mathbb{J}).$
\end{proposition}
\begin{proof}
Let $ P_{x,\tau}= I_{\ell^2(\mathbb{J})}.$ Now $\ell^2(\mathbb{J})=P_{x,\tau}(\mathcal{H})=\theta_xS^{-1}_{x,\tau}\theta^*_\tau(\mathcal{H})\subseteq\theta_x(\mathcal{H})\subseteq\ell^2(\mathbb{J}).$
If $ \theta_x(\mathcal{H})=\ell^2(\mathbb{J})$ and $ y \in\ell^2(\mathbb{J}), $ then there exists an $ h \in \mathcal{H}$ such that $y=\theta_xh .$ Now $ y=\theta_xS_{x,\tau}^{-1}S_{x,\tau}h=(\theta_xS_{x,\tau}^{-1}\theta_\tau^*)\theta_xh =P_{x,\tau}y.$ Definition of frame is symmetric, thus another  `if and only if'.
\end{proof}

\begin{proposition}
A frame  $ (\{x_j\}_{j\in \mathbb{J}}, \{\tau_j\}_{j\in \mathbb{J}}) $ for  $ \mathcal{H}$ is an orthonormal frame   if and only if it is a Parseval frame  and $ \langle x_j,\tau_k \rangle =\delta_{j,k},\forall j, k \in \mathbb{J}$. 
\end{proposition}
\begin{proof}
Let $\{e_j\}_{j\in \mathbb{J}} $ be the standard orthonormal basis for $ \ell^2(\mathbb{J})$.
$(\Rightarrow)$ $\langle x_j,\tau_k \rangle =\langle \theta_x^*e_j,\theta_\tau^*e_k \rangle =\langle e_j,\theta_x\theta_\tau^*e_k \rangle=\langle e_j,I_{\ell^2(\mathbb{J})}e_k \rangle=\delta_{j,k},\forall j, k \in \mathbb{J}.$
$(\Leftarrow)$ For all $\{a_j\}_{j\in \mathbb{J}}\in\ell^2(\mathbb{J})$, $\theta_x\theta_\tau^*(\{a_j\}_{j\in \mathbb{J}})=\theta_x(\sum_{j\in \mathbb{J}}a_j\tau_j)= \sum_{k\in \mathbb{J}}\langle \sum_{j\in \mathbb{J}}a_j\tau_j , x_k\rangle e_k =\sum_{k\in \mathbb{J}}a_ke_k$. Hence $ (\{x_j\}_{j\in \mathbb{J}}, \{\tau_j\}_{j\in \mathbb{J}}) $ is an orthonormal frame.

\end{proof}

Following is  a dilation theorem.
 \begin{theorem}\label{SEQUENTIALDILATION}
 Let $(\{x_j\}_{j\in \mathbb{J}},\{\tau_j\}_{j\in \mathbb{J}} )$ be  a Parseval frame  for  $ \mathcal{H}$ such that $ \theta_x(\mathcal{H})=\theta_\tau(\mathcal{H})$ and $P_{x,\tau}$ be  a projection. Then there exist a Hilbert space $ \mathcal{H}_1 $ which contains $ \mathcal{H}$ isometrically and  an orthonormal frame $(\{y_j\}_{j\in \mathbb{J}},\{\omega_j\}_{j\in \mathbb{J}} )$ for  $ \mathcal{H}_1$ such that $ x_j=Py_j,\tau_j=P\omega_j, \forall j \in \mathbb{J}$, where $P$ is the orthogonal projection from $\mathcal{H}_1$ onto $\mathcal{H}$. 
 \end{theorem}
 \begin{proof}
 Let  $\{e_j\}_{j\in \mathbb{J}} $ be the standard orthonormal basis for $ \ell^2(\mathbb{J})$, and define $ \mathcal{H}_1\coloneqq\mathcal{H}\oplus \theta_x(\mathcal{H})^\perp$. Then $\mathcal{H} \ni h \mapsto h\oplus 0 \in \mathcal{H}_1 $ is isometry.  Denote the orthogonal projection from $\mathcal{H}_1 $ onto $\mathcal{H} $ by $P$. Define $ y_j\coloneqq x_j\oplus P_{x,\tau}^\perp e_j, \omega_j\coloneqq\tau_j\oplus P_{x,\tau}^\perp e_j , \forall j \in \mathbb{J}$. Then, clearly $ Py_j=x_j, P\omega_j=\tau_j, \forall j \in \mathbb{J}$. Now $ \theta_y(h\oplus g)=\{\langle h\oplus g, x_j\oplus P_{x,\tau}^\perp e_j\rangle \}_{j\in \mathbb{J}}=\{\langle h, x_j \rangle +\langle g,P_{x,\tau}^\perp e_j\rangle \}_{j\in \mathbb{J}}=\{\langle h, x_j \rangle\}_{j\in \mathbb{J}} +\{\langle P_{x,\tau}^\perp g, e_j\rangle \}_{j\in \mathbb{J}}=\theta_xh+P_{x,\tau}^\perp g,\forall h\oplus g\in \mathcal{H}_1$ and $ \langle \theta_y^*(\{a_j\}_{j\in \mathbb{J}}), h\oplus g\rangle = \langle \{a_j\}_{j\in \mathbb{J}}, \theta_y(h\oplus g)\rangle=\langle \{a_j\}_{j\in \mathbb{J}}, \theta_xh+ P_{x,\tau}^\perp g\rangle =\langle \theta_x^*(\{a_j\}_{j\in \mathbb{J}}), h\rangle+\langle  P_{x,\tau}^\perp\{a_j\}_{j\in \mathbb{J}},  g\rangle =\langle \theta_x^*(\{a_j\}_{j\in \mathbb{J}})\oplus  P_{x,\tau}^\perp(\{a_j\}_{j\in \mathbb{J}}),h\oplus g\rangle, \forall \{a_j\}_{j\in \mathbb{J}} \in \ell^2(\mathbb{J}) , \forall h\oplus g\in \mathcal{H}_1$. Hence $ \theta_y(h\oplus g)=\theta_xh+P_{x,\tau}^\perp g,\forall h\oplus g\in \mathcal{H}_1,$ and $ \theta_y^*(\{a_j\}_{j\in \mathbb{J}})=\theta_x^*(\{a_j\}_{j\in \mathbb{J}})\oplus  P_{x,\tau}^\perp\{a_j\}_{j\in \mathbb{J}},\forall \{a_j\}_{j\in \mathbb{J}} \in \ell^2(\mathbb{J}).$ Similarly we can find $ \theta_\omega$ and $ \theta_\omega^*$. Therefore,  by  using $\theta_x(\mathcal{H})=\theta_\tau(\mathcal{H}) $ and $\theta_\tau^*P_{x,\tau}^\perp=0=P_{x,\tau}^\perp\theta_x ,$ we get  $ S_{y,\omega}(h\oplus g)= \theta_\omega^*(\theta_xh+  P_{x,\tau}^\perp g)=\theta_\tau^*(\theta_xh+ P_{x,\tau}^\perp g)\oplus  P_{x,\tau}^\perp(\theta_xh+ P_{x,\tau}^\perp g)=(S_{x,\tau}h+0)\oplus(0+ P_{x,\tau}^\perp g)=S_{x,\tau}h\oplus P_{x,\tau}^\perp g=I_\mathcal{H}h\oplus I_{\theta_x(\mathcal{H})^\perp} g, \forall h\oplus g\in \mathcal{H}_1$.  Hence $(\{y_j\}_{j\in \mathbb{J}},\{\omega_j\}_{j\in \mathbb{J}} )$ is a Parseval frame for  $ \mathcal{H}_1$. We see $ P_{y,\omega}(\{a_j\}_{j\in \mathbb{J}})=\theta_BS_{y,\omega}^{-1}(\{a_j\}_{j\in \mathbb{J}}), \forall \{a_j\}_{j\in \mathbb{J}} \in \ell^2(\mathbb{J})$. Hence $(\{y_j\}_{j\in \mathbb{J}},\{\omega_j\}_{j\in \mathbb{J}} )$ is a Riesz frame for  $ \mathcal{H}_1$. Therefore   $(\{y_j\}_{j\in \mathbb{J}},\{\omega_j\}_{j\in \mathbb{J}} )$ is an orthonormal frame for  $ \mathcal{H}_1$. 
  \end{proof}
 \begin{definition}
 A frame   $(\{y_j\}_{j\in \mathbb{J}}, \{\omega_j\}_{j\in \mathbb{J}})$  for  $\mathcal{H}$ is said to be a dual of frame  $ ( \{x_j\}_{j\in \mathbb{J}}, \{\tau_j\}_{j\in \mathbb{J}})$ for  $\mathcal{H}$  if $ \theta_\omega^*\theta_x= \theta_y^*\theta_\tau=I_{\mathcal{H}}$. The `frame' $(   \{\widetilde{x}_j\coloneqq S_{x,\tau}^{-1}x_j\}_{j\in \mathbb{J}},\{\widetilde{\tau}_j\coloneqq S_{x,\tau}^{-1}\tau_j\}_{j \in \mathbb{J}} )$, which is a `dual' of $ (\{x_j\}_{j\in \mathbb{J}}, \{\tau_j\}_{j\in \mathbb{J}})$ is called the canonical dual of $ (\{x_j\}_{j\in \mathbb{J}}, \{\tau_j\}_{j\in \mathbb{J}})$.
 \end{definition}
 Definition is symmetric, and  $ (\{y_j\}_{j\in \mathbb{J}}, \{\omega_j\}_{j\in \mathbb{J}})$ is a dual of $ (\{x_j\}_{j\in \mathbb{J}}, \{\tau_j\}_{j\in \mathbb{J}})$ if and only if both 
 $ (\{x_j\}_{j\in \mathbb{J}}, \{\omega_j\}_{j\in \mathbb{J}})$ and $ (\{\tau_j\}_{j\in \mathbb{J}}, \{y_j\}_{j\in \mathbb{J}})$ are Parseval frames. It is also true that if  $\{y_j\}_{j\in \mathbb{J}}, \{z_j\}_{j\in \mathbb{J}}\in \mathscr{F}_\tau $ are duals of $ \{x_j\}_{j\in \mathbb{J}} \in \mathscr{F}_\tau$, then the `frame' $(\left\{(y_j+z_j)/2\right\}_{j\in \mathbb{J}},\{\tau_j\}_{j\in \mathbb{J}}) $ is also a dual of $( \{x_j\}_{j\in \mathbb{J}}, \{\tau_j\}_{j\in \mathbb{J}}).$
 \begin{proposition}
 Let $( \{x_j\}_{j\in \mathbb{J}},\{\tau_j\}_{j\in \mathbb{J}} )$ be a frame for  $\mathcal{H}.$ If $ h \in \mathcal{H}$ has representation  $ h=\sum_{j\in\mathbb{J}}c_jx_j= \sum_{j\in\mathbb{J}}d_j\tau_j, $ for some scalar sequences  $ \{c_j\}_{j\in \mathbb{J}},\{d_j\}_{j\in \mathbb{J}}$,  then 
 $$ \sum\limits_{j\in \mathbb{J}}c_j\bar{d}_j =\sum\limits_{j\in \mathbb{J}}\langle h, \widetilde{\tau}_j\rangle\langle \widetilde{x}_j , h \rangle+\sum\limits_{j\in \mathbb{J}}(\langle c_j-\langle h, \widetilde{\tau}_j\rangle)(\bar{d}_j-\langle \widetilde{x}_j, h\rangle). $$
\end{proposition}
\begin{proof}
Right side $ =$
\begin{align*}
 &\sum\limits_{j\in \mathbb{J}}\langle S_{x, \tau}^{-1}h, \tau_j\rangle\langle x_j , S_{x, \tau}^{-1}h \rangle+\sum\limits_{j\in \mathbb{J}}c_j\bar{d}_j-\sum\limits_{j\in \mathbb{J}}c_j\langle x_j, S_{x, \tau}^{-1}h\rangle-\sum\limits_{j\in \mathbb{J}}\bar{d}_j\langle S_{x, \tau}^{-1}h, \tau_j\rangle+\sum\limits_{j\in \mathbb{J}}\langle S_{x, \tau}^{-1}h, \tau_j\rangle\langle x_j, S_{x, \tau}^{-1}h\rangle\\
&= 2\langle S_{x, \tau}S_{x, \tau}^{-1}h, S_{x, \tau}^{-1}h\rangle+\sum\limits_{j\in \mathbb{J}}c_j\bar{d}_j-\left\langle\sum\limits_{j\in \mathbb{J}}c_jx_j,S_{x, \tau}^{-1}h \right\rangle -\left\langle S_{x, \tau}^{-1}h,\sum\limits_{j\in \mathbb{J}} d_j\tau_j \right\rangle\\
&=2\langle h, S_{x, \tau}^{-1}h\rangle+\sum\limits_{j\in \mathbb{J}}c_j\bar{d}_j-\langle h, S_{x, \tau}^{-1}h\rangle-\langle  S_{x, \tau}^{-1}h, h\rangle
 =\text{Left side.}
 \end{align*}
 \end{proof} 
 \begin{theorem}\label{CANONICALDUALFRAMEPROPERTYSEQUENTIALVERSION}
 Let $( \{x_j\}_{j\in \mathbb{J}},\{\tau_j\}_{j\in \mathbb{J}} )$ be a frame for $ \mathcal{H}$ with frame bounds $ a$ and $ b.$ Then
  \begin{enumerate}[\upshape(i)]
 \item The canonical dual frame of the canonical dual frame  of $ (\{x_j\}_{j\in \mathbb{J}} ,\{\tau_j\}_{j\in \mathbb{J}} )$ is itself.
 \item$ \frac{1}{b}, \frac{1}{a}$ are frame bounds for the  canonical dual of $ (\{x_j\}_{j\in \mathbb{J}},\{\tau_j\}_{j\in \mathbb{J}}).$
 \item If $ a, b $ are optimal frame bounds for $( \{x_j\}_{j\in \mathbb{J}} , \{\tau_j\}_{j\in \mathbb{J}}),$ then $ \frac{1}{b}, \frac{1}{a}$ are optimal  frame bounds for its canonical dual.
 \end{enumerate} 
 \end{theorem} 
 \begin{proof}
  For $ h \in \mathcal{H},$ 
 $$ \sum\limits_{j\in \mathbb{J}}\langle h, \widetilde{x}_j \rangle\widetilde{\tau}_j = \sum\limits_{j\in \mathbb{J}}\langle h, S_{x,\tau}^{-1} x_j\rangle S_{x,\tau}^{-1}\tau_j = S_{x,\tau}^{-1}\left(\sum\limits_{j\in \mathbb{J}}\langle h, S_{x,\tau}^{-1} x_j\rangle \tau_j\right)=S_{x,\tau}^{-1}h.$$
 Thus the frame operator for the canonical dual $(\{\widetilde{x}_j\}_{j \in \mathbb{J}},   \{\widetilde{\tau}_j\}_{j\in \mathbb{J}} )$ is  $ S_{x,\tau}^{-1}. $
  Therefore, its canonical dual is $(\{ S_{x,\tau}S_{x,\tau}^{-1}x_j\}_{j \in \mathbb{J}},   \{ S_{x,\tau} S_{x,\tau}^{-1}\tau_j\}_{j\in \mathbb{J}} ).$ Others can be proved as in the  earlier  consideration `operator-valued frame'.
  \end{proof}
 \begin{proposition}
 Let  $ (\{x_j\}_{j\in \mathbb{J}}, \{\tau_j\}_{j\in \mathbb{J}}) $ and $ (\{y_j\}_{j\in \mathbb{J}}, \{\omega_j\}_{j\in \mathbb{J}}) $ be  frames for   $\mathcal{H}$. Then the following are equivalent.
 \begin{enumerate}[\upshape(i)]
 \item $ (\{y_j\}_{j\in \mathbb{J}},\{\omega_j\}_{j\in \mathbb{J}}) $ is a dual of $( \{x_j\}_{j\in \mathbb{J}}, \{\tau_j\}_{j\in \mathbb{J}}) $. 
 \item $ \sum_{j\in \mathbb{J}}\langle h, x_j\rangle \omega_j= \sum_{j\in \mathbb{J}}\langle h, \tau_j\rangle y_j=h, \forall h \in  \mathcal{H}.$ 
 \end{enumerate}
 \end{proposition}
 \begin{proof}
 $\theta_\omega^*\theta_xh= \sum_{j\in \mathbb{J}}\langle h, x_j\rangle \omega_j, \theta_y^*\theta_\tau h= \sum_{j\in \mathbb{J}}\langle h, \tau_j\rangle y_j.$
 \end{proof}
 \begin{theorem}
 Let $ (\{x_j\}_{j\in \mathbb{J}}, \{\tau_j\}_{j\in \mathbb{J}})$  be a  frame for   $\mathcal{H}$. If $ (\{x_j\}_{j\in \mathbb{J}}, \{\tau_j\}_{j\in \mathbb{J}})$ is a Riesz  basis  for   $\mathcal{H}$, then $ (\{x_j\}_{j\in \mathbb{J}}, \{\tau_j\}_{j\in \mathbb{J}}) $ has unique dual. Converse holds if $ \theta_x(\mathcal{H})=\theta_\tau(\mathcal{H})$.
 \end{theorem}
 \begin{proof}
 Let $ (\{y_j\}_{j\in \mathbb{J}}, \{\omega_j\}_{j\in \mathbb{J}})$ and  $ (\{z_j\}_{j\in \mathbb{J}}, \{\rho_j\}_{j\in \mathbb{J}})$ be  dual frames  of $ (\{x_j\}_{j\in \mathbb{J}}, \{\tau_j\}_{j\in \mathbb{J}})$.
 Since $ (\{x_j\}_{j\in \mathbb{J}}, \{\tau_j\}_{j\in \mathbb{J}})$ is a Riesz  basis, there exist  invertible $ U,V\in \mathcal{B}(\mathcal{H})$ and an orthonormal basis $\{f_j\}_{j\in \mathbb{J}} $ for $\mathcal{H}$ such that $x_j=Uf_j, \tau_j=Vf_j, \forall j \in \mathbb{J}$ and $ VU^*\geq 0.$ Then $ \sum_{j\in \mathbb{J}}\langle h,y_j-z_j \rangle \tau_j=\sum_{j\in \mathbb{J}}\langle h,y_j \rangle \tau_j-\sum_{j\in \mathbb{J}}\langle h,z_j \rangle \tau_j=h-h=0,\forall h\in \mathcal{H}$ $\Rightarrow$ $ U(\sum_{j\in \mathbb{J}}\langle h,y_j-z_j \rangle f_j)=0$ $\Rightarrow$ $\sum_{j\in \mathbb{J}}\langle h,y_j-z_j \rangle f_j=0,\forall j \in \mathbb{J}$ $\Rightarrow$ $\langle h,y_j-z_j \rangle=0,\forall h\in \mathcal{H},\forall j \in \mathbb{J}$ $\Rightarrow$ $ y_j=z_j,\forall j \in \mathbb{J}$. Similarly $ \omega_j=\rho_j, \forall j \in \mathbb{J}.$ Hence the dual of $(\{x_j\}_{j\in \mathbb{J}}, \{\tau_j\}_{j\in \mathbb{J}})$ is unique.

 To prove the converse, we further have  $ \theta_x(\mathcal{H})=\theta_\tau(\mathcal{H})$. Suppose $ (\{x_j\}_{j\in \mathbb{J}}, \{\tau_j\}_{j\in \mathbb{J}}) $ is not a Riesz basis. Then from Proposition \ref{RIESZFRAMECHARACTERIZATION}, $ \theta_x(\mathcal{H}) \subsetneq \ell^2(\mathbb{J})$. Let $P :\ell^2(\mathbb{J}) \rightarrow  \theta_x(\mathcal{H})^\perp $ be the orthogonal projection, $ T:\theta_x(\mathcal{H})^\perp \rightarrow \mathcal{H}$ be a nonzero bounded linear operator, and $\{e_j\}_{j\in \mathbb{J}} $ be the standard orthonormal basis for $\ell^2(\mathbb{J})$.  Since $\{Pe_j\}_{j\in \mathbb{J}} $ contains a Schauder basis for $ \theta_x(\mathcal{H})^\perp$, there exists $ k \in \mathbb{J}$ such that $ TPe_k\neq0$. Therefore if we define $y_j=S_{x,\tau}^{-1}x_j+TPe_j, \forall j \in \mathbb{J} $, then $\{y_j\}_{j\in \mathbb{J}} $ and  $\{S_{x,\tau}^{-1}x_j\}_{j\in \mathbb{J}} $ are different. Now define $\omega_j=S_{x,\tau}^{-1}\tau_j+TPe_j, \forall j \in \mathbb{J}.$ We see $\theta_y=\theta_xS_{x,\tau}^{-1}+PT^*$, $\theta_\omega=\theta_\tau S_{x,\tau}^{-1}+PT^* $ and $S_{y,\omega}=\theta_\omega^*\theta_y=( S_{x,\tau}^{-1}\theta_\tau^*+TP)(\theta_xS_{x,\tau}^{-1}+PT^*)=S_{x,\tau}^{-1}+S_{x,\tau}^{-1}0T^*+T0S_{x,\tau}^{-1}+TPT^*=S_{x,\tau}^{-1}+TPT^* $ which is positive invertible. Hence $(\{y_j\}_{j\in \mathbb{J}}, \{\omega_j\}_{j\in \mathbb{J}})$ is a frame for $\mathcal{H}$. This is a dual of $(\{x_j\}_{j\in \mathbb{J}}, \{\tau_j\}_{j\in \mathbb{J}})$. In fact, $\theta_x^*\theta_\omega=\theta_x^*(\theta_\tau S_{x,\tau}^{-1}+PT^*)=I_\mathcal{H}+0=I_\mathcal{H} $, and $ \theta_\tau^*\theta_y=\theta_\tau^*(\theta_xS_{x,\tau}^{-1}+PT^*) =I_\mathcal{H}+0=I_\mathcal{H}$. Therefore $(\{x_j\}_{j\in \mathbb{J}}, \{\tau_j\}_{j\in \mathbb{J}})$ has two  distinct duals, one - the canonical dual, and two - $(\{y_j\}_{j\in \mathbb{J}}, \{\omega_j\}_{j\in \mathbb{J}})$, which is a contradiction.
 \end{proof}
 
 \begin{proposition}
 Let $ (\{x_j\}_{j\in \mathbb{J}}, \{\tau_j\}_{j\in \mathbb{J}}) $  be a  frame for   $\mathcal{H}$. If $ (\{y_j\}_{j\in \mathbb{J}}, \{\omega_j\}_{j\in \mathbb{J}}) $ is a dual of $ (\{x_j\}_{j\in \mathbb{J}}, \{\tau_j\}_{j\in \mathbb{J}}) $, then there exist Bessel sequences $ \{z_j\}_{j\in \mathbb{J}}$ and $ \{\rho_j\}_{j\in \mathbb{J}} $ (w.r.t. themselves) for  $\mathcal{H}$ such that $ y_j=S_{x,\tau}^{-1}x_j+z_j, \omega_j=S_{x,\tau}^{-1}\tau_j+\rho_j,\forall j \in \mathbb{J}$,  and $\theta_z(\mathcal{H})\perp \theta_\tau(\mathcal{H}),\theta_\rho(\mathcal{H})\perp \theta_x(\mathcal{H})$. Converse holds if  $ \theta_\rho^*\theta_z \geq 0$.
 \end{proposition}
 \begin{proof}
 $(\Rightarrow)$ Define $z_j\coloneqq  y_j-S_{x,\tau}^{-1}x_j, \rho_j\coloneqq\omega_j-S_{x,\tau}^{-1}\tau_j+\rho_j,\forall j \in \mathbb{J}$. Since  $ \{x_j\}_{j\in \mathbb{J}} $, $ \{\tau_j\}_{j\in \mathbb{J}} $, $ \{y_j\}_{j\in \mathbb{J}} $, $ \{\rho_j\}_{j\in \mathbb{J}} $ are all Bessel (w.r.t. themselves), $ \{z_j\}_{j\in \mathbb{J}} $ and $ \{\rho_j\}_{j\in \mathbb{J}} $ are Bessel (w.r.t. themselves). Next, $\theta_z=\theta_y-\theta_xS_{x,\tau}^{-1}$, $\theta_\rho=\theta_\omega-\theta_\tau S_{x,\tau}^{-1}$. Now $ \theta_\tau^*\theta_z=\theta_\tau^*(\theta_y-\theta_xS_{x,\tau}^{-1})=I_\mathcal{H}-I_\mathcal{H}=0$, $ \theta_x^*\theta_\rho=\theta_x^*(\theta_\omega-\theta_\tau S_{x,\tau}^{-1})=I_\mathcal{H}-I_\mathcal{H}=0$.
 
 $(\Leftarrow)$ Surely  $ \{y_j\}_{j\in \mathbb{J}}$ and $ \{\omega_j\}_{j\in \mathbb{J}} $ are Bessel (w.r.t. themselves), $\theta_y=\theta_xS_{x,\tau}^{-1}+\theta_z$, $\theta_\omega=\theta_\tau S_{x,\tau}^{-1}+\theta_\rho$. From $ \theta_\rho^*\theta_z \geq 0$, $S_{y,\omega}=\theta_\omega^*\theta_y=(S_{x,\tau}^{-1}\theta_\tau^*+\theta_\rho^*)(\theta_xS_{x,\tau}^{-1}+\theta_z)=S_{x,\tau}^{-1}+0+0+\theta_\rho^*\theta_z\geq  S_{x,\tau}^{-1}$, hence $S_{y,\omega}$ is positive invertible. Frame $ (\{y_j\}_{j\in \mathbb{J}}, \{\omega_j\}_{j\in \mathbb{J}}) $ is a dual of $ (\{x_j\}_{j\in \mathbb{J}}, \{\tau_j\}_{j\in \mathbb{J}}) $. Indeed, $\theta_\tau^*\theta_y=\theta_\tau^*(\theta_xS_{x,\tau}^{-1}+\theta_z)=I_\mathcal{H}+0$, $\theta_x^*\theta_\omega=\theta_x^*(\theta_\tau S_{x,\tau}^{-1}+\theta_\rho)=I_\mathcal{H}+0 $.
 \end{proof}
 \begin{lemma}\label{DUALCHARATERIZATIONLEMMA1}
 Let  $ (\{x_j\}_{j\in \mathbb{J}}, \{\tau_j\}_{j\in \mathbb{J}}) $  be a  frame for   $\mathcal{H}$ and $ \{e_j\}_{j\in \mathbb{J}}$ be the standard orthonormal basis for $ \ell^2(\mathbb{J})$. Then the dual frames  of $ (\{x_j\}_{j\in \mathbb{J}}, \{\tau_j\}_{j\in \mathbb{J}})$ are precisely $ (\{y_j=Ue_j\}_{j\in \mathbb{J}}, \{\omega_j=Ve_j\}_{j\in \mathbb{J}})$, where $ U,V: \ell^2(\mathbb{J}) \rightarrow \mathcal{H}$ are bounded left-inverses of $ \theta_\tau, \theta_x$, respectively, such that $ VU^*$ is positive invertible.
\end{lemma}
 \begin{proof}
 $(\Leftarrow)$ We see $ \theta_yh=\{\langle h, Ue_j \rangle\}_{{j\in \mathbb{J}}} =\sum_{j\in \mathbb{J}}\langle h, Ue_j \rangle e_j=\sum_{j\in \mathbb{J}}\langle U^*h, e_j \rangle e_j=U^*h, \forall h \in \mathcal{H}.$ Similarly $ \theta_\omega=V^*$. Then  $ S_{y,\omega}=\theta_\omega^*\theta_y=VU^*$ is positive invertible. Hence $ (\{y_j\}_{j\in \mathbb{J}}, \{\omega_j\}_{j\in \mathbb{J}})$ is a  frame. We further see $ \theta_y^*\theta_\tau=U\theta_\tau=I_\mathcal{H},\theta_\omega^*\theta_x=V\theta_x=I_\mathcal{H}.$ Hence $ (\{y_j\}_{j\in \mathbb{J}}, \{\omega_j\}_{j\in \mathbb{J}})$ is a dual of  $ (\{x_j\}_{j\in \mathbb{J}}, \{\tau_j\}_{j\in \mathbb{J}})$.
 
 $(\Rightarrow)$ Let $ (\{y_j\}_{j\in \mathbb{J}}, \{\omega_j\}_{j\in \mathbb{J}})$ be a dual frame of $ (\{x_j\}_{j\in \mathbb{J}}, \{\tau_j\}_{j\in \mathbb{J}})$.  Then $ \theta_y^*\theta_\tau=I_\mathcal{H}=\theta_\omega^*\theta_x$. Define $ U\coloneqq\theta_y^*, V\coloneqq\theta_\omega^*.$ Then  $ U,V: \ell^2(\mathbb{J}) \rightarrow \mathcal{H}$ are bounded left-inverses of $ \theta_\tau, \theta_x$, respectively, such that $ VU^*=\theta_\omega^*\theta_y=S_{y,\omega}$ is positive invertible. Further, $ Ue_j=\theta_y^*e_j=y_j, Ve_j=\theta_\omega^*e_j=\omega_j,\forall j \in \mathbb{J}$.
 \end{proof}
\begin{lemma}\label{DUALCHARATERIZATIONLEMMA2}
 Let $ (\{x_j\}_{j\in \mathbb{J}}, \{\tau_j\}_{j\in \mathbb{J}}) $  be a  frame for   $\mathcal{H}$. Then the bounded left-inverses of 
\begin{enumerate}[\upshape(i)]
\item $ \theta_x$ are precisely $S_{x,\tau}^{-1}\theta_\tau^*+U(I_{\ell^2(\mathbb{J})}-\theta_xS_{x,\tau}^{-1}\theta_\tau^*)$, where $U\in \mathcal{B}( \ell^2(\mathbb{J}), \mathcal{H})$.
\item $ \theta_\tau$ are precisely $S_{x,\tau}^{-1}\theta_x^*+V(I_{\ell^2(\mathbb{J})}-\theta_\tau S_{x,\tau}^{-1}\theta_x^*)$, where $V\in \mathcal{B}( \ell^2(\mathbb{J}),\mathcal{H})$. 
\end{enumerate}	
 \end{lemma}
 \begin{proof}
 We prove (i). $(\Leftarrow)$ Let $U: \ell^2(\mathbb{J})\rightarrow \mathcal{H}$ be a bounded operator. Then $(S_{x,\tau}^{-1}\theta_\tau^*+U(I_{\ell^2(\mathbb{J})}-\theta_xS_{x,\tau}^{-1}\theta_\tau^*))\theta_x=I_\mathcal{H}+U\theta_x-U\theta_xI_\mathcal{H}=I_\mathcal{H}$. Therefore  $S_{x,\tau}^{-1}\theta_\tau^*+U(I_{\ell^2(\mathbb{J})}-\theta_xS_{x,\tau}^{-1}\theta_\tau^*)$ is a bounded left-inverse of $\theta_x$.
  
 $(\Rightarrow)$ Let $ L:\ell^2(\mathbb{J})\rightarrow \mathcal{H}$ be a bounded left-inverse of $ \theta_x$. Define $U\coloneqq L$. Then $S_{x,\tau}^{-1}\theta_\tau^*+U(I_{\ell^2(\mathbb{J})}-\theta_xS_{x,\tau}^{-1}\theta_\tau^*) =S_{x,\tau}^{-1}\theta_\tau^*+L(I_{\ell^2(\mathbb{J})}-\theta_xS_{x,\tau}^{-1}\theta_\tau^*)=S_{x,\tau}^{-1}\theta_\tau^*+L-I_{\mathcal{H}}S_{x,\tau}^{-1}\theta_\tau^*= L$.  
\end{proof}
\begin{theorem}
Let  $ (\{x_j\}_{j\in \mathbb{J}}, \{\tau_j\}_{j\in \mathbb{J}}) $  be a  frame for   $\mathcal{H}$. The dual frames 	 $ (\{y_j\}_{j\in \mathbb{J}}, \{\omega_j\}_{j\in \mathbb{J}}) $ of $ (\{x_j\}_{j\in \mathbb{J}}, \{\tau_j\}_{j\in \mathbb{J}}) $ are precisely  
 \begin{align*}
 (\{y_j=S_{x,\tau}^{-1}x_j+Ve_j-V\theta_\tau S_{x,\tau}^{-1}x_j\}_{j\in \mathbb{J}},
 \{\omega_j=S_{x,\tau}^{-1}\tau_j+Ue_j-U\theta_xS_{x,\tau}^{-1}\tau_j\}_{j\in \mathbb{J}})
 \end{align*}
 such that 
 $$S_{x,\tau}^{-1}+UV^*-U\theta_xS_{x,\tau}^{-1}\theta_\tau^*V^* $$
 is positive invertible, where  $ \{e_j\}_{j\in \mathbb{J}}$ is  the standard orthonormal basis for $ \ell^2(\mathbb{J})$, and $U, V \in \mathcal{B} (\ell^2(\mathbb{J}),\mathcal{H})$.
\end{theorem}
\begin{proof}
 From Lemma \ref{DUALCHARATERIZATIONLEMMA1} and Lemma \ref{DUALCHARATERIZATIONLEMMA2} we can spell the characterization   for the dual frames of  $ (\{x_j\}_{j\in \mathbb{J}}, \{\tau_j\}_{j\in \mathbb{J}}) $ as the families 
 \begin{align*}
  &(\left\{y_j=S_{x,\tau}^{-1}\theta_x^*e_j+Ve_j-V\theta_\tau S_{x,\tau}^{-1}\theta_x^*e_j=S_{x,\tau}^{-1}x_j+Ve_j-V\theta_\tau S_{x,\tau}^{-1}x_j\right\}_{j\in \mathbb{J}},\\
 &\left\{\omega_j=S_{x,\tau}^{-1}\theta_\tau^*e_j+Ue_j-U\theta_xS_{x,\tau}^{-1}\theta_\tau^*e_j=S_{x,\tau}^{-1}\tau_j+Ue_j-U\theta_xS_{x,\tau}^{-1}\tau_j\right\}_{j\in \mathbb{J}})
  \end{align*}
  such that 
   $$ (S_{x,\tau}^{-1}\theta_\tau^*+U(I_{\ell^2(\mathbb{J})}-\theta_xS_{x,\tau}^{-1}\theta_\tau^*))(S_{x,\tau}^{-1}\theta_x^*+V(I_{\ell^2(\mathbb{J})}-\theta_\tau S_{x,\tau}^{-1}\theta_x^*))^*$$
  is positive invertible, where  $ \{e_j\}_{j\in \mathbb{J}}$ is  the standard orthonormal basis for $ \ell^2(\mathbb{J})$, and  $U, V \in \mathcal{B} (\ell^2(\mathbb{J}),\mathcal{H})$.  A direct expansion gives 
  \begin{align*}
  &(S_{x,\tau}^{-1}\theta_\tau^*+U(I_{\ell^2(\mathbb{J})}-\theta_xS_{x,\tau}^{-1}\theta_\tau^*))(S_{x,\tau}^{-1}\theta_x^*+V(I_{\ell^2(\mathbb{J})}-\theta_\tau S_{x,\tau}^{-1}\theta_x^*))^*\\
 &=(S_{x,\tau}^{-1}\theta_\tau^*+U(I_{\ell^2(\mathbb{J})}-\theta_xS_{x,\tau}^{-1}\theta_\tau^*))(\theta_xS_{x,\tau}^{-1}+(I_{\ell^2(\mathbb{J})}-\theta_x S_{x,\tau}^{-1}\theta_\tau^*)V^*)
 =S_{x,\tau}^{-1}+UV^*-U\theta_xS_{x,\tau}^{-1}\theta_\tau^*V^*.
  \end{align*}
 \end{proof}
\begin{definition}
 A frame   $(\{y_j\}_{j\in \mathbb{J}},  \{\omega_j\}_{j\in \mathbb{J}})$  for  $\mathcal{H}$ is said to be orthogonal to a frame   $( \{x_j\}_{j\in \mathbb{J}}, \{\tau_j\}_{j\in \mathbb{J}})$ for $\mathcal{H}$ if $\theta_\omega^*\theta_x= \theta_y^*\theta_\tau=0.$
 \end{definition}
 Orthogonality is symmetric. Similar to the observation we made in Section \ref{MK}, dual frames cannot be orthogonal to each other and orthogonal frames can not be dual to each other. If $ (\{y_j\}_{j\in \mathbb{J}}, \{\omega_j\}_{j\in \mathbb{J}})$ is orthogonal to $ (\{x_j\}_{j\in \mathbb{J}}, \{\tau_j\}_{j\in \mathbb{J}})$, then  both $ (\{x_j\}_{j\in \mathbb{J}}, \{\omega_j\}_{j\in \mathbb{J}})$ and $ (\{y_j\}_{j\in \mathbb{J}}, \{\tau_j\}_{j\in \mathbb{J}})$ are not frames.
 \begin{proposition}
 Let  $ (\{x_j\}_{j\in \mathbb{J}}, \{\tau_j\}_{j\in \mathbb{J}}) $, $ (\{y_j\}_{j\in \mathbb{J}}, \{\omega_j\}_{j\in \mathbb{J}})$ be  frames for  $\mathcal{H}$. Then the following are equivalent.
 \begin{enumerate}[\upshape(i)]
 \item $(\{y_j\}_{j\in \mathbb{J}},\{\omega_j\}_{j\in \mathbb{J}}) $ is orthogonal to  $( \{x_j\}_{j\in \mathbb{J}},  \{\tau_j\}_{j\in \mathbb{J}}) $.
 \item $  \sum_{j\in \mathbb{J}}\langle h, x_j\rangle \omega_j=0=\sum_{j\in \mathbb{J}}\langle g, \tau_j\rangle y_j, \forall h \in  \mathcal{H}$. 
\end{enumerate}
\end{proposition}
\begin{proposition}
Two orthogonal frames  have common dual frame.	
\end{proposition}
\begin{proof}
Let  $ (\{x_j\}_{j\in \mathbb{J}}, \{\tau_j\}_{j\in \mathbb{J}}) $ and $ (\{y_j\}_{j\in \mathbb{J}}, \{\omega_j\}_{j\in \mathbb{J}}) $ be  two orthogonal frames for  $\mathcal{H}$. Define $ z_j\coloneqq S_{x,\tau}^{-1}x_j+S_{y,\omega}^{-1}y_j,\rho_j\coloneqq S_{x,\tau}^{-1}\tau_j+S_{y,\omega}^{-1}\omega_j, \forall j \in \mathbb{J}$. Then $ \theta_z=\theta_xS_{x,\tau}^{-1}+\theta_yS_{y,\omega}^{-1}, \theta_\rho=\theta_\tau S_{x,\tau}^{-1}+\theta_\omega S_{y,\omega}^{-1}$,  $ S_{z,\rho}=\theta_\rho^*\theta_z=(S_{x,\tau}^{-1}\theta_\tau^* +S_{y,\omega}^{-1}\theta_\omega^* )(\theta_xS_{x,\tau}^{-1}+\theta_yS_{y,\omega}^{-1})=S_{x,\tau}^{-1}+S_{y,\omega}^{-1} $ which is positive and $ \langle S_{z,\rho}h, h\rangle =\langle S_{x,\tau}^{-1}h,h \rangle +\langle S_{y,\omega}^{-1}h, h\rangle \geq \min\left\{ \|S_{x,\tau}\|^{-1},\|S_{y,\omega}\|^{-1}\right\}\|h\|^2, \forall h \in \mathcal{H}$, hence $S_{z,\rho}$ is invertible. Therefore $(\{z_j\}_{j\in \mathbb{J}}, \{\rho_j\}_{j\in \mathbb{J}})$ is a frame for $\mathcal{H}$. This is a common dual of  $ (\{x_j\}_{j\in \mathbb{J}}, \{\tau_j\}_{j\in \mathbb{J}}) $ and $ (\{y_j\}_{j\in \mathbb{J}}, \{\omega_j\}_{j\in \mathbb{J}}).$ In fact, $\theta_z^*\theta_\tau=(S_{x,\tau}^{-1}\theta_x^*+S_{y,\omega}^{-1}\theta_y^*)\theta_\tau=I_\mathcal{H}+0, \theta_\rho^*\theta_x=(S_{x,\tau}^{-1}\theta_\tau^* +S_{y,\omega}^{-1}\theta_\omega^*)\theta_x=I_\mathcal{H}+0$, and $\theta_z^*\theta_\omega= (S_{x,\tau}^{-1}\theta_x^*+S_{y,\omega}^{-1}\theta_y^*)\theta_\omega=0+I_\mathcal{H}, \theta_\rho^*\theta_y=(S_{x,\tau}^{-1}\theta_\tau^* +S_{y,\omega}^{-1}\theta_\omega^*)\theta_y=0+I_\mathcal{H} $.
\end{proof}
Following is an interpolation result.
\begin{proposition}
Let $ (\{x_j\}_{j\in \mathbb{J}}, \{\tau_j\}_{j\in \mathbb{J}}) $ and $ (\{y_j\}_{j\in \mathbb{J}}, \{\omega_j\}_{j\in \mathbb{J}}) $ be  two Parseval frames for  $\mathcal{H}$ which are  orthogonal. If $A,B,C,D \in \mathcal{B}(\mathcal{H})$ are such that $ AC^*+BD^*=I_\mathcal{H}$, then  $ (\{Ax_j+By_j\}_{j\in \mathbb{J}}, \{C\tau_j+D\omega_j\}_{j\in \mathbb{J}}) $ is a  Parseval frame for  $\mathcal{H}$. In particular,  if scalars $ a,b,c,d$ satisfy $a\bar{c}+b\bar{d} =1$, then $ (\{ax_j+by_j\}_{j\in \mathbb{J}}, \{c\tau_j+d\omega_j\}_{j\in \mathbb{J}}) $ is a  Parseval frame for  $\mathcal{H}$.
\end{proposition} 
\begin{proof}
We see $ \theta_{Ax+By} h = \{\langle h,Ax_j+By_j \rangle\}_{j\in \mathbb{J}}=\{\langle A^*h,x_j \rangle\}_{j\in \mathbb{J}}+\{\langle B^*h,y_j \rangle\}_{j\in \mathbb{J}}=\theta_xA^*h+\theta_yB^*h,\forall h \in \mathcal{H}$ and similarly $ \theta_{C\tau+D\omega}=\theta_\tau C^*+\theta_\omega D^* $. Therefore 	$S_{Ax+By,C\tau+D\omega} =\theta^*_{C\tau+D\omega} \theta_{Ax+By}= (\theta_\tau C^*+\theta_\omega D^*)^*(\theta_xA^*+\theta_yB^*)=C\theta_\tau^*\theta_xA^*+C\theta_\tau^*\theta_yB^*+D\theta_\omega^*\theta_xA^*+D\theta_\omega^*\theta_yB^*=CS_{x,\tau}A^*+C0B^*+D0A^*+DS_{y,\omega}B^*=CI_\mathcal{H}A^*+DI_\mathcal{H}B^*=I_\mathcal{H}.$
\end{proof}

\begin{definition}
Two frames  $(\{x_j\}_{j\in \mathbb{J}},  \{\tau_j\}_{j\in \mathbb{J}}) $ and $(\{y_j\}_{j\in \mathbb{J}},\{\omega_j\}_{j\in \mathbb{J}})$  for $ \mathcal{H}$ are called disjoint if $(\{x_j\oplus y_j\}_{j\in \mathbb{J}},\{\tau_j\oplus\omega_j\}_{j\in \mathbb{J}})$ is a frame for $\mathcal{H}\oplus\mathcal{H} $.
\end{definition} 
\begin{proposition}\label{SEQUENTIALDISJOINTFRAMEPROPOSITION}
If $(\{x_j\}_{j\in \mathbb{J}},\{\tau_j\}_{j\in \mathbb{J}} )$  and $ (\{y_j\}_{j\in \mathbb{J}}, \{\omega_j\}_{j\in \mathbb{J}} )$  are  disjoint  frames  for $\mathcal{H}$, then  they  are disjoint. Further, if both $(\{x_j\}_{j\in \mathbb{J}},\{\tau_j\}_{j\in \mathbb{J}} )$  and $ (\{y_j\}_{j\in \mathbb{J}}, \{\omega_j\}_{j\in \mathbb{J}} )$ are  Parseval, then $(\{x_j\oplus y_j\}_{j \in \mathbb{J}},\{\tau_j\oplus \omega_j\}_{j \in \mathbb{J}})$ is Parseval.
\end{proposition} 
\begin{proof}
We find $(S_{x,\tau}\oplus S_{y,\omega})(h\oplus g) =S_{x,\tau}h\oplus S_{y,\omega}g  =\sum_{j\in \mathbb{J}}\langle h, x_j\rangle\tau_j\oplus\sum_{j\in \mathbb{J}}\langle g, y_j\rangle\omega_j=\left(\sum_{j\in \mathbb{J}}\langle h, x_j\rangle\tau_j+\sum_{j\in \mathbb{J}}\langle g, y_j\rangle\tau_j\right)\oplus \left(\sum_{j\in \mathbb{J}}\langle h, x_j\rangle\omega_j+\sum_{j\in \mathbb{J}}\langle g, y_j\rangle\omega_j\right) =\sum_{j\in \mathbb{J}}(\langle h, x_j \rangle+\langle g, y_j\rangle)(\tau_j\oplus\omega_j)=\sum_{j\in \mathbb{J}}\langle h\oplus g,x_j\oplus y_j\rangle(\tau_j\oplus\omega_j)= S_{x\oplus y, \tau \oplus \omega}(h\oplus g), \forall h\oplus g \in \mathcal{H}\oplus\mathcal{H} $. Thus $S_{x\oplus y, \tau \oplus \omega}=S_{x,\tau}\oplus S_{y,\omega} $ is bounded,  positive  and   $ S_{x\oplus y, \tau \oplus \omega}^{-1}=S_{x,\tau}^{-1}\oplus S_{y,\omega}^{-1}$. Next, $\theta_{x\oplus y}(h\oplus g)=\{\langle h\oplus g, x_j\oplus y_j\rangle \}_{j\in \mathbb{J}}=\{\langle h,x_j\rangle +\langle g,y_j\rangle\}_{j\in \mathbb{J}} =\{\langle h,x_j\rangle \}_{j\in \mathbb{J}}+\{\langle g,y_j \rangle \}_{j\in \mathbb{J}}=\theta_xh+\theta_yg $ exists for all $ h\oplus g \in \mathcal{H}\oplus\mathcal{H}$, $ \|\theta_{x\oplus y}(h\oplus g)\|=\|\theta_xh+\theta_yg\|\leq\|\theta_x\|\|h\|+\|\theta_y\|\|g\|\leq \max\{\|\theta_x\|,\|\theta_y\|\}(\|h\|+\|g\|)\leq 2\max\{\|\theta_x\|,\|\theta_y\|\}(\|h\|^2+\|g\|^2)^{1/2}=2\max\{\|\theta_x\|,\|\theta_y\|\}\|h\oplus g\| , \forall h\oplus g \in \mathcal{H}\oplus\mathcal{H}$ and similarly $\theta_{\tau\oplus \omega} $ exists and is also bounded. Therefore $(\{x_j\oplus y_j\}_{j\in \mathbb{J}}, \{\tau_j\oplus \omega_j\}_{j\in \mathbb{J}}) $ is a frame for $\mathcal{H}\oplus\mathcal{H} $. From the expression of $S_{x\oplus y, \tau \oplus \omega} $ we get the last conclusion.
\end{proof}
  
\textbf{Characterizations for sequential version} 
\begin{theorem}\label{SEQUENTIALCHARACTERIZATIONHILBERT1}
Let $ \{f_j\}_{j \in \mathbb{J}}$ be an arbitrary orthonormal basis for $ \mathcal{H}.$  Then 
\begin{enumerate}[\upshape(i)]
\item The orthonormal  bases  $ ( \{x_j\}_{j \in \mathbb{J}},\{\tau_j\}_{j \in \mathbb{J}})$ for $ \mathcal{H}$  are precisely $( \{Uf_j\}_{j \in \mathbb{J}},\{c_jUf_j\}_{j \in \mathbb{J}}) $, where $ U\in \mathcal{B}(\mathcal{H}) $ is unitary and $ c_j \in \mathbb{R}, \forall j \in \mathbb{J}$ such that $ 0<\inf\{c_j\}_{j\in \mathbb{J}}\leq \sup\{c_j\}_{j\in \mathbb{J}}< \infty.$
\item The Riesz bases  $ ( \{x_j\}_{j \in \mathbb{J}},\{\tau_j\}_{j \in \mathbb{J}})$ for $ \mathcal{H}$  are precisely $( \{Uf_j\}_{j \in \mathbb{J}},\{Vf_j\}_{j \in \mathbb{J}}) $, where $ U,V  \in \mathcal{B}(\mathcal{H}) $ are  invertible  such that  $ VU^*$ is positive. 
\item The frames $ ( \{x_j\}_{j \in \mathbb{J}},\{\tau_j\}_{j \in \mathbb{J}})$ for $ \mathcal{H}$  are precisely $( \{Uf_j\}_{j \in \mathbb{J}},\{Vf_j\}_{j \in \mathbb{J}}) $, where $ U,V \in \mathcal{B}(\mathcal{H}) $ are   such that  $ VU^*$ is positive invertible.
\item The Bessel sequences  $ ( \{x_j\}_{j \in \mathbb{J}},\{\tau_j\}_{j \in \mathbb{J}})$ for $ \mathcal{H}$  are precisely $( \{Uf_j\}_{j \in \mathbb{J}},\{Vf_j\}_{j \in \mathbb{J}}) $, where $ U,V  \in \mathcal{B}(\mathcal{H}) $ are such that   $ VU^*$ is positive.
\item The Riesz frames $ ( \{x_j\}_{j \in \mathbb{J}},\{\tau_j\}_{j \in \mathbb{J}})$ for $ \mathcal{H}$  are precisely $( \{Uf_j\}_{j \in \mathbb{J}},\{Vf_j\}_{j \in \mathbb{J}}) $, where $ U,V \in \mathcal{B}(\mathcal{H})$ are such that  $ VU^*$ is positive invertible and $ U^*(VU^*)^{-1}V=I_{\mathcal{H}}$.  
\item The orthonormal frames $ ( \{x_j\}_{j \in \mathbb{J}},\{\tau_j\}_{j \in \mathbb{J}})$ for $ \mathcal{H}$  are precisely $( \{Uf_j\}_{j \in \mathbb{J}},\{Vf_j\}_{j \in \mathbb{J}}) $, where $ U,V  \in \mathcal{B}(\mathcal{H}) $ are such that  $ VU^*=I_\mathcal{H}= U^*V$.  
\end{enumerate}
\end{theorem}
\begin{proof}
\begin{enumerate}[\upshape(i)]
\item Since a unitary operator carries orthonormal basis to orthonormal basis, collections of the form $( \{Uf_j\}_{j \in \mathbb{J}},\{c_jUf_j\}_{j \in \mathbb{J}}) $, $ U : \mathcal{H} \rightarrow  \mathcal{H} $ is unitary,  $ c_j \in \mathbb{R}, \forall j \in \mathbb{J}$,  $ 0<\inf\{c_j\}_{j\in \mathbb{J}}\leq \sup\{c_j\}_{j\in \mathbb{J}}< \infty$ are orthonormal bases. For the other way, let $ ( \{x_j\}_{j \in \mathbb{J}},\{\tau_j\}_{j \in \mathbb{J}})$ be an orthonormal basis for $ \mathcal{H}.$ Now we may assume $ \{x_j\}_{j \in \mathbb{J}}$ is an orthonormal basis for $ \mathcal{H}$ and $ \tau_j=c_jx_j, \forall j \in \mathbb{J},$ for some $ c_j \in \mathbb{R}, \forall j \in \mathbb{J}$ with   $ 0<\inf\{c_j\}_{j\in \mathbb{J}}\leq \sup\{c_j\}_{j\in \mathbb{J}}< \infty.$ Since both $ \{f_j\}_{j \in \mathbb{J}}$  and $\{x_j\}_{j \in \mathbb{J}} $ are orthonormal bases for $ \mathcal{H}$, there exists a unitary $ U: \mathcal{H} \rightarrow  \mathcal{H}$ such that $ Uf_j=x_j,  \forall j \in \mathbb{J}.$ This gives $ c_jUf_j= c_jx_j=\tau_j,  \forall j \in \mathbb{J}.$
\item  We need to prove only the direct part.  Let $ \{e_j\}_{j \in \mathbb{J}}$ be an orthonormal basis for $ \mathcal{H}$ and $ R,S:\mathcal{H}\rightarrow\mathcal{H} $  be bounded invertible such that $x_j=Re_j, \tau_j=Se_j, \forall j \in \mathbb{J}$ and $ SR^*\geq 0.$ Let $T:\mathcal{H}\rightarrow\mathcal{H} $ be the unitary operator obtained by defining $ Tf_j\coloneqq e_j, \forall j \in \mathbb{J}.$ Define $ U\coloneqq RT, V\coloneqq ST.$ Then $ U,V$ are invertible, $ Uf_j=RTf_j=Re_j=x_j, Vf_j=STf_j=Se_j=\tau_j, \forall j \in \mathbb{J}$ and $ VU^*=STT^*R^*=SR^*\geq 0.$
\item  $ (\Leftarrow)$ $ \theta_xh=\{\langle h , x_j\rangle \}_{j \in \mathbb{J}}=\{\langle U^*h , f_j\rangle \}_{j \in \mathbb{J}}= \theta_f(U^*h), \theta_\tau h=\theta_f(V^*h) , \forall h\in \mathcal{H}$. We find $\theta_f^*\theta_fh=\theta_f^*(\{\langle h , f_j\rangle \}_{j \in \mathbb{J}}) =\sum_{j\in \mathbb{J}}\langle h , f_j\rangle f_j=h, \forall h \in \mathcal{H}$. Hence  $ S_{x,\tau}=\theta_\tau^*\theta_x=V\theta_f^*\theta_fU^*=VU^*$ is positive invertible.
  
$(\Rightarrow)$
Let $ \{e_j\}_{j \in \mathbb{J}}$ be the standard orthonormal basis for $ \ell^2(\mathbb{J}).$ Let $ T : \mathcal{H} \rightarrow  \ell^2(\mathbb{J})$ be the unitary  isomorphism obtained by defining $ Tf_j\coloneqq e_j.$ Define $U\coloneqq \theta_x^*T,  V\coloneqq \theta_\tau^*T $. Then $ Uf_j=\theta_x^*Tf_j=\theta_x^*e_j=x_j,  Vf_j=\theta_\tau^*Tf_j=\theta_\tau^*e_j=\tau_j, \forall j \in \mathbb{J}$ and $ VU^*=\theta_\tau^*TT^*\theta_x=\theta_\tau^*I_{\ell^2(\mathbb{J})}\theta_x=S_{x,\tau}$ which is positive invertible.
\item Similar to (iii).
\item We refer to the proof of (iii). $(\Leftarrow)$ We first find $ \theta_f\theta_f^*(\{a_j\}_{j\in \mathbb{J}})=\sum_{j\in \mathbb{J}}a_j\theta_ff_j=\sum_{j\in \mathbb{J}}a_j(\sum_{k\in \mathbb{J}}\langle f_j, f_k\rangle e_k)=\sum_{j\in \mathbb{J}}a_je_j=\{a_j\}_{j\in \mathbb{J}}, \forall \{a_j\}_{j\in \mathbb{J}} \in \ell^2(\mathbb{J})$. Then $ P_{x,\tau}=\theta_xS_{x,\tau}^{-1}\theta_\tau^*=\theta_fU^*(VU^*)^{-1}V\theta_f^*=\theta_fI_\mathcal{H}\theta_f^*=I_{\ell^2(\mathbb{J})}$. 

$(\Rightarrow)$ $U^*(VU^*)^{-1}V=( T^*\theta_x)S_{x,\tau}^{-1}(\theta_\tau^*V)=T^*P_{x,\tau}T=T^*I_{\ell^2(\mathbb{J})}T=T^*T= I_\mathcal{H}.$
\item Naturally, we  refer to the proof of (v). $ (\Leftarrow)$ $S_{x,\tau}=VU^*=I_\mathcal{H}, P_{x,\tau}=\theta_xS_{x,\tau}^{-1}\theta_\tau^*=\theta_x\theta_\tau^*=\theta_fU^*V\theta_f^* =\theta_fI_\mathcal{H}\theta_f^* =\theta_f\theta_f^* =I_{\ell^2(\mathbb{J})}.$

$(\Rightarrow)$ $VU^*=S_{x,\tau}=I_\mathcal{H},U^*V=T^*\theta_x\theta_\tau^*T =T^*P_{x,\tau}T=T^*I_{\ell^2(\mathbb{J})}T=I_\mathcal{H}.  $
\end{enumerate}
\end{proof}  
\begin{corollary}
\begin{enumerate}[\upshape(i)]
\item If $ (\{x_j\}_{j \in \mathbb{J}},\{\tau_j=c_jx_j\}_{j \in \mathbb{J}})$ is an orthonormal basis for $ \mathcal{H}$, then $ \|x_j\|=1, \forall j \in \mathbb{J}, \|\tau_j\|=c_j, \forall j \in \mathbb{J}.$
\item If $ (\{x_j\}_{j \in \mathbb{J}},\{\tau_j\}_{j \in \mathbb{J}})$ is a Riesz basis for $ \mathcal{H}$, then
$$ \frac{1}{\|U^{-1}\|} \leq \|x_j\| \leq \|U\|, ~\forall j \in \mathbb{J},~   \frac{1}{\|V^{-1}\|} \leq \|\tau_j\| \leq \|V\|, ~\forall j \in \mathbb{J}.$$
\item If $ (\{x_j\}_{j \in \mathbb{J}},\{\tau_j\}_{j \in \mathbb{J}})$ is a Bessel sequence for $ \mathcal{H}$, then
$  \|x_j\| \leq \|U\|, \forall j \in \mathbb{J},  \|\tau_j\| \leq \|V\|, \forall j \in \mathbb{J}.$
\end{enumerate}
\end{corollary}
\begin{corollary}\label{CORRECTEXTENSION}
Let $ \{f_j\}_{j \in \mathbb{J}}$ be an arbitrary orthonormal basis for $ \mathcal{H}.$  Then 
\begin{enumerate}[\upshape(i)]
\item The orthonormal  bases  $ ( \{x_j\}_{j \in \mathbb{J}},\{x_j\}_{j \in \mathbb{J}})$ for $ \mathcal{H}$  are precisely $( \{Uf_j\}_{j \in \mathbb{J}},\{Uf_j\}_{j \in \mathbb{J}}) $, where $ U\in \mathcal{B}(\mathcal{H}) $ is unitary.
\item The Riesz bases  $ ( \{x_j\}_{j \in \mathbb{J}},\{x_j\}_{j \in \mathbb{J}})$ for $ \mathcal{H}$  are precisely $( \{Uf_j\}_{j \in \mathbb{J}},\{Uf_j\}_{j \in \mathbb{J}}) $, where $ U \in \mathcal{B}(\mathcal{H}) $ is  invertible. 
\item The frames $ (\{x_j\}_{j \in \mathbb{J}},\{x_j\}_{j \in \mathbb{J}})$ for $ \mathcal{H}$  are precisely $( \{Uf_j\}_{j \in \mathbb{J}},\{Uf_j\}_{j \in \mathbb{J}}) $, where $ U \in \mathcal{B}(\mathcal{H}) $ is  such that  $ UU^*$ is  invertible.
\item The Bessel sequences  $ ( \{x_j\}_{j \in \mathbb{J}},\{x_j\}_{j \in \mathbb{J}})$ for $ \mathcal{H}$  are precisely $( \{Uf_j\}_{j \in \mathbb{J}},\{Uf_j\}_{j \in \mathbb{J}}) $, where $ U \in \mathcal{B}(\mathcal{H}) $.
\item The Riesz frames $ ( \{x_j\}_{j \in \mathbb{J}},\{x_j\}_{j \in \mathbb{J}})$ for $ \mathcal{H}$  are precisely $( \{Uf_j\}_{j \in \mathbb{J}},\{Uf_j\}_{j \in \mathbb{J}}) $, where $ U \in \mathcal{B}(\mathcal{H})$ is  such that  $ UU^*$ is  invertible and $ U^*(UU^*)^{-1}U=I_{\mathcal{H}}$.  
\item The orthonormal frames $ ( \{x_j\}_{j \in \mathbb{J}},\{x_j\}_{j \in \mathbb{J}})$ for $ \mathcal{H}$  are precisely $( \{Uf_j\}_{j \in \mathbb{J}},\{Uf_j\}_{j \in \mathbb{J}}) $, where $ U  \in \mathcal{B}(\mathcal{H}) $ is such that  $ UU^*=I_\mathcal{H}= U^*U$.
\item The frames $ (\{x_j\}_{j \in \mathbb{J}},\{x_j\}_{j \in \mathbb{J}})$ for $ \mathcal{H}$  are precisely $( \{Uf_j\}_{j \in \mathbb{J}},\{Uf_j\}_{j \in \mathbb{J}}) $, where $ U \in \mathcal{B}(\mathcal{H}) $ is    surjective.
\item $ (\{x_j\}_{j \in \mathbb{J}},\{x_j\}_{j \in \mathbb{J}})$  is an orthonormal basis for $ \mathcal{H}$ if and only if it is an orthonormal frame.
\end{enumerate}
\end{corollary}
\begin{theorem}\label{SEQUENTIALCHARACTERIZATIONHILBERT2}
Let $\{x_j\}_{j\in\mathbb{J}},\{\tau_j\}_{j\in\mathbb{J}}$ be in $ \mathcal{H}.$ Then $(\{x_j\}_{j\in\mathbb{J}},\{\tau_j\}_{j\in\mathbb{J}})$  is a frame with bounds $ a $ and $ b$ (resp. Bessel with bound $ b$)
 \begin{enumerate}[\upshape(i)]
\item  if and only if 
$$U:\ell^2(\mathbb{J}) \ni \{c_j\}_{j\in\mathbb{J}}\mapsto\sum\limits_{j\in\mathbb{J}}c_jx_j \in \mathcal{H}, ~\text{and} ~ V:\ell^2(\mathbb{J}) \ni \{d_j\}_{j\in\mathbb{J}}\mapsto\sum\limits_{j\in\mathbb{J}}d_j\tau_j \in \mathcal{H} $$ 
are well-defined, $ U,V \in  \mathcal{B}(\ell^2(\mathbb{J}), \mathcal{H})$ such that $ aI_\mathcal{H}\leq VU^*\leq bI_\mathcal{H}$   (resp. $ 0\leq VU^*\leq bI_\mathcal{H}).$
\item  if and only if 
$$U:\ell^2(\mathbb{J}) \ni\{c_j\}_{j\in\mathbb{J}}\mapsto\sum\limits_{j\in\mathbb{J}}c_jx_j \in \mathcal{H}, ~\text{and} ~ S: \mathcal{H} \ni g\mapsto\{\langle g, \tau_j\rangle\}_{j\in\mathbb{J}} \in \ell^2(\mathbb{J}) $$ 
are well-defined, $ U \in  \mathcal{B}(\ell^2(\mathbb{J}), \mathcal{H})$, $ S \in  \mathcal{B}(\mathcal{H}, \ell^2(\mathbb{J}))$  such that  $ aI_\mathcal{H}\leq S^*U^*\leq bI_\mathcal{H}$ (resp.  $ 0\leq S^*U^*\leq bI_\mathcal{H}).$
\item   if and only if 
$$ R: \mathcal{H} \ni h\mapsto\{\langle h, x_j\rangle\}_{j\in\mathbb{J}} \in \ell^2(\mathbb{J}), ~\text{and} ~  V:\ell^2(\mathbb{J}) \ni\{d_j\}_{j\in\mathbb{J}}\mapsto\sum\limits_{j\in\mathbb{J}}d_j\tau_j \in \mathcal{H} $$ 
are well-defined, $ R \in  \mathcal{B}(\mathcal{H}, \ell^2(\mathbb{J}))$, $ V \in  \mathcal{B}(\ell^2(\mathbb{J}), \mathcal{H})$ such that  $ aI_\mathcal{H}\leq VR\leq bI_\mathcal{H}$  (resp.  $ 0\leq VR\leq bI_\mathcal{H} ).$
\item  if and only if 
$$ R: \mathcal{H} \ni h\mapsto\{\langle h, x_j\rangle\}_{j\in\mathbb{J}} \in \ell^2(\mathbb{J}), ~\text{and} ~  S: \mathcal{H} \ni g\mapsto\{\langle g, \tau_j\rangle\}_{j\in\mathbb{J}} \in \ell^2(\mathbb{J}) $$
are well-defined, $ R,S \in  \mathcal{B}(\mathcal{H}, \ell^2(\mathbb{J}))$ such that  $ aI_\mathcal{H}\leq S^*R\leq bI_\mathcal{H}$ (resp.   $ 0\leq S^*R\leq bI_\mathcal{H}).$   
\end{enumerate}
\end{theorem}
\begin{proof}
We prove the first one for Bessel sequences, remainings are similar by considering adjoints.
 
$(\Rightarrow)$ The given operators $ U,V$ are adjoints of bounded operators  $\mathcal{H}\ni h\mapsto\{\langle h,x_j\rangle\}_{j\in\mathbb{J}}\in \ell^2(\mathbb{J})$ and $ \mathcal{H}\ni h\mapsto\{\langle h,\tau_j\rangle\}_{j\in\mathbb{J}}\in \ell^2(\mathbb{J})$, respectively. Therefore $ U^*h=\{\langle h,x_j\rangle\}_{j\in\mathbb{J}}$ and  $ V^*g=\{\langle g,\tau_j\rangle\}_{j\in\mathbb{J}}, \forall h, g \in \mathcal{H}.$ From this,
\begin{align*}
\langle VU^*h, g\rangle &=\langle U^*h, V^*g\rangle=\sum_{j\in\mathbb{J}}\langle h, x_j\rangle\langle \tau_j, g\rangle=\left\langle \sum_{j\in\mathbb{J}}\langle h, x_j\rangle\tau_j,g \right\rangle \\
&=\left\langle \sum_{j\in\mathbb{J}}\langle h, \tau_j\rangle x_j,g \right\rangle=\sum_{j\in\mathbb{J}}\langle h, \tau_j\rangle\langle x_j, g\rangle= \langle UV^*h, g\rangle= \langle h, VU^*g\rangle,~\forall h,g \in \mathcal{H},
\end{align*}
$\langle VU^*h,h\rangle= \langle U^*h,V^*h\rangle=\sum_{j\in\mathbb{J}}\langle h,x_j\rangle\langle \tau_j, h\rangle\geq0, \forall h \in \mathcal{H} $ and $\langle VU^*h,h\rangle\leq b \|h\|^2, \forall h \in \mathcal{H} .$ Thus we got  $ VU^*$ is self-adjoint and from this, $\|VU^*\|=\sup_{h \in \mathcal{H}, \|h\|=1}\langle VU^*h,h\rangle \leq b. $
     
$ (\Leftarrow)$ From the  presence of $ U^* $ and $V^*,$ we get  $\mathcal{H}\ni h\mapsto\{\langle h,x_j\rangle\}_{j\in\mathbb{J}}\in \ell^2(\mathbb{J}), \mathcal{H}\ni h\mapsto\{\langle h,\tau_j\rangle\}_{j\in\mathbb{J}}\in \ell^2(\mathbb{J})$ are well-defined and bounded, respectively;  hence $ \{x_j\}_{j\in\mathbb{J}}, \{ \tau_j\}_{j\in\mathbb{J}}$ are Bessel sequences (w.r.t. themselves). Finally,  $ 0\leq \langle VU^*h,h\rangle=\sum_{j\in\mathbb{J}}\langle h,x_j\rangle\langle\tau_j,h\rangle\leq\|VU^*\|\langle h,h\rangle \leq b\langle h,h\rangle, \forall h \in \mathcal{H}.$
\end{proof}
 
\textbf{Similarity  and tensor product}
\begin{definition}
 A frame  $ (\{y_j\}_{j\in \mathbb{J}},\{\omega_j\}_{j\in \mathbb{J}})$ for $ \mathcal{H}$ is said to be  similar to a frame $ (\{x_j\}_{j\in \mathbb{J}},\{\tau_j\}_{j\in \mathbb{J}})$ for $ \mathcal{H}$ if there are invertible operators $ T_{x,y}, T_{\tau,\omega} \in \mathcal{B}(\mathcal{H})$ such that $ y_j=T_{x,y}x_j, \omega_j=T_{\tau,\omega}\tau_j,   \forall j \in \mathbb{J}.$
 \end{definition}     
 Definition is symmetric.     
 \begin{proposition}
 Let $ \{x_j\}_{j\in \mathbb{J}}\in \mathscr{F}_\tau$  with frame bounds $a, b,$  let $T_{x,y} , T_{\tau,\omega}\in \mathcal{B}(\mathcal{H})$ be positive, invertible, commute with each other, commute with $ S_{x, \tau}$, and let $y_j=T_{x,y}x_j , \omega_j=T_{\tau,\omega}\tau_j,  \forall j \in \mathbb{J}.$ Then 
 \begin{enumerate}[\upshape(i)]
 \item $ \{y_j\}_{j\in \mathbb{J}}\in \mathscr{F}_\tau$ and $ \frac{a}{\|T_{x,y}^{-1}\|\|T_{\tau,\omega}^{-1}\|}\leq S_{y, \omega} \leq b\|T_{x,y}T_{\tau,\omega}\|.$ Assuming that $ (\{x_j\}_{j\in \mathbb{J}},\{\tau_j\}_{j\in \mathbb{J}})$ is Parseval, then $(\{y_j\}_{j\in \mathbb{J}},  \{\omega_j\}_{j\in \mathbb{J}})$ is Parseval  if and only if   $ T_{\tau, \omega}T_{x,y}=I_\mathcal{H}.$  
 \item $ \theta_y=\theta_x T_{x,y}, \theta_\omega=\theta_\tau T_{\tau,\omega}, S_{y,\omega}=T_{\tau,\omega}S_{x, \tau}T_{x,y},  P_{y,\omega} =P_{x, \tau}.$
 \end{enumerate}
 \end{proposition}     
 \begin{proof}
 For $ h \in \mathcal{H},$ $\theta_x T_{x,y}h = \{\langle T_{x,y}h, x_j\rangle\}_{j\in \mathbb{J}}=\{\langle h , T_{x,y} x_j\rangle\}_{j\in \mathbb{J}}.$ Hence $ \theta_y$ exists and $ \theta_y=\theta_x T_{x,y}.$  Likewise we get $ \theta_\omega.$ Now $S_{y,\omega}=\theta_\omega^*\theta_y=(\theta_\tau T_{\tau,\omega})^*\theta_x T_{x,y}=T_{\tau,\omega}S_{x, \tau}T_{x,y}, P_{y,\omega}= \theta_yS_{y,\omega}^{-1}\theta_\omega^*=(\theta_xT_{x,y})(T_{x,y}^{-1}S_{x, \tau}^{-1}T_{\tau,\omega}^{-1})(T_{\tau,\omega}\theta_\tau^*)=P_{x, \tau}$. Other parts are simple.
  \end{proof}     
 \begin{lemma}\label{SEQUENTIALSIMILARITYLEMMA}
 Let $ \{x_j\}_{j\in \mathbb{J}}\in \mathscr{F}_\tau,$ $ \{y_j\}_{j\in \mathbb{J}}\in \mathscr{F}_\omega$ and   $y_j=T_{x, y}x_j , \omega_j=T_{\tau,\omega}\tau_j,  \forall j \in \mathbb{J}$, for some invertible $T_{x,y}, T_{\tau,\omega}\in \mathcal{B}(\mathcal{H}).$ Then 
 $ \theta_y=\theta_x T^*_{x,y}, \theta_\omega=\theta_\tau T^*_{\tau,\omega}, S_{y,\omega}=T_{\tau,\omega}S_{x, \tau}T_{x,y}^*,  P_{y,\omega}=P_{x, \tau}.$ Assuming that $ (\{x_j\}_{j\in \mathbb{J}},\{\tau_j\}_{j\in \mathbb{J}})$ is Parseval frame, then $ (\{y_j\}_{j\in \mathbb{J}},\{\omega_j\}_{j\in \mathbb{J}})$ is Parseval frame if and only if $T_{\tau,\omega}T_{x,y}^*=I_\mathcal{H}.$
 \end{lemma}     
 \begin{proof}
 Let $ h \in \mathcal{H}$. Then $\theta_x T^*_{x,y}h = \{\langle T^*_{x,y}h, x_j\rangle\}_{j\in \mathbb{J}}=\{\langle h , T_{x,y} x_j\rangle\}_{j\in \mathbb{J}}=\{\langle h ,  y_j\rangle\}_{j\in \mathbb{J}} =\theta_yh; \theta_\omega=\theta_\tau T^*_{\tau,\omega}, S_{y,\omega}=(\theta_\tau T_{\tau,\omega}^*)^*\theta_x T_{x,y}^*=T_{\tau,\omega}S_{x, \tau}T_{x,y}^*, P_{y,\omega}=(\theta_xT_{x,y}^*)({T^*}_{x,y}^{-1}S_{x,\tau}^{-1}{T}_{\tau,\omega}^{-1})(T_{\tau,\omega}\theta_\tau^*)=P_{x, \tau}. $ 
 \end{proof}  
 \begin{theorem}\label{SEQUENTIALSIMILARITYCHARACTERIZATION}
 Let $ \{x_j\}_{j\in \mathbb{J}}\in \mathscr{F}_\tau,$ $ \{y_j\}_{j\in \mathbb{J}}\in \mathscr{F}_\omega.$ The following are equivalent.
 \begin{enumerate}[\upshape(i)]
 \item $y_j=T_{x,y}x_j , \omega_j=T_{\tau, \omega}\tau_j ,  \forall j \in \mathbb{J},$ for some invertible  $ T_{x,y} ,T_{\tau, \omega} \in \mathcal{B}(\mathcal{H}). $
 \item $\theta_y=\theta_x{T'}_{x,y}^* , \theta_\omega=\theta_\tau {T'}_{\tau, \omega}^* $ for some invertible  $ {T'}_{x,y} ,{T'}_{\tau, \omega} \in \mathcal{B}(\mathcal{H}). $
 \item $P_{y,\omega}=P_{x,\tau}.$
 \end{enumerate}
 If one of the above conditions is satisfied, then  invertible operators in  $\operatorname{(i)}$ and  $\operatorname{(ii)}$ are unique and are given by  $T_{x,y}=\theta_y^*\theta_\tau S_{x,\tau}^{-1}, T_{\tau, \omega}=\theta_\omega^*\theta_x S_{x,\tau}^{-1}.$
 In the case that $(\{x_j\}_{j\in \mathbb{J}},  \{\tau_j\}_{j\in \mathbb{J}})$ is Parseval, then $(\{y_j\}_{j\in \mathbb{J}},  \{\omega_j\}_{j\in \mathbb{J}})$ is  Parseval if and only if $T_{\tau, \omega}T_{x,y}^* =I_\mathcal{H}$   if and only if $ T_{x,y}^*T_{\tau, \omega} =I_\mathcal{H}$. 
 \end{theorem}
 \begin{proof}
 The implications (i)  $\Rightarrow $ (ii)   $\Rightarrow $ (iii) come from Lemma \ref{SEQUENTIALSIMILARITYLEMMA}. (ii)  $\Rightarrow $ (i) Let $ \{e_j\}_{j \in \mathbb{J}}$ be the standard orthonormal basis for $ \ell^2(\mathbb{J})$. Then $ y_j=\theta_y^*e_j=(\theta_x{T'}_{x,y}^*)^*e_j=T'_{x,y}\theta_x^*e_j=T'_{x,y}x_j$, $\omega_j=\theta_\omega^*e_j=(\theta_\tau{T'}_{\tau,\omega}^*)^*e_j=T'_{\tau,\omega}\theta_\tau^*e_j=T'_{\tau,\omega}\tau_j,\forall j \in \mathbb{J} $. (iii)  $\Rightarrow $ (ii) $\theta_y=P_{y,\omega}\theta_y=P_{x,\tau}\theta_y=\theta_x(S_{x,\tau}^{-1}\theta_\tau^*\theta_y),\theta_\omega=P_{y,\omega}^*\theta_\omega=P_{x,\tau}^*\theta_\omega=\theta_\tau(S_{x,\tau}^{-1}\theta_x^*\theta_\omega) $ and $ (S_{x,\tau}^{-1}\theta_\tau^*\theta_y)(S_{y,\omega}^{-1}\theta_\omega^*\theta_x)=S_{x,\tau}^{-1}\theta_\tau^*P_{y,\omega}\theta_x=S_{x,\tau}^{-1}\theta_\tau^*P_{x,\tau}\theta_x =S_{x,\tau}^{-1}\theta_\tau^*\theta_x=I_\mathcal{H}, $  $(S_{y,\omega}^{-1}\theta_\omega^*\theta_x)(S_{x,\tau}^{-1}\theta_\tau^*\theta_y)=S_{y,\omega}^{-1}\theta_\omega^*P_{x,\tau}\theta_y=S_{y,\omega}^{-1}\theta_\omega^*P_{y,\omega}\theta_y=S_{y,\omega}^{-1}\theta_\omega^*\theta_y=I_\mathcal{H},$  $ (S_{x,\tau}^{-1}\theta_x^*\theta_\omega)(S_{y,\omega}^{-1}\theta_y^*\theta_\tau)=S_{x,\tau}^{-1}\theta_x^*P_{y,\omega}^*\theta_\tau=S_{x,\tau}^{-1}\theta_x^*P_{x,\tau}^*\theta_\tau=S_{x,\tau}^{-1}\theta_x^*\theta_\tau=I_\mathcal{H} ,$  $(S_{y,\omega}^{-1}\theta_y^*\theta_\tau)(S_{x,\tau}^{-1}\theta_x^*\theta_\omega)=S_{y,\omega}^{-1}\theta_y^*P_{x,\tau}^*\theta_\omega=S_{y,\omega}^{-1}\theta_y^*P_{y,\omega}^*\theta_\omega =S_{y,\omega}^{-1}\theta_y^*\theta_\omega=I_\mathcal{H}$.\\ 
 Now let $ y_j=T_{x,y}x_j, \omega_j=T_{\tau,\omega}\tau_j,\forall j \in \mathbb{J}$. Then $ \theta_y=\theta_xT_{x,y}^*$ $ \Rightarrow$ $ \theta_\tau^*\theta_y=\theta_\tau^*\theta_xT_{x,y}^*=S_{x,\tau}T_{x,y}^*$ $ \Rightarrow$ $T_{x,y}=\theta_y^*\theta_\tau S_{x,\tau}^{-1} ,$ and $\theta_\omega=\theta_\tau T_{\tau, \omega}^* $ $ \Rightarrow$ $\theta_x^*\theta_\omega=\theta_x^*\theta_\tau T_{\tau, \omega}^*= S_{x,\tau}T_{\tau, \omega}^*$ $ \Rightarrow$ $T_{\tau, \omega}=\theta_\omega^*\theta_xS_{x,\tau}^{-1} .$
 \end{proof}
 \begin{corollary}
 For any given frame $ (\{x_j\}_{j \in \mathbb{J}} , \{\tau_j\}_{j \in \mathbb{J}})$, the canonical dual of $ (\{x_j\}_{j \in \mathbb{J}} , \{\tau_j\}_{j \in \mathbb{J}}  )$ is the only dual frame that is similar to $ (\{x_j\}_{j \in \mathbb{J}} , \{\tau_j\}_{j \in \mathbb{J}} )$.
 \end{corollary}
 \begin{proof}
 If $ (\{x_j\}_{j \in \mathbb{J}} , \{\tau_j\}_{j \in \mathbb{J}})$ and $ (\{y_j\}_{j \in \mathbb{J}} , \{\omega_j\}_{j \in \mathbb{J}})$  are similar and dual to each other, then there exist invertible  $T_{x,y}, T_{\tau,\omega} \in \mathcal{B}(\mathcal{H})$  such that $ y_j=T_{x,y}x_j,\omega_j=T_{\tau,\omega}\tau_j ,\forall j \in \mathbb{J}$. Theorem \ref{SEQUENTIALSIMILARITYCHARACTERIZATION} gives $T_{x,y}=\theta_y^*\theta_\tau S_{x,\tau}^{-1}=I_\mathcal{H}S_{x,\tau}^{-1}, T_{\tau, \omega}=\theta_\omega^*\theta_x S_{x,\tau}^{-1}=I_\mathcal{H}S_{x,\tau}^{-1}=S_{x,\tau}^{-1}.$ Hence $ (\{y_j\}_{j \in \mathbb{J}} , \{\omega_j\}_{j \in \mathbb{J}})$ is the canonical dual of $ (\{x_j\}_{j \in \mathbb{J}} , \{\tau_j\}_{j \in \mathbb{J}})$.	
 \end{proof}
 \begin{corollary}
 Two similar  frames cannot be orthogonal.
 \end{corollary}
 \begin{proof}
 Let $ (\{x_j\}_{j \in \mathbb{J}} , \{\tau_j\}_{j \in \mathbb{J}})$ and $ (\{y_j\}_{j \in \mathbb{J}} , \{\omega_j\}_{j \in \mathbb{J}})$  be similar.  Then there exist invertible  $T_{x,y}, T_{\tau,\omega} \in \mathcal{B}(\mathcal{H})$  such that $ y_j=T_{x,y}x_j,\omega_j=T_{\tau,\omega}\tau_j ,\forall j \in \mathbb{J}$. From Theorem \ref{SEQUENTIALSIMILARITYCHARACTERIZATION}, $\theta_y=\theta_x{T}_{x,y}^* , \theta_\omega=\theta_\tau {T}_{\tau, \omega}^* $.   Therefore $\theta_y^*\theta_\tau=(\theta_xT_{x,y}^*)^*\theta_\tau=T_{x,y}\theta_x^*\theta_\tau =T_{x,y}S_{x,\tau} \neq0. $
 \end{proof}
 \begin{remark}
 For every frame  $(\{x_j\}_{j \in \mathbb{J}}, \{\tau_j\}_{j \in \mathbb{J}}),$ each  of `frames'  $( \{S_{x, \tau}^{-1}x_j\}_{j \in \mathbb{J}}, \{\tau_j\}_{j \in \mathbb{J}})$,    $( \{S_{x, \tau}^{-1/2}x_j\}_{j \in \mathbb{J}}, \{S_{x,\tau}^{-1/2}\tau_j\}_{j \in \mathbb{J}}),$ and  $ (\{x_j \}_{j \in \mathbb{J}}, \{S_{x,\tau}^{-1}\tau_j\}_{j \in \mathbb{J}})$ is a  Parseval frame  which is similar to  $ (\{x_j\}_{j \in \mathbb{J}} , \{\tau_j\}_{j \in \mathbb{J}}  ).$  Hence each frame is similar to  Parseval frames.
\end{remark}  

\textbf{Tensor product  of frames}: Let $(\{x_j\}_{j \in \mathbb{J}}, \{\tau_j\}_{j \in \mathbb{J}})$ be a frame  for   $ \mathcal{H}$ and $(\{y_l\}_{l \in \mathbb{L}}, \{\omega_l\}_{l \in \mathbb{L}})$ be a frame  for   $ \mathcal{H}_1.$ The frame  $(\{z_{(j, l)}\coloneqq x_j\otimes y_l\}_{(j, l)\in \mathbb{J}\bigtimes  \mathbb{L}},\{\rho_{(j, l)}\coloneqq \tau_j\otimes\omega_l\}_{(j, l)\in \mathbb{J}\bigtimes  \mathbb{L}})$   for $ \mathcal{H}\otimes\mathcal{H}_1$ is called  as tensor product  of frames $( \{x_j\}_{j \in \mathbb{J}}, \{\tau_j\}_{j\in \mathbb{J}})$ and $( \{y_l\}_{l \in \mathbb{L}},  \{\omega_l\}_{l\in \mathbb{L}}).$ 
 \begin{proposition}
 Let  $(\{z_{(j, l)}\coloneqq x_j\otimes y_l\}_{(j, l)\in \mathbb{J}\bigtimes  \mathbb{L}},\{\rho _{(j, l)}\coloneqq \tau_j\otimes \omega_l\}_{(j, l)\in \mathbb{J}\bigtimes  \mathbb{L}}) $ be the  tensor product of frames  $( \{x_j\}_{j \in \mathbb{J}}, \{\tau_j\}_{j \in \mathbb{J}}) $  for  $ \mathcal{H},$ and $( \{y_l\}_{l \in \mathbb{L}}, \{\omega_l\}_{l \in \mathbb{L}} )$ for $ \mathcal{H}_1.$  Then  $\theta_z=\theta_x\otimes\theta_y, \theta_\rho=\theta_\tau\otimes\theta_\omega, S_{z, \rho}=S_{x, \tau}\otimes S_{y, \omega}, P_{z, \rho}=P_{x, \tau}\otimes P_{y, \omega}.$ If  $( \{x_j\}_{j \in \mathbb{J}}, \{\tau_j\}_{j \in \mathbb{J}}) $ and $ (\{y_l\}_{l \in \mathbb{L}},  \{\omega_l\}_{l \in \mathbb{L}} )$ are Parseval, then $(\{z_{(j, l)}\}_{(j, l)\in \mathbb{J}\bigtimes  \mathbb{L}} ,\{\rho_{(j,l)}\}_{(j,l)\in \mathbb{J}\bigtimes \mathbb{L}})$ is Parseval.
 \end{proposition}
 \begin{proof}
 At elementary tensor $ h\otimes g \in \mathcal{H}\otimes\mathcal{H}_1$, $ (\theta_x\otimes \theta_y)(h\otimes g)=\theta_xh\otimes\theta_yg=\{\langle h,  x_j\rangle \}_{j \in \mathbb{J}}\otimes\{\langle g,  y_k\rangle \}_{k \in \mathbb{L}}=\{\langle h,  x_j\rangle\langle g,  y_j\rangle \}_{(j,k) \in \mathbb{J}\times \mathbb{L}}=\{\langle h\otimes g,  x_j\otimes \tau_k\rangle\}_{(j,k) \in \mathbb{J}\times \mathbb{L}}=\theta_z(h\otimes g)$. Similarly $\theta_\rho=\theta_\tau\otimes\theta_\omega$. Then $S_{z, \rho}=\theta_\rho^*\theta_z=(\theta_\tau^*\otimes\theta_\omega^*)(\theta_x\otimes\theta_y)=\theta_\tau^*\theta_x\otimes\theta_\omega^*\theta_y=S_{x, \tau}\otimes S_{y, \omega} $ and  $P_{z, \rho}=\theta_zS_{z, \rho}^{-1}\theta_\rho^*=(\theta_x\otimes\theta_y )(S_{x, \tau}^{-1}\otimes S_{y, \omega}^{-1})(\theta_\tau^*\otimes\theta_\omega^*)=\theta_xS_{x, \tau}^{-1}\theta_\tau^*\otimes\theta_y S_{y, \omega}^{-1}\theta_\omega^*=P_{x, \tau}\otimes P_{y, \omega}$.
 \end{proof}
    
\textbf{Frames and  discrete group representations}
\begin{definition}
Let $ \pi$ be a unitary representation of a discrete group $ G$ on a Hilbert space $ \mathcal{H}.$ An element  $ x$ in $ \mathcal{H}$ is called a  frame generator (resp. a  Parseval frame generator) w.r.t. $ \tau$ in $ \mathcal{H}$  if $(\{x_g\coloneqq  \pi_{g}x\}_{g\in G},\{\tau_g\coloneqq  \pi_{g}\tau\}_{g\in G}) $ is a  frame  (resp. Parseval frame) for  $ \mathcal{H}$.
In this case we write $ (x,\tau)$ is a frame generator for $\pi$. 
\end{definition}
\begin{proposition}\label{REPRESENATIONLEMMAGROUP}
Let $ (x,\tau)$ and $ (y,\omega)$ be  frame generators    in $\mathcal{H}$ for a unitary representation $ \pi$ of  $G$ on $ \mathcal{H}.$ Then
\begin{enumerate}[\upshape(i)]
 \item $ \theta_x\pi_g=\lambda_g\theta_x,  \theta_\tau \pi_g=\lambda_g\theta_\tau,   \forall g \in G.$
 \item $ \theta_x^*\theta_y,   \theta_\tau^*\theta_\omega,\theta_x^*\theta_\omega$ are in the commutant $ \pi(G)'$ of $ \pi(G)''.$ Further, $ S_{x,\tau} \in \pi(G)'$ and $(S_{x, \tau}^{-1/2}x, S_{x, \tau}^{-1/2}\tau)$ is a Parseval frame generator. 
 \item $ \theta_xT\theta_\tau^*,\theta_xT\theta_y^*, \theta_\tau T\theta_\omega^* \in \mathscr{R}(G), \forall T \in \pi (G)'.$ In particular, $ P_{x, \tau} \in \mathscr{R}(G). $
 \end{enumerate}
 \end{proposition}   
  \begin{proof} Let $ g,p,q, \in G $ and $ h \in \mathcal{H}.$
  \begin{enumerate}[\upshape(i)]
  \item As earlier, $ \{\chi_g\}_{g\in G}$ is the standard orthonormal basis of $ \ell^2(G).$ Then
  \begin{align*}
  \lambda_g\theta_xh&= \lambda_g(\{\langle h, x_p\rangle\}_{p\in G})=\lambda_g\left( \sum\limits_{p\in G}\langle h, x_p\rangle\chi_p\right)= \sum\limits_{p\in G}\langle h, x_p\rangle\lambda_g\chi_p=\sum\limits_{p\in G}\langle h, x_p\rangle\chi_{gp}\\
  &=\sum\limits_{p\in G}\langle h, \pi_px\rangle\chi_{gp}=\sum\limits_{q\in G}\langle h, \pi_{g^{-1}q}x\rangle\chi_q=\sum\limits_{q\in G}\langle h, \pi_{g^{-1}}\pi_qx\rangle\chi_q=\sum\limits_{q\in G}\langle \pi_gh,\pi_qx\rangle\chi_q\\
  &=\sum\limits_{q\in G}\langle \pi_gh,x_q\rangle\chi_q=\{\langle \pi_gh, x_q\rangle\}_{q\in G}=\theta_x(\pi_gh).
  \end{align*}
  \item $\theta_x^*\theta_y\pi_g=\theta_x^*\lambda_g\theta_y=(\lambda_{g^{-1}}\theta_x)^*\theta_y=(\theta_x\pi_{g^{-1}})^*\theta_y=\pi_g\theta_x^*\theta_y.$ By taking $ y=x $ and $ \omega=\tau$ we get  $ S_{x,\tau} \in \pi(G)'.$ From this, $ S_{x,\tau}^{-1/2} \in \pi(G)'.$ We next show that $S_{x, \tau}^{-1/2}x$ is a Parseval frame generator w.r.t. $ S_{x, \tau}^{-1/2}\tau.$ Consider 
  
 \begin{align*}
 S_{S_{x, \tau}^{-\frac{1}{2}}x,S_{x, \tau}^{-\frac{1}{2}}\tau}h&=\sum_{g\in G}\langle h, \pi_gS_{x, \tau}^{-\frac{1}{2}}x \rangle \pi_gS_{x, \tau}^{-\frac{1}{2}}\tau =\sum_{g\in G}\langle h, S_{x, \tau}^{-\frac{1}{2}}\pi_gx \rangle S_{x, \tau}^{-\frac{1}{2}}\pi_g\tau\\
 &= S_{x, \tau}^{-\frac{1}{2}}\left(\sum_{g\in G}\langle  S_{x, \tau}^{-\frac{1}{2}}h, \pi_gx \rangle \pi_g\tau\right)=S_{x, \tau}^{-\frac{1}{2}}S_{x,\tau }(S_{x, \tau}^{-\frac{1}{2}}h)=I_\mathcal{H}h.
 \end{align*}
 \item 
Let  $ T \in \pi(G)'.$ Then 
$$\theta_xT\theta_\tau^*\lambda_g = \theta_xT(\lambda_{g^{-1}}\theta_\tau)^*=\theta_xT(\theta_\tau\pi_{g^{-1}})^*=\theta_xT\pi_{g}\theta_\tau^*=\theta_x\pi_{g}T\theta_\tau^*=\lambda_g\theta_xT\theta_\tau^*.$$
Therefore  $\theta_xT\theta_\tau^* \in \mathscr{L}(G)'=\mathscr{R}(G).$ To get  $ P_{x, \tau} \in \mathscr{R}(G)$ we take $ T=S_{x, \tau}^{-1}.$
\end{enumerate}
\end{proof} 
For the direct part of the following theorem and corollaries of that, there is no need of Parseval condition.
\begin{theorem}\label{GROUPC1}
Let $ G$ be a discrete group with identity $ e$ and $( \{x_g\}_{g\in G},  \{\tau_g\}_{g\in G})$ be a Parseval  frame  for  $\mathcal{H}.$ Then there is a unitary representation $ \pi$  of $ G$ on  $ \mathcal{H}$  for which 
$$ x_g=\pi_gx_e, ~\tau_g=\pi_g\tau_e ,~ \forall g \in G$$  if and only if 
$$\langle x_{gp} , x_{gq}\rangle=\langle x_p, x_q \rangle, ~ \langle x_{gp},  \tau_{gq}\rangle=\langle x_p, \tau_q \rangle, ~\langle \tau_{gp} , \tau_{gq}\rangle=\langle \tau_p, \tau_q \rangle,  ~\forall  g,p,q \in G.  $$
\end{theorem} 
\begin{proof} 
Proof 1. $(\Rightarrow)$  For all $ g, p, q \in G$, we have 
$$\langle x_{gp}, x_{gq}\rangle=\langle \pi_{gp}x_e, \pi_{gq}x_e\rangle = \langle \pi_g\pi_px_e, \pi_g\pi_qx_e \rangle = \langle \pi_{g^{-1}}\pi_g\pi_px_e, \pi_qx_e \rangle = \langle\pi_px_e, \pi_qx_e \rangle=\langle x_p, x_q \rangle .$$
Two others are  similar.
 
$(\Leftarrow)$ We claim the   following three, among them we prove third, others are similar.
$$   \lambda_g\theta_x\theta_x^*=\theta_x\theta_x^*\lambda_g,~  \lambda_g\theta_x\theta_\tau^*=\theta_x\theta_\tau^*\lambda_g,~ \lambda_g\theta_\tau\theta_\tau^*=\theta_\tau\theta_\tau^*\lambda_g,~ \forall g \in G.$$
Consider 
\begin{align*}
\lambda_g \theta_\tau\theta_\tau^*\lambda_{g^{-1}}\chi_p&=\lambda_g \theta_\tau\theta_\tau^*\chi_{g^{-1}p}=\lambda_g\theta_\tau \tau_{g^{-1}p}=\lambda_g\{\langle \tau_{g^{-1}p}, \tau_q \rangle\}_{q\in G}=\lambda_g\left(\sum\limits_{q\in G}\langle \tau_{g^{-1}p}, \tau_q \rangle \chi_q\right)\\
&=\sum\limits_{q\in G}\langle \tau_{g^{-1}p}, \tau_q \rangle \chi_{gq}=\sum\limits_{r\in G}\langle \tau_{g^{-1}p}, \tau_{g^{-1}r}\rangle \chi_r=\sum\limits_{r\in G}\langle \tau_{p}, \tau_{r}\rangle \chi_r=\{\langle \tau_{p}, \tau_{r}\rangle \}_{r\in G}\\
&=\theta_\tau\tau_p=\theta_\tau\theta_\tau^*\chi_p.
\end{align*}
Define $ \pi : G \ni g  \mapsto \pi_g\coloneqq \theta_\tau^*\lambda_g\theta_x  \in \mathcal{B}(\mathcal{H}).$ We use the fact that  given frame is a Parseval frame to get  $ \pi_g\pi_h=\theta_\tau^*\lambda_g \theta_x \theta_\tau^*\lambda_h \theta_x =\theta_\tau^*\lambda_g \lambda_h\theta_x \theta_\tau^* \theta_x = \theta_\tau^*\lambda_g \lambda_h\theta_x = \theta_\tau^*\lambda_{gh} \theta_x =\pi_{gh}$ for all $g, h \in G,$ and $\pi_g\pi_g^*=\theta_\tau^*\lambda_g\theta_x\theta_x^*\lambda_{g^{-1}}\theta_\tau=\theta_\tau^*\theta_x\theta_x^*\lambda_g\lambda_{g^{-1}}\theta_\tau=I_\mathcal{H},  \pi_g^*\pi_g= \theta_x^*\lambda_{g^{-1}}\theta_\tau\theta_\tau^*\lambda_g\theta_x= \theta_x^*\lambda_{g^{-1}}\lambda_g\theta_\tau\theta_\tau^*\theta_x=I_\mathcal{H} $ for all $ g \in G$.   Due to the  discrete topology of $ G$,  $ \pi$ is a unitary representation. We now establish   $ x_g=\pi_gx_e, \tau_g=\pi_g\tau_e  $ for all $ g \in G$. For,

\begin{align*}
\pi_gx_e &=\theta_\tau^*\lambda_g\theta_xx_e=\theta_\tau^*\lambda_g\{\langle x_e, x_p\rangle\}_{p\in G}=\theta_\tau^*\lambda_g\left(\sum\limits_{p\in G}\langle x_e, x_p\rangle\chi_p\right)=\theta_\tau^*\left(\sum\limits_{p\in G}\langle x_e, x_p\rangle\chi_{gp}\right)\\
&=\theta_\tau^*\left(\sum\limits_{q\in G}\langle x_e, x_{g^{-1}q}\rangle\chi_q\right)=\theta_\tau^*\left(\sum\limits_{q\in G}\langle x_{g^{-1}g}, x_{g^{-1}q}\rangle\chi_q\right)=\theta_\tau^*\left(\sum\limits_{q\in G}\langle x_{g}, x_{q}\rangle\chi_q\right)\\
&=\theta_\tau^*\{\langle x_g,x_q \rangle\}_{q\in G}=\sum\limits_{q\in G}\langle x_g, x_{q}\rangle\tau_q=x_g,
\end{align*}
and
\begin{align*}
\pi_g\tau_e &=\theta_\tau^*\lambda_g\theta_x\tau_e=\theta_\tau^*\lambda_g\{\langle \tau_e, x_p\rangle\}_{p\in G}=\theta_\tau^*\lambda_g\left(\sum\limits_{p\in G}\langle \tau_e, x_p\rangle\chi_p\right)=\theta_\tau^*\left(\sum\limits_{p\in G}\langle \tau_e, x_p\rangle\chi_{gp}\right)\\
&=\theta_\tau^*\left(\sum\limits_{q\in G}\langle \tau_e, x_{g^{-1}q}\rangle\chi_q\right)=\theta_\tau^*\left(\sum\limits_{q\in G}\langle \tau_{g^{-1}g}, x_{g^{-1}q}\rangle\chi_q\right)=\theta_\tau^*\left(\sum\limits_{q\in G}\langle \tau_{g}, x_{q}\rangle\chi_q\right)\\
&=\theta_\tau^*\{\langle \tau_g,x_q \rangle\}_{q\in G}=\sum\limits_{q\in G}\langle \tau_g, x_{q}\rangle\tau_q=\tau_g.
\end{align*}
Proof 2. Define $A_g: \mathcal{H} \ni h \mapsto \langle h, x_g \rangle \in \mathbb{K} $, $\Psi_g: \mathcal{H} \ni h \mapsto \langle h, \tau_g \rangle \in \mathbb{K}, \forall g \in G $. Then, from Theorem \ref{OVFTOSEQUENCEANDVICEVERSATHEOREM},  $(\{x_g\}_{g\in G}, \{\tau_g\}_{g\in G})$ is a frame for  $\mathcal{H}$ if and only if  $(\{A_g\}_{g\in G}, \{\Psi_g\}_{g\in G})$ is an (ovf)  in $\mathcal{B}(\mathcal{H},\mathbb{K})$. Further, from the proof of Theorem \ref{OVFTOSEQUENCEANDVICEVERSATHEOREM}, we also see that $(\{x_g\}_{g\in G}, \{\tau_g\}_{g\in G})$ is a Parseval frame  if and only if  $(\{A_g\}_{g\in G}, \{\Psi_g\}_{g\in G})$ is a Parseval (ovf). Now applying Theorem \ref{gc1} to the  Parseval (ovf) $(\{A_g\}_{g\in G}, \{\Psi_g\}_{g\in G})$ yields - there is a unitary representation $ \pi$  of $ G$ on  $ \mathcal{H}$  for which 
\begin{align}\label{PROOF2FIRSTEQUATION}
 A_g=A_e\pi_{g^{-1}}, ~\Psi_g=\Psi_e\pi_{g^{-1}}, ~\forall  g \in G
\end{align}
if and only if 
\begin{align}\label{PROOF2SECONDEQUATION}
A_{gp}A_{gq}^*=A_pA_q^* ,~ A_{gp}\Psi_{gq}^*=A_p\Psi_q^*,~ \Psi_{gp}\Psi_{gq}^*=\Psi_p\Psi_q^*, ~ \forall g,p,q \in G.
\end{align}
But Equation (\ref{PROOF2FIRSTEQUATION}) holds if and only if 
\begin{align*}
&\langle h, x_g \rangle=A_gh=A_e\pi_{g^{-1}}h=\langle \pi_{g^{-1}}h, x_e \rangle=\langle h, \pi_gx_e \rangle, \\
&\langle h, \tau_g \rangle=\Psi_gh=\Psi_e\pi_{g^{-1}}h=\langle \pi_{g^{-1}}h, \tau_g \rangle=\langle h, \pi_g\tau_e \rangle, ~\forall  g \in G, \forall h \in \mathcal{H}\\
&\iff x_g=\pi_gx_e, ~\tau_g=\pi_g\tau_e ,~ \forall g \in G.
\end{align*}
Also, Equation (\ref{PROOF2SECONDEQUATION}) holds if and only if 
\begin{align*}
&\langle \alpha x_{gq}, x_{gp}\rangle =A_{gp}A_{gq}^*\alpha=A_pA_q^*\alpha=\langle \alpha x_{q}, x_{p}\rangle,\\
&\langle \alpha \tau_{gq}, x_{gp}\rangle=A_{gp}\Psi_{gq}^*\alpha=A_p\Psi_q^*\alpha=\langle \alpha \tau_{q}, x_{p}\rangle,\\
& \langle \alpha \tau_{gq}, \tau_{gp}\rangle=\Psi_{gp}\Psi_{gq}^*=\Psi_p\Psi_q^*=\langle \alpha \tau_{q}, \tau_{p}\rangle  , ~\forall \alpha \in \mathbb{C}\\
&\iff \langle x_{gp} , x_{gq}\rangle=\langle x_p, x_q \rangle, ~ \langle x_{gp},  \tau_{gq}\rangle=\langle x_p, \tau_q \rangle, ~\langle \tau_{gp} , \tau_{gq}\rangle=\langle \tau_p, \tau_q \rangle,  ~\forall  g,p,q \in G.
\end{align*}
\end{proof}
\begin{corollary}
Let $ G$ be a discrete group with identity $ e$ and $( \{x_g\}_{g\in G},  \{\tau_g\}_{g\in G})$ be a  frame  for  $\mathcal{H}.$ Then there is a unitary representation $ \pi$  of $ G$ on  $ \mathcal{H}$  for which
\begin{enumerate}[\upshape(i)]
\item  $ x_g=S_{x,\tau}\pi_gS_{x,\tau}^{-1}x_e, \tau_g=\pi_g\tau_e $ for all $ g \in G$  if and only if $\langle S_{x,\tau}^{-2} x_{gp} , x_{gq}\rangle=\langle S_{x,\tau}^{-2} x_p, x_q \rangle,  \langle S_{x,\tau}^{-1} x_{gp},  \tau_{gq}\rangle=\langle S_{x,\tau}^{-1} x_p, \tau_q \rangle, \langle \tau_{gp} , \tau_{gq}\rangle=\langle \tau_p, \tau_q \rangle  $ for all $ g,p,q \in G.$
\item  $ x_g=S_{x,\tau}^{1/2}\pi_gS_{x,\tau}^{-1/2}x_e, \tau_g=S_{x,\tau}^{1/2}\pi_gS_{x,\tau}^{-1/2}\tau_e $ for all $ g \in G$  if and only if $\langle S_{x,\tau}^{-1} x_{gp} , x_{gq}\rangle=\langle S_{x,\tau}^{-1}x_p, x_q \rangle$, $  \langle S_{x,\tau}^{-1}x_{gp},  \tau_{gq}\rangle=\langle S_{x,\tau}^{-1}x_p, \tau_q \rangle, \langle S_{x,\tau}^{-1}\tau_{gp} , \tau_{gq}\rangle=\langle S_{x,\tau}^{-1}\tau_p, \tau_q \rangle  $ for all $ g,p,q \in G.$
\item  $ x_g=\pi_gx_e, \tau_g=S_{x,\tau}\pi_gS_{x,\tau}^{-1}\tau_e $ for all $ g \in G$  if and only if $\langle x_{gp} , x_{gq}\rangle=\langle x_p, x_q \rangle,  \langle x_{gp}, S_{x,\tau}^{-1} \tau_{gq}\rangle=\langle x_p, S_{x,\tau}^{-1}\tau_q \rangle, \langle \tau_{gp} , S_{x,\tau}^{-2}\tau_{gq}\rangle=\langle \tau_p, S_{x,\tau}^{-2}\tau_q \rangle  $ for all $ g,p,q \in G.$
\end{enumerate}
\end{corollary}   
\begin{proof} Apply Theorem \ref{GROUPC1} to 
\begin{enumerate}[\upshape(i)]
\item $(\{S^{-1}_{x,\tau}x_g\}_{g\in G},  \{\tau_g\}_{g\in G}) $ to get: there is a unitary representation $ \pi$  of $ G$ on  $ \mathcal{H}$  for which $ S_{x,\tau}^{-1}x_g=\pi_gS_{x,\tau}^{-1}x_e, \tau_g=\pi_g\tau_e $ for all $ g \in G$  if and only if $\langle S_{x,\tau}^{-1}x_{gp} , S_{x,\tau}^{-1}x_{gq}\rangle=\langle S_{x,\tau}^{-1}x_p,S_{x,\tau}^{-1} x_q \rangle,  \langle S_{x,\tau}^{-1} x_{gp},  \tau_{gq}\rangle=\langle S_{x,\tau}^{-1} x_p, \tau_q \rangle, \langle \tau_{gp} , \tau_{gq}\rangle=\langle \tau_p, \tau_q \rangle  $ for all $ g,p,q \in G.$
\item $(\{S^{-1/2}_{x,\tau}x_g\}_{g\in G},  \{S^{-1/2}_{x,\tau}\tau_g\}_{g\in G}) $ to get: there is a unitary representation $ \pi$  of $ G$ on  $ \mathcal{H}$  for which $ S_{x,\tau}^{-1/2}x_g=\pi_g(S_{x,\tau}^{-1/2}x_e), S_{x,\tau}^{-1/2}\tau_g=\pi_g(S_{x,\tau}^{-1/2}\tau_e) $ for all $ g \in G$  if and only if $\langle S_{x,\tau}^{-1/2}x_{gp} , S_{x,\tau}^{-1/2}x_{gq}\rangle=\langle S_{x,\tau}^{-1/2}x_p, S_{x,\tau}^{-1/2}x_q \rangle,  \langle S_{x,\tau}^{-1/2}x_{gp},  S_{x,\tau}^{-1/2}\tau_{gq}\rangle=\langle S_{x,\tau}^{-1/2}x_p, S_{x,\tau}^{-1/2}\tau_q \rangle, \langle S_{x,\tau}^{-1/2}\tau_{gp},
S_{x,\tau}^{-1/2}\tau_{gq}\rangle=\langle S_{x,\tau}^{-1/2}\tau_p, S_{x,\tau}^{-1/2}\tau_q \rangle  $ for all $ g,p,q \in G.$
\item $(\{x_g\}_{g\in G},  \{S^{-1}_{x,\tau}\tau_g\}_{g\in G}) $ to get: there is a unitary representation $ \pi$  of $ G$ on  $ \mathcal{H}$  for which $ x_g=\pi_gx_e, S_{x,\tau}^{-1}\tau_g=\pi_g(S_{x,\tau}^{-1}\tau_e) $ for all $ g \in G$  if and only if $\langle x_{gp} , x_{gq}\rangle=\langle x_p, x_q \rangle,  \langle x_{gp},  S_{x,\tau}^{-1}\tau_{gq}\rangle=\langle x_p, S_{x,\tau}^{-1}\tau_q \rangle,  \langle S_{x,\tau}^{-1}\tau_{gp} , S_{x,\tau}^{-1}\tau_{gq}\rangle=\langle S_{x,\tau}^{-1}\tau_p, S_{x,\tau}^{-1}\tau_q \rangle  $ for all $ g,p,q \in G.$
\end{enumerate}
\end{proof}
\textbf{Frames and group-like unitary systems}
 \begin{definition}
 Let $ \pi$ be a unitary representation of a  group-like unitary system  $ \mathcal{U}$ on a Hilbert space $ \mathcal{H}.$ An element  $ x$ in $ \mathcal{H}$ is called a  frame generator (resp. a  Parseval frame generator) w.r.t. $ \tau$ in $ \mathcal{H}$  if $(\{x_U\coloneqq  \pi(U)x\}_{U \in \mathcal{U}},\{\tau_g\coloneqq  \pi(U)\tau\}_{U \in \mathcal{U}}) $ is a  frame  (resp. Parseval frame) for  $ \mathcal{H}$. 
 In this case we write $ (x,\tau)$ is  a frame generator for $\pi$. 
 \end{definition}     
 \begin{theorem}\label{UNITARY1}
 Let $ \mathcal{U}$ be a  group-like unitary system  with identity $ I$ and $( \{x_U\}_{U \in \mathcal{U}},  \{\tau_U\}_{U \in \mathcal{U}})$ be a Parseval  frame  for  $\mathcal{H}$ with $ \theta_x^*$ or $\theta_\tau^*$ is injective. Then there is a unitary representation $ \pi$  of $  \mathcal{U}$ on  $ \mathcal{H}$  for which 
 $$ x_U=\pi(U)x_I, ~\tau_U=\pi(U)\tau_I , ~ \forall U \in \mathcal{U}$$
   if and only if 
 \begin{align*}
\langle x_{\sigma(UV)}, x_{\sigma(UW)}\rangle&=\overline{f(UV)}f(UW)\langle x_V, x_W \rangle, \\
\langle x_{\sigma(UV)},  \tau_{\sigma(UW)}\rangle&=\overline{f(UV)}f(UW)\langle x_V, \tau_W \rangle, \\
\langle \tau_{\sigma(UV)} , \tau_{\sigma(UW)}\rangle&=\overline{f(UV)}f(UW)\langle \tau_V, \tau_W \rangle  
\end{align*} 
 for all $ U, V, W \in \mathcal{U}.$
 \end{theorem}
 \begin{proof}
Proof 1. $(\Rightarrow)$ Let $ U, V, W \in \mathcal{U}.$ Then 	
\begin{align*}
 \langle x_{\sigma(UV)}, x_{\sigma(UW)}\rangle&=\langle \pi({\sigma(UV)})x_I, \pi({\sigma(UW)})x_I\rangle = \langle \overline{f(UV)}\pi(U)\pi(V)x_I, \overline{f(UW)}\pi(U)\pi(W)x_I \rangle \\
 & =  \overline{f(UV)}f(UW)\langle\pi(V)x_I, \pi(W)x_I \rangle
 =\overline{f(UV)}f(UW)\langle x_V, x_W \rangle, \forall U, V, W \in \mathcal{U}. 
\end{align*}
Others are similar.

$(\Leftarrow)$
We claim the following three, among them we prove first.
$$  \lambda_U\theta_x\theta_x^*=\theta_x\theta_x^*\lambda_U,~ \lambda_U\theta_x\theta_\tau^*=\theta_x\theta_\tau^*\lambda_U,~ \lambda_U\theta_\tau\theta_\tau^*=\theta_\tau\theta_\tau^*\lambda_U,~ \forall U \in \mathcal{U}. $$
We verify $  \lambda_U\theta_x\theta_x^*=\theta_x\theta_x^*\lambda_U$. For $\{c_V\}_{V\in \mathcal{U}} \in \ell^2(\mathcal{U}) $, consider
\begin{align*}
\lambda_U\theta_x\theta_x^*(\{c_V\}_{V\in \mathcal{U}})&=\lambda_U\theta_x\left(\sum_{V \in\mathcal{U} }c_Vx_V\right)=\lambda_U\left(\left\{\left\langle\sum_{V \in\mathcal{U}}c_Vx_V, x_W\right\rangle\right\}_{W\in \mathcal{U}}\right)\\
&=\lambda_U\left(\sum_{W \in\mathcal{U}}\left\langle\sum_{V \in\mathcal{U}}c_Vx_V, x_W\right\rangle\chi_W\right)=\sum_{W \in\mathcal{U}}\left\langle\sum_{V \in\mathcal{U}}c_Vx_V, x_W\right\rangle f(UW)\chi_{\sigma(UW)}\\
&=\sum_{W \in\mathcal{U}}\sum_{V \in\mathcal{U}}c_V  f(UV)\overline{ f(UV)}f(UW)\langle x_V, x_W\rangle \chi_{\sigma(UW)}\\
&=\sum_{W \in\mathcal{U}}\sum_{V \in\mathcal{U}}c_Vf(UV)\langle x_{\sigma(UV)}, x_{\sigma(UW)}\rangle \chi_{\sigma(UW)}=\sum_{V \in\mathcal{U}}\sum_{W \in\mathcal{U}}c_Vf(UV)\langle x_{\sigma(UV)}, x_{\sigma(UW)}\rangle \chi_{\sigma(UW)}\\
&=\sum_{V \in\mathcal{U}}\sum_{W \in\mathcal{U}}c_Vf(UV)\langle x_{\sigma(UV)}, x_{W}\rangle \chi_{W}=\sum_{W \in\mathcal{U}}\sum_{V \in\mathcal{U}}c_Vf(UV)\langle x_{\sigma(UV)}, x_{W}\rangle \chi_{W}\\
&= \sum_{W \in\mathcal{U}}\left \langle\sum_{V \in\mathcal{U}}c_Vf(UV) x_{\sigma(UV)}, x_W\right \rangle=\left\{\left \langle \sum_{V \in\mathcal{U}}c_Vf(UV) x_{\sigma(UV)}, x_W \right \rangle \right\}_{W \in\mathcal{U}}
\\
&=\theta_x\left(\sum_{V \in\mathcal{U}}c_Vf(UV) x_{\sigma(UV)} \right)=\theta_x\left(\sum_{V \in\mathcal{U}}c_Vf(UV) \theta_x^*\chi_{\sigma(UV)} \right)
\\
&=\theta_x\theta_x^*\left(\sum_{V \in\mathcal{U}}c_V\lambda_U \chi_{V} \right) =\theta_x\theta_x^*\lambda_U(\{c_V\}_{V\in \mathcal{U}}).
\end{align*}

 We next claim $ (\lambda_U\theta_x)^*\{c_V\}_{V\in \mathcal{U}}=\sum_{V\in \mathcal{U}}\overline{f(UV)}c_{\sigma(UV)}x_V, \forall\{c_V\}_{V\in \mathcal{U}} \in \ell^2(\mathcal{U}) .$ In fact, $ \langle(\lambda_U\theta_x)^*\{c_V\}_{V\in \mathcal{U}}, h\rangle =\langle \{c_V\}_{V\in \mathcal{U}},\lambda_U(\{\langle h, x_W  \rangle \}_{W \in \mathcal{U}}) \rangle   =\langle \{c_V\}_{V\in \mathcal{U}},\sum_{W\in \mathcal{U}}\langle h, x_W  \rangle f(UW)\chi_{\sigma(UW)} \rangle = \sum_{W\in \mathcal{U}}\overline{f(UW)}\langle x_W, h  \rangle c_{\sigma(UW)}=\langle\sum_{V\in \mathcal{U}}\overline{f(UV)}c_{\sigma(UV)}x_V, h \rangle ,  \forall\{c_V\}_{V\in \mathcal{U}} \in \ell^2(\mathcal{U}), \forall h \in \mathcal{H}. $
Similarly $ (\lambda_U\theta_\tau)^*\{c_V\}_{V\in \mathcal{U}}=\sum_{V\in \mathcal{U}}\overline{f(UV)}c_{\sigma(UV)}\tau_V$, $ \forall\{c_V\}_{V\in \mathcal{U}} \in \ell^2(\mathcal{U}) .$

Define $ \pi : \mathcal{U} \ni U  \mapsto \pi(U)\coloneqq \theta_\tau^*\lambda_U\theta_x  \in \mathcal{B}(\mathcal{H}).$  Then, $ \pi(U)\pi(V)=\theta_\tau^*\lambda_U\theta_x \theta_\tau^*\lambda_V\theta_x =\theta_\tau^*\theta_x \theta_\tau^*\lambda_U \lambda_V\theta_x = \theta_\tau^*\lambda_U\lambda_V\theta_x =\theta_\tau^*f(UV)\lambda_{\sigma(UV)} \theta_x  =f(UV) \theta_\tau^*\lambda_{\sigma(UV)}\theta_x =f(UV)\pi({\sigma(UV)})$ for all $U, V \in \mathcal{U},$ and
$\pi(U)\pi(U)^*=\theta_\tau^*\lambda_U\theta_x\theta_x^*\lambda^*_U\theta_\tau=\theta_\tau^*\theta_x\theta_x^*\lambda_U\lambda^*_U\theta_\tau=\theta_x^*\lambda_U\lambda^*_U\theta_\tau=\theta_x^*\theta_\tau=I_\mathcal{H}$ for all $ U \in \mathcal{U}$. Similarly $\pi(U)^*\pi(U)=\theta_x^*\lambda^*_U\theta_\tau\theta_\tau^*\lambda_U\theta_x=I_\mathcal{H} $ for all $U \in \mathcal{U}$.  We further find  
\begin{align*}
\pi(U)f(U^{-1})\pi(\sigma(U^{-1}))&=\theta_\tau^*\lambda_U\theta_xf(U^{-1})\theta_\tau^*\lambda_{\sigma(U^{-1})}\theta_x=f(U^{-1})\theta_\tau^*\theta_x\theta_\tau^*\lambda_U\lambda_{\sigma(U^{-1})}\theta_x\\
&=f(U^{-1})\theta_\tau^*\lambda_U\lambda_{\sigma(U^{-1})}\theta_x=f(U^{-1})\theta_\tau^*f(U\sigma(U^{-1}))\lambda_{\sigma(U\sigma(U^{-1}))}\theta_x\\
&=\theta_\tau^*f(U\sigma(U^{-1}I))f(U^{-1}I)\lambda_{\sigma(U\sigma(U^{-1}I))}\theta_x\\
&=\theta_\tau^*f(\sigma(UU^{-1})I)f(UU^{-1})\lambda_{\sigma(\sigma(UU^{-1})I)}\theta_x=\theta_\tau^*\theta_x=I_\mathcal{H} 
\end{align*}
for all $ U \in \mathcal{U}$. Therefore ${\pi(U)}^{-1}=f(U^{-1})\pi(\sigma(U^{-1}))$ for all $ U \in \mathcal{U}$. Using $ \theta_x^*$ is injective (or $ \theta_\tau^*$ is injective)  we show $ \pi$ is injective. Let $\pi(U)=\pi(V)$. Then $ \theta_\tau^*\lambda_U\theta_x=\theta_\tau^*\lambda_V\theta_x$ $\Rightarrow $ $ \theta_\tau^*\lambda_U\theta_x\theta_x^*=\theta_\tau^*\lambda_V\theta_x\theta_x^*$ $\Rightarrow $ $ \theta_\tau^*\theta_x\theta_x^*\lambda_U=\theta_\tau^*\theta_x\theta_x^*\lambda_V$ $\Rightarrow $ $\theta_x^*\lambda_U=\theta_x^*\lambda_V$ $\Rightarrow $ $\lambda_U=\lambda_V $ $\Rightarrow $ $U=V$.

Finally 

\begin{align*}
\pi(U)x_I&=\theta_\tau^*\lambda_U\theta_xx_I=(\lambda_U^*\theta_\tau)^*(\{\langle x_I,x_V \rangle \}_{V\in \mathcal{U}})=(f(U^{-1})\lambda_{\sigma(U^{-1})}\theta_\tau)^*(\{\langle x_I,x_V \rangle \}_{V\in \mathcal{U}})\\
&=\overline{f(U^{-1})}(\lambda_{\sigma(U^{-1})}\theta_\tau)^*(\{\langle x_I,x_V \rangle \}_{V\in \mathcal{U}})=\overline{f(U^{-1})}\sum_{V\in \mathcal{U}}\overline{f(\sigma(U^{-1})V)}\langle x_I,x_{\sigma(\sigma(U^{-1})V)} \rangle\tau_V\\
&=\overline{f(U^{-1})}\sum_{V\in \mathcal{U}}\overline{f(\sigma(U^{-1})V)}\langle x_{\sigma( \sigma(IU^{-1})U)},x_{\sigma(\sigma(U^{-1})V)} \rangle\tau_V\\
&=\overline{f(U^{-1})}\sum_{V\in \mathcal{U}}\overline{f(\sigma(U^{-1})V)}\overline{f(\sigma(IU^{-1})U)}f(\sigma(U^{-1})V)\langle x_{U},x_{V} \rangle\tau_V\\
&=\overline{f(U^{-1})}\overline{f(\sigma(IU^{-1})U)}\sum_{V\in\mathcal{U}}\langle x_{U},x_{V} \rangle\tau_V=\overline{f(\sigma(IU^{-1})U)}\overline{f(U^{-1})}x_U=\overline{f(\sigma(IU^{-1})U)}\overline{f(IU^{-1})}x_U\\
&=\overline{f(I\sigma(U^{-1}U))}\overline{f(U^{-1}U)}x_U=x_U
\end{align*}
and 
\begin{align*}
\pi(U)\tau_I&=\theta_\tau^*\lambda_U\theta_x\tau_I=(\lambda_U^*\theta_\tau)^*(\{\langle \tau_I,x_V \rangle \}_{V\in \mathcal{U}})=(f(U^{-1})\lambda_{\sigma(U^{-1})}\theta_\tau)^*(\{\langle \tau_I,x_V \rangle \}_{V\in \mathcal{U}})\\
&=\overline{f(U^{-1})}(\lambda_{\sigma(U^{-1})}\theta_\tau)^*(\{\langle \tau_I,x_V \rangle \}_{V\in \mathcal{U}})=\overline{f(U^{-1})}\sum_{V\in \mathcal{U}}\overline{f(\sigma(U^{-1})V)}\langle \tau_I,x_{\sigma(\sigma(U^{-1})V)} \rangle\tau_V\\
&=\overline{f(U^{-1})}\sum_{V\in \mathcal{U}}\overline{f(\sigma(U^{-1})V)}\langle \tau_{\sigma( \sigma(IU^{-1})U)},x_{\sigma(\sigma(U^{-1})V)} \rangle\tau_V\\
&=\overline{f(U^{-1})}\sum_{V\in \mathcal{U}}\overline{f(\sigma(U^{-1})V)}\overline{f(\sigma(IU^{-1})U)}f(\sigma(U^{-1})V)\langle \tau_{U},x_{V} \rangle\tau_V\\
&=\overline{f(U^{-1})}\overline{f(\sigma(IU^{-1})U)}\sum_{V\in\mathcal{U}}\langle \tau_{U},x_{V} \rangle\tau_V=\overline{f(\sigma(IU^{-1})U)}\overline{f(U^{-1})}\tau_U=\overline{f(\sigma(IU^{-1})U)}\overline{f(IU^{-1})}\tau_U\\
&=\overline{f(I\sigma(U^{-1}U))}\overline{f(U^{-1}U)}\tau_U=\tau_U.
\end{align*}
Proof 2. If we define $A_U: \mathcal{H} \ni h \mapsto \langle h, x_U \rangle \in \mathbb{K} $, $\Psi_U: \mathcal{H} \ni h \mapsto \langle h, \tau_U \rangle \in \mathbb{K}, \forall U \in \mathcal{U} $, then   $(\{x_U\}_{U\in \mathcal{U}}, \{\tau_U\}_{U\in \mathcal{U}})$ is a Parseval frame for  $\mathcal{H}$ if and only if  $(\{A_U\}_{U\in \mathcal{U}}, \{\Psi_U\}_{U\in \mathcal{U}})$ is a Parseval (ovf)  in $\mathcal{B}(\mathcal{H},\mathbb{K})$. It is also easy to see that $ \theta_x^*$ (resp. $ \theta_\tau^*$) is injective if and only if $ \theta_A^*$ (resp. $ \theta_\Psi^*$) is injective. At this point we naturally use Theorem \ref{CHARACTERIZATIONGROUPLIKE} to get -  there is a unitary representation $ \pi$  of $ \mathcal{U}$ on  $ \mathcal{H}$  for which 
$$ A_U=A_I\pi(U)^{-1}, ~\Psi_U=\Psi_I\pi(U)^{-1},~\forall  U \in \mathcal{U}$$
if and only if 
\begin{align*}
A_{\sigma(UV)}A_{\sigma(UW)}^*&=f(UV)\overline{f(UW)} A_VA_W^* ,
A_{\sigma(UV)}\Psi_{\sigma(UW)}^*=f(UV)\overline{f(UW)} A_V\Psi_W^*,\\ \Psi_{\sigma(UV)}\Psi_{\sigma(UW)}^*&=f(UV)\overline{f(UW)} \Psi_V\Psi_W^*,~ \forall  U,V,W \in \mathcal{U}.
\end{align*}
Now for all $h\in \mathcal{H}$, for all $\alpha \in \mathbb{C}$, and for all $U \in \mathcal{U} $,
\begin{align*}
&\langle h,x_U \rangle =A_Uh=A_I\pi(U)^{-1}h=\langle \pi(U)^{-1}h,x_U \rangle=\langle h,\pi(U)x_U \rangle,\\
&\langle h,\tau_U \rangle =\Psi_Uh=\Psi_I\pi(U)^{-1}h=\langle \pi(U)^{-1}h,\tau_U \rangle=\langle h,\pi(U)\tau_U \rangle,
\end{align*}
and 
\begin{align*}
\langle \alpha x_{\sigma(UW)},x_{\sigma(UV)} \rangle =A_{\sigma(UV)}A_{\sigma(UW)}^*\alpha&=f(UV)\overline{f(UW)} A_VA_W^*\alpha=\langle \alpha x_{\sigma(W)},x_{\sigma(V)} \rangle ,\\
\langle \alpha \tau_{\sigma(UW)},x_{\sigma(UV)} \rangle=A_{\sigma(UV)}\Psi_{\sigma(UW)}^*\alpha&=f(UV)\overline{f(UW)} A_V\Psi_W^*\alpha=\langle \alpha \tau_{\sigma(W)},x_{\sigma(V)} \rangle,\\ 
\langle \alpha \tau_{\sigma(UW)},\tau_{\sigma(UV)} \rangle=\Psi_{\sigma(UV)}\Psi_{\sigma(UW)}^*\alpha&=f(UV)\overline{f(UW)} \Psi_V\Psi_W^*\alpha=\langle \alpha \tau_{\sigma(W)},\tau_{\sigma(V)} \rangle.
\end{align*}
 Hence the theorem.
\end{proof}
\begin{corollary}
 Let $ \mathcal{U}$ be a  group-like unitary system  with identity $ I$ and $( \{x_U\}_{U \in \mathcal{U}},  \{\tau_U\}_{U \in \mathcal{U}})$ be a   frame  for  $\mathcal{H}$ with $ \theta_x^*$ or $\theta_\tau^*$ is injective. Then there is a unitary representation $ \pi$ of  $  \mathcal{U}$ on  $ \mathcal{H}$  for which
\begin{enumerate}[\upshape(i)]
\item  $ x_U=S_{x,\tau}\pi(U)S_{x,\tau}^{-1}x_I, \tau_U=\pi(U)\tau_I $ for all $ U \in \mathcal{U}$  if and only if $\langle S_{x,\tau}^{-2}x_{\sigma(UV)}, x_{\sigma(UW)}\rangle=\overline{f(UV)}f(UW)\langle S_{x,\tau}^{-2}x_V, x_W \rangle$, $
\langle S_{x,\tau}^{-1}x_{\sigma(UV)},  \tau_{\sigma(UW)}\rangle=\overline{f(UV)}f(UW)\langle S_{x,\tau}^{-1} x_V, \tau_W \rangle, \langle \tau_{\sigma(UV)} , \tau_{\sigma(UW)}\rangle=\overline{f(UV)}f(UW)\langle \tau_V, \tau_W \rangle  $ for all $ U, V, W \in \mathcal{U}.$
\item  $ x_U=S_{x,\tau}^{1/2}\pi(U)S_{x,\tau}^{-1/2}x_I, \tau_U=S_{x,\tau}^{1/2}\pi(U)S_{x,\tau}^{-1/2}\tau_I $ for all $ U \in \mathcal{U}$  if and only if $\langle S_{x,\tau}^{-1}x_{\sigma(UV)}, x_{\sigma(UW)}\rangle=\overline{f(UV)}f(UW)\langle S_{x,\tau}^{-1}x_V, x_W \rangle, 
\langle S_{x,\tau}^{-1}x_{\sigma(UV)},  \tau_{\sigma(UW)}\rangle=\overline{f(UV)}f(UW)\langle S_{x,\tau}^{-1} x_V, \tau_W \rangle, \langle S_{x,\tau}^{-1}\tau_{\sigma(UV)} , \tau_{\sigma(UW)}\rangle=\overline{f(UV)}f(UW)\langle S_{x,\tau}^{-1}\tau_V, \tau_W \rangle  $ for all $ U, V, W \in \mathcal{U}.$
\item  $ x_U=\pi(U)x_I, \tau_U=\pi(U)S_{x,\tau}^{-1}\tau_I $ for all $ U \in \mathcal{U}$  if and only if $\langle x_{\sigma(UV)}, x_{\sigma(UW)}\rangle=\overline{f(UV)}f(UW)\langle x_V, x_W \rangle$, $ \langle x_{\sigma(UV)}, S_{x,\tau}^{-1} \tau_{\sigma(UW)}\rangle=\overline{f(UV)}f(UW)\langle  x_V, S_{x,\tau}^{-1}\tau_W \rangle, \langle \tau_{\sigma(UV)} ,S_{x,\tau}^{-2}\tau_{\sigma(UW)}\rangle=\overline{f(UV)}f(UW)\langle \tau_V, S_{x,\tau}^{-2}\tau_W \rangle  $ for all $ U, V, W \in \mathcal{U}.$
\end{enumerate}
\end{corollary}  
\begin{proof} Apply Theorem \ref{UNITARY1} to 
\begin{enumerate}[\upshape(i)]
\item $(\{S^{-1}_{x,\tau}x_U\}_{U \in \mathcal{U}},  \{\tau_U\}_{U \in \mathcal{U}}) $ to get: there is a unitary representation $ \pi$  of $ \mathcal{U}$ on  $ \mathcal{H}$  for which $ S_{x,\tau}^{-1}x_U=\pi(U)(S_{x,\tau}^{-1}x_I), \tau_U=\pi(U)\tau_I $ for all $U \in \mathcal{U}$  if and only if $\langle S_{x,\tau}^{-1}x_{\sigma(UV)} , S_{x,\tau}^{-1}x_{\sigma(UW)}\rangle=\overline{f(UV)}f(UW)\langle S_{x,\tau}^{-1}x_V,S_{x,\tau}^{-1} x_W \rangle $,  $ \langle S_{x,\tau}^{-1} x_{\sigma(UV)},  \tau_{\sigma(UW)}\rangle=\overline{f(UV)}f(UW)\langle S_{x,\tau}^{-1} x_V, \tau_W \rangle $,  $\langle \tau_{\sigma(UV)} , \tau_{\sigma(UW)}\rangle=\overline{f(UV)}f(UW)\langle \tau_V, \tau_W \rangle  $ for all $ U, V, W \in \mathcal{U}.$
\item $(\{S^{-1/2}_{x,\tau}x_U\}_{U \in \mathcal{U}},  \{S^{-1/2}_{x,\tau}\tau_U\}_{U \in \mathcal{U}}) $ to get: there is a unitary representation $ \pi$  of $ \mathcal{U}$ on  $ \mathcal{H}$  for which $ S_{x,\tau}^{-1/2}x_U=\pi(U)(S_{x,\tau}^{-1/2}x_I), S_{x,\tau}^{-1/2}\tau_U=\pi(U)(S_{x,\tau}^{-1/2}\tau_I) $ for all $ U \in \mathcal{U}$  if and only if $\langle S_{x,\tau}^{-1/2}x_{\sigma(UV)} , S_{x,\tau}^{-1/2}x_{\sigma(UW)}\rangle=\overline{f(UV)}f(UW)\langle S_{x,\tau}^{-1/2}x_V, S_{x,\tau}^{-1/2}x_W \rangle, $ $ \langle S_{x,\tau}^{-1/2}x_{\sigma(UV)},  S_{x,\tau}^{-1/2}\tau_{\sigma(UW)}\rangle=\overline{f(UV)}f(UW)\langle S_{x,\tau}^{-1/2}x_V, S_{x,\tau}^{-1/2}\tau_W \rangle, $ $\langle S_{x,\tau}^{-1/2}\tau_{\sigma(UV)},
S_{x,\tau}^{-1/2}\tau_{\sigma(UW)}\rangle=\overline{f(UV)}f(UW)\langle S_{x,\tau}^{-1/2}\tau_V, S_{x,\tau}^{-1/2}\tau_W \rangle  $ for all $ U, V, W \in \mathcal{U}.$
\item $(\{x_g\}_{U \in \mathcal{U}},  \{S^{-1}_{x,\tau}\tau_g\}_{U \in \mathcal{U}}) $ to get: there is a unitary representation $ \pi$  of $ \mathcal{U}$ on  $ \mathcal{H}$  for which $ x_U=\pi(U)x_I, S_{x,\tau}^{-1}\tau_U=\pi(U)(S_{x,\tau}^{-1}\tau_I) $ for all $ U \in \mathcal{U}$  if and only if $\langle x_{\sigma(UV)} , x_{\sigma(UW)}\rangle=\overline{f(UV)}f(UW)\langle x_V, x_W \rangle, $ $ \langle x_{\sigma(UV)},  S_{x,\tau}^{-1}\tau_{\sigma(UW)}\rangle=\overline{f(UV)}f(UW)\langle x_V, S_{x,\tau}^{-1}\tau_W \rangle,  $ $\langle S_{x,\tau}^{-1}\tau_{\sigma(UV)} , S_{x,\tau}^{-1}\tau_{\sigma(UW)}\rangle=\overline{f(UV)}f(UW)\langle S_{x,\tau}^{-1}\tau_V, S_{x,\tau}^{-1}\tau_W \rangle  $ for all $ U, V, W \in \mathcal{U}.$
\end{enumerate}
\end{proof}

\textbf{Perturbations}  
\begin{theorem}\label{PERTURBATION1SV}
Let $ (\{x_j\}_{j \in \mathbb{J}},\{\tau_j\}_{j \in \mathbb{J}} )$ be a frame for  $\mathcal{H}. $ Suppose $ \{y_j\}_{j \in \mathbb{J}}$  in $ \mathcal{H}$ is  such that $ \langle h,y_j \rangle \tau_j=\langle h,\tau_j \rangle y_j, \langle h, y_j\rangle  \langle \tau_j,h \rangle \geq 0, \forall h \in \mathcal{H}, \forall j \in \mathbb{J}$ and there exist $ \alpha, \beta, \gamma \geq0$ with  $ \max\{\alpha+\gamma\|\theta_\tau S_{x,\tau}^{-1}\|, \beta\}<1$ and for every finite subset $ \mathbb{S}$ of $ \mathbb{J}$
\begin{align}\label{PERTURBATIONINEQUALITYAEQUENTIALFIRST}
\left\|\sum\limits_{j\in \mathbb{S}}c_j(x_j-y_j) \right\|\leq \alpha\left\|\sum\limits_{j\in \mathbb{S}}c_jx_j\right \|+\gamma \left(\sum\limits_{j\in \mathbb{S}}|c_j|^2\right)^\frac{1}{2}+\beta\left\|\sum\limits_{j\in \mathbb{S}}c_jy_j\right \|,~ \forall c_j \in \mathbb{K}, \forall j \in \mathbb{S}.
\end{align}	 
Then  $ (\{y_j\}_{j \in \mathbb{J}},\{\tau_j\}_{j \in \mathbb{J}} )$ is  a frame  with bounds $ \frac{1-(\alpha+\gamma\|\theta_\tau S_{x,\tau}^{-1}\|)}{(1+\beta)\|S_{x,\tau}^{-1}\|}$ and $\frac{\|\theta_\tau\|((1+\alpha)\|\theta_x\|+\gamma)}{1-\beta} $.
\end{theorem}
\begin{proof}
Let $ \mathbb{S}\subseteq\mathbb{J}$ be finite  and  $ \{c_j\}_{j \in \mathbb{J}} \in \ell^2(\mathbb{J}) $. Then
\begin{align*}
\left\|\sum_{j \in \mathbb{S}}c_jy_j\right\| \leq \left\|\sum_{j \in \mathbb{S}}c_j(x_j-y_j)\right\| +\left\|\sum_{j \in \mathbb{S}}c_jx_j\right\|
=(1+\alpha)\left\|\sum_{j \in \mathbb{S}}c_jx_j\right\|+\beta\left\|\sum_{j \in \mathbb{S}}c_jy_j\right\|+\gamma \left(\sum\limits_{j\in \mathbb{S}}|c_j|^2\right)^\frac{1}{2}.
\end{align*}
Therefore 
\begin{align*}
\left\|\sum_{j \in \mathbb{S}}c_jy_j\right\|\leq \frac{1+\alpha}{1-\beta}\left\|\sum_{j \in \mathbb{S}}c_jx_j\right\|+\frac{\gamma}{1-\beta} \left(\sum\limits_{j\in \mathbb{S}}|c_j|^2\right)^\frac{1}{2}.
\end{align*}
Hence $ \ell^2(\mathbb{J}) \ni \{c_j\}_{j \in \mathbb{J}} \mapsto \sum_{j \in \mathbb{J}}c_jy_j  \in \mathcal{H}$ is a well-defined bounded linear operator with norm $ \leq \frac{1+\alpha}{1-\beta}\|\theta_x^*\|+\frac{\gamma}{1-\beta}$ and hence $ \theta_y$ exists, $\|\theta_y\|\leq \frac{1+\alpha}{1-\beta}\|\theta_x\|+\frac{\gamma}{1-\beta} $. Now from Inequality (\ref{PERTURBATIONINEQUALITYAEQUENTIALFIRST}),
\begin{equation}\label{INTERMEDIATE}
\|\theta_x^*(\{c_j\}_{j \in \mathbb{J}})-\theta_y^*(\{c_j\}_{j \in \mathbb{J}})\|\leq \alpha\|\theta_x^*(\{c_j\}_{j \in \mathbb{J}})\|+\beta\|\theta_y^*(\{c_j\}_{j \in \mathbb{J}})\|+\gamma \left(\sum\limits_{j\in \mathbb{J}}|c_j|^2\right)^\frac{1}{2} ,\quad \forall \{c_j\}_{j \in \mathbb{J}} \in \ell^2(\mathbb{J}).
\end{equation}
Let $ \{e_j\}_{j \in \mathbb{J}}$ be the standard orthonormal basis for $\ell^2(\mathbb{J}) $.  Since $ \{\tau_j\}_{j \in \mathbb{J}}$ is a Bessel sequence (w.r.t. itself), $\{\langle \theta_\tau S_{x,\tau}^{-1}h, e_j\rangle \}_{j \in \mathbb{J}} \in \ell^2(\mathbb{J}) $ for each $ h \in \mathcal{H}$. This implies $\{\langle  S_{x,\tau}^{-1}h, \theta_\tau^*e_j\rangle =\langle  S_{x,\tau}^{-1}h, \tau_j\rangle \}_{j \in \mathbb{J}} \in \ell^2(\mathbb{J}) $ for each $ h \in \mathcal{H}$. So by considering  $\theta_\tau S_{x,\tau}^{-1}h=\{\langle  S_{x,\tau}^{-1}h, \tau_j\rangle \}_{j \in \mathbb{J}} \in \ell^2(\mathbb{J})  $ in Inequality (\ref{INTERMEDIATE}),
\begin{align*}
 \|h-S_{y,\tau}S_{x,\tau}^{-1}h\|&\leq \alpha \|h\|+\beta\|S_{y,\tau}S_{x,\tau}^{-1}h\|+\gamma\left(\sum\limits_{j\in \mathbb{J}}|\langle  S_{x,\tau}^{-1}h, \tau_j\rangle|^2\right)^\frac{1}{2}\\
 &=\alpha \|h\|+\beta\|S_{y,\tau}S_{x,\tau}^{-1}h\|+\gamma\left(\sum\limits_{j\in \mathbb{J}}|\langle  \theta_\tau S_{x,\tau}^{-1}h, e_j\rangle|^2\right)^\frac{1}{2}\\
 &=\alpha \|h\|+\beta\|S_{y,\tau}S_{x,\tau}^{-1}h\|+\gamma\| \theta_\tau S_{x,\tau}^{-1}h\|=(\alpha+\gamma\|\theta_\tau S_{x,\tau}^{-1}\|)\|h\|+\beta\|S_{y,\tau}S_{x,\tau}^{-1}h\| , ~ \forall h \in \mathcal{H}.
 \end{align*}
Thus $S_{y,\tau}S_{x,\tau}^{-1} $ is invertible. Rest is similar to the proof of Theorem \ref{PERTURBATION RESULT 1}.
\end{proof} 
\begin{corollary}
Let $ (\{x_j\}_{j \in \mathbb{J}},\{\tau_j\}_{j \in \mathbb{J}} )$ be a frame for  $\mathcal{H}. $ Suppose $ \{y_j\}_{j \in \mathbb{J}}$  in $ \mathcal{H}$ is  such that $ \langle h,y_j \rangle \tau_j=\langle h,\tau_j \rangle y_j, \langle h, y_j\rangle  \langle \tau_j,h \rangle \geq 0, \forall h \in \mathcal{H}, \forall j \in \mathbb{J}$ and 
$$ r\coloneqq \sum_{j\in \mathbb{J}}\|x_j-y_j\|^2<\frac{1}{\|\theta_\tau S_{x,\tau}^{-1}\|^2}.$$
Then  $ (\{y_j\}_{j \in \mathbb{J}},\{\tau_j\}_{j \in \mathbb{J}} )$ is  a frame  with bounds $ \frac{1-\sqrt{r}\|\theta_\tau S_{x,\tau}^{-1}\|}{\|S_{x,\tau}^{-1}\|}$ and $\|\theta_\tau\|(\|\theta_x\|+\sqrt{r}) $.
\end{corollary}
\begin{proof}
Take $ \alpha =0, \beta=0, \gamma=\sqrt{r}$. Then $ \max\{\alpha+\gamma\|\theta_\tau S_{x,\tau}^{-1}\|, \beta\}<1$ and for every finite subset $ \mathbb{S}$ of $ \mathbb{J}$,
$$\left\|\sum\limits_{j\in \mathbb{S}}c_j(x_j-y_j) \right\|\leq  \left(\sum\limits_{j\in \mathbb{S}}|c_j|^2\right)^\frac{1}{2} \left(\sum\limits_{j\in \mathbb{S}}\|x_j-y_j\|^2\right)^\frac{1}{2}\leq \gamma  \left(\sum\limits_{j\in \mathbb{S}}|c_j|^2\right)^\frac{1}{2},~\forall c_j \in \mathbb{K}, \forall j \in \mathbb{S}.$$	
 Now we can apply Theorem \ref{PERTURBATION1SV}.
\end{proof}
\begin{theorem}\label{PERTURBATION2SV}
Let $ (\{x_j\}_{j \in \mathbb{J}},\{\tau_j\}_{j \in \mathbb{J}} )$ be a frame for  $\mathcal{H}$ with bounds $ a$ and $b.$ Suppose $ \{y_j\}_{j \in \mathbb{J}}$  in $ \mathcal{H}$ is  such that   $ \sum_{j\in\mathbb{J}}\langle h,y_j \rangle \langle \tau_j, h \rangle$ exists for all $ h \in \mathcal{H}$ and  is nonnegative for all $ h \in \mathcal{H}$ and there exist $ \alpha, \beta, \gamma \geq0$ with  $ \max\{\alpha+\frac{\gamma}{\sqrt{a}}, \beta\}<1$ and
\begin{align*}
\left| \sum\limits_{j\in\mathbb{J}}\langle h,x_j-y_j\rangle \langle \tau_j, h\rangle \right|^\frac{1}{2} \leq \alpha\left(\sum\limits_{j\in\mathbb{J}}\langle h,x_j\rangle \langle \tau_j,h \rangle \right)^\frac{1}{2} + \beta\left(\sum\limits_{j\in\mathbb{J}}\langle h,y_j\rangle \langle \tau_j,h \rangle \right)^\frac{1}{2} +\gamma \|h\|, ~\forall h \in \mathcal{H}.
\end{align*}
Then $ (\{y_j\}_{j \in \mathbb{J}},\{\tau_j\}_{j \in \mathbb{J}} )$ is  a frame with bounds $a\left(1-\frac{\alpha+\beta+\frac{\gamma}{\sqrt{a}}}{1+\beta}\right)^2 $ and $b\left(1+\frac{\alpha+\beta+\frac{\gamma}{\sqrt{b}}}{1-\beta}\right)^2.$
\end{theorem}
\begin{proof}
Similar to proof of Theorem \ref{PERTURBATION RESULT 2}.
\end{proof} 
\begin{theorem}\label{PERTURBATION3SV}
Let $ (\{x_j\}_{j\in \mathbb{J}}, \{\tau_j\}_{j\in \mathbb{J}}) $  be  a frame for  $\mathcal{H}$. Suppose  $\{y_j\}_{j\in \mathbb{J}} $ in $\mathcal{H}$ is such that $ \langle h,y_j \rangle \tau_j=\langle h,\tau_j \rangle y_j, \langle h, y_j\rangle  \langle \tau_j,h \rangle \geq 0, \forall h \in \mathcal{H}, \forall j \in \mathbb{J}$, $   \sum_{j\in \mathbb{J}}\|x_j-y_j\|^2$ converges, and 
$\sum_{j\in \mathbb{J}}\|x_j-y_j\|\|S_{x,\tau}^{-1}\tau_j\|<1.$
Then  $ (\{y_j\}_{j\in \mathbb{J}}, \{\tau_j\}_{j\in \mathbb{J}}) $ is a frame with bounds  $\frac{1-\sum_{j\in \mathbb{J}}\|x_j-y_j\|\|S_{x,\tau}^{-1}\tau_j\|}{\|S_{x,\tau}^{-1}\|}$ and $\|\theta_\tau\|(\sum_{j\in \mathbb{J}}\|x_j-y_j\|^2+\|\theta_x\|) $.
\end{theorem}
\begin{proof}
Let $\alpha=\sum_{j\in \mathbb{J}}\|x_j-y_j\|^2 $,  $\beta=\sum_{j\in \mathbb{J}}\|x_j-y_j\|\|S_{x,\tau}^{-1}\tau_j\| $,  $ \mathbb{S}\subseteq\mathbb{J}$ be finite, $h \in \mathcal{H}$  and  $ \{c_j\}_{j \in \mathbb{J}} \in \ell^2(\mathbb{J}) $. Then
\begin{align*}
\left\|\sum_{j \in \mathbb{S}}c_jy_j\right\| &\leq \left\|\sum_{j \in \mathbb{S}}c_j(x_j-y_j)\right\| +\left\|\sum_{j \in \mathbb{S}}c_jx_j\right\|
\leq \left(\sum\limits_{j\in \mathbb{S}}|c_j|^2\right)^\frac{1}{2}\left(\sum\limits_{j\in \mathbb{S}}\|x_j-y_j\|^2\right)^\frac{1}{2}+\left\|\sum_{j \in \mathbb{S}}c_jx_j\right\|\\
&\leq \alpha\left(\sum\limits_{j\in \mathbb{S}}|c_j|^2\right)^\frac{1}{2}+\left\|\sum_{j \in \mathbb{S}}c_jx_j\right\| ,
\end{align*}
\begin{align*}
\|h-S_{y,\tau}S_{x,\tau}^{-1}h\|&=\left\|\sum_{j \in \mathbb{J}}\langle h,x_j \rangle S_{x,\tau}^{-1}\tau_j-\sum_{j \in \mathbb{J}}\langle h,y_j \rangle S_{x,\tau}^{-1}\tau_j \right\|=\left\|\sum_{j \in \mathbb{J}}\langle h,x_j-y_j \rangle S_{x,\tau}^{-1}\tau_j\right\|\\
&\leq\sum_{j \in \mathbb{J}}\|h\|\|x_j-y_j\|\| S_{x,\tau}^{-1}\tau_j\|=\beta \|h\|.
\end{align*}	
Others are similar to the proof of Theorem \ref{OVFQUADRATICPERTURBATION}.
\end{proof}
\section{The finite dimensional case}\label{THEFINITEDIMENSIONALCASE}
\begin{theorem}\label{FINITEDIMENSIONALCHARATERIZATIONHILBERT}
Let $ \mathcal{H}$ be a finite dimensional Hilbert space, $ \{x_j\}_{j=1}^n , \{\tau_j\}_{j=1}^n $ be  a finite set of vectors in $ \mathcal{H}$  such that $\langle h,  x_j \rangle \tau_j =\langle h, \tau_j \rangle x_j , \langle h, x_j\rangle \langle\tau_j, h\rangle\geq0, \forall h \in  \mathcal{H}, \forall j=1,...,n. $ Then $ (\{x_j\}_{j=1}^n , \{\tau_j\}_{j=1}^n)$ is a frame for $ \mathcal{H}$ if and only if for every set $\{x_{r_1},...,x_{r_k}, \tau_{s_1},...,\tau_{s_l}: r_1,...,r_k,s_1,...,s_l\in \{1, 2, ..., n\}, k+l=n\} $ of $ n$ vectors with $\{r_1,...,r_k\} \cap \{s_1,...,s_l\}=\emptyset$ one has $ \operatorname{span}\{x_{r_1},...,x_{r_k}, \tau_{s_1},...,\tau_{s_l}\}= \mathcal{H}.$
\end{theorem} 
\begin{proof} We can assume $\mathcal{H}$ is not trivial.
	
$(\Leftarrow)$ Proof 1. There exists $1\leq  j\leq n$  such that $x_j\neq0\neq\tau_j$ (otherwise $ S_{x, \tau}$ is zero). We note that $ S_{x, \tau}$ is self-adjoint and the upper frame bound condition is satisfied. In fact, $\sum_{j=1}^{n}\langle h, x_j\rangle \langle\tau_j, h\rangle \leq (\sum_{j=1}^{n}\|x_j\|\|\tau_j\|)\|h\|^2,  \forall h \in  \mathcal{H}.$ Define $ \phi :  \mathcal{H} \ni h \mapsto  \sum_{j=1}^{n}\langle h, x_j\rangle \langle\tau_j, h\rangle \in \mathbb{R}.$ Then $ \phi$ is continuous. Since the unit sphere of  $\mathcal{H}$ is compact, there exists $ g \in  \mathcal{H}$ with $ \|g\|=1$ such that $ a\coloneqq\sum_{j=1}^{n}\langle g, x_j\rangle \langle\tau_j, g\rangle=\inf\{\sum_{j=1}^{n}\langle h, x_j\rangle \langle\tau_j, h\rangle: h \in\mathcal{H}, \|h\|=1 \}.$ We wish to say $ a>0.$ If this is false:  since $\langle g, x_j\rangle \langle\tau_j, g\rangle\geq0, \forall j  $ we must have $ \langle g, x_j\rangle \langle\tau_j, g\rangle=0, \forall j .$ Let $ x_{r_1},...,x_{r_k}$ be the only $x_j$'s such that $\langle g, x_j\rangle=0.$ Then there exist $ \tau_{s_1},...,\tau_{s_l}$ such that $ \langle g, \tau_{s_1}\rangle=\cdots=\langle g, \tau_{s_l}\rangle =0$, $ k+l=n$ and $\{r_1,...,r_k\} \cap \{s_1,...,s_l\}=\emptyset$. But then $ g \perp \operatorname{span}\{x_{r_1},...,x_{r_k}, \tau_{s_1},...,\tau_{s_l}\}= \mathcal{H}$ which implies $ g=0$ which is forbidden. Now the argument is:  $ a$ is lower frame bound. Indeed, for a nonzero $ h$, $ a\|h\|^2\leq(\sum_{j=1}^{n}\langle \frac{h}{\|h\|}, x_j\rangle \langle\tau_j, \frac{h}{\|h\|}\rangle)\|h\|^2=\sum_{j=1}^{n}\langle h, x_j\rangle \langle\tau_j, h\rangle. $
	
Proof 2. We prove by contrapositive. Suppose $ (\{x_j\}_{j=1}^n , \{\tau_j\}_{j=1}^n)$ is not a frame for $ \mathcal{H}.$ Then lower frame bound condition is violated. Then for each $ m \in \mathbb{N}$, there exists $ y_m \in \mathcal{H}$ such that $\sum_{j=1}^{n}\langle y_m, x_j\rangle \langle\tau_j, y_m\rangle \leq \frac{1}{m}\|y_m\|^2.$ By normalizing and omitting $y_m$'s which are zeros we may assume $ \|y_m\|=1, \forall m.$ Now we take the finite dimensionality of $ \mathcal{H}$ to get a convergent subsequence $ \{y_{m_k}\}_{k=1}^\infty$ of $ \{y_m\}_{m=1}^\infty$ converging  to $y \in \mathcal{H}$, as $ k \rightarrow \infty$. Then $\lim_{k\rightarrow\infty}\sum_{j=1}^{n}\langle y_{m_k}, x_j\rangle \langle\tau_j, y_{m_k}\rangle= \sum_{j=1}^{n}\langle y, x_j\rangle \langle\tau_j, y\rangle=0 $ and $\|y\|=1.$ This gives there exists a set $\{x_{r_1},...,x_{r_k}, \tau_{s_1},...,\tau_{s_l}: r_1,...,r_k,s_1,...,s_l\in \{1, 2, ..., n\}, k+l=n\} $ of $ n$ vectors with  $\{r_1,...,r_k\} \cap \{s_1,...,s_l\}=\emptyset$ and  $ \operatorname{span}\{x_{r_1},...,x_{r_k}, \tau_{s_1},...,\tau_{s_l}\}\subsetneq\mathcal{H}.$
	
$ (\Rightarrow)$ Again contrapositive. Suppose  there exists a set $\{x_{r_1},...,x_{r_k}, \tau_{s_1},...,\tau_{s_l}: r_1,...,r_k,s_1,...,s_l\in \{1, 2, ..., n\}, k+l=n\} $ of $ n$ vectors with  $\{r_1,...,r_k\} \cap \{s_1,...,s_l\}=\emptyset$ and  $ \operatorname{span}\{x_{r_1},...,x_{r_k}, \tau_{s_1},...,\tau_{s_l}\}\subsetneq\mathcal{H}.$ Let $ h \in \mathcal{H}$ be nonzero such that $ h\perp \operatorname{span}\{x_{r_1},...,x_{r_k}, \tau_{s_1},...,\tau_{s_l}\}.$ Now because of  $\{r_1,...,r_k\} \cap \{s_1,...,s_l\}=\emptyset$ and $k+l=n $ we get $ \sum_{j=1}^{n}\langle h, x_j\rangle \langle\tau_j, h\rangle =0 $ which says the lower frame bound condition fails.
\end{proof} 
\begin{proposition}
Let $ (\{x_j\}_{j=1}^n , \{\tau_j\}_{j=1}^n) $ be  a frame for a finite dimensional Hilbert space $\mathcal{H}$ with a lower frame bound  $a>1 $ and $ \|x_j\|=\|\tau_j\|=1, \forall j =1,...,n$. If $ \mathbb{S}$ is any subset of $ \{1, ..., n\}$ such that $\operatorname{Card}(\mathbb{S})<a ,$  then $ (\{x_j\}_{j\notin\mathbb{S}},\{\tau_j\}_{j\notin\mathbb{S}})$ is a frame for $\mathcal{H} $ with lower frame bound $a-\operatorname{Card}(\mathbb{S})$.
\end{proposition}
\begin{proof}
\begin{align*}
a\|h\|^2&\leq \sum_{j=1}^n\langle h,x_j \rangle\langle \tau_j, h\rangle=\sum_{j\in \mathbb{S}}\langle h,x_j \rangle\langle \tau_j, h\rangle+\sum_{j \notin \mathbb{S}}\langle h,x_j \rangle\langle \tau_j, h\rangle \\
&\leq\sum_{j\in \mathbb{S}}\|h\|^2\|x_j\|\|\tau_j\| +\sum_{j \notin \mathbb{S}}\langle h,x_j \rangle\langle \tau_j,h\rangle=\operatorname{Card}(\mathbb{S})\|h\|^2+\sum_{j \notin \mathbb{S}}\langle h,x_j \rangle\langle \tau_j,h\rangle, ~\forall h \in \mathcal{H}.
\end{align*}
\end{proof}

\begin{theorem}\label{FINITEDIMENSIONALPASTINGTHEOREM}
Let   $ (\{x_j\}_{j=1}^n , \{\tau_j\}_{j=1}^n) $ be  a frame for a finite dimensional complex Hilbert space $ \mathcal{H}$ of dimension $ m$. Then we have the following. 
\begin{enumerate}[\upshape(i)]
\item The optimal lower frame bound (resp. optimal upper frame bound) is the smallest (resp. largest) eigenvalue for $ S_{x, \tau}.$ 
\item If $ \{\lambda_j\}_{j=1}^m$ denote the eigenvalues for $S_{x, \tau},$  each appears as many times as its algebraic multiplicity, then  
$$ \sum_{j=1}^m\lambda_j=\sum_{j=1}^n\langle x_j,\tau_j \rangle =\sum_{j=1}^n\langle \tau_j, x_j \rangle.$$
\item There exist $m-1$ vectors $ \{y_j\}_{j=2}^{m}$ in $ \mathcal{H}$ such that $(\{x_j\}_{j=1}^n \cup\{y_j\}_{j=2}^{m} , \{\tau_j\}_{j=1}^n\cup\{y_j\}_{j=2}^{m})$ is a tight frame for $ \mathcal{H}.$
\item Condition number for  $S_{x, \tau}$ is  equal to the ratio between the optimal upper frame bound and the optimal lower frame bound.
\item If the optimal upper frame bound is $b,$  then 
$$ b\leq\sum_{j=1}^n\langle x_j,\tau_j \rangle =\sum_{j=1}^n\langle \tau_j, x_j \rangle\leq mb. $$
\item \begin{align*}
&\operatorname{Trace}(S_{x,\tau})=\sum_{j=1}^n\langle x_j,\tau_j \rangle=\sum_{j=1}^n\langle \tau_j,x_j \rangle;\\
 &\operatorname{Trace}(S^2_{x,\tau}) =\sum_{j=1}^n\sum_{k=1}^n\langle \tau_j,x_k\rangle \langle \tau_k,x_j\rangle=\sum_{j=1}^n\sum_{k=1}^n\langle \tau_j,\tau_k\rangle \langle x_k,x_j\rangle .
\end{align*}
\item If the frame is tight, then the optimal frame bound $b=\frac{1}{m}\sum_{j=1}^n\langle x_j,\tau_j \rangle=\frac{1}{m}\sum_{j=1}^n\langle \tau_j,x_j \rangle.$ In particular, if $\langle x_j,\tau_j\rangle=1, \forall j=1,...,n,$ then $b=n/m.$ Further, 
$$ h=\frac{1}{b}\sum_{j=1}^n\langle h,x_j \rangle \tau_j=\frac{1}{b}\sum_{j=1}^n\langle h,\tau_j \rangle x_j, ~\forall h \in \mathcal{H} ;~ \|h\|^2=\frac{1}{b}\sum_{j=1}^n\langle h,x_j \rangle \langle \tau_j, h \rangle=\frac{1}{b}\sum_{j=1}^n\langle h,\tau_j \rangle \langle x_j, h \rangle,~\forall h \in \mathcal{H}.$$
\item If the frame is tight, then 
\begin{align*}
\text{(Extended variation formula)}\quad  &\sum_{j=1}^n\sum_{k=1}^n\langle \tau_j,x_k\rangle \langle \tau_k,x_j\rangle =\frac{1}{\dim\mathcal{H}}\left(\sum_{j=1}^n\langle x_j,\tau_j \rangle\right)^2\\
&=\frac{1}{\dim\mathcal{H}}\left(\sum_{j=1}^n\langle \tau_j,x_j \rangle\right)^2 =\sum_{j=1}^n\sum_{k=1}^n\langle \tau_j,\tau_k\rangle \langle x_k,x_j\rangle . 
\end{align*}
\item If the frame is tight with the optimal bound $b$ and is such that $\langle x_1,\tau_1\rangle=\cdots=\langle x_n,\tau_n\rangle$, and $\langle \tau_j,x_k\rangle \langle \tau_k,x_j\rangle $ is constant for all $j\neq k$, $1\leq j,k\leq n$, then 
$$\langle x_j,\tau_j\rangle=\frac{b}{n}\dim\mathcal{H}, ~\forall j, \quad \langle \tau_j,x_k\rangle \langle \tau_k,x_j\rangle= \left(\frac{b}{n}\right)^2\frac{n-\dim\mathcal{H}}{n-1}\dim\mathcal{H}, ~ \forall j, k,  j \neq k .$$
\item If the frame is Parseval, then 
$$\text{(Extended dimension formula)}\quad \quad\dim\mathcal{H}=\sum_{j=1}^n\langle x_j,\tau_j \rangle =\sum_{j=1}^n\langle \tau_j, x_j \rangle. $$
\item If the frame is Parseval, then for every $T \in \mathcal{B}(\mathcal{H}),$
$$ \text{(Extended trace formula)}\quad \quad\operatorname{Trace}(T)=\sum_{j=1}^n\langle Tx_j,\tau_j \rangle =\sum_{j=1}^n\langle T\tau_j, x_j \rangle.  $$
\end{enumerate}
\end{theorem}
\begin{proof}
\begin{enumerate}[\upshape(i)]
\item Using spectral theorem,  $\mathcal{H}$ has an orthonormal basis $ \{e_j\}_{j=1}^m$ consisting of eigenvectors for $S_{x, \tau}.$ Let $\{\lambda_j\}_{j=1}^m $ denote the corresponding eigenvalues. Then $S_{x, \tau}h=\sum_{j=1}^m\langle h, e_j \rangle S_{x, \tau}e_j=\sum_{j=1}^m\lambda_j\langle h, e_j \rangle e_j, \forall h \in \mathcal{H}.$ Since  $S_{x, \tau}$ is positive invertible, $ \lambda_j>0, \forall j =1, ..., n.$ Therefore 
\begin{align*}
 \min\{\lambda_j\}_{j=1}^m \|h\|^2 \leq \sum_{j=1}^m\lambda_j|\langle h, e_j \rangle|^2 =\langle S_{x, \tau}h,h \rangle = \sum_{j=1}^n\langle h, x_j \rangle \langle \tau_j, h \rangle\leq \max\{\lambda_j\}_{j=1}^m\|h\|^2 ,\forall h \in \mathcal{H}.
\end{align*}
To get optimal frame bounds we take eigenvectors corresponding to $\min\{\lambda_j\}_{j=1}^m $ and $\max\{\lambda_j\}_{j=1}^m.$
\item $ \sum_{j=1}^m\lambda_j= \sum_{j=1}^m\lambda_j\|e_j\|^2 =\sum_{j=1}^m\langle S_{x, \tau}e_j,e_j \rangle =\sum_{j=1}^m \sum_{k=1}^n\langle e_j, x_k \rangle \langle \tau_k, e_j \rangle 
=\sum_{k=1}^n\sum_{j=1}^m\langle e_j, x_k \rangle \langle \tau_k, e_j \rangle =\sum_{k=1}^n\langle x_k,\tau_k \rangle .$ Since $S_{x, \tau}=S_{\tau, x} $, we get  $\sum_{j=1}^m\lambda_j=\sum_{j=1}^n\langle \tau_j, x_j \rangle.$
\item Let  $ \{e_j\}_{j=1}^m$ and $\{\lambda_j\}_{j=1}^m $ be as in (i). We may assume $ \lambda_1\geq \cdots \geq \lambda_m$. Define $ y_j\coloneqq\sqrt{\lambda_1-\lambda_j}e_j, j=2,...,m$. To show $(\{x_j\}_{j=1}^n \cup\{y_j\}_{j=2}^{m} , \{\tau_j\}_{j=1}^n\cup\{y_j\}_{j=2}^{m})$ is a tight frame for $ \mathcal{H}$: $ \sum_{j=1}^n\langle h, x_j \rangle\langle \tau_j, h\rangle+\sum_{j=2}^m\langle h, y_j \rangle\langle y_j, h \rangle=\langle S_{x,\tau}h,h \rangle+\sum_{j=2}^m (\lambda_1-\lambda_j)|\langle h, e_j \rangle|^2=\sum_{j=1}^m\lambda_j|\langle h, e_j \rangle|^2+\sum_{j=2}^m (\lambda_1-\lambda_j)|\langle h, e_j \rangle|^2=\lambda_1|\langle h, e_1 \rangle|^2+\lambda_1\sum_{j=2}^m |\langle h, e_j \rangle|^2= \lambda_1\sum_{j=1}^m |\langle h, e_j \rangle|^2=\lambda_1\|h\|^2,\forall h \in \mathcal{H}.$
\item This follows from (i).
\item As in (iii), we may assume $ \lambda_1\geq \cdots \geq \lambda_m$. Then (i) gives $b=\lambda_1.$ Now  use  (ii): $ b=\lambda_1 \leq \sum_{j=1}^m\lambda_j=\sum_{j=1}^n\langle x_j,\tau_j \rangle =\sum_{j=1}^n\langle \tau_j, x_j \rangle \leq bm=\lambda_1m .$
\item Let $ \{f_j\}_{j=1}^m$ be an orthonormal basis for $\mathcal{H}$. Then  $$\operatorname{Trace}(S_{x,\tau})=\sum_{k=1}^m\langle S_{x,\tau}f_k,f_k \rangle=\sum_{k=1}^m\left\langle\sum_{j=1}^n\langle f_k, x_j\rangle \tau_j ,f_k \right\rangle=\sum_{j=1}^n\sum_{k=1}^m\langle f_k, x_j\rangle \langle \tau_j ,f_k \rangle =\sum_{j=1}^n\langle \tau_j, x_j\rangle,  $$
and 
\begin{align*}
\operatorname{Trace}(S^2_{x,\tau})&=\sum_{l=1}^m\langle S_{x,\tau}f_l,S_{x,\tau}f_l \rangle=\sum_{l=1}^m\left\langle \sum_{j=1}^n\langle f_l, x_j\rangle \tau_j, \sum_{k=1}^n\langle f_l, \tau_k\rangle  x_k\right \rangle\\ &=\sum_{j=1}^n\sum_{k=1}^n\langle \tau_j,x_k \rangle \sum_{l=1}^m\langle f_l, x_j\rangle\langle \tau_k,f_l\rangle=\sum_{j=1}^n\sum_{k=1}^n\langle \tau_j,x_k \rangle \langle \tau_k,x_j \rangle,
\end{align*} 
\begin{align*}
\operatorname{Trace}(S^2_{x,\tau})&=\sum_{l=1}^m\langle S_{x,\tau}f_l,S_{x,\tau}f_l \rangle=\sum_{l=1}^m\left\langle \sum_{j=1}^n\langle f_l, x_j\rangle \tau_j, \sum_{k=1}^n\langle f_l, x_k\rangle  \tau_k\right \rangle\\ &=\sum_{j=1}^n\sum_{k=1}^n\langle \tau_j,\tau_k \rangle \sum_{l=1}^m\langle f_l, x_j\rangle\langle x_k,f_l\rangle=\sum_{j=1}^n\sum_{k=1}^n\langle \tau_j,\tau_k \rangle \langle x_k,x_j \rangle.
\end{align*} 
\item Now $S_{x, \tau}=\lambda I_{\mathcal{H}}, $ for some  positive $ \lambda.$ This gives $ \lambda_1=\cdots=\lambda_m=\lambda=b.$ From (ii) we get the conclusions.
\item Let the  optimal frame bound be $b$. From (vi) and (vii), 
\begin{align*}
\sum_{j=1}^n\sum_{k=1}^n\langle \tau_j,\tau_k\rangle \langle x_k,x_j\rangle&=\sum_{j=1}^n\sum_{k=1}^n\langle \tau_j,x_k\rangle \langle \tau_k,x_j\rangle= \operatorname{Trace}(S^2_{x,\tau})\\
&= \operatorname{Trace}(b^2I_\mathcal{H})=b^2m=\left(\frac{1}{m}\sum_{j=1}^n\langle x_j, \tau_j \rangle\right)^2m .
\end{align*}
\item From (vii), $b=n\langle x_j,\tau_j \rangle/{\dim\mathcal{H}}, \forall j$. Using this in (viii)
\begin{align*}
b^2\dim\mathcal{H}&=\sum_{j=1}^n\sum_{k=1}^n\langle \tau_j,x_k\rangle \langle \tau_k,x_j\rangle=\sum_{j=1}^n\langle \tau_j,x_j\rangle \langle \tau_j,x_j\rangle+\sum_{j,k=1,j\neq k}^n\langle \tau_j,x_k\rangle \langle \tau_k,x_j\rangle\\
&=\frac{b^2(\dim\mathcal{H})^2}{n^2}n+\langle \tau_j,x_k\rangle \langle \tau_k,x_j\rangle(n^2-n),\forall j \neq k, \text{which gives the answer}.
\end{align*}
\item $\dim\mathcal{H}=\sum_{j=1}^m\|e_j\|^2 =\sum_{j=1}^m\sum_{k=1}^n\langle e_j,  x_k \rangle\langle  \tau_k, e_j \rangle =\sum_{j=1}^m\sum_{k=1}^n\langle e_j,  \tau_k \rangle\langle  x_k, e_j \rangle =\sum_{k=1}^n\sum_{j=1}^m\langle e_j,  x_k \rangle\langle  \tau_k, e_j \rangle=\sum_{k=1}^n\sum_{j=1}^m\langle e_j,  \tau_k \rangle\langle  x_k, e_j \rangle=\sum_{k=1}^n\langle x_k,\tau_k \rangle =\sum_{k=1}^n\langle \tau_k, x_k \rangle.$
\item Let $ \{f_j\}_{j=1}^m$ be an orthonormal basis for $\mathcal{H}$. Then  $\operatorname{Trace}(T)=\sum_{j=1}^m\langle Tf_j,f_j \rangle = \sum_{j=1}^m\langle \sum_{k=1}^n\langle Tf_j,x_k \rangle \tau_k,f_j \rangle=\sum_{k=1}^n\sum_{j=1}^m\langle \tau_k,f_j \rangle\langle Tf_j,x_k \rangle=\sum_{k=1}^n\sum_{j=1}^m\langle \tau_k,f_j \rangle\langle f_j,T^*x_k \rangle=\sum_{k=1}^n\langle \tau_k,T^*x_k \rangle=\sum_{k=1}^n\langle T\tau_k,x_k \rangle$. Similarly by using $Tf_j=\sum_{k=1}^n\langle Tf_j,\tau_k \rangle x_k $, we get $\operatorname{Trace}(T) =\sum_{k=1}^n\langle Tx_k, \tau_k \rangle. $
\end{enumerate}	
\end{proof}
\begin{remark}
Trace formula even holds in infinite dimensions, namely - 
\begin{proposition}
Let $(\{x_j\}_{j\in \mathbb{J}} , \{\tau_j\}_{j\in \mathbb{J}})$ be a  frame for $\mathcal{H}$. Then 
\begin{enumerate}[\upshape(i)]
\item $\operatorname{Trace}(S_{x,\tau})=\sum_{j\in \mathbb{J}}\langle x_j,\tau_j \rangle=\sum_{j\in \mathbb{J}}\langle \tau_j,x_j \rangle$;\\
 $ \operatorname{Trace}(S^2_{x,\tau}) =\sum_{j\in\mathbb{J}}\sum_{k\in \mathbb{J}}\langle \tau_j,x_k\rangle \langle \tau_k,x_j\rangle=\sum_{j\in \mathbb{J}}\sum_{k\in \mathbb{J}}\langle \tau_j,\tau_k\rangle \langle x_k,x_j\rangle.$
\item  If the frame is Parseval, and   $ T \in \mathcal{B}(\mathcal{H}),$ then 
$\operatorname{Trace}(T)=\sum_{j\in\mathbb{J}}\langle Tx_j,\tau_j \rangle =\sum_{j\in\mathbb{J}}\langle T\tau_j, x_j \rangle. $
\end{enumerate}
\end{proposition}
\end{remark}
\begin{proposition}
Let $G$ be a finite group, $(x,\tau)$ be a frame generator for a finite  dimensional complex Hilbert space $ \mathcal{H}$ whose generated frame has bounds $a$ and $b$.  Then  
$$ a \leq \frac{\operatorname{order}(G)}{\operatorname{dim}\mathcal{H}}\langle x,\tau \rangle=\frac{\operatorname{order}(G)}{\operatorname{dim}\mathcal{H}}\langle \tau, x \rangle \leq b.$$
In particular, if the frame is tight with the optimal bound $c$, then $\langle x,\tau \rangle=\langle\tau,x \rangle=c\operatorname{dim}\mathcal{H}/\operatorname{order}(G). $
\end{proposition}
\begin{proof}
Let $\operatorname{dim}\mathcal{H}=m$, $\{\lambda_j\}_{j=1}^m $ be eigenvalues for $S_{x,\tau}$. Define $T_g:\mathcal{H} \ni h \mapsto \langle h,\pi_g\tau \rangle \pi_gx \in \mathcal{H}, \forall g \in G $. Then $\operatorname{Trace}(T_g) = \langle \pi_gx,\pi_g\tau \rangle=\langle x,\tau \rangle, \forall g \in G$, and $S_{x,\tau}=\sum_{g\in G}T_g $. From Theorem \ref{FINITEDIMENSIONALPASTINGTHEOREM}, $a\leq \min\{\lambda_j\}_{j=1}^m\leq \max\{\lambda_j\}_{j=1}^m\leq b$. This gives $am\leq \sum_{j=1}^m\lambda_j=\operatorname{Trace}(S_{x,\tau})=\sum_{g\in G}\operatorname{Trace}(T_g)=\sum_{g\in G}\langle x,\tau \rangle\leq bm.$ 
\end{proof}
\begin{theorem}\label{REALTOCOMPLEX}
If a frame $(\{x_j\}_{j=1}^n , \{\tau_j\}_{j=1}^n)$ for $\mathbb{R}^m $ is such that 
$$ \sum_{j=1}^n\langle h, x_j \rangle \langle \tau_j, g \rangle =\sum_{j=1}^n\langle g, x_j \rangle \langle \tau_j, h \rangle, ~\forall h, g \in \mathbb{R}^n,$$
then it is  also a frame for $\mathbb{C}^m. $ Further, if $(\{x_j\}_{j=1}^n , \{\tau_j\}_{j=1}^n) $ is tight (resp. Parseval) frame for $\mathbb{R}^m $, then it is also  a tight (resp. Parseval) frame for $\mathbb{C}^m. $
\end{theorem}
\begin{proof}
Let $a, b$  be lower and upper frame bounds, in order. For $ z\in \mathbb{C}^m$ we write $ z=\operatorname{Re}z+i\operatorname{Im}z,\operatorname{Re}z,\operatorname{Im}z \in \mathbb{R}^m.$ Then
\begin{align*}
a\|z\|^2&=a\|\operatorname{Re}z\|^2+a\|\operatorname{Im}z\|^2\leq a\sum\limits_{j=1}^n\langle \operatorname{Re}z, x_j\rangle\langle \tau_j, \operatorname{Re}z\rangle+a\sum\limits_{j=1}^n\langle \operatorname{Im}z, x_j\rangle\langle \tau_j, \operatorname{Im}z\rangle \\
&=a\left(\sum\limits_{j=1}^n\langle \operatorname{Re}z, x_j\rangle\langle \tau_j, \operatorname{Re}z\rangle+i\sum\limits_{j=1}^n\langle \operatorname{Im}z, x_j\rangle\langle \tau_j, \operatorname{Re}z\rangle\right) \\
&\quad-ai\left( \sum\limits_{j=1}^n\langle \operatorname{Re}z, x_j\rangle\langle \tau_j, \operatorname{Im}z\rangle +i\sum\limits_{j=1}^n\langle \operatorname{Im}z, x_j\rangle\langle \tau_j, \operatorname{Im}z\rangle\right)\\
&=a\sum\limits_{j=1}^n\langle \operatorname{Re}z+i\operatorname{Im}z, x_j\rangle\langle \tau_j, \operatorname{Re}z\rangle-ai \sum\limits_{j=1}^n\langle \operatorname{Re}z+i\operatorname{Im}z, x_j\rangle\langle \tau_j, \operatorname{Im}z\rangle\\
&=a\left(\sum\limits_{j=1}^n\langle z,x_j\rangle \langle \tau_j, \operatorname{Re}z\rangle +\sum\limits_{j=1}^n\langle z,x_j\rangle \langle \tau_j, i\operatorname{Im}z\rangle\right)=a\left(\sum\limits_{j=1}^n\langle z,x_j\rangle \langle \tau_j, z\rangle \right)\\
&\leq b\|\operatorname{Re}z\|^2+b\|\operatorname{Im}z\|^2=b\|z\|^2.
\end{align*}
Whenever frame is  tight or Parseval for $ \mathbb{R}^m$, we continuously get  equalities  in  the last chain.
\end{proof}
\begin{theorem}\label{COMPLEXTOREAL}
If  $(\{x_j\}_{j=1}^n , \{\tau_j\}_{j=1}^n)$ is a frame  for  $\mathbb{C}^m $  such that 
$$ \sum_{j=1}^n\langle h, \operatorname{Re}x_j \rangle \langle \operatorname{Im}\tau_j, h \rangle =\sum_{j=1}^n\langle h, \operatorname{Im}x_j \rangle \langle \operatorname{Re}\tau_j, h \rangle, ~\forall h \in \mathbb{C}^m,$$  
then  $(\{\operatorname{Re}x_j\}_{j=1}^n\cup\{\operatorname{Im}x_j\}_{j=1}^n , \{\operatorname{Re}\tau_j\}_{j=1}^n\cup\{\operatorname{Im}\tau_j\}_{j=1}^n)$ is   a frame for $\mathbb{R}^m. $ Further, if $(\{x_j\}_{j=1}^n , \{\tau_j\}_{j=1}^n) $ is tight (resp. Parseval)  for $\mathbb{C}^m $, then  $(\{\operatorname{Re}x_j\}_{j=1}^n\cup\{\operatorname{Im}x_j\}_{j=1}^n , \{\operatorname{Re}\tau_j\}_{j=1}^n\cup\{\operatorname{Im}\tau_j\}_{j=1}^n)$ is  tight (resp. Parseval) for $\mathbb{R}^m. $
\end{theorem}
\begin{proof}
Taking $a, b$  as  lower and upper frame bounds, respectively; since $\mathbb{R}^m $	is inside $\mathbb{C}^m, $ we get $ a\|h\|^2\leq \sum_{j=1}^n\langle h,x_j\rangle \langle \tau_j, h\rangle\leq b\|h\|^2, \forall h \in \mathbb{R}^m.$ We find 

\begin{align*}
&\sum_{j=1}^n\langle h,x_j\rangle \langle \tau_j, h\rangle=\sum_{j=1}^n\langle h,\operatorname{Re}x_j+i\operatorname{Im}x_j\rangle \langle \operatorname{Re}\tau_j+i\operatorname{Im}\tau_j, h\rangle\\
&=\sum_{j=1}^n\langle h,\operatorname{Re}x_j\rangle \langle \operatorname{Re}\tau_j, h\rangle+i\sum_{j=1}^n\langle h,\operatorname{Re}x_j\rangle \langle \operatorname{Im}\tau_j, h\rangle
-i\sum_{j=1}^n\langle h,\operatorname{Im}x_j\rangle \langle \operatorname{Re}\tau_j, h\rangle+\sum_{j=1}^n\langle h,\operatorname{Im}x_j\rangle \langle \operatorname{Im}\tau_j, h\rangle\\
&=\sum_{j=1}^n\langle h,\operatorname{Re}x_j\rangle \langle \operatorname{Re}\tau_j, h\rangle+\sum_{j=1}^n\langle h,\operatorname{Im}x_j\rangle \langle \operatorname{Im}\tau_j, h\rangle, ~\forall h \in \mathbb{R}^m.
\end{align*}
Therefore $ a\|h\|^2\leq \sum_{j=1}^n\langle h,\operatorname{Re}x_j\rangle \langle \operatorname{Re}\tau_j, h\rangle+\sum_{j=1}^n\langle h,\operatorname{Im}x_j\rangle \langle \operatorname{Im}\tau_j, h\rangle\leq b\|h\|^2, \forall h \in \mathbb{R}^m.$
Hence $(\{\operatorname{Re}x_j\}_{j=1}^n\cup\{\operatorname{Im}x_j\}_{j=1}^n , \{\operatorname{Re}\tau_j\}_{j=1}^n\cup\{\operatorname{Im}\tau_j\}_{j=1}^n)$ is a frame for $\mathbb{R}^m$. Other facts are clear.
\end{proof}
\begin{remark}
Number of elements in the resulting frame in Theorem \ref{COMPLEXTOREAL} is twice the number of elements in the original frame.
\end{remark}
\begin{proposition}
Let $(\{x_j\}_{j\in \mathbb{J}} , \{\tau_j\}_{j\in \mathbb{J}})$ be a  frame for $\mathcal{H}.$  Then $\mathcal{H} $ is finite dimensional if and only if 
$$ \sum_{j\in\mathbb{J}}\langle x_j,\tau_j \rangle =\sum_{j\in\mathbb{J}}\langle \tau_j, x_j \rangle<\infty.$$
In particular, if dimension of $\mathcal{H} $ is finite, $\langle x_j,\tau_j \rangle>0,\forall j\in \mathbb{J} $ and $ \inf\{\langle x_j,\tau_j \rangle\}_{j\in \mathbb{J}}$ is  positive, then $ \operatorname{Card}(\mathbb{J})<\infty.$
\end{proposition}
\begin{proof}
Let $\{e_k\}_{k\in \mathbb{L}}$ be an orthonormal basis for $ \mathcal{H}$, and  $a, b$ be frame bounds for $(\{x_j\}_{j\in \mathbb{J}} , \{\tau_j\}_{j\in \mathbb{J}})$.
$(\Rightarrow)$ Now $ \mathbb{L}$  is finite. Then $ \sum_{j\in\mathbb{J}}\langle x_j,\tau_j \rangle =\sum_{j\in\mathbb{J}}\sum_{k\in\mathbb{L}}\langle x_j,e_k \rangle\langle e_k,\tau_j \rangle=\sum_{k\in\mathbb{L}}\sum_{j\in\mathbb{J}}\langle x_j,e_k \rangle\langle e_k,\tau_j \rangle \leq b\sum_{k\in\mathbb{L}}\|e_k\|^2=b\operatorname{Card}(\mathbb{L})<\infty.$ Since $(\{x_j\}_{j\in \mathbb{J}} , \{\tau_j\}_{j\in \mathbb{J}})$ is a frame, we must have $\sum_{j\in\mathbb{J}}\langle x_j,e_k \rangle\langle e_k,\tau_j \rangle=\overline{\sum_{j\in\mathbb{J}}\langle x_j,e_k \rangle\langle e_k,\tau_j \rangle} $. Therefore $ \sum_{j\in\mathbb{J}}\langle x_j,\tau_j \rangle =\sum_{j\in\mathbb{J}}\langle \tau_j, x_j \rangle<\infty.$

 $(\Leftarrow)$	$ \operatorname{dim}\mathcal{H}=\sum_{k\in\mathbb{L}}\|e_k\|^2\leq (1/a)\sum_{k\in\mathbb{L}}\sum_{j\in\mathbb{J}}\langle e_k,x_j \rangle\langle \tau_j, e_k \rangle=(1/a)\sum_{j\in\mathbb{J}}\sum_{k\in\mathbb{L}}\langle e_k, x_j\rangle\langle \tau_j, e_k \rangle=(1/a)\sum_{j\in\mathbb{J}}\langle x_j,\tau_j\rangle<\infty.$
 
 For the second, $ 0<\inf\{\langle x_j,\tau_j \rangle\}_{j\in \mathbb{J}}\operatorname{Card}(\mathbb{J})\leq\sum_{j\in\mathbb{J}}\langle x_j,\tau_j \rangle<\infty. $
\end{proof}
\begin{proposition}\label{INTERESTINGEXAMPLEPROPOSITION}
Let  $ a_j, b_j \geq 0,\theta_j, \phi_j \in \mathbb{R}, j=1,...,n$. Then 	
\begin{enumerate}[\upshape(i)]
\item $\left\{ x_j\coloneqq\begin{bmatrix}
a_j\cos\theta_j \\
a_j\sin\theta_j
\end{bmatrix}\right\}_{j=1}^n$ is a tight frame for $ \mathbb{R}^2$ w.r.t.  $\left\{ \tau_j\coloneqq \begin{bmatrix}
b_j\cos\phi_j \\
b_j\sin\phi_j
\end{bmatrix}\right\}_{j=1}^n$ if and only if 
$$ \sum\limits_{j=1}^{n}\begin{bmatrix}
a_jb_j\cos(\theta_j+\phi_j) \\
a_jb_j\sin(\theta_j+\phi_j)\\
a_jb_j\sin(\theta_j-\phi_j)
\end{bmatrix}=\begin{bmatrix}
0 \\
0\\
0
\end{bmatrix}.$$
\item For all $ k,l $ with $kl\geq 3$, $\left\{ \begin{bmatrix}
\cos\left(\frac{2\pi j}{k}\right) \\
\sin\left(\frac{2\pi j}{k}\right)
\end{bmatrix}\right\}_{j=0}^{kl-1}$ is a tight frame for $ \mathbb{R}^2$ w.r.t.  $\left\{ \begin{bmatrix}
\cos\left(\frac{2\pi j}{l}\right) \\
\sin\left(\frac{2\pi j}{l}\right)
\end{bmatrix}\right\}_{j=0}^{kl-1}$.
\end{enumerate}	
\end{proposition}
\begin{proof}
\begin{enumerate}[\upshape(i)]
\item A direct computation shows that the matrix of the frame operator $ S_{x,\tau}$ for the given  collection is $$\begin{bmatrix}
\sum\limits_{j=1}^{n} a_jb_j\cos\theta_j\cos\phi_j& \sum\limits_{j=1}^{n} a_jb_j\sin\theta_j\cos\phi_j \\
\sum\limits_{j=1}^{n} a_jb_j\cos\theta_j\sin\phi_j &\sum\limits_{j=1}^{n} a_jb_j\sin\theta_j\sin\phi_j	\\
\end{bmatrix}.$$
Another computation shows (using trigonometric formulas for compound angles) that the frame is tight, i.e., $ S_{x,\tau}=aI_{\mathbb{R}^2}$ for some $ a>0$ if and only if $$ \sum\limits_{j=1}^{n}\begin{bmatrix}
a_jb_j\cos(\theta_j+\phi_j) \\
a_jb_j\sin(\theta_j+\phi_j)\\
a_jb_j\sin(\theta_j-\phi_j)
\end{bmatrix}=\begin{bmatrix}
0 \\
0\\
0
\end{bmatrix}.$$
\item We take the advantage of `if' in (i). We even  show $$ \sum\limits_{j=0}^{kl-1}\begin{bmatrix}
\cos\left(\frac{2\pi j}{k}+\frac{2\pi j}{l}\right) \\
\sin\left(\frac{2\pi j}{k}+\frac{2\pi j}{l}\right)\\
\cos\left(\frac{2\pi j}{k}-\frac{2\pi j}{l}\right)\\
\sin\left(\frac{2\pi j}{k}-\frac{2\pi j}{l}\right)
\end{bmatrix}=\begin{bmatrix}
0 \\
0\\
0\\
0
\end{bmatrix} ~ \text{for}~ k\neq l.$$ We get this from:   $\sum_{j=0}^{kl-1} \left(\cos\left(\frac{2\pi j}{k}+\frac{2\pi j}{l}\right) +i
\sin\left(\frac{2\pi j}{k}+\frac{2\pi j}{l}\right)\right)=\sum_{j=0}^{kl-1}e^{\left(\frac{2\pi j}{k}+\frac{2\pi j}{l}\right)i}=\sum_{j=0}^{kl-1}e^{\left(\frac{2\pi i}{k}+\frac{2\pi i}{l}\right)j}=0$, and $\sum_{j=0}^{kl-1} \left(\cos\left(\frac{2\pi j}{k}-\frac{2\pi j}{l}\right) +i
\sin\left(\frac{2\pi j}{k}-\frac{2\pi j}{l}\right)\right)=\sum_{j=0}^{kl-1}e^{\left(\frac{2\pi j}{k}-\frac{2\pi j}{l}\right)i}=\sum_{j=0}^{kl-1}e^{\left(\frac{2\pi i}{k}-\frac{2\pi i}{l}\right)j}=0 $. For the case $ k=l$, clearly  
$$ \sum\limits_{j=0}^{k^2-1}\begin{bmatrix}
\cos\left(\frac{4\pi j}{k}\right) \\
\sin\left(\frac{4\pi j}{k}\right)\\
\sin0
\end{bmatrix}=\begin{bmatrix}
0 \\
0\\
0
\end{bmatrix}.$$
\end{enumerate}
\end{proof}

\section{Further extension}\label{FURTHEREXTENSION}
\begin{definition}
 A collection $ \{A_j\}_{j \in \mathbb{J}} $  in $ \mathcal{B}(\mathcal{H}, \mathcal{H}_0)$ is said to be a weak \textit{operator-valued frame} (we write weak (ovf)) in $ \mathcal{B}(\mathcal{H}, \mathcal{H}_0) $  with respect to  collection  $ \{\Psi_j\}_{j \in \mathbb{J}}  $ in $ \mathcal{B}(\mathcal{H}, \mathcal{H}_0) $ if  the series $ S_{A, \Psi} \coloneqq  \sum_{j\in \mathbb{J}} \Psi_j^*A_j$  converges in the strong-operator topology on $ \mathcal{B}(\mathcal{H})$ to a  bounded positive invertible operator.

 Notions of frame bounds, optimal bounds, tight frame, Parseval frame, Bessel are in same fashion  as in Definition \ref{1}.
  
 For fixed $ \mathbb{J}$, $\mathcal{H}, \mathcal{H}_0, $ and $ \{\Psi_j \}_{j \in \mathbb{J}}$ the set of all weak operator-valued frames in $ \mathcal{B}(\mathcal{H}, \mathcal{H}_0) $ with respect to collection  $ \{\Psi_j \}_{j \in \mathbb{J}}$ is denoted by $ \mathscr{F}^\text{w}_\Psi.$
  \end{definition}
 Above definition is equivalent to 
 \begin{definition}
 A collection $ \{A_j\}_{j \in \mathbb{J}} $   in $ \mathcal{B}(\mathcal{H}, \mathcal{H}_0)$ is said to be a weak (ovf)  w.r.t. $ \{\Psi_j\}_{j \in \mathbb{J}}  $ in $ \mathcal{B}(\mathcal{H}, \mathcal{H}_0) $ if there exist $ a, b, r >0$ such that 
 \begin{enumerate}[\upshape(i)]
 \item $\|\sum_{j\in \mathbb{J}}\Psi_j^*A_jh \|\leq r\|h\|, \forall h \in \mathcal{H},$
 \item $\sum_{j\in \mathbb{J}}\Psi_j^*A_jh =\sum_{j\in \mathbb{J}}A_j^*\Psi_jh, \forall h \in \mathcal{H},$
  \item $a\|h\|^2\leq\sum_{j\in \mathbb{J}}\langle A_jh, \Psi_jh\rangle  \leq b\|h\|^2, \forall h \in \mathcal{H}.$
\end{enumerate}
If the Hilbert space is complex, then  condition \text{\upshape(ii)} can be dropped.
 \end{definition}
 Conditions  (i)  and (ii) in the previous definition say  $ S_{A,\Psi}$ exists, bounded and self adjoint and (i) says it is positive invertible.

Now we try to get concepts and results devoloped earlier it this further extension. Care should be taken that ``statements and proofs  should not use analysis as well as synthesis operators''.
  
Results we derived earlier, for the extension, if it does not use analysis operators, we state them here, for others we give the proof.
\begin{proposition}
Let $(\{A_j\}_{j\in \mathbb{J}}, \{\Psi_j\}_{j\in \mathbb{J}}) $ be  a weak (ovf)   in $ \mathcal{B}(\mathcal{H}, \mathcal{H}_0)$  with upper frame bound $b$. If $\Psi_j^*A_j\geq 0, \forall j \in \mathbb{J} ,$ then $ \|\Psi_j^*A_j\|\leq b, \forall j \in \mathbb{J}.$
\end{proposition}

\begin{proposition}
If  $(\{A_j\}_{j\in \mathbb{J}}, \{\Psi_j\}_{j\in \mathbb{J}}) $ is  a weak Bessel   in $ \mathcal{B}(\mathcal{H}, \mathcal{H}_0)$, then there exists a $ B \in \mathcal{B}(\mathcal{H}, \mathcal{H}_0)$ such that $(\{A_j\}_{j\in \mathbb{J}}\cup\{B\}, \{\Psi_j\}_{j\in \mathbb{J}}\cup\{B\}) $ is a tight weak (ovf). In particular,  if $(\{A_j\}_{j\in \mathbb{J}}, \{\Psi_j\}_{j\in \mathbb{J}}) $ is   weak (ovf)   in $ \mathcal{B}(\mathcal{H}, \mathcal{H}_0)$, then there exists a $ B \in \mathcal{B}(\mathcal{H}, \mathcal{H}_0)$ such that $(\{A_j\}_{j\in \mathbb{J}}\cup\{B\}, \{\Psi_j\}_{j\in \mathbb{J}}\cup\{B\}) $ is a tight weak (ovf). 
\end{proposition}

\begin{definition}
A weak (ovf)  $(\{B_j\}_{j\in \mathbb{J}}, \{\Phi_j\}_{j\in \mathbb{J}})$  in $\mathcal{B}(\mathcal{H}, \mathcal{H}_0)$ is said to be dual of a weak (ovf) $ ( \{A_j\}_{j\in \mathbb{J}}, \{\Psi_j\}_{j\in \mathbb{J}})$ in $\mathcal{B}(\mathcal{H}, \mathcal{H}_0)$  if  $ \sum_{j \in \mathbb{J}}B_j^*\Psi_j= \sum_{j \in \mathbb{J}}\Phi^*_jA_j=I_{\mathcal{H}}$. The `weak (ovf)' $( \{\widetilde{A}_j\coloneqq A_jS_{A,\Psi}^{-1}\}_{j\in \mathbb{J}},\{\widetilde{\Psi}_j\coloneqq\Psi_jS_{A,\Psi}^{-1}\}_{j \in \mathbb{J}})$, which is a `dual' of $ (\{A_j\}_{j\in \mathbb{J}}, \{\Psi_j\}_{j\in \mathbb{J}})$ is called the canonical dual of $ (\{A_j\}_{j\in \mathbb{J}}, \{\Psi_j\}_{j\in \mathbb{J}})$. 
\end{definition}  
\begin{proposition}
Let $( \{A_j\}_{j\in \mathbb{J}}, \{\Psi_j\}_{j\in \mathbb{J}} )$ be a weak (ovf) in $ \mathcal{B}(\mathcal{H}, \mathcal{H}_0).$  If $ h \in \mathcal{H}$ has representation  $ h=\sum_{j\in\mathbb{J}}A_j^*y_j= \sum_{j\in\mathbb{J}}\Psi_j^*z_j, $ for some $ \{y_j\}_{j\in \mathbb{J}},\{z_j\}_{j\in \mathbb{J}}$ in $ \mathcal{H}_0$, then 
$$ \sum\limits_{j\in \mathbb{J}}\langle y_j,z_j\rangle =\sum\limits_{j\in \mathbb{J}}\langle \widetilde{\Psi}_jh,\widetilde{A}_jh\rangle+\sum\limits_{j\in \mathbb{J}}\langle y_j-\widetilde{\Psi}_jh,z_j-\widetilde{A}_jh\rangle. $$
\end{proposition}  
  
\begin{theorem}\label{CANONICALDUALFRAMEPROPERTYOPERATORVERSIONWEAK}
Let $ (\{A_j\}_{j\in \mathbb{J}},\{\Psi_j\}_{j\in \mathbb{J}}) $ be a weak (ovf) with frame bounds $ a$ and $ b.$ Then
\begin{enumerate}[\upshape(i)]
\item The canonical dual weak (ovf) of the canonical dual weak (ovf)  of $ (\{A_j\}_{j\in \mathbb{J}} ,\{\Psi_j\}_{j\in \mathbb{J}} )$ is itself.
\item$ \frac{1}{b}, \frac{1}{a}$ are frame bounds for the canonical dual of $ (\{A_j\}_{j\in \mathbb{J}},\{\Psi_j\}_{j\in \mathbb{J}}).$
\item If $ a, b $ are optimal frame bounds for $( \{A_j\}_{j\in \mathbb{J}} , \{\Psi_j\}_{j\in \mathbb{J}}),$ then $ \frac{1}{b}, \frac{1}{a}$ are optimal  frame bounds for its canonical dual.
\end{enumerate} 
\end{theorem} 
\begin{proof}
Proof of Theorem \ref{CANONICALDUALFRAMEPROPERTYOPERATORVERSION} carries, since we have not used analysis operators there.
\end{proof} 
\begin{definition}
A weak (ovf)  $(\{B_j\}_{j\in \mathbb{J}},  \{\Phi_j\}_{j\in \mathbb{J}})$  in $\mathcal{B}(\mathcal{H}, \mathcal{H}_0)$ is said to be orthogonal to a weak (ovf)  $( \{A_j\}_{j\in \mathbb{J}}, \{\Psi_j\}_{j\in \mathbb{J}})$ in $\mathcal{B}(\mathcal{H}, \mathcal{H}_0)$ if $ \sum_{j \in \mathbb{J}}B_j^*\Psi_j= \sum_{j \in \mathbb{J}}\Phi^*_jA_j=0$.
\end{definition}  
\begin{proposition}
Two orthogonal weak operator-valued frames  have common dual weak (ovf).	
\end{proposition} 
\begin{proof}
Let  $ (\{A_j\}_{j\in \mathbb{J}}, \{\Psi_j\}_{j\in \mathbb{J}}) $ and $ (\{B_j\}_{j\in \mathbb{J}}, \{\Phi_j\}_{j\in \mathbb{J}}) $ be  two orthogonal weak  operator-valued frames in $\mathcal{B}(\mathcal{H}, \mathcal{H}_0)$. Define $ C_j\coloneqq A_jS_{A,\Psi}^{-1}+B_jS_{B,\Phi}^{-1},\Xi_j\coloneqq \Psi _jS_{A,\Psi}^{-1}+\Phi_jS_{B,\Phi}^{-1}, \forall j \in \mathbb{J}$. Then  
\begin{align*}
&S_{C,\Xi}= \sum_{j\in \mathbb{J}}\Xi_j^*C_j=\sum_{j\in \mathbb{J}}(S_{A,\Psi}^{-1}\Psi_j^*+S_{B,\Phi}^{-1}\Phi_j^*)(A_jS_{A,\Psi}^{-1}+B_jS_{B,\Phi}^{-1})\\
&=S_{A,\Psi}^{-1}\left(\sum_{j\in \mathbb{J}}\Psi_j^*A_j\right)S_{A,\Psi}^{-1}+S_{A,\Psi}^{-1}\left(\sum_{j\in \mathbb{J}}\Psi_j^*B_j\right)S_{B,\Phi}^{-1}+S_{B,\Phi}^{-1}\left(\sum_{j\in \mathbb{J}}\Phi_j^*A_j\right)S_{A,\Psi}^{-1}+S_{B,\Phi}^{-1}\left(\sum_{j\in \mathbb{J}}\Phi_j^*B_j\right)S_{B,\Phi}^{-1}\\
&=S_{A,\Psi}^{-1}S_{A,\Psi}S_{A,\Psi}^{-1}+S_{A,\Psi}^{-1}0S_{B,\Phi}^{-1}+S_{B,\Phi}^{-1}0S_{A,\Psi}^{-1}+S_{B,\Phi}^{-1}S_{B,\Phi}S_{B,\Phi}^{-1} =S_{A,\Psi}^{-1}+S_{B,\Phi}^{-1}
\end{align*}
which is positive  invertible. Therefore $(\{C_j\}_{j\in \mathbb{J}}, \{\Xi_j\}_{j\in \mathbb{J}})$ is a weak (ovf) in $\mathcal{B}(\mathcal{H}, \mathcal{H}_0)$. Further,  $\sum_{j\in \mathbb{J}}\Psi_j^*C_j=\sum_{j\in \mathbb{J}}\Psi_j^*(A_jS_{A,\Psi}^{-1}+B_jS_{B,\Phi}^{-1})=I_\mathcal{H}+0S_{B,\Phi}^{-1}=I_\mathcal{H}$, $\sum_{j\in \mathbb{J}}A_j^*\Xi_j=\sum_{j\in \mathbb{J}}A_j^*(\Psi _jS_{A,\Psi}^{-1}+\Phi_jS_{B,\Phi}^{-1})=I_\mathcal{H}+0S_{B,\Phi}^{-1}=I_\mathcal{H} $, and $\sum_{j\in \mathbb{J}}\Phi_j^*C_j=\sum_{j\in \mathbb{J}}\Phi_j^*(A_jS_{A,\Psi}^{-1}+B_jS_{B,\Phi}^{-1})=0S_{A,\Psi}^{-1}+I_\mathcal{H}=I_\mathcal{H}$, $ \sum_{j\in \mathbb{J}}B_j^*\Xi_j=\sum_{j\in \mathbb{J}}B_j^*(\Psi _jS_{A,\Psi}^{-1}+\Phi_jS_{B,\Phi}^{-1})=0S_{A,\Psi}^{-1}+I_\mathcal{H}=I_\mathcal{H}$. This says $(\{C_j\}_{j\in \mathbb{J}}, \{\Xi_j\}_{j\in \mathbb{J}})$ is a common dual of  $ (\{A_j\}_{j\in \mathbb{J}}, \{\Psi_j\}_{j\in \mathbb{J}}) $ and $ (\{B_j\}_{j\in \mathbb{J}}, \{\Phi_j\}_{j\in \mathbb{J}}).$	
\end{proof}
\begin{proposition}
Let $ (\{A_j\}_{j\in \mathbb{J}}, \{\Psi_j\}_{j\in \mathbb{J}}) $ and $ (\{B_j\}_{j\in \mathbb{J}}, \{\Phi_j\}_{j\in \mathbb{J}}) $ be  two Parseval weak operator-valued frames in   $\mathcal{B}(\mathcal{H}, \mathcal{H}_0)$ which are  orthogonal. If $C,D,E,F \in \mathcal{B}(\mathcal{H})$ are such that $ C^*E+D^*F=I_\mathcal{H}$, then  $ (\{A_jC+B_jD\}_{j\in \mathbb{J}}, \{\Psi_jE+\Phi_jF\}_{j\in \mathbb{J}}) $ is a  Parseval weak (ovf) in  $\mathcal{B}(\mathcal{H}, \mathcal{H}_0)$. In particular,  if scalars $ c,d,e,f$ satisfy $\bar{c}e+\bar{d}f =1$, then $ (\{cA_j+dB_j\}_{j\in \mathbb{J}}, \{e\Psi_j+f\Phi_j\}_{j\in \mathbb{J}}) $ is   a Parseval weak (ovf).
\end{proposition} 
\begin{proof}
$ S_{AC+BD,\Psi E+\Phi F} =\sum_{j\in\mathbb{J}}(\Psi_jE+\Phi_jF)^*(A_jC+B_jD)=E^*S_{A,\Psi}C+E^*(\sum_{j\in\mathbb{J}}\Psi_j^*B_j)D+F^*(\sum_{j\in\mathbb{J}}\Phi_j^*A_j)C+F^*S_{B,\Phi}D=E^*I_\mathcal{H}C+E^*0D+F^*0C+F^*I_\mathcal{H}D=I_\mathcal{H}$.
\end{proof}
\begin{definition}
Two weak operator-valued frames $(\{A_j\}_{j\in \mathbb{J}},\{\Psi_j\}_{j\in \mathbb{J}} )$  and $ (\{B_j\}_{j\in \mathbb{J}}, \{\Phi_j\}_{j\in \mathbb{J}} )$   in $ \mathcal{B}(\mathcal{H}, \mathcal{H}_0)$  are called 
disjoint if $(\{A_j\oplus B_j\}_{j \in \mathbb{J}},\{\Psi_j\oplus \Phi_j\}_{j \in \mathbb{J}})$ is a weak (ovf) in $ \mathcal{B}(\mathcal{H}\oplus \mathcal{H}, \mathcal{H}_0).$   
\end{definition}
\begin{proposition}
If $(\{A_j\}_{j\in \mathbb{J}},\{\Psi_j\}_{j\in \mathbb{J}} )$  and $ (\{B_j\}_{j\in \mathbb{J}}, \{\Phi_j\}_{j\in \mathbb{J}} )$  are  weak orthogonal operator-valued frames  in $ \mathcal{B}(\mathcal{H}, \mathcal{H}_0)$, then  they  are disjoint. Further, if both $(\{A_j\}_{j\in \mathbb{J}},\{\Psi_j\}_{j\in \mathbb{J}} )$  and $ (\{B_j\}_{j\in \mathbb{J}}, \{\Phi_j\}_{j\in \mathbb{J}} )$ are  Parseval weak, then $(\{A_j\oplus B_j\}_{j \in \mathbb{J}},\{\Psi_j\oplus \Phi_j\}_{j \in \mathbb{J}})$ is Parseval weak.
\end{proposition}
\begin{proof}
Let $ h \oplus g \in \mathcal{H}\oplus \mathcal{H}$. Then 
\begin{align*}
&S_{A\oplus B, \Psi\oplus \Phi}(h\oplus g)=\sum_{j\in \mathbb{J}}(\Psi_j\oplus \Phi_j)^*(A_j\oplus B_j)(h\oplus g)=\sum_{j\in \mathbb{J}}(\Psi_j\oplus \Phi_j)^*(A_jh+ B_jg)\\
&=\sum_{j\in \mathbb{J}}(\Psi_j^*(A_jh+B_jg)\oplus \Phi_j^*(A_jh+B_jg))
=\left(\sum_{j\in \mathbb{J}}\Psi_j^*A_jh+\sum_{j\in \mathbb{J}}\Psi_j^*B_jg\right)\oplus \left(\sum_{j\in \mathbb{J}}\Phi_j^*A_jh+\sum_{j\in \mathbb{J}}\Phi_j^*B_jg\right)\\
&=(S_{A,\Psi}h+0)\oplus(0+S_{B,\Phi}g) =(S_{A,\Psi}\oplus S_{B,\Phi})(h\oplus g).
\end{align*}	
\end{proof}

\textbf{Characterizations}
\begin{theorem}\label{WEAKSEQUENTIALCHARACTERIZATION}
Let $ \{A_j\}_{j\in\mathbb{J}}, \{\Psi_j\}_{j\in\mathbb{J}}$ be in $ \mathcal{B}(\mathcal{H},\mathcal{H}_0).$ Suppose $ \{e_{j,k}\}_{k\in\mathbb{L}_j}$ is an orthonormal basis for $ \mathcal{H}_0,$ for each $j \in \mathbb{J}.$ Let  $ u_{j,k}=A_j^*e_{j,k}, v_{j,k}=\Psi_j^*e_{j,k}, \forall k \in  \mathbb{L}_j, j\in \mathbb{J}.$ Then $ (\{A_j\}_{j\in\mathbb{J}}, \{\Psi_j\}_{j\in\mathbb{J}})$ is 
\begin{enumerate}[\upshape(i)]
\item   a weak (ovf)  in $ \mathcal{B}(\mathcal{H},\mathcal{H}_0)$  with bounds $a $ and $ b$  if and only if  the map 
$$ T: \mathcal{H} \ni h \mapsto\sum\limits_{j\in \mathbb{J}}\sum\limits_{k\in \mathbb{L}_j}\langle h, u_{j,k}\rangle v_{j,k} \in  \mathcal{H} $$
is a well-defined bounded positive invertible operator such that $ a\|h\|^2 \leq \langle Th,h \rangle \leq b\|h\|^2, \forall h \in \mathcal{H}. $
\item    weak Bessel in  $ \mathcal{B}(\mathcal{H},\mathcal{H}_0)$  with bound  $ b$  if and only if  the map 
 $$ T: \mathcal{H} \ni h \mapsto\sum\limits_{j\in \mathbb{J}}\sum\limits_{k\in \mathbb{L}_j}\langle h, u_{j,k}\rangle v_{j,k} \in  \mathcal{H} $$
 is well-defined bounded positive  operator such that $ \langle Th,h \rangle \leq b\|h\|^2, \forall h \in \mathcal{H} .$ 
 \item a  weak (ovf)  in $ \mathcal{B}(\mathcal{H},\mathcal{H}_0)$  with bounds $a $ and $ b$  if and only if there exists $  r >0$ such that 
 $$\left \|\sum\limits_{j\in \mathbb{J}}\sum\limits_{k\in \mathbb{L}_j}\langle h, u_{j,k}\rangle v_{j,k}\right\|\leq r\|h\|,\forall h \in \mathcal{H}   ; ~\sum\limits_{j\in \mathbb{J}}\sum\limits_{k\in \mathbb{L}_j}\langle h, u_{j,k}\rangle v_{j,k} =\sum\limits_{j\in \mathbb{J}}\sum\limits_{k\in \mathbb{L}_j}\langle h, v_{j,k}\rangle u_{j,k} ,\forall h \in \mathcal{H} ;$$
 $$a\|h\|^2\leq \sum\limits_{j\in \mathbb{J}}\sum\limits_{k\in \mathbb{L}_j}\langle h, u_{j,k}\rangle \langle  v_{j,k} , h\rangle  \leq b\|h\|^2 ,~ \forall h \in \mathcal{H}.$$
 \item weak Bessel in $ \mathcal{B}(\mathcal{H},\mathcal{H}_0)$  with bound $ b$  if and only if there exists $  r >0$ such that 
 $$\left \|\sum\limits_{j\in \mathbb{J}}\sum\limits_{k\in \mathbb{L}_j}\langle h, u_{j,k}\rangle v_{j,k}\right\|\leq r\|h\|,\forall h \in \mathcal{H}   ; ~\sum\limits_{j\in \mathbb{J}}\sum\limits_{k\in \mathbb{L}_j}\langle h, u_{j,k}\rangle v_{j,k} =\sum\limits_{j\in \mathbb{J}}\sum\limits_{k\in \mathbb{L}_j}\langle h, v_{j,k}\rangle u_{j,k} ,\forall h \in \mathcal{H} ;$$
 $$0\leq \sum\limits_{j\in \mathbb{J}}\sum\limits_{k\in \mathbb{L}_j}\langle h, u_{j,k}\rangle \langle  v_{j,k} , h\rangle  \leq b\|h\|^2 ,~ \forall h \in \mathcal{H}.$$
 \end{enumerate}	
 \end{theorem}
 \begin{theorem}\label{WEAK CHARACTERIZATION}
 Let $ \{A_j\}_{j\in\mathbb{J}}, \{\Psi_j\}_{j\in\mathbb{J}}$ be in $ \mathcal{B}(\mathcal{H})$ such that $ \Psi_j^*A_j\geq 0, \forall j \in \mathbb{J}.$ Then $ (\{A_j\}_{j\in\mathbb{J}},\{\Psi_j\}_{j\in\mathbb{J}})$ is a  weak (ovf) in $ \mathcal{B}(\mathcal{H})$ if and only if 
  $$ T : \ell^2(\mathbb{J})\otimes \mathcal{H} \ni y \mapsto \sum\limits_{j\in \mathbb{J}} (\Psi_j^*A_j)^\frac{1}{2}L_j^*y\in \mathcal{H}$$
  is a  well-defined bounded surjective operator.
 \end{theorem}
 \begin{proof}
 $ (\Rightarrow)$ For every finite subset $ \mathbb{S}$ of $ \mathbb{J}$ and every $ y \in \ell^2(\mathbb{J})\otimes \mathcal{H},$ we have 
 
 \begin{align*}
 \left\|\sum\limits_{j\in \mathbb{S}} (\Psi_j^*A_j)^\frac{1}{2}L_j^*y\right\|&=\sup\limits_{h\in \mathcal{H}, \|h\|=1}\left|\left\langle \sum\limits_{j\in \mathbb{S}} (\Psi_j^*A_j)^\frac{1}{2}L_j^*y,h \right\rangle\right|=\sup\limits_{h\in \mathcal{H}, \|h\|=1}\left|\sum\limits_{j\in \mathbb{S}}\langle L_j^*y,(\Psi_j^*A_j)^\frac{1}{2}h \rangle \right|\\
 &\leq \sup\limits_{h\in \mathcal{H}, \|h\|=1}\sum\limits_{j\in \mathbb{S}}\| L_j^*y\|\|(\Psi_j^*A_j)^\frac{1}{2}h\|
  \leq \sup\limits_{h\in \mathcal{H}, \|h\|=1}\left(\sum\limits_{j\in \mathbb{S}}\|L^*_jy\|^2\right)^\frac{1}{2}\left(\sum\limits_{j\in \mathbb{S}}\|(\Psi_j^*A_j)^\frac{1}{2}h\|^2\right)^\frac{1}{2}\\
 &= \sup\limits_{h\in \mathcal{H}, \|h\|=1}\left(\sum\limits_{j\in \mathbb{S}}\|L^*_jy\|^2\right)^\frac{1}{2}\left(\sum\limits_{j\in \mathbb{S}}\langle(\Psi_j^*A_j)^\frac{1}{2}h,(\Psi_j^*A_j)^\frac{1}{2}h \rangle \right)^\frac{1}{2}\\
 &=\sup\limits_{h\in \mathcal{H}, \|h\|=1}\left(\sum\limits_{j\in \mathbb{S}}\|L^*_jy\|^2\right)^\frac{1}{2}\left(\sum\limits_{j\in \mathbb{S}}\langle\Psi_j^*A_jh,h \rangle \right)^\frac{1}{2}
 \leq \sup\limits_{h\in \mathcal{H}, \|h\|=1}\left(\sum\limits_{j\in \mathbb{S}}\|L^*_jy\|^2\right)^\frac{1}{2}\|S_{A,\Psi}h\|^\frac{1}{2}\\
& \leq\left(\sum\limits_{j\in \mathbb{S}}\|L^*_jy\|^2\right)^\frac{1}{2}\|S_{A,\Psi}\|^\frac{1}{2},
 \end{align*}
and  $\sum_{j\in \mathbb{J}}\|L^*_jy\|^2$  exists (it equals $ \|y\|^2$). Therefore $ T$ is bounded linear with $\| T\| \leq \|S_{A,\Psi}\|^{1/2}.$ Since $S_{A,\Psi} $ is invertible, for given $ g \in \mathcal{H}$ there exists $ h \in \mathcal{H}$ such that $ g= S_{A,\Psi} h.$ Now we  need to ensure the convergence of  $\sum_{k\in \mathbb{J}} L_k(\Psi_k^*A_k)^{1/2}h .$ In fact, for finite $ \mathbb{S} \subseteq \mathbb{J}$,  $ \|\sum_{k\in \mathbb{S}} L_k(\Psi_k^*A_k)^{1/2}h\|^2=\langle \sum_{k\in \mathbb{S}} \Psi_k^*A_kh, h \rangle $, this goes to $\langle S_{A,\Psi}h, h \rangle$. But then, 
 $$ g= \sum\limits_{j\in \mathbb{J}} \Psi_j^*A_jh=\sum\limits_{j\in \mathbb{J}} (\Psi_j^*A_j)^\frac{1}{2}L_j^*\left( \sum\limits_{k\in \mathbb{J}} L_k(\Psi_k^*A_k)^\frac{1}{2}h\right)=T\left( \sum\limits_{k\in \mathbb{J}} L_k(\Psi_k^*A_k)^\frac{1}{2}h\right).$$
   
 $(\Leftarrow)$ First we show that $ \sum_{j\in \mathbb{J}}\langle \Psi_j^*A_jh, h \rangle $ converges, $ \forall h \in \mathcal {H}$ and using this, we show $\sum_{j\in \mathbb{J}}\Psi_j^*A_jh $ converges, $ \forall h \in \mathcal {H}$.  Let $h \in \mathcal{H} $, $ \mathbb{S}$ be a finite subset of $ \mathbb{J}.$ Then
 \begin{align*}
\sum_{j\in \mathbb{S}}\langle \Psi_j^*A_jh, h \rangle &\leq \left\|\sum_{j\in \mathbb{S}} \Psi_j^*A_jh \right\|\|h\|= \left\|\sum\limits_{j\in \mathbb{S}}(\Psi_j^*A_j)^\frac{1}{2}L_j^*\left( \sum\limits_{k\in \mathbb{S}} L_k(\Psi_k^*A_k)^\frac{1}{2}h\right) \right\|\|h\|\\
&= \left\|T\left( \sum\limits_{k\in \mathbb{S}} L_k(\Psi_k^*A_k)^\frac{1}{2}h+\sum\limits_{k\in \mathbb{J}\setminus\mathbb{S} }0\right) \right\|\|h\|
\leq \|T\|\left \|\sum\limits_{k\in \mathbb{S}} L_k(\Psi_k^*A_k)^\frac{1}{2}h \right\|\|h\|\\
&=\|T\|\left \langle \sum\limits_{k\in \mathbb{S}}\Psi_k^*A_kh, h\right \rangle^\frac{1}{2}\|h\|.
 \end{align*}
Therefore $\sum_{j\in \mathbb{S}}\langle\Psi_j^*A_jh, h \rangle   \leq \|T\|^2\|h\|^2 $. Since $ \Psi_j^*A_j \geq 0, \forall j \in \mathbb{J},$ $\{\sum_{j\in \mathbb{S}}\langle\Psi_j^*A_jh, h \rangle : \mathbb{S} \subseteq \mathbb{J} \text{ is finite}\} $ is monotonically increasing,  and this gives  the convergence of $ \sum_{j\in \mathbb{J}}\langle \Psi_j^*A_jh, h \rangle. $ Next, 

 \begin{align*}
 \left\|\sum\limits_{j\in \mathbb{S}}\Psi_j^*A_jh \right\|& = \sup\limits_{g\in \mathcal{H}, \|g\|=1}\left|\left\langle \sum\limits_{j\in \mathbb{S}} \Psi_j^*A_jh, g \right\rangle\right|=\sup\limits_{g\in \mathcal{H}, \|g\|=1}\left| \sum\limits_{j\in \mathbb{S}} \langle (\Psi_j^*A_j)^\frac{1}{2}h,(\Psi_j^*A_j)^\frac{1}{2}g \rangle\right|\\
 &\leq \sup\limits_{g\in \mathcal{H}, \|g\|=1}\left(\sum\limits_{j\in \mathbb{S}}\|(\Psi_j^*A_j)^\frac{1}{2}h\|^2\right)^\frac{1}{2}\left(\sum\limits_{j\in \mathbb{S}}\|(\Psi_j^*A_j)^\frac{1}{2}g\|^2\right)^\frac{1}{2}\\
 &=\left(\sum\limits_{j\in \mathbb{S}}\langle \Psi_j^*A_jh, h \rangle \right)^\frac{1}{2}\sup\limits_{g\in \mathcal{H}, \|g\|=1}\left(\sum\limits_{j\in \mathbb{S}}\langle \Psi_j^*A_jg,g \rangle \right)^\frac{1}{2}\leq \left(\sum\limits_{j\in \mathbb{S}}\langle \Psi_j^*A_jh, h \rangle \right)^\frac{1}{2}\|T\|.
\end{align*} 
So $ \sum_{j\in \mathbb{J}}\Psi_j^*A_jh$ converges and $\|S_{A, \Psi}h\|=\|\sum_{j\in \mathbb{S}}\Psi_j^*A_jh \|\leq \|T\|^2\|h\|. $ Clearly $S_{A, \Psi} $ is positive.
 Since range of $ T$ is closed,  there exists bounded linear $ T^\dagger: \mathcal{H} \rightarrow \ell^2(\mathbb{J})\otimes \mathcal{H} $ such that $ TT^\dagger =I_\mathcal{H}.$ With this, for $h \in \mathcal{H}$ 
 \begin{align*}
 \|h\|^4= \langle TT^\dagger h, h\rangle^2 &=\left\langle \sum\limits_{j\in \mathbb{J}} (\Psi_j^*A_j)^\frac{1}{2}L_j^*T^\dagger h, h \right \rangle^2= \left(\sum\limits_{j\in \mathbb{J}}\langle  L_j^*T^\dagger h, (\Psi_j^*A_j)^\frac{1}{2}h  \rangle\right)^2
 \\
& \leq \left(\sum\limits_{j\in \mathbb{J}}\|L_j^*T^\dagger h\|^2\right) \left(\sum\limits_{j\in \mathbb{J}}\|(\Psi_j^*A_j)^\frac{1}{2}h\|^2\right)= \left\langle \sum_{j \in \mathbb{J}}L_jL_j^*T^\dagger h, T^\dagger h\right\rangle \sum\limits_{j\in \mathbb{J}}\langle \Psi_j^*A_jh, h \rangle \\
&=\|T^\dagger h\|^2\sum\limits_{j\in \mathbb{J}}\langle \Psi_j^*A_jh, h \rangle \leq \|T^\dagger\| \|h\|^2\sum\limits_{j\in \mathbb{J}}\langle \Psi_j^*A_jh, h \rangle 
 \end{align*}
 and this gives $    \|T^\dagger\|^{-1} \|h\|^2 \leq \langle  S_{A, \Psi}h, h\rangle  , \forall h \in  \mathcal{H}.$
 \end{proof}
 \textbf{Similarity and tensor product of weak operator-valued frames }
 \begin{definition}
A weak (ovf)  $(\{B_j\}_{j\in \mathbb{J}} ,\{\Phi_j\}_{j\in \mathbb{J}})$ in $ \mathcal{B}(\mathcal{H}, \mathcal{H}_0)$  is said to be right-similar  to a weak (ovf) $(\{A_j\}_{j\in \mathbb{J}},\{\Psi_j\}_{j\in \mathbb{J}})$ in $ \mathcal{B}(\mathcal{H}, \mathcal{H}_0)$ if there exist invertible  $ R_{A,B}, R_{\Psi, \Phi} \in \mathcal{B}(\mathcal{H})$  such that $B_j=A_jR_{A,B} , \Phi_j=\Psi_jR_{\Psi, \Phi}, \forall j \in \mathbb{J}. $
 \end{definition}
 \begin{proposition}
 Let $ \{A_j\}_{j\in \mathbb{J}}\in \mathscr{F}^\text{w}_\Psi$  with frame bounds $a, b,$  let $R_{A,B}, R_{\Psi, \Phi} \in \mathcal{B}(\mathcal{H})$ be positive, invertible, commute with each other, commute with $ S_{A, \Psi}$, and let $B_j=A_jR_{A,B} , \Phi_j=\Psi_jR_{\Psi, \Phi},  \forall j \in \mathbb{J}.$ Then 
 $ \{B_j\}_{j\in \mathbb{J}}\in \mathscr{F}^\text{w}_\Phi,$  $ S_{B,\Phi}=R_{\Psi,\Phi}S_{A, \Psi}R_{A,B} $, and    $ \frac{a}{\|R_{A,B}^{-1}\|\|R_{\Psi,\Phi}^{-1}\|}\leq S_{B, \Phi} \leq b\|R_{A,B}R_{\Psi,\Phi}\|.$ Assuming that $( \{A_j\}_{j\in \mathbb{J}}, \{\Psi_j\}_{j\in \mathbb{J}})$ is a Parseval weak (ovf), then $ (\{B_j\}_{j\in \mathbb{J}},\{\Phi_j\}_{j\in \mathbb{J}} ) $ is a Parseval  weak (ovf) if and only if   $ R_{\Psi, \Phi}R_{A,B}=I_\mathcal{H}.$  
 \end{proposition}
 \begin{proof}
 Proof is there in the proof of Proposition \ref{RIGHTSIMILARITYPROPOSITIONOPERATORVERSION}.
 \end{proof}
 \begin{proposition}
 Let $ \{A_j\}_{j\in \mathbb{J}}\in \mathscr{F}^\text{w}_\Psi,$ $ \{B_j\}_{j\in \mathbb{J}}\in \mathscr{F}^\text{w}_\Phi$ and   $B_j=A_jR_{A,B} , \Phi_j=\Psi_jR_{\Psi, \Phi},  \forall j \in \mathbb{J}$, for some invertible $ R_{A,B} ,R_{\Psi, \Phi} \in \mathcal{B}(\mathcal{H}).$ Then $  S_{B,\Phi}=R_{\Psi,\Phi}^*S_{A, \Psi}R_{A,B}.$ Assuming that $ (\{A_j\}_{j\in \mathbb{J}},\{\Psi_j\}_{j\in \mathbb{J}})$ is a  Parseval weak (ovf), then $(\{B_j\}_{j\in \mathbb{J}},  \{\Phi_j\}_{j\in \mathbb{J}})$ is a Parseval  weak (ovf) if and only if   $ R_{\Psi, \Phi}^*R_{A,B}=I_\mathcal{H}.$
 \end{proposition}
\begin{proof}
$S_{B,\Phi}= \sum_{j\in \mathbb{J}}\Phi_j^*B_j=\sum_{j\in \mathbb{J}}R_{\Psi,\Phi}^*\Psi_j^*A_jR_{A,B}=R_{\Psi,\Phi}^*S_{A, \Psi}R_{A,B}.$
\end{proof}
\begin{remark}
For every weak (ovf) $(\{A_j\}_{j \in \mathbb{J}},\{\Psi_j\}_{j \in \mathbb{J}})$, each  of `weak operator-valued frames'  $( \{A_jS_{A, \Psi}^{-1}\}_{j \in \mathbb{J}}, \{\Psi_j\}_{j \in \mathbb{J}}),$   $( \{A_jS_{A, \Psi}^{-1/2}\}_{j \in \mathbb{J}}, \{\Psi_jS_{A,\Psi}^{-1/2}\}_{j \in \mathbb{J}}),$ and  $ (\{A_j \}_{j \in \mathbb{J}}, \{\Psi_jS_{A,\Psi}^{-1}\}_{j \in \mathbb{J}})$ is  a Parseval weak  (ovf)  which is right-similar to  $ (\{A_j\}_{j \in \mathbb{J}} , \{\Psi_j\}_{j \in \mathbb{J}}).$  
\end{remark}

 \textbf{Tensor product}: Let $ \{A_j\}_{j \in \mathbb{J}} $  be a weak (ovf) w.r.t. $ \{\Psi_j\}_{j \in \mathbb{J}} $  in  $ \mathcal{B}(\mathcal{H}, \mathcal{H}_0),$ and $ \{B_l\}_{l \in \mathbb{L}} $  be a weak (ovf) w.r.t. $ \{\Phi_l\}_{l \in \mathbb{L}} $  in  $ \mathcal{B}(\mathcal{H}_1, \mathcal{H}_2).$ The weak (ovf)  $(\{C_{(j, l)}\coloneqq A_j\otimes B_l\}_{(j, l)\in \mathbb{J}\bigtimes  \mathbb{L}}, \{\Xi_{(j, l)}\coloneqq \Psi_j\otimes\Phi_l\}_{(j, l)\in \mathbb{J}\bigtimes  \mathbb{L}}) $ in $ \mathcal{B}(\mathcal{H}\otimes\mathcal{H}_1, \mathcal{H}_0\otimes\mathcal{H}_2)$ is called  as tensor product  of  $( \{A_j\}_{j \in \mathbb{J}}, \{\Psi_j\}_{j\in \mathbb{J}})$ and $( \{B_l\}_{l \in \mathbb{L}},  \{\Phi_l\}_{l\in \mathbb{L}}).$
 \begin{proposition}
 Let $(\{C_{(j, l)}\coloneqq A_j\otimes B_l\}_{(j, l)\in \mathbb{J}\bigtimes  \mathbb{L}}, \{\Xi _{(j, l)}\coloneqq \Psi_j\otimes \Phi_l\}_{(j, l)\in \mathbb{J}\bigtimes  \mathbb{L}})$ be the  tensor product of   weak operator-valued frames  $( \{A_j\}_{j \in \mathbb{J}}, \{\Psi_j\}_{j \in \mathbb{J}}) $  in  $ \mathcal{B}(\mathcal{H}, \mathcal{H}_0),$ and $( \{B_l\}_{l \in \mathbb{L}}, \{\Phi_l\}_{l \in \mathbb{L}} )$ in $ \mathcal{B}(\mathcal{H}_1, \mathcal{H}_2).$  Then  $ S_{C, \Xi}=S_{A, \Psi}\otimes S_{B, \Phi}.$ If  $( \{A_j\}_{j \in \mathbb{J}}, \{\Psi_j\}_{j \in \mathbb{J}}) $ and $ (\{B_l\}_{l \in \mathbb{L}},  \{\Phi_l\}_{l \in \mathbb{L}} )$ are Parseval  weak frames , then $(\{C_{(j, l)}\}_{(j, l)\in \mathbb{J}\bigtimes  \mathbb{L}} ,\{\Xi_{(j,l)}\}_{(j,l)\in \mathbb{J}\bigtimes \mathbb{L}})$ is Parseval  weak.
 \end{proposition}
 \begin{proof}
 At each  elementary tensor $ h\otimes g \in  \mathcal{H}\otimes\mathcal{H}_1,$ we have
 \begin{align*}
 &(S_{A, \Psi}\otimes S_{B, \Phi})(h\otimes g)=S_{A, \Psi}h\otimes S_{B, \Phi}g=\left(\sum_{j\in \mathbb{J}}\Psi_j^*A_jh\right) \otimes \left(\sum_{k\in \mathbb{L}}\Phi_k^*B_kg\right)\\
 &=\sum_{j\in \mathbb{J}}\sum_{k\in \mathbb{L}}(\Psi_j^*\otimes \Phi_k^*)(A_j\otimes B_k)(h\otimes g) =\sum_{(j,k)\in \mathbb{J}\bigtimes \mathbb{L}}(\Psi_j\otimes \Phi_k)^*(A_j\otimes B_k)(h\otimes g) =S_{C,\Xi}(h\otimes g).
 \end{align*}	
 \end{proof}
 \textbf{Sequential version of weak frames}
 \begin{definition}\label{WEAKFRAMEDEFINITION}
 A collection $ \{x_j\}_{j\in \mathbb{J}}$ in  $ \mathcal{H}$ is called a weak frame w.r.t. a collection $ \{\tau_j\}_{j\in \mathbb{J}}$ in  $ \mathcal{H}$ if  $S_{x, \tau} :\mathcal{H} \ni  h \mapsto \sum_{j\in\mathbb{J}}\langle h,  x_j\rangle\tau_j \in \mathcal{H} $ is a well-defined  bounded positive  invertible operator.
 Notions of frame bounds, optimal bounds, tight frame, Parseval frame, Bessel are in parallel  to the same  in Definition \ref{SEQUENTIAL2}.
 
 For fixed $ \mathbb{J}, \mathcal{H},$  and $ \{\tau_j\}_{j\in \mathbb{J}}$   the set of all weak frames for $ \mathcal{H}$  w.r.t.  $ \{\tau_j\}_{j\in \mathbb{J}}$ is denoted by $ \mathscr{F}^w_\tau.$
 \end{definition}
 Above definition is equivalent to
 \begin{definition}
 A collection $ \{x_j\}_{j\in \mathbb{J}}$ in  $ \mathcal{H}$ is called a weak frame w.r.t. a collection $ \{\tau_j\}_{j\in \mathbb{J}}$ in  $ \mathcal{H}$ if there are $ a,b, r >0$ such that 
 \begin{enumerate}[\upshape(i)]
  \item $ \|\sum_{j\in \mathbb{J}}\langle h, x_j\rangle\tau_j\|\leq r\|h\|, \forall h \in  \mathcal{H} ,$
 \item $ \sum_{j\in \mathbb{J}}\langle h, x_j\rangle\tau_j= \sum_{j\in \mathbb{J}}\langle h, \tau_j\rangle x_j, \forall h \in  \mathcal{H} ,$
 \item  $ a\|h\|^2\leq\sum_{j\in \mathbb{J}}\langle h,x_j \rangle \langle \tau_j, h\rangle \leq b\|h\|^2, \forall h \in  \mathcal{H}$.
\end{enumerate}
 If the space is complex, then   \text{\upshape(ii)} is not required. 
\end{definition}
Note that the first  condition in the above definition gives the existence and boundedness of $ S_{x, \tau}$ and second and third conditions  give the self-adjointness, positivity and invertibility of $ S_{x, \tau}$.
\begin{theorem}
Let $\{x_j\}_{j\in \mathbb{J}}, \{\tau_j\}_{j\in \mathbb{J}}$ be in $\mathcal{H}$. Define $A_j: \mathcal{H} \ni h \mapsto \langle h, x_j \rangle \in \mathbb{K} $, $\Psi_j: \mathcal{H} \ni h \mapsto \langle h, \tau_j \rangle \in \mathbb{K}, \forall j \in \mathbb{J} $. Then   $(\{x_j\}_{j\in \mathbb{J}}, \{\tau_j\}_{j\in \mathbb{J}})$ is a weak frame for  $\mathcal{H}$ if and only if  $\{A_j\}_{j\in \mathbb{J}}, \{\Psi_j\}_{j\in \mathbb{J}}$ is a  weak operator-valued frame  in $\mathcal{B}(\mathcal{H},\mathbb{K})$.
\end{theorem} 
\begin{proposition}
If $(\{x_j\}_{j\in \mathbb{J}}, \{\tau_j\}_{j\in \mathbb{J}})$ is a weak frame for  $\mathcal{H}$, then  every $ h \in \mathcal{H}$ can be written as 
$$h =\sum\limits_{j\in \mathbb{J}}\langle h, S^{-1}_{x, \tau}\tau_j\rangle  x_j=\sum\limits_{j\in \mathbb{J}}\langle h,\tau_j \rangle S^{-1}_{x, \tau}  x_j =\sum\limits_{j\in \mathbb{J}}\langle h, S^{-1}_{x, \tau}x_j\rangle  \tau_j=\sum\limits_{j\in \mathbb{J}}\langle h, x_j\rangle S^{-1}_{x, \tau}  \tau_j.$$
\end{proposition}
\begin{proof}
For all $h \in \mathcal{H}$, $h=S^{-1}_{x, \tau}S_{x, \tau}h=S^{-1}_{x, \tau}(\sum_{j\in \mathbb{J}}\langle h, x_j\rangle  \tau_j)=\sum_{j\in \mathbb{J}}\langle h, x_j\rangle  S^{-1}_{x, \tau}\tau_j$,  $h=S_{x, \tau}S_{x, \tau}^{-1}h=\sum_{j\in \mathbb{J}}\langle  S^{-1}_{x, \tau}h, x_j\rangle  \tau_j=\sum_{j\in \mathbb{J}}\langle  h, S^{-1}_{x, \tau}x_j\rangle  \tau_j $. Other two expansions will come from $S_{x,\tau}=S_{\tau,x}$.
\end{proof}
\begin{proposition}
Let $ \{x_j\}_{j\in \mathbb{J}}$ and  $ \{\tau_j\}_{j\in \mathbb{J}}$ be in  $ \mathcal{H}$ such that $\langle h, x_j\rangle\langle \tau_j, h\rangle  \geq 0,\forall h \in  \mathcal{H}, \forall j \in \mathbb{J} $, and 
$$ h=\sum_{j\in \mathbb{J}}\langle h, x_j\rangle \tau_j=\sum_{j\in \mathbb{J}}\langle h, \tau_j\rangle x_j,~ \forall h \in  \mathcal{H}.$$
Then  $ (\{x_j\}_{j\in \mathbb{J}}, \{\tau_j\}_{j\in \mathbb{J}})$ is a tight weak frame for  $ \mathcal{H}$.
\end{proposition}
\begin{proof}
Given equation gives existence, boundedness and self-adjointness of $ S_{x, \tau}$.  We have to show the frame inequality. Now
$$ \|h\|^2 =\lim\left\langle \sum_{j\in \mathbb{S},~\mathbb{S}~\text{ finite}}\langle h, x_j\rangle \tau_j, h \right \rangle=\lim\sum_{j\in \mathbb{S},~\mathbb{S}~\text{finite}}\langle h, x_j\rangle \langle \tau_j, h\rangle =\sum_{j\in \mathbb{J}}\langle h, x_j\rangle \langle \tau_j, h\rangle,~\forall h \in  \mathcal{H},$$
where to get  the last equality we have used the fact that a monotonically increasing net of real numbers converges.
\end{proof}
\begin{proposition}
Let  $( \{x_j\}_{j \in \mathbb{J}},\{\tau_j\}_{j \in \mathbb{J}} )$ be a weak frame for  $ \mathcal{H}$  with an upper frame bound $b$. If for some $ j \in \mathbb{J} $ we have  $  \langle x_j, x_l \rangle\langle \tau_l, x_j \rangle \geq0, \forall l  \in \mathbb{J},$ then $ \langle x_j, \tau_j \rangle\leq b$ for that $j. $
\end{proposition}
\begin{proposition}
Every weak Bessel sequence 	$(\{x_j\}_{j \in \mathbb{J}},\{\tau_j\}_{j \in \mathbb{J}} )$ for  $ \mathcal{H}$ can be extended to a tight weak frame for $ \mathcal{H}$. In particular,  every weak frame	$(\{x_j\}_{j \in \mathbb{J}},\{\tau_j\}_{j \in \mathbb{J}} )$ for $ \mathcal{H}$ can be extended to a tight weak frame for $ \mathcal{H}$.
\end{proposition}
\begin{proposition}
Let $( \{x_n\}_{n=1}^\infty,\{\tau_n\}_{n=1}^\infty )$ be a weak frame for  $ \mathcal{H}$	with bounds $a$ and  $b$. For  $ h \in \mathcal{H}$ define 
$$ h_0\coloneqq0,~ h_n\coloneqq h_{n-1}+\frac{2}{a+b}S_{x,\tau}(h-h_{n-1}), ~\forall n \geq1.
~\text{Then}~ \|h_n-h\|\leq \left(\frac{b-a}{b+a}\right)^n\|h\|, ~\forall n \geq1.$$
\end{proposition} 
 \begin{definition}\label{WEAKDUALDEFINITIONSEQUENTIAL}
A weak frame   $(\{y_j\}_{j\in \mathbb{J}}, \{\omega_j\}_{j\in \mathbb{J}})$  for  $\mathcal{H}$ is said to be a dual of weak frame  $ ( \{x_j\}_{j\in \mathbb{J}}, \{\tau_j\}_{j\in \mathbb{J}})$ for  $\mathcal{H}$  if $ \sum_{j\in \mathbb{J}}\langle h, x_j\rangle \omega_j= \sum_{j\in \mathbb{J}}\langle h, \tau_j\rangle y_j=h, \forall h \in  \mathcal{H}$. The `weak frame' $(   \{\widetilde{x}_j\coloneqq S_{x,\tau}^{-1}x_j\}_{j\in \mathbb{J}},\{\widetilde{\tau}_j\coloneqq S_{x,\tau}^{-1}\tau_j\}_{j \in \mathbb{J}} )$, which is a `dual' of $ (\{x_j\}_{j\in \mathbb{J}}, \{\tau_j\}_{j\in \mathbb{J}})$ is called the canonical dual of $ (\{x_j\}_{j\in \mathbb{J}}, \{\tau_j\}_{j\in \mathbb{J}})$.
 \end{definition}
 Symmetry of Definition \ref{WEAKFRAMEDEFINITION} implies symmetry of Definition \ref{WEAKDUALDEFINITIONSEQUENTIAL}.
\begin{proposition}
Let $( \{x_j\}_{j\in \mathbb{J}},\{\tau_j\}_{j\in \mathbb{J}} )$ be a weak frame for  $\mathcal{H}.$ If $ h \in \mathcal{H}$ has representation  $ h=\sum_{j\in\mathbb{J}}c_jx_j= \sum_{j\in\mathbb{J}}d_j\tau_j, $ for some scalar sequences  $ \{c_j\}_{j\in \mathbb{J}},\{d_j\}_{j\in \mathbb{J}}$,  then 
$$ \sum\limits_{j\in \mathbb{J}}c_j\bar{d}_j =\sum\limits_{j\in \mathbb{J}}\langle h, \widetilde{\tau}_j\rangle\langle \widetilde{x}_j , h \rangle+\sum\limits_{j\in \mathbb{J}}(\langle c_j-\langle h, \widetilde{\tau}_j\rangle)(\bar{d}_j-\langle \widetilde{x}_j, h\rangle). $$
\end{proposition}
\begin{theorem}
Let $( \{x_j\}_{j\in \mathbb{J}},\{\tau_j\}_{j\in \mathbb{J}} )$ be a weak frame for $ \mathcal{H}$ with frame bounds $ a$ and $ b.$ Then
\begin{enumerate}[\upshape(i)]
\item The canonical dual weak frame of the canonical dual weak frame  of $ (\{x_j\}_{j\in \mathbb{J}} ,\{\tau_j\}_{j\in \mathbb{J}} )$ is itself.
\item$ \frac{1}{b}, \frac{1}{a}$ are frame bounds for the canonical dual of $ (\{x_j\}_{j\in \mathbb{J}},\{\tau_j\}_{j\in \mathbb{J}}).$
\item If $ a, b $ are optimal frame bounds for $( \{x_j\}_{j\in \mathbb{J}} , \{\tau_j\}_{j\in \mathbb{J}}),$ then $ \frac{1}{b}, \frac{1}{a}$ are optimal  frame bounds for its canonical dual.
\end{enumerate} 
\end{theorem}
\begin{proof}
We   refer to the proof of Theorem \ref{CANONICALDUALFRAMEPROPERTYSEQUENTIALVERSION}.
\end{proof}
\begin{definition}\label{WEAKORTHOGONALDEFINITIONSEQUENTIAL}
A weak frame   $(\{y_j\}_{j\in \mathbb{J}},  \{\omega_j\}_{j\in \mathbb{J}})$  for  $\mathcal{H}$ is said to be orthogonal to a weak frame   $( \{x_j\}_{j\in \mathbb{J}}, \{\tau_j\}_{j\in \mathbb{J}})$ for $\mathcal{H}$ if $\sum_{j\in \mathbb{J}}\langle h, x_j\rangle \omega_j= \sum_{j\in \mathbb{J}}\langle h, \tau_j\rangle y_j=0, \forall h \in  \mathcal{H}.$
\end{definition}
Since Definition \ref{WEAKFRAMEDEFINITION} is symmetric, Definition \ref{WEAKORTHOGONALDEFINITIONSEQUENTIAL} is also symmetric.
\begin{proposition}
Two orthogonal weak frames  have common dual weak frame.	
\end{proposition}
\begin{proof}
Let  $ (\{x_j\}_{j\in \mathbb{J}}, \{\tau_j\}_{j\in \mathbb{J}}) $ and $ (\{y_j\}_{j\in \mathbb{J}}, \{\omega_j\}_{j\in \mathbb{J}}) $ be   orthogonal frames for  $\mathcal{H}$. Define $ z_j\coloneqq S_{x,\tau}^{-1}x_j+S_{y,\omega}^{-1}y_j,\rho_j\coloneqq S_{x,\tau}^{-1}\tau_j+S_{y,\omega}^{-1}\omega_j, \forall j \in \mathbb{J}$. For $h \in \mathcal{H}$,  
\begin{align*}
&S_{z,\rho}h=\sum_{j\in \mathbb{J}}\langle h, z_j\rangle \rho_j=\sum_{j\in \mathbb{J}}\langle h,S_{x,\tau}^{-1}x_j+S_{y,\omega}^{-1}y_j \rangle (S_{x,\tau}^{-1}\tau_j+S_{y,\omega}^{-1}\omega_j)\\
&=S_{x,\tau}^{-1}\left(\sum_{j\in \mathbb{J}}\langle S_{x,\tau}^{-1}h,x_j \rangle\tau_j\right)+S_{y,\omega}^{-1}\left(\sum_{j\in \mathbb{J}}\langle S_{x,\tau}^{-1}h,x_j \rangle\omega_j\right)+S_{x,\tau}^{-1}\left(\sum_{j\in \mathbb{J}}\langle S_{y,\omega}^{-1} h,y_j \rangle\tau_j\right)+S_{y,\omega}^{-1}\left(\sum_{j\in \mathbb{J}}\langle S_{y,\omega}^{-1}h, y_j\rangle\omega_j\right)\\
&=S_{x,\tau}^{-1}S_{x,\tau}S_{x,\tau}^{-1}h+S_{y,\omega}^{-1}0+S_{x,\tau}^{-1}0+S_{y,\omega}^{-1}S_{y,\omega}S_{y,\omega}^{-1}h=S_{x,\tau}^{-1}h+S_{y,\omega}^{-1}h=(S_{x,\tau}^{-1}+S_{y,\omega}^{-1})h.
\end{align*}
Therefore $ (\{z_j\}_{j\in \mathbb{J}}, \{\rho_j\}_{j\in \mathbb{J}}) $ is a  weak frame for  $\mathcal{H}$. For duality: $\sum_{j \in \mathbb{J}}\langle h, x_j\rangle \rho_j=\sum_{j \in \mathbb{J}}\langle h, x_j \rangle S_{x,\tau}^{-1}\tau_j+\sum_{j \in \mathbb{J}}\langle h, x_j \rangle S_{y,\omega}^{-1}\omega_j =S_{x, \tau}^{-1}(\sum_{j \in \mathbb{J}}\langle h, x_j \rangle \tau_j)+ S_{y,\omega}^{-1}(\sum_{j \in \mathbb{J}}\langle h, x_j \rangle \omega_j)=h+S_{y,\omega}^{-1}0=h$, $\sum_{j \in \mathbb{J}}\langle h, \tau_j\rangle z_j=S_{x, \tau}^{-1}(\sum_{j \in \mathbb{J}}\langle h, \tau_j\rangle x_j)+S_{y,\omega}^{-1}(\sum_{j \in \mathbb{J}}\langle h, \tau_j\rangle y_j)=h+S_{y,\omega}^{-1}0=h $, and $\sum_{j \in \mathbb{J}}\langle h, y_j\rangle \rho_j=S_{x,\tau}^{-1}(\sum_{j \in \mathbb{J}}\langle h, y_j\rangle \tau_j)+S_{y,\omega}^{-1}(\sum_{j \in \mathbb{J}}\langle h, y_j\rangle \omega_j)=S_{x,\tau}^{-1}0+h=h $, $\sum_{j \in \mathbb{J}}\langle h, \omega_j\rangle z_j=S_{x,\tau}^{-1}(\sum_{j \in \mathbb{J}}\langle h, \omega_j\rangle x_j)+S_{y,\omega}^{-1}(\sum_{j \in \mathbb{J}}\langle h, \omega_j\rangle y_j)=S_{x,\tau}^{-1}0+h=h, \forall h \in \mathcal{H}$.
\end{proof}
\begin{proposition}
Let $ (\{x_j\}_{j\in \mathbb{J}}, \{\tau_j\}_{j\in \mathbb{J}}) $ and $ (\{y_j\}_{j\in \mathbb{J}}, \{\omega_j\}_{j\in \mathbb{J}}) $ be  two Parseval weak frames for  $\mathcal{H}$ which are  orthogonal. If $A,B,C,D \in \mathcal{B}(\mathcal{H})$ are such that $ AC^*+BD^*=I_\mathcal{H}$, then  $ (\{Ax_j+By_j\}_{j\in \mathbb{J}}, \{C\tau_j+D\omega_j\}_{j\in \mathbb{J}}) $ is a  Parseval weak frame for  $\mathcal{H}$. In particular,  if scalars $ a,b,c,d$ satisfy $a\bar{c}+b\bar{d} =1$, then $ (\{ax_j+by_j\}_{j\in \mathbb{J}}, \{c\tau_j+d\omega_j\}_{j\in \mathbb{J}}) $ is a  Parseval weak frame for  $\mathcal{H}$.
\end{proposition} 
\begin{proof}
For all $h \in \mathcal{H}$,
\begin{align*}
&S_{Ax+By,C\tau+D\omega}h =\sum_{j\in \mathbb{J}}\langle h,Ax_j+By_j \rangle (C\tau_j+D\omega_j)\\
&=C\left(\sum_{j\in \mathbb{J}}\langle A^* h, x_j\rangle \tau_j\right)+D\left(\sum_{j\in \mathbb{J}}\langle A^*h, x_j\rangle \omega_j\right)+C\left(\sum_{j\in \mathbb{J}}\langle B^*h, x_j\rangle \tau_j\right)+D\left(\sum_{j\in \mathbb{J}}\langle B^*h, y_j\rangle \omega_j\right)\\
&=CA^*h+D0+C0+DB^*h=h.
\end{align*}
\end{proof}
\begin{definition}
Two weak frames  $(\{x_j\}_{j\in \mathbb{J}},  \{\tau_j\}_{j\in \mathbb{J}}) $ and $(\{y_j\}_{j\in \mathbb{J}},\{\omega_j\}_{j\in \mathbb{J}})$  for $ \mathcal{H}$ are called disjoint if $(\{x_j\oplus y_j\}_{j\in \mathbb{J}},\{\tau_j\oplus\omega_j\}_{j\in \mathbb{J}})$ is a weak frame for $\mathcal{H}\oplus\mathcal{H} $.
\end{definition} 
\begin{proposition}
If $(\{x_j\}_{j\in \mathbb{J}},\{\tau_j\}_{j\in \mathbb{J}} )$  and $ (\{y_j\}_{j\in \mathbb{J}}, \{\omega_j\}_{j\in \mathbb{J}} )$  are  disjoint  weak frames  for $\mathcal{H}$, then  they  are disjoint. Further, if both $(\{x_j\}_{j\in \mathbb{J}},\{\tau_j\}_{j\in \mathbb{J}} )$  and $ (\{y_j\}_{j\in \mathbb{J}}, \{\omega_j\}_{j\in \mathbb{J}} )$ are  Parseval weak, then $(\{x_j\oplus y_j\}_{j \in \mathbb{J}},\{\tau_j\oplus \omega_j\}_{j \in \mathbb{J}})$ is Parseval weak.
\end{proposition}
\begin{proof}
Proof is inside the proof of Proposition \ref{SEQUENTIALDISJOINTFRAMEPROPOSITION}. 
\end{proof}

\textbf{Similarity  and tensor product}
\begin{definition}
A weak frame  $ (\{y_j\}_{j\in \mathbb{J}},\{\omega_j\}_{j\in \mathbb{J}})$ for $ \mathcal{H}$ is said to be  similar to  a weak frame $ (\{x_j\}_{j\in \mathbb{J}},\{\tau_j\}_{j\in \mathbb{J}})$ for $ \mathcal{H}$ if there are invertible operators $ T_{x,y}, T_{\tau,\omega} \in \mathcal{B}(\mathcal{H})$ such that $ y_j=T_{x,y}x_j, \omega_j=T_{\tau,\omega}\tau_j,   \forall j \in \mathbb{J}.$
\end{definition} 
\begin{proposition}
Let $ \{x_j\}_{j\in \mathbb{J}}\in \mathscr{F}^w_\tau$  with frame bounds $a, b,$  let $T_{x,y} , T_{\tau,\omega}\in \mathcal{B}(\mathcal{H})$ be positive, invertible, commute with each other, commute with $ S_{x, \tau}$, and let $y_j=T_{x,y}x_j , \omega_j=T_{\tau,\omega}\tau_j,  \forall j \in \mathbb{J}.$ Then $ \{y_j\}_{j\in \mathbb{J}}\in \mathscr{F}^w_\tau$, $S_{y,\omega}=T_{\tau,\omega}S_{x, \tau}T_{x,y}$, and $ \frac{a}{\|T_{x,y}^{-1}\|\|T_{\tau,\omega}^{-1}\|}\leq S_{y, \omega} \leq b\|T_{x,y}T_{\tau,\omega}\|$. Assuming that $ (\{x_j\}_{j\in \mathbb{J}},\{\tau_j\}_{j\in \mathbb{J}})$ is Parseval weak, then $(\{y_j\}_{j\in \mathbb{J}},  \{\omega_j\}_{j\in \mathbb{J}})$ is Parseval weak if and only if   $ T_{\tau, \omega}T_{x,y}=I_\mathcal{H}.$  
\end{proposition}     
 \begin{proof}
 For $ h \in \mathcal{H}$, $ S_{y,\omega}=\sum_{j \in \mathbb{J}}\langle h,y_j \rangle \omega_j=\sum_{j \in \mathbb{J}}\langle h,T_{x,y}x_j \rangle T_{\tau,\omega}\tau_j=T_{\tau,\omega}(\sum_{j \in \mathbb{J}}\langle T_{x,y}h,x_j \rangle \tau_j)=T_{\tau,\omega}S_{x, \tau}T_{x,y}h.$
 \end{proof}     
 \begin{proposition}
 Let $ \{x_j\}_{j\in \mathbb{J}}\in \mathscr{F}^w_\tau,$ $ \{y_j\}_{j\in \mathbb{J}}\in \mathscr{F}^w_\omega$ and   $y_j=T_{x, y}x_j , \omega_j=T_{\tau,\omega}\tau_j,  \forall j \in \mathbb{J}$, for some invertible $T_{x,y}, T_{\tau,\omega}\in \mathcal{B}(\mathcal{H}).$ Then 
 $ \theta_y=\theta_x T^*_{x,y}, \theta_\omega=\theta_\tau T^*_{\tau,\omega}, S_{y,\omega}=T_{\tau,\omega}S_{x, \tau}T_{x,y}^*,  P_{y,\omega}=P_{x, \tau}.$ Assuming that $ (\{x_j\}_{j\in \mathbb{J}},\{\tau_j\}_{j\in \mathbb{J}})$ is a Parseval weak frame, then $ (\{y_j\}_{j\in \mathbb{J}},\{\omega_j\}_{j\in \mathbb{J}})$ is a Parseval weak frame if and only if $T_{\tau,\omega}T_{x,y}^*=I_\mathcal{H}.$
 \end{proposition}     
\begin{remark}
For every weak frame  $(\{x_j\}_{j \in \mathbb{J}}, \{\tau_j\}_{j \in \mathbb{J}}),$ each  of `weak frames'  $( \{S_{x, \tau}^{-1}x_j\}_{j \in \mathbb{J}}, \{\tau_j\}_{j \in \mathbb{J}})$,    $( \{S_{x, \tau}^{-1/2}x_j\}_{j \in \mathbb{J}}, \{S_{x,\tau}^{-1/2}\tau_j\}_{j \in \mathbb{J}}),$ and  $ (\{x_j \}_{j \in \mathbb{J}}, \{S_{x,\tau}^{-1}\tau_j\}_{j \in \mathbb{J}})$ is a Parseval  weak frame which is similar to  $ (\{x_j\}_{j \in \mathbb{J}} , \{\tau_j\}_{j \in \mathbb{J}} ).$  
\end{remark}

\textbf{Tensor product}: Let $(\{x_j\}_{j \in \mathbb{J}}, \{\tau_j\}_{j \in \mathbb{J}})$ be a weak  frame  for   $ \mathcal{H},$ and $(\{y_l\}_{l \in \mathbb{L}}, \{\omega_l\}_{l \in \mathbb{L}})$ be a weak frame  for   $ \mathcal{H}_1.$ The weak frame  $(\{z_{(j, l)}\coloneqq x_j\otimes y_l\}_{(j, l)\in \mathbb{J}\bigtimes  \mathbb{L}},\{\rho_{(j, l)}\coloneqq \tau_j\otimes\omega_l\}_{(j, l)\in \mathbb{J}\bigtimes  \mathbb{L}})$   for $ \mathcal{H}\otimes\mathcal{H}_1$ is called  as tensor product  of weak frames $( \{x_j\}_{j \in \mathbb{J}}, \{\tau_j\}_{j\in \mathbb{J}})$ and $( \{y_l\}_{l \in \mathbb{L}},  \{\omega_l\}_{l\in \mathbb{L}}).$ 
\begin{proposition}
Let  $(\{z_{(j, l)}\coloneqq x_j\otimes y_l\}_{(j, l)\in \mathbb{J}\bigtimes  \mathbb{L}},\{\rho _{(j, l)}\coloneqq \tau_j\otimes \omega_l\}_{(j, l)\in \mathbb{J}\bigtimes  \mathbb{L}}) $ be the  tensor product of weak frames  $( \{x_j\}_{j \in \mathbb{J}}, \{\tau_j\}_{j \in \mathbb{J}}) $  for  $ \mathcal{H},$ and $( \{y_l\}_{l \in \mathbb{L}}, \{\omega_l\}_{l \in \mathbb{L}} )$ for $ \mathcal{H}_1.$  Then  $ S_{z, \rho}=S_{x, \tau}\otimes S_{y, \omega}.$ If  $( \{x_j\}_{j \in \mathbb{J}}, \{\tau_j\}_{j \in \mathbb{J}}) $ and $ (\{y_l\}_{l \in \mathbb{L}},  \{\omega_l\}_{l \in \mathbb{L}} )$ are Parseval weak, then $(\{z_{(j, l)}\}_{(j, l)\in \mathbb{J}\bigtimes  \mathbb{L}} ,\{\rho_{(j,l)}\}_{(j,l)\in \mathbb{J}\bigtimes \mathbb{L}})$ is Parseval weak.
\end{proposition}
\begin{proof}
For all elementary tensors $ h\otimes g \in  \mathcal{H}\otimes\mathcal{H}_1,$ 
\begin{align*}
(S_{x, \tau}\otimes S_{y, \omega})(h\otimes g)&=S_{x, \tau}h\otimes S_{y, \omega}g=\left(\sum_{j\in \mathbb{J}}\langle h, x_j\rangle \tau_j\right) \otimes \left(\sum_{k\in \mathbb{L}}\langle g, y_k\rangle \omega_k\right)\\
&=\sum_{j\in \mathbb{J}}\sum_{k\in \mathbb{L}}\langle h, x_j\rangle\langle g, y_k\rangle (\tau_j\otimes \omega_k)=\sum_{(j,k)\in \mathbb{J}\bigtimes \mathbb{L}}\langle h\otimes g, x_j\otimes y_k\rangle(\tau_j\otimes \omega_k)=S_{z,\rho}(h\otimes g).
\end{align*}	
\end{proof}

\section{Extension of homomorphism-valued frames and bases }\label{EXTENSIONOFHOMOMORPHISM-VALUEDFRAMESFORHILBERTC*-MODULES}
Kaplansky \cite{KAPLANSKYORIGINAL} was the first to introduce inner products on modules which take values in commutative C*-algebras. Paschke \cite{PASCHKE1} studied these spaces without assuming commutativity.
 
Books  \cite{MANUILOV1, LANCE1, WEGGEOLSEN1} contain Hilbert C*-modules. 
The definition of frames for Hilbert C*-modules (over unital C*-algebra) for the first time, appears in 2002, in the paper of Frank and Larson \cite{FRANKLARSON1}.
 
 \begin{definition}\cite{FRANKLARSON1}
A set $\{x_j\}_{j\in \mathbb{J}} $ in $ \mathscr{E}$  is said to be a frame for $ \mathscr{E}$ if there are real $ a, b >0$ such that
$$a\langle x, x  \rangle \leq \sum\limits_{j\in \mathbb{J}}\langle x, x_j\rangle\langle x_j, x\rangle \leq  b \langle x, x  \rangle, \quad\forall x  \in  \mathscr{E}$$
where the sum converges in the norm of $ \mathscr{A}.$ We say  $\{x_j\}_{j\in \mathbb{J}} $ is Bessel if  $ \sum_{j\in \mathbb{J}}\langle x, x_j\rangle\langle x_j, x\rangle \leq  b \langle x, x  \rangle, \forall x  \in  \mathscr{E} .$ 
\end{definition}
\begin{definition}\label{HMDEFINITION1}
Let  $\mathscr{E}, \mathscr{E}_0 $  be Hilbert C*-modules over a  unital C*-algebra $ \mathscr{A}$. Define 
$ L_j : \mathscr{E}_0 \ni y \mapsto e_j\otimes y \in   \mathscr{H}_\mathscr{A} \otimes \mathscr{E}_0$,  where $\{e_j\}_{j \in \mathbb{J}} $ is  the standard orthonormal basis for $ \mathscr{H}_\mathscr{A}$,  for each $ j \in \mathbb{J}$. A collection $ \{A_j\}_{j \in \mathbb{J}} $  in $ \operatorname{Hom}^*_\mathscr{A}(\mathscr{E}, \mathscr{E}_0)$ is said to be a homomorphism-valued frame (we write (hvf))  in $ \operatorname{Hom}^*_\mathscr{A}(\mathscr{E}, \mathscr{E}_0)$  with respect to a collection  $ \{\Psi_j\}_{j \in \mathbb{J}}  $ in $ \operatorname{Hom}^*_\mathscr{A}(\mathscr{E}, \mathscr{E}_0)$ if
\begin{enumerate}[\upshape(i)]
\item the series $S_{A, \Psi}\coloneqq \sum_{j\in \mathbb{J}} \Psi_j^*A_j$ (frame homomorphism)  converges in the strict topology on $ \operatorname{End}_\mathscr{A}^*(\mathscr{E})$ to   an  adjointable (hence bounded homomorphism) positive  invertible homomorphism.  
\item both $ \theta_A \coloneqq\sum_{j\in \mathbb{J}} L_jA_j$, $\theta_\Psi \coloneqq \sum_{j\in \mathbb{J}} L_j\Psi_j$ (analysis homomorphisms) converge in the strict topology on $ \operatorname{Hom}_\mathscr{A}^*(\mathscr{E},\mathscr{H}_\mathscr{A} \otimes \mathscr{E}_0 )$ to    adjointable   homomorphisms.
\end{enumerate}
We call, real $ \alpha, \beta  >0 $ satisfying $\alpha I_\mathscr{E} \leq S_{A, \Psi}\ \leq \beta I_\mathscr{E} $ are called as lower and upper frame bounds, respectively.
Let $ a= \sup\{\alpha : \alpha I_\mathscr{E} \leq S_{A, \Psi}\}$, $ b= \inf\{\beta :  S_{A, \Psi}\leq \beta I_\mathscr{E} \}$.   We term $ a$ as the optimal lower frame bound, $ b$ as the optimal upper frame bound.  If $ a=b$, then the frame is called as tight frame or exact frame.  A tight frame is said to be Parseval if $ a=1$. 

For fixed $ \mathbb{J}$, $\mathscr{E}, \mathscr{E}_0, $ and $ \{\Psi_j \}_{j \in \mathbb{J}}$,  the set of all homomorphism-valued frames in  $ \operatorname{Hom}^*_\mathscr{A}(\mathscr{E}, \mathscr{E}_0)$ with respect to collection  $ \{\Psi_j \}_{j \in \mathbb{J}}$ is denoted by $ \mathscr{F}_\Psi.$
 
If we do not  demand the invertibility condition in \text{\upshape(i)}, then we say that $  \{A_j \}_{j \in \mathbb{J}}$  is Bessel w.r.t. $ \{\Psi_j \}_{j \in \mathbb{J}}$. 
 
Like Hilbert space situation, we write $ (\{A_j \}_{j \in \mathbb{J}},  \{\Psi_j \}_{j \in \mathbb{J}})$ is (hvf)/Bessel.
\end{definition}

\begin{proposition}
If  $(\{A_j\}_{j\in \mathbb{J}}, \{\Psi_j\}_{j\in \mathbb{J}}) $ is  Bessel   in $ \operatorname{Hom}^*_\mathscr{A}(\mathscr{E}, \mathscr{E}_0)$, then there exists a $ B \in  \operatorname{Hom}^*_\mathscr{A}(\mathscr{E}, \mathscr{E}_0)$ such that $(\{A_j\}_{j\in \mathbb{J}}\cup\{B\}, \{\Psi_j\}_{j\in \mathbb{J}}\cup\{B\}) $ is a tight (hvf). In particular, if  $(\{A_j\}_{j\in \mathbb{J}}, \{\Psi_j\}_{j\in \mathbb{J}}) $ is a  (hvf)   in $ \operatorname{Hom}^*_\mathscr{A}(\mathscr{E}, \mathscr{E}_0)$, then there exists a $ B \in  \operatorname{Hom}^*_\mathscr{A}(\mathscr{E}, \mathscr{E}_0)$ such that $(\{A_j\}_{j\in \mathbb{J}}\cup\{B\}, \{\Psi_j\}_{j\in \mathbb{J}}\cup\{B\}) $ is a tight (hvf).
\end{proposition}

\begin{definition}
Let  $ \{A_j\}_{j \in \mathbb{J}}$ in $ \operatorname{Hom}^*_\mathscr{A}(\mathscr{E}, \mathscr{E}_0).$ We say 
\begin{enumerate}[\upshape(i)]
\item $ \{A_j\}_{j \in \mathbb{J}}$ is   an orthogonal set   in   $ \operatorname{Hom}^*_\mathscr{A}(\mathscr{E}, \mathscr{E}_0)$ if 
$$\langle A_j^*y, A_k^*z\rangle=0 , ~\forall y, z \in \mathscr{E}_0, ~\forall j, k \in \mathbb{J}, j\neq k.$$
\item $ \{A_j\}_{j \in \mathbb{J}}$ is   an orthonormal set  in   $ \operatorname{Hom}^*_\mathscr{A}(\mathscr{E}, \mathscr{E}_0)$ if 
$$\langle A_j^*y, A_k^*z\rangle=\delta_{j,k}\langle y, z\rangle  , ~\forall y, z \in \mathscr{E}_0, ~\forall j, k \in \mathbb{J}~ \text{and } ~ \sum\limits_{j \in \mathbb{J}}\langle A_jx,A_jx\rangle\leq\langle x, x\rangle, ~ \forall x \in \mathscr{E}.$$  
\item $ \{A_j\}_{j \in \mathbb{J}}$ is an orthonormal basis in  $ \operatorname{Hom}^*_\mathscr{A}(\mathscr{E}, \mathscr{E}_0)$ if 
$$\langle A_j^*y, A_k^*z\rangle=\delta_{j,k}\langle y, z\rangle  , ~\forall y, z \in \mathscr{E}_0, ~\forall j, k \in \mathbb{J} ~ \text{and } ~ \sum\limits_{j \in \mathbb{J}}\langle A_jx,A_jx\rangle=\langle x,x\rangle, ~ \forall x \in \mathscr{E}.$$
 \end{enumerate}
 \end{definition}
 \begin{theorem}
 \begin{enumerate}[\upshape(i)]
 \item Every orthonormal set $Y$ in  $ \operatorname{Hom}^*_\mathscr{A}(\mathscr{E}, \mathscr{E}_0)$ is contained in a maximal orthonormal set.
 \item If $ \operatorname{Hom}^*_\mathscr{A}(\mathscr{E}, \mathscr{E}_0)$ has a homomorphism    $T$ such that $TT^*$ is bounded invertible, then $ \operatorname{Hom}^*_\mathscr{A}(\mathscr{E}, \mathscr{E}_0)$ has a maximal orthonormal set.
 \end{enumerate}
 \end{theorem}
 \begin{lemma}\label{FIRSTIMPLIESSECONDLEMMAMODULE}
 If $ \{A_j\}_{j \in \mathbb{J}}$ in $ \operatorname{Hom}^*_\mathscr{A}(\mathscr{E}, \mathscr{E}_0)$  satisfies $\langle A_j^*y, A_k^*z\rangle=\delta_{j,k}\langle y, z\rangle, \forall y, z \in \mathscr{E}_0, \forall j, k \in \mathbb{J}$, and $\sum_{j \in \mathbb{J}}\langle  A_jx, A_jx\rangle $ converges for all $ x \in \mathscr{E}$, then $\sum_{j \in \mathbb{J}}\langle A_jx, A_jx \rangle\leq \langle x,x \rangle , \forall x \in \mathscr{E} $.
 \end{lemma}
 \begin{proof}
  For  $ x \in \mathscr{E}$ and  each finite subset $ \mathbb{S} \subseteq \mathbb{J},$
  \begin{align*}
  \left\| \sum_{j\in \mathbb{S}} A_j^*A_jx\right\|^2
  &= \left\|\left\langle  \sum_{j\in \mathbb{S}} A_j^*A_jx, \sum_{k\in \mathbb{S}} A_k^*A_kx\right\rangle\right\|=\left\|\sum_{j\in \mathbb{S}}  \langle A_jx, A_jx\rangle \right \|, 
  \end{align*}

  which is convergent. Therefore $ \sum_{j\in \mathbb{J}} A_j^*A_jx$ exists. Then 
  \begin{align*}
  0\leq \left\langle x-\sum_{j\in \mathbb{J}} A_j^*A_jx ,x-\sum_{k\in \mathbb{J}} A_k^*A_kx \right\rangle  
  &=\langle x, x \rangle-2\sum_{j\in \mathbb{J}}\langle A_jx, A_jx \rangle +\sum_{j\in \mathbb{J}}\langle A_jx, A_jx \rangle\\
  &=\langle x, x \rangle-\sum_{j\in \mathbb{J}}\langle A_jx, A_jx \rangle,
  \end{align*}
  $ \Rightarrow\sum_{j\in \mathbb{J}}\langle A_jx, A_jx \rangle  \leq \langle x, x \rangle,\forall x \in \mathscr{E} .$ 
 \end{proof}
 \begin{proposition}
 Let $ \{A_j\}_{j \in \mathbb{J}}$  be   orthogonal   in $ \operatorname{Hom}^*_\mathscr{A}(\mathscr{E}, \mathscr{E}_0)$ such that $\sum_{j \in \mathbb{S}}\langle  A_jx, A_jx\rangle $ converges for all $ x \in \mathscr{E}$. If  $ A_jA_j^*$s are bounded invertible, then $ \{U_j \coloneqq (A_jA_j^*)^{-1/2}A_j\}_{j \in \mathbb{J}}$ is  orthonormal   in $ \operatorname{Hom}^*_\mathscr{A}(\mathscr{E}, \mathscr{E}_0)$ and $\overline{\operatorname{span}}_{\operatorname{End}^*_\mathscr{A}(\mathscr{E}_0)}\{U_j\}_{j \in \mathbb{J}}=\overline{\operatorname{span}}_{ \operatorname{End}^*_\mathscr{A}(\mathscr{E}_0)}\{A_j\}_{j \in \mathbb{J}}$.
 \end{proposition} 
 \begin{proof}
 Like Hilbert space situation, but one has to use Lemma \ref{FIRSTIMPLIESSECONDLEMMAMODULE} to get the inequality in the definition of orthonormality.
\end{proof}
 \begin{theorem}
 \begin{enumerate}[\upshape(i)]
 \item If  $ \{A_n\}_{n=1}^m $ is  orthogonal  in $ \operatorname{Hom}^*_\mathscr{A}(\mathscr{E}, \mathscr{E}_0)$, then 
 $$\left\langle \sum_{n=1}^{m}A_n^*y_n, \sum_{k=1}^{m}A_k^*y_k \right\rangle  =\sum_{n=1}^{m}\langle A_n^*y_n,A_n^*y_n \rangle, ~\forall y_1,...,y_n \in \mathscr{E}_0 .$$
In particular,
$\left\| \sum_{n=1}^{m}A_n^*y_n \right\|^2  =\left\|\sum_{n=1}^{m}\langle A_n^*y_n,A_n^*y_n \rangle\right\|, \forall y_1,...,y_n \in \mathscr{E}_0$
 and  $ \|A_1+\cdots+A_m\|^2\leq\|A_1\|^2+\cdots+\|A_m\|^2.$
\item If  $ \{A_n\}_{n=1}^m $ is  orthonormal  in  $ \operatorname{Hom}^*_\mathscr{A}(\mathscr{E}, \mathscr{E}_0)$, then $\left\langle \sum_{n=1}^{m}A_n^*y_n, \sum_{k=1}^{m}A_k^*y_k \right\rangle  =\sum_{n=1}^{m}\langle y_n,y_n \rangle, \forall y_1,...,y_n \in \mathscr{E}_0$. In particular,  $\| \sum_{n=1}^{m}A_n^*y_n \|^2 =\|\sum_{n=1}^{m}\langle y_n, y_n \rangle\| , \forall y_1,...,y_n \in \mathscr{E}_0$ and  $ \|A_1+\cdots+A_m\|^2= m.$
\item If $ \{A_j\}_{j \in \mathbb{J}} $ is  orthogonal   in  $ \operatorname{Hom}^*_\mathscr{A}(\mathscr{E}, \mathscr{E}_0)$ such that  $A_jA_j^* $ is invertible, then $ \{A_j\}_{j \in \mathbb{J}} $ is linearly independent over $ \mathscr{A}$ as well as over $ \operatorname{End}^*_\mathscr{A}(\mathscr{E}_0)$. In particular, if $ \{A_j\}_{j \in \mathbb{J}} $ is  orthonormal   in $ \operatorname{Hom}^*_\mathscr{A}(\mathscr{E}, \mathscr{E}_0)$, then it is linearly independent over $ \mathscr{A}$ as well as over $ \operatorname{End}^*_\mathscr{A}(\mathscr{E}_0)$.
 \end{enumerate}
  \end{theorem}
 \begin{theorem}\label{RIESZFISCHERHOMOMORPHISMVERSION}
 Let $ \{A_j\}_{j \in \mathbb{J}} $ be  orthonormal  in $ \operatorname{Hom}^*_\mathscr{A}(\mathscr{E}, \mathscr{E}_0)$,  $ \{U_j\}_{j \in \mathbb{J}} $ be  in $\operatorname{End}^*_\mathscr{A}(\mathscr{E}_0)$ and $ y \in \mathscr{E}_0$. Then
 $$\sum_{j\in\mathbb{J}}A_j^*U_jy ~\text{converges in} ~  \mathscr{E} ~ \text{if and only if}~ \sum_{j\in\mathbb{J}}\langle U_jy, U_jy \rangle  ~\text{converges in }~\mathscr{A}. $$ 
 \end{theorem} 
 \begin{corollary}
 Let $ \{A_j\}_{j \in \mathbb{J}} $ be  orthonormal  in $ \operatorname{Hom}^*_\mathscr{A}(\mathscr{E}, \mathscr{E}_0)$,  $ \{c_j\}_{j \in \mathbb{J}} $ be a sequence with elements from $\mathscr{A} $  and $ y \in \mathscr{E}_0$. Then
$$\sum_{j\in\mathbb{J}}c_jA_j^*y ~\text{converges in} ~  \mathscr{E} ~ \text{if and only if}~\{c_j\langle y, y \rangle^\frac{1}{2}\}_{j \in \mathbb{J}} \in  \mathscr{H}_\mathscr{A}.$$ 	
In particular, if $y \in \mathscr{E}_0 $ is such that $\langle y, y \rangle^{1/2} $ is invertible (in $\mathscr{A}$), then  $\sum_{j\in\mathbb{J}}c_jA_j^*y $ converges in $ \mathscr{E}$ if and only if  $\{c_j\}_{j \in \mathbb{J}}  \in  \mathscr{H}_\mathscr{A}.$ 	
\end{corollary}
\begin{definition}
Let  $ \{A_j\}_{j \in \mathbb{J}},\{\Psi_j\}_{j \in \mathbb{J}} $ be in $ \operatorname{Hom}^*_\mathscr{A}(\mathscr{E}, \mathscr{E}_0)$.
We say 
\begin{enumerate}[\upshape(i)]
\item $ \{A_j\}_{j \in \mathbb{J}}$ is an orthonormal set (resp. basis)  w.r.t. $\{\Psi_j\}_{j \in \mathbb{J}} $ if  $ \{A_j\}_{j \in \mathbb{J}}$ or   $\{\Psi_j\}_{j \in \mathbb{J}} $ is an orthonormal set (resp. basis) in $ \operatorname{Hom}^*_\mathscr{A}(\mathscr{E}, \mathscr{E}_0)$, say $ \{A_j\}_{j \in \mathbb{J}}$ is an orthonormal set  (resp. basis) in $ \operatorname{Hom}^*_\mathscr{A}(\mathscr{E}, \mathscr{E}_0)$,   and there exists a sequence  $\{c_j\}_{j \in \mathbb{J}} $  of positive invertible elements in the center of  $ \mathscr{A}$ such that  $ 0<\inf\{\|c_j\|\}_{j \in \mathbb{J}}\leq \sup\{\|c_j\|\}_{j \in \mathbb{J}}<\infty$   and $ \Psi_j=c_jA_j, \forall j \in \mathbb{J}.$ We write $ (\{A_j\}_{j \in \mathbb{J}}, \{\Psi_j\}_{j \in \mathbb{J}})$ is an orthonormal set (resp. basis).
\item $ \{A_j\}_{j \in \mathbb{J}}$ is a Riesz basis  w.r.t. $\{\Psi_j\}_{j \in \mathbb{J}} $ if there exists an orthonormal basis $ \{F_j\}_{j \in \mathbb{J}}$ in $ \operatorname{Hom}^*_\mathscr{A}(\mathscr{E}, \mathscr{E}_0)$ and   invertible $U,V \in \operatorname{End}^*_\mathscr{A}(\mathscr{E})$ with $ V^*U$ is positive  such that $ A_j=F_jU, \Psi_j=F_jV,  \forall j \in \mathbb{J}.$ We write $ (\{A_j\}_{j \in \mathbb{J}}, \{\Psi_j\}_{j \in \mathbb{J}})$ is a Riesz basis.
 \end{enumerate}
\end{definition}
Since $ c_j$'s are invertible, previous definition is symmetric.
\begin{theorem}\label{GBEOVHM}
Let $( \{A_j\}_{j \in \mathbb{J}},\{\Psi_j=c_jA_j\}_{j \in \mathbb{J}}) $  be   orthonormal  in $ \operatorname{Hom}^*_\mathscr{A}(\mathscr{E}, \mathscr{E}_0)$ such that $c_j\leq 2e, \forall j \in \mathbb{J}$ ($ e$ is identity of $ \mathscr{A}$). Then 
\begin{enumerate}[\upshape(i)]
 \item  $ \sum_{j\in\mathbb{J}}(2e-c_j)\langle A_jx, \Psi_jx\rangle\leq \langle x, x \rangle , \forall x \in \mathscr{E},$ i.e.,  $ \sum_{j\in \mathbb{J}}(2e-c_j)\Psi_j^*A_j\leq I_\mathscr{E}. $
 \item For $ x \in \mathscr{E},$ 
  \begin{align*}
  x= \sum \limits_{j\in \mathbb{J}}\Psi_j^*A_jx \iff \sum\limits_{j\in\mathbb{J}}(2e-c_j)\langle A_jx, \Psi_jx\rangle = \langle x, x\rangle \iff  \sum\limits_{j\in\mathbb{J}}c_j^2\langle A_jx, A_jx\rangle  = \langle x, x\rangle.
  \end{align*}
  If $ c_j \leq e , \forall j ,$  then $x= \sum_{j\in \mathbb{J}}\Psi_j^*A_jx \iff (e-c_j)A_jx=0 , \forall j \in \mathbb{J} \iff (e-c_j)A_j^*A_jx\perp x, \forall j \in \mathbb{J}.$
\end{enumerate}
 \end{theorem} 
\begin{proof}
\begin{enumerate}[\upshape(i)]
\item For  $ x \in \mathscr{E}$ and  each finite subset $ \mathbb{S} \subseteq \mathbb{J},$
\begin{align*}
\left\| \sum_{j\in \mathbb{S}} \Psi_j^*A_jx\right\|^2
&= \left\|\left\langle  \sum_{j\in \mathbb{S}} c_jA_j^*A_jx, \sum_{k\in \mathbb{S}} c_kA_k^*A_kx\right\rangle\right\|=\left\|\sum_{j\in \mathbb{S}} c_j^2 \langle A_jx, A_jx\rangle \right \|\\
&\leq \left(\sup\{\|c_j\|^2\}_{j \in \mathbb{J}}\right)\left\|\sum\limits_{j \in \mathbb{S}}\langle  A_jx, A_jx\rangle \right\| , ~ \text{which is convergent.}
\end{align*}
Therefore $ \sum_{j\in \mathbb{J}} \Psi_j^*A_jx$ exists and similarly $\sum_{j\in \mathbb{J}}(2e-c_j)\Psi_j^*A_jx $ also exists. Then 
\begin{align*}
0&\leq \left\langle x-\sum_{j\in \mathbb{J}} \Psi_j^*A_jx ,x-\sum_{k\in \mathbb{J}} \Psi_k^*A_kx \right\rangle  = \left\langle x-\sum_{j\in \mathbb{J}}c_jA_j^*A_jx, x-\sum_{k\in \mathbb{J}}c_kA_k^*A_kx\right\rangle \\
&=\langle x, x \rangle-2\sum_{j\in \mathbb{J}}c_j\langle A_jx, A_jx \rangle +\sum_{j\in \mathbb{J}}c^2_j\langle A_jx, A_jx \rangle=\langle x, x \rangle-\sum_{j\in \mathbb{J}}(2c_j-c_j^2)\langle A_jx, A_jx \rangle
\end{align*}
$ \Rightarrow\sum_{j\in \mathbb{J}}(2c_j-c_j^2)\langle A_jx, A_jx \rangle= \sum_{j\in\mathbb{J}}(2e-c_j)\langle A_jx,\Psi_jx\rangle  \leq \langle x, x \rangle.$ 
\item For  this we  use the fact that  the set of all positive elements in a C*-algebra forms a cone.
\end{enumerate}
\end{proof}
\begin{corollary}
  Let $( \{A_j\}_{j \in \mathbb{J}},\{\Psi_j=c_jA_j\}_{j \in \mathbb{J}}) $  be an orthonormal basis in   $ \operatorname{Hom}^*_\mathscr{A}(\mathscr{E}, \mathscr{E}_0)$. Then
 $$ \frac{1}{\sup\{\|c_j\|\}_{j \in \mathbb{J}}}\sum_{j\in \mathbb{J}}\Psi_j^*A_j\leq I_\mathscr{E}\leq \frac{1}{\inf\{\|c_j\|\}_{j \in \mathbb{J}}}\sum_{j\in \mathbb{J}}\Psi_j^*A_j.$$
 \end{corollary}
 \begin{proof}
 $ \frac{1}{\sup\{\|c_j\|\}_{j \in \mathbb{J}}}\sum_{j\in \mathbb{J}}\Psi_j^*A_j=\frac{1}{\sup\{\|c_j\|\}_{j \in \mathbb{J}}}\sum_{j\in \mathbb{J}}c_jA_j^*A_j\leq \sum_{j\in \mathbb{J}}A_j^*A_j= I_\mathscr{E}\leq\frac{1}{\inf\{\|c_j\|\}_{j \in \mathbb{J}}}\sum_{j\in \mathbb{J}}c_jA_j^*A_j= \frac{1}{\inf\{\|c_j\|\}_{j \in \mathbb{J}}}\sum_{j\in \mathbb{J}}\Psi_j^*A_j.$
 \end{proof}
 \begin{corollary}
 If  $( \{A_j\}_{j \in \mathbb{J}},\{\Psi_j=c_jA_j\}_{j \in \mathbb{J}}) $  is an orthonormal basis in $ \operatorname{Hom}^*_\mathscr{A}(\mathscr{E}, \mathscr{E}_0)$, then $\inf\{\|c_j\|\}_{j \in \mathbb{J}} \leq \|\sum_{j\in \mathbb{J}}\Psi_j^*A_j\| \leq \sup\{\|c_j\|\}_{j \in \mathbb{J}}.$
 \end{corollary}
\begin{corollary}
If $( \{A_j\}_{j \in \mathbb{J}},\{\Psi_j=c_jA_j\}_{j \in \mathbb{J}}) $  is  orthonormal  in $ \operatorname{Hom}^*_\mathscr{A}(\mathscr{E}, \mathscr{E}_0)$ such that $c_j\leq 2e, \forall j \in \mathbb{J},$ then $ \|\sum_{j\in\mathbb{J}}(2e-c_j)\langle A_jx, \Psi_jx\rangle\|\leq \| x\|^2 , \forall x \in \mathscr{E}$ and   $\| \sum_{j\in \mathbb{J}}(2e-c_j)\Psi_j^*A_j\|\leq 1. $	
\end{corollary}
\begin{theorem}
If $( \{A_j\}_{j \in \mathbb{J}},\{\Psi_j=c_jA_j\}_{j \in \mathbb{J}}) $  is  orthonormal  in $ \operatorname{Hom}^*_\mathscr{A}(\mathscr{E}, \mathscr{E}_0)$ with $c_j\leq 2e, \forall j \in \mathbb{J},$ then for each $ x \in  \mathscr{E}$, the set 
$$ Y_x\coloneqq\bigcup\limits_{n=1}^\infty\left\{A_j: (2e-c_j)\langle A_jx, \Psi_jx \rangle > \frac{1}{n}\langle x, x \rangle, j\in \mathbb{J} \right\}$$
is either  finite or countable.
\end{theorem}
\begin{proof}
For $ n \in \mathbb{N}$, define 
$$ Y_{n,x}\coloneqq\left\{x_j: (2e-c_j)\langle A_jx, \Psi_jx \rangle  > \frac{1}{n}\langle x, x \rangle, j\in \mathbb{J} \right\}.$$	

Suppose, for some $n$, $ Y_{n,x}$ has more than $n-1$ elements, say $x_1,...,x_n$. Then $ \sum_{j=1 }^n(2e-c_j)\langle A_jx, \Psi_jx \rangle  > n\frac{1}{n}\langle x, x \rangle= \langle x, x \rangle$. From Theorem \ref{GBEOVHM},  $ \sum_{j\in\mathbb{J}}(2e-c_j)\langle x, x_j\rangle \langle \tau_j, x \rangle\leq \langle x, x \rangle $. This gives $ \langle x, x \rangle< \langle x, x \rangle $ which is  impossible. Therefore $ \operatorname{Card}(Y_{n,x})\leq n-1$ and hence  
$Y_x=\cup_{n=1}^\infty Y_{n,x}$  is finite or countable.
\end{proof}
\begin{proposition}
\begin{enumerate}[\upshape(i)]
\item If $ (\{A_j\}_{j \in \mathbb{J}},\{\Psi_j\}_{j \in \mathbb{J}}) $ is an orthonormal basis in  $ \operatorname{Hom}^*_\mathscr{A}(\mathscr{E}, \mathscr{E}_0)$, then it is a Riesz basis.
\item If $ (\{A_j=F_jU\}_{j \in \mathbb{J}},\{\Psi_j=F_jV\}_{j \in \mathbb{J}}) $ is a Riesz  basis in  $ \operatorname{Hom}^*_\mathscr{A}(\mathscr{E}, \mathscr{E}_0)$, then it  is a (hvf)  with optimal frame bounds $ \|(V^*U)^{-1}\|^{-1}$ and  $\|V^*U\| $.
\end{enumerate}
\end{proposition} 
\begin{proof}
\begin{enumerate}[\upshape(i)]
\item We may assume  $ \{A_j\}_{j \in \mathbb{J}}$ is an orthonormal basis.  Then there exists a sequence  $\{c_j\}_{j \in \mathbb{J}} $  of   positive invertible elements in the center of $ \mathscr{A}$ such that  $ 0<\inf\{\|c_j\|\}_{j \in \mathbb{J}}\leq \sup\{\|c_j\|\}_{j \in \mathbb{J}}<\infty$   and $ \Psi_j=c_jA_j, \forall j \in \mathbb{J}.$ Define $ F_j\coloneqq A_j, \forall j \in \mathbb{J},  U\coloneqq I_\mathscr{E}$ and  $ V \coloneqq \sum_{j\in \mathbb{J}}c_jA_j^*A_j.$ Since $c_j$'s are in the center  of $ \mathscr{A}$ and $ A_j$'s are orthonormal,  $ V$ is well-defined bounded adjointable homomorphism (with bound $ \sup\{\|c_j\|\}_{j \in \mathbb{J}}$). Then $ F_jU=A_j, F_jV=\sum_{k\in \mathbb{J}}c_kA_jA_k^*A_k=c_jA_j, \forall j \in \mathbb{J}.$  Now positivity of  $c_j $'s imply   $ V$ is positive invertible, whose inverse is $\sum_{j\in \mathbb{J}}c_j^{-1}A_j^*A_j.$ Note that $ V^*U=V\geq 0.$
 \item For every finite subset $\mathbb{S} $ of $\mathbb{J}$ and $ x \in \mathscr{E},$ we get  $ \|\sum_{j\in \mathbb{S}}L_jF_jx\|^2=\|\sum_{j\in \mathbb{S}}\langle F_jx, F_jx\rangle\|$ and this converges to $ \|x\|^2.$ Therefore $ \theta_A=\theta_FU$ exists and is bounded adjointable. Also, $ \|\sum_{j\in \mathbb{S}}L_j\Psi_jx\|^2=\|\sum_{j\in \mathbb{S}} c_j^2\langle A_jx,A_jx\rangle\| \leq (\sup\{\|c_j\|^2\}_{j \in \mathbb{J}})\|\sum_{j\in \mathbb{S}}\langle A_jx, A_jx\rangle \|  $ and this converges. Therefore $ \theta_\Psi=\theta_FV$ is also bounded  adjointable. Now $ S_{A,\Psi}= \sum_{j\in\mathbb{J}}V^*F_j^*F_jU=V^*U$ which is positive invertible. 
\end{enumerate}
 \end{proof} 
\begin{theorem}
 Let $ (\{A_j=F_jU\}_{j \in \mathbb{J}},\{\Psi_j=F_jV\}_{j \in \mathbb{J}}) $ be  a Riesz  basis in  $ \operatorname{Hom}^*_\mathscr{A}(\mathscr{E}, \mathscr{E}_0).$
\begin{enumerate}[\upshape(i)]
\item There exist unique  $ \{B_j\}_{j\in\mathbb{J}}, \{\Phi_j\}_{j\in\mathbb{J}}$ in $\operatorname{Hom}^*_\mathscr{A}(\mathscr{E}, \mathscr{E}_0) $ such that 
$$ I_\mathscr{E}= \sum\limits_{j\in\mathbb{J}}B_j^*A_j = \sum\limits_{j\in\mathbb{J}}\Phi_j^*\Psi_j, $$ and $( \{B_j\}_{j\in\mathbb{J}}, \{\Phi_j\}_{j\in\mathbb{J}})$ is Riesz (hvf).
\item $\{x \in \mathscr{E}: A_jx=0, \forall j \in \mathbb{J}\} =\{0\}= \{x \in \mathscr{E}: \Psi_jx=0, \forall j \in \mathbb{J}\} $. If $ VU^*\geq0,$ then there are real $ a, b >0$ such that for every finite subset $ \mathbb{S}$ of $ \mathbb{J}$,
$$ a\sum\limits_{j\in \mathbb{S}}c_jc_j^*\langle x, x \rangle \leq \left\langle \sum\limits_{j\in\mathbb{S}}c_jA^*_jx ,\sum\limits_{k\in\mathbb{S}}c_k\Psi_k^*x  \right \rangle \leq b\sum\limits_{j\in \mathbb{S}}c_jc_j^*\langle x, x \rangle, ~\forall x \in \mathscr{E},\forall c_j\in \mathscr{A}, \forall j\in \mathbb{S}.$$ 
\end{enumerate}
\end{theorem}
\begin{proposition}
Let $ (\{A_j\}_{j\in \mathbb{J}}, \{\Psi_j\}_{j\in \mathbb{J}}) $  be a  (hvf) in    $ \operatorname{Hom}^*_\mathscr{A}(\mathscr{E}, \mathscr{E}_0) $. Then the bounded adjointable left-inverses of 
\begin{enumerate}[\upshape(i)]
\item $ \theta_A$ are precisely   $S_{A,\Psi}^{-1}\theta_\Psi^*+U(I_{\mathscr{H}_\mathscr{A}\otimes\mathscr{E}_0}-\theta_AS_{A,\Psi}^{-1}\theta_\Psi^*)$, where $U\in \operatorname{Hom}^*_\mathscr{A}(\mathscr{H}_\mathscr{A}\otimes\mathscr{E}_0, \mathscr{E})$.
\item $ \theta_\Psi$ are precisely  $S_{A,\Psi}^{-1}\theta_A^*+V(I_{\mathscr{H}_\mathscr{A}\otimes\mathscr{E}_0}-\theta_\Psi S_{A,\Psi}^{-1}\theta_A^*)$, where $V\in \operatorname{Hom}^*_\mathscr{A}(\mathscr{H}_\mathscr{A}\otimes\mathscr{E}_0, \mathscr{E})$.
\end{enumerate}	
\end{proposition}
It is well-known that,  in contrast with Hilbert spaces, a closed subspace of a Hilbert C*-module need not be orthogonally complemented. With regard to this,  we are going to use the following theorem  in the next result.
\begin{theorem}(cf. \cite{MANUILOV1})\label{MANUILOV2}
If   $T \in \operatorname{Hom}^*_\mathscr{A}(\mathscr{E}, \mathscr{E}_0) $ is such that $T(\mathscr{E})$ is closed, then $ \operatorname{Ker}(T)$ (resp. $ T(\mathscr{E})$) is an orthogonally complementable submodule in $ \mathscr{E} $ (resp. $ \mathscr{E}_0$).
\end{theorem}

\begin{proposition}\label{HMOP}
For every $ \{A_j\}_{j \in \mathbb{J}}  \in \mathscr{F}_\Psi$,
 \begin{enumerate}[ \upshape (i)]
\item $ \theta_A^* \theta_A =  \sum_{j\in \mathbb{J}} A_j^*A_j$.
\item $ S_{A, \Psi} = \theta_\Psi^*\theta_A=\theta_A^*\theta_\Psi =S_{\Psi,A}.$
\item $(\{A_j\}_{j \in \mathbb{J}}, \{\Psi_j \}_ {j \in \mathbb{J}})$ is Parseval  if and only if $  \theta_\Psi^*\theta_A =I_\mathscr{E}.$ 
\item  $(\{A_j\}_{j \in \mathbb{J}}, \{\Psi_j \}_ {j \in \mathbb{J}})$ is Parseval  if and only if $ \theta_A\theta_\Psi^* $ is idempotent.
\item $  A_j=L_j^*\theta_A, \forall j\in \mathbb{J}.$
\item $\theta_AS_{A,\Psi}^{-1}\theta_\Psi^* $ is idempotent.
\item $\theta_A $ and $ \theta_\Psi$ are injective and  their  ranges are closed.
\item  $\theta_A^* $ and $ \theta_\Psi^*$ are surjective.
 \item $\operatorname{Ker}(\theta_A)$ and $\operatorname{ Ker}(\theta_\Psi)$ (resp. $\theta_A(\mathscr{E})$ and $\theta_\Psi(\mathscr{E})$) are orthogonally complementable submodules of $\mathscr{E}$ (resp. $\mathscr{H}_\mathscr{A}\otimes \mathscr{E}_0$).
\end{enumerate}
\end{proposition}
\begin{proof}
Proof is similar to the corresponding proposition  in `operator-valued frames'. For (ix), we have to use (vii)  and Theorem \ref{MANUILOV2}.
\end{proof}
Like we noticed  in Hilbert space setting, the previous proposition  says Definition \ref{HMDEFINITION1} is equivalent to 
\begin{definition}
A collection $ \{A_j\}_{j \in \mathbb{J}} $   in $ \operatorname{Hom}^*_\mathscr{A}(\mathscr{E}, \mathscr{E}_0)$ is said to be a (hvf)  w.r.t. $ \{\Psi_j\}_{j \in \mathbb{J}}  $ in $ \operatorname{Hom}^*_\mathscr{A}(\mathscr{E}, \mathscr{E}_0) $ if there exist $ a, b, c , d>0$ such that 
\begin{enumerate}[\upshape(i)]
\item $\sum_{j\in \mathbb{J}}\Psi_j^*A_j=\sum_{j\in \mathbb{J}}A_j^*\Psi_j,$	
\item $a\langle x, x \rangle \leq\sum_{j\in \mathbb{J}}\langle A_jx, \Psi_jx \rangle  \leq b\langle x, x \rangle, \forall x \in \mathscr{E},$	
\item $\|\sum_{j\in \mathbb{J}}\langle  A_jx, A_jx\rangle \| \leq c\|x\|^2, \forall x \in \mathscr{E}; \|\sum_{j\in \mathbb{J}} \langle \Psi_jx, \Psi_jx \rangle \|  \leq d\|x\|^2, \forall x \in \mathscr{E}.$
\end{enumerate}
\end{definition}
We call the idempotent homomorphism  $P_{A,\Psi}\coloneqq\theta_AS_{A,\Psi}^{-1}\theta_\Psi^* $ as the  \textit{frame idempotent}.
\begin{definition}
A (hvf)  $(\{A_j\}_{j\in \mathbb{J}}, \{\Psi_j\}_{j\in \mathbb{J}})$  in  $\operatorname{Hom}^*_\mathscr{A}(\mathscr{E}, \mathscr{E}_0)$ is said to be a Riesz (hvf)  if $ P_{A,\Psi}= I_{\mathscr{H}_\mathscr{A}}\otimes I_{\mathscr{E}_0}$. A Parseval and  Riesz (hvf) (i.e., $\theta_A\theta_\Psi^*=I_{\mathscr{H}_\mathscr{A}}\otimes I_{\mathscr{E}_0} $  and $\theta_\Psi^*\theta_A=I_\mathscr{E} $) is called as an orthonormal (hvf).
\end{definition}
\begin{proposition}
 \begin{enumerate}[\upshape(i)]
 \item If $ (\{A_j\}_{j \in \mathbb{J}},\{\Psi_j\}_{j \in \mathbb{J}}) $ is a Riesz basis  in  $ \operatorname{Hom}^*_\mathscr{A}(\mathscr{E}, \mathscr{E}_0)$, then it is a Riesz (hvf).
\item If $ (\{A_j\}_{j \in \mathbb{J}},\{\Psi_j\}_{j \in \mathbb{J}}) $ is an orthonormal   basis in  $ \operatorname{Hom}^*_\mathscr{A}(\mathscr{E}, \mathscr{E}_0)$, then it is a Riesz (hvf).
\end{enumerate}
\end{proposition} 
\begin{proposition}
A (hvf) $(\{A_j\}_{j\in \mathbb{J}}, \{\Psi_j\}_{j\in \mathbb{J}}) $ in $ \operatorname{Hom}^*_\mathscr{A}(\mathscr{E}, \mathscr{E}_0)$ is a Riesz (hvf) if and only if $\theta_A(\mathscr{E})=\mathscr{H}_\mathscr{A}\otimes \mathscr{E}_0  $ if and only if $\theta_\Psi(\mathscr{E})=\mathscr{H}_\mathscr{A}\otimes \mathscr{E}_0.$
\end{proposition}

\begin{proposition}
A (hvf) $ (\{A_j\}_{j\in \mathbb{J}}, \{\Psi_j\}_{j\in \mathbb{J}}) $ in  $ \operatorname{Hom}^*_\mathscr{A}(\mathscr{E}, \mathscr{E}_0)$ is an orthonormal (hvf)  if and only if it is a Parseval (hvf) and $ A_j\Psi_k^*=\delta_{j,k}I_{\mathscr{E}_0},\forall j, k \in \mathbb{J}$. 
\end{proposition}

 \begin{theorem}
 Let $(\{A_j\}_{j\in \mathbb{J}},\{\Psi_j\}_{j\in \mathbb{J}} )$ be  a Parseval (hvf) in $ \operatorname{Hom}^*_\mathscr{A}(\mathscr{E}, \mathscr{E}_0)$  such that $ \theta_A(\mathscr{E})=\theta_\Psi(\mathscr{E})$ and $ P_{A,\Psi}$ is projection. Then there exist a Hilbert C*-module $ \mathscr{E}_1$ which contains $ \mathscr{E}$ isometrically and  bounded  homomorphisms $B_j,\Phi_j:\mathscr{E}_1\rightarrow \mathscr{E}_0, \forall j \in \mathbb{J} $ such that $(\{B_j\}_{j\in \mathbb{J}},\{\Phi_j\}_{j\in \mathbb{J}})$ is an orthonormal (hvf) in $ \operatorname{Hom}^*_\mathscr{A}(\mathscr{E}_1, \mathscr{E}_0)$ and $B_j|_{\mathscr{E}}=  A_j,\Phi_j|_{\mathscr{E}}=\Psi_j, \forall j \in \mathbb{J}$. 
 \end{theorem}
 \begin{proof}
One has to ensure that $\theta_A(\mathscr{E}) $ is orthogonally complementable in  $\mathscr{H}_\mathscr{A}\otimes \mathscr{E}_0$. This comes from (ix) of Proposition \ref{HMOP}. Construction of $ \mathscr{E}_1 $ and other arguments are similar to the corresponding dilation result in Section \ref{MK}.
 \end{proof}
\begin{definition}
A (hvf)  $(\{B_j\}_{j\in \mathbb{J}}, \{\Phi_j\}_{j\in \mathbb{J}})$  in $\operatorname{Hom}^*_\mathscr{A}(\mathscr{E}, \mathscr{E}_0)$ is said to be a dual of  (hvf) $( \{A_j\}_{j\in \mathbb{J}}, \{\Psi_j\}_{j\in \mathbb{J}})$ in $\operatorname{Hom}^*_\mathscr{A}(\mathscr{E}, \mathscr{E}_0)$  if $ \theta_\Phi^*\theta_A= \theta_B^*\theta_\Psi=I_{\mathscr{E}}$. The `homomorphism-valued frame'  $(\{\widetilde{A}_j\coloneqq A_jS_{A,\Psi}^{-1}\}_{j \in \mathbb{J}}, \{\widetilde{\Psi}_j\coloneqq\Psi_jS_{A,\Psi}^{-1}\}_{j\in \mathbb{J}})$, which is  a `dual' of $ (\{  A_j\}_{j\in \mathbb{J}},\{\Psi_j\}_{j\in \mathbb{J}})$ is called the canonical dual of $ (\{A_j\}_{j\in \mathbb{J}}, \{\Psi_j\}_{j\in \mathbb{J}})$.
\end{definition}
\begin{proposition}
 Let $( \{A_j\}_{j\in \mathbb{J}}, \{\Psi_j\}_{j\in \mathbb{J}} )$ be a (hvf) in $ \operatorname{Hom}^*_\mathscr{A}(\mathscr{E}, \mathscr{E}_0).$  If $ \{y_j\}_{j\in \mathbb{J}},\{z_j\}_{j\in \mathbb{J}}$ in $ \mathscr{E}_0$ are such that $ x=\sum_{j\in\mathbb{J}}A_j^*y_j= \sum_{j\in\mathbb{J}}\Psi_j^*z_j, \forall x \in \mathscr{E}, $ then 
 $$ \sum\limits_{j\in \mathbb{J}}\langle y_j,z_j\rangle =\sum\limits_{j\in \mathbb{J}}\langle \widetilde{\Psi}_jx,\widetilde{A}_jx\rangle+\sum\limits_{j\in \mathbb{J}}\langle y_j-\widetilde{\Psi}_jx,z_j-\widetilde{A}_jx\rangle. $$
 \end{proposition}
\begin{theorem}
Let $(\{A_j\}_{j\in \mathbb{J}},\{\Psi_j\}_{j\in \mathbb{J}} )$ be a (hvf) with frame bounds $ a$ and $ b.$ Then
\begin{enumerate}[\upshape(i)]
\item The canonical dual (hvf) of the canonical dual (hvf)  of $ (\{A_j\}_{j\in \mathbb{J}},    \{\Psi_j\}_{j\in \mathbb{J}}) $ is itself.
\item$ \frac{1}{b}, \frac{1}{a}$ are frame bounds for the canonical dual of $ (\{A_j\}_{j\in \mathbb{J}},\{\Psi_j\}_{j\in \mathbb{J}}).$
\item If $ a, b $ are optimal frame bounds for $ (\{A_j\}_{j\in \mathbb{J}}, \{\Psi_j\}_{j\in \mathbb{J}}),$ then $ \frac{1}{b}, \frac{1}{a}$ are optimal  frame bounds for its canonical dual.
\end{enumerate} 
\end{theorem}
\begin{proposition}
Let  $( \{A_j\}_{j\in \mathbb{J}}, \{\Psi_j\}_{j\in \mathbb{J}}) $ and $ (\{B_j\}_{j\in \mathbb{J}},\{\Phi_j\}_{j\in \mathbb{J}}) $ be homomorphism-valued frames in  $\operatorname{Hom}^*_\mathscr{A}(\mathscr{E}, \mathscr{E}_0)$. Then the following are equivalent.
\begin{enumerate}[\upshape(i)]
\item $ (\{B_j\}_{j\in \mathbb{J}}, \{\Phi_j\}_{j\in \mathbb{J}}) $ is dual of $ (\{A_j\}_{j\in \mathbb{J}}, \{\Psi_j\}_{j\in \mathbb{J}}) $. 
\item $ \sum_{j\in \mathbb{J}}\Phi_j^*A_j = \sum_{j\in \mathbb{J}}B_j^*\Psi_j=I_\mathscr{E}.$ 
\end{enumerate}
\end{proposition}
 \begin{theorem}
Let $ (\{A_j\}_{j\in \mathbb{J}}, \{\Psi_j\}_{j\in \mathbb{J}})$  be a  (hvf) in   $\operatorname{Hom}^*_\mathscr{A}(\mathscr{E}, \mathscr{E}_0)$. If $ (\{A_j\}_{j\in \mathbb{J}}, \{\Psi_j\}_{j\in \mathbb{J}})$ is a Riesz  basis  in  $\operatorname{Hom}^*_\mathscr{A}(\mathscr{E}, \mathscr{E}_0)$, then $ (\{A_j\}_{j\in \mathbb{J}}, \{\Psi_j\}_{j\in \mathbb{J}}) $ has unique dual. Converse holds if $ \theta_A(\mathscr{E})=\theta_\Psi(\mathscr{E})$.
 \end{theorem}
\begin{definition}
A (hvf)   $(\{B_j\}_{j\in \mathbb{J}}, \{\Phi_j\}_{j\in \mathbb{J}})$  in $\operatorname{Hom}^*_\mathscr{A}(\mathscr{E}, \mathscr{E}_0)$ is said to be orthogonal to a (hvf)  $(\{A_j\}_{j\in \mathbb{J}}, \{\Psi_j\}_{j\in \mathbb{J}})$ in $\operatorname{Hom}^*_\mathscr{A}(\mathscr{E}, \mathscr{E}_0)$ if $ \theta_\Phi^*\theta_A= \theta_B^*\theta_\Psi=0.$
\end{definition}
\begin{proposition}
Let  $( \{A_j\}_{j\in \mathbb{J}}, \{\Psi_j\}_{j\in \mathbb{J}}) $ and $ (\{B_j\}_{j\in \mathbb{J}},\{\Phi_j\}_{j\in \mathbb{J}}) $ be homomorphism-valued frames in  $\operatorname{Hom}^*_\mathscr{A}(\mathscr{E}, \mathscr{E}_0)$. Then the following are equivalent.
\begin{enumerate}[\upshape(i)]
\item $ (\{B_j\}_{j\in \mathbb{J}}, \{\Phi_j\}_{j\in \mathbb{J}}) $ is orthogonal to  $ (\{A_j\}_{j\in \mathbb{J}}, \{\Psi_j\}_{j\in \mathbb{J}}) $. 
\item $ \sum_{j\in \mathbb{J}}\Phi_j^*A_j = \sum_{j\in \mathbb{J}}B_j^*\Psi_j=0.$ 
\end{enumerate}
\end{proposition}
\begin{proposition}
Two  orthogonal homomorphism-valued frames have a common dual (hvf).
\end{proposition}
\begin{proposition}
Let $ (\{A_j\}_{j\in \mathbb{J}}, \{\Psi_j\}_{j\in \mathbb{J}}) $ and $ (\{B_j\}_{j\in \mathbb{J}}, \{\Phi_j\}_{j\in \mathbb{J}}) $ be  two Parseval homomorphism-valued frames in  $\operatorname{Hom}^*_\mathscr{A}(\mathscr{E}, \mathscr{E}_0)$ which are  orthogonal. If $C,D,E,F \in \operatorname{End}_\mathscr{A}(\mathscr{E})$ are such that $ C^*E+D^*F=I_\mathscr{E}$, then  $ (\{A_jC+B_jD\}_{j\in \mathbb{J}}, \{\Psi_jE+\Phi_jF\}_{j\in \mathbb{J}}) $ is a  Parseval (hvf) in  $\operatorname{Hom}^*_\mathscr{A}(\mathscr{E}, \mathscr{E}_0)$. In particular,  if  $ a,b,c,d \in \mathscr{A}$ satisfy $a^*c+b^*d =e$ (the identity of $\mathscr{A}$), then $ (\{cA_j+dB_j\}_{j\in \mathbb{J}}, \{e\Psi_j+f\Phi_j\}_{j\in \mathbb{J}}) $ is   a Parseval (hvf).
\end{proposition} 
\begin{definition}
Two homomorphism-valued frames $(\{A_j\}_{j\in \mathbb{J}},\{\Psi_j\}_{j\in \mathbb{J}} )$  and $ (\{B_j\}_{j\in \mathbb{J}}, \{\Phi_j\}_{j\in \mathbb{J}} )$   in $\operatorname{Hom}^*_\mathscr{A}(\mathscr{E}, \mathscr{E}_0)$  are called 
disjoint if $(\{A_j\oplus B_j\}_{j \in \mathbb{J}},\{\Psi_j\oplus \Phi_j\}_{j \in \mathbb{J}})$ is a  (hvf) in $\operatorname{Hom}^*_\mathscr{A}(\mathscr{E}\oplus\mathscr{E},\mathscr{E}_0)$.
\end{definition}
\begin{proposition}
If $(\{A_j\}_{j\in \mathbb{J}},\{\Psi_j\}_{j\in \mathbb{J}} )$  and $ (\{B_j\}_{j\in \mathbb{J}}, \{\Phi_j\}_{j\in \mathbb{J}} )$  are  orthogonal homomorphism-valued frames  in $\operatorname{Hom}^*_\mathscr{A}(\mathscr{E}\oplus\mathscr{E},\mathscr{E}_0)$, then  they  are disjoint. Further, if both $(\{A_j\}_{j\in \mathbb{J}},\{\Psi_j\}_{j\in \mathbb{J}} )$  and $ (\{B_j\}_{j\in \mathbb{J}}, \{\Phi_j\}_{j\in \mathbb{J}} )$ are  Parseval, then $(\{A_j\oplus B_j\}_{j \in \mathbb{J}},\{\Psi_j\oplus \Phi_j\}_{j \in \mathbb{J}})$ is Parseval.
\end{proposition} 
\textbf{Characterizations}
 \begin{theorem}
Let $ \{F_j\}_{j \in \mathbb{J}}$ be an arbitrary orthonormal basis in $\operatorname{Hom}^*_\mathscr{A}(\mathscr{E}, \mathscr{E}_0).$ Then 
 \begin{enumerate}[\upshape(i)]
 \item The orthonormal  bases   $ (\{A_j\}_{j \in \mathbb{J}},\{\Psi_j\}_{j \in \mathbb{J}})$ in  $ \operatorname{Hom}^*_\mathscr{A}(\mathscr{E}, \mathscr{E}_0)$ are precisely $( \{F_jU\}_{j \in \mathbb{J}},\{c_jF_jU\}_{j \in \mathbb{J}}) $, 
 where $ U \in \operatorname{End}^*_\mathscr{A}(\mathscr{E}) $ is unitary and $ c_j'$s are positive invertible elements in the center of  $ \mathscr{A}$ such that $ 0<\inf\{\|c_j\|\}_{j \in \mathbb{J}}\leq \sup\{\|c_j\|\}_{j \in \mathbb{J}}< \infty.$
 \item The Riesz bases   $ (\{A_j\}_{j \in \mathbb{J}},\{\Psi_j\}_{j \in \mathbb{J}})$ in  $ \operatorname{Hom}^*_\mathscr{A}(\mathscr{E}, \mathscr{E}_0)$   are precisely $( \{F_jU\}_{j \in \mathbb{J}},\{F_jV\}_{j \in \mathbb{J}}) $, where $ U,V \in \operatorname{End}^*_\mathscr{A}(\mathscr{E}) $ are  invertible  such that  $ V^*U$ is positive.
 \item The homomorphism-valued frames $ (\{A_j\}_{j \in \mathbb{J}},\{\Psi_j\}_{j \in \mathbb{J}})$ in   $ \operatorname{Hom}^*_\mathscr{A}(\mathscr{E}, \mathscr{E}_0)$  are precisely $( \{F_jU\}_{j \in \mathbb{J}},\{F_jV\}_{j \in \mathbb{J}}) $, where $ U,V \in \operatorname{End}^*_\mathscr{A}(\mathscr{E}) $  are such that $ V^*U$ is positive invertible.
 \item The Bessel sequences  $(\{A_j\}_{j \in \mathbb{J}},\{\Psi_j\}_{j \in \mathbb{J}})$ in  $ \operatorname{Hom}^*_\mathscr{A}(\mathscr{E}, \mathscr{E}_0)$  are precisely $( \{F_jU\}_{j \in \mathbb{J}},\{F_jV\}_{j \in \mathbb{J}}) $, where $ U,V \in \operatorname{End}^*_\mathscr{A}(\mathscr{E}) $  are such that  $ V^*U$ is positive. 
 \item The Riesz homomorphism-valued frames $ (\{A_j\}_{j \in \mathbb{J}},\{\Psi_j\}_{j \in \mathbb{J}})$ in  $ \operatorname{Hom}^*_\mathscr{A}(\mathscr{E}, \mathscr{E}_0)$  are precisely $( \{F_jU\}_{j \in \mathbb{J}},\{F_jV\}_{j \in \mathbb{J}}) $, where $ U,V \in \operatorname{End}^*_\mathscr{A}(\mathscr{E}) $ are such that  $ V^*U$ is positive invertible and $ U(V^*U)^{-1}V^* =I_{\mathscr{E}}$.  
 \item The orthonormal homomorphism-valued frames $ ( \{A_j\}_{j \in \mathbb{J}},\{\Psi_j\}_{j \in \mathbb{J}})$ in  $ \operatorname{Hom}^*_\mathscr{A}(\mathscr{E}, \mathscr{E}_0)$  are precisely $( \{F_jU\}_{j \in \mathbb{J}},\{F_jV\}_{j \in \mathbb{J}}) $, where $ U,V \in \operatorname{End}^*_\mathscr{A}(\mathscr{E}) $ are such that $ V^*U=I_\mathscr{E}= UV^*$.
\end{enumerate}
\end{theorem} 
\begin{proof}
Similar to proof of Theorem \ref{OPERATORVALUEDCHARACTERIZATIONSOFTHEEXTENSION}, but  note that 
\begin{align*}
\sum_{j\in \mathbb{J}}\langle F_jUx,F_jUx \rangle=\left\langle\sum_{j\in \mathbb{J}}F_j^*F_j Ux , Ux\right\rangle =\langle Ux,Ux \rangle=\langle x,x \rangle, 
\end{align*}
and 
\begin{align*}
\left\langle\sum_{j\in \mathbb{S}}F_j^*A_jx, \sum_{k\in \mathbb{S}}F_k^*A_kx\right\rangle= \sum_{j\in \mathbb{S}}\left\langle A_jx, F_j\left(\sum_{k\in \mathbb{S}}F_k^*A_kx\right) \right\rangle=\sum_{j\in \mathbb{S}}\langle A_jx,A_jx \rangle
\end{align*}
for every finite subset $\mathbb{S}$ of $\mathbb{J}$, for every unitary $U \in\operatorname{End}^*_\mathscr{A}(\mathscr{E}) $, and for all $x\in\mathscr{E}$.
\end{proof}
\begin{corollary}
\begin{enumerate}[\upshape(i)]
\item If   $ (\{A_j\}_{j \in \mathbb{J}},\{\Psi_j=c_jA_j\}_{j \in \mathbb{J}})$ is an orthonormal basis  in $ \operatorname{Hom}^*_\mathscr{A}(\mathscr{E}, \mathscr{E}_0)$, then 
$$ \sup\{\|A_j\|\}_{j\in\mathbb{J}}\leq1, ~\sup\{\|\Psi_j\|\}_{j\in\mathbb{J}}\leq \sup\{\|c_j\|\}_{j\in\mathbb{J}}, ~ A_j\Psi_j^*=c_jI_{\mathscr{E}_0},  ~\forall j \in \mathbb{J}.$$
\item If   $ (\{A_j\}_{j \in \mathbb{J}},\{\Psi_j\}_{j \in \mathbb{J}})$ is  Bessel  in $ \operatorname{Hom}^*_\mathscr{A}(\mathscr{E}, \mathscr{E}_0)$, then 
$$ \sup\{\|A_j\|\}_{j\in\mathbb{J}}\leq\|U\|, ~\sup\{\|\Psi_j\|\}_{j\in\mathbb{J}}\leq\|V\| , ~ \sup\{\|A_j\Psi_j^*\|\}_{j\in\mathbb{J}}\leq\|UV^*\|.$$
\end{enumerate}	
\end{corollary}
\begin{corollary}\label{OVMODULECOROLLARY}
Let $ \{F_j\}_{j \in \mathbb{J}}$ be an arbitrary orthonormal basis in $\operatorname{Hom}^*_\mathscr{A}(\mathscr{E}, \mathscr{E}_0).$ Then 
\begin{enumerate}[\upshape(i)]
\item The orthonormal  bases   $ (\{A_j\}_{j \in \mathbb{J}},\{A_j\}_{j \in \mathbb{J}})$ in  $ \operatorname{Hom}^*_\mathscr{A}(\mathscr{E}, \mathscr{E}_0)$ are precisely $( \{F_jU\}_{j \in \mathbb{J}},\{F_jU\}_{j \in \mathbb{J}}) $, where $ U \in \operatorname{End}^*_\mathscr{A}(\mathscr{E}) $ is unitary.
\item The Riesz bases   $ (\{A_j\}_{j \in \mathbb{J}},\{A_j\}_{j \in \mathbb{J}})$ in  $ \operatorname{Hom}^*_\mathscr{A}(\mathscr{E}, \mathscr{E}_0)$   are precisely $( \{F_jU\}_{j \in \mathbb{J}},\{F_jU\}_{j \in \mathbb{J}}) $, where $ U \in \operatorname{End}^*_\mathscr{A}(\mathscr{E}) $ is invertible.
\item The homomorphism-valued frames $ (\{A_j\}_{j \in \mathbb{J}},\{A_j\}_{j \in \mathbb{J}})$ in   $ \operatorname{Hom}^*_\mathscr{A}(\mathscr{E}, \mathscr{E}_0)$  are precisely $( \{F_jU\}_{j \in \mathbb{J}},\{F_jU\}_{j \in \mathbb{J}}) $, where $ U\in \operatorname{End}^*_\mathscr{A}(\mathscr{E}) $  is such that $ U^*U$ is invertible.
\item The Bessel sequences  $(\{A_j\}_{j \in \mathbb{J}},\{A_j\}_{j \in \mathbb{J}})$ in  $ \operatorname{Hom}^*_\mathscr{A}(\mathscr{E}, \mathscr{E}_0)$  are precisely $( \{F_jU\}_{j \in \mathbb{J}},\{F_jU\}_{j \in \mathbb{J}}) $, where $ U \in \operatorname{End}^*_\mathscr{A}(\mathscr{E}) $. 
\item The Riesz homomorphism-valued frames $ (\{A_j\}_{j \in \mathbb{J}},\{A_j\}_{j \in \mathbb{J}})$ in  $ \operatorname{Hom}^*_\mathscr{A}(\mathscr{E}, \mathscr{E}_0)$  are precisely $( \{F_jU\}_{j \in \mathbb{J}},\{F_jU\}_{j \in \mathbb{J}}) $, where $ U \in \operatorname{End}^*_\mathscr{A}(\mathscr{E}) $ is  such that  $ U^*U$ is  invertible and $ U(U^*U)^{-1}U^* =I_{\mathscr{E}}$.  
\item The orthonormal homomorphism-valued frames $ ( \{A_j\}_{j \in \mathbb{J}},\{A_j\}_{j \in \mathbb{J}})$ in  $ \operatorname{Hom}^*_\mathscr{A}(\mathscr{E}, \mathscr{E}_0)$  are precisely $( \{F_jU\}_{j \in \mathbb{J}},\{F_jU\}_{j \in \mathbb{J}}) $, where $ U \in \operatorname{End}^*_\mathscr{A}(\mathscr{E}) $ is  such that $ U^*U=I_\mathscr{E}= UU^*$.
\item  $ ( \{A_j\}_{j \in \mathbb{J}},\{A_j\}_{j \in \mathbb{J}})$ is an orthonormal   basis  in  $ \operatorname{Hom}^*_\mathscr{A}(\mathscr{E}, \mathscr{E}_0)$ if and only if it is an orthonormal (hvf).
\end{enumerate}	
\end{corollary}
\begin{caution}
Why  is the  result corresponding to \text{\upshape(vii)} in Corollary \ref{OVCOROLLARYHILBERT}  missing in the statement of  Corollary \ref{OVMODULECOROLLARY}? The reason is that a closed submodule of a Hilbert C*-module need not be orthogonally complemented.
\end{caution}
\begin{theorem}\label{OPERATORVERSIONCHARACTERIZATIONMODULES}
Let $\{A_j\}_{j\in\mathbb{J}},\{\Psi_j\}_{j\in\mathbb{J}}$ be in  $ \operatorname{Hom}^*_\mathscr{A}(\mathscr{E}, \mathscr{E}_0).$  Then $ (\{A_j\}_{j\in \mathbb{J}}, \{\Psi_j\}_{j\in \mathbb{J}})$  is a (hvf) with bounds  $a $ and  $ b$ (resp. Bessel with bound $ b$)
\begin{enumerate}[\upshape(i)]
\item  if and only if $$U: \mathscr{H}_\mathscr{A}\otimes \mathscr{E}_0 \ni y\mapsto\sum\limits_{j\in\mathbb{J}}A_j^*L_j^*y \in \mathscr{E}, ~\text{and} ~ V: \mathscr{H}_\mathscr{A}\otimes \mathscr{E}_0 \ni z\mapsto\sum\limits_{j\in\mathbb{J}}\Psi_j^*L^*_jz \in \mathscr{E} $$ 
are well-defined, $U,V \in \operatorname{Hom}^*_\mathscr{A}(\mathscr{H}_\mathscr{A}\otimes\mathscr{E}_0, \mathscr{E}) $  such that  $ aI_\mathscr{E}\leq VU^*\leq bI_\mathscr{E}$ (resp. $ 0\leq VU^*\leq bI_\mathscr{E}$).
\item   if and only if $$U: \mathscr{H}_\mathscr{A}\otimes \mathscr{E}_0 \ni y\mapsto\sum\limits_{j\in\mathbb{J}}A_j^*L_j^*y \in \mathscr{E}, ~\text{and} ~ S: \mathscr{E} \ni g\mapsto \sum\limits_{j\in \mathbb{J}}L_j\Psi_jg \in  \mathscr{H}_\mathscr{A}\otimes \mathscr{E}_0 $$ 
are well-defined, $U \in \operatorname{Hom}^*_\mathscr{A}(\mathscr{H}_\mathscr{A}\otimes\mathscr{E}_0, \mathscr{E}) $, $S \in \operatorname{Hom}^*_\mathscr{A}(\mathscr{E}, \mathscr{H}_\mathscr{A}\otimes\mathscr{E}_0) $  such that  $ aI_\mathscr{E}\leq S^*U^*\leq bI_\mathscr{E}$ (resp. $ 0\leq S^*U^*\leq bI_\mathscr{E}$).
\item  if and only if  $$R:   \mathscr{E} \ni h\mapsto \sum\limits_{j\in \mathbb{J}}L_jA_jh \in  \mathscr{H}_\mathscr{A}\otimes \mathscr{E}_0, ~\text{and} ~ V:  \mathscr{H}_\mathscr{A}\otimes \mathscr{E}_0 \ni z\mapsto\sum\limits_{j\in\mathbb{J}}\Psi_j^*L_j^*z \in \mathscr{E} $$ 
are well-defined, $R \in \operatorname{Hom}^*_\mathscr{A}(\mathscr{E}, \mathscr{H}_\mathscr{A}\otimes\mathscr{E}_0) $, $V \in \operatorname{Hom}^*_\mathscr{A}(\mathscr{H}_\mathscr{A}\otimes\mathscr{E}_0, \mathscr{E}) $ such that  $ aI_\mathscr{E}\leq VR\leq bI_\mathscr{E}$ (resp. $ 0\leq VR\leq bI_\mathscr{E}$).
\item  if and only if  $$ R:   \mathscr{E} \ni h\mapsto \sum\limits_{j\in \mathbb{J}}L_jA_jh \in  \mathscr{H}_\mathscr{A}\otimes \mathscr{E}_0, ~\text{and} ~  S:   \mathscr{E} \ni g\mapsto \sum\limits_{j\in \mathbb{J}}L_j\Psi_jg \in  \mathscr{H}_\mathscr{A}\otimes \mathscr{E}_0 $$
are well-defined, $R,S \in \operatorname{Hom}^*_\mathscr{A}(\mathscr{E}, \mathscr{H}_\mathscr{A}\otimes\mathscr{E}_0) $ such that  $ aI_\mathscr{E}\leq S^*R\leq bI_\mathscr{E}$ (resp. $ 0\leq S^*R\leq bI_\mathscr{E}$). 
\end{enumerate}
\end{theorem} 
One more charactrization (Theorem \ref{HOMOMORPHISMTOSEQUENTIALVICEVERSAMODULE}) will appear in Section \ref{SEQUENTIALVERSIONOFHOMOMORPHISM-VALUEDFRAMES}. 

\begin{theorem}\cite{PASCHKE1}\label{PASCHKE}
 Let $T:\mathscr{E} \rightarrow \mathscr{E}_0$ be linear. The following are equivalent.
 \begin{enumerate}[\upshape(i)]
\item Operator $ T$ is bounded, and $ \mathscr{A}$-linear.
\item  There exists a real  $K>0 $ such that $ \langle Tx, Tx\rangle\leq K\langle x, x\rangle, \forall x \in \mathscr{E}.$
\end{enumerate}
\end{theorem}
\begin{theorem}\cite{ARAMBASIC}\label{ARAMBASIC1}
 If $T \in \operatorname{End}^*_\mathscr{A}(\mathscr{E}) $ is self-adjoint, then the following are equivalent.
\begin{enumerate}[\upshape(i)]
\item $ T$ is surjective.
\item  There are $ a,b >0$ such that $ a\|x\|\leq\|Tx\|\leq b\|x\|, \forall x \in \mathscr{E}.$
\item  There are $ c,d >0$ such that $c \langle x,x\rangle \leq \langle Tx,Tx \rangle \leq d\langle x,x\rangle , \forall x \in \mathscr{E}.$
\end{enumerate}
\end{theorem}
\begin{theorem}\label{CHARACTERIZATION}
Let $ \{A_j\}_{j\in \mathbb{J}}, \{\Psi_j\}_{j\in \mathbb{J}} $ be in $\operatorname{Hom}^*_\mathscr{A}(\mathscr{E}, \mathscr{E}_0) $. If  $(  \{A_j\}_{j\in \mathbb{J}}, \{\Psi_j\}_{j\in \mathbb{J}} )$ is (hvf), then  
\begin{enumerate}[\upshape(i)]
\item there are $a ,b> 0$ such that
 $$a\|x\|^2\leq \left\|\sum_{j\in \mathbb{J}} \langle A_jx, \Psi_jx\rangle\right\|=\left\|\sum_{j\in \mathbb{J}} \langle \Psi_jx, A_jx\rangle\right\| \leq b\|x\|^2 , ~\forall x \in \mathscr{E},$$  
\item there are $ c,d> 0$ such that 
$$ \sum\limits_{j \in \mathbb{J}} \langle A_jx, A_jx\rangle \leq c\langle x, x\rangle , ~ \forall x \in \mathscr{E} ~ \text{and}~ \sum\limits_{j \in \mathbb{J}} \langle \Psi_jx, \Psi_jx\rangle \leq d\langle x, x \rangle , ~ \forall x \in \mathscr{E}.$$
\end{enumerate}
If $ \Psi_j^*A_j\geq 0, \forall j \in \mathbb{J},$ then the converse holds.
\end{theorem}
\begin{proof}
\begin{enumerate}[\upshape(i)]
\item Norm preserves the order relation among positive elements of a C*-algebra.
\item We apply Theorem \ref{PASCHKE} to $ \theta_A$ and $ \theta_\Psi.$ 
\end{enumerate}

Let $ \Psi_j^*A_j\geq 0, \forall j \in \mathbb{J}$ and (i), (ii) hold.  That  (ii) gives $ \theta_A,  \theta_\Psi \in \operatorname{Hom}^*_\mathscr{A}(\mathscr{E}, \mathscr{H}_\mathscr{A}\otimes\mathscr{E}_0).$  Thus $ S_{A,\Psi}=\theta_\Psi^*\theta_A$  exists as a bounded homomorphism. Further, $ \Psi_j^*A_j\geq 0, \forall j \in \mathbb{J}$ gives $ S_{A,\Psi}=\theta_\Psi^*\theta_A$ is  positive. Even if  $ \Psi_j^*A_j> 0$ for atleast one $  j, $ we get $ S_{A,\Psi}>0.$ But this happens (else $ \Psi_j^*A_j= 0, \forall j \in \mathbb{J}$ implies $x=0, \forall x \in \mathscr{E} $, from (i) of the assumptions. Thus $ \mathscr{E}=\{0\} $ which is too trivial). Hence  $ S^{1/2}_{A,\Psi}$ exists. Using this in (i),  $ \sqrt{a}\|x\|\leq \|S^{1/2}_{A,\Psi}x\|\leq \sqrt{b}\|x\|, \forall x \in \mathscr{E}.$ We now refer Theorem \ref{ARAMBASIC1} to end the proof.
\end{proof}
\begin{corollary}
Let $ \{A_j\}_{j\in \mathbb{J}}, \{\Psi_j\}_{j\in \mathbb{J}} $ be in $\operatorname{Hom}^*_\mathscr{A}(\mathscr{E}, \mathscr{E}_0) $. If  $ ( \{A_j\}_{j\in \mathbb{J}} ,\{\Psi_j\}_{j\in \mathbb{J}} )$ is Bessel, then  
\begin{enumerate}[\upshape(i)]
\item there is  $b> 0$ such that
$$  \left\|\sum_{j\in \mathbb{J}} \langle A_jx, \Psi_jx\rangle\right\|=\left\|\sum_{j\in \mathbb{J}} \langle \Psi_jx, A_jx\rangle\right\|\leq b\|x\|^2 , ~\forall x \in \mathscr{E},$$  
\item there are $ c,d> 0$ such that 
$$ \sum\limits_{j \in \mathbb{J}} \langle A_jx, A_jx\rangle \leq c\langle x, x\rangle , ~ \forall x \in \mathscr{E} ~ \text{and}~ \sum\limits_{j \in \mathbb{J}} \langle \Psi_jx, \Psi_jx\rangle \leq d\langle x, x \rangle , ~ \forall x \in \mathscr{E}.$$
\end{enumerate}
If $ \Psi_j^*A_j\geq 0, \forall j \in \mathbb{J},$ then the converse holds.
\end{corollary}
\textbf{Similarity, composition and tensor product}
\begin{definition}
A (hvf) $(\{B_j\}_{j\in \mathbb{J}}, \{\Phi_j\}_{j\in \mathbb{J}})$  in  $ \operatorname{Hom}^*_\mathscr{A}(\mathscr{E}, \mathscr{E}_0)$   is said to be right-similar (resp. left-similar) to a (hvf) $(\{A_j\}_{j\in \mathbb{J}}, \{\Psi_j\}_{j\in \mathbb{J}})$ in $ \operatorname{Hom}^*_\mathscr{A}(\mathscr{E}, \mathscr{E}_0)$ if there exist invertible  $ R_{A,B}, R_{\Psi, \Phi} \in \operatorname{End}^*_\mathscr{A}(\mathscr{E})$  (resp. $ L_{A,B}, L_{\Psi, \Phi}$ $ \in \operatorname{End}^*_\mathscr{A}(\mathscr{E}_0)$) such that $B_j=A_jR_{A,B} , \Phi_j=\Psi_jR_{\Psi, \Phi} $  (resp. $ B_j=L_{A,B}A_j , \Phi_j=L_{\Psi, \Phi}\Psi_j$), $\forall j \in \mathbb{J}. $
\end{definition}

\begin{proposition}
Let $ \{A_j\}_{j\in \mathbb{J}}\in \mathscr{F}_\Psi$  with frame bounds $a, b,$  let $R_{A,B}, R_{\Psi, \Phi} \in \operatorname{End}^*_\mathscr{A}(\mathscr{E})$ be positive, invertible, commute with each other, commute with $ S_{A, \Psi}$, and let $B_j=A_jR_{A,B} , \Phi_j=\Psi_jR_{\Psi, \Phi},  \forall j \in \mathbb{J}.$ Then 
\begin{enumerate}[\upshape(i)]
\item $ \{B_j\}_{j\in \mathbb{J}}\in \mathscr{F}_\Phi$ and $ \frac{a}{\|R_{A,B}^{-1}\|\|R_{\Psi,\Phi}^{-1}\|}\leq S_{B, \Phi} \leq b\|R_{A,B}R_{\Psi,\Phi}\|.$ Assuming that $( \{A_j\}_{j\in \mathbb{J}},\{\Psi_j\}_{j\in \mathbb{J}})$ is Parseval, then $( \{B_j\}_{j\in \mathbb{J}}, \{\Phi_j\}_{j\in \mathbb{J}})$  is Parseval if and only if   $ R_{\Psi, \Phi}R_{A,B}=I_\mathscr{E}.$  
\item $ \theta_B=\theta_A R_{A,B}, \theta_\Phi=\theta_\Psi R_{\Psi,\Phi}, S_{B,\Phi}=R_{\Psi,\Phi}S_{A, \Psi}R_{A,B},  P_{B,\Phi}=P_{A, \Psi}.$
\end{enumerate}
\end{proposition}
\begin{lemma}
 Let $ \{A_j\}_{j\in \mathbb{J}}\in \mathscr{F}_\Psi,$ $ \{B_j\}_{j\in \mathbb{J}}\in \mathscr{F}_\Phi$ and   $B_j=A_jR_{A,B} , \Phi_j=\Psi_jR_{\Psi, \Phi},  \forall j \in \mathbb{J}$, for some invertible $ R_{A,B} ,R_{\Psi, \Phi} \in \operatorname{End}^*_\mathscr{A}(\mathscr{E}).$ Then 
  $ \theta_B=\theta_A R_{A,B}, \theta_\Phi=\theta_\Psi R_{\Psi,\Phi}, S_{B,\Phi}=R_{\Psi,\Phi}^*S_{A, \Psi}R_{A,B},  P_{B,\Phi}=P_{A, \Psi}.$ Assuming that $ (\{A_j\}_{j\in \mathbb{J}},\{\Psi_j\}_{j\in \mathbb{J}})$ is Parseval, then $(\{B_j\}_{j\in \mathbb{J}},  \{\Phi_j\}_{j\in \mathbb{J}})$ is Parseval  if and only if   $ R_{\Psi, \Phi}^*R_{A,B}=I_\mathscr{E}.$
 \end{lemma}
 \begin{theorem}
 Let $ \{A_j\}_{j\in \mathbb{J}}\in \mathscr{F}_\Psi,$ $ \{B_j\}_{j\in \mathbb{J}}\in \mathscr{F}_\Phi.$ The following are equivalent.
 \begin{enumerate}[\upshape(i)]
 \item $B_j=A_jR_{A,B} , \Phi_j=\Psi_jR_{\Psi, \Phi} ,  \forall j \in \mathbb{J},$ for some invertible  $ R_{A,B} ,R_{\Psi, \Phi} \in \operatorname{End}^*_\mathscr{A}(\mathscr{E}). $
 \item $\theta_B=\theta_AR_{A,B}' , \theta_\Phi=\theta_\Psi R_{\Psi, \Phi}' $ for some invertible  $ R_{A,B}' ,R_{\Psi, \Phi}' \in \operatorname{End}^*_\mathscr{A}(\mathscr{E}). $
 \item $P_{B,\Phi}=P_{A,\Psi}.$
 \end{enumerate}
 If one of the above conditions is satisfied, then  invertible homomorphisms in  $\operatorname{(i)}$ and  $\operatorname{(ii)}$ are unique and are given by $R_{A,B}=S_{A,\Psi}^{-1}\theta_\Psi^*\theta_B, R_{\Psi, \Phi}=S_{A,\Psi}^{-1}\theta_A^*\theta_\Phi.$
 In the case that $(\{A_j\}_{j \in \mathbb{J}} , \{\Psi_j\}_{j \in \mathbb{J}})$ is Parseval, then $ (\{B_j\}_{j \in \mathbb{J}} , \{\Phi_j\}_{j \in \mathbb{J}})$ is  Parseval if and only if $R_{\Psi, \Phi}^*R_{A,B}=I_\mathscr{E} $  if and only if $R_{A,B}R_{\Psi, \Phi}^*=I_\mathscr{E}  $.  
 \end{theorem}
 \begin{corollary}
 For any given (hvf) $ (\{A_j\}_{j \in \mathbb{J}} , \{\Psi_j\}_{j \in \mathbb{J}})$, the canonical dual of $ (\{A_j\}_{j \in \mathbb{J}} , \{\Psi_j\}_{j \in \mathbb{J}}  )$ is the only dual (hvf) that is right-similar to $ (\{A_j\}_{j \in \mathbb{J}} , \{\Psi_j\}_{j \in \mathbb{J}} )$.
 \end{corollary}
 \begin{corollary}
 Two right-similar homomorphism-valued frames cannot be orthogonal.
 \end{corollary}
 \begin{remark}
 For every (hvf) $(\{A_j\}_{j \in \mathbb{J}},\{\Psi_j\}_{j \in \mathbb{J}})$, each  of `homomorphism-valued frames'  $( \{A_jS_{A, \Psi}^{-1}\}_{j \in \mathbb{J}}, \{\Psi_j\}_{j \in \mathbb{J}}),$   $( \{A_jS_{A, \Psi}^{-1/2}\}_{j \in \mathbb{J}}, \{\Psi_jS_{A,\Psi}^{-1/2}\}_{j \in \mathbb{J}}),$ and  $ (\{A_j \}_{j \in \mathbb{J}}, \{\Psi_jS_{A,\Psi}^{-1}\}_{j \in \mathbb{J}})$ is a  Parseval (hvf) which is right-similar to  $ (\{A_j\}_{j \in \mathbb{J}} , \{\Psi_j\}_{j \in \mathbb{J}}  ).$  
 \end{remark}
 \begin{proposition}
  Let $ \{A_j\}_{j\in \mathbb{J}}\in \mathscr{F}_\Psi,$ $ \{B_j\}_{j\in \mathbb{J}}\in \mathscr{F}_\Phi$ and   $B_j=L_{A,B}A_j , \Phi_j=L_{\Psi, \Phi}\Psi_j,  \forall j \in \mathbb{J}$, for some invertible $ L_{A,B} ,L_{\Psi, \Phi} \in \operatorname{End}^*_\mathscr{A}(\mathscr{E}_0).$ Then 
  \begin{enumerate}[\upshape(i)]
  \item $ \theta_B=(I_{\mathscr{H}_\mathscr{A}}\otimes L_{A,B})\theta_A , \theta_\Phi=(I_{\mathscr{H}_\mathscr{A}}\otimes L_{\Psi,\Phi})\theta_\Psi, S_{B,\Phi}=\theta_\Psi^*(I_{\mathscr{H}_\mathscr{A}}\otimes L_{\Psi,\Phi}^*L_{A,B})\theta_A,  P_{B,\Phi}=(I_{\mathscr{H}_\mathscr{A}}\otimes L_{A,B})\theta_A(\theta_\Psi^*(I_{\mathscr{H}_\mathscr{A}}\otimes L_{\Psi,\Phi}^*L_{A,B})\theta_A)^{-1}\theta_\Psi^*(I_{\mathscr{H}_\mathscr{A}}\otimes L_{\Psi,\Phi}^*).$ 
  \item Assuming $(\{A_j\}_{j\in \mathbb{J}}, \{\Psi_j\}_{j\in \mathbb{J}}) $ is Parseval, then $(\{B_j\}_{j\in \mathbb{J}},\{\Phi_j\}_{j\in \mathbb{J}}) $ is Parseval if and only if $ P_{A, \Psi}(I_{\mathscr{H}_\mathscr{A}}\otimes L_{\Psi,\Phi}^*L_{A,B})P_{A,\Psi}=P_{A,\Psi}$ if and only if $ P_{B,\Phi}=(I_{\mathscr{H}_\mathscr{A}}\otimes L_{A,B})P_{A,\Psi}(I_{\mathscr{H}_\mathscr{A}}\otimes L_{\Psi, \Phi}^*).$
   \end{enumerate}
  \end{proposition}
  \begin{definition}
  Let $(\{A_j\}_{j \in \mathbb{J}}, \{\Psi_j \}_ {j \in \mathbb{J}})$ and $(\{B_j\}_{j \in \mathbb{J}}, \{\Phi_j \}_ {j \in \mathbb{J}})$ be two homomorphism-valued frames in  $ \operatorname{Hom}^*_\mathscr{A}(\mathscr{E}, \mathscr{E}_0) $. We say that  $( \{B_j\}_{j \in \mathbb{J}}, \{\Phi_j \}_ {j \in \mathbb{J}})$ is   
  \begin{enumerate}[\upshape(i)]
  \item RL-similar (right-left-similar) to $( \{A_j\}_{j \in \mathbb{J}},  \{\Psi_j \}_ {j \in \mathbb{J}})$ if there exist invertible $R_{A,B}, L_{\Psi, \Phi} $ in $  \operatorname{End}^*_\mathscr{A}(\mathscr{E})$, $ \operatorname{End}^*_\mathscr{A}(\mathscr{E}_0)  $, respectively  such that $B_j=A_jR_{A,B} , \Phi_j=L_{\Psi, \Phi}\Psi_j, \forall j \in \mathbb{J}.$
  \item LR-similar (left-right-similar) to $ (\{A_j\}_{j \in \mathbb{J}}, \{\Psi_j \}_ {j \in \mathbb{J}})$ if there exist invertible $L_{A,B}, R_{\Psi, \Phi}  $ in $ \operatorname{End}^*_\mathscr{A}(\mathscr{E}_0)$, $ \operatorname{End}^*_\mathscr{A}(\mathscr{E})$, respectively  such that $B_j=L_{A,B}A_j , \Phi_j=\Psi_jR_{\Psi, \Phi} , \forall j \in \mathbb{J}.$
  \end{enumerate}
  \end{definition}
  \begin{proposition}
Let $ \{A_j\}_{j\in \mathbb{J}}\in \mathscr{F}_\Psi,$ $ \{B_j\}_{j\in \mathbb{J}}\in \mathscr{F}_\Phi$ and   $B_j=A_jR_{A,B}, \Phi_j=L_{\Psi, \Phi}\Psi_j,  \forall j \in \mathbb{J}$, for some invertible $ R_{A,B} ,L_{\Psi, \Phi} $ in $\operatorname{End}^*_\mathscr{A}(\mathscr{E}),\operatorname{End}^*_\mathscr{A}(\mathscr{E}_0) $, respectively. Then 
$ \theta_B=\theta_A R_{A,B}, \theta_\Phi=(I_{\mathscr{H}_\mathscr{A}}\otimes L_{\Psi, \Phi})\theta_\Psi , S_{B,\Phi}=\theta_\Psi^*(I_{\mathscr{H}_\mathscr{A}}\otimes L^*_{\Psi, \Phi})\theta_AR_{A,B},  P_{B,\Phi}=\theta_A(\theta_\Psi^*(I_{\mathscr{H}_\mathscr{A}}\otimes L^*_{\Psi, \Phi})\theta_A)^{-1}\theta_\Psi^*(I_{\mathscr{H}_\mathscr{A}}\otimes L^*_{\Psi, \Phi}).$
\end{proposition}
 \begin{proposition}
Let $ \{A_j\}_{j\in \mathbb{J}}\in \mathscr{F}_\Psi,$ $ \{B_j\}_{j\in \mathbb{J}}\in \mathscr{F}_\Phi$ and   $B_j=L_{A,B}A_j, \Phi_j=\Psi_jR_{\Psi, \Phi},  \forall j \in \mathbb{J}$, for some invertible $ L_{A,B} ,R_{\Psi, \Phi} $ in $\operatorname{End}^*_\mathscr{A}(\mathscr{E}_0),\operatorname{End}^*_\mathscr{A}(\mathscr{E}) $, respectively. Then 
$ \theta_B=(I_{\mathscr{H}_\mathscr{A}}\otimes L_{A,B})\theta_A , \theta_\Phi=  \theta_\Psi R_{\Psi,\Phi} , S_{B,\Phi}=R_{\Psi,\Phi}^*\theta_\Psi^*(I_{\mathscr{H}_\mathscr{A}}\otimes L_{A,B})\theta_A ,  P_{B,\Phi}=(I_{\mathscr{H}_\mathscr{A}}\otimes L_{A,B})\theta_A(\theta_\Psi^*(I_{\mathscr{H}_\mathscr{A}}\otimes L_{A,B})\theta_A)^{-1}\theta_\Psi^*.$
\end{proposition}
\textbf{Composition of frames}: Let $ \{A_j\}_{j \in \mathbb{J}} $  be a (hvf) w.r.t. $ \{\Psi_j\}_{j \in \mathbb{J}} $  in  $ \operatorname{Hom}^*_\mathscr{A}(\mathscr{E}, \mathscr{E}_0),$ and $ \{B_l\}_{l \in \mathbb{L}} $  be a (hvf) w.r.t. $ \{\Phi_l\}_{l \in \mathbb{L}} $  in  $ \operatorname{Hom}^*_\mathscr{A}(\mathscr{E}_0, \mathscr{E}_1).$ Suppose  $\{C_{(l, j)}\coloneqq B_lA_j\}_{(l, j)\in \mathbb{L}\bigtimes  \mathbb{J}} $ is  a (hvf) w.r.t. $\{\Xi_{(l, j)}\coloneqq \Phi_l\Psi_j\}_{(l, j)\in \mathbb{L}\bigtimes  \mathbb{J}} $ in $ \operatorname{Hom}^*_\mathscr{A}(\mathscr{E}, \mathscr{E}_1)$. Then the frame $ (\{C_{(l, j)}\}_{(l, j)\in \mathbb{L}\bigtimes  \mathbb{J}}, \{\Xi_{(l, j)}\}_{(l, j)\in \mathbb{L}\bigtimes  \mathbb{J}}) $ is called  as composition of frames $( \{A_j\}_{j \in \mathbb{J}} , \{\Psi_j\}_{j\in \mathbb{J}})$  and $ (\{B_l\}_{l \in \mathbb{L}}, \{\Phi_l\}_{l\in \mathbb{L}}).$ 
\begin{proposition}
Suppose $\{C_{(l, j)}\coloneqq B_lA_j\}_{(l, j)\in \mathbb{L}\bigtimes  \mathbb{J}} $ w.r.t. $\{\Xi_{(l, j)}\coloneqq \Phi_l\Psi_j\}_{(l, j)\in \mathbb{L}\bigtimes  \mathbb{J}} $ is composition of  homomorphism-valued frames  $( \{A_j\}_{j \in \mathbb{J}}, \{\Psi_j\}_{j \in \mathbb{J}} )$  in  $ \operatorname{Hom}^*_\mathscr{A}(\mathscr{E}, \mathscr{E}_0),$ and $ (\{B_l\}_{l \in \mathbb{L}} ,\{\Phi_l\}_{l \in \mathbb{L}}) $ in $ \operatorname{Hom}^*_\mathscr{A}(\mathscr{E}_0, \mathscr{E}_1).$ Then 
\begin{enumerate}[\upshape(i)]
\item $ \theta_C=(I_{\mathscr{H}_\mathscr{A}}\otimes\theta_B)\theta_A, \theta_\Xi=(I_{\mathscr{H}_\mathscr{A}}\otimes\theta_\Phi)\theta_\Psi,  S_{C,\Xi}=\theta_\Psi^*(I_{\mathscr{H}_\mathscr{A}}\otimes S_{B, \Phi})\theta_A,  P_{C, \Xi}=(I_{\mathscr{H}_\mathscr{A}}\otimes\theta_B)\theta_A (\theta_\Psi^*(I_{\mathscr{H}_\mathscr{A}}\otimes S_{B, \Phi})\theta_A)^{-1}\theta_\Psi^*(I_{\mathscr{H}_\mathscr{A}}\otimes\theta_\Phi^*).$ If $( \{A_j\}_{j \in \mathbb{J}},\{\Psi_j\}_{j \in \mathbb{J}})$ and $( \{B_l\}_{l \in \mathbb{J}} , \{\Phi_l\}_{l \in \mathbb{J}}) $   are Parseval frames, then $(\{C_{(l, j)}\}_{(l, j)\in \mathbb{J}\bigtimes  \mathbb{J}}, \{\Xi_{(l,j)}\}_{(l,j)\in \mathbb{J}\bigtimes \mathbb{J}})$ is Parseval.
\item If $P_{A, \Psi}$ commutes with  $I_{\mathscr{H}_\mathscr{A}}\otimes S_{B, \Phi} $, then $ P_{C, \Xi}=(I_{\mathscr{H}_\mathscr{A}}\otimes\theta_B)P_{A,\Psi} (I_{\mathscr{H}_\mathscr{A}}\otimes S_{B, \Phi}^{-1})P_{A,\Psi}(I_{\mathscr{H}_\mathscr{A}}\otimes\theta_\Phi^*) $.
\end{enumerate}
\end{proposition}
\textbf{Tensor product  of frames}: Let $(\{A_j\}_{j \in \mathbb{J}}, \{\Psi_j\}_{j \in \mathbb{J}})$ be a (hvf) in  $ \operatorname{Hom}^*_\mathscr{A}(\mathscr{E}, \mathscr{E}_0),$ and $(\{B_l\}_{l \in \mathbb{L}}, \{\Phi_l\}_{l \in \mathbb{L}})$ be a  (hvf) in  $ \operatorname{Hom}^*_\mathscr{A}(\mathscr{E}_1, \mathscr{E}_2).$ The (hvf)   $(\{C_{(j, l)}\coloneqq A_j\otimes B_l\}_{(j, l)\in \mathbb{J}\bigtimes  \mathbb{L}},\{\Xi_{(j, l)}\coloneqq \Psi_j\otimes\Phi_l\}_{(j, l)\in \mathbb{J}\bigtimes  \mathbb{L}})$  in $ \operatorname{Hom}^*_\mathscr{A}(\mathscr{E}\otimes\mathscr{E}_1, \mathscr{E}_0\otimes\mathscr{E}_2)$ is called  as tensor product  of frames $ (\{A_j\}_{j \in \mathbb{J}}, \{\Psi_j\}_{j\in \mathbb{J}})$ and $( \{B_l\}_{l \in \mathbb{L}} ,\{\Phi_l\}_{l\in \mathbb{L}}).$
\begin{proposition}
Let $(\{C_{(j, l)}\coloneqq A_j\otimes B_l\}_{(j, l)\in \mathbb{J}\bigtimes  \mathbb{L}},\{\Xi _{(j, l)}\coloneqq \Psi_j\otimes \Phi_l\}_{(j, l)\in \mathbb{J}\bigtimes  \mathbb{L}} )$ be the  tensor product of  homomorphism-valued frames  $( \{A_j\}_{j \in \mathbb{J}}, \{\Psi_j\}_{j \in \mathbb{J}} )$  in  $ \operatorname{Hom}^*_\mathscr{A}(\mathscr{E}, \mathscr{H}_\mathscr{A}(\mathbb{L})),$ and $( \{B_l\}_{l \in \mathbb{L}},\{\Phi_l\}_{l \in \mathbb{L}}) $ in $ \operatorname{Hom}^*_\mathscr{A}(\mathscr{E}_1, \mathscr{E}_2).$ Then  $\theta_C=\theta_A\otimes\theta_B, \theta_\Xi=\theta_\Psi\otimes\theta_\Phi, S_{C, \Xi}=S_{A, \Psi}\otimes S_{B, \Phi}, P_{C, \Xi}=P_{A, \Psi}\otimes P_{B, \Phi}.$ If  $( \{A_j\}_{j \in \mathbb{J}},\{\Psi_j\}_{j \in \mathbb{J}}) $ and $( \{B_l\}_{l \in \mathbb{L}}, \{\Phi_l\}_{l \in \mathbb{L}}), $ are Parseval, then $(\{C_{(j, l)}\}_{(j, l)\in \mathbb{J}\bigtimes  \mathbb{L}}, \{\Xi_{(j,l)}\}_{(j,l)\in \mathbb{J}\bigtimes \mathbb{L}})$ is Parseval.
\end{proposition}

\textbf{Perturbations in Hilbert C*-modules}
\begin{theorem}\cite{HANMOHAPATRA1}
Let $ \mathscr{E}$ be a finitely or countably generated Hilbert C*-module over  $ \mathscr{A}$ and  $ \{x_n\}_{n=1}^\infty$ be a frame for $ \mathscr{E}$ with bounds   $ a$ and $ b$. Suppose that $\{y_n\}_{n=1}^\infty$ is in $ \mathscr{E}$ and that there exist $ \alpha, \beta, \gamma \geq 0$ with $ \max\{\alpha+\frac{\gamma}{\sqrt{a}}, \beta\}<1$ and one of the following two conditions is satisfied: 
\begin{equation*}
\left\| \sum \limits_{n=1}^\infty\langle x, x_n-y_n \rangle \langle x, x_n-y_n \rangle^*\right\|^\frac{1}{2}\leq \alpha\left\| \sum \limits_{n=1}^\infty\langle x, x_n \rangle \langle x, x_n \rangle^*\right\|^\frac{1}{2}+\beta\left\| \sum \limits_{n=1}^\infty\langle x, y_n \rangle \langle x, y_n \rangle^*\right\|^\frac{1}{2}+\gamma\|x\|, \quad \forall x \in \mathscr{E}
\end{equation*}
or
\begin{equation*}
\left\| \sum \limits_{n=1}^m c_n(x_n-y_n)\right\| \leq \left\| \sum \limits_{n=1}^m c_nx_n\right\|+\left\| \sum \limits_{n=1}^m c_ny_n\right\|+\left\| \sum \limits_{n=1}^m c_nc^*_n\right\|^\frac{1}{2},~\forall c_j \in \mathscr{A}, m=1,2,....
\end{equation*}
Then $ \{y_n\}_{n=1}^\infty$ is a frame for $ \mathscr{E}$ with bounds $a\left(1-\frac{\alpha+\beta+\frac{\gamma}{\sqrt{a}}}{1+\beta}\right)^2 $ and $b\left(1+\frac{\alpha+\beta+\frac{\gamma}{\sqrt{b}}}{1-\beta}\right)^2 .$
\end{theorem}
\begin{theorem}
Let $(\{A_j\}_{j\in \mathbb{J}}, \{\Psi_j\}_{j\in \mathbb{J}} )$ be a (hvf) in  $ \operatorname{Hom}^*_\mathscr{A}(\mathscr{E}, \mathscr{E}_0)$. Suppose  $\{B_j\}_{j\in \mathbb{J}} $ in $ \operatorname{Hom}^*_\mathscr{A}(\mathscr{E}, \mathscr{E}_0)$ is such that $ \Psi_j^*B_j\geq 0, \forall j \in \mathbb{J}$ and there exist $\alpha, \beta, \gamma \geq 0  $ with $ \max\{\alpha+\gamma\|\theta_\Psi S_{A,\Psi}^{-1}\|, \beta\}<1$ and for every finite subset $ \mathbb{S}$ of $ \mathbb{J}$
 \begin{equation*}
 \left\|\sum\limits_{j\in \mathbb{S}}(A_j^*-B_j^*)L_j^*y\right\|\leq \alpha\left\|\sum\limits_{j\in \mathbb{S}}A_j^*L_j^*y\right\|+\beta\left\|\sum\limits_{j\in \mathbb{S}}B_j^*L_j^*y\right\|+\gamma \left\|\sum\limits_{j\in \mathbb{S}}\langle L_j^*y, L_j^*y\rangle\right\|^\frac{1}{2},\quad \forall y \in \mathscr{H}_\mathscr{A}\otimes \mathscr{E}_0.
 \end{equation*} 
 Then  $( \{B_j\}_{j\in \mathbb{J}} ,\{\Psi_j\}_{j\in \mathbb{J}}) $ is a (hvf) with bounds $ \frac{1-(\alpha+\gamma\|\theta_\Psi S_{A,\Psi}^{-1}\|)}{(1+\beta)\|S_{A,\Psi}^{-1}\|}$ and $\frac{\|\theta_\Psi\|((1+\alpha)\|\theta_A\|+\gamma)}{1-\beta} $.
 \end{theorem}
 \begin{proof}
 Proof is along with the lines of Theorem \ref{PERTURBATION RESULT 1}, by observing
 $$ \langle y,y\rangle =\langle (I_{\ell^2(\mathbb{J})}\otimes I_{\mathcal{H}_0})y,y\rangle=\left\langle\sum\limits_{j\in \mathbb{J}}L_jL_j^* y,y\right\rangle=\sum\limits_{j\in \mathbb{J}}\langle L_j^*y, L_j^*y\rangle, \quad \forall y \in \mathscr{H}_\mathscr{A}\otimes \mathscr{E}_0.$$
 \end{proof}
 \begin{theorem}
 Let $ (\{A_j\}_{j\in \mathbb{J}},\{\Psi_j\}_{j\in \mathbb{J}} )$ be a (hvf) in $\operatorname{Hom}^*_ \mathscr{A}(\mathscr{E}, \mathscr{E}_0)$ with bounds $ a$ and $b$. Suppose  $\{B_j\}_{j\in \mathbb{J}} $ is Bessel (w.r.t. itself) in $ \operatorname{Hom}^*_\mathscr{A}(\mathscr{E}, \mathscr{E}_0)$  such that $ \theta_\Psi^*\theta_B\geq0$ and 
 there exist real $\alpha, \beta, \gamma \geq 0  $ with $ \max\{\alpha+\frac{\gamma}{\sqrt{a}}, \beta\}<1$ and 
\begin{equation*} 
\left\|\sum\limits_{j\in \mathbb{J}}\langle(A_j-B_j)x,\Psi_jx \rangle\right\|^\frac{1}{2}\leq \alpha \left\|\sum\limits_{j\in \mathbb{J}}\langle A_jx,\Psi_jx \rangle\right\|^\frac{1}{2}+\beta\left\|\sum\limits_{j\in \mathbb{J}}\langle B_jx,\Psi_jx \rangle\right\|^\frac{1}{2}+\gamma \|x\|,\quad  \forall x \in \mathscr{E}.
\end{equation*}
Then  $ (\{B_j\}_{j\in \mathbb{J}}, \{\Psi_j\}_{j\in \mathbb{J}}) $ is a (hvf) with bounds $a\left(1-\frac{\alpha+\beta+\frac{\gamma}{\sqrt{a}}}{1+\beta}\right)^2 $ and $b\left(1+\frac{\alpha+\beta+\frac{\gamma}{\sqrt{b}}}{1-\beta}\right)^2 .$
 \end{theorem}
 \begin{proof}
 By replacing $ |\cdot|$  and $ (\cdot)$  by $ \|\cdot\|$ in Theorem \ref{PERTURBATION RESULT 2} and using  Theorem \ref{ARAMBASIC1}  we get this theorem.
 \end{proof}
 
\section{Sequential version of homomorphism-valued frames and bases} \label{SEQUENTIALVERSIONOFHOMOMORPHISM-VALUEDFRAMES}
\begin{definition}\label{SEQUENTIAL VERSION HOMOMORPHISM VAUED DEFINITION}
A set of vectors   $ \{x_j\}_{j\in \mathbb{J}}$  in a Hilbert C*-module $\mathscr{E}$ is said to be a frame    w.r.t.  set $ \{\tau_j\}_{j\in \mathbb{J}}$ in $\mathscr{E}$  if there are  real $ c,  d  >0$ such that
\begin{enumerate}[\upshape(i)]
\item the map  $S_{x,\tau}:\mathscr{E} \ni  x \mapsto \sum_{j\in\mathbb{J}}\langle x,  x_j\rangle\tau_j \in \mathscr{E} $ is well-defined  adjointable  positive  invertible homomorphism.
\item $   \sum_{j \in \mathbb{J}}\langle x,x_j\rangle\langle x_j, x \rangle \leq c\langle x,x\rangle , \forall x \in \mathscr{E}; 
\sum_{j \in \mathbb{J}}\langle x,\tau_j\rangle\langle \tau_j, x \rangle \leq d\langle x,x\rangle , \forall x \in \mathscr{E}.$
\end{enumerate}
Let $a,b>0 $ be such that $ a I_\mathscr{E}\leq S_{x,\tau}\leq b I_\mathscr{E}$. We call $ a$ and $ b$ as lower and upper frame bounds, respectively. Supremum of the set of all lower frame bounds is called the optimal lower frame bound and infimum of the set of all upper frame bounds is called the optimal upper frame bound. If optimal bounds are equal, we say the frame is exact. Whenever optimal bound of an exact frame is one,  we say it is Parseval.
The bounded homomorphisms $\theta_x: \mathscr{E}\ni x\mapsto\{\langle x,x_j\rangle\}_{j\in\mathbb{J}}\in \mathscr{H}_\mathscr{A}, $ $\theta_\tau : \mathscr{E}\ni x\mapsto\{\langle x,\tau_j\rangle\}_{j\in\mathbb{J}}\in \mathscr{H}_\mathscr{A}$ as analysis homomorphisms and adjoints of these as synthesis homomorphisms. The map defined by $S_{x,\tau}: \mathscr{E}\ni x\mapsto  \sum_{j \in \mathbb{J}}\langle x,x_j\rangle \tau_j \in \mathscr{E}$ is called as frame homomorphism.
If condition \text{\upshape(i)} is replaced by
$$ \text{ the map } \mathscr{E} \ni  x \mapsto \sum_{j\in\mathbb{J}}\langle x,  x_j\rangle\tau_j \in \mathscr{E}  ~\text{is well-defined  adjointable positive homomorphism},$$
then we call $\{x_j\}_{j\in\mathbb{J}}$ is  Bessel w.r.t. $\{\tau_j\}_{j\in\mathbb{J}}$.
      
For fixed $ \mathbb{J}, \mathscr{E}$  and $ \{\tau_j\}_{j\in \mathbb{J}}$   the set of all frames for $ \mathscr{E}$  w.r.t.  $ \{\tau_j\}_{j\in \mathbb{J}}$ is denoted by $ \mathscr{F}_\tau.$
\end{definition}
\begin{proposition}
 Definition \ref{SEQUENTIAL VERSION HOMOMORPHISM VAUED DEFINITION} holds if and only if  there are  real $ a, b, c,  d  >0$ such that
\begin{enumerate}[\upshape(i)]
\item $ a\langle x,x\rangle\leq \sum_{j \in \mathbb{J}}\langle x,x_j\rangle\langle \tau_j, x \rangle\leq b\langle x,x\rangle , \forall x \in \mathscr{E}.$
\item $   \sum_{j \in \mathbb{J}}\langle x,x_j\rangle\langle x_j, x \rangle \leq c\langle x,x\rangle ,\forall x \in \mathscr{E}; 
\sum_{j \in \mathbb{J}}\langle x,\tau_j\rangle\langle \tau_j, x \rangle \leq d\langle x,x\rangle , \forall x \in \mathscr{E}.$
\item $  \sum_{j \in \mathbb{J}}\langle x,x_j\rangle\tau_j=\sum_{j \in \mathbb{J}}\langle x,\tau_j\rangle x_j,  \forall x \in \mathscr{E}.$
\end{enumerate}
\end{proposition}
\begin{theorem}
 Let $\{x_j\}_{j\in \mathbb{J}}, \{\tau_j\}_{j\in \mathbb{J}}$ be in $\mathscr{E}$ over $\mathscr{A}$. Define $A_j: \mathscr{E} \ni x \mapsto \langle x, x_j \rangle \in \mathscr{A} $, $\Psi_j: \mathscr{E} \ni x \mapsto \langle x, \tau_j \rangle \in \mathscr{A}, \forall j \in \mathbb{J} $. Then   $(\{x_j\}_{j\in \mathbb{J}}, \{\tau_j\}_{j\in \mathbb{J}})$ is a frame for  $\mathscr{E}$ if and only if  $(\{A_j\}_{j\in \mathbb{J}}, \{\Psi_j\}_{j\in \mathbb{J}})$ is a homomorphism-valued frame  in $\operatorname{Hom}_\mathscr{A}^*(\mathscr{E},\mathscr{A})$.
\end{theorem} 
 \begin{proposition}
Let  $( \{x_j\}_{j \in \mathbb{J}},\{\tau_j\}_{j \in \mathbb{J}} )$ be a frame for  $ \mathscr{E}$  with upper frame bound $b$. If for some $ j \in \mathbb{J} $ we have  $  \langle x_j, x_l \rangle\langle \tau_l, x_j \rangle \geq0, \forall l  \in \mathbb{J},$ then $ \|\langle x_j, \tau_j \rangle\|\leq b$ for that $j. $
\end{proposition}
\begin{proof}
 We take norm on $\langle x_j, x_j \rangle\langle \tau_j, x_j \rangle \leq \sum_{l\in \mathbb{J}}\langle x_j, x_l \rangle\langle \tau_l, x_j \rangle\leq b\langle x_j, x_j\rangle. $
\end{proof}
\begin{definition}\cite{FRANKLARSON1}.
A  collection $ \{x_j\}_{j \in \mathbb{J}} $  in  $ \mathscr{E}$ over $\mathscr{A}$ is said to be  orthogonal  if $ \langle x_j,x_k \rangle=0, \forall j , k \in \mathbb{J}, j\neq k.$	
\end{definition}
 \begin{definition}\cite{MANUILOV1}
A  collection $ \{x_j\}_{j \in \mathbb{J}} $  in  $ \mathscr{E}$ over $(\mathscr{A},e)$ is said to be  orthonormal  if $ \langle x_j,x_k \rangle=\delta_{j,k}e, \forall j , k \in \mathbb{J}.$
\end{definition}
\begin{proposition}
Let $ \{x_j\}_{j \in \mathbb{J}}$  be   orthogonal for   $ \mathscr{E}$ over $(\mathscr{A},e)$. If  $ \langle x_j,x_j \rangle$'s are invertible (in $ \mathscr{A}$)  for all $ j \in \mathbb{J}$, then $ \{y_j \coloneqq \langle x_j,x_j \rangle^{-1/2}x_j\}_{j \in \mathbb{J}}$ is  orthonormal   for $ \mathscr{E}$ and  $\overline{\operatorname{span}}_{\mathscr{A}}\{x_j\}_{j \in \mathbb{J}}=\overline{\operatorname{span}}_{\mathscr{A}}\{x_j\}_{j \in \mathbb{J}}$.
\end{proposition}
\begin{proof}
Let $ j, k \in \mathbb{J}$. If $ j=k$, then $\langle y_j,y_j \rangle=\langle \langle x_j,x_j \rangle^{-1/2}x_j,\langle x_j,x_j \rangle^{-1/2}x_j \rangle = \langle x_j,x_j \rangle^{-1/2}\langle x_j,x_j \rangle\langle x_j,x_j \rangle^{-1/2}=e$. If $ j\neq k$, then $\langle y_j,y_k \rangle=\langle \langle x_j,x_j \rangle^{-1/2}x_j,\langle x_k,x_k \rangle^{-1/2}x_k \rangle = \langle x_j,x_j \rangle^{-1/2}\langle x_j,x_k \rangle\langle x_k,x_k \rangle^{-1/2}=0$.
\end{proof}
\begin{theorem}
Let $ \mathscr{E}$ be over  $(\mathscr{A},e)$.
\begin{enumerate}[\upshape(i)]
\item If $ \{x_n\}_{n=1}^m $ is  orthogonal  for $ \mathscr{E}$, then $ \langle \sum_{n=1}^{m}x_n, \sum_{k=1}^{m}x_k\rangle =\sum_{n=1}^{m}\langle x_n, x_n \rangle$. In particular, $ \|\sum_{n=1}^{m}x_n\|^2=\|\sum_{n=1}^{m}\langle x_n, x_n \rangle\|^2$ and  $ \|\sum_{n=1}^{m}x_n\|^2\leq\sum_{n=1}^{m}\| x_n\|^2.$
\item If $ \{x_n\}_{n=1}^m $ is  orthonormal  for $ \mathscr{E}$, then $ \langle \sum_{n=1}^{m}x_n, \sum_{k=1}^{m}x_k\rangle=me $. In particular, $ \| \sum_{n=1}^{m}x_n\|= \sqrt{m}$.
\item If $ \{x_j\}_{j \in \mathbb{J}} $ is  orthogonal   for  $ \mathscr{E}$ such that $\langle x_j, x_j \rangle $ is invertible (in $ \mathscr{A}$), $\forall j \in \mathbb{J}$,  then $ \{x_j\}_{j \in \mathbb{J}} $ is linearly independent over $ \mathscr{A}$. In particular, if $ \{x_j\}_{j \in \mathbb{J}} $ is  orthonormal   in  $ \mathscr{E}$, then it is linearly independent over $ \mathscr{A}$.
\end{enumerate}
\end{theorem}
\begin{proof}
(i) and (ii) are clear, we prove (iii). Let $ \mathbb{S}$ be a finite subset of $ \mathbb{J}$ and $ a_j \in \mathscr{A}, j \in \mathbb{S}$  be such that $ \sum_{j\in\mathbb{S}}a_jx_j=0$. Then for each  fixed $ k \in  \mathbb{S} $, we get $a_k\langle x_k, x_k \rangle=\langle\sum_{j\in\mathbb{S}}c_jx_j, x_k \rangle=0\Rightarrow a_k=0$. When  $ \{x_j\}_{j \in \mathbb{J}} $ is orthonormal,  $\langle x_j, x_j \rangle=e, \forall j \in \mathbb{J} $,  hence it is linearly independent over $ \mathscr{A}$.
\end{proof}
\begin{proposition}
If $ \{x_j\}_{j \in \mathbb{J}} $ is orthonormal for  $ \mathscr{E}$ over  $(\mathscr{A},e)$, then it  is closed.
\end{proposition}
\begin{proof}
Whenever $\{ x_{j_n}\}_{n=1}^\infty$	is in  $ \{x_j\}_{j \in \mathbb{J}} $ converging to an element $ x \in \mathcal{X}$. Then $ \|x_{j_l}-x_{j_m}\|^p=\|2e\|^p=2, \forall  x_{j_l}\neq x_{j_m}$. This implies, from the Cauchyness of  $\{ x_{j_n}\}_{n=1}^\infty$, that $x_{j_n}=x$  except for finitely many $ j_n$'s.
\end{proof}
\begin{proposition}
Let   $ \mathscr{E}$ be  over  $(\mathscr{A},e)$ having a dense subset indexed with $ \mathbb{J}$. If    $ \{y_l\}_{l \in \mathbb{L}} $ is  orthonormal  for  $ \mathscr{E}$, then  $ \operatorname{Card}(\mathbb{L}) \leq  \operatorname{Card}(\mathbb{J})$.
\end{proposition}
\begin{proof}
Let $ \{x_j\}_{j \in \mathbb{J}} $ be dense in $ \mathscr{E}$. For all $ y_l\neq y_m$, we have $ \|y_l-y_m\|^2=\|\langle y_l-y_m,y_l-y_m\rangle \|=\|\langle y_l,y_l\rangle +\langle y_m,y_m\rangle \|=\|e+e\|=2$. Hence the collection of balls $ \{B(y_l,1/\sqrt{2})\}_{l \in \mathbb{L}}$ is disjoint, where $B(y_l,1/\sqrt{2})\coloneqq\{ x \in \mathscr{E}:\|x-y_l\|<1/\sqrt{2}\} $. Since $ \{x_j\}_{j \in \mathbb{J}} $ is  dense, it must contain atleast one point from each of the balls  $B(y_l,1/\sqrt{2})$. Thus $ \operatorname{Card}(\mathbb{L}) \leq  \operatorname{Card}(\mathbb{J})$.
\end{proof}
\begin{corollary}
If $ \mathscr{E}$ is separable, then every orthonormal set  for  $ \mathscr{E}$ is also separable.
\end{corollary}
\begin{definition}
We say  a  collection $ \{x_j\}_{j \in \mathbb{J}} $ in   $ \mathscr{E}$ over  $(\mathscr{A},e)$ is  
\begin{enumerate}[\upshape(i)]
\item a Hamel basis for $ \mathscr{E}$ if every element of  \ $\mathscr{E}$ can be written uniquely as a finite linear combination of $x_j$'s over $\mathscr{A}$.
\item a  basis  for $ \mathscr{E}$ if every $ x \in \mathscr{E}$ has a  unique representation $x= \sum_{j\in \mathbb{J}} c_jx_j  ,c_j \in\mathscr{A}, \forall j \in \mathbb{J}$, convergence is in the norm  of  $ \mathscr{E}.$
\item an orthonormal basis  for $ \mathscr{E}$ if $ \{x_j\}_{j \in \mathbb{J}} $ is both  orthonormal set and a basis.
\item a Riesz basis  for $ \mathscr{E}$ if there exist an orthonormal basis $ \{f_j\}_{j \in \mathbb{J}} $ for  $ \mathscr{E}$ and  an invertible $ U \in \operatorname{End}_\mathscr{A}^*(\mathscr{E})$ such that $ x_j=Uf_j, \forall j \in \mathbb{J}.$
\end{enumerate}
\end{definition}
 Clearly,  if $ \{x_j\}_{j \in \mathbb{J}} $ is an orthonormal  basis and $x= \sum_{j\in \mathbb{J}} c_jx_j$, then (i) $ c_j$'s are precisely given by $ c_j=\langle x , x_j \rangle , \forall j \in \mathbb{J}$, (ii) $\langle x , x \rangle= \sum_{j\in\mathbb{J}}\langle x , x_j \rangle\langle x_j , x \rangle,\forall x \in \mathscr{E} $. Hence an orthonormal basis is a Bessel sequence.

 \begin{lemma}\label{BESSELCONVERGENCE}
 If $ \{x_j\}_{j \in \mathbb{J}} $ is a Bessel sequence for  $ \mathscr{E}$, then 
 \begin{enumerate}[\upshape(i)]
 \item $\sum_{j \in \mathbb{J}}a_j x_j $ converges in $ \mathscr{E}, \forall \{a_j\}_{j\in \mathbb{J}} \in \mathscr{H}_\mathscr{A}$ and 
 \item the map $\mathscr{H}_\mathscr{A} \ni \{a_j\}_{j\in \mathbb{J}} \mapsto \sum_{j \in \mathbb{J}}a_j x_j \in \mathscr{E} $ is well-defined bounded homomorphism.
\end{enumerate}
 \end{lemma}
 \begin{proof}
 \begin{enumerate}[\upshape(i)]
 \item  Let $ b>0$ be such that $ \sum_{j \in \mathbb{J}}\langle x, x_j \rangle \langle  x_j, x \rangle\leq b \langle x, x\rangle, \forall x \in \mathscr{E}.$ For  finite $\mathbb{S}\subseteq\mathbb{J} $ we have 
 \begin{align*}
 \left\|\sum_{j \in \mathbb{S}}a_j x_j\right\|&=\sup_{y \in \mathscr{E}, \|y\|=1}\left\|\left\langle \sum_{j \in \mathbb{S}}a_jx_j, y\right\rangle \right\|=\sup_{y \in \mathscr{E}, \|y\|=1}\left\| \sum_{j \in \mathbb{S}}a_j \langle  x_j, y \rangle \right\|\\
 &\leq \left\|\sum_{j \in \mathbb{S}}a_ja_j^*\right\|^\frac{1}{2}\sup_{y \in \mathscr{E}, \|y\|=1}\left\| \sum_{j \in \mathbb{S}}\langle y, x_j \rangle \langle  x_j, y \rangle \right\|^\frac{1}{2}\leq \sqrt{b}\left\|\sum_{j \in \mathbb{S}}a_ja_j^*\right\|^\frac{1}{2}
 \end{align*}
 which is convergent.
 \item follows from (i).  Note that the norm of  homomorphism is less than or equal to $  \sqrt{b}$.
\end{enumerate}
\end{proof}	
\begin{theorem}\label{CHARACTERIZATION3}
 Let   $ \{x_j\}_{j \in \mathbb{J}} $  be  orthonormal  for   $ \mathscr{E}$ over  $\mathscr{A}$. If  $ \{x_j\}_{j \in \mathbb{J}} $ is  Bessel, then  the following are equivalent.
 \begin{enumerate}[\upshape(i)]
 \item $ \{x_j\}_{j \in \mathbb{J}} $ is an orthonormal basis for $ \mathscr{E}$.
 \item $ x= \sum_{j\in \mathbb{J}}\langle x , x_j \rangle x_j, \forall x \in \mathscr{E}.$
 \item $ \langle x, y\rangle =\sum_{j\in \mathbb{J}}\langle x , x_j \rangle \langle x_j, y \rangle,  \forall x, y \in \mathscr{E}.$
 \item $ \langle x, x\rangle =\sum_{j\in \mathbb{J}}\langle x , x_j \rangle \langle x_j, x \rangle,  \forall x \in \mathscr{E}.$
 \item $ \overline{\operatorname{span}}_{\mathscr{A} }\{x_j\}_{j \in \mathbb{J}}= \mathscr{E}.$
 \item If $\langle x, x_j\rangle =0, \forall j \in \mathbb{J},$ then $ x=0.$
 \end{enumerate}
 \end{theorem}
 \begin{proof}
 (i) $\Rightarrow $ (ii)  $\Rightarrow $ (iii) $\Rightarrow $ (iv) are direct. For (iv) $\Rightarrow $ (v), we simply note that for each $ x $ in $ \mathscr{E}$, the net  $\{\sum_{j\in \mathbb{S}}\langle x , x_j \rangle x_j:\mathbb{S} \subseteq \mathbb{J}, \mathbb{S} ~\text{is finite}\} $ converges to $x$. Now, for (v) $\Rightarrow $ (vi), let $\langle x, x_j\rangle =0, \forall j \in \mathbb{J}$.  Choose a sequence $ \{y_n\}_{n=1}^\infty$ in  $ \operatorname{span}_{\mathscr{A}}\{x_j\}_{j \in \mathbb{J}}$ such that $ y_n\rightarrow x$ as $n \rightarrow \infty$. Then $ 0=\lim_{n\rightarrow \infty}\langle x, y_n\rangle = \langle x, \lim_{n\rightarrow \infty} y_n\rangle =\langle x, x\rangle.$ Thus $x=0$.
 Finally for (vi) $\Rightarrow $ (i),  we first find that - since $ \{x_j\}_{j \in \mathbb{J}} $ is Bessel, $ \{\langle x, x_j\rangle \}_{j\in \mathbb{J}} \in\mathscr{H}_\mathscr{A}. $  Therefore from Lemma \ref{BESSELCONVERGENCE},  $\sum_{j \in \mathbb{J}}\langle x, x_j \rangle x_j $ converges in $ \mathscr{E}, \forall x \in \mathscr{E}.$ Let $x \in \mathscr{E} $ be fixed. Define $y\coloneqq\sum_{j \in \mathbb{J}}\langle x, x_j \rangle x_j.$ Then $ \langle y-x, x_j\rangle =\langle x, x_j\rangle-\langle x, x_j\rangle =0, \forall j \in \mathbb{J}.$ Therefore (vi) gives $ y=x.$ If $ x$ also has representation $ \sum_{j \in \mathbb{J}}a_j x_j,$ for some $ a_j \in \mathscr{A}, j\in \mathbb{J}$, then $ \langle x, x_k \rangle =\langle \sum_{j \in \mathbb{J}}a_j x_j, x_k \rangle =a_k, \forall k \in \mathbb{J}.$ Thus $ \{x_j\}_{j \in \mathbb{J}} $ is an orthonormal basis for $\mathscr{E}.$
 \end{proof}
\begin{corollary}
Let   $ \{x_j\}_{j \in \mathbb{J}} $  be  Bessel  for   $ \mathscr{E}$ over  $(\mathscr{A},e)$ such that $\langle x, x\rangle =e, \forall j \in \mathbb{J} $, and 
\begin{align}\label{COROLLARYEQUALTOE}
\langle x, x\rangle =\sum_{j\in \mathbb{J}}\langle x , x_j \rangle \langle x_j, x \rangle,  ~\forall x \in \mathscr{E}.
\end{align}
Then     $ \{x_j\}_{j \in \mathbb{J}} $  is an  orthonormal  basis for   $ \mathscr{E}$.
\end{corollary}
\begin{proof}
For each fixed $ k \in \mathbb{J}$, Equation (\ref{COROLLARYEQUALTOE})  gives $e=\langle x_k, x_k\rangle =\sum_{j\in \mathbb{J}}\langle x_k , x_j \rangle \langle x_j, x_k \rangle=e+\sum_{j\in \mathbb{J}, j \neq k}\langle x_k , x_j \rangle \langle x_j, x_k \rangle$. Since $ \langle x_k , x_j \rangle \langle x_j, x_k \rangle \geq 0, \forall k$, we must get $ \langle x_k , x_j \rangle=0,\forall k \neq j$. Thus $ \{x_j\}_{j \in \mathbb{J}} $ is orthonormal. Now Theorem  \ref{CHARACTERIZATION3}  applies and completes the proof.
\end{proof}
Note that, in  Theorem \ref{CHARACTERIZATION3},  only in proving (vi) implies (i) we have used the fact that  $ \{x_j\}_{j \in \mathbb{J}} $ is Bessel. Hence we have 
\begin{corollary}\label{CHARACTERIZATION4}
Let $ \{x_j\}_{j \in \mathbb{J}} $  be orthonormal in  $ \mathscr{E}$. Then $\text{\upshape(i)}\Rightarrow\text{\upshape(ii)}\Rightarrow\text{\upshape(iii)}\Rightarrow \text{\upshape(iv)}\Rightarrow \text{\upshape(v)}\Rightarrow \text{\upshape(vi)}$, where \text{\upshape(i)} $ \{x_j\}_{j \in \mathbb{J}} $ is an orthonormal basis for $ \mathscr{E}$,
\text{\upshape(ii)} $x= \sum_{j\in \mathbb{J}}\langle x , x_j \rangle x_j, \forall x \in \mathscr{E}$,
\text{\upshape(iii)} $ \langle x, y\rangle =\sum_{j\in \mathbb{J}}\langle x , x_j \rangle \langle x_j, y \rangle,  \forall x, y \in \mathscr{E}$,
\text{\upshape(iv)} $ \langle x, x\rangle =\sum_{j\in \mathbb{J}}\langle x , x_j \rangle \langle x_j, x \rangle,  \forall x \in \mathscr{E}$,
\text{\upshape(v)} $ \overline{\operatorname{span}}_{\mathscr{A} }\{x_j\}_{j \in \mathbb{J}}= \mathscr{E}$,
\text{\upshape(vi)} If $\langle x, x_j\rangle =0, \forall j \in \mathbb{J},$ then $ x=0.$
 \end{corollary}
Result which we promised after Theorem \ref{OPERATORVERSIONCHARACTERIZATIONMODULES} is here.
 
Eventhough we do not have Riesz representation theorem for Hilbert C*-modules, we can  derive an analogous to  Theorem \ref{SEQUENTIAL CHARACTERIZATION}. 

\begin{theorem}\label{HOMOMORPHISMTOSEQUENTIALVICEVERSAMODULE}
Let $ \{A_j\}_{j\in\mathbb{J}}, \{\Psi_j\}_{j\in\mathbb{J}}$ be in  $ \operatorname{Hom}^*_\mathscr{A}(\mathscr{E}, \mathscr{E}_0).$  Suppose $ \{e_{j,k}\}_{k\in\mathbb{L}_j}$ is an orthonormal basis for $ \mathscr{E}_0,$ for each $j \in \mathbb{J}.$ Let  $ u_{j,k}=A_j^*e_{j,k}, v_{j,k}=\Psi_j^*e_{j,k}, \forall k \in  \mathbb{L}_j, j\in \mathbb{J}.$ Then $ (\{A_j\}_{j\in\mathbb{J}}, \{\Psi_j\}_{j\in\mathbb{J}})$ is 
\begin{enumerate}[\upshape(i)]
\item   an orthonormal set (resp. basis)  in $ \operatorname{Hom}^*_\mathscr{A}(\mathscr{E}, \mathscr{E}_0)$ if and only if  $ \{u_{j,k}: k\in \mathbb{L}_j, j \in \mathbb{J}\} $ or $ \{v_{j,k}: k\in \mathbb{L}_j, j \in \mathbb{J}\} $ is an orthonormal set (resp. basis) for $\mathscr{E}$, say $ \{u_{j,k}: k\in \mathbb{L}_j, j \in \mathbb{J}\} $ is an orthonormal set (resp. basis) and there is a sequence  $ \{c_j\}_{j \in \mathbb{J}} $ of positive invertible elements in the center of $ \mathscr{A}$  such that $ 0<\inf\{\|c_j\|\}_{ j\in \mathbb{J}}\leq\sup\{\|c_j\|\}_{j\in \mathbb{J}}<\infty$ and $v_{j,k}=c_ju_{j,k} , \forall k \in  \mathbb{L}_j, ~\text{for each}~j \in  \mathbb{J}.$
\item a Riesz basis in $ \operatorname{Hom}^*_\mathscr{A}(\mathscr{E}, \mathscr{E}_0)$ if and only if  there are  invertible $ U,V \in \operatorname{End}^*_\mathscr{A}(\mathscr{E}) $  and an orthonormal basis $ \{f_{j,k}: k\in \mathbb{L}_j, j \in \mathbb{J}\} $ for $\mathscr{E}$ such that $ VU^*\geq0$ and $u_{j,k}=Uf_{j,k}, v_{j,k}=Vf_{j,k},  \forall k \in  \mathbb{L}_j, j\in \mathbb{J}. $
\item  a (hvf)  in $ \operatorname{Hom}^*_\mathscr{A}(\mathscr{E}, \mathscr{E}_0)$  with bounds $a $ and $ b$  if and only if there exist $ c,d >0$ such that the map 
$$ T: \mathscr{E} \ni x \mapsto\sum\limits_{j\in \mathbb{J}}\sum\limits_{k\in \mathbb{L}_j}\langle x, u_{j,k}\rangle v_{j,k} \in  \mathscr{E} $$
is a well-defined adjointable  positive invertible homomorphism  such that $ a\langle x, x \rangle\leq \langle Tx,x \rangle \leq b\langle x, x \rangle ,  \forall x \in \mathscr{E} $, and 
$$ \left\| \sum\limits_{j\in \mathbb{J}}\sum\limits_{k\in \mathbb{L}_j}\langle x, u_{j,k}\rangle \langle u_{j,k},x \rangle \right\| \leq c\|x\|^2 ,~ \forall x \in \mathscr{E}; ~  \left\|\sum\limits_{j\in \mathbb{J}}\sum\limits_{k\in \mathbb{L}_j} \langle x, v_{j,k}\rangle \langle v_{j,k},x \rangle\right\|\leq d\|x\|^2 ,~ \forall x \in \mathscr{E}.$$
\item Bessel in $ \operatorname{Hom}^*_\mathscr{A}(\mathscr{E}, \mathscr{E}_0)$  with bound  $ b$  if and only if there exist $ c,d >0$ such that the map 
$$ T: \mathscr{E} \ni x \mapsto\sum\limits_{j\in \mathbb{J}}\sum\limits_{k\in \mathbb{L}_j}\langle x, u_{j,k}\rangle v_{j,k} \in  \mathscr{E} $$
is a well-defined adjointable  positive  homomorphism  such that $  \langle Tx,x \rangle \leq b\langle x, x \rangle ,  \forall x \in \mathscr{E} $, and 
$$ \left\| \sum\limits_{j\in \mathbb{J}}\sum\limits_{k\in \mathbb{L}_j}\langle x, u_{j,k}\rangle \langle u_{j,k},x \rangle \right\| \leq c\|x\|^2 ,~ \forall x \in \mathscr{E}; ~  \left\|\sum\limits_{j\in \mathbb{J}}\sum\limits_{k\in \mathbb{L}_j} \langle x, v_{j,k}\rangle \langle v_{j,k},x \rangle\right\|\leq d\|x\|^2 ,~ \forall x \in \mathscr{E}.$$
\item a (hvf) in $ \operatorname{Hom}^*_\mathscr{A}(\mathscr{E}, \mathscr{E}_0)$  with bounds $a $ and $ b$ if and only if there exist $ c,d, r >0$ such that 
$$\left \|\sum\limits_{j\in \mathbb{J}}\sum\limits_{k\in \mathbb{L}_j}\langle x, u_{j,k}\rangle v_{j,k}\right\|\leq r\|x\|,~\forall x \in \mathscr{E}  ; ~\sum\limits_{j\in \mathbb{J}}\sum\limits_{k\in \mathbb{L}_j}\langle x, u_{j,k}\rangle v_{j,k} =\sum\limits_{j\in \mathbb{J}}\sum\limits_{k\in \mathbb{L}_j}\langle x, v_{j,k}\rangle u_{j,k} ,~\forall x \in \mathscr{E} ;$$
$$ a\langle x, x \rangle\leq \sum\limits_{j\in \mathbb{J}}\sum\limits_{k\in \mathbb{L}_j}\langle x, u_{j,k}\rangle \langle  v_{j,k} , x\rangle= \sum\limits_{j\in \mathbb{J}}\sum\limits_{k\in \mathbb{L}_j}\langle x, v_{j,k}\rangle \langle  u_{j,k} , x\rangle \leq b\langle x, x \rangle ,~ \forall x \in \mathscr{E}, ~\text{and} $$
$$ \left\| \sum\limits_{j\in \mathbb{J}}\sum\limits_{k\in \mathbb{L}_j}\langle x, u_{j,k}\rangle \langle u_{j,k},x \rangle \right\| \leq c\|x\|^2 ,~ \forall x \in \mathscr{E}; ~  \left\|\sum\limits_{j\in \mathbb{J}}\sum\limits_{k\in \mathbb{L}_j} \langle x, v_{j,k}\rangle \langle v_{j,k},x \rangle\right\|\leq d\|x\|^2 ,~ \forall x \in \mathscr{E}.$$
\item  Bessel  in $ \operatorname{Hom}^*_\mathscr{A}(\mathscr{E}, \mathscr{E}_0)$  with bound  $ b$ if and only if there exist $ c,d, r >0$ such that 
$$\left \|\sum\limits_{j\in \mathbb{J}}\sum\limits_{k\in \mathbb{L}_j}\langle x, u_{j,k}\rangle v_{j,k}\right\|\leq r\|x\|,~\forall x \in \mathscr{E}  ; ~\sum\limits_{j\in \mathbb{J}}\sum\limits_{k\in \mathbb{L}_j}\langle x, u_{j,k}\rangle v_{j,k} =\sum\limits_{j\in \mathbb{J}}\sum\limits_{k\in \mathbb{L}_j}\langle x, v_{j,k}\rangle u_{j,k} ,~\forall x \in \mathscr{E} ;$$
$$ 0 \leq \sum\limits_{j\in \mathbb{J}}\sum\limits_{k\in \mathbb{L}_j}\langle x, u_{j,k}\rangle \langle  v_{j,k} , x\rangle =\sum\limits_{j\in \mathbb{J}}\sum\limits_{k\in \mathbb{L}_j}\langle x, v_{j,k}\rangle \langle  u_{j,k} , x\rangle \leq b\langle x, x \rangle ,~ \forall x \in \mathscr{E}, ~\text{and} $$
$$ \left\| \sum\limits_{j\in \mathbb{J}}\sum\limits_{k\in \mathbb{L}_j}\langle x, u_{j,k}\rangle \langle u_{j,k},x \rangle \right\| \leq c\|x\|^2 ,~ \forall x \in \mathscr{E}; ~  \left\|\sum\limits_{j\in \mathbb{J}}\sum\limits_{k\in \mathbb{L}_j} \langle x, v_{j,k}\rangle \langle v_{j,k},x \rangle\right\|\leq d\|x\|^2 ,~ \forall x \in \mathscr{E}.$$
\end{enumerate}
\end{theorem}
\begin{proof}
\begin{enumerate}[\upshape(i)]
\item $(\Rightarrow)$ For the set case, arguments are similar to the proof of  Theorem \ref{SEQUENTIAL CHARACTERIZATION}.

Basis case: Now $\{A_j\}_{j\in\mathbb{J}} $  satisfies $ \sum_{j\in\mathbb{J}} \langle A_jx, A_jx\rangle =\langle x, x \rangle , \forall x \in \mathscr{E}$ and this implies 
\begin{align*}
\langle x, x \rangle = \sum_{j\in\mathbb{J}} \left\langle \sum_{k\in \mathbb{L}_j}\langle x, u_{j,k}\rangle e_{j,k},\sum_{m\in \mathbb{L}_j}\langle x, u_{j,m}\rangle e_{j,m}\right\rangle =\sum_{j\in\mathbb{J}} \sum_{k\in \mathbb{L}_j}  \langle x, u_{j,k}\rangle \langle  u_{j,k},x\rangle, ~\forall x \in \mathscr{E}.
\end{align*}
Next we identify $ \{u_{j,k}: k\in \mathbb{L}_j, j \in \mathbb{J}\} $ is a Bessel sequence. In fact,  
\begin{align*}
\sum\limits_{j\in \mathbb{J}}\sum\limits_{k\in \mathbb{L}_j}\langle x, u_{j,k}\rangle \langle u_{j,k},x \rangle=\sum\limits_{j\in \mathbb{J}}\sum\limits_{k\in \mathbb{L}_j}\langle A_jx, e_{j,k}\rangle \langle e_{j,k},A_jx \rangle=\sum\limits_{j\in \mathbb{J}}\langle A_jx, A_jx\rangle=\langle x, x\rangle, ~\forall x \in \mathscr{E}.
\end{align*}
Theorem \ref{CHARACTERIZATION3} now couples and tells  $ \{u_{j,k}: k\in \mathbb{L}_j, j \in \mathbb{J}\}$ it is an orthonormal basis.

$(\Leftarrow)$ We  note 
\begin{align*}
\sum_{j\in \mathbb{J}}\langle A_jx ,A_jx \rangle &=\sum_{j\in \mathbb{J}}\left\langle\sum_{k\in \mathbb{L}_j}\langle A_jx ,e_{j,k} \rangle e_{j,k},  \sum_{m\in \mathbb{L}_j}\langle A_jx ,e_{j,m} \rangle e_{j,m}\right\rangle\\
&= \sum_{j\in\mathbb{J}}\sum_{k\in \mathbb{L}_j}\langle x,u_{j,k} \rangle\langle u_{j,k}, x \rangle\leq  \langle x, x \rangle ,~ \forall x \in \mathscr{E}, 
\end{align*}
 then similar to the proof of Theorem \ref{SEQUENTIAL CHARACTERIZATION}.

Basis case: Again note $\sum_{j\in \mathbb{J}}\langle A_jx , A_jx\rangle = \sum_{j\in\mathbb{J}}\sum_{k\in \mathbb{L}_j}\langle x,u_{j,k} \rangle\langle u_{j,k}, x \rangle = \langle x, x \rangle, \forall x \in \mathscr{E}$,   then similar to the proof of Theorem \ref{SEQUENTIAL CHARACTERIZATION}.

\item Similar to the proof of (ii) in Theorem \ref{SEQUENTIAL CHARACTERIZATION} with the following.

$(\Rightarrow)$
\begin{align*}
\sum_{j \in \mathbb{J},k\in \mathbb{L}_j}\langle x,F_j^*e_{j,k} \rangle\langle F_j^*e_{j,k},x \rangle&=\sum_{j\in \mathbb{J}}\sum_{k\in \mathbb{L}_j} \langle F_jx, e_{j,k} \rangle\langle  e_{j,k},F_jx \rangle\\
&= \sum_{j\in \mathbb{J}} \langle F_jx ,F_jx\rangle= \langle x, x \rangle, ~\forall x \in  \mathscr{E}.
\end{align*}
$(\Leftarrow)$ For $x=\sum_{l \in \mathbb{J},k \in \mathbb{L}_l}a_{l,k} f_{l,k}\in  \mathscr{E}$,
\begin{align*}
 \sum_{j \in \mathbb{J}} \langle F_jx ,F_jx\rangle =\sum_{j \in \mathbb{J}}\left\langle  \sum_{k \in \mathbb{L}_j} a_{j,k}e_{j,k},\sum_{m \in \mathbb{L}_j} a_{j,m}e_{j,m}\right \rangle =\sum_{j \in \mathbb{J}} \sum_{k \in \mathbb{L}_j} a_{j,k}a_{j,k}^* =\langle x, x \rangle.
\end{align*}
\item  For all $ x\in \mathscr{E},$
\begin{align*}
\sum\limits_{j \in \mathbb{J}}\langle A_jx, \Psi_jx\rangle &=\sum\limits_{j \in \mathbb{J}} \left \langle\sum\limits_{k \in \mathbb{L}_j}\langle x, A_j^*e_{j,k} \rangle e_{j,k}, \sum\limits_{l \in \mathbb{L}_j} \langle x, \Psi_j^*e_{j,l}\rangle e_{j,l} \right\rangle\\
&=\sum\limits_{j \in \mathbb{J}}\left \langle\sum\limits_{k \in \mathbb{L}_j}\langle x, u_{j,k}\rangle e_{j,k}, \sum\limits_{l \in \mathbb{L}_j} \langle x, v_{j,l}\rangle e_{j,l} \right\rangle=\sum\limits_{j \in \mathbb{J}}\sum\limits_{k \in \mathbb{L}_j}\langle x, u_{j,k}\rangle \langle v_{j,k}, x\rangle, 
\end{align*}
\begin{align*}
\left\langle\sum\limits_{j \in \mathbb{J}}L_jA_jx,\sum\limits_{k \in \mathbb{J}}L_kA_kx\right\rangle =\sum\limits_{j \in \mathbb{J}}\langle A_jx, A_jx \rangle =\sum\limits_{j \in \mathbb{J}}\sum\limits_{k \in \mathbb{L}_j}\langle x, u_{j,k}\rangle\langle  u_{j,k}, x \rangle;\\ 
\left\langle \sum\limits_{j \in \mathbb{J}}L_j\Psi_jx,\sum\limits_{k \in \mathbb{J}}L_k\Psi_kx \right\rangle =\sum\limits_{j \in \mathbb{J}}\langle \Psi_jx, \Psi_jx \rangle =\sum\limits_{j \in \mathbb{J}}\sum\limits_{k \in \mathbb{L}_j}\langle x, v_{j,k}\rangle\langle x, v_{j,k}\rangle,
\end{align*}
and
\begin{align*}
\sum\limits_{j\in \mathbb{J}} \Psi_j^*A_j\geq 0&\iff \sum\limits_{j \in \mathbb{J}}\langle A_jy, \Psi_jy\rangle=\sum\limits_{j \in \mathbb{J}}\langle  \Psi_jy, A_jy\rangle \geq 0 , \forall y \in \mathscr{E}\\
& \iff \sum\limits_{j\in \mathbb{J}}\sum\limits_{k\in \mathbb{L}_j}\langle y, u_{j,k}\rangle \langle  v_{j,k} , y\rangle = \sum\limits_{j\in \mathbb{J}}\sum\limits_{k\in \mathbb{L}_j}\langle y, v_{j,k}\rangle \langle  u_{j,k} , y\rangle , ~\forall y \in \mathscr{E}.
\end{align*} 
\item Similar to (iii).
\item Similar to (v) in Theorem \ref{SEQUENTIAL CHARACTERIZATION}.
\item Similar to (v).
\end{enumerate}   
\end{proof} 
\begin{theorem}
Let $ \{x_j\}_{j \in \mathbb{J}} $  be  orthonormal for $ \mathscr{E}$ with  the property: If $\langle x, x_j\rangle =0, \forall j \in \mathbb{J},$ then $ x=0.$ Then
$ \{x_j\}_{j \in \mathbb{J}} $ is  a maximal orthonormal set for  $ \mathscr{E}$.
 \end{theorem}
 \begin{proof}
 If the statement fails, then there is   a proper  orthonormal set $Y$ for $ \mathscr{E}$ which properly contains  $ \{x_j\}_{j \in \mathbb{J}} $.  Choose $ y $ in $ Y \setminus\{x_j\}_{j \in \mathbb{J}} $. But then $\ \langle y, x_j \rangle =0, \forall j \in \mathbb{J}$ and hence $y=0$, from the assumption. This is clearly impossible, because $ y$ is in the orthonormal set $ Y$.
 \end{proof}
\begin{corollary}\label{MAXIMALCOROLLARY}
Let $ \{x_j\}_{j \in \mathbb{J}} $  be  orthonormal for  $ \mathscr{E}$. Consider the following statements.
 \text{\upshape(i)} $ \{x_j\}_{j \in \mathbb{J}} $ is an orthonormal basis for $ \mathscr{E}$.
 \text{\upshape(ii)} $ x= \sum_{j\in \mathbb{J}}\langle x , x_j \rangle x_j, \forall x \in \mathscr{E}.$
\text{\upshape(iii)} $ \langle x, y\rangle =\sum_{j\in \mathbb{J}}\langle x , x_j \rangle \langle x_j, y \rangle,  \forall x, y \in \mathscr{E}.$
\text{\upshape(iv)} $ \langle x, x\rangle =\sum_{j\in \mathbb{J}}\langle x , x_j \rangle \langle x_j, x \rangle,  \forall x \in \mathscr{E}.$
 \text{\upshape(v)} $ \overline{\operatorname{span}}_{\mathscr{A} }\{x_j\}_{j \in \mathbb{J}}= \mathscr{E}.$
 \text{\upshape(vi)} If $\langle x, x_j\rangle =0, \forall j \in \mathbb{J},$ then $ x=0.$
 \text{\upshape(vii)} $ \{x_j\}_{j \in \mathbb{J}} $ is a maximal orthonormal set for $ \mathscr{E}$.
 Then $\text{\upshape(i)}\Rightarrow\text{\upshape(ii)}\Rightarrow\text{\upshape(iii)}\Rightarrow \text{\upshape(iv)}\Rightarrow \text{\upshape(v)}\Rightarrow \text{\upshape(vi)}\Rightarrow \text{\upshape(vii)}$.
 \end{corollary}
 \begin{theorem}
 Let $ \mathscr{E}$  over $ (\mathscr{A},e)$  be such that $\langle x, x\rangle $ is invertible (in $\mathscr{A}$) for all non zero $ x \in \mathscr{E}$. Suppose  $ \{x_j\}_{j \in \mathbb{J}} $ is  orthonormal   for   $ \mathscr{E}$ which is also  Bessel. Then the following are equivalent.
 \begin{enumerate}[\upshape(i)]
 \item  $ \{x_j\}_{j \in \mathbb{J}} $ is an orthonormal basis for  $ \mathscr{E}$.
 \item  $ \{x_j\}_{j \in \mathbb{J}} $ is a maximal orthonormal set for $ \mathscr{E}$.
 \end{enumerate}
 \end{theorem}
 \begin{proof}
 Corollary \ref{MAXIMALCOROLLARY} gives (i) $ \Rightarrow$ (ii). In view of  Theorem \ref{CHARACTERIZATION3}, to prove (ii) $ \Rightarrow$ (i) it is enough to show that (ii) $ \Rightarrow$ (vi) of Theorem \ref{CHARACTERIZATION3}. If (vi)  of Theorem \ref{CHARACTERIZATION3} fails, choose a nonzero $ x \in \mathscr{E}$ such that  $\langle x, x_j\rangle =0, \forall j \in \mathbb{J}$. Define $ y\coloneqq \langle x, x\rangle^{-1/2}x$. Then $\langle y, y\rangle =e $. Note that $ y\neq x_j, \forall j \in \mathbb{J}$ (otherwise, if $ y=x_j$ for some $j \in \mathbb{J}$, then $ \langle \langle x, x\rangle^{-1/2}x-x_j, \langle x, x\rangle^{-1/2}x-x_j\rangle=0 \Rightarrow e-0-0+e=0 $ which is impossible). Note also that $ \{x_j\}_{j \in \mathbb{J}} \cup\{y\} \subsetneq \mathscr{E} $ (for instance, $2y \notin \{x_j\}_{j \in \mathbb{J}} \cup\{y\} $). Then the proper orthonormal set $ \{x_j\}_{j \in \mathbb{J}} \cup\{y\}$ properly contains $ \{x_j\}_{j \in \mathbb{J}}$, thus (ii) collapses.
 \end{proof}
 \begin{theorem}
 Let   $ \{x_j\}_{j \in \mathbb{J}} $  be  Bessel  for   $ \mathscr{E}$ over $ (\mathscr{A},e)$  which is also  orthonormal  for  $ \mathscr{E}$. If there exists an $ x \in  \mathscr{E}$ such that  $\langle x, x \rangle -\sum_{j\in \mathbb{J}}\langle x, x_j \rangle \langle x_j, x \rangle$ is invertible (in $ \mathscr{A}$), then  $ \{x_j\}_{j \in \mathbb{J}} $ can not be a maximal orthonormal set for  $ \mathscr{E}$.
 \end{theorem}
 \begin{proof}
 Since $ \{x_j\}_{j \in \mathbb{J}}$ is Bessel,  we see $ \sum_{j\in \mathbb{J}}\langle x, x_j \rangle  x_j$ exists, say $y$. Then $\langle x-y, x-y \rangle =\langle x-\sum_{j\in \mathbb{J}}\langle x, x_j \rangle x_j, x-\sum_{k\in \mathbb{J}}\langle x, x_k \rangle x_k \rangle=\langle x, x \rangle -\sum_{j\in \mathbb{J}}\langle x, x_j \rangle \langle x_j, x \rangle $, which is invertible.  Let $z=(x-y)^{-1/2}(x-y)$. Then  $ \langle z, z \rangle =e $, $\langle z, x_j \rangle =0, \forall j \in \mathbb{J}$ and $z\neq x_j, \forall j \in \mathbb{J}$. So the proper orthonormal set $ \{x_j\}_{j \in \mathbb{J}}\cup\{y\} $ properly contains $ \{x_j\}_{j \in \mathbb{J}}$.
 \end{proof}
 \begin{theorem}
 For  $\mathscr{E}$    over  $\mathscr{A}$,
 \begin{enumerate}[\upshape(i)]
 \item Every orthonormal set $Y$ in  $\mathscr{E}$ is contained in a maximal orthonormal set.
 \item If  $\mathscr{E}$ has a vector $x$ such that $\langle x, x\rangle $ is invertible in $\mathscr{A}$, then  $\mathscr{E}$ has a maximal orthonormal set.
 \end{enumerate}
 \end{theorem}
 \begin{proof}
  \begin{enumerate}[\upshape(i)]
 \item This is an application of Zorn's lemma to the poset $(\mathscr{P},\preceq)$, where 
 $$\mathscr{P}=\{Z : Z ~\text{is an orthonormal set for}~ \mathscr{E} ~ \text{such that}~ Z \supseteq Y\} $$
 and for $ Z_1,Z_2 \in \mathscr{P}, Z_1\preceq Z_2 $ if $ Z_1\subseteq Z_2.$
 \item Apply (i) to the orthonormal set $ Y=\{\langle x, x \rangle^{-1/2}x\}$.
 \end{enumerate}
 \end{proof}
 \begin{theorem}
 If  $\mathscr{E} $  over   $\mathscr{A}$ has an orthonormal basis  $ \{x_j\}_{j \in \mathbb{J}}$, then  $\mathscr{E}$ is unitarily isomorphic to  $\mathscr{H}_\mathscr{A}.$
 \end{theorem}
 \begin{proof}
 Let $ \{e_j\}_{j \in \mathbb{J}} $ be the standard  orthonormal basis for $\mathscr{H}_\mathscr{A}.$ Define $ U:\mathscr{E} \ni \sum_{j \in \mathbb{J}}a_jx_j \mapsto \sum_{j \in \mathbb{J}}a_je_j \in  \mathscr{H}_\mathscr{A}.$ Then $ U$ is  bijective, and  $ \langle U(\sum_{j \in \mathbb{J}}a_jx_j), U(\sum_{k \in \mathbb{J}}a_kx_k)\rangle= \langle \sum_{j \in \mathbb{J}}a_je_j, \sum_{k \in \mathbb{J}}a_ke_k\rangle = \sum_{j \in \mathbb{J}}a_ja^*_j=\langle \sum_{j \in \mathbb{J}}a_jx_j, \sum_{k \in \mathbb{J}}a_kx_k\rangle,  \forall \sum_{j \in \mathbb{J}}a_jx_j \in \mathscr{E},$ which says it is  a unitary.
 \end{proof}

 \begin{theorem}
 Let  $ \{x_j\}_{j \in \mathbb{J}} $ be an orthonormal basis for $\mathscr{E} $ over  $\mathscr{A}$. Then the orthonormal bases for  $\mathscr{E} $ which are  also  Bessel sequences are precisely  $ \{Ux_j\}_{j \in \mathbb{J}} ,$ where $ U: \mathscr{E} \rightarrow \mathscr{E}$ is a unitary homomorphism.
 \end{theorem}
 \begin{proof}
 $(\Rightarrow)$ Let  $ \{y_j\}_{j \in \mathbb{J}} $ be an orthonormal basis as well as a Bessel sequence  for $\mathscr{E}$. Define $U:\mathscr{E} \ni \sum_{j \in \mathbb{J}}a_jx_j \mapsto \sum_{j \in \mathbb{J}}a_jy_j\in  \mathscr{E} $. Then $ U$ is a bounded adjointable invertible homomorphism with $ U^* :\mathscr{E} \ni \sum_{j \in \mathbb{J}}b_jy_j \mapsto \sum_{j \in \mathbb{J}}b_jx_j\in  \mathscr{E} $ and $ Ux_j=y_j, \forall j \in \mathbb{J}$. We find $ \langle U^*Ux,y \rangle=\langle Ux, Uy\rangle = \langle \sum_{j \in \mathbb{J}}\langle x, x_j \rangle y_j, \sum_{k \in \mathbb{J}}\langle y, x_k \rangle y_k\rangle=\sum_{j \in \mathbb{J}}\langle x, x_j \rangle \langle x_j, y \rangle=\langle I_\mathscr{E}x, y\rangle,\forall x, y \in \mathscr{E} $ and hence $ U$ is unitary.
 
 $(\Leftarrow)$ Since $ U$ is invertible, $ \{Ux_j\}_{j \in \mathbb{J}} $ is a Schauder basis for $\mathscr{E} $ over  $\mathscr{A}$. Unitariness of  $ U$ gives orthonormality of $ \{Ux_j\}_{j \in \mathbb{J}}$. Now $\sum_{j \in \mathbb{J}}\langle x,Ux_j \rangle\langle Ux_j,x \rangle=\sum_{j \in \mathbb{J}}\langle U^*x,x_j \rangle\langle x_j,U^*x \rangle =\langle U^*x,U^*x \rangle=\langle x,x \rangle,\forall x \in \mathscr{E}$ and hence $ \{Ux_j\}_{j \in \mathbb{J}} $ is Bessel.
 \end{proof}
 \begin{proposition}
If  $ \{x_j\}_{j \in \mathbb{J}} $   in $ \mathscr{E}$ is such that $ \langle x_j,x_j \rangle =e, \forall j \in \mathbb{J}$,  and $ \sum_{j\in \mathbb{J}}\langle x,x_j \rangle\langle x_j,x \rangle=\langle x,x \rangle, \forall x \in \mathscr{E},$ then $ \{x_j\}_{j \in \mathbb{J}} $ is an orthonormal basis for $ \mathscr{E}.$
 \end{proposition}
 \begin{proof}
 Equality $ \sum_{j\in \mathbb{J}}\langle x,x_j \rangle\langle x_j,x \rangle=\langle x,x \rangle, \forall x \in \mathscr{E}$ says that $ \{x_j\}_{j \in \mathbb{J}} $ is Bessel (with bound 1). Evaluating this equality at $ x_k,$ for each $ k \in \mathbb{J}$, we get $ \sum_{j\in \mathbb{J}}\langle x_k,x_j \rangle\langle x_j,x_k \rangle=\langle x_k,x_k \rangle \Rightarrow\sum_{j\in \mathbb{J},j\neq k}\langle x_k,x_j \rangle\langle x_j,x_k \rangle=0$. Since the set of all positive elements in a unital C*-algebra is a cone and each term in the last sum is positive,  we must have $\langle x_k,x_j \rangle\langle x_j,x_k \rangle=0, \forall j \neq k.$ This implies $ \|\langle  x_k,x_j \rangle\|=0, \forall j \neq k$, by C*-condition. Now Theorem  \ref{CHARACTERIZATION3} says that  $ \{x_j\}_{j \in \mathbb{J}} $ is an orthonormal basis for $ \mathscr{E}.$ 
 \end{proof}
 \begin{theorem}
 Let $ \mathscr{E} $   over  $( \mathscr{A},e)$  be such that $\langle x, x\rangle $ is invertible (in $\mathscr{A}$) for all non zero $ x \in \mathscr{E}$.  Let  $ \{x_n\}_{n=1}^\infty $  be  linearly independent (over  $\mathscr{A}$) in $ \mathscr{E}$. Define 
 $y_1\coloneqq x_1,  z_1\coloneqq\langle y_1, y_1 \rangle^{-{1/2}}y_1 $, 
 \begin{align*}
 y_n\coloneqq x_n-\sum_{k=1}^{n-1}\langle x_n, z_k\rangle z_k, ~ z_n\coloneqq\langle y_n, y_n \rangle^{-\frac{1}{2}}y_n, ~\forall n \geq 2.
 \end{align*}
Then $ \{y_n\}_{n=1}^\infty $  is  orthogonal  for  $ \mathscr{E}$, $ \{z_n\}_{n=1}^\infty $  is  orthonormal for $ \mathscr{E}$ and
$$ \operatorname{span}_\mathscr{A}\{z_n\}_{n=1}^m=\operatorname{span}_\mathscr{A}\{x_n\}_{n=1}^m=\operatorname{span}_\mathscr{A}\{y_n\}_{n=1}^m, ~\forall m \in \mathbb{N}.$$
In particular, $ \overline{\operatorname{span}}_\mathscr{A}\{z_n\}_{n=1}^\infty=\overline{\operatorname{span}}_\mathscr{A}\{x_n\}_{n=1}^\infty=\overline{\operatorname{span}}_\mathscr{A}\{y_n\}_{n=1}^\infty$. 
 \end{theorem}
 \begin{proof}
 Similar to the Gram-Schmidt orthogonalization in Hilbert spaces.
 \end{proof}
\begin{theorem}\label{RIESZFISCHERGENERAL}
Let  $ \{x_j\}_{j \in \mathbb{J}}$ be  orthonormal for  $\mathscr{E}$  over $\mathscr{A}$, and $ \{a_j\}_{j \in \mathbb{J}}$ be  in  $\mathscr{A}$. Then 
$$ \sum_{j\in \mathbb{J}}a_jx_j ~ \text{converges in}~ \mathscr{E}~ \text{if and only if }~\sum_{j\in \mathbb{J}}a_ja_j^* ~\text{converges in}~ \mathscr{A}.$$
In this case, define $ x\coloneqq\sum_{j\in \mathbb{J}}a_jx_j$. Then $a_j=\langle x, x_j\rangle, \forall j \in \mathbb{J}$.
\end{theorem}
\begin{proof}
For every finite subset $ \mathbb{S}$ of  $\mathbb{J}$, we get
$$ \left\|\sum_{j\in \mathbb{S}}a_jx_j \right\|^2=\left\|\left\langle\sum_{j\in \mathbb{S}}a_jx_j,\sum_{k\in \mathbb{S}}a_kx_k \right \rangle \right\|=\left\|\sum_{j\in \mathbb{S}}a_ja^*_j \right\|.$$
Thus $\{\sum_{j\in \mathbb{S}}a_jx_j:\mathbb{S}~\text{ is a  finite subset of }~ \mathbb{J}\}$ is Cauchy net if and only if $\{\sum_{j\in \mathbb{S}}a_ja^*_j:\mathbb{S}~\text{ is a finite subset of }~ \mathbb{J}\} $ is a Cauchy net. Since $\mathscr{E}$ and $\mathscr{A}$ are complete, the result follows.
\end{proof}
\begin{remark}
 Theorem \ref{RIESZFISCHERGENERAL} is an extension of Riesz-Fischer theorem.
\end{remark}
\begin{theorem}
\begin{enumerate}[\upshape(i)]	
\item If $\mathscr{E}$ has an orthonormal basis which is also a Hamel basis for $\mathscr{E}$,  then $\mathscr{E}$ is finite dimensional.
\item An infinite orthonormal  basis for $\mathscr{E}$ can not be a Hamel basis  for $\mathscr{E}$.
\end{enumerate}
\end{theorem}
\begin{proof}
 Since we have Theorem \ref{RIESZFISCHERGENERAL}, the proof is similar to the proof of statement  ``if Hilbert space $\mathcal{H}$ has an orthonormal basis which is also a Hamel basis for $\mathcal{H}$,  then $\mathcal{H}$ is finite dimensional". 
\end{proof}
\begin{proposition}
Let $ \{x_j=Ue_j\}_{j \in \mathbb{J}} $ be a Riesz basis for  $\mathscr{E}$ such that $ \{e_j\}_{j \in \mathbb{J}} $ is a Bessel sequence.  Then 
\begin{enumerate}[\upshape(i)]
\item $ \{x_j\}_{j \in \mathbb{J}} $ is a frame (w.r.t. itself) for $\mathscr{E}$.
\item There exists a unique sequence $ \{y_j\}_{j \in \mathbb{J}} $ in $\mathscr{E} $ such that $ x=\sum_{j\in\mathbb{J}}\langle x,y_j \rangle x_j, \forall x \in \mathscr{E}.$ Further, $ \{y_j\}_{j \in \mathbb{J}} $ is Riesz.
\end{enumerate}
\end{proposition}
\begin{proof}
\begin{enumerate}[\upshape(i)]
\item $ \|U\|^{-2}\langle x,x\rangle \leq \sum_{j\in \mathbb{J}}\langle x,x_j \rangle\langle x_j,x \rangle =\sum_{j\in \mathbb{J}}\langle U^*x,e_j \rangle\langle e_j,U^*x \rangle =\langle U^*x,U^*x\rangle \leq \|U\|^2\langle x,x\rangle, \forall x \in \mathscr{E}.$
\item Define $ y_j\coloneqq (U^{-1})^*e_j, \forall j \in \mathbb{J}$. Then $ \{y_j\}_{j \in \mathbb{J}} $ is Riesz. Also, $\sum_{j\in\mathbb{J}}\langle x,y_j \rangle x_j =\sum_{j\in \mathbb{J}}\langle x,(U^{-1})^*e_j \rangle Ue_j=U(\sum_{j\in \mathbb{J}}\langle U^{-1}x,e_j \rangle e_j)=x, \forall x \in \mathscr{E}$. If $ \{z_j\}_{j \in \mathbb{J}} $ in $\mathscr{E} $ also fulfills  $ x=\sum_{j\in\mathbb{J}}\langle x,z_j \rangle x_j, \forall x \in \mathscr{E},$ then $\sum_{j\in\mathbb{J}}\langle x,z_j-y_j \rangle x_j=0,\forall x \in \mathscr{E} $ $\Rightarrow $ $z_j=y_j,\forall j \in \mathbb{J}$.
\end{enumerate}
\end{proof}

\begin{lemma}\label{HOMOMORPHISMEXISTENCELEMMA}
Let $ \{x_j\}_{j \in \mathbb{J}}$ be complete in $\mathscr{E}$ (over $\mathscr{A}$), $ \{y_j\}_{j \in \mathbb{J}}$ be  in $\mathscr{E}_0$ (over $\mathscr{A}$). Suppose $a, b>0$ are  such that for all finite subsets $ \mathbb{S}_{\mathscr{E}},\mathbb{S}_{\mathscr{E}_0}\subseteq\mathbb{J}$,
\begin{align} 
a\left\|\sum_{j\in \mathbb{S}_{\mathscr{E}}}c_jc_j^*\right\|&\leq \left\|\sum_{j\in \mathbb{S}_{\mathscr{E}}}c_jx_j \right\|^2, ~\forall c_j \in \mathscr{A}, \forall j \in \mathbb{S}_{\mathscr{E}} ,\label{RIESZMODULEINEQUALITY}\\
\left\|\sum_{j\in \mathbb{S}_{\mathscr{E}_0}}d_jy_j \right\|^2&\leq b\left\|\sum_{j\in \mathbb{S}_{\mathscr{E}_0}}d_jd_j^*\right\|, ~ \forall d_j \in \mathscr{A}, \forall j \in \mathbb{S}_{\mathscr{E}_0}. \label{RIESZMODULEINEQUALITY1}
\end{align}
Then $ T : \operatorname{span}_\mathscr{A}\{x_j\}_{j \in \mathbb{J}}\ni\sum_{\operatorname{finite}}a_jx_j \mapsto\sum_{\operatorname{finite}}a_jy_j  \in \operatorname{span}_\mathscr{A}\{y_j\}_{j \in \mathbb{J}}$ defines a bounded homomorphism  and extends uniquely as a bounded homomorphism  from  $\mathscr{E}$ into $\mathscr{E}_0$, the norm of the extended homomorphism  is less than or equal to $ (\frac{b}{a})^{1/2}$.
\end{lemma}
\begin{proof}
Since $ \mathscr{A}$ is a C*-algebra, $ \sum_{j\in \mathbb{S}_{\mathscr{E}}}c_jc_j^*=0$ implies $c_j=0,\forall j \in \mathbb{S}_{\mathscr{E}}$. Therefore $ \{x_j\}_{j \in \mathbb{J}}$ is linearly independent over $ \mathscr{A}$. This gives well-definedness of $T$. Linearity of $T$ is obvious. For the boundedness of $T$,
\begin{align*}
\left\|T\left(\sum_{j\in \mathbb{S}_{\mathscr{E}}}c_jx_j \right)\right\|^2=\left\|\sum_{j\in \mathbb{S}_{\mathscr{E}}}c_jy_j \right\|^2 \leq b\left\|\sum_{j\in \mathbb{S}_{\mathscr{E}}}c_jc_j^*\right\|\leq \frac{b}{a} \left\|\sum_{j\in \mathbb{S}_{\mathscr{E}}}c_jx_j \right\|^2.
\end{align*}
Continuity of $T$ and completeness of $ \{x_j\}_{j \in \mathbb{J}}$ give the remaining last conclusion of lemma.
\end{proof}
\begin{lemma}\label{HOMOMORPHISMEXISTENCELEMMAORDER}
Let $ \{x_j\}_{j \in \mathbb{J}}$ be complete in $\mathscr{E}$ (over $\mathscr{A}$), $ \{y_j\}_{j \in \mathbb{J}}$ be  in $\mathscr{E}_0$ (over $\mathscr{A}$). Suppose $a, b>0$ are  such that for all finite subsets $ \mathbb{S}_{\mathscr{E}},\mathbb{S}_{{\mathscr{E}_0}}\subseteq\mathbb{J}$,
\begin{align*} 
a\sum_{j\in \mathbb{S}_{\mathscr{E}}}c_jc_j^*&\leq \left\langle\sum_{j\in \mathbb{S}_{\mathscr{E}}}c_jx_j,\sum_{k\in \mathbb{S}_{\mathscr{E}}}c_kx_k \right\rangle, ~\forall c_j \in \mathscr{A}, \forall j,k \in \mathbb{S}_{\mathscr{E}} ,\\
\left\langle\sum_{j\in \mathbb{S}_{{\mathscr{E}_0}}}d_jy_j,\sum_{k\in \mathbb{S}_{{\mathscr{E}_0}}}d_ky_k \right\rangle  &\leq b\sum_{j\in \mathbb{S}_{{\mathscr{E}_0}}}d_jd_j^*,~ \forall d_j \in \mathscr{A}, \forall j,k \in \mathbb{S}_{\mathscr{E}_0}.
\end{align*}
Then $ T : \operatorname{span}_\mathscr{A}\{x_j\}_{j \in \mathbb{J}}\ni\sum_{\operatorname{finite}}a_jx_j \mapsto\sum_{\operatorname{finite}}a_jy_j  \in \operatorname{span}_\mathscr{A}\{y_j\}_{j \in \mathbb{J}}$ defines a bounded homomorphism  and extends uniquely as a bounded homomorphism  from  $\mathscr{E}$ into $\mathscr{E}_0$, the norm of the extended homomorphism  is less than or equal to $ (\frac{b}{a})^{1/2}$.
\end{lemma}
\begin{proof}
We take norm in given inequalities and apply Lemma \ref{HOMOMORPHISMEXISTENCELEMMA}.
\end{proof}
\begin{theorem}\label{RIESZCHARACTERIZATIONSELFDUAL}
Let $ \{x_j\}_{j \in \mathbb{J}}$ be in $\mathscr{E}$ over $\mathscr{A}$. Assume that $\mathscr{E} $ is self-dual. Then the following are equivalent.
\begin{enumerate}[\upshape(i)]
\item $ \{x_j\}_{j \in \mathbb{J}}$ is a Riesz basis for $\mathscr{E}$.
\item $\overline{\operatorname{span}}_\mathscr{A}\{x_j\}_{j \in \mathbb{J}}=\mathscr{E}$, and there exist $a, b>0$ such that for every finite subset $ \mathbb{S}\subseteq\mathbb{J}$,
$$ a\sum\limits_{j\in \mathbb{S}}c_jc_j^*\leq \left\langle \sum\limits_{j\in\mathbb{S}}c_jx_j ,\sum\limits_{k\in\mathbb{S}}c_kx_k  \right \rangle \leq b\sum\limits_{j\in \mathbb{S}}c_jc_j^*, ~\forall c_j\in \mathscr{A}, \forall j,k\in \mathbb{S}.$$ 
\item $\overline{\operatorname{span}}_\mathscr{A}\{x_j\}_{j \in \mathbb{J}}=\mathscr{E}$, and there exist $a, b>0$ such that for every finite subset $ \mathbb{S}\subseteq\mathbb{J}$,
$$ a\left\|\sum_{j\in \mathbb{S}}c_jc_j^*\right\|\leq \left\|\sum_{j\in \mathbb{S}}c_jx_j \right\|^2\leq b\left\|\sum_{j\in \mathbb{S}}c_jc_j^*\right\|,~ \forall c_j \in \mathscr{A}, \forall j \in \mathbb{S} .$$
\end{enumerate}
\end{theorem}
\begin{proof}
We prove (i) $ \Rightarrow$	(ii) $ \Rightarrow$ (iii) $ \Rightarrow$ (i). Clearly, (ii) $ \Rightarrow$ (iii). For (i) $ \Rightarrow$	(ii),  let  $ \{f_j\}_{j \in \mathbb{J}} $ be an orthonormal basis for  $ \mathscr{E}$ and  $ U \in \operatorname{End}_\mathscr{A}^*(\mathscr{E})$ be invertible such that $ x_j=Uf_j, \forall j \in \mathbb{J}.$ Since $U$ is invertible, $\overline{\operatorname{span}}_\mathscr{A}\{x_j\}_{j \in \mathbb{J}}=\mathscr{E}$. Further for every finite subset $ \mathbb{S}\subseteq\mathbb{J}$,
\begin{align*}
\frac{1}{\|U^{-1}\|^2}\sum\limits_{j\in \mathbb{S}}c_jc_j^*&= \frac{1}{\|U^{-1}\|^2}\left\langle \sum\limits_{j\in\mathbb{S}}c_jf_j ,\sum\limits_{k\in\mathbb{S}}c_kf_k  \right \rangle \leq\left\langle \sum\limits_{j\in\mathbb{S}}c_jx_j ,\sum\limits_{k\in\mathbb{S}}c_kx_k  \right \rangle \\
&\leq \|U\|^2\left\langle \sum\limits_{j\in\mathbb{S}}c_jf_j ,\sum\limits_{k\in\mathbb{S}}c_kf_k  \right \rangle= \|U\|^2\sum\limits_{j\in \mathbb{S}}c_jc_j^*, ~\forall c_j\in \mathscr{A}, \forall j,k\in \mathbb{S}.
\end{align*}
We next show (iii) $ \Rightarrow$	(i). Let   $ \{f_j\}_{j \in \mathbb{J}} $ be  an orthonormal basis for  $ \mathscr{E}$. We try to apply Lemma \ref{HOMOMORPHISMEXISTENCELEMMA} for $\{f_j\}_{j \in \mathbb{J}} $ and $\{x_j\}_{j \in \mathbb{J}} $. Since $ \{f_j\}_{j \in \mathbb{J}} $ is an orthonormal basis, it satisfies Inequality (\ref{RIESZMODULEINEQUALITY})  (now it is an equality). Second inequality in the assumption gives Inequality (\ref{RIESZMODULEINEQUALITY1}). Already $\{f_j\}_{j \in \mathbb{J}} $ is complete in $\mathscr{E}$. Therefore, if we define $Uf_j\coloneqq x_j, \forall j \in \mathbb{J} $, then it extends as a bounded homomorphism on $\mathscr{E}$. Next we try to apply Lemma \ref{HOMOMORPHISMEXISTENCELEMMA} for $\{x_j\}_{j \in \mathbb{J}} $ and $\{f_j\}_{j \in \mathbb{J}} $. Now the first inequality in the assumption gives Inequality (\ref{RIESZMODULEINEQUALITY}) and orthonormality of $\{f_j\}_{j \in \mathbb{J}} $ gives Inequality (\ref{RIESZMODULEINEQUALITY1}). Given that $\{x_j\}_{j \in \mathbb{J}} $ is complete in $\mathscr{E}$. So,  if we define $Vx_j\coloneqq f_j, \forall j \in \mathbb{J} $, then it extends as a bounded homomorphism on $\mathscr{E}$. We easily see $ UV=VU=I_\mathscr{E}$. Thus $U$ is invertible. Still we are not done - we cannot tell that $\{x_j\}_{j \in \mathbb{J}} $ is a Riesz basis. Why? What about the adjointability of $U$? (then only we can conclude $\{x_j\}_{j \in \mathbb{J}} $ is a Riesz basis). This comes from a result of Paschke, which says a bounded homomorphism from a Hilbert C*-module to a pre-Hilbert C*-module has an adjoint (Proposition 3.4 in \cite{PASCHKE1}).
\end{proof}
\begin{remark}
In Theorem \ref{RIESZCHARACTERIZATIONSELFDUAL}, self-duality of $\mathscr{E} $ was used only in the last conclusion of \text{\upshape(iii)} $ \Rightarrow$	\text{\upshape(i)}.  
\end{remark}

\begin{corollary}\label{ONBCOROLLARY}
Let $\mathscr{E} $ be self-dual, $ \overline{\operatorname{span}}_ \mathscr{A}\{x_j\}_{j \in \mathbb{J}}= \mathscr{E}$, and for each finite subset $ \mathbb{S} \subseteq \mathbb{J}$,
\begin{align}\label{EQUATIONTORIESZMODULE}
\left\| \sum_{j\in \mathbb{S}}c_jx_j\right\|^2= \left\|\sum_{j\in \mathbb{S}}c_jc_j^*\right\| ,~  \forall c_j \in  \mathscr{A}, \forall j \in \mathbb{S}.
\end{align}
Then $\{x_j\}_{j \in \mathbb{J}}$ is an adjointable isometric isomorphism image of an  orthonormal basis for $\mathscr{E}.$
\end{corollary}
\begin{proof}
Equality (\ref{EQUATIONTORIESZMODULE}) implies $\{x_j\}_{j \in \mathbb{J}}$ is a Riesz basis	(from Theorem \ref{RIESZCHARACTERIZATIONSELFDUAL}). Let  $ \{f_j\}_{j \in \mathbb{J}} $ be  an orthonormal basis for  $ \mathscr{E}$ and  $ U \in \operatorname{End}_\mathscr{A}^*(\mathscr{E})$ be invertible such that $ x_j=Uf_j, \forall j \in \mathbb{J}.$ Equality (\ref{EQUATIONTORIESZMODULE}) also tells whenever  $ \{c_j\}_{j \in \mathbb{J}}\in \mathscr{H}_\mathscr{A} $, then $\sum_{j\in \mathbb{J}}c_jx_j$ converges in $\mathscr{E}$. Then for each $ x=\sum_{j\in \mathbb{J}}c_jf_j \in\mathscr{E} $,
\begin{align}\label{ISOMETRICUEQUALITY}
\|Ux\|^2=\left\| \sum_{j\in \mathbb{J}}c_jx_j\right\|^2= \left\|\sum_{j\in \mathbb{J}}c_jc_j^*\right\|=\|x\|^2 
\end{align}
Threfore $U$ is an adjointable isometric isomorphism.
\end{proof}
\begin{corollary}
Let $\mathscr{E} $ be self-dual, $ \overline{\operatorname{span}}_ \mathscr{A}\{x_j\}_{j \in \mathbb{J}}= \mathscr{E}$, and for each finite subset $ \mathbb{S} \subseteq \mathbb{J}$,
\begin{align*}
\left\langle \sum_{j\in \mathbb{S}}c_jx_j, \sum_{k\in \mathbb{S}}c_kx_k \right\rangle = \sum_{j\in \mathbb{S}}c_jc_j^* ,~  \forall c_j \in  \mathscr{A}, \forall j,k \in \mathbb{S}.
\end{align*}
Then $\{x_j\}_{j \in \mathbb{J}}$ is an orthonormal basis for $\mathscr{E}.$
\end{corollary}
\begin{proof}
Similar to the proof of Corollary \ref{ONBCOROLLARY} but rather writing Equation (\ref{ISOMETRICUEQUALITY}) we have to write
\begin{align*}
\langle Ux, Ux\rangle =\left\langle \sum_{j\in \mathbb{J}}c_jx_j, \sum_{k\in \mathbb{J}}c_kx_k \right\rangle = \sum_{j\in \mathbb{J}}c_jc_j^*=\langle x, x\rangle,
\end{align*}
which tells $ U$ is unitary.	
\end{proof}
\begin{definition}
Let  $ \{x_j\}_{j \in \mathbb{J}},\{\tau_j\}_{j \in \mathbb{J}} $ be in $ \mathscr{E}$ over  $ \mathscr{A}$.  We say 
\begin{enumerate}[\upshape(i)]
\item $ \{x_j\}_{j \in \mathbb{J}}$ is an orthonormal set (resp. basis)  w.r.t. $\{\tau_j\}_{j \in \mathbb{J}} $ if  $ \{x_j\}_{j \in \mathbb{J}}$ or   $\{\tau_j\}_{j \in \mathbb{J}} $ is an orthonormal set (resp. basis) for $ \mathscr{E}$, say $ \{x_j\}_{j \in \mathbb{J}}$ is an orthonormal set (resp. basis) for $ \mathscr{E}$,  and there exists  a sequence $\{c_j\}_{j \in \mathbb{J}} $ of positive invertible elements in the center of  $ \mathscr{A}$  such that  $ 0<\inf\{\|c_j\|\}_{j \in \mathbb{J}}\leq \sup\{\|c_j\|\}_{j \in \mathbb{J}}<\infty$   and $ \tau_j=c_jx_j, \forall j \in \mathbb{J}.$ We write $ (\{x_j\}_{j\in \mathbb{J}}, \{\tau_j\}_{j\in \mathbb{J}})$ is an  orthonormal  set (resp. basis).
 \item  $ \{x_j\}_{j \in \mathbb{J}}$ is a Riesz  basis  w.r.t. $\{\tau_j\}_{j \in \mathbb{J}} $ if there  are  invertible  $ U,V \in \operatorname{End}_\mathscr{A}^*(\mathscr{E}) $ and an orthonormal basis $\{f_j\}_{j\in \mathbb{J}} $ for $\mathscr{E}$ such that $x_j=Uf_j, \tau_j=Vf_j, \forall j \in \mathbb{J}$ and $ VU^*\geq 0.$ We write $ (\{x_j\}_{j\in \mathbb{J}}, \{\tau_j\}_{j\in \mathbb{J}})$ is a Riesz  basis.
 \end{enumerate}
 \end{definition}
  Due to invertibility of $c_j$'s, definition is symmetric.
 \begin{theorem}\label{GBISVHM}
 Let $( \{x_j\}_{j \in \mathbb{J}},\{\tau_j =c_jx_j\}_{j \in \mathbb{J}}) $  be  orthonormal   for $\mathscr{E}$ over  $(\mathscr{A},e)$  such that $ c_j\leq 2e, \forall j \in \mathbb{J}.$ Then 
  \begin{enumerate}[\upshape(i)]
 \item 
 $ \sum_{j\in\mathbb{J}}(2e-c_j)\langle x, x_j\rangle \langle \tau_j, x \rangle\leq \langle x, x \rangle , \forall x \in \mathscr{E}.$
 \item For $ x \in \mathscr{E},$ 
 \begin{align*}
  x= \sum \limits_{j\in \mathbb{J}}\langle x, x_j\rangle \tau_j \iff \sum\limits_{j\in\mathbb{J}}(2e-c_j)\langle x, x_j\rangle\langle \tau_j, x\rangle = \langle x, x \rangle \iff  \sum\limits_{j\in\mathbb{J}}c_j^2\langle x, x_j\rangle \langle x_j, x\rangle  = \langle x, x \rangle.
  \end{align*}
 If $ c_j \leq e, \forall j$, then $\iff (e-c_j)\langle x, x_j\rangle=0 , \forall j \in \mathbb{J} \iff (e-c_j)x_j\perp x, \forall j \in \mathbb{J}.$
\end{enumerate}
 \end{theorem}
 \begin{proof}
 \begin{enumerate}[\upshape(i)]
\item For  $ x \in \mathscr{E}$ and  each finite subset $ \mathbb{S} \subseteq \mathbb{J},$
\begin{align*}
\left\| \sum\limits_{j\in \mathbb{S}}\langle x,  x_j \rangle \tau_j\right\|^2
&= \left\|\left\langle \sum_{j\in \mathbb{S}}c_j\langle x, x_j\rangle x_j, \sum_{k\in \mathbb{S}}c_k\langle x, x_k\rangle x_k \right\rangle\right\|=\left\|\sum\limits_{j \in \mathbb{S}} c_j^2 \langle x,  x_j \rangle\langle x_j,x \rangle\right\|\\
&\leq \left(\sup\{\|c_j\|^2\}_{j \in \mathbb{J}}\right)\left\|\sum\limits_{j \in \mathbb{S}} \langle x,  x_j \rangle\langle x_j,x \rangle\right\|,
\end{align*}
the last sum converges. Hence $\sum_{j\in \mathbb{S}}\langle x,  x_j \rangle \tau_j$ exists. In a similar manner $\sum_{j\in\mathbb{J}}(2e-c_j)\langle x, x_j\rangle \langle \tau_j, x \rangle $ also exists.  Then 
\begin{align*}
0&\leq \left\langle x-\sum\limits_{j\in \mathbb{J}}\langle x,  x_j \rangle \tau_j\right\rangle  = \left\langle x-\sum_{j\in \mathbb{J}}c_j\langle x, x_j\rangle x_j, x-\sum_{k\in \mathbb{J}}c_k\langle x, x_k\rangle x_k \right\rangle \\
&=\langle x,x \rangle-2\sum_{j\in \mathbb{J}}c_j\langle x, x_j\rangle \langle x_j,x \rangle+\sum_{j\in \mathbb{J}}c^2_j\langle x, x_j\rangle\langle x_j,x \rangle 
=\langle x,x \rangle-\sum_{j\in \mathbb{J}}(2c_j-c_j^2)\langle x, x_j\rangle \langle x_j,x \rangle,
\end{align*}
$ \Rightarrow\sum_{j\in \mathbb{J}}(2c_j-c_j^2)\langle x, x_j\rangle \langle x_j, x\rangle=\sum_{j\in\mathbb{J}}(2e-c_j)\langle x, x_j\rangle \langle \tau_j, x \rangle\leq \langle x,x \rangle.$ 
\item Arguments are  simple.
\end{enumerate}
\end{proof}
\begin{corollary}
 Let $( \{x_j\}_{j \in \mathbb{J}},\{\tau_j=c_jx_j\}_{j \in \mathbb{J}}) $  be  an orthonormal basis for  $ \mathscr{E}.$ Then
$$ \frac{1}{\sup\{\|c_j\|\}_{j \in \mathbb{J}}}\sum_{j\in \mathbb{J}}\langle x, x_j\rangle \langle\tau_j, x \rangle \leq \langle x, x \rangle\leq \frac{1}{\inf\{\|c_j\|\}_{j \in \mathbb{J}}}\sum_{j\in \mathbb{J}}\langle x, x_j\rangle \langle\tau_j, x \rangle, ~\forall x \in \mathscr{E}.$$
\end{corollary}
\begin{proof}
$ \frac{1}{\sup\{\|c_j\|\}_{j \in \mathbb{J}}}\sum_{j\in \mathbb{J}}\langle x, x_j\rangle \langle\tau_j, x\rangle   =\frac{1}{\sup\{\|c_j\|\}_{j \in \mathbb{J}}}\sum_{j\in \mathbb{J}}c_j\langle x, x_j\rangle \langle x_j , x \rangle   \leq \sum_{j\in \mathbb{J}}\langle x, x_j\rangle\langle x_j , x \rangle = \langle x, x \rangle\leq \frac{1}{\inf\{\|c_j\|\}_{j \in \mathbb{J}}}\sum_{j\in \mathbb{J}} c_j\langle x, x_j\rangle \langle x_j , x \rangle\leq  \frac{1}{\inf\{\|c_j\|\}_{j \in \mathbb{J}}}\sum_{j\in \mathbb{J}}\langle x, x_j\rangle \langle\tau_j, x \rangle, ~\forall x \in \mathscr{E}.$
\end{proof}
\begin{corollary}
Let $( \{x_j\}_{j \in \mathbb{J}},\{\tau_j=c_jx_j\}_{j \in \mathbb{J}}) $  be  an orthonormal basis for  $ \mathscr{E}.$ Then
$$ \frac{1}{\sup\{\|c_j\|\}_{j \in \mathbb{J}}}\left\|\sum_{j\in \mathbb{J}}\langle x, x_j\rangle \langle\tau_j, x \rangle \right\| \leq \|x\|^2\leq \frac{1}{\inf\{\|c_j\|\}_{j \in \mathbb{J}}}\left\|\sum_{j\in \mathbb{J}}\langle x, x_j\rangle \langle\tau_j, x \rangle \right\|, ~\forall x \in \mathscr{E}.$$
\end{corollary}
\begin{corollary}
 If  $(\{x_j\}_{j \in \mathbb{J}},\{\tau_j =c_jx_j\}_{j \in \mathbb{J}}) $  is   orthonormal  for $\mathscr{E}$ over  $(\mathscr{A},e)$  such that $ c_j\leq 2e, \forall j \in \mathbb{J}$, then $\| \sum_{j\in\mathbb{J}}(2e-c_j)\langle x, x_j\rangle \langle \tau_j, x \rangle\|\leq\|x\|^2 , \forall x \in \mathscr{E}.$	
\end{corollary}
\begin{theorem}
If $ (\{x_j\}_{j\in \mathbb{J}}, \{\tau_j=c_jx_j\}_{j\in \mathbb{J}})$  is  orthonormal   for  $\mathscr{E}$ over  $(\mathscr{A},e)$  such that $ c_j\leq 2e, \forall j \in \mathbb{J}$, then for each $ x \in  \mathscr{E}$, the set 
$$ Y_x\coloneqq\bigcup\limits_{n=1}^\infty\left\{x_j: (2e-c_j)\langle x, x_j\rangle \langle \tau_j, x \rangle > \frac{1}{n}\langle x, x \rangle, j\in \mathbb{J} \right\}$$
 is either  finite or countable.
\end{theorem}
\begin{proof}
For $ n \in \mathbb{N}$, define 
$$ Y_{n,x}\coloneqq\left\{x_j: (2e-c_j)\langle x, x_j\rangle \langle \tau_j, x \rangle > \frac{1}{n}\langle x, x \rangle, j\in \mathbb{J} \right\}.$$	

Suppose, for some $n$, $ Y_{n,x}$ has more than $n-1$ elements, say $x_1,...,x_n$. Then $ \sum_{j=1 }^n(2e-c_j)\langle x, x_j\rangle \langle \tau_j, x \rangle > n\frac{1}{n}\langle x, x \rangle= \langle x, x \rangle$. From Theorem \ref{GBISVHM},  $ \sum_{j\in\mathbb{J}}(2e-c_j)\langle x, x_j\rangle \langle \tau_j, x \rangle\leq \langle x, x \rangle $. This gives $ \langle x, x \rangle< \langle x, x \rangle $ which is impossible. Hence $ \operatorname{Card}(Y_{n,x})\leq n-1$ and hence  
$Y_x=\cup_{n=1}^\infty Y_{n,x}$  is finite or countable.
\end{proof}
\begin{theorem}
 \begin{enumerate}[\upshape(i)]
 \item If $ (\{x_j\}_{j\in \mathbb{J}}, \{\tau_j\}_{j\in \mathbb{J}})$  is an orthonormal basis for $\mathscr{E}$, then it is a Riesz basis.
 \item If $(\{x_j=Uf_j\}_{j\in\mathbb{J}}, \{\tau_j=Vf_j\}_{j\in\mathbb{J}})$ is a Riesz basis for  $\mathscr{E}$, then it is a frame  with  optimal frame bounds $ \|(VU^*)^{-1}\|^{-1}$ and $ \|VU^*\|.$
 \end{enumerate}
 \end{theorem}
 \begin{proof}
  \begin{enumerate}[\upshape(i)]
 \item  We may assume  that $ \{x_j\}_{j \in \mathbb{J}}$ is an orthonormal basis.  Then there exists a sequence  $\{c_j\}_{j \in \mathbb{J}} $  of   positive invertible elements in the center of $ \mathscr{A}$ such that  $ 0<\inf\{\|c_j\|\}_{j \in \mathbb{J}}\leq \sup\{\|c_j\|\}_{j \in \mathbb{J}}<\infty$   and $ \tau_j=c_jx_j, \forall j \in \mathbb{J}.$ Define $ f_j \coloneqq x_j, \forall j \in \mathbb{J},  U\coloneqq I_\mathscr{E}$ and  $ V : \mathscr{E} \ni x \mapsto\sum_{j\in \mathbb{J}}c_j\langle x, x_j\rangle x_j \in \mathscr{E}.$ Using $ c_j$'s are in the center of $ \mathscr{A}$, we get $  V \in \operatorname{End}_\mathscr{A}^*(\mathscr{E}) $. Then $ x_j=Uf_j, Vf_j=\sum_{k\in \mathbb{J}}c_k\langle f_j,x_k \rangle x_k= c_jx_j=\tau_j, \forall j \in \mathbb{J}.$  Since all $c_j $'s are positive invertible,  $ V$ is positive invertible, whose inverse (at $ x \in \mathscr{E} $) is $\sum_{j\in \mathbb{J}}c_j^{-1}\langle x, x_j\rangle x_j , $ and  $ VU^*=V\geq 0.$
 \item Following  show that  $ \{x_j\}_{j\in\mathbb{J}}, \{ \tau_j\}_{j\in\mathbb{J}}$ are Bessel sequences (w.r.t. themselves): 
 $ \sum_{j\in\mathbb{S}}\langle x,x_j\rangle\langle x_j, x\rangle = \sum_{j\in\mathbb{S}}\langle x,Uf_j\rangle\langle Uf_j, x\rangle = \sum_{j\in\mathbb{S}}\langle U^*x,f_j\rangle\langle f_j, U^*x\rangle \leq \langle U^*x,U^*x \rangle\leq\|U^*\|^2\langle x, x\rangle ,
  \sum_{j\in\mathbb{S}}\langle y,\tau_j\rangle\langle \tau_j, y\rangle = \sum_{j\in\mathbb{S}}\langle y,Vf_j\rangle\langle Vf_j,y \rangle\leq \langle V^*y, V^*y \rangle \leq\|V^*\|^2\langle y, y \rangle, \forall x,y\in \mathscr{E}$, for each finite $\mathbb{S}\subseteq\mathbb{J}.$ Further, $  \|(VU^*)^{-1}\|^{-1}\langle x,x  \rangle\leq \langle VU^*x,x \rangle \leq  \|VU^*\|\langle x, x \rangle, \forall x \in \mathscr{E},$ $ \langle VU^*x,x \rangle =\langle U^*x,V^*x \rangle =\sum_{j\in\mathbb{J}}\langle U^*x,f_j\rangle\langle f_j,V^*x\rangle=\sum_{j\in\mathbb{J}}\langle x,Uf_j\rangle\langle Vf_j,x\rangle=\sum_{j\in\mathbb{J}}\langle x,x_j\rangle \langle \tau_j,x\rangle, \forall x \in \mathscr{E} $
 and $ \sum_{j\in\mathbb{J}}\langle x,x_j\rangle \tau_j=\sum_{j\in\mathbb{J}}\langle x,Uf_j\rangle Vf_j=V(\sum_{j\in\mathbb{J}}\langle U^*x,f_j\rangle f_j )=VU^*x=UV^*x=U(\sum_{j\in\mathbb{J}}\langle V^*x,f_j\rangle f_j )=\sum_{j\in\mathbb{J}}\langle x,Vf_j\rangle Uf_j=\sum_{j\in\mathbb{J}}\langle x,\tau_j\rangle x_j, \forall x \in \mathscr{E}.$
 \end{enumerate}
\end{proof}
 \begin{theorem}
 Let $ (\{x_j=Uf_j\}_{j\in\mathbb{J}}, \{\tau_j=Vf_j\}_{j\in\mathbb{J}})$ be a Riesz basis for  $\mathscr{E}$ over $ \mathscr{A}.$ Then
\begin{enumerate}[\upshape(i)]
 \item There exist unique  $ \{y_j\}_{j\in\mathbb{J}}, \{\omega_j\}_{j\in\mathbb{J}}$ in $\mathscr{E}$ such that 
$$ x= \sum\limits_{j\in\mathbb{J}}\langle x, y_j\rangle x_j = \sum\limits_{j\in\mathbb{J}}\langle x, \omega_j\rangle\tau_j, ~ \forall x \in \mathscr{E}$$ and $( \{y_j\}_{j\in\mathbb{J}}, \{\omega_j\}_{j\in\mathbb{J}})$ is Riesz.
\item $ \{x_j\}_{j\in\mathbb{J}}, \{\tau_j\}_{j\in\mathbb{J}}$ are complete in  $\mathscr{E}$. If $ V^*U\geq0,$ then there are real $ a, b >0$ such that for every finite subset $ \mathbb{S}$ of $ \mathbb{J}$,
$$ a\sum\limits_{j\in \mathbb{S}}c_jc_j^*\leq \left\langle \sum\limits_{j\in\mathbb{S}}c_jx_j ,\sum\limits_{k\in\mathbb{S}}c_k\tau_k  \right \rangle \leq b\sum\limits_{j\in \mathbb{S}}c_jc_j^*, ~\forall c_j\in \mathscr{A}, \forall j\in \mathbb{S}.$$ 
\end{enumerate}
\end{theorem}
\begin{proof}
(i) is similar to Hilbert space situation. For (ii), from the invertibility of  $ U$, and  from (v) of Corollary \ref{CHARACTERIZATION4}, we see the completeness of $ \{x_j\}_{j\in\mathbb{J}}$  in  $\mathscr{E}$. Similarly $ \{\tau_j\}_{j\in\mathbb{J}}$ is also complete  in  $\mathscr{E}$. If $ V^*U\geq0,$ then
\begin{align*}
\frac{1}{\|(V^*U)^{-1}\|}\sum\limits_{j\in \mathbb{S}}c_jc_j^*&=\frac{1}{\|(V^*U)^{-1}\|}\left\langle \sum\limits_{j\in\mathbb{S}}c_jf_j ,\sum\limits_{k\in\mathbb{S}}c_kf_k  \right \rangle\leq \left\langle \sum\limits_{j\in\mathbb{S}}c_jUf_j ,\sum\limits_{k\in\mathbb{S}}c_kVf_k  \right \rangle\\
&=\left\langle \sum\limits_{j\in\mathbb{S}}c_jx_j ,\sum\limits_{k\in\mathbb{S}}c_k\tau_k  \right \rangle \leq \left\langle V^*U\left(\sum\limits_{j\in\mathbb{S}}c_jf_j \right),\sum\limits_{k\in\mathbb{S}}c_kf_k  \right \rangle \leq \|V^*U\| \sum\limits_{j\in \mathbb{S}}c_jc_j^*
\end{align*}
for all $c_j\in \mathscr{A}, \forall j\in \mathbb{S}.$
\end{proof}
\begin{theorem}
Let $ \{x_j\}_{j\in\mathbb{J}},  \{\tau_j\}_{j\in\mathbb{J}}$ be in $\mathscr{E}$ over  $ \mathscr{A}$. Then the following are equivalent.
\begin{enumerate}[\upshape(i)]
\item $( \{x_j\}_{j\in\mathbb{J}},  \{\tau_j\}_{j\in\mathbb{J}})$ is a Riesz basis for  $\mathscr{E}$.
\item $ \{x_j\}_{j\in\mathbb{J}},  \{\tau_j\}_{j\in\mathbb{J}}$ are complete  in $\mathscr{E}$, and there exist $ a, b, c, d >0$ such that for every finite subset $ \mathbb{S}\subseteq\mathbb{J}$, 
\begin{align*}
a\left\|\sum_{j\in \mathbb{S}}c_jc_j^*\right\|^2\leq \left\|\sum_{j\in \mathbb{S}}c_jx_j\right\|^2\leq b\left\|\sum_{j\in \mathbb{S}}c_jc_j^*\right\|^2, ~\forall c_j \in \mathscr{A},\forall j\in \mathbb{S}, \\
c\left\|\sum_{j\in \mathbb{S}}d_jd_j^*\right\|^2\leq \left\|\sum_{j\in \mathbb{S}}d_j\tau_j\right\|^2\leq d\left\|\sum_{j\in \mathbb{S}}d_jd_j^*\right\|^2, ~\forall d_j \in \mathscr{A},\forall j\in \mathbb{S},
\end{align*}
and 
$$ \sum_{j\in \mathbb{J}}\langle x, x_j\rangle \langle \tau_j, x\rangle \geq 0, \forall x \in \mathscr{E}.$$
\end{enumerate}
\end{theorem} 
\begin{proof}
We have to use Lemma \ref{HOMOMORPHISMEXISTENCELEMMA} in  proving (ii) $ \Rightarrow$ (i). Other facts are similar to proof of the corresponding result in Hilbert space.	 
\end{proof}
\begin{proposition}\label{HILBERTSEQUENTIALPROP}
For every $ \{x_j\}_{j \in \mathbb{J}}  \in \mathscr{F}_\tau$,
\begin{enumerate}[\upshape (i)]
 \item $\theta_x^*(\{c_j\}_{j\in\mathbb{J}})=\sum_{j\in \mathbb{J}}c_jx_j, \theta_\tau^*(\{c_j\}_{j\in\mathbb{J}})=\sum_{j\in \mathbb{J}}c_j\tau_j,  \forall  \{c_j\}_{j\in\mathbb{J}} \in \mathscr{H}_\mathscr{A} ;$\\
 $\theta_x^* \theta_xy=  \sum_{j\in \mathbb{J}} \langle y,x_j\rangle x_j, \theta_\tau^* \theta_\tau y =  \sum_{j\in \mathbb{J}} \langle y,\tau_j\rangle \tau_j,  \forall y \in \mathscr{E}.$ 
 \item $ S_{x, \tau} = \theta_\tau^*\theta_x=\theta_x^*\theta_\tau .$ In particular,
 $$ S_{x, \tau}y =\sum\limits_{j\in \mathbb{J}}\langle y,x_j\rangle\tau_j=\sum\limits_{j\in \mathbb{J}}\langle y,\tau_j\rangle x_j,~ \forall y \in \mathscr{E} ~\operatorname{and}~$$
 $$\langle S_{x, \tau}y, z\rangle =\sum\limits_{j\in \mathbb{J}}\langle y,x_j\rangle\langle\tau_j, z\rangle=\sum\limits_{j\in \mathbb{J}}\langle y,\tau_j\rangle \langle x_j,z\rangle, ~ \forall y,z  \in \mathscr{E}.$$
 \item Every $ y \in \mathscr{E}$ can be written as 
 $$y =\sum\limits_{j\in \mathbb{J}}\langle y, S^{-1}_{x, \tau}\tau_j\rangle  x_j=\sum\limits_{j\in \mathbb{J}}\langle y,\tau_j \rangle S^{-1}_{x, \tau}  x_j =\sum\limits_{j\in \mathbb{J}}\langle y, S^{-1}_{x, \tau}x_j\rangle  \tau_j=\sum\limits_{j\in \mathbb{J}}\langle y, x_j\rangle S^{-1}_{x, \tau}  \tau_j.$$
 \item $(\{x_j\}_{j \in \mathbb{J}}, \{\tau_j \}_ {j \in \mathbb{J}})$ is Parseval if and only if $  \theta_\tau^*\theta_x=I_\mathscr{E}.$ 
 \item $(\{x_j\}_{j \in \mathbb{J}}, \{\tau_j \}_ {j \in \mathbb{J}})$ is Parseval  if and only if $ \theta_x\theta_\tau^* $ is idempotent.
 \item $\theta_xS_{x,\tau}^{-1}\theta_\tau^* $ is idempotent.
 \item $\theta_x$ and $ \theta_\tau$ are injective and  their  ranges are closed.
 \item $\theta_x^* $ and $ \theta_\tau^*$ are surjective.
 \item $\operatorname{Ker}(\theta_x)$ and $\operatorname{ Ker}(\theta_x)$ (resp. $\theta_x(\mathscr{E})$ and $\theta_x(\mathscr{E})$) are orthogonally complementable submodules of $\mathscr{E}$ (resp. $\mathscr{H}_\mathscr{A}$).
 \end{enumerate}
 \end{proposition} 
 \begin{proof}
Proof of (i)-(vi) are similar to  Hilbert space situation. For (ix), we have to use (vii) and Theorem \ref{MANUILOV2}.
\end{proof}
 We call $ P_{x,\tau}\coloneqq\theta_xS_{x,\tau}^{-1}\theta_\tau^* $ as the  \textit{frame idempotent}.  
\begin{definition}
 A frame  $( \{x_j\}_{j \in \mathbb{J}}, \{\tau_j\}_{j \in \mathbb{J}})$ for $ \mathscr{E}$ is called a Riesz frame if $P_{x,\tau} = I_{\mathscr{H}_\mathscr{A}} $. A Parseval and Riesz frame (i.e., $\theta_\tau^*\theta_x=I_\mathscr{E} $ and $\theta_x\theta_\tau^*= I_{\mathscr{H}_\mathscr{A}}$) is called as an orthonormal frame.
 \end{definition}
\begin{proposition}
\begin{enumerate}[\upshape(i)]
\item  If $( \{x_j\}_{j \in \mathbb{J}}, \{\tau_j\}_{j \in \mathbb{J}})$ is a Riesz basis for $\mathscr{E}$, then it is a  Riesz frame.
\item  If $( \{x_j\}_{j \in \mathbb{J}}, \{\tau_j=c_jx_j\}_{j \in \mathbb{J}})$ is an  orthonormal basis for $\mathscr{E}$, then it is  a Riesz frame.
\end{enumerate}
\end{proposition}
\begin{proposition}
A frame  $( \{x_j\}_{j \in \mathbb{J}}, \{\tau_j\}_{j \in \mathbb{J}})$ for $ \mathscr{E}$ is a Riesz frame if  and only if $ \theta_x(\mathscr{E})={\mathscr{H}_\mathscr{A}}$ if  and only if $ \theta_\tau(\mathscr{E})={\mathscr{H}_\mathscr{A}}.$
\end{proposition} 
\begin{proposition}
A frame  $ (\{x_j\}_{j\in \mathbb{J}}, \{\tau_j\}_{j\in \mathbb{J}}) $ for  $ \mathscr{E}$ over $ (\mathscr{A},e)$ is an orthonormal frame   if and only if it is a  Parseval frame  and $ \langle x_j,\tau_k \rangle =\delta_{j,k}e,\forall j, k \in \mathbb{J}$. 
\end{proposition}

\begin{theorem}
Let $(\{x_j\}_{j\in \mathbb{J}},\{\tau_j\}_{j\in \mathbb{J}} )$ be  a Parseval frame  for $\mathscr{E}$ such that $ \theta_x(\mathscr{E})=\theta_\tau(\mathscr{E})$ and $P_{x,\tau} $ is a projection. Then there exist a Hilbert C*-module  $ \mathscr{E}_1$ which contains $\mathscr{E}$ isometrically and  an orthonormal  frame  $(\{y_j\}_{j\in \mathbb{J}},\{\omega_j\}_{j\in \mathbb{J}} )$ for  $ \mathscr{E}_1$ such that $ x_j=Py_j,\tau_j=P\omega_j, \forall j \in \mathbb{J}$, where $P$ is the orthogonal projection from $\mathscr{E}_1$ onto $\mathscr{E}$. 
\end{theorem}
\begin{definition}
 A frame   $(\{y_j\}_{j\in \mathbb{J}}, \{\omega_j\}_{j\in \mathbb{J}})$  for  $\mathscr{E}$ is said to be a dual of frame  $(\{x_j\}_{j\in \mathbb{J}}, \{\tau_j\}_{j\in \mathbb{J}})$ for $\mathscr{E}$  if $ \theta_\omega^*\theta_x= \theta_y^*\theta_\tau=I_{\mathscr{E}}$. The `frame' $( \{\widetilde{x}_j\coloneqq S_{x,\tau}^{-1}x_j\}_{j\in \mathbb{J}},\{\widetilde{\tau}_j\coloneqq S_{x,\tau}^{-1}\tau_j\}_{j \in \mathbb{J}})$, which is a `dual' of $ (\{x_j\}_{j\in \mathbb{J}}, \{\tau_j\}_{j\in \mathbb{J}})$ is called the canonical dual of $ (\{x_j\}_{j\in \mathbb{J}}, \{\tau_j\}_{j\in \mathbb{J}})$.
 \end{definition}
 \begin{proposition}
 Let $( \{x_j\}_{j\in \mathbb{J}},\{\tau_j\}_{j\in \mathbb{J}} )$ be a frame for  $\mathscr{E}.$ If $ y \in \mathscr{E}$ has representation  $ y=\sum_{j\in\mathbb{J}}c_jx_j= \sum_{j\in\mathbb{J}}d_j\tau_j, $ for some sequences  $ \{c_j\}_{j\in \mathbb{J}},\{d_j\}_{j\in \mathbb{J}}$ in $\mathscr{A}$,  then 
 $$ \sum\limits_{j\in \mathbb{J}}c_jd^*_j =\sum\limits_{j\in \mathbb{J}}\langle S_{x, \tau}^{-1}y, \tau_j\rangle\langle x_j , S_{x, \tau}^{-1}y \rangle+\sum\limits_{j\in \mathbb{J}}(\langle c_j-\langle S_{x, \tau}^{-1}y, \tau_j\rangle)(d^*_j-\langle x_j, S_{x, \tau}^{-1}y\rangle). $$
 \end{proposition}
 \begin{proof}
 Right side $ =$
 \begin{align*}
 &\sum\limits_{j\in \mathbb{J}}\langle S_{x, \tau}^{-1}y, \tau_j\rangle\langle x_j , S_{x, \tau}^{-1}y \rangle+\sum\limits_{j\in \mathbb{J}}c_jd^*_j-\sum\limits_{j\in \mathbb{J}}c_j\langle x_j, S_{x, \tau}^{-1}y\rangle-\sum\limits_{j\in \mathbb{J}}\langle S_{x, \tau}^{-1}y, \tau_j\rangle d^*_j+\sum\limits_{j\in \mathbb{J}}\langle S_{x, \tau}^{-1}y, \tau_j\rangle\langle x_j, S_{x, \tau}^{-1}y\rangle\\
 &= 2\langle S_{x, \tau}S_{x, \tau}^{-1}y, S_{x, \tau}^{-1}y\rangle+\sum\limits_{j\in \mathbb{J}}c_jd^*_j-\left\langle\sum\limits_{j\in \mathbb{J}}c_jx_j,S_{x, \tau}^{-1}y \right\rangle -\left\langle S_{x, \tau}^{-1}y,\sum\limits_{j\in \mathbb{J}} d_j\tau_j \right\rangle\\
 &=2\langle y, S_{x, \tau}^{-1}y\rangle+\sum\limits_{j\in \mathbb{J}}c_jd^*_j-\langle y, S_{x, \tau}^{-1}y\rangle-\langle  S_{x, \tau}^{-1}y, y\rangle
 =\text{Left side.}
 \end{align*}
 \end{proof} 
\begin{theorem}
 Let $( \{x_j\}_{j\in \mathbb{J}},\{\tau_j\}_{j\in \mathbb{J}} )$ be a frame for $ \mathscr{E}$ with frame bounds $ a$ and $ b.$ Then 
 \begin{enumerate}[\upshape(i)]
 \item The canonical dual frame of the canonical dual frame  of $ (\{x_j\}_{j\in \mathbb{J}} ,\{\tau_j\}_{j\in \mathbb{J}} )$ is itself. 
 \item $ 1/b, 1/a$ are frame bounds for the canonical dual of $ (\{x_j\}_{j\in \mathbb{J}},\{\tau_j\}_{j\in \mathbb{J}})$.  
 \item If $ a, b $ are optimal frame bounds for $( \{x_j\}_{j\in \mathbb{J}} , \{\tau_j\}_{j\in \mathbb{J}}),$ then $ 1/b, 1/a$ are optimal  frame bounds for its canonical dual.
 \end{enumerate}
 \end{theorem}
\begin{proposition}
 Let  $ (\{x_j\}_{j\in \mathbb{J}}, \{\tau_j\}_{j\in \mathbb{J}}) $ and $ (\{y_j\}_{j\in \mathbb{J}}, \{\omega_j\}_{j\in \mathbb{J}}) $ be  frames for  $\mathscr{E}$. Then the following are equivalent.
 \begin{enumerate}[\upshape(i)]
 \item $ (\{y_j\}_{j\in \mathbb{J}},\{\omega_j\}_{j\in \mathbb{J}}) $ is dual of $( \{x_j\}_{j\in \mathbb{J}}, \{\tau_j\}_{j\in \mathbb{J}}) $. 
 \item $ \sum_{j\in \mathbb{J}}\langle z, x_j\rangle \omega_j= \sum_{j\in \mathbb{J}}\langle z, \tau_j\rangle y_j=z, \forall z \in  \mathscr{E}.$ 
 \end{enumerate}
 \end{proposition}
\begin{theorem}
Let $ (\{x_j\}_{j\in \mathbb{J}}, \{\tau_j\}_{j\in \mathbb{J}})$  be a  frame for   $ \mathscr{E}$. If $ (\{x_j\}_{j\in \mathbb{J}}, \{\tau_j\}_{j\in \mathbb{J}})$ is a Riesz  basis  for   $ \mathscr{E}$, then $ (\{x_j\}_{j\in \mathbb{J}}, \{\tau_j\}_{j\in \mathbb{J}}) $ has unique dual. Converse holds if $ \theta_x( \mathscr{E})=\theta_\tau( \mathscr{E})$.
\end{theorem}
\begin{proof}
From Proposition \ref{HILBERTSEQUENTIALPROP},  $\theta_x(\mathscr{E})$ and $ \theta_\tau(\mathscr{E}) $ are orthogonally complementable submodules of  $\mathscr{H}_\mathscr{A}$. Rest are similar to Hilbert space situation.
\end{proof} 
\begin{proposition}
Let $ (\{x_j\}_{j\in \mathbb{J}}, \{\tau_j\}_{j\in \mathbb{J}}) $  be a  frame for   $\mathscr{E}$. If $( \{y_j\}_{j\in \mathbb{J}}, \{\omega_j\}_{j\in \mathbb{J}})$ is a dual of $ (\{x_j\}_{j\in \mathbb{J}}, \{\tau_j\}_{j\in \mathbb{J}}) $, then  there exist Bessel sequences $ \{z_j\}_{j\in \mathbb{J}}$ and $ \{\rho_j\}_{j\in \mathbb{J}} $ for  $\mathscr{E}$ such that $ y_j=S_{x,\tau}^{-1}x_j+z_j, \omega_j=S_{x,\tau}^{-1}\tau_j+\rho_j,\forall j \in \mathbb{J}$, and $\theta_z(\mathscr{E})\perp \theta_\tau(\mathscr{E}),\theta_\rho(\mathscr{E})\perp \theta_x(\mathscr{E})$. Converse holds if  $ \theta_\rho^*\theta_z \geq 0$.
\end{proposition} 
\begin{lemma}
Let  $ (\{x_j\}_{j\in \mathbb{J}}, \{\tau_j\}_{j\in \mathbb{J}}) $  be a  frame for   $\mathscr{E}$ and $ \{e_j\}_{j\in \mathbb{J}}$ be the standard orthonormal basis for $ \mathscr{H}_\mathscr{A}$. Then the dual frames  of $ (\{x_j\}_{j\in \mathbb{J}}, \{\tau_j\}_{j\in \mathbb{J}})$ are precisely $ (\{y_j=Ue_j\}_{j\in \mathbb{J}}, \{\omega_j=Ve_j\}_{j\in \mathbb{J}})$, where $ U,V: \mathscr{H}_\mathscr{A} \rightarrow \mathscr{E}$ are bounded left-inverses of $ \theta_\tau, \theta_x$, respectively, such that $ VU^*$ is positive invertible.
\end{lemma} 
\begin{lemma}
Let $ (\{x_j\}_{j\in \mathbb{J}}, \{\tau_j\}_{j\in \mathbb{J}}) $  be a  frame for   $ \mathscr{E}$ over $ \mathscr{A}$. Then the  adjointable left-inverses of 
\begin{enumerate}[\upshape(i)]
\item $ \theta_x$ are precisely   $S_{x,\tau}^{-1}\theta_\tau^*+U(I_{\mathscr{H}_\mathscr{A}}-\theta_xS_{x,\tau}^{-1}\theta_\tau^*)$, where $U\in \operatorname{Hom}_\mathscr{A}^*(\mathscr{H}_\mathscr{A}, \mathscr{E})$.
\item $ \theta_\tau$ are precisely  $S_{x,\tau}^{-1}\theta_x^*+V(I_{\mathscr{H}_\mathscr{A}}-\theta_\tau S_{x,\tau}^{-1}\theta_x^*)$, where $V\in \operatorname{Hom}_\mathscr{A}^*( \mathscr{H}_\mathscr{A}, \mathscr{E})$.
\end{enumerate}	
\end{lemma}
\begin{theorem}
Let  $ (\{x_j\}_{j\in \mathbb{J}}, \{\tau_j\}_{j\in \mathbb{J}}) $  be a  frame for  $ \mathscr{E}$ over $ \mathscr{A}$. The dual frames 	 $ (\{y_j\}_{j\in \mathbb{J}}, \{\omega_j\}_{j\in \mathbb{J}}) $ of $ (\{x_j\}_{j\in \mathbb{J}}, \{\tau_j\}_{j\in \mathbb{J}}) $ are precisely  
\begin{align*}
(\{y_j=S_{x,\tau}^{-1}x_j+Ve_j-V\theta_\tau S_{x,\tau}^{-1}x_j\}_{j\in \mathbb{J}},
\{\omega_j=S_{x,\tau}^{-1}\tau_j+Ue_j-U\theta_xS_{x,\tau}^{-1}\tau_j\}_{j\in \mathbb{J}})
\end{align*}
such that 
$$S_{x,\tau}^{-1}+UV^*-U\theta_xS_{x,\tau}^{-1}\theta_\tau^*V^* $$
is positive invertible, where  $ \{e_j\}_{j\in \mathbb{J}}$ is  the standard orthonormal basis for $ \mathscr{H}_\mathscr{A}$, and $U, V\in \operatorname{Hom}_\mathscr{A}^*( \mathscr{H}_\mathscr{A}, \mathscr{E})$.
\end{theorem}
\begin{definition}
A frame   $(\{y_j\}_{j\in \mathbb{J}},  \{\omega_j\}_{j\in \mathbb{J}})$  for  $\mathscr{E}$ is said to be orthogonal to a frame   $( \{x_j\}_{j\in \mathbb{J}}, \{\tau_j\}_{j\in \mathbb{J}})$ for  $\mathscr{E}$ if $ \theta_\omega^*\theta_x= \theta_y^*\theta_\tau=0.$
\end{definition}
\begin{proposition}
 Let  $ (\{x_j\}_{j\in \mathbb{J}}, \{\tau_j\}_{j\in \mathbb{J}}) $ and $ (\{y_j\}_{j\in \mathbb{J}}, \{\omega_j\}_{j\in \mathbb{J}}) $ be  frames for   $\mathscr{E}$. Then the following are equivalent.
 \begin{enumerate}[\upshape(i)]
 \item $ (\{y_j\}_{j\in \mathbb{J}},\{\omega_j\}_{j\in \mathbb{J}}) $ is orthogonal to  $( \{x_j\}_{j\in \mathbb{J}},  \{\tau_j\}_{j\in \mathbb{J}}) $.
 \item $ \sum_{j\in \mathbb{J}}\langle z, x_j\rangle \omega_j= \sum_{j\in \mathbb{J}}\langle z, \tau_j\rangle y_j=0, \forall z \in  \mathscr{E}.$ 
 \end{enumerate}
 \end{proposition}
\begin{proposition}
 Two orthogonal frames   have common dual frame.	
\end{proposition} 
\begin{proposition}
Let $ (\{x_j\}_{j\in \mathbb{J}}, \{\tau_j\}_{j\in \mathbb{J}}) $ and $ (\{y_j\}_{j\in \mathbb{J}}, \{\omega_j\}_{j\in \mathbb{J}}) $ be  two Parseval frames for  $\mathscr{E}$ over $(\mathscr{A},e)$ which are  orthogonal. If $A,B,C,D \in \operatorname{End}^*_\mathscr{A}(\mathscr{E})$ are such that $ AC^*+BD^*=I_\mathscr{E}$, then  $ (\{Ax_j+By_j\}_{j\in \mathbb{J}}, \{C\tau_j+D\omega_j\}_{j\in \mathbb{J}}) $ is a  Parseval frame for  $\mathscr{E}$. In particular,  if  $ a,b,c,d \in \mathscr{A}$ satisfy $ac^*+bd^* =e$, then $ (\{ax_j+by_j\}_{j\in \mathbb{J}}, \{c\tau_j+d\omega_j\}_{j\in \mathbb{J}}) $ is a  Parseval frame for  $\mathscr{E}$.
\end{proposition}  
\begin{definition}
Two frames  $(\{x_j\}_{j\in \mathbb{J}},  \{\tau_j\}_{j\in \mathbb{J}}) $ and $(\{y_j\}_{j\in \mathbb{J}},\{\omega_j\}_{j\in \mathbb{J}})$  for $ \mathscr{E}$ are called disjoint if $(\{x_j\oplus y_j\}_{j\in \mathbb{J}}, \{\tau_j\oplus\omega_j\}_{j\in \mathbb{J}})$ is a frame for $\mathscr{E}\oplus\mathscr{E} $.
\end{definition} 
\begin{proposition}
If $(\{x_j\}_{j\in \mathbb{J}},\{\tau_j\}_{j\in \mathbb{J}} )$  and $ (\{y_j\}_{j\in \mathbb{J}}, \{\omega_j\}_{j\in \mathbb{J}} )$  are  disjoint  frames  for $\mathscr{E}$, then  they  are disjoint. Further, if both $(\{x_j\}_{j\in \mathbb{J}},\{\tau_j\}_{j\in \mathbb{J}} )$  and $ (\{y_j\}_{j\in \mathbb{J}}, \{\omega_j\}_{j\in \mathbb{J}} )$ are  Parseval, then $(\{x_j\oplus y_j\}_{j \in \mathbb{J}},\{\tau_j\oplus \omega_j\}_{j \in \mathbb{J}})$ is Parseval.
\end{proposition}

 \textbf{Characterizations}
\begin{theorem}
Let $ \{f_j\}_{j \in \mathbb{J}}$ be an arbitrary orthonormal basis for $ \mathscr{E}$ over  $\mathscr{A}.$ Then 
\begin{enumerate}[\upshape(i)]
\item The orthonormal  bases  $ ( \{x_j\}_{j \in \mathbb{J}},\{\tau_j\}_{j \in \mathbb{J}})$ for $ \mathscr{E}$  are precisely $( \{Uf_j\}_{j \in \mathbb{J}},\{c_jUf_j\}_{j \in \mathbb{J}}) $, where $ U \in  \operatorname{End}^*_\mathscr{A}(\mathscr{E}) $ is unitary and $ c_j$'s are positive invertible elements in the center of $ \mathscr{A}$  such that $ 0<\inf\{\|c_j\|\}_{j \in \mathbb{J}}\leq \sup\{\|c_j\|\}_{j \in \mathbb{J}}< \infty.$
\item The Riesz bases  $ ( \{x_j\}_{j \in \mathbb{J}},\{\tau_j\}_{j \in \mathbb{J}})$ for $ \mathscr{E}$  are precisely $( \{Uf_j\}_{j \in \mathbb{J}},\{Vf_j\}_{j \in \mathbb{J}}) $, where $ U,V \in  \operatorname{End}^*_\mathscr{A}(\mathscr{E})$ are invertible  such that  $ VU^*$ is positive.
\item The frames $ ( \{x_j\}_{j \in \mathbb{J}},\{\tau_j\}_{j \in \mathbb{J}})$ for $ \mathscr{E}$  are precisely $( \{Uf_j\}_{j \in \mathbb{J}},\{Vf_j\}_{j \in \mathbb{J}}) $, where $ U,V \in  \operatorname{End}^*_\mathscr{A}(\mathscr{E}) $ are such that   $ VU^*$ is positive invertible.
\item The Bessel sequences  $ ( \{x_j\}_{j \in \mathbb{J}},\{\tau_j\}_{j \in \mathbb{J}})$ for $ \mathscr{E}$  are precisely $( \{Uf_j\}_{j \in \mathbb{J}},\{Vf_j\}_{j \in \mathbb{J}}) $, where $ U,V  \in  \operatorname{End}^*_\mathscr{A}(\mathscr{E}) $  are such that  $ VU^*$ is positive.
\item The Riesz frames $ ( \{x_j\}_{j \in \mathbb{J}},\{\tau_j\}_{j \in \mathbb{J}})$ for $ \mathscr{E}$  are precisely $( \{Uf_j\}_{j \in \mathbb{J}},\{Vf_j\}_{j \in \mathbb{J}}) $, where $ U,V  \in  \operatorname{End}^*_\mathscr{A}(\mathscr{E}) $ are such that  $ VU^*$ is positive invertible and $ U^*(VU^*)^{-1}V=I_{\mathscr{E}}$.  
\item The orthonormal frames $ ( \{x_j\}_{j \in \mathbb{J}},\{\tau_j\}_{j \in \mathbb{J}})$ for $ \mathscr{E}$  are precisely $( \{Uf_j\}_{j \in \mathbb{J}},\{Vf_j\}_{j \in \mathbb{J}}) $, where $ U,V  \in  \operatorname{End}^*_\mathscr{A}(\mathscr{E}) $  are such that $ VU^*=I_\mathscr{E}= U^*V$.   
\end{enumerate}
\end{theorem} 
\begin{corollary}
\begin{enumerate}[\upshape(i)]
\item If $ (\{x_j\}_{j \in \mathbb{J}},\{\tau_j=c_jx_j\}_{j \in \mathbb{J}})$ is an orthonormal basis for $ \mathscr{E}$, then $ \|x_j\|=1, \forall j \in \mathbb{J}, \|\tau_j\|=\|c_j\|, \forall j \in \mathbb{J}.$
\item If $ (\{x_j\}_{j \in \mathbb{J}},\{\tau_j\}_{j \in \mathbb{J}})$ is a Riesz basis for $ \mathscr{E}$, then
$$ \frac{1}{\|U^{-1}\|} \leq \|x_j\| \leq \|U\|, ~\forall j \in \mathbb{J}, ~  \frac{1}{\|V^{-1}\|} \leq \|\tau_j\| \leq \|V\|, ~\forall j \in \mathbb{J}.$$
\item If $ (\{x_j\}_{j \in \mathbb{J}},\{\tau_j\}_{j \in \mathbb{J}})$ is a Bessel sequence for $ \mathscr{E}$, then
$  \|x_j\| \leq \|U\|, \forall j \in \mathbb{J},  \|\tau_j\| \leq \|V\|, \forall j \in \mathbb{J}.$
\end{enumerate}
\end{corollary} 
\begin{corollary}
Let $ \{f_j\}_{j \in \mathbb{J}}$ be an arbitrary orthonormal basis for $ \mathscr{E}$ over  $\mathscr{A}.$ Then 
\begin{enumerate}[\upshape(i)]
\item The orthonormal  bases  $ ( \{x_j\}_{j \in \mathbb{J}},\{x_j\}_{j \in \mathbb{J}})$ for $ \mathscr{E}$  are precisely $( \{Uf_j\}_{j \in \mathbb{J}},\{Uf_j\}_{j \in \mathbb{J}}) $, where $ U \in  \operatorname{End}^*_\mathscr{A}(\mathscr{E}) $ is unitary.
\item The Riesz bases  $ ( \{x_j\}_{j \in \mathbb{J}},\{x_j\}_{j \in \mathbb{J}})$ for $ \mathscr{E}$  are precisely $( \{Uf_j\}_{j \in \mathbb{J}},\{Uf_j\}_{j \in \mathbb{J}}) $, where $ U\in  \operatorname{End}^*_\mathscr{A}(\mathscr{E})$ is  invertible.
\item The frames $ ( \{x_j\}_{j \in \mathbb{J}},\{x_j\}_{j \in \mathbb{J}})$ for $ \mathscr{E}$  are precisely $( \{Uf_j\}_{j \in \mathbb{J}},\{Uf_j\}_{j \in \mathbb{J}}) $, where $ U\in  \operatorname{End}^*_\mathscr{A}(\mathscr{E}) $ is such  that   $ UU^*$ is invertible.
\item The Bessel sequences  $ ( \{x_j\}_{j \in \mathbb{J}},\{x_j\}_{j \in \mathbb{J}})$ for $ \mathscr{E}$  are precisely $( \{Uf_j\}_{j \in \mathbb{J}},\{Uf_j\}_{j \in \mathbb{J}}) $, where $ U  \in  \operatorname{End}^*_\mathscr{A}(\mathscr{E}) $. 
\item The Riesz frames $ ( \{x_j\}_{j \in \mathbb{J}},\{x_j\}_{j \in \mathbb{J}})$ for $ \mathscr{E}$  are precisely $( \{Uf_j\}_{j \in \mathbb{J}},\{Uf_j\}_{j \in \mathbb{J}}) $, where $ U  \in  \operatorname{End}^*_\mathscr{A}(\mathscr{E}) $ is  such that  $ UU^*$ is invertible and $ U^*(UU^*)^{-1}U=I_{\mathscr{E}}$.  
\item The orthonormal frames $ ( \{x_j\}_{j \in \mathbb{J}},\{x_j\}_{j \in \mathbb{J}})$ for $ \mathscr{E}$  are precisely $( \{Uf_j\}_{j \in \mathbb{J}},\{Uf_j\}_{j \in \mathbb{J}}) $, where $ U  \in  \operatorname{End}^*_\mathscr{A}(\mathscr{E}) $  is  such that $ UU^*=I_\mathscr{E}= U^*U$.  
\item $ ( \{x_j\}_{j \in \mathbb{J}},\{\tau_j\}_{j \in \mathbb{J}})$ is an orthonormal basis for $ \mathscr{E}$ if and only if it an orthonormal frame. 
\end{enumerate}
\end{corollary} 
\begin{theorem}
Let $\{x_j\}_{j\in\mathbb{J}},\{\tau_j\}_{j\in\mathbb{J}}$ be in $ \mathscr{E}$ over $ \mathscr{A}$. Then $(\{x_j\}_{j\in\mathbb{J}},\{\tau_j\}_{j\in\mathbb{J}})$  is a  frame with bounds $ a $ and $ b$ (resp. Bessel with bound $ b$) 
\begin{enumerate}[\upshape(i)]
\item  if and only if $$U:\mathscr{H}_\mathscr{A} \ni \{c_j\}_{j\in\mathbb{J}}\mapsto\sum\limits_{j\in\mathbb{J}}c_jx_j \in  \mathscr{E}, ~\text{and} ~ V:\mathscr{H}_\mathscr{A} \ni \{d_j\}_{j\in\mathbb{J}}\mapsto\sum\limits_{j\in\mathbb{J}}d_j\tau_j \in  \mathscr{E} $$ 
are well-defined, $ U,V \in \operatorname{Hom}^*_\mathscr{A}(\mathscr{H}_\mathscr{A} ,\mathscr{E})$  such that $ aI_ \mathscr{E}\leq VU^*\leq bI_ \mathscr{E}$   (resp. $ 0\leq VU^*\leq bI_ \mathscr{E}).$
\item  if and only if $$U:\mathscr{H}_\mathscr{A} \ni\{c_j\}_{j\in\mathbb{J}}\mapsto\sum\limits_{j\in\mathbb{J}}c_jx_j \in    \mathscr{E}, ~\text{and} ~ S:  \mathscr{E} \ni z\mapsto\{\langle z, \tau_j\rangle\}_{j\in\mathbb{J}} \in \mathscr{H}_\mathscr{A} $$ 
are well-defined, $ U\in \operatorname{Hom}^*_\mathscr{A}(\mathscr{H}_\mathscr{A} ,\mathscr{E})$, $ S \in \operatorname{Hom}^*_\mathscr{A}(\mathscr{E},\mathscr{H}_\mathscr{A})$  such that  $ aI_ \mathscr{E}\leq S^*U^*\leq bI_ \mathscr{E}$ (resp.  $ 0\leq S^*U^*\leq bI_ \mathscr{E}).$
\item  if and only if $$ R:  \mathscr{E} \ni y\mapsto\{\langle y, x_j\rangle\}_{j\in\mathbb{J}} \in \mathscr{H}_\mathscr{A}, ~\text{and} ~  V:\mathscr{H}_\mathscr{A} \ni\{d_j\}_{j\in\mathbb{J}}\mapsto\sum\limits_{j\in\mathbb{J}}d_j\tau_j \in  \mathscr{E}$$ 
are well-defined, $ R \in \operatorname{Hom}^*_\mathscr{A}(\mathscr{E},\mathscr{H}_\mathscr{A})$, $ V \in \operatorname{Hom}^*_\mathscr{A}(\mathscr{H}_\mathscr{A} ,\mathscr{E})$ such that  $ aI_ \mathscr{E}\leq VR\leq bI_ \mathscr{E}$  (resp.  $ 0\leq VR\leq bI_ \mathscr{E}).$
\item if and only if $$ R: \mathscr{E} \ni y\mapsto\{\langle y, x_j\rangle\}_{j\in\mathbb{J}} \in \mathscr{H}_\mathscr{A}, ~\text{and} ~  S:  \mathscr{E} \ni z\mapsto\{\langle z, \tau_j\rangle\}_{j\in\mathbb{J}} \in \mathscr{H}_\mathscr{A} $$
are well-defined, $ R,S \in \operatorname{Hom}^*_\mathscr{A}(\mathscr{E},\mathscr{H}_\mathscr{A})$ such that  $ aI_ \mathscr{E}\leq S^*R\leq bI_ \mathscr{E}$ (resp.   $ 0\leq S^*R\leq bI_ \mathscr{E}).$   
\end{enumerate}
\end{theorem}
\begin{theorem}
Let $ \{x_j\}_{j\in \mathbb{J}}, \{\tau_j\}_{j\in \mathbb{J}} $ be in $\mathscr{E}$. If  $(\{x_j\}_{j\in \mathbb{J}},\{\tau_j\}_{j\in \mathbb{J}})$ is a frame, then  
\begin{enumerate}[\upshape(i)]
\item there are $a ,b> 0$ such that
$$a\|x\|^2\leq \left\|\sum_{j\in \mathbb{J}} \langle x, x_j\rangle\langle \tau_j, x\rangle\right\|=\left\|\sum_{j\in \mathbb{J}} \langle x, \tau_j\rangle\langle x_j, x \rangle\right\| \leq b\|x\|^2 , ~\forall x \in \mathscr{E},$$  
\item there are $ c,d> 0$ such that 
$$ \sum\limits_{j \in \mathbb{J}}  \langle x, x_j\rangle\langle x_j, x\rangle  \leq c\langle x, x \rangle , ~ \forall x \in \mathscr{E} ~ \text{and} ~ \sum\limits_{j \in \mathbb{J}}  \langle x, \tau_j\rangle\langle \tau_j, x\rangle \leq d \langle x, x \rangle, ~ \forall x \in \mathscr{E}.$$
 \end{enumerate}
 If $ \langle x,x_j \rangle \tau_j=\langle x,\tau_j \rangle x_j, \langle x, x_j\rangle  \langle \tau_j,x \rangle \geq 0, \forall x \in \mathscr{E}, \forall j \in \mathbb{J}$,  then the converse holds.
 \end{theorem} 
 \begin{proof}
$ (\Rightarrow)$ (resp. $(\Leftarrow)$) is an application of Theorem \ref{PASCHKE} (resp. Theorem \ref{ARAMBASIC1}).
 \end{proof} 
 \begin{corollary}
 Let $ \{x_j\}_{j\in \mathbb{J}}, \{\tau_j\}_{j\in \mathbb{J}} $ be in $\mathscr{E}$. If  $(\{x_j\}_{j\in \mathbb{J}},\{\tau_j\}_{j\in \mathbb{J}})$ is Bessel, then  
 \begin{enumerate}[\upshape(i)]
 \item there is  $b> 0$ such that
 $$ \left\|\sum_{j\in \mathbb{J}} \langle x, x_j\rangle\langle \tau_j, x\rangle\right\|=\left\|\sum_{j\in \mathbb{J}} \langle x, \tau_j\rangle\langle x_j, x \rangle\right\| \leq b\|x\|^2 , ~\forall x \in \mathscr{E},$$  
\item there are $ c,d> 0$ such that 
 $$ \sum\limits_{j \in \mathbb{J}}  \langle x, x_j\rangle\langle x_j, x\rangle \leq c\langle x, x \rangle , ~ \forall x \in \mathscr{E} ~ \text{and} ~\sum\limits_{j \in \mathbb{J}}  \langle x, \tau_j\rangle\langle \tau_j, x\rangle\leq d \langle x, x \rangle, ~ \forall x \in \mathscr{E}.$$
 \end{enumerate}
 If $ \langle x,x_j \rangle \tau_j=\langle x,\tau_j \rangle x_j, \langle x, x_j\rangle  \langle \tau_j,x \rangle \geq 0, \forall x \in \mathscr{E}, \forall j \in \mathbb{J}$,  then the converse holds.	
 \end{corollary} 
  
 \textbf{Similarity  and tensor product}
 \begin{definition}
A frame $ (\{y_j\}_{j\in \mathbb{J}},\{\omega_j\}_{j\in \mathbb{J}})$ for  $ \mathscr{E}$ is said to  similar to a frame  $ (\{x_j\}_{j\in \mathbb{J}},\{\tau_j\}_{j\in \mathbb{J}})$ for  $ \mathscr{E}$  if there are invertible   $ T_{x,y}, T_{\tau,\omega} \in \operatorname{End}^*_\mathscr{A}(\mathscr{E})$ such that $ y_j=T_{x,y}x_j, \omega_j=T_{\tau,\omega}\tau_j,   \forall j \in \mathbb{J}.$
 \end{definition} 
 \begin{proposition}
  Let $ \{x_j\}_{j\in \mathbb{J}}\in \mathscr{F}_\tau$  with frame bounds $a, b,$  let $T_{x,y} , T_{\tau,\omega}\in \operatorname{End}^*_\mathscr{A}(\mathscr{E})$ be positive, invertible, commute with each other, commute with $ S_{x, \tau}$, and let $y_j=T_{x,y}x_j , \omega_j=T_{\tau,\omega}\tau_j,  \forall j \in \mathbb{J}.$ Then 
 \begin{enumerate}[\upshape(i)]
 \item $ \{y_j\}_{j\in \mathbb{J}}\in \mathscr{F}_\tau$ and $ \frac{a}{\|T_{x,y}^{-1}\|\|T_{\tau,\omega}^{-1}\|}\leq S_{y, \omega} \leq b\|T_{x,y}T_{\tau,\omega}\|.$ Assuming that $ (\{x_j\}_{j\in \mathbb{J}},\{\tau_j\}_{j\in \mathbb{J}})$ is Parseval, then $(\{y_j\}_{j\in \mathbb{J}},  \{\omega_j\}_{j\in \mathbb{J}})$ is Parseval  if and only if   $ T_{\tau, \omega}T_{x,y}=I_\mathscr{E}.$  
 \item $ \theta_y=\theta_x T_{x,y}, \theta_\omega=\theta_\tau T_{\tau,\omega}, S_{y,\omega}=T_{\tau,\omega}S_{x, \tau}T_{x,y},  P_{y,\omega} =P_{x, \tau}.$
 \end{enumerate}
 \end{proposition} 
 \begin{lemma}
 Let $ \{x_j\}_{j\in \mathbb{J}}\in \mathscr{F}_\tau,$ $ \{y_j\}_{j\in \mathbb{J}}\in \mathscr{F}_\omega$ and   $y_j=T_{x, y}x_j , \omega_j=T_{\tau,\omega}\tau_j,  \forall j \in \mathbb{J}$, for some invertible $T_{x,y}, T_{\tau,\omega}\in \operatorname{End}^*_\mathscr{A}(\mathscr{E}).$ Then 
 $ \theta_y=\theta_x T^*_{x,y}, \theta_\omega=\theta_\tau T^*_{\tau,\omega}, S_{y,\omega}=T_{\tau,\omega}S_{x, \tau}T_{x,y}^*,  P_{y,\omega}=P_{x, \tau}.$ Assuming that $ (\{x_j\}_{j\in \mathbb{J}},\{\tau_j\}_{j\in \mathbb{J}})$ is Parseval frame, then $ (\{y_j\}_{j\in \mathbb{J}},\{\omega_j\}_{j\in \mathbb{J}})$ is Parseval frame if and only if $T_{\tau,\omega}T_{x,y}^*=I_\mathscr{E}.$
 \end{lemma}
 \begin{theorem}
 Let $ \{x_j\}_{j\in \mathbb{J}}\in \mathscr{F}_\tau,$ $ \{y_j\}_{j\in \mathbb{J}}\in \mathscr{F}_\omega.$ The following are equivalent.
 \begin{enumerate}[\upshape(i)]
 \item $y_j=T_{x,y}x_j , \omega_j=T_{\tau, \omega}\tau_j ,  \forall j \in \mathbb{J},$ for some invertible  $ T_{x,y} ,T_{\tau, \omega} \in \operatorname{End}^*_\mathscr{A}(\mathscr{E}). $
 \item $\theta_y=\theta_x{T'}_{x,y}^* , \theta_\omega=\theta_\tau {T'}_{\tau, \omega}^* $ for some invertible  $ {T'}_{x,y},{T'}_{\tau, \omega} \in \operatorname{End}^*_\mathscr{A}(\mathscr{E}). $
 \item $P_{y,\omega}=P_{x,\tau}.$
 \end{enumerate}
 If one of the above conditions is satisfied, then  invertible homomorphisms in $ \operatorname{(i)}$ and  $\operatorname{(ii)}$ are unique and are given by $T_{x,y}=\theta_y^*\theta_\tau S_{x,\tau}^{-1}, T_{\tau, \omega}=\theta_\omega^*\theta_x S_{x,\tau}^{-1}.$
 In the case that $(\{x_j\}_{j\in \mathbb{J}},  \{\tau_j\}_{j\in \mathbb{J}})$ is Parseval, then $(\{y_j\}_{j\in \mathbb{J}},  \{\omega_j\}_{j\in \mathbb{J}})$ is  Parseval if and only if $T_{\tau, \omega}T_{x,y}^* =I_\mathscr{E}$   if and only if $ T_{x,y}^*T_{\tau, \omega} =I_\mathscr{E}$.
 \end{theorem}
 \begin{corollary}
 For any given frame $ (\{x_j\}_{j \in \mathbb{J}} , \{\tau_j\}_{j \in \mathbb{J}})$, the canonical dual of $ (\{x_j\}_{j \in \mathbb{J}} , \{\tau_j\}_{j \in \mathbb{J}}  )$ is the only dual frame that is similar to $ (\{x_j\}_{j \in \mathbb{J}} , \{\tau_j\}_{j \in \mathbb{J}} )$.
 \end{corollary}
 \begin{corollary}
 Two similar  frames cannot be orthogonal.
\end{corollary}
\begin{remark}
For every frame  $(\{x_j\}_{j \in \mathbb{J}}, \{\tau_j\}_{j \in \mathbb{J}}),$ each  of `frames'  $( \{S_{x, \tau}^{-1}x_j\}_{j \in \mathbb{J}}, \{\tau_j\}_{j \in \mathbb{J}})$,    $( \{S_{x, \tau}^{-1/2}x_j\}_{j \in \mathbb{J}}, \{S_{x,\tau}^{-1/2}\tau_j\}_{j \in \mathbb{J}}),$ and  $ (\{x_j \}_{j \in \mathbb{J}}, \{S_{x,\tau}^{-1}\tau_j\}_{j \in \mathbb{J}})$ is a  Parseval frame which is similar to  $ (\{x_j\}_{j \in \mathbb{J}} , \{\tau_j\}_{j \in \mathbb{J}} ).$  
 \end{remark} 
\textbf{Tensor product  of frames}: Let $(\{x_j\}_{j \in \mathbb{J}}, \{\tau_j\}_{j \in \mathbb{J}})$ be  a frame  for   $ \mathscr{E},$ and $(\{y_l\}_{l \in \mathbb{L}}, \{\omega_l\}_{l \in \mathbb{L}})$  be a frame for   $ \mathscr{E}_1.$ The frame   $(\{z_{(j, l)}\coloneqq x_j\otimes y_l\}_{(j, l)\in \mathbb{J}\bigtimes  \mathbb{L}},\{\rho_{(j, l)}\coloneqq \tau_j\otimes\omega_l\}_{(j, l)\in \mathbb{J}\bigtimes  \mathbb{L}} )$   for $ \mathscr{E}\otimes\mathscr{E}_1$ is called  as tensor product  of frames $( \{x_j\}_{j \in \mathbb{J}}, \{\tau_j\}_{j\in \mathbb{J}})$ and $( \{y_l\}_{l \in \mathbb{L}},  \{\omega_l\}_{l\in \mathbb{L}}).$ 
 \begin{proposition}
 Let  $(\{z_{(j, l)}\coloneqq x_j\otimes y_l\}_{(j, l)\in \mathbb{J}\bigtimes  \mathbb{L}},\{\rho _{(j, l)}\coloneqq \tau_j\otimes \omega_l\}_{(j, l)\in \mathbb{J}\bigtimes  \mathbb{L}}) $ be the  tensor product of frames  $( \{x_j\}_{j \in \mathbb{J}}, \{\tau_j\}_{j \in \mathbb{J}}) $  for  $ \mathscr{E},$ and $( \{y_l\}_{l \in \mathbb{L}}, \{\omega_l\}_{l \in \mathbb{L}} )$ for $\mathscr{E}_1.$  Then  $\theta_z=\theta_x\otimes\theta_y, \theta_\rho=\theta_\tau\otimes\theta_\omega, S_{z, \rho}=S_{x, \tau}\otimes S_{y, \omega}, P_{z, \rho}=P_{x, \tau}\otimes P_{y, \omega}.$ If  $( \{x_j\}_{j \in \mathbb{J}}, \{\tau_j\}_{j \in \mathbb{J}}) $ and $ (\{y_l\}_{l \in \mathbb{L}},  \{\omega_l\}_{l \in \mathbb{L}} )$ are Parseval, then $(\{z_{(j, l)}\}_{(j, l)\in \mathbb{J}\bigtimes  \mathbb{L}} ,\{\rho_{(j,l)}\}_{(j,l)\in \mathbb{J}\bigtimes \mathbb{L}})$ is Parseval.
 \end{proposition}

 \textbf{Perturbations}
 \begin{theorem}\label{PERTURBATIONRESULTMODULE1}
 Let $ (\{x_j\}_{j \in \mathbb{J}},\{\tau_j\}_{j \in \mathbb{J}} )$ be a frame for  $\mathscr{E}$ over  $\mathscr{A}.$ Suppose $ \{y_j\}_{j \in \mathbb{J}}$  in $ \mathscr{E}$ is  such that $ \langle x,y_j \rangle \tau_j=\langle x,\tau_j \rangle y_j, \langle x, y_j\rangle  \langle \tau_j,x \rangle \geq 0, \forall x \in \mathscr{E}, \forall j \in \mathbb{J}$ and there exist $ \alpha, \beta, \gamma \geq0$ with  $ \max\{\alpha+\gamma\|\theta_\tau S_{x,\tau}^{-1}\|, \beta\}<1$ and for every finite subset $ \mathbb{S}$ of $ \mathbb{J}$
 \begin{align*}
 \left\|\sum\limits_{j\in \mathbb{S}}c_j(x_j-y_j) \right\|\leq \alpha\left\|\sum\limits_{j\in \mathbb{S}}c_jx_j\right \|+\gamma \left\|\sum\limits_{j\in \mathbb{S}}c_jc_j^*\right\|^\frac{1}{2}+\beta\left\|\sum\limits_{j\in \mathbb{S}}c_jy_j\right \|,~\forall c_j \in \mathscr{A}, \forall j \in \mathbb{S}.
 \end{align*}	 
Then  $ (\{y_j\}_{j \in \mathbb{J}},\{\tau_j\}_{j \in \mathbb{J}} )$ is  a frame  with bounds $ \frac{1-(\alpha+\gamma\|\theta_\tau S_{x,\tau}^{-1}\|)}{(1+\beta)\|S_{x,\tau}^{-1}\|}$ and $\frac{\|\theta_\tau\|((1+\alpha)\|\theta_x\|+\gamma)}{1-\beta} $.
 \end{theorem}
 \begin{corollary}
 Let $ (\{x_j\}_{j \in \mathbb{J}},\{\tau_j\}_{j \in \mathbb{J}} )$ be a frame for  $\mathscr{E}$ over  $\mathscr{A}.$ Suppose $ \{y_j\}_{j \in \mathbb{J}}$  in $ \mathscr{E}$ is  such that $ \langle x,y_j \rangle \tau_j=\langle x,\tau_j \rangle y_j, \langle x, y_j\rangle  \langle \tau_j,x \rangle \geq 0, \forall x \in \mathscr{E}, \forall j \in \mathbb{J}$ and
 $$ r\coloneqq \left\|\sum_{j\in\mathbb{J}}(x_j-y_j)( x^*_j-y^*_j)\right\|<\frac{1}{\|\theta_\tau S_{x,\tau}^{-1}\|^2} .$$	
 Then  $ (\{y_j\}_{j \in \mathbb{J}},\{\tau_j\}_{j \in \mathbb{J}} )$ is a frame  with bounds $ \frac{1-\sqrt{r}\|\theta_\tau S_{x,\tau}^{-1}\|}{\|S_{x,\tau}^{-1}\|}$ and $\|\theta_\tau\|(\|\theta_x\|+\sqrt{r}) $.
 \end{corollary}
 \begin{proof}
 Take $ \alpha =0, \beta=0, \gamma=\sqrt{r}$. Then $ \max\{\alpha+\gamma\|\theta_\tau S_{x,\tau}^{-1}\|, \beta\}<1$ and for every finite subset $ \mathbb{S}$ of $ \mathbb{J}$,
 $$\left\|\sum\limits_{j\in \mathbb{S}}c_j(x_j-y_j) \right\|\leq  \left\|\sum\limits_{j\in \mathbb{S}}c_jc_j^*\right\|^\frac{1}{2} \left\|\sum\limits_{j\in \mathbb{S}}(x_j-y_j)(x_j^*-y_j^*)\right\|^\frac{1}{2}\leq \gamma  \left\|\sum\limits_{j\in \mathbb{S}}c_jc_j^*\right\|^\frac{1}{2}, ~\forall c_j \in \mathscr{A}, \forall j \in \mathbb{S}.$$	
 Now we can apply Theorem \ref{PERTURBATIONRESULTMODULE1}.
 \end{proof}
 \begin{theorem}
 Let $ (\{x_j\}_{j \in \mathbb{J}},\{\tau_j\}_{j \in \mathbb{J}} )$ be a frame for  $\mathscr{E}$  over  $\mathscr{A}$ with bounds $ a$ and $b.$ Suppose $ \{y_j\}_{j \in \mathbb{J}}$  in $ \mathscr{E}$ is  such that  $ \sum_{j\in\mathbb{J}}\langle x,y_j \rangle \langle \tau_j, x \rangle$ exist for all $ x \in \mathscr{E}$ and  is positive  (in the C*-algebra $ \mathscr{A}$) for all $ x \in \mathscr{E}$ and there exist $ \alpha, \beta, \gamma \geq0$ with  $ \max\{\alpha+\frac{\gamma}{\sqrt{a}}, \beta\}<1$ and
 \begin{align*}
 \left\| \sum\limits_{j\in\mathbb{J}}\langle x,x_j-y_j\rangle \langle \tau_j, x\rangle \right\|^\frac{1}{2} \leq \alpha\left\|\sum\limits_{j\in\mathbb{J}}\langle x,x_j\rangle \langle \tau_j,x \rangle \right\|^\frac{1}{2} + \beta\left\|\sum\limits_{j\in\mathbb{J}}\langle x,y_j\rangle \langle \tau_j,x \rangle \right\|^\frac{1}{2} +\gamma \|x\|, ~\forall x \in \mathscr{E}.
 \end{align*}
 Then $ (\{y_j\}_{j \in \mathbb{J}},\{\tau_j\}_{j \in \mathbb{J}} )$ is  a frame with bounds $a\left(1-\frac{\alpha+\beta+\frac{\gamma}{\sqrt{a}}}{1+\beta}\right)^2 $ and $b\left(1+\frac{\alpha+\beta+\frac{\gamma}{\sqrt{b}}}{1-\beta}\right)^2.$
 \end{theorem}
 
 \section{Further extension in modules} \label{FURTHEREXTENSIONINMODULES}
 We follow the same strategy as we did for `further extension' in Hilbert spaces.
 \begin{definition}
 A collection $ \{A_j\}_{j \in \mathbb{J}} $  in $ \operatorname{Hom}^*_\mathscr{A}(\mathscr{E},\mathscr{E}_0)$ is said to be a weak \textit{homomorphism-valued frame} (we write weak (hvf)) in $ \operatorname{Hom}^*_\mathscr{A}(\mathscr{E},\mathscr{E}_0)$  with respect to  collection  $ \{\Psi_j\}_{j \in \mathbb{J}}  $ in $ \operatorname{Hom}^*_\mathscr{A}(\mathscr{E},\mathscr{E}_0)$ if  the series $ S_{A, \Psi} \coloneqq  \sum_{j\in \mathbb{J}} \Psi_j^*A_j$  converges in the strict  topology on $ \operatorname{End}^*_\mathscr{A}(\mathscr{E})$ to a  bounded positive invertible homomorphism.

Notions of frame bounds, optimal bounds, tight frame, Parseval frame, Bessel are very similar to the corresponding  in Definition \ref{HMDEFINITION1}.
 	
For fixed $ \mathbb{J}$, $\mathscr{E}, \mathscr{E}_0, $ and $ \{\Psi_j \}_{j \in \mathbb{J}}$, the set of all weak homomorphism-valued frames in $ \operatorname{Hom}^*_\mathscr{A}(\mathscr{E},\mathscr{E}_0)$  with respect to collection  $ \{\Psi_j \}_{j \in \mathbb{J}}$ is denoted by $ \mathscr{F}^\text{w}_\Psi.$
\end{definition}
Last definition holds if and only if 
  \begin{definition}
 A collection $ \{A_j\}_{j \in \mathbb{J}} $   in $ \operatorname{Hom}^*_\mathscr{A}(\mathscr{E},\mathscr{E}_0)$ is said to be a weak (hvf)  w.r.t. $ \{\Psi_j\}_{j \in \mathbb{J}}  $ in $ \operatorname{Hom}^*_\mathscr{A}(\mathscr{E},\mathscr{E}_0) $ if there exist $ a, b, r >0$ such that 
 \begin{enumerate}[\upshape(i)]
 \item $\|\sum_{j\in \mathbb{J}}\Psi_j^*A_jx\|\leq r\|x\|, \forall x \in \mathscr{E},$
 \item  $\sum_{j\in \mathbb{J}}\Psi_j^*A_jx=\sum_{j\in \mathbb{J}}A_j^*\Psi_jx, \forall x \in \mathscr{E},$
 \item $a\langle x, x \rangle \leq\sum_{j\in \mathbb{J}}\langle A_jx, \Psi_jx\rangle  \leq b\langle x, x \rangle,  \forall x \in \mathscr{E}.$
 \end{enumerate}
 \end{definition}
\begin{proposition}
If  $(\{A_j\}_{j\in \mathbb{J}}, \{\Psi_j\}_{j\in \mathbb{J}}) $ is  weak Bessel   in $ \operatorname{Hom}^*_\mathscr{A}(\mathscr{E},\mathscr{E}_0)$, then there exists a $ B \in \operatorname{Hom}^*_\mathscr{A}(\mathscr{E},\mathscr{E}_0)$ such that $(\{A_j\}_{j\in \mathbb{J}}\cup\{B\}, \{\Psi_j\}_{j\in \mathbb{J}}\cup\{B\}) $ is a tight weak (hvf). In particular,  if $(\{A_j\}_{j\in \mathbb{J}}, \{\Psi_j\}_{j\in \mathbb{J}}) $ is  a weak (hvf)   in $ \mathcal{B}(\mathcal{H}, \mathcal{H}_0)$, then there exists a $ B \in \operatorname{Hom}^*_\mathscr{A}(\mathscr{E},\mathscr{E}_0)$ such that $(\{A_j\}_{j\in \mathbb{J}}\cup\{B\}, \{\Psi_j\}_{j\in \mathbb{J}}\cup\{B\}) $ is a tight weak (hvf). 
\end{proposition} 
\begin{definition}
A weak (hvf)  $(\{B_j\}_{j\in \mathbb{J}}, \{\Phi_j\}_{j\in \mathbb{J}})$  in $\operatorname{Hom}^*_\mathscr{A}(\mathscr{E},\mathscr{E}_0)$ is said to be a dual of  weak (hvf) $ ( \{A_j\}_{j\in \mathbb{J}}, \{\Psi_j\}_{j\in \mathbb{J}})$ in $\operatorname{Hom}^*_\mathscr{A}(\mathscr{E},\mathscr{E}_0)$  if  $ \sum_{j \in \mathbb{J}}B_j^*\Psi_j= \sum_{j \in \mathbb{J}}\Phi^*_jA_j=I_{\mathscr{E}}$. The `weak (hvf)' $( \{\widetilde{A}_j\coloneqq A_jS_{A,\Psi}^{-1}\}_{j\in \mathbb{J}},\{\widetilde{\Psi}_j\coloneqq\Psi_jS_{A,\Psi}^{-1}\}_{j \in \mathbb{J}})$, which is a `dual' of $ (\{A_j\}_{j\in \mathbb{J}}, \{\Psi_j\}_{j\in \mathbb{J}})$ is called the canonical dual of $ (\{A_j\}_{j\in \mathbb{J}}, \{\Psi_j\}_{j\in \mathbb{J}})$. 
\end{definition}   
\begin{proposition}
Let $( \{A_j\}_{j\in \mathbb{J}}, \{\Psi_j\}_{j\in \mathbb{J}} )$ be a weak (hvf) in $\operatorname{Hom}^*_\mathscr{A}(\mathscr{E},\mathscr{E}_0) $.  If $ x \in \mathscr{E}$ has representation  $ x=\sum_{j\in\mathbb{J}}A_j^*y_j= \sum_{j\in\mathbb{J}}\Psi_j^*z_j, $ for some $ \{y_j\}_{j\in \mathbb{J}},\{z_j\}_{j\in \mathbb{J}}$ in $ \mathscr{E}_0$, then 
$$ \sum\limits_{j\in \mathbb{J}}\langle y_j,z_j\rangle =\sum\limits_{j\in \mathbb{J}}\langle \widetilde{\Psi}_jx,\widetilde{A}_jx\rangle+\sum\limits_{j\in \mathbb{J}}\langle y_j-\widetilde{\Psi}_jx,z_j-\widetilde{A}_jx\rangle. $$
\end{proposition}  
\begin{theorem}
Let $ (\{A_j\}_{j\in \mathbb{J}},\{\Psi_j\}_{j\in \mathbb{J}}) $ be a weak (hvf) with frame bounds $ a$ and $ b.$ Then
\begin{enumerate}[\upshape(i)]
\item The canonical dual weak (hvf) of the canonical dual weak (hvf)   of $ (\{A_j\}_{j\in \mathbb{J}} ,\{\Psi_j\}_{j\in \mathbb{J}} )$ is itself.
\item$ \frac{1}{b}, \frac{1}{a}$ are frame bounds for the canonical dual of $ (\{A_j\}_{j\in \mathbb{J}},\{\Psi_j\}_{j\in \mathbb{J}}).$
\item If $ a, b $ are optimal frame bounds for $( \{A_j\}_{j\in \mathbb{J}} , \{\Psi_j\}_{j\in \mathbb{J}}),$ then $ \frac{1}{b}, \frac{1}{a}$ are optimal  frame bounds for its canonical dual.
\end{enumerate} 
\end{theorem}  
\begin{definition}
A weak (hvf)  $(\{B_j\}_{j\in \mathbb{J}},  \{\Phi_j\}_{j\in \mathbb{J}})$  in $\operatorname{Hom}^*_\mathscr{A}(\mathscr{E},\mathscr{E}_0)$ is said to be orthogonal to a weak (hvf)  $( \{A_j\}_{j\in \mathbb{J}}, \{\Psi_j\}_{j\in \mathbb{J}})$ in $\operatorname{Hom}^*_\mathscr{A}(\mathscr{E},\mathscr{E}_0)$ if $ \sum_{j \in \mathbb{J}}B_j^*\Psi_j= \sum_{j \in \mathbb{J}}\Phi^*_jA_j=0$.
\end{definition}  
\begin{proposition}
Two orthogonal weak homomorphism-valued frames  have common dual weak (hvf).	
\end{proposition} 
\begin{proposition}
Let $ (\{A_j\}_{j\in \mathbb{J}}, \{\Psi_j\}_{j\in \mathbb{J}}) $ and $ (\{B_j\}_{j\in \mathbb{J}}, \{\Phi_j\}_{j\in \mathbb{J}}) $ be  two Parseval weak homomorphism-valued frames in   $\operatorname{Hom}^*_\mathscr{A}(\mathscr{E},\mathscr{E}_0)$ which are  orthogonal. If $C,D,E,F \in \operatorname{End}^*_\mathscr{A}(\mathscr{E})$ are such that $ C^*E+D^*F=I_\mathscr{E}$, then  $ (\{A_jC+B_jD\}_{j\in \mathbb{J}}, \{\Psi_jE+\Phi_jF\}_{j\in \mathbb{J}}) $ is a  Parseval weak (hvf) in  $\operatorname{Hom}^*_\mathscr{A}(\mathscr{E},\mathscr{E}_0)$. In particular,  if  $ a,b,c,d\in\mathscr{A} $ satisfy $a^*c+b^*d =e$ (the identity of $\mathscr{A}$), then $ (\{cA_j+dB_j\}_{j\in \mathbb{J}}, \{e\Psi_j+f\Phi_j\}_{j\in \mathbb{J}}) $ is   a Parseval weak (hvf).
\end{proposition}  
\begin{definition}
Two weak homomorphism-valued frames $(\{A_j\}_{j\in \mathbb{J}},\{\Psi_j\}_{j\in \mathbb{J}} )$  and $ (\{B_j\}_{j\in \mathbb{J}}, \{\Phi_j\}_{j\in \mathbb{J}} )$   in $\operatorname{Hom}^*_\mathscr{A}(\mathscr{E},\mathscr{E}_0)$  are called 
disjoint if $(\{A_j\oplus B_j\}_{j \in \mathbb{J}},\{\Psi_j\oplus \Phi_j\}_{j \in \mathbb{J}})$ is a weak (hvf) in $\operatorname{Hom}^*_\mathscr{A}(\mathscr{E}\oplus\mathscr{E} ,\mathscr{E}_0) $.   
\end{definition}
\begin{proposition}
If $(\{A_j\}_{j\in \mathbb{J}},\{\Psi_j\}_{j\in \mathbb{J}} )$  and $ (\{B_j\}_{j\in \mathbb{J}}, \{\Phi_j\}_{j\in \mathbb{J}} )$  are  weak orthogonal homomorphism-valued frames  in $ \operatorname{Hom}^*_\mathscr{A}(\mathscr{E},\mathscr{E}_0)$, then  they  are disjoint. Further, if both $(\{A_j\}_{j\in \mathbb{J}},\{\Psi_j\}_{j\in \mathbb{J}} )$  and $ (\{B_j\}_{j\in \mathbb{J}}, \{\Phi_j\}_{j\in \mathbb{J}} )$ are  Parseval weak, then $(\{A_j\oplus B_j\}_{j \in \mathbb{J}},\{\Psi_j\oplus \Phi_j\}_{j \in \mathbb{J}})$ is Parseval weak.
\end{proposition}

\textbf{Characterizations}
\begin{theorem}
Let $ \{A_j\}_{j\in\mathbb{J}}, \{\Psi_j\}_{j\in\mathbb{J}}$ be in $ \operatorname{Hom}^*_\mathscr{A}(\mathscr{E},\mathscr{E}_0).$ Suppose $ \{e_{j,k}\}_{k\in\mathbb{L}_j}$ is an orthonormal basis for $ \mathscr{E}_0,$ for each $j \in \mathbb{J}.$ Let  $ u_{j,k}=A_j^*e_{j,k}, v_{j,k}=\Psi_j^*e_{j,k}, \forall k \in  \mathbb{L}_j, \forall j\in \mathbb{J}.$ Then  $ ( \{A_j\}_{j\in \mathbb{J}} ,\{\Psi_j\}_{j\in \mathbb{J}} )$ is a
\begin{enumerate}[\upshape(i)]
\item  weak (hvf) in $ \operatorname{Hom}^*_\mathscr{A}(\mathscr{E}, \mathscr{E}_0)$  with bounds $a $ and $ b$  if and only if there exist $ c,d >0$ such that the map 
$$ T: \mathscr{E} \ni x \mapsto\sum\limits_{j\in \mathbb{J}}\sum\limits_{k\in \mathbb{L}_j}\langle x, u_{j,k}\rangle v_{j,k} \in  \mathscr{E} $$
is a well-defined adjointable  positive invertible homomorphism  such that $ a\langle x, x \rangle\leq \langle Tx,x \rangle \leq b\langle x, x \rangle $, $ \forall x \in \mathscr{E} $, and 
$$ \left\| \sum\limits_{j\in \mathbb{J}}\sum\limits_{k\in \mathbb{L}_j}\langle x, u_{j,k}\rangle \langle u_{j,k},x \rangle \right\| \leq c\|x\|^2 ,~ \forall x \in \mathscr{E}; ~  \left\|\sum\limits_{j\in \mathbb{J}}\sum\limits_{k\in \mathbb{L}_j} \langle x, v_{j,k}\rangle \langle v_{j,k},x \rangle\right\|\leq d\|x\|^2 ,~ \forall x \in \mathscr{E}.$$
\item  weak Bessel in $ \operatorname{Hom}^*_\mathscr{A}(\mathscr{E}, \mathscr{E}_0)$  with bound $ b$  if and only if there exist $ c,d >0$ such that the map 
$$ T: \mathscr{E} \ni x \mapsto\sum\limits_{j\in \mathbb{J}}\sum\limits_{k\in \mathbb{L}_j}\langle x, u_{j,k}\rangle v_{j,k} \in  \mathscr{E} $$
is a well-defined adjointable  positive  homomorphism  such that $ \langle Tx,x \rangle \leq b\langle x, x \rangle ,  \forall x \in \mathscr{E} $, and 
$$ \left\| \sum\limits_{j\in \mathbb{J}}\sum\limits_{k\in \mathbb{L}_j}\langle x, u_{j,k}\rangle \langle u_{j,k},x \rangle \right\| \leq c\|x\|^2 ,~ \forall x \in \mathscr{E}; ~  \left\|\sum\limits_{j\in \mathbb{J}}\sum\limits_{k\in \mathbb{L}_j} \langle x, v_{j,k}\rangle \langle v_{j,k},x \rangle\right\|\leq d\|x\|^2 ,~ \forall x \in \mathscr{E}.$$
\item  weak (hvf)  in $ \operatorname{Hom}^*_\mathscr{A}(\mathscr{E},\mathscr{E}_0)$  with bounds $a $ and $ b$ if and only if there is $r>0$ such that 
$$\left \|\sum\limits_{j\in \mathbb{J}}\sum\limits_{k\in \mathbb{L}_j}\langle x, u_{j,k}\rangle v_{j,k}\right\|\leq r\|x\|,~\forall x \in \mathscr{E}  ; ~\sum\limits_{j\in \mathbb{J}}\sum\limits_{k\in \mathbb{L}_j}\langle x, u_{j,k}\rangle v_{j,k} =\sum\limits_{j\in \mathbb{J}}\sum\limits_{k\in \mathbb{L}_j}\langle x, v_{j,k}\rangle u_{j,k} ,~\forall x \in \mathscr{E} ;$$
$$ a\langle x, x \rangle \leq \sum\limits_{j\in \mathbb{J}}\sum\limits_{k\in \mathbb{L}_j}\langle x, u_{j,k}\rangle \langle  v_{j,k} , x\rangle=\sum\limits_{j\in \mathbb{J}}\sum\limits_{k\in \mathbb{L}_j}\langle x, v_{j,k}\rangle \langle  u_{j,k} , x\rangle \leq b \langle x, x \rangle ,~ \forall x \in \mathscr{E}.$$

\item  weak Bessel  in $ \operatorname{Hom}^*_\mathscr{A}(\mathscr{E},\mathscr{E}_0)$  with bound  $ b$ if and only if  there is $r>0$ such that 
$$\left \|\sum\limits_{j\in \mathbb{J}}\sum\limits_{k\in \mathbb{L}_j}\langle x, u_{j,k}\rangle v_{j,k}\right\|\leq r\|x\|,~\forall x \in \mathscr{E}  ; ~\sum\limits_{j\in \mathbb{J}}\sum\limits_{k\in \mathbb{L}_j}\langle x, u_{j,k}\rangle v_{j,k} =\sum\limits_{j\in \mathbb{J}}\sum\limits_{k\in \mathbb{L}_j}\langle x, v_{j,k}\rangle u_{j,k} ,~\forall x \in \mathscr{E} ;$$
$$ 0 \leq \sum\limits_{j\in \mathbb{J}}\sum\limits_{k\in \mathbb{L}_j}\langle x, u_{j,k}\rangle \langle  v_{j,k} , x \rangle=\sum\limits_{j\in \mathbb{J}}\sum\limits_{k\in \mathbb{L}_j}\langle x, v_{j,k}\rangle \langle  u_{j,k} , x\rangle  \leq b\langle x, x \rangle  ,~ \forall x \in \mathscr{E}. $$
\end{enumerate}	
\end{theorem} 

\begin{lemma}\label{PSEUDOINVERSE}
If  $T \in \operatorname{Hom}^*_\mathscr{A}(\mathscr{E}, \mathscr{E}_0) $ has closed range, then there exists a bounded homomorphism $ R : \mathscr{E}_0\rightarrow \mathscr{E}$  such that  $ TRy=y, \forall y \in \mathscr{E}_0.$
\end{lemma}
\begin{proof}
From Theorem \ref{MANUILOV2}, $\operatorname{Ker}(T)$ and $T(\mathscr{E}) $ are orthogonally complementable submodules in $ \mathscr{E}$ and $\mathscr{E}_0$, respectively. Rest is routine.
\end{proof}
 \begin{theorem}
 Let $ \{A_j\}_{j\in\mathbb{J}}, \{\Psi_j\}_{j\in\mathbb{J}}$ be in $ \operatorname{End}^*_\mathscr{A}(\mathscr{E})$ such that $ \Psi_j^*A_j\geq 0, \forall j \in \mathbb{J}.$ Then  $ ( \{A_j\}_{j\in \mathbb{J}} ,\{\Psi_j\}_{j\in \mathbb{J}} )$ is  a  weak (hvf)  in $  \operatorname{End}^*_\mathscr{A}(\mathscr{E})$ if and only if 
 $$  T: \mathscr{H}_{\mathscr{A}}\otimes \mathscr{E} \ni y \mapsto \sum\limits_{j\in \mathbb{J}} (\Psi_j^*A_j)^\frac{1}{2}L_j^*y\in \mathscr{E}$$
 is a well-defined bounded surjective adjointable homomorphism.
 \end{theorem}
 \begin{proof}
$ (\Rightarrow)$
 For every finite subset $ \mathbb{S}$ of $ \mathbb{J}$ and every $ y \in  \mathscr{H}_{\mathscr{A}}\otimes \mathscr{E},$ we have 
 
 \begin{align*}
 \left\|\sum\limits_{j\in \mathbb{S}} (\Psi_j^*A_j)^\frac{1}{2}L_j^*y\right\|&=\sup\limits_{x\in \mathscr{E}, \|x\|=1}\left\|\left\langle \sum\limits_{j\in \mathbb{S}} (\Psi_j^*A_j)^\frac{1}{2}L_j^*y,x \right\rangle\right\|=\sup\limits_{x\in \mathscr{E}, \|x\|=1 }\left\|\sum\limits_{j\in \mathbb{S}}\langle L_j^*y,(\Psi_j^*A_j)^\frac{1}{2}x \rangle \right\|\\
 & \leq \sup\limits_{ x\in \mathscr{E}, \|x\|=1}\left\|\sum\limits_{j\in \mathbb{S}}\langle L^*_jy,L^*_jy \rangle\right\|^\frac{1}{2}\left\|\sum\limits_{j\in \mathbb{S}}\langle (\Psi_j^*A_j)^\frac{1}{2}x, (\Psi_j^*A_j)^\frac{1}{2}x\rangle \right\|^\frac{1}{2}\\
 &=\sup\limits_{x\in \mathscr{E}, \|x\|=1}\left\|\sum\limits_{j\in \mathbb{S}}\langle L^*_jy,L^*_jy \rangle\right\|^\frac{1}{2}\left\|\sum\limits_{j\in \mathbb{S}}\langle\Psi_j^*A_jx,x \rangle \right\|^\frac{1}{2}\\
 &\leq \sup\limits_{x\in \mathscr{E}, \|x\|=1}\left\|\sum\limits_{j\in \mathbb{S}}\langle L^*_jy,L^*_jy \rangle\right\|^\frac{1}{2}\|S_{A,\Psi}x\|^\frac{1}{2}\leq\left\|\sum\limits_{j\in \mathbb{S}}\langle L^*_jy,L^*_jy \rangle\right\|^\frac{1}{2}\|S_{A,\Psi}\|^\frac{1}{2},
 \end{align*}
 $\sum_{j\in \mathbb{J}}\langle L^*_jy,L^*_jy \rangle$  exists (it equals to $ \langle y, y\rangle$). Therefore $ T$ is bounded linear with $\| T\| \leq \|S_{A,\Psi}\|^\frac{1}{2}.$ We can show the surjectivity of $ T$ as we showed the same in Theorem \ref{WEAK CHARACTERIZATION}. The adjoint of $ T $ is $ \sum_{j\in \mathbb{J}}L_j\Psi_j^*A_j.$ 
 
 $(\Leftarrow)$  We show  $\| \sum_{j\in \mathbb{J}}\langle \Psi_j^*A_jx, x \rangle\| $ converges, $ \forall x \in \mathscr{E}$ and using this, we show $\sum_{j\in \mathbb{J}}\Psi_j^*A_jx $ converges, $ \forall x \in \mathscr{E}$.  Let $x \in \mathscr{E} $, and  $ \mathbb{S}$ be a finite subset of $ \mathbb{J}.$ Then
 \begin{align*}
 \left\|\sum_{j\in \mathbb{S}}\langle \Psi_j^*A_jx, x \rangle \right\|&\leq \left\|\sum_{j\in \mathbb{S}} \Psi_j^*A_jx \right\|\|x\|= \left\|\sum\limits_{j\in \mathbb{S}}(\Psi_j^*A_j)^\frac{1}{2}L_j^*\left( \sum\limits_{k\in \mathbb{S}} L_k(\Psi_k^*A_k)^\frac{1}{2}x\right) \right\|\|x\|\\
 &
 = \left\|T\left( \sum\limits_{k\in \mathbb{S}} L_k(\Psi_k^*A_k)^\frac{1}{2}x+\sum\limits_{k\in \mathbb{J}\setminus\mathbb{S} }0\right) \right\|\|x\|
 \leq \|T\|\left \|\sum\limits_{k\in \mathbb{S}} L_k(\Psi_k^*A_k)^\frac{1}{2}x \right\|\|x\|\\
 &=\|T\|\left\|\left \langle \sum\limits_{k\in \mathbb{S}}\Psi_k^*A_kx, x\right \rangle\right\|^\frac{1}{2}\|x\|.
 \end{align*}
 Therefore $\|\sum_{j\in \mathbb{S}}\langle\Psi_j^*A_jx, x \rangle  \| \leq \|T\|^2\|x\|^2 $. Since $ \Psi_j^*A_j \geq 0, \forall j \in \mathbb{J},$ we get the convergence of $ \|\sum_{j\in \mathbb{J}}\langle \Psi_j^*A_jx, x \rangle\|. $ Next, 
 \begin{align*}
 \left\|\sum\limits_{j\in \mathbb{S}}\Psi_j^*A_jx \right\|& = \sup\limits_{z\in \mathscr{E}, \|z\|=1}\left\|\left\langle \sum\limits_{j\in \mathbb{S}} \Psi_j^*A_jx, z \right\rangle\right\|=\sup\limits_{ z\in \mathscr{E}, \|z\|=1}\left\| \sum\limits_{j\in \mathbb{S}} \langle (\Psi_j^*A_j)^\frac{1}{2}x,(\Psi_j^*A_j)^\frac{1}{2}z \rangle\right\|\\
 &\leq \sup\limits_{ z\in \mathscr{E}, \|z\|=1}\left\|\sum\limits_{j\in \mathbb{S}}\langle (\Psi_j^*A_j)^\frac{1}{2}x,(\Psi_j^*A_j)^\frac{1}{2}x \rangle\right\|^\frac{1}{2}\left\|\sum\limits_{j\in \mathbb{S}}\langle (\Psi_j^*A_j)^\frac{1}{2}z,(\Psi_j^*A_j)^\frac{1}{2}z \rangle\right\|^\frac{1}{2}\\
 &=\left\|\sum\limits_{j\in \mathbb{S}}\langle \Psi_j^*A_jx,x \rangle \right\|^\frac{1}{2}\sup\limits_{ z\in \mathscr{E}, \|z\|=1 }\left\|\sum\limits_{j\in \mathbb{S}}\langle \Psi_j^*A_jz,z \rangle \right\|^\frac{1}{2}\leq \left\|\sum\limits_{j\in \mathbb{S}}\langle \Psi_j^*A_jx, x \rangle \right\|^\frac{1}{2}\|T\|.
 \end{align*} 
 So $ \sum_{j\in \mathbb{J}}\Psi_j^*A_jx$ converges and $\|S_{A, \Psi}x\|=\|\sum_{j\in \mathbb{S}}\Psi_j^*A_jx \|\leq \|T\|^2\|x\|. $ Clearly $S_{A, \Psi} $ is positive. Now using Theorem \ref{PASCHKE} there exists a positive $ b$ such that $ \langle S_{A, \Psi}x,x \rangle=\langle S_{A, \Psi}^{1/2} x,S_{A, \Psi}^{1/2}x \rangle \leq b \langle x, x\rangle , \forall x \in \mathscr{E}.$

From Lemma \ref{PSEUDOINVERSE}, there exists a bounded homomorphism $R: \mathscr{E}\rightarrow  \mathscr{H}_{\mathscr{A}}\otimes \mathscr{E}$ such that $TRx=x,\forall x \in \mathscr{E}$. This gives
\begin{align*}
\|x\|^2&=\|\langle TRx, x\rangle\|=\left\|\left\langle\sum\limits_{j\in \mathbb{J}} (\Psi_j^*A_j)^\frac{1}{2}L_j^*Rx,x \right\rangle\right\|=\left\|\left\langle Rx ,\sum\limits_{j\in \mathbb{J}} L_j(\Psi_j^*A_j)^\frac{1}{2}x \right\rangle\right\|\\
&\leq\|Rx\|\left\|\sum\limits_{j\in \mathbb{J}} L_j(\Psi_j^*A_j)^\frac{1}{2}x\right\|= \|Rx\|\left\|\left\langle\sum\limits_{j\in \mathbb{J}} L_j(\Psi_j^*A_j)^\frac{1}{2}x,\sum\limits_{k\in \mathbb{J}} L_k(\Psi_k^*A_k)^\frac{1}{2}x \right\rangle\right\|^\frac{1}{2}\\
&=\|Rx\|\left\|\sum\limits_{j\in \mathbb{J}}\langle(\Psi_j^*A_j)^\frac{1}{2}x , (\Psi_j^*A_j)^\frac{1}{2}x\rangle \right\|^\frac{1}{2}=\|Rx\|\left\|\left \langle \sum\limits_{j\in \mathbb{J}}\Psi_j^*A_jx, x \right\rangle \right\|\leq \|R\|\|x\|\|\langle S_{A,\Psi}x,x\rangle \|
\end{align*}
implies $ \|R\|^{-1}\|x\|\leq \|S_{A,\Psi}^{1/2}x\|.$ Applying Theorem \ref{ARAMBASIC1} now completes the proof.
\end{proof} 

\textbf{Similarity and tensor product of weak homomorphism-valued frames }
\begin{definition}
A weak (hvf)  $(\{B_j\}_{j\in \mathbb{J}} ,\{\Phi_j\}_{j\in \mathbb{J}})$ in $ \operatorname{Hom}^*_\mathscr{A}(\mathscr{E}, \mathscr{E}_0)$  is said to be right-similar  to weak (hvf) $(\{A_j\}_{j\in \mathbb{J}},\{\Psi_j\}_{j\in \mathbb{J}})$ in $\operatorname{Hom}^*_\mathscr{A}(\mathscr{E}, \mathscr{E}_0)$ if there exist invertible  $ R_{A,B}, R_{\Psi, \Phi} \in  \operatorname{End}^*_\mathscr{A}(\mathscr{E})$  such that $B_j=A_jR_{A,B} , \Phi_j=\Psi_jR_{\Psi, \Phi}, \forall j \in \mathbb{J}$.
\end{definition}
\begin{proposition}
Let $ \{A_j\}_{j\in \mathbb{J}}\in \mathscr{F}^\text{w}_\Psi$  with frame bounds $a, b,$  let $R_{A,B}, R_{\Psi, \Phi} \in \operatorname{End}^*_\mathscr{A}(\mathscr{E})$ be positive, invertible, commute with each other, commute with $ S_{A, \Psi}$, and let $B_j=A_jR_{A,B} , \Phi_j=\Psi_jR_{\Psi, \Phi},  \forall j \in \mathbb{J}.$ Then 
$ \{B_j\}_{j\in \mathbb{J}}\in \mathscr{F}^\text{w}_\Phi,$  $ S_{B,\Phi}=R_{\Psi,\Phi}S_{A, \Psi}R_{A,B} $, and    $ \frac{a}{\|R_{A,B}^{-1}\|\|R_{\Psi,\Phi}^{-1}\|}\leq S_{B, \Phi} \leq b\|R_{A,B}R_{\Psi,\Phi}\|.$ Assuming that $( \{A_j\}_{j\in \mathbb{J}}, \{\Psi_j\}_{j\in \mathbb{J}})$ is a  Parseval weak (hvf), then $ (\{B_j\}_{j\in \mathbb{J}},\{\Phi_j\}_{j\in \mathbb{J}} ) $ is a Parseval  weak (hvf) if and only if   $ R_{\Psi, \Phi}R_{A,B}=I_\mathscr{E}.$  
\end{proposition}
\begin{proposition}
Let $ \{A_j\}_{j\in \mathbb{J}}\in \mathscr{F}^\text{w}_\Psi,$ $ \{B_j\}_{j\in \mathbb{J}}\in \mathscr{F}^\text{w}_\Phi$ and   $B_j=A_jR_{A,B} , \Phi_j=\Psi_jR_{\Psi, \Phi},  \forall j \in \mathbb{J}$, for some invertible $ R_{A,B} ,R_{\Psi, \Phi} \in \operatorname{End}^*_\mathscr{A}(\mathscr{E}).$ Then $  S_{B,\Phi}=R_{\Psi,\Phi}^*S_{A, \Psi}R_{A,B}.$ Assuming that $ (\{A_j\}_{j\in \mathbb{J}},\{\Psi_j\}_{j\in \mathbb{J}})$ is Parseval weak, then $(\{B_j\}_{j\in \mathbb{J}},  \{\Phi_j\}_{j\in \mathbb{J}})$ is Parseval weak if and only if   $ R_{\Psi, \Phi}^*R_{A,B}=I_\mathscr{E}.$
\end{proposition}
\begin{remark}
For every weak (hvf) $(\{A_j\}_{j \in \mathbb{J}},\{\Psi_j\}_{j \in \mathbb{J}})$, each  of `weak homomorphism-valued frames'  $( \{A_jS_{A, \Psi}^{-1}\}_{j \in \mathbb{J}}, \{\Psi_j\}_{j \in \mathbb{J}}),$   $( \{A_jS_{A, \Psi}^{-1/2}\}_{j \in \mathbb{J}}, \{\Psi_jS_{A,\Psi}^{-1/2}\}_{j \in \mathbb{J}}),$ and  $ (\{A_j \}_{j \in \mathbb{J}}, \{\Psi_jS_{A,\Psi}^{-1}\}_{j \in \mathbb{J}})$ is a  Parseval weak  (hvf) which is right-similar to  $ (\{A_j\}_{j \in \mathbb{J}} , \{\Psi_j\}_{j \in \mathbb{J}}  ).$  
\end{remark}

\textbf{Tensor product}: Let $ \{A_j\}_{j \in \mathbb{J}} $  be a weak (hvf) w.r.t. $ \{\Psi_j\}_{j \in \mathbb{J}} $  in  $ \operatorname{Hom}^*_\mathscr{A}(\mathscr{E}, \mathscr{E}_0),$ and $ \{B_l\}_{l \in \mathbb{L}} $  be a weak (hvf) w.r.t. $ \{\Phi_l\}_{l \in \mathbb{L}} $  in  $ \operatorname{Hom}^*_\mathscr{A}(\mathscr{E}_1, \mathscr{E}_2).$ The weak (hvf)  $(\{C_{(j, l)}\coloneqq A_j\otimes B_l\}_{(j, l)\in \mathbb{J}\bigtimes  \mathbb{L}} ,\{\Xi_{(j, l)}\coloneqq \Psi_j\otimes\Phi_l\}_{(j, l)\in \mathbb{J}\bigtimes  \mathbb{L}})$ in $ \operatorname{Hom}^*_\mathscr{A}(\mathscr{E}\otimes\mathscr{E}_1, \mathscr{E}_0\otimes\mathscr{E}_2)$ is called  as tensor product  of  $( \{A_j\}_{j \in \mathbb{J}}, \{\Psi_j\}_{j\in \mathbb{J}})$ and $( \{B_l\}_{l \in \mathbb{L}},  \{\Phi_l\}_{l\in \mathbb{L}}).$
\begin{proposition}
Let $(\{C_{(j, l)}\coloneqq A_j\otimes B_l\}_{(j, l)\in \mathbb{J}\bigtimes  \mathbb{L}}, \{\Xi _{(j, l)}\coloneqq \Psi_j\otimes \Phi_l\}_{(j, l)\in \mathbb{J}\bigtimes  \mathbb{L}})$ be the  tensor product of   weak homomorphism-valued frames  $( \{A_j\}_{j \in \mathbb{J}}, \{\Psi_j\}_{j \in \mathbb{J}}) $  in  $ \operatorname{Hom}^*_\mathscr{A}(\mathscr{E}, \mathscr{E}_1),$ and $( \{B_l\}_{l \in \mathbb{L}}, \{\Phi_l\}_{l \in \mathbb{L}} )$ in $ \operatorname{Hom}^*_\mathscr{A}(\mathscr{E}_1, \mathscr{E}_2).$  Then  $ S_{C, \Xi}=S_{A, \Psi}\otimes S_{B, \Phi}.$ If  $( \{A_j\}_{j \in \mathbb{J}}, \{\Psi_j\}_{j \in \mathbb{J}}) $ and $ (\{B_l\}_{l \in \mathbb{L}},  \{\Phi_l\}_{l \in \mathbb{L}} )$ are Parseval  weak (hvf), then $(\{C_{(j, l)}\}_{(j, l)\in \mathbb{J}\bigtimes  \mathbb{L}} ,\{\Xi_{(j,l)}\}_{(j,l)\in \mathbb{J}\bigtimes \mathbb{L}})$ is a Parseval  weak (hvf).
\end{proposition}

\textbf{Sequential version}
\begin{definition}
A collection $ \{x_j\}_{j\in \mathbb{J}}$ in   $ \mathscr{E}$ is called a weak frame w.r.t. collection $ \{\tau_j\}_{j\in \mathbb{J}}$ in  $ \mathscr{E}$ if  $S_{x, \tau}: \mathscr{E} \ni  x \mapsto \sum_{j\in\mathbb{J}}\langle x,  x_j\rangle\tau_j \in \mathscr{E} $ is a well-defined   adjointable positive invertible homomorphism. 
Notions of frame bounds, optimal bounds, tight frame, Parseval frame, Bessel are same as in Definition \ref{SEQUENTIAL VERSION HOMOMORPHISM VAUED DEFINITION}.

 For fixed $ \mathbb{J}, \mathscr{E},$  and $ \{\tau_j\}_{j\in \mathbb{J}}$   the set of all weak frames for $ \mathscr{E}$  w.r.t.  $ \{\tau_j\}_{j\in \mathbb{J}}$ is denoted by $ \mathscr{F}^w_\tau.$
\end{definition}
Previous definition is equivalent to
\begin{definition}
A collection $ \{x_j\}_{j\in \mathbb{J}}$ in  $ \mathscr{E}$ is called a weak frame w.r.t. collection $ \{\tau_j\}_{j\in \mathbb{J}}$ in  $ \mathscr{E}$ if there are $ a,b, r >0$ such that 
\begin{enumerate}[\upshape(i)]
\item $ \|\sum_{j\in \mathbb{J}}\langle x, x_j\rangle\tau_j\|\leq r\|x\|, \forall x \in  \mathscr{E} ,$
\item  $ \sum_{j\in \mathbb{J}}\langle x, x_j\rangle\tau_j=\sum_{j\in \mathbb{J}}\langle x, \tau_j\rangle x_j, \forall x \in  \mathscr{E},$
\item  $ a\langle x, x \rangle\leq\sum_{j\in \mathbb{J}}\langle x,x_j \rangle \langle \tau_j, x\rangle \leq b\langle x, x \rangle, \forall x \in  \mathscr{E}$.
\end{enumerate}
\end{definition} 
\begin{theorem}
Let $\{x_j\}_{j\in \mathbb{J}}, \{\tau_j\}_{j\in \mathbb{J}}$ be in $\mathscr{E}$ over $\mathscr{A}$. Define $A_j: \mathscr{E} \ni x \mapsto \langle x, x_j \rangle \in \mathscr{A} $, $\Psi_j: \mathscr{E} \ni x \mapsto \langle x, \tau_j \rangle \in \mathscr{A}, \forall j \in \mathbb{J} $. Then   $(\{x_j\}_{j\in \mathbb{J}}, \{\tau_j\}_{j\in \mathbb{J}})$ is a weak frame for  $\mathscr{E}$ if and only if  $(\{A_j\}_{j\in \mathbb{J}}, \{\Psi_j\}_{j\in \mathbb{J}})$ is a weak homomorphism-valued frame  in $\operatorname{Hom}_\mathscr{A}^*(\mathscr{E},\mathscr{A})$.
\end{theorem}
\begin{proposition}
If $(\{x_j\}_{j\in \mathbb{J}}, \{\tau_j\}_{j\in \mathbb{J}})$ is a weak frame for  $\mathscr{E}$, then  every $ y \in \mathscr{E}$ can be written as 
$$y =\sum\limits_{j\in \mathbb{J}}\langle y, S^{-1}_{x, \tau}\tau_j\rangle  x_j=\sum\limits_{j\in \mathbb{J}}\langle y,\tau_j \rangle S^{-1}_{x, \tau}  x_j =\sum\limits_{j\in \mathbb{J}}\langle y, S^{-1}_{x, \tau}x_j\rangle  \tau_j=\sum\limits_{j\in \mathbb{J}}\langle y, x_j\rangle S^{-1}_{x, \tau}  \tau_j.$$
\end{proposition}
\begin{proposition}
Let  $( \{x_j\}_{j \in \mathbb{J}},\{\tau_j\}_{j \in \mathbb{J}} )$ be a weak frame for  $\mathscr{E}$  with upper  frame bound $b$. If for some $ j \in \mathbb{J} $ we have  $  \langle x_j, x_l \rangle\langle \tau_l, x_j \rangle \geq0, \forall l  \in \mathbb{J},$ then $ \langle x_j, \tau_j \rangle\leq b$ for that $j. $
\end{proposition}
\begin{proposition}
Every weak Bessel sequence 	$(\{x_j\}_{j \in \mathbb{J}},\{\tau_j\}_{j \in \mathbb{J}} )$ for  $ \mathscr{E}$ can be extended to a tight weak frame for $ \mathscr{E}$. In particular,  every weak frame	$(\{x_j\}_{j \in \mathbb{J}},\{\tau_j\}_{j \in \mathbb{J}} )$ for $ \mathscr{E}$ can be extended to a  tight weak frame for $\mathscr{E}$.
\end{proposition} 
\begin{definition}
A weak frame   $(\{y_j\}_{j\in \mathbb{J}}, \{\omega_j\}_{j\in \mathbb{J}})$  for  $\mathscr{E}$ is said to be a dual of weak frame  $ ( \{x_j\}_{j\in \mathbb{J}}, \{\tau_j\}_{j\in \mathbb{J}})$ for  $\mathscr{E}$  if $ \sum_{j\in \mathbb{J}}\langle x, x_j\rangle \omega_j= \sum_{j\in \mathbb{J}}\langle x, \tau_j\rangle y_j=x, \forall x \in  \mathscr{E}$. The `weak frame' $(   \{\widetilde{x}_j\coloneqq S_{x,\tau}^{-1}x_j\}_{j\in \mathbb{J}},\{\widetilde{\tau}_j\coloneqq S_{x,\tau}^{-1}\tau_j\}_{j \in \mathbb{J}} )$, which is a `dual' of $ (\{x_j\}_{j\in \mathbb{J}}, \{\tau_j\}_{j\in \mathbb{J}})$ is called the canonical dual of $ (\{x_j\}_{j\in \mathbb{J}}, \{\tau_j\}_{j\in \mathbb{J}})$.
\end{definition}
\begin{proposition}
Let $( \{x_j\}_{j\in \mathbb{J}},\{\tau_j\}_{j\in \mathbb{J}} )$ be a weak frame for  $\mathscr{E}$ over $\mathscr{A}$. If $ x \in \mathscr{E}$ has representation  $ x=\sum_{j\in\mathbb{J}}c_jx_j= \sum_{j\in\mathbb{J}}d_j\tau_j, $ for some  sequences  $ \{c_j\}_{j\in \mathbb{J}},\{d_j\}_{j\in \mathbb{J}}$ in $\mathscr{A}$,  then 
$$ \sum\limits_{j\in \mathbb{J}}c_jd^*_j =\sum\limits_{j\in \mathbb{J}}\langle x, \widetilde{\tau}_j\rangle\langle \widetilde{x}_j , x \rangle+\sum\limits_{j\in \mathbb{J}}(\langle c_j-\langle x, \widetilde{\tau}_j\rangle)(d^*_j-\langle \widetilde{x}_j, x\rangle). $$
\end{proposition} 
\begin{theorem}
Let $( \{x_j\}_{j\in \mathbb{J}},\{\tau_j\}_{j\in \mathbb{J}} )$ be a weak frame for $ \mathscr{E}$ with frame bounds $ a$ and $ b.$ Then
\begin{enumerate}[\upshape(i)]
\item The canonical dual weak frame of canonical dual weak frame  of $ (\{x_j\}_{j\in \mathbb{J}} ,\{\tau_j\}_{j\in \mathbb{J}} )$ is itself.
\item$ \frac{1}{b}, \frac{1}{a}$ are frame bounds for canonical dual of $ (\{x_j\}_{j\in \mathbb{J}},\{\tau_j\}_{j\in \mathbb{J}}).$
\item If $ a, b $ are optimal frame bounds for $( \{x_j\}_{j\in \mathbb{J}} , \{\tau_j\}_{j\in \mathbb{J}}),$ then $ \frac{1}{b}, \frac{1}{a}$ are optimal  frame bounds for its canonical dual.
\end{enumerate} 
\end{theorem}  
\begin{definition}
A weak frame   $(\{y_j\}_{j\in \mathbb{J}},  \{\omega_j\}_{j\in \mathbb{J}})$  for  $\mathscr{E}$ is said to be orthogonal to a weak frame   $( \{x_j\}_{j\in \mathbb{J}}, \{\tau_j\}_{j\in \mathbb{J}})$ for $\mathscr{E}$ if $\sum_{j\in \mathbb{J}}\langle x, x_j\rangle \omega_j= \sum_{j\in \mathbb{J}}\langle x, \tau_j\rangle y_j=0, \forall x \in  \mathscr{E}.$
\end{definition}
\begin{proposition}
Two orthogonal weak frames  have common dual weak frame.	
\end{proposition}
\begin{proposition}
Let $ (\{x_j\}_{j\in \mathbb{J}}, \{\tau_j\}_{j\in \mathbb{J}}) $ and $ (\{y_j\}_{j\in \mathbb{J}}, \{\omega_j\}_{j\in \mathbb{J}}) $ be  two Parseval weak frames for  $\mathscr{E}$ over $(\mathscr{A},e)$ which are  orthogonal. If $A,B,C,D \in \operatorname{End}^*_\mathscr{A}(\mathscr{E})$ are such that $ AC^*+BD^*=I_\mathscr{E}$, then  $ (\{Ax_j+By_j\}_{j\in \mathbb{J}}, \{C\tau_j+D\omega_j\}_{j\in \mathbb{J}}) $ is a  Parseval weak frame for  $\mathscr{E}$. In particular,  if $ a,b,c,d \in\mathscr{A} $ satisfy $ac^*+bd^* =e$, then $ (\{ax_j+by_j\}_{j\in \mathbb{J}}, \{c\tau_j+d\omega_j\}_{j\in \mathbb{J}}) $ is a  Parseval weak frame for  $\mathscr{E}$.
\end{proposition} 
\begin{definition}
Two weak frames  $(\{x_j\}_{j\in \mathbb{J}},  \{\tau_j\}_{j\in \mathbb{J}}) $ and $(\{y_j\}_{j\in \mathbb{J}},\{\omega_j\}_{j\in \mathbb{J}})$  for $\mathscr{E}$ are called disjoint if $(\{x_j\oplus y_j\}_{j\in \mathbb{J}},\{\tau_j\oplus\omega_j\}_{j\in \mathbb{J}})$ is a weak frame for $\mathscr{E}\oplus\mathscr{E} $.
\end{definition} 
\begin{proposition}
If $(\{x_j\}_{j\in \mathbb{J}},\{\tau_j\}_{j\in \mathbb{J}} )$  and $ (\{y_j\}_{j\in \mathbb{J}}, \{\omega_j\}_{j\in \mathbb{J}} )$  are  disjoint  weak frames  for $\mathscr{E}$, then  they  are disjoint. Further, if both $(\{x_j\}_{j\in \mathbb{J}},\{\tau_j\}_{j\in \mathbb{J}} )$  and $ (\{y_j\}_{j\in \mathbb{J}}, \{\omega_j\}_{j\in \mathbb{J}} )$ are  Parseval weak, then $(\{x_j\oplus y_j\}_{j \in \mathbb{J}},\{\tau_j\oplus \omega_j\}_{j \in \mathbb{J}})$ is Parseval weak.
\end{proposition} 
\textbf{Similarity  and tensor product}
\begin{definition}
A weak frame  $ (\{y_j\}_{j\in \mathbb{J}},\{\omega_j\}_{j\in \mathbb{J}})$ for $\mathscr{E}$ is said to be  similar to  a weak frame $ (\{x_j\}_{j\in \mathbb{J}},\{\tau_j\}_{j\in \mathbb{J}})$ for $\mathscr{E}$ if there are invertible  $ T_{x,y}, T_{\tau,\omega} \in \operatorname{End}^*_\mathscr{A}(\mathscr{E})$ such that $ y_j=T_{x,y}x_j, \omega_j=T_{\tau,\omega}\tau_j,   \forall j \in \mathbb{J}.$
\end{definition} 
\begin{proposition}
Let $ \{x_j\}_{j\in \mathbb{J}}\in \mathscr{F}^w_\tau$  with frame bounds $a, b,$  let $T_{x,y} , T_{\tau,\omega}\in \operatorname{End}^*_\mathscr{A}(\mathscr{E})$ be positive, invertible, commute with each other, commute with $ S_{x, \tau}$, and let $y_j=T_{x,y}x_j , \omega_j=T_{\tau,\omega}\tau_j,  \forall j \in \mathbb{J}.$ Then $ \{y_j\}_{j\in \mathbb{J}}\in \mathscr{F}^w_\tau$, $S_{y,\omega}=T_{\tau,\omega}S_{x, \tau}T_{x,y}$, and $ \frac{a}{\|T_{x,y}^{-1}\|\|T_{\tau,\omega}^{-1}\|}\leq S_{y, \omega} \leq b\|T_{x,y}T_{\tau,\omega}\|$. Assuming that $ (\{x_j\}_{j\in \mathbb{J}},\{\tau_j\}_{j\in \mathbb{J}})$ is Parseval weak, then $(\{y_j\}_{j\in \mathbb{J}},  \{\omega_j\}_{j\in \mathbb{J}})$ is Parseval weak if and only if   $ T_{\tau, \omega}T_{x,y}=I_\mathscr{E}.$  
\end{proposition}     
\begin{proposition}
Let $ \{x_j\}_{j\in \mathbb{J}}\in \mathscr{F}^w_\tau,$ $ \{y_j\}_{j\in \mathbb{J}}\in \mathscr{F}^w_\omega$ and   $y_j=T_{x, y}x_j , \omega_j=T_{\tau,\omega}\tau_j,  \forall j \in \mathbb{J}$, for some invertible $T_{x,y}, T_{\tau,\omega}\in \operatorname{End}^*_\mathscr{A}(\mathscr{E}).$ Then 
$ \theta_y=\theta_x T^*_{x,y}, \theta_\omega=\theta_\tau T^*_{\tau,\omega}, S_{y,\omega}=T_{\tau,\omega}S_{x, \tau}T_{x,y}^*,  P_{y,\omega}=P_{x, \tau}.$ Assuming that $ (\{x_j\}_{j\in \mathbb{J}},\{\tau_j\}_{j\in \mathbb{J}})$ is Parseval weak, then $ (\{y_j\}_{j\in \mathbb{J}},\{\omega_j\}_{j\in \mathbb{J}})$ is Parseval weak  if and only if $T_{\tau,\omega}T_{x,y}^*=I_\mathscr{E}.$
\end{proposition}   
\begin{remark}
For every weak frame  $(\{x_j\}_{j \in \mathbb{J}}, \{\tau_j\}_{j \in \mathbb{J}}),$ each  of `weak frames'  $( \{S_{x, \tau}^{-1}x_j\}_{j \in \mathbb{J}}, \{\tau_j\}_{j \in \mathbb{J}})$,    $( \{S_{x, \tau}^{-1/2}x_j\}_{j \in \mathbb{J}}, \{S_{x,\tau}^{-1/2}\tau_j\}_{j \in \mathbb{J}}),$ and  $ (\{x_j \}_{j \in \mathbb{J}}, \{S_{x,\tau}^{-1}\tau_j\}_{j \in \mathbb{J}})$ is a  Parseval weak  frame which is similar to  $ (\{x_j\}_{j \in \mathbb{J}} , \{\tau_j\}_{j \in \mathbb{J}}).$  
\end{remark} 
\textbf{Tensor product}: Let $(\{x_j\}_{j \in \mathbb{J}}, \{\tau_j\}_{j \in \mathbb{J}})$ be a weak frame  for   $\mathscr{E}$,  and $(\{y_l\}_{l \in \mathbb{L}}, \{\omega_l\}_{l \in \mathbb{L}})$ be a weak frame  for   $ \mathscr{E}_1.$ The weak frame  $(\{z_{(j, l)}\coloneqq x_j\otimes y_l\}_{(j, l)\in \mathbb{J}\bigtimes  \mathbb{L}},\{\rho_{(j, l)}\coloneqq \tau_j\otimes\omega_l\}_{(j, l)\in \mathbb{J}\bigtimes  \mathbb{L}})$   for $\mathscr{E}\otimes\mathscr{E}_1$ is called  as tensor product  of frames $( \{x_j\}_{j \in \mathbb{J}}, \{\tau_j\}_{j\in \mathbb{J}})$ and $( \{y_l\}_{l \in \mathbb{L}},  \{\omega_l\}_{l\in \mathbb{L}}).$ 
\begin{proposition}
Let  $(\{z_{(j, l)}\coloneqq x_j\otimes y_l\}_{(j, l)\in \mathbb{J}\bigtimes  \mathbb{L}},\{\rho _{(j, l)}\coloneqq \tau_j\otimes \omega_l\}_{(j, l)\in \mathbb{J}\bigtimes  \mathbb{L}}) $ be the  tensor product of weak frames  $( \{x_j\}_{j \in \mathbb{J}}, \{\tau_j\}_{j \in \mathbb{J}}) $  for  $ \mathscr{E},$ and $( \{y_l\}_{l \in \mathbb{L}}, \{\omega_l\}_{l \in \mathbb{L}} )$ for $ \mathscr{E}_1.$  Then  $ S_{z, \rho}=S_{x, \tau}\otimes S_{y, \omega}.$ If  $( \{x_j\}_{j \in \mathbb{J}}, \{\tau_j\}_{j \in \mathbb{J}}) $ and $ (\{y_l\}_{l \in \mathbb{L}},  \{\omega_l\}_{l \in \mathbb{L}} )$ are Parseval weak, then $(\{z_{(j, l)}\}_{(j, l)\in \mathbb{J}\bigtimes  \mathbb{L}} ,\{\rho_{(j,l)}\}_{(j,l)\in \mathbb{J}\bigtimes \mathbb{L}})$ is Parseval weak.
\end{proposition}

\section{p-operator-valued frames and p-bases} \label{OPERATOR-VALUEDFRAMESFRAMESFORBANACHSPACES}
Idea: we write the frame inequality 
\begin{align}\label{FRAMEINEQUALITYFUNDAMENTAL}
 a\|h\|^2 \leq \langle S_{A,\Psi}h, h\rangle \leq b\|h\|^2, ~\forall h \in \mathcal{H} 
\end{align}
free from inner product as
\begin{align}\label{FUNDAMENTALIDEA}
a\|h\|^2 \leq \| S^\frac{1}{2}_{A,\Psi}h\|^2 \leq b\|h\|^2, ~\forall h \in \mathcal{H} .
\end{align}
Now to form the definition we want the notion of ``real powers" of an operator, which is available in \cite{KOMATSU1}. 
\begin{definition}(cf. \cite{KOMATSU1})\label{KOMATSU}
Let  $\mathcal{X}$ be a Banach space,  $A \in  \mathcal{B}(\mathcal{X})$ be such that the resolvent of $ A $ contains $( -\infty, 0].$ For each $ \alpha \in \mathbb{C},$ we define 
$$ A^\alpha=\frac{1}{2\pi i}\int\limits_\gamma \zeta(\zeta I_\mathcal{X}-A)^{-1}d\zeta,$$
where the path $ \gamma$ encircles the spectrum of $ A$ counterclockwise avoiding the negative real axis and $ \zeta^\alpha$ takes the principal branch.
\end{definition}
\begin{definition}\label{HM1}
Let  $\mathcal{X}, \mathcal{X}_0 $  be Banach spaces, $ p \in [1, \infty )$. Define 
$ L_j : \mathcal{X}_0 \ni x \mapsto e_j\otimes x \in  \ell^p(\mathbb{J}) \otimes \mathcal{X}_0$,  where $\{e_j\}_{j \in \mathbb{J}} $ is  the standard Schauder basis for $ \ell^p(\mathbb{J})$,  for each $ j \in \mathbb{J}$ and $ \widehat{L}_j:\ell^p(\mathbb{J}) \otimes \mathcal{X}_0 \rightarrow \mathcal{X}_0$ by $\widehat{L}_j(\{a_j\}_{j\in \mathbb{J}}\otimes x)=a_jx $, at elementary tensors $ \{a_j\}_{j\in \mathbb{J}}\otimes x \in \ell^p(\mathbb{J}) \otimes \mathcal{X}_0$,  extend by linearity,  for each $ j \in \mathbb{J}$. A collection $ \{A_j\}_{j \in \mathbb{J}} $  in $ \mathcal{B}(\mathcal{X}, \mathcal{X}_0)$ is said to be a p-operator-valued frame in  $ \mathcal{B}(\mathcal{X}, \mathcal{X}_0)$  with respect to a collection  $ \{\widehat{\Psi}_j\}_{j \in \mathbb{J}}  $ in $ \mathcal{B}(\mathcal{X}_0, \mathcal{X}) $ if 
\begin{enumerate}[\upshape(i)]
\item the series $ \widehat{S}_{A, \Psi}\coloneqq\sum_{j\in \mathbb{J}} \widehat{\Psi}_jA_j$ (p-(ovf) operator)  converges in the pointwise limit  in $ \mathcal{B}(\mathcal{X})$ to a  bounded  operator whose resolvent contains $ (-\infty, 0]$.
\item both $ \theta_A \coloneqq \sum_{j\in \mathbb{J}} L_jA_j$ (analysis operator), $\widehat{\theta}_\Psi\coloneqq \sum_{j\in \mathbb{J}}\widehat{\Psi}_j \widehat{L}_j$ (synthesis operator) converge in the pointwise limit  in $ \mathcal{B}(\mathcal{X}, \ell^p(\mathbb{J})\otimes \mathcal{X}_0 )$ and $ \mathcal{B}(\ell^p(\mathbb{J}) \otimes \mathcal{X}_0 , \mathcal{X})$, respectively, to  bounded operators. 
\end{enumerate}
In this situation, we write $(\{A_j\}_{j \in \mathbb{J}},\{\widehat{\Psi}_j\}_{j \in \mathbb{J}})$ is a  p-(ovf) in $ \mathcal{B}(\mathcal{X}, \mathcal{X}_0) $.
 Positive  $ \alpha, \beta $ satisfying 
\begin{equation}\label{FRAMEINEQUALITYBANACH}
\alpha^{1/p} \|x\| \leq \|\widehat{S}_{A, \Psi}^{1/p}x\| \leq \beta^{1/p}\|x\|, \quad \forall x \in \mathcal{X} 
\end{equation}
are called as lower and upper p-operator-valued frame bounds, taken in order.

Let $ a= \sup\{\alpha :\alpha^{1/p} \|x\| \leq \|\widehat{S}^{1/p}_{A, \Psi}x\|, \forall x \in \mathcal{X}\}$, $ b= \inf\{\beta :  \|\widehat{S}^{1/p}_{A, \Psi}x\|\leq \beta^{1/p} \|x\|, \forall x \in \mathcal{X} \}$ which are precisely  $ a=\|\widehat{S}_{A,\Psi}^{-1/p}\|^{-p}$ and $ b = \|\widehat{S}^{1/p}_{A,\Psi}\|^p$.  We call $ a$ as the optimal lower p-(ovf) bound, $ b$ as the optimal upper p-(ovf) bound for the p-(ovf) $(\{A_j \}_{j \in \mathbb{J}}, \{\widehat{\Psi}_j\}_{j \in \mathbb{J}}).$   Whenever $\widehat{S}_{A,\Psi}=\alpha I_\mathcal{X}$, for some $ \alpha \in \mathbb{K}$,    we say p-(ovf) is tight and whenever 
\begin{equation}\label{PARDEFINITION}
\widehat{S}_{A,\Psi}=I_\mathcal{X}
 \end{equation}
 we call p-(ovf) as Parseval p-(ovf).

For fixed $ \mathbb{J}$, $\mathcal{X}, \mathcal{X}_0, p $ and $ \{\widehat{\Psi}_j \}_{j \in \mathbb{J}}$, the set of all $p$-operator-valued frames in $ \mathcal{B}(\mathcal{X}, \mathcal{X}_0)$
 with respect to collection  $ \{\widehat{\Psi}_j \}_{j \in \mathbb{J}}$ is denoted by $ \widehat{\mathscr{F}}_{\Psi, p}.$
\end{definition}
\begin{caution}
\begin{enumerate}[\upshape(i)]
\item Definition \ref{HM1}  when considered on Hilbert spaces, includes a lot more than the Definition \ref{1}. It is obvious that if  $(\{A_j \}_{j \in \mathbb{J}}, \{\widehat{\Psi}_j \}_{j \in \mathbb{J}})$ is an (ovf)  in $ \mathcal{B}(\mathcal{H},\mathcal{H}_0)$, then $(\{A_j \}_{j \in \mathbb{J}}, \{\Psi^*_j \}_{j \in \mathbb{J}})$ is a 2-(ovf) because the spectrum  of $S_{A,\Psi}$ is in $(0,\infty)$. Consider  $\{A_n\coloneqq\frac{1+i}{n}I_{\ell^2(\mathbb{N})}\}_{n=1}^\infty $ and  $\{\widehat{\Psi}_n\coloneqq\frac{i}{n}I_{\ell^2(\mathbb{N})}\}_{n=1}^\infty $. Then  $(\{A_n\}_{n=1}^\infty ,\{\widehat{\Psi}_n\}_{n=1}^\infty)$  is a 2-(ovf)   in 
$ \mathcal{B}(\ell^2(\mathbb{N}),\ell^2(\mathbb{N}))$  for the reason that  the spectrum of $\widehat{S}_{A,\Psi} $ is $\{(-1+i)\pi^2/6\}$. But $(\{A_n=\frac{1+i}{n}I_{\ell^2(\mathbb{N})}\}_{n=1}^\infty,  \{(\widehat{\Psi}_n)^*=\frac{-i}{n}I_{\ell^2(\mathbb{N})}\}_{n=1}^\infty )$ is not an (ovf) in $ \mathcal{B}(\ell^2(\mathbb{N}),\ell^2(\mathbb{N}))$  because  the spectrum of $S_{A,\Psi} $  is $\{(-1+i)\pi^2/6\}$  and hence it is not self-adjoint (also it is not positive).
\item  We have not defined the definition of Parsevalness as -
\begin{equation}\label{PARDIFFERENCE}
\text{p-(ovf) is Parseval if $\alpha=\beta=1$ in}~  \operatorname{Inequality} ~ (\ref{FRAMEINEQUALITYBANACH}). 
\end{equation}
 In Hilbert spaces,  Equation (\ref{PARDEFINITION}) and Sentence (\ref{PARDIFFERENCE}) are equivalent. But these are  not same in Banach spaces, one can consider $ \widehat{S}_{A,\Psi}=iI_{\ell^1(\mathbb{J})}$. Of course,  Equation (\ref{PARDEFINITION}) implies  Sentence (\ref{PARDIFFERENCE}). Similar comments hold for the definition of p-tight (ovf).
\end{enumerate}
\end{caution}
 We note the following. 
\begin{enumerate}[(i)]
\item $ \theta_Ax =\sum_{j\in \mathbb{J}}e_j\otimes A_jx , \forall  x \in \mathcal{X}.$
\item The operators $ L_j$'s defined in Definition \ref{HM1} are  isometries from $\mathcal{X}_0 $ to $ \ell^p(\mathbb{J})\otimes \mathcal{X}_0$, and  for
 $  j,k \in \mathbb{J }$ we have  $$\widehat{L}_jL_k =
\left\{
\begin{array}{ll}
I_{\mathcal{X}_0 } & \mbox{if } j=k \\
0 & \mbox{if } j\neq k
\end{array}
\right. $$
and 
$$ \sum\limits_{j\in \mathbb{J}} L_j\widehat{L}_j=I_{\ell^p(\mathbb{J})}\otimes I_{\mathcal{X}_0}$$
where  the convergence is in the pointwise limit. 
\item If  $ \{A_j\}_{j \in \mathbb{J}}, $ $ \{B_j\}_{j \in \mathbb{J}}  \in \widehat{\mathscr{F}}_{\Psi,p} $, then $\{A_j+B_j\}_{j \in \mathbb{J}} \in \widehat{\mathscr{F}}_{\Psi, p} $, and $\{\alpha A_j\}_{j \in \mathbb{J}} \in \widehat{\mathscr{F}}_{\Psi,p}, \forall \alpha >0.$
\end{enumerate}
\begin{definition}\label{BANACHOPERATORORTHOGONALSET}
A collection  $ \{A_j\}_{j \in \mathbb{J}} $ in  $ \mathcal{B}(\mathcal{X},  \mathcal{X}_0 )$ is said to be a p-orthogonal set if the following    conditions hold.
\begin{enumerate}[\upshape(i)]
\item There exists  unique   $ \widehat{A}_k$ in $\mathcal{B}(\mathcal{X}_0,  \mathcal{X}) $ for each  $ A_k, k \in \mathbb{J}$ such that $A_j\widehat{A}_k=0,  \forall j \neq k, \forall j \in \mathbb{J}$.
\item If $\mathbb{L}\subseteq\mathbb{J} $  is such that $\sum_{j\in \mathbb{L}}L_jA_j \in \mathcal{B}(\mathcal{X}, \ell^p(\mathbb{J})\otimes \mathcal{X}_0) $, then  $\|\sum_{j\in \mathbb{L}}L_jA_jx\|^p=\sum_{j\in \mathbb{L}}\|A_jx\|^p, \forall x \in \mathcal{X}. $
\item If $\mathbb{L}\subseteq\mathbb{J} $  is such that $\sum_{j\in \mathbb{L}}\widehat{A}_j \widehat{L}_j\in \mathcal{B}(\ell^p(\mathbb{J}) \otimes \mathcal{X}_0 , \mathcal{X})$, then $ \|\sum_{j\in \mathbb{L}}\widehat{A}_j\widehat{L}_jy\|^p=\sum_{j\in \mathbb{L}}\|\widehat{A}_j\widehat{L}_jy\|^p, \forall y \in \mathcal{X}_0$.
\end{enumerate}
In this case, we say that  $\widehat{A}_j$ is the right p-orthogonal inverse of $A_j$ for each $ j \in \mathbb{J}$.
\end{definition}
\begin{theorem}
\begin{enumerate}[\upshape(i)]
\item  Let  $ \{A_n\}_{n=1}^m $ be  p-orthogonal  in  $ \mathcal{B}(\mathcal{X}, \mathcal{X}_0)$,  $\widehat{A}_n$ be the  right p-orthogonal inverse of $A_n$ for each $ n=1,...,m.$ Then $ \|\sum_{n=1}^m\widehat{A}_ny_n\|^p=\sum_{n=1}^m\|\widehat{A}_ny_n\|^p, \forall y_1,...,y_n \in \mathcal{X}_0$. In particular, $ \|\widehat{A}_1+\cdots+\widehat{A}_m\|^p\leq \|\widehat{A}_1\|^p+\cdots+\|\widehat{A}_m\|^p.$
\item If $ \{A_j\}_{j \in \mathbb{J}} $ is p-orthogonal in   $ \mathcal{B}(\mathcal{X}, \ell^p(\mathbb{J})\otimes \mathcal{X}_0 )$ such that $A_j\widehat{A}_j $ is invertible for each $ j \in \mathbb{J}$, where $\widehat{A}_j$ is the  right p-orthogonal inverse of $A_j$, then it is linearly independent over $ \mathbb{K}$ as well as over $ \mathcal{B}(\mathcal{X}_0).$
\end{enumerate}
\end{theorem}
\begin{proof}
\begin{enumerate}[\upshape(i)]
\item  $\|\sum_{n=1}^{m}\widehat{A}_ny_n\|^p = \|\sum_{n=1}^{m}\widehat{A}_n\widehat{L}_n(\sum_{k=1}^{m}L_ky_k)\|^p=\sum_{n=1}^{m}\|\widehat{L}_n(\sum_{k=1}^{m}L_ky_k)\|^p  =\sum_{n=1}^{m}\|\widehat{A}_ny_n\|^p$, $ \forall y_1,...,y_n \in \mathcal{X}_0$ and  $\|\widehat{A}_1+\cdots+\widehat{A}_m\|=\sup_{y \in \mathcal{X}_0, \|y\|=1}\|(\widehat{A}_1+\cdots+\widehat{A}_m)y\|  \leq( \sum_{n=1}^{m}\|\widehat{A}_n\|^p)^{1/p}$.
\item Let $ \mathbb{S} \subseteq \mathbb{J}$ be finite  and $ c_j \in \mathbb{K}$ (resp. $T_j \in \mathcal{B}(\mathcal{X}_0)$), $j \in \mathbb{S}$  be such that $ \sum_{j\in\mathbb{S}}c_jA_j=0$ (resp. $\sum_{j\in\mathbb{S}}T_jA_j=0$). Fixing $ k \in  \mathbb{S} $, we get $c_kA_k\widehat{A}_k=\sum_{j\in\mathbb{S}}c_jA_j\widehat{A}_k=0$ (resp. $T_kA_k\widehat{A}_k=\sum_{j\in\mathbb{S}}T_jA_j\widehat{A}_k=0$) which implies $ c_k=0$ (resp. $ T_k=0$).
\end{enumerate}	
\end{proof}

\begin{definition}\label{BANACHOPERATORORTHONORMALSET}
A collection  $ \{A_j\}_{j \in \mathbb{J}} $ in  $ \mathcal{B}(\mathcal{X},  \mathcal{X}_0 )$ is said to be  a p-orthonormal set  if the following conditions hold.
\begin{enumerate}[\upshape(i)]
\item There exists  unique   $ \widehat{A}_k$ in $\mathcal{B}(\mathcal{X}_0,  \mathcal{X}) $ for each  $ A_k, k \in \mathbb{J}$ such that $A_j\widehat{A}_k=\delta_{j,k}I_{\mathcal{X}_0},  \forall j \in \mathbb{J}$.
\item If $\mathbb{L}\subseteq\mathbb{J} $  is such that $\sum_{j\in \mathbb{L}}L_jA_j \in \mathcal{B}(\mathcal{X}, \ell^p(\mathbb{J})\otimes \mathcal{X}_0) $, then  $\|\sum_{j\in \mathbb{L}}L_jA_jx\|^p=\sum_{j\in \mathbb{L}}\|A_jx\|^p\leq \|x\|^p, \forall x \in \mathcal{X}. $
\item If $\mathbb{L}\subseteq\mathbb{J} $  is such that $\sum_{j\in \mathbb{L}}\widehat{A}_j \widehat{L}_j\in \mathcal{B}(\ell^p(\mathbb{J}) \otimes \mathcal{X}_0 , \mathcal{X})$, then $ \|\sum_{j\in \mathbb{L}}\widehat{A}_j\widehat{L}_jy\|^p=\sum_{j\in \mathbb{L}}\|\widehat{L}_jy\|^p, \forall y \in \mathcal{X}_0$.
\item The maps $ \theta_A : \mathcal{X} \ni x \mapsto \sum_{j\in \mathbb{J}}L_jA_jx\in \ell^p(\mathbb{J})\otimes \mathcal{X}_0 $ and $ \widehat{\theta}_A :  \ell^p(\mathbb{J})\otimes \mathcal{X}_0 \ni y \mapsto \sum_{j\in \mathbb{J}}\widehat{A}_j\widehat{L}_jy \in \mathcal{X}$ are well-defined bounded linear operators.

\end{enumerate}
In this case, we write $\widehat{A}_j$ is the right p-orthonormal inverse of $A_j$ for each $ j \in \mathbb{J}$.
\end{definition}

As a consequence of (ii) in the previous definition, we see that whenever $ \{A_j\}_{j \in \mathbb{J}} $ is p-orthonormal in  $ \mathcal{B}(\mathcal{X},  \mathcal{X}_0)$, then $ \|A_j\|\leq 1,\forall  j \in \mathbb{J}.$
\begin{theorem}
Let $ \{A_j\}_{j \in \mathbb{J}} $ be  p-orthonormal  in $ \mathcal{B}(\mathcal{X}, \mathcal{X}_0)$,  $ \{U_j\}_{j \in \mathbb{J}} $ be  in $ \mathcal{B}(\mathcal{X}_0)$,  $\widehat{A}_j$ be the right p-orthonormal inverse of $A_j$ for each $ j \in \mathbb{J}$ and $ y \in \mathcal{X}_0$.  Then
$$\sum_{j\in\mathbb{J}}\widehat{A}_jU_jy ~\text{converges in} ~  \mathcal{X} ~ \text{if and only if}~ \sum_{j\in\mathbb{J}}\|U_jy\|^p ~\text{converges} . $$ 
\end{theorem}
\begin{proof}
We use condition (iii) in Definition \ref{BANACHOPERATORORTHONORMALSET}. Let $ \mathbb{S} \subseteq \mathbb{J}$ be finite. Then
\begin{align*}
\left\|\sum_{j\in\mathbb{S}}\widehat{A}_jU_jy\right\|^p=\left\|\sum_{j\in\mathbb{S}}\widehat{A}_j\widehat{L}_j\left(\sum_{k\in\mathbb{S}}L_kU_ky\right)\right\|^p=\sum_{j\in\mathbb{S}}\left\|\widehat{L}_j\left(\sum_{k\in\mathbb{S}}L_kU_ky\right)\right\|^p=\sum_{j\in\mathbb{S}}\|U_jy\|^p .
\end{align*}
Now, using the completeness of $  \mathcal{X}$ and $ \mathbb{K}$, we see that  $\sum_{j\in\mathbb{J}}\widehat{A}_jU_jy $ exists if and only if  $\sum_{j\in\mathbb{J}}\|U_jy\|^p$ exists. 
\end{proof}
\begin{corollary}
Let $ \{A_j\}_{j \in \mathbb{J}} $ be  p-orthonormal  in $ \mathcal{B}(\mathcal{X}, \mathcal{X}_0)$,  $ \{c_j\}_{j \in \mathbb{J}} $ be a sequence of scalars, $\widehat{A}_j$ be the right p-orthonormal inverse of $A_j$ for each $ j \in \mathbb{J}$  and $ y \in \mathcal{X}_0$. Then
$$\sum_{j\in\mathbb{J}}c_j\widehat{A}_jy ~\text{converges in} ~  \mathcal{X} ~ \text{if and only if}~\{c_j\|y\|\}_{j \in \mathbb{J}} \in \ell^p(\mathbb{J}).$$ 	
In particular, if $y \in \mathcal{X}_0 $ is nonzero, then  $\sum_{j\in\mathbb{J}}c_j\widehat{A}_jy $ converges in $ \mathcal{X}$ if and only if $\{c_j\}_{j \in \mathbb{J}} \in \ell^p(\mathbb{J}).$ 	
\end{corollary}
\begin{definition}\label{BANACHOPERATORORTHONORMALBASIS}
A collection  $ \{A_j\}_{j \in \mathbb{J}} $ in  $ \mathcal{B}(\mathcal{X},  \mathcal{X}_0)$ is said to be  a p-orthonormal basis  if the following conditions hold.
\begin{enumerate}[\upshape(i)]
\item  There exists  unique   $ \widehat{A}_k$ in $\mathcal{B}(\mathcal{X}_0,  \mathcal{X}) $ for each  $ A_k, k \in \mathbb{J}$ such that $A_j\widehat{A}_k=\delta_{j,k}I_{\mathcal{X}_0},  \forall j \in \mathbb{J}$ and  $ \sum_{j\in\mathbb{J}}\widehat{A}_jA_j=I_\mathcal{X}$.
\item If $\mathbb{L}\subseteq\mathbb{J} $  is such that $\sum_{j\in \mathbb{L}}L_jA_j \in \mathcal{B}(\mathcal{X}, \ell^p(\mathbb{J})\otimes \mathcal{X}_0) $, then  $\|\sum_{j\in \mathbb{L}}L_jA_jx\|^p=\sum_{j\in \mathbb{L}}\|A_jx\|^p\leq \|x\|^p, \forall x \in \mathcal{X}. $
\item If $\mathbb{L}\subseteq\mathbb{J} $  is such that $\sum_{j\in \mathbb{L}}\widehat{A}_j \widehat{L}_j\in \mathcal{B}(\ell^p(\mathbb{J}) \otimes \mathcal{X}_0 , \mathcal{X})$, then $ \|\sum_{j\in \mathbb{L}}\widehat{A}_j\widehat{L}_jy\|^p=\sum_{j\in \mathbb{L}}\|\widehat{L}_jy\|^p, \forall y \in \mathcal{X}_0$.
\item The maps $ \theta_A : \mathcal{X} \ni x \mapsto \sum_{j\in \mathbb{J}}L_jA_jx\in \ell^p(\mathbb{J})\otimes \mathcal{X}_0 $ and $ \widehat{\theta}_A :  \ell^p(\mathbb{J})\otimes \mathcal{X}_0 \ni y \mapsto \sum_{j\in \mathbb{J}}\widehat{A}_j\widehat{L}_jy \in \mathcal{X}$ are well-defined bounded linear isometries.
\end{enumerate}
\end{definition}
\begin{example}
Let $\mathcal{H}$ and $\mathcal{H}_0$ be Hilbert spaces. If  $ \{A_j\}_{j \in \mathbb{J}} $ is an orthonormal basis in  $ \mathcal{B}(\mathcal{H},  \mathcal{H}_0)$, then we argue that it is a 2-orthonormal basis in  $ \mathcal{B}(\mathcal{H},  \mathcal{H}_0)$. In fact, 
\begin{enumerate}[\upshape(i)]
\item For each $A_k, k\in \mathbb{J}$, $ A_k^*$ in $ \mathcal{B}(\mathcal{H}_0,  \mathcal{H})$ satisfies $A_jA_k^*=\delta_{j,k}I_{\mathcal{H}_0},  \forall j \in \mathbb{J}$ and  $ \sum_{j\in\mathbb{J}}A^*_jA_j=I_\mathcal{H}$. Let $k\in \mathbb{J}$ be fixed. If there exists $T$ in  $ \mathcal{B}(\mathcal{H}_0,  \mathcal{H})$ such that $A_jT=\delta_{j,k}I_{\mathcal{H}_0},  \forall j \in \mathbb{J}$, then $ T=\sum_{j\in \mathbb{J}}A_j^*A_jT=A_k^*A_kT+\sum_{j\in \mathbb{J}, j \neq k}A_j^*A_jT=A_k^*I_{\mathcal{H}_0}+\sum_{j\in \mathbb{J}, j \neq k}A_j^*0=A_k^*$. Hence the right 2-orthonormal inverse of $A_k$ is unique. Thus $\widehat{A}_k=A_k^*,\forall k \in \mathbb{J}$.
\item If $\mathbb{L}\subseteq\mathbb{J} $  is such that $\sum_{j\in \mathbb{L}}L_jA_j \in \mathcal{B}(\mathcal{H}, \ell^2(\mathbb{J})\otimes \mathcal{H}_0) $, then  $\|\sum_{j\in \mathbb{L}}L_jA_jh\|^2= \langle\sum_{j\in \mathbb{L}}L_jA_jh ,\sum_{k\in \mathbb{L}}L_kA_kh \rangle =\sum_{j\in \mathbb{L}}\langle A_jh,L_j^*(\sum_{k\in \mathbb{L}}L_kA_kh) \rangle =\sum_{j\in \mathbb{L}}\langle A_jh,A_jh\rangle  =\sum_{j\in \mathbb{L}}\|A_jh\|^2\leq \sum_{j\in \mathbb{J}}\|A_jh\|^2=\|h\|^2, \forall h \in \mathcal{H}. $
\item We first identify that $\widehat{L}_j=L_j^*,\forall j \in \mathbb{J}$. If $\mathbb{L}\subseteq\mathbb{J} $  is such that $\sum_{j\in \mathbb{L}}A^*_j L^*_j\in \mathcal{B}(\ell^2(\mathbb{J}) \otimes \mathcal{H}_0 , \mathcal{H})$, then $ \|\sum_{j\in \mathbb{L}}A^*_jL^*_jy\|^2=\sum_{j\in \mathbb{L}}\langle L^*_jy, A_j(\sum_{k\in \mathbb{L}}A^*_kL^*_ky)\rangle =\sum_{j\in \mathbb{L}}\langle L^*_jy,L^*_jy \rangle =\sum_{j\in \mathbb{L}}\|L^*_jy\|^2, \forall y \in \mathcal{H}_0$.
\item $\|\theta_Ah\|^2=\|\sum_{j\in \mathbb{J}}L_jA_jh\|^2= \sum_{j\in \mathbb{J}}\|A_jh\|^2=\|h\|^2,\forall h \in \mathcal{H}$. From \text{\upshape(i)} and \text{\upshape(iii)}, we have $ \widehat{\theta}_A=\theta_A^*$. Hence  $\| \widehat{\theta}_Ay\|^2=\|\theta_A^*y\|^2=\|\sum_{j\in \mathbb{J}}A^*_jL^*_jy\|^2=\sum_{j\in \mathbb{J}}\langle L^*_jy,L^*_jy \rangle=\langle \sum_{j\in \mathbb{J}}L_jL_j^*y, y\rangle=\langle (I_{\ell^2(\mathbb{J})}\otimes I_{\mathcal{H}_0})y, y\rangle =\|y\|^2 ,\forall y \in \mathcal{H}_0$.
\end{enumerate}
\end{example}
\begin{remark}\label{P-ONBREMARK}
 If $ \{A_j\}_{j \in \mathbb{J}} $ is a p-orthonormal basis in  $ \mathcal{B}(\mathcal{X},  \mathcal{X}_0 )$, then from \text{\upshape{(iii)}} of Definition \ref{BANACHOPERATORORTHONORMALBASIS} we see $\sum_{j\in \mathbb{J}}L_jA_j$ and  $\sum_{j\in \mathbb{J}}\widehat{A}_j\widehat{L}_j $ are  isometries and using this in \text{\upshape{(ii)}} we get $\|x\|^p=\|\sum_{j\in \mathbb{J}}L_jA_jx\|^p =\sum_{j\in \mathbb{J}}\|A_jx\|^p, \forall x \in \mathcal{X}$, and $ \|y\|^p=\|\sum_{j\in \mathbb{J}}\widehat{A}_j\widehat{L}_jy\|^p=\sum_{j\in \mathbb{J}}\|\widehat{L}_jy\|^p, \forall y \in \mathcal{X}_0$.
\end{remark}

\begin{definition}\label{RELATIVEPORTHONORMALANDRIESZBANACH}
Let  $ \{A_j\}_{j \in \mathbb{J}} $ be in $ \mathcal{B}(\mathcal{X}, \mathcal{X}_0)$, and $ \{\widehat{\Psi}_j\}_{j \in \mathbb{J}} $ be in $ \mathcal{B}(\mathcal{X}_0, \mathcal{X}).$
We say 
\begin{enumerate}[\upshape(i)]
\item $ \{A_j\}_{j \in \mathbb{J}}$ is a p-orthonormal set (resp. basis)  w.r.t. $\{\widehat{\Psi}_j\}_{j \in \mathbb{J}} $ if  $ \{A_j\}_{j \in \mathbb{J}}$  is a p-orthonormal set (resp. basis), say $ \{A_j\}_{j \in \mathbb{J}}$ is a p-orthonormal set  (resp. basis),   and there exists a sequence  $\{c_j\}_{j \in \mathbb{J}} $  of   reals such that  $ 0<\inf\{c_j\}_{j \in \mathbb{J}}\leq \sup\{c_j\}_{j \in \mathbb{J}}<\infty$   and $ \widehat{\Psi}_j=c_j\widehat{A}_j$, where $\widehat{A}_j$ is  the right p-orthonormal inverse of $A_j$ for each $ j \in \mathbb{J}$. We write $ (\{A_j\}_{j \in \mathbb{J}}, \{\widehat{\Psi}_j\}_{j \in \mathbb{J}})$ is a p-orthonormal set (resp. basis).
\item $ \{A_j\}_{j \in \mathbb{J}}$ is a Riesz p-basis  w.r.t. $\{\widehat{\Psi}_j\}_{j \in \mathbb{J}} $ if there exists a  p-orthonormal basis $ \{F_j\}_{j \in \mathbb{J}}$ in $ \mathcal{B}(\mathcal{X}, \mathcal{X}_0)$ and   invertible $ U, V \in \mathcal{B}(\mathcal{X})$   with the resolvent of $ VU$ contains  $ (-\infty, 0]$    such that $ A_j=F_jU, \widehat{\Psi}_j=V\widehat{F}_j,  \forall j \in \mathbb{J}$, where  $\widehat{F}_j$ is the right p-orthonormal inverse of $F_j$ for each $ j \in \mathbb{J}$. We write $ (\{A_j\}_{j \in \mathbb{J}}, \{\widehat{\Psi}_j\}_{j \in \mathbb{J}})$ is a Riesz p-basis.
\end{enumerate}
\end{definition}
\begin{theorem}
\begin{enumerate}[\upshape(i)]
\item If $ (\{A_j\}_{j \in \mathbb{J}},\{\widehat{\Psi}_j\}_{j \in \mathbb{J}}) $ is a p-orthonormal basis in  $ \mathcal{B}(\mathcal{X}, \mathcal{X}_0)$, then it is a Riesz p-basis.
\item If $ (\{A_j=F_jU\}_{j \in \mathbb{J}},\{\widehat{\Psi}_j=V\widehat{F}_j\}_{j \in \mathbb{J}}) $ is a Riesz  p-basis in  $ \mathcal{B}(\mathcal{X}, \mathcal{X}_0)$, then it  is a p-(ovf) with optimal frame bounds $ \|(VU)^{-1/p}\|^{-p}$ and  $\|(VU)^{1/p}\|^p $.
\end{enumerate}
\end{theorem} 
\begin{proof}
\begin{enumerate}[\upshape(i)]
\item Let  $\{c_j\}_{j \in \mathbb{J}} $ be a sequence of     reals such that  $ 0<\inf\{c_j\}_{j \in \mathbb{J}}\leq \sup\{c_j\}_{j \in \mathbb{J}}<\infty$   and $ \widehat{\Psi}_j=c_j\widehat{A}_j, \forall j \in \mathbb{J}.$ Define $ F_j\coloneqq A_j, \forall j \in \mathbb{J},  U\coloneqq I_\mathcal{X}$ and  $ V \coloneqq \sum_{j\in \mathbb{J}}c_j\widehat{A}_jA_j.$ Since $\sup\{c_j\}_{j \in \mathbb{J}}<\infty $ and $ \sum_{j\in \mathbb{J}}\widehat{A}_jA_j$  converges, $ V$ is a well-defined bounded operator. Then $ A_j=F_jU,V\widehat{F}_j =\sum_{k\in \mathbb{J}}c_k\widehat{A}_kA_k\widehat{F}_j=\sum_{k\in \mathbb{J}}c_k\widehat{A}_kA_k\widehat{A}_j=c_j\widehat{A}_j=\widehat{\Psi}_j, \forall j \in \mathbb{J}.$  Since  $c_j>0, \forall j \in \mathbb{J} $, we see   $ V$ is invertible, whose inverse is $\sum_{j\in \mathbb{J}}c_j^{-1}\widehat{A}_jA_j. $ An  important fact which  remained is to show  the containment of $(-\infty, 0]$ in the   resolvent of $ VU=VI_\mathcal{X}=V$. Let $ \lambda \in  (-\infty, 0]$. Then $\lambda-c_j\neq 0, \forall j \in \mathbb{J} $. Therefore $ \lambda I_\mathcal{X}-V=\lambda I_\mathcal{X}-\sum_{j\in \mathbb{J}}c_j\widehat{A}_jA_j=\lambda\sum_{j\in \mathbb{J}}\widehat{A}_jA_j-\sum_{j\in \mathbb{J}}c_j\widehat{A}_jA_j=\sum_{j\in \mathbb{J}}(\lambda-c_j)\widehat{A}_jA_j$ is bounded invertible with inverse $ \sum_{j\in \mathbb{J}}(\lambda-c_j)^{-1}\widehat{A}_jA_j$. Thus the  resolvent of $ VU$ contains  $(-\infty, 0]$.
\item $ \theta_A=\sum_{j\in \mathbb{J}}L_jA_j=\sum_{j\in \mathbb{J}}L_jF_jU=\theta_FU,\widehat{\theta}_\Psi=\sum_{j\in \mathbb{J}}\widehat{\Psi}_j\widehat{L}_j=V\sum_{j\in \mathbb{J}}\widehat{F}_j\widehat{L}_j=V\widehat{\theta}_F,$ and $\widehat{S}_{A,\Psi}=\sum_{j\in \mathbb{J}}\widehat{\Psi}_jA_j=V\sum_{j\in \mathbb{J}}\widehat{F}_jF_jU =VI_\mathcal{X}U=VU$ whose resolvent contains $(-\infty, 0]$. Optimal bounds are clear.
\end{enumerate}	
\end{proof}
\begin{proposition}
Let $ (\{A_j\}_{j\in \mathbb{J}}, \{\widehat{\Psi}_j\}_{j\in \mathbb{J}}) $  be a  p-(ovf) in    $\mathcal{B}(\mathcal{X}, \mathcal{X}_0)$. Then the bounded 
\begin{enumerate}[\upshape(i)]
\item left-inverses of  $ \theta_A$ are precisely   $\widehat{S}_{A,\Psi}^{-1}\widehat{\theta}_\Psi+U(I_{\ell^p(\mathbb{J})\otimes\mathcal{X}_0}-\theta_A\widehat{S}_{A,\Psi}^{-1}\widehat{\theta}_\Psi)$, where $U\in \mathcal{B}( \ell^p(\mathbb{J})\otimes\mathcal{X}_0, \mathcal{X})$.
\item right-inverses of $ \widehat{\theta}_\Psi$ are precisely   $\theta_A\widehat{S}_{A,\Psi}^{-1}+(I_{\ell^p(\mathbb{J})\otimes\mathcal{X}_0}-\theta_A \widehat{S}_{A,\Psi}^{-1}\widehat{\theta}_\Psi)V$, where $V\in \mathcal{B}(\mathcal{X}, \ell^p(\mathbb{J})\otimes\mathcal{X}_0)$.
\end{enumerate}	
\end{proposition}
\begin{proof}
\begin{enumerate}[\upshape(i)]
\item $(\Leftarrow)$ Let $U: \ell^p(\mathbb{J})\otimes\mathcal{X}_0\rightarrow \mathcal{X}$ be a bounded operator. Then $(\widehat{S}_{A,\Psi}^{-1}\widehat{\theta}_\Psi+U(I_{\ell^p(\mathbb{J})\otimes\mathcal{X}_0}-\theta_A\widehat{S}_{A,\Psi}^{-1}\widehat{\theta}_\Psi))\theta_A=I_\mathcal{X}+U\theta_A-U\theta_A I_\mathcal{X}=I_\mathcal{X}$. Therefore  $\widehat{S}_{A,\Psi}^{-1}\widehat{\theta}_\Psi+U(I_{\ell^p(\mathbb{J})\otimes\mathcal{X}_0}-\theta_A\widehat{S}_{A,\Psi}^{-1}\widehat{\theta}_\Psi)$ is a bounded left-inverse of $\theta_A$.

$(\Rightarrow)$ Let $ L:\ell^p(\mathbb{J})\otimes\mathcal{X}_0\rightarrow \mathcal{X}$ be a bounded left-inverse of $ \theta_A$. Define $U\coloneqq L$. Then $\widehat{S}_{A,\Psi}^{-1}\widehat{\theta}_\Psi+U(I_{\ell^p(\mathbb{J})\otimes\mathcal{X}_0}-\theta_A\widehat{S}_{A,\Psi}^{-1}\widehat{\theta}_\Psi) =\widehat{S}_{A,\Psi}^{-1}\widehat{\theta}_\Psi+L(I_{\ell^p(\mathbb{J})\otimes\mathcal{X}_0}-\theta_A\widehat{S}_{A,\Psi}^{-1}\widehat{\theta}_\Psi)=\widehat{S}_{A,\Psi}^{-1}\widehat{\theta}_\Psi+L-I_{\mathcal{X}}\widehat{S}_{A,\Psi}^{-1}\widehat{\theta}_\Psi= L$. 
\item  $(\Leftarrow)$ Let $V: \mathcal{X} \mapsto \ell^p(\mathbb{J})\otimes\mathcal{X}_0$ be a bounded operator. Then $\widehat{\theta}_\Psi(\theta_A\widehat{S}_{A,\Psi}^{-1}+(I_{\ell^p(\mathbb{J})\otimes\mathcal{X}_0}-\theta_A \widehat{S}_{A,\Psi}^{-1}\widehat{\theta}_\Psi)V)=I_\mathcal{X}+\widehat{\theta}_\Psi V- I_\mathcal{X}\widehat{\theta}_\Psi V=I_\mathcal{X}$. Therefore  $\theta_A\widehat{S}_{A,\Psi}^{-1}+(I_{\ell^p(\mathbb{J})\otimes\mathcal{X}_0}-\theta_A \widehat{S}_{A,\Psi}^{-1}\widehat{\theta}_\Psi)V$ is a bounded right-inverse of $\widehat{\theta}_\Psi $.

$(\Rightarrow)$ Let $ R:\mathcal{X} \rightarrow \ell^p(\mathbb{J})\otimes\mathcal{X}_0$ be a bounded right-inverse of $\widehat{\theta}_\Psi$. Define $V\coloneqq R$. Then $ \theta_A\widehat{S}_{A,\Psi}^{-1}+(I_{\ell^p(\mathbb{J})\otimes\mathcal{X}_0}-\theta_A \widehat{S}_{A,\Psi}^{-1}\widehat{\theta}_\Psi)V=\theta_A\widehat{S}_{A,\Psi}^{-1}+(I_{\ell^p(\mathbb{J})\otimes\mathcal{X}_0}-\theta_A \widehat{S}_{A,\Psi}^{-1}\widehat{\theta}_\Psi)R=\theta_A\widehat{S}_{A,\Psi}^{-1}+R-\theta_A \widehat{S}_{A,\Psi}^{-1}I_\mathcal{X}= R$.

\end{enumerate}

\end{proof}

\begin{proposition}\label{BANACHSPACEOVFFUNDAMENTALLEMMA}
For every $ \{A_j\}_{j \in \mathbb{J}}  \in \widehat{\mathscr{F}}_{\Psi,p}$,
\begin{enumerate}[\upshape (i)]
\item $ \widehat{S}_{A, \Psi} = \widehat{\theta}_\Psi\theta_A .$  
\item $( \{A_j\}_{j \in \mathbb{J}}, \{\widehat{\Psi}_j \}_ {j \in \mathbb{J}})$ is Parseval if and only if $  \widehat{\theta}_\Psi\theta_A =I_\mathcal{X}.$ 
\item $( \{A_j\}_{j \in \mathbb{J}}, \{\widehat{\Psi}_j \}_ {j \in \mathbb{J}})$ is Parseval  if and only if $ \theta_A\widehat{\theta}_\Psi $ is idempotent.
\item $  A_j=\widehat{L}_j\theta_A, \forall j\in \mathbb{J}.$
\item $ \widehat{\Psi}_j=\widehat{\theta}_{\Psi}L_j, \forall j\in \mathbb{J}.$
\item $\theta_A\widehat{S}_{A,\Psi}^{-1}\widehat{\theta}_\Psi$ is idempotent.
\item $\theta_A $ is  injective whose range is closed.
\item  $\widehat{\theta}_\Psi $ is  surjective.
\end{enumerate}
\end{proposition}
\begin{proof}
A direct verification gives (i), and (ii) is a consequence of that. Justifications for   the remainings are similar to the proof of Proposition \ref{2.2}.
\end{proof}
The idempotent operator $\widehat{P}_{A, \Psi}\coloneqq \theta_A\widehat{S}_{A,\Psi}^{-1}\widehat{\theta}_\Psi$ is called as the \textit{frame idempotent} for $(\{A_j\}_{j \in \mathbb{J}}, \{\widehat{\Psi}_j \}_ {j \in \mathbb{J}}).$
\begin{definition}
A p-(ovf)  $(\{A_j\}_{j\in \mathbb{J}}, \{\widehat{\Psi}_j\}_{j\in \mathbb{J}})$  in $\mathcal{B}(\mathcal{X}, \mathcal{X}_0)$ is said to be a Riesz  p-(ovf)  if $ \widehat{P}_{A,\Psi}= I_{\ell^p(\mathbb{J})}\otimes I_{\mathcal{X}_0}$. A Parseval and  Riesz p-(ovf) (i.e., $\widehat{\theta}_\Psi\theta_A=I_\mathcal{X} $ and  $\theta_A\widehat{\theta}_\Psi=I_{\ell^p(\mathbb{J})}\otimes I_{\mathcal{X}_0} $) is called as an orthonormal p-(ovf).
\end{definition}
 \begin{proposition}
 A p-(ovf) $ (\{A_j\}_{j\in \mathbb{J}}, \{\widehat{\Psi}_j\}_{j\in \mathbb{J}}) $ in $ \mathcal{B}(\mathcal{X}, \mathcal{X}_0)$ is a Riesz p-(ovf)  if and only if   $\theta_A(\mathcal{X})=\ell^p(\mathbb{J})\otimes \mathcal{X}_0.$ 
 \end{proposition}
\begin{definition}
A p-(ovf) $ (\{B_j\}_{j\in \mathbb{J}}, \{\widehat{\Phi}_j\}_{j\in \mathbb{J}}) $  in $ \mathcal{B}(\mathcal{X}, \mathcal{X}_0)$ is said to be a dual of  p-(ovf) $ (\{A_j\}_{j\in \mathbb{J}}, \{\widehat{\Psi}_j\}_{j\in \mathbb{J}}) $ in $ \mathcal{B}(\mathcal{X}, \mathcal{X}_0)$  if  $ \widehat{\theta}_\Phi\theta_A= \widehat{\theta}_\Psi\theta_B=I_{\mathcal{X}}$. The `p-(ovf)' $(\{\widetilde{A}_j\coloneqq A_j\widehat{S}_{A,\Psi}^{-1}\}_{j\in \mathbb{J}}, \{\widetilde{\Psi}_j\coloneqq\widehat{S}_{A,\Psi}^{-1}\widehat{\Psi}_j\}_{j \in \mathbb{J}} )$, which is a `dual' of $ (\{A_j\}_{j\in \mathbb{J}}, \{\widehat{\Psi}_j\}_{j\in \mathbb{J}})$ is called the canonical dual of $ (\{A_j\}_{j\in \mathbb{J}}, \{\widehat{\Psi}_j\}_{j\in \mathbb{J}})$.
\end{definition}
 \begin{theorem}
 Let $(\{A_j\}_{j\in \mathbb{J}},\{\widehat{\Psi}_j\}_{j\in \mathbb{J}})$ be a p-(ovf) with frame bounds $ a$ and $ b.$ Then
 \begin{enumerate}[\upshape(i)]
 \item The canonical dual p-(ovf) of the canonical dual p-(ovf)  of $ (\{A_j\}_{j\in \mathbb{J}} ,\{\widehat{\Psi}_j\}_{j\in \mathbb{J}} )$ is itself.
 \item$ \frac{1}{b}, \frac{1}{a}$ are frame bounds for the canonical dual of $ (\{A_j\}_{j\in \mathbb{J}},\{\widehat{\Psi}_j\}_{j\in \mathbb{J}}).$
 \item If $ a, b $ are optimal frame bounds for $( \{A_j\}_{j\in \mathbb{J}} , \{\widehat{\Psi}_j\}_{j\in \mathbb{J}}),$ then $ \frac{1}{b}, \frac{1}{a}$ are optimal  frame bounds for its canonical dual.
 \end{enumerate} 
 \end{theorem} 
\begin{proof}
 Frame operator for the canonical dual $(\{\widetilde{A}_j\coloneqq A_j\widehat{S}_{A,\Psi}^{-1}\}_{j\in \mathbb{J}}, \{\widetilde{\Psi}_j\coloneqq\widehat{S}_{A,\Psi}^{-1}\widehat{\Psi}_j\}_{j \in \mathbb{J}} )$ is 
 $$ \sum\limits_{j\in \mathbb{J}}(\widehat{S}_{A,\Psi}^{-1}\widehat{\Psi}_j ) (A_j\widehat{S}_{A,\Psi}^{-1}) =\widehat{S}_{A,\Psi}^{-1}\left(\sum\limits_{j\in \mathbb{J}}\widehat{\Psi}_jA_j\right)\widehat{S}_{A,\Psi}^{-1} =\widehat{S}_{A,\Psi}^{-1}\widehat{S}_{A,\Psi}\widehat{S}_{A,\Psi}^{-1}= \widehat{S}_{A,\Psi}^{-1}.$$
 Therefore, its canonical dual is $(\{(A_j\widehat{S}_{A,\Psi}^{-1})\widehat{S}_{A,\Psi}^{-1}\}_{j \in \mathbb{J}} ,\{\widehat{S}_{A,\Psi}(\widehat{S}_{A,\Psi}^{-1}\widehat{\Psi}_j)\}_{j \in \mathbb{J}} )$. Rest follows from the consideration of $ \widehat{S}_{A,\Psi}^{-1/p}$ and frame bound definition.
\end{proof}
\begin{proposition}
Let  $ (\{A_j\}_{j\in \mathbb{J}}, \{\widehat{\Psi}_j\}_{j\in \mathbb{J}}) $ and $ (\{B_j\}_{j\in \mathbb{J}}, \{\widehat{\Phi}_j\}_{j\in \mathbb{J}}) $ be p-operator-valued frames in $ \mathcal{B}(\mathcal{X}, \mathcal{X}_0)$. Then the following are equivalent.
\begin{enumerate}[\upshape(i)]
\item $ (\{B_j\}_{j\in \mathbb{J}},\{\widehat{\Phi}_j\}_{j\in \mathbb{J}}) $ is dual of $( \{A_j\}_{j\in \mathbb{J}},   \{\widehat{\Psi}_j\}_{j\in \mathbb{J}}) $. 
\item $ \sum_{j\in \mathbb{J}}\widehat{\Phi}_jA_j = \sum_{j\in \mathbb{J}}\widehat{\Psi}_jB_j=I_\mathcal{X}$. 
\end{enumerate}
\end{proposition}
\begin{proof}
$\widehat{\theta}_\Phi\theta_A= \sum_{j\in \mathbb{J}}\widehat{\Phi}_jA_j, \widehat{\theta}_\Psi\theta_B= \sum_{j\in \mathbb{J}}\widehat{\Psi}_jB_j.$
\end{proof}
\begin{definition}
A p-(ovf)  $(\{B_j\}_{j\in \mathbb{J}},  \{\widehat{\Phi}_j\}_{j\in \mathbb{J}})$  in $ \mathcal{B}(\mathcal{X}, \mathcal{X}_0)$ is said to be orthogonal to a p-(ovf)  $( \{A_j\}_{j\in \mathbb{J}}, \{\widehat{\Psi}_j\}_{j\in \mathbb{J}})$ in $ \mathcal{B}(\mathcal{X}, \mathcal{X}_0)$ if $ \widehat{\theta}_\Phi\theta_A= \widehat{\theta}_\Psi\theta_B=0$.
\end{definition}
\begin{proposition}
Let  $ (\{A_j\}_{j\in \mathbb{J}}, \{\widehat{\Psi}_j\}_{j\in \mathbb{J}}) $ and $ (\{B_j\}_{j\in \mathbb{J}}, \{\widehat{\Phi}_j\}_{j\in \mathbb{J}}) $ be p-operator-valued frames in  $ \mathcal{B}(\mathcal{X}, \mathcal{X}_0)$. Then the following are equivalent.
\begin{enumerate}[\upshape(i)]
\item $ (\{B_j\}_{j\in \mathbb{J}},\{\widehat{\Phi}_j\}_{j\in \mathbb{J}}) $ is orthogonal to  $(\{A_j\}_{j\in \mathbb{J}}, \{\widehat{\Psi}_j\}_{j\in \mathbb{J}}) $. 
\item $ \sum_{j\in \mathbb{J}}\widehat{\Phi}_jA_j = \sum_{j\in \mathbb{J}}\widehat{\Psi}_jB_j=0$. 
\end{enumerate}
\end{proposition}
\begin{proposition}
Let $ (\{A_j\}_{j\in \mathbb{J}}, \{\widehat{\Psi}_j\}_{j\in \mathbb{J}}) $ and $ (\{B_j\}_{j\in \mathbb{J}}, \{\widehat{\Phi}_j\}_{j\in \mathbb{J}}) $ be  two Parseval p-operator-valued frames in   $\mathcal{B}(\mathcal{X}, \mathcal{X}_0)$ which are  orthogonal. If $C,D,E,F \in \mathcal{B}(\mathcal{X})$ are such that $ EC+FD=I_\mathcal{X}$, then  $ (\{A_jC+B_jD\}_{j\in \mathbb{J}}, \{E\widehat{\Psi}_j+F\widehat{\Phi}_j\}_{j\in \mathbb{J}}) $ is a  Parseval p-(ovf) in  $\mathcal{B}(\mathcal{X}, \mathcal{X}_0)$. In particular,  if scalars $ c,d,e,f$ satisfy $ec+fd =1$, then $ (\{cA_j+dB_j\}_{j\in \mathbb{J}}, \{e\widehat{\Phi}_j+f\widehat{\Phi}_j\}_{j\in \mathbb{J}}) $ is  a  Parseval p-(ovf).
\end{proposition}   
\begin{proof}
We see $ \theta_{AC+BD} =\sum_{j\in \mathbb{J}}L_j(A_jC+B_jD)=\theta_AC+\theta_BD, $ $\widehat{\theta}_{E\Psi +F\Phi }=\sum_{j\in \mathbb{J}}(E\widehat{\Psi}_j+F\widehat{\Phi}_j)\widehat{L}_j=E\widehat{\theta}_\Psi +F\widehat{\theta}_\Phi $ and hence $\widehat{S}_{AC+BD,E\Psi +F\Phi }= \widehat{\theta}_{E\Psi +F\Phi } \theta_{AC+BD}=(E\widehat{\theta}_\Psi +F\widehat{\theta}_\Phi )(\theta_AC+\theta_BD)=E\widehat{\theta}_\Psi\theta_AC+E\widehat{\theta}_\Psi\theta_BD+F\widehat{\theta}_\Phi\theta_AC+F\widehat{\theta}_\Phi\theta_BD=E\widehat{S}_{A,\Psi}C+E0D+F0C+F\widehat{S}_{B,\Phi }D=EI_\mathcal{X}C+FI_\mathcal{X}D=I_\mathcal{X}.$
\end{proof}

\textbf{Characterization}
\begin{theorem}\label{POVFCHARACTERIZATIONBANACH}
Let $ \{F_j\}_{j \in \mathbb{J}}$ be an arbitrary p-orthonormal basis in $ \mathcal{B}(\mathcal{X},\mathcal{X}_0),$ $\widehat{F}_j$ be the right p-orthonormal inverse of $F_j$ for each $ j \in \mathbb{J}$. Then 
\begin{enumerate}[\upshape(i)]
\item The p--orthonormal  bases   $ (\{A_j\}_{j \in \mathbb{J}},\{\widehat{\Psi}_j\}_{j \in \mathbb{J}})$ in  $ \mathcal{B}(\mathcal{X},\mathcal{X}_0)$ are precisely $( \{F_jU\}_{j \in \mathbb{J}},\{c_jU^{-1}\widehat{F}_j\}_{j \in \mathbb{J}}) $, 
where $ U \in \mathcal{B}(\mathcal{X}) $ is invertible isometry  and $ c_j \in \mathbb{R}, \forall j \in \mathbb{J}$ such that $ 0<\inf\{c_j\}_{j \in \mathbb{J}}\leq \sup\{c_j\}_{j \in \mathbb{J}}< \infty.$
\item The Riesz p-bases   $ (\{A_j\}_{j \in \mathbb{J}},\{\widehat{\Psi}_j\}_{j \in \mathbb{J}})$ in  $ \mathcal{B}(\mathcal{X},\mathcal{X}_0)$  are precisely $( \{F_jU\}_{j \in \mathbb{J}},\{V \widehat{F}_j\}_{j \in \mathbb{J}}) $, where $ U,V \in \mathcal{B}(\mathcal{X}) $ are  invertible  such that   the resolvent of $ VU$ contains $ (-\infty, 0]$.
\item The p-operator-valued frames $ (\{A_j\}_{j \in \mathbb{J}},\{\widehat{\Psi}_j\}_{j \in \mathbb{J}})$ in  $ \mathcal{B}(\mathcal{X},\mathcal{X}_0)$  are precisely $( \{F_jU\}_{j \in \mathbb{J}},\{V \widehat{F}_j\}_{j \in \mathbb{J}}) $, where $ U,V \in \mathcal{B}(\mathcal{X}) $ are such that  the resolvent of $ VU$ contains $ (-\infty, 0]$.
\item The Riesz p-operator-valued frames $ (\{A_j\}_{j \in \mathbb{J}},\{\widehat{\Psi}_j\}_{j \in \mathbb{J}})$ in  $ \mathcal{B}(\mathcal{X},\mathcal{X}_0)$  are precisely $( \{F_jU\}_{j \in \mathbb{J}},\{V \widehat{F}_j\}_{j \in \mathbb{J}}) $, where $ U,V \in \mathcal{B}(\mathcal{X}) $ are such that  the resolvent of $ VU$ contains $ (-\infty, 0]$ and $ U(VU)^{-1}V=I_\mathcal{X}$.
\item The orthonormal p-operator-valued frames $ (\{A_j\}_{j \in \mathbb{J}},\{\widehat{\Psi}_j\}_{j \in \mathbb{J}})$ in  $ \mathcal{B}(\mathcal{X},\mathcal{X}_0)$  are precisely $( \{F_jU\}_{j \in \mathbb{J}},\{V \widehat{F}_j\}_{j \in \mathbb{J}}) $, where $ U,V \in \mathcal{B}(\mathcal{X}) $ are such that the resolvent of $ VU$ contains $ (-\infty, 0]$ and $VU=I_\mathcal{X}=UV$.
\end{enumerate}
\end{theorem}  
\begin{proof}
\begin{enumerate}[\upshape(i)]
\item $(\Leftarrow)$ Let  $ U : \mathcal{X} \rightarrow  \mathcal{X} $ be invertible isometry,  $ c_j \in \mathbb{R}, \forall j \in \mathbb{J}$ with $ 0<\inf\{c_j\}_{j \in \mathbb{J}}\leq \sup\{c_j\}_{j \in \mathbb{J}}< \infty$. We need to show that $( \{F_jU\}_{j \in \mathbb{J}},\{c_jF_jU\}_{j \in \mathbb{J}}) $ is a p-orthonormal basis in  $ \mathcal{B}(\mathcal{X},\mathcal{X}_0).$ To show this, we try to get   $ \{F_jU\}_{j \in \mathbb{J}}$ is a p-orthonormal basis. Note that $U^{-1}\widehat{F}_j $ is the unique right bounded inverse of $F_jU$  for each $ j \in \mathbb{J}.$ In fact, $F_jUU^{-1}\widehat{F}_j=I_{\mathcal{X}_0} $, and if $ T \in \mathcal{B}(\mathcal{X}_0,\mathcal{X})$
is any right bounded inverse of $F_jU$, then $ F_jUT=I_{\mathcal{X}_0}\Rightarrow \widehat{F}_j=UT$ (from the uniqueness of right bounded inverse of $ F_j$) $\Rightarrow U^{-1}\widehat{F}_j=T.$ Then $F_jU\widehat{F_kU}=F_jUU^{-1}\widehat{F}_k=F_j\widehat{F}_k =\delta_{j,k}I_{\mathcal{X}_0},\forall j , k \in \mathbb{J}$.  If $\mathbb{L}\subseteq\mathbb{J} $  is such that $\sum_{j\in \mathbb{L}}L_jF_jU \in \mathcal{B}(\mathcal{X}, \ell^p(\mathbb{J})\otimes \mathcal{X}_0) $, then  $\|\sum_{j\in \mathbb{L}}L_jF_jUx\|^p=\sum_{j\in \mathbb{L}}\|F_jUx\|^p\leq \|Ux\|^p=\|x\|^p, \forall x \in \mathcal{X}. $ If $\mathbb{L}\subseteq\mathbb{J} $  is such that $\sum_{j\in \mathbb{L}}\widehat{F_jU} \widehat{L}_j\in \mathcal{B}(\ell^p(\mathbb{J}) \otimes \mathcal{X}_0 , \mathcal{X})$, then $ \|\sum_{j\in \mathbb{L}}\widehat{F_jU}\widehat{L}_jy\|^p=\|\sum_{j\in \mathbb{L}}U^{-1}\widehat{F}_j\widehat{L}_jy\|^p= \|U^{-1}(\sum_{j\in \mathbb{L}}\widehat{F}_j\widehat{L}_jy)\|^p=\|\sum_{j\in \mathbb{L}}\widehat{F}_j\widehat{L}_jy\|^p =\sum_{j\in \mathbb{L}}\|\widehat{L}_jy\|^p, \forall y \in \mathcal{X}_0$. In a similar fashion, we get    the maps $ \theta_A : \mathcal{X} \ni x \mapsto \sum_{j\in \mathbb{J}}L_jF_jUx\in \ell^p(\mathbb{J})\otimes \mathcal{X}_0 $ and $ \widehat{\theta}_A :  \ell^p(\mathbb{J})\otimes \mathcal{X}_0 \ni y \mapsto \sum_{j\in \mathbb{J}}\widehat{F_jU}\widehat{L}_jy \in \mathcal{X}$ are well-defined bounded linear isometries.

$(\Rightarrow)$ We may take  $ \{A_j\}_{j \in \mathbb{J}}$  as a p-orthonormal basis in  $ \mathcal{B}(\mathcal{X},\mathcal{X}_0)$. Then there exists $ c_j \in \mathbb{R}, \forall j \in \mathbb{J} $ with $ 0<\inf\{c_j\}_{j \in \mathbb{J}}\leq \sup\{c_j\}_{j \in \mathbb{J}}< \infty$ and $ \widehat{\Psi}_j=c_j\widehat{A}_j, \forall j \in \mathbb{J}.$ Define $ U\coloneqq \sum_{j\in \mathbb{J}}\widehat {F}_jA_j.$ This operator exists in pointwise limit, since for every finite subset $\mathbb{S}$ of  $\mathbb{J}$ and $ x \in \mathcal{X},$ 
\begin{align*}
\left\|\sum_{j\in \mathbb{S}}\widehat{F}_jA_jx\right\|^p=\left\|\sum_{j\in \mathbb{S}}\widehat{F}_j\widehat{L}_j\left(\sum_{k\in \mathbb{S}}L_kA_kx\right)\right\|^p=\sum_{j\in \mathbb{S}}\left\|\widehat{L}_j\left(\sum_{k\in \mathbb{S}}L_kA_kx\right)\right\|^p =\sum_{j\in \mathbb{S}}\|A_jx\|^p.
\end{align*}
We see $U(\sum_{j\in \mathbb{J}}\widehat{A}_jF_j)=(\sum_{k\in \mathbb{J}}\widehat {F}_kA_k)(\sum_{j\in \mathbb{J}}\widehat{A}_jF_j)=\sum_{k\in \mathbb{J}}\widehat {F}_k(\sum_{j\in \mathbb{J}}A_k\widehat{A}_jF_j) =\sum_{k\in \mathbb{J}}\widehat {F}_kF_k=I_\mathcal{X}$, $(\sum_{j\in \mathbb{J}}\widehat{A}_jF_j)U=(\sum_{j\in \mathbb{J}}\widehat{A}_jF_j)(\sum_{k\in \mathbb{J}}\widehat {F}_kA_k) =\sum_{j\in \mathbb{J}}\widehat{A}_j(\sum_{k\in \mathbb{J}}F_j\widehat {F}_kA_k)=\sum_{j\in \mathbb{J}}\widehat{A}_jA_j=I_\mathcal{X}$. Therefore $U^{-1}=\sum_{j\in \mathbb{J}}\widehat{A}_jF_j.$ For $ x \in \mathcal{X}$,
\begin{align*}
\|Ux\|^p=\left\|\sum_{j\in \mathbb{J}}\widehat {F}_jA_jx\right\|^p&=\left\|\sum_{j\in \mathbb{J}}\widehat {F}_j\widehat {L}_j\left(\sum_{k\in \mathbb{J}} L_kA_kx\right)\right\|^p=\sum_{j\in \mathbb{J}}\left\| \widehat {L}_j\left(\sum_{k\in \mathbb{J}} L_kA_kx\right)\right\|^p\\
&=\sum_{j\in \mathbb{J}}\|A_jx\|^p=\|x\|^p,
\end{align*}
where we used Remark \ref{P-ONBREMARK} to get the last equality. Thus $U$ is invertible  isometry. Now, $ F_jU=F_j(\sum_{k\in \mathbb{J}}\widehat{F}_kA_k)=A_j, c_jU^{-1}\widehat{F}_j=c_j\sum_{k\in \mathbb{J}}\widehat{A}_kF_k\widehat{F}_j=c_j\widehat{A}_j=\widehat{\Psi}_j,  \forall j \in \mathbb{J}.$ 
\item $(\Leftarrow)$ This is the definition of Riesz p-basis.

$(\Rightarrow)$ There exists a  p-orthonormal basis $ \{G_j\}_{j \in \mathbb{J}}$ in $ \mathcal{B}(\mathcal{X}, \mathcal{X}_0)$ and  bounded invertible operators $ R, S : \mathcal{X} \rightarrow \mathcal{X}$  with the resolvent of $ SR$ contains  $ (-\infty, 0]$    such that $ A_j=G_jR, \widehat{\Psi}_j=S\widehat{G}_j,  \forall j \in \mathbb{J}$, where  $\widehat{G}_j$ is the right p-orthonormal inverse of $G_j$ for each $ j \in \mathbb{J}$. Define $ U\coloneqq\sum_{j\in \mathbb{J}}\widehat{F}_jG_jR, V\coloneqq S\sum_{j\in \mathbb{J}}\widehat{G}_jF_j.$ Since $ \{\widehat{F}_j\}_{j \in \mathbb{J}}$ and  $\{\widehat{G}_j\}_{j \in \mathbb{J}}$ are p-orthonormal bases, as in the proof of (i), $ U,V$ are well-defined. We find $F_jU=F_j\sum_{k\in \mathbb{J}}\widehat{F}_kG_kR=G_jR=A_j,V\widehat{F}_j=S\sum_{k\in\mathbb{J}}\widehat{G}_kF_k\widehat{F}_j=S\widehat{G}_j=\widehat{\Psi}_j,\forall j \in \mathbb{J}$. Consider $ U(R^{-1}\sum_{j\in \mathbb{J}}\widehat{G}_jF_j)=(\sum_{k\in \mathbb{J}}\widehat{F}_kG_kR)(R^{-1}\sum_{j\in \mathbb{J}}\widehat{G}_jF_j)=I_\mathcal{X}$, $(R^{-1}\sum_{j\in \mathbb{J}}\widehat{G}_jF_j)U=(R^{-1}\sum_{j\in \mathbb{J}}\widehat{G}_jF_j)(\sum_{k\in \mathbb{J}}\widehat{F}_kG_kR) =I_\mathcal{X}$, and $V(\sum_{j\in \mathbb{J}}\widehat{F}_jG_jS^{-1})=(S\sum_{k\in \mathbb{J}}\widehat{G}_kF_k)(\sum_{j\in \mathbb{J}}\widehat{F}_jG_jS^{-1})=I_\mathcal{X}$,
$(\sum_{j\in \mathbb{J}}\widehat{F}_jG_jS^{-1})V=(\sum_{j\in \mathbb{J}}\widehat{F}_jG_jS^{-1})(S\sum_{k\in \mathbb{J}}\widehat{G}_kF_k)=I_\mathcal{X}$. Hence $U$, and $V$ are invertible. At the end, the resolvent of $ VU=(S\sum_{j\in \mathbb{J}}\widehat{G}_jF_j)(\sum_{k\in \mathbb{J}}\widehat{F}_kG_kR)=SR$ contains $(-\infty,0]$.
\item $(\Leftarrow)$ $ \theta_FU= (\sum_{j\in\mathbb{J}}L_jF_j)U=\sum_{j\in\mathbb{J}}L_j(F_jU),V\widehat{\theta}_{\widehat{F}}=V (\sum_{j\in\mathbb{J}}\widehat{F}_j\widehat{L}_j) =\sum_{j\in\mathbb{J}}(V\widehat{F}_j)\widehat{L}_j .$ Therefore $ \theta_{FU}=\theta_FU$, and $\widehat{\theta}_{V\widehat{F}}=V\widehat{\theta}_{\widehat{F}}$. Now $ \widehat{S}_{FU,V\widehat{F}}=\sum_{j\in\mathbb{J}}V\widehat{F}_jF_jU=V(\sum_{j\in\mathbb{J}}\widehat{F}_jF_j)U=VI_\mathcal{X}U=VU$ exists and whose resolvent contains  $(-\infty,0]$.
$(\Rightarrow)$  Define $U\coloneqq \widehat{\theta}_F\theta_A=\sum_{j\in \mathbb{J}}\widehat{F}_jA_j $,  $V\coloneqq \widehat{\theta}_\Psi\theta_F=\sum_{j\in \mathbb{J}}\widehat{\Psi}_jF_j $. Then $F_jU=F_j\sum_{k\in \mathbb{J}}\widehat{F}_kA_k=A_j$, $V\widehat{F}_j=\sum_{k\in \mathbb{J}}\widehat{\Psi}_kF_k\widehat{F}_j=\widehat{\Psi}_j ,\forall j \in \mathbb{J}$, and the resolvent of $VU=(\sum_{j\in \mathbb{J}}\widehat{\Psi}_jF_j)(\sum_{k\in \mathbb{J}}\widehat{F}_kA_k)=\sum_{j\in \mathbb{J}}\widehat{\Psi}_jA_j =\widehat{S}_{A,\Psi}$ contains $ (-\infty,0].$
\item From (iii).  $ (\Leftarrow)$ $ \widehat{P}_{FU,V\widehat{F}}=\theta_{FU}\widehat{S}^{-1}_{FU,V\widehat{F}}\widehat{\theta}_{V\widehat{F}}=\theta_FU(VU)^{-1}V\widehat{\theta}_{\widehat{F}}=\theta_FI_\mathcal{X}\widehat{\theta}_{\widehat{F}}=(\sum_{j\in \mathbb{J}}L_jF_j)(\sum_{k\in \mathbb{J}}\widehat{F}_k\widehat{L}_k)=I_{\ell^p(\mathbb{J})}\otimes I_{\mathcal{X}_0}$.
$(\Rightarrow)$ $ U(VU)^{-1}V=\widehat{\theta}_F\theta_A   \widehat{S}_{A,\Psi}^{-1}\widehat{\theta}_\Psi\theta_F =\widehat{\theta}_F\widehat{P}_{A,Psi}\theta_F =\widehat{\theta}_F(I_{\ell^p(\mathbb{J})}\otimes I_{\mathcal{X}_0})\theta_F=(\sum_{j\in \mathbb{J}}\widehat{F}_j\widehat{L}_j)(\sum_{k\in \mathbb{J}}L_kF_k)=I_\mathcal{X}.$
\item From (iv). $ (\Leftarrow)$ $\widehat{S}_{FU,V\widehat{F}}=VU=I_\mathcal{X} $, $\widehat{P}_{FU,V\widehat{F}}=\theta_{FU}\widehat{S}_{FU,V\widehat{F}}^{-1}\widehat{\theta}_{V\widehat{F}}=\theta_{FU}I_\mathcal{X}\widehat{\theta}_{V\widehat{F}}=\theta_FUV\widehat{\theta}_{\widehat{F}}=\theta_FI_{\mathcal{X}}\widehat{\theta}_{\widehat{F}}=I_{\ell^p(\mathbb{J})}\otimes I_{\mathcal{X}_0}.$
 $(\Rightarrow)$ 
 $VU=\widehat{S}_{A,\Psi}=I_\mathcal{X},$ 
\begin{align*}
UV&= \left(\sum_{j\in \mathbb{J}}\widehat{F}_jA_j\right)\left( \sum_{k\in \mathbb{J}}\widehat{\Psi}_kF_k\right) =
\left(\sum_{j\in \mathbb{J}}\widehat{F}_j\widehat{L}_j \left(\sum_{l\in \mathbb{J}}L_lA_l\right)\right) \left( \sum_{k\in \mathbb{J}}\widehat{\Psi}_k\widehat{L}_k \left( \sum_{m\in \mathbb{J}}L_mF_m\right)\right) \\ &=\widehat{\theta}_F\theta_A\widehat{\theta}_\Psi\theta_F=\widehat{\theta}_F\widehat{P}_{A,\Psi}\theta_F =\widehat{\theta}_F(I_{\ell^p(\mathbb{J})}\otimes I_{\mathcal{X}_0})\theta_F =I_\mathcal{X}. 
\end{align*}
\end{enumerate}
\end{proof}
\begin{corollary}
\begin{enumerate}[\upshape(i)]
\item If   $ (\{A_j\}_{j \in \mathbb{J}},\{\widehat{\Psi}_j\}_{j \in \mathbb{J}})$ is a p-orthonormal basis  in  $ \mathcal{B}(\mathcal{X}, \mathcal{X}_0)$, then 
$$ \sup\{\|A_j\|\}_{j\in\mathbb{J}}\leq1,~ \|\widehat{\Psi}_j\|\leq c_j \|\widehat{F}_j\|, ~\forall j \in \mathbb{J}, ~ A_j \widehat{\Psi}_j=c_jI_{\mathcal{X}_0},  ~\forall j \in \mathbb{J}.$$
\item If   $ (\{A_j\}_{j \in \mathbb{J}},\{\widehat{\Psi}_j\}_{j \in \mathbb{J}})$ is  p-(ovf)  in $ \mathcal{B}(\mathcal{X}, \mathcal{X}_0)$, then 
$$ \sup\{\|A_j\|\}_{j\in\mathbb{J}}\leq\|U\|, ~\|\widehat{\Psi}_j\|\leq\|V\|\|\widehat{F}_j\|,~ \forall j \in \mathbb{J},  ~ \|A_j\widehat{\Psi}_j\|\leq\|UV\widehat{F}_j\|, ~ \forall j \in \mathbb{J}.$$
\end{enumerate}	
\end{corollary}

\textbf{Similarity}
\begin{definition}
Let $(\{A_j\}_{j \in \mathbb{J}}, \{\widehat{\Psi}_j \}_ {j \in \mathbb{J}})$ and $(\{B_j\}_{j \in \mathbb{J}}, \{\widehat{\Phi}_j \}_ {j \in \mathbb{J}})$ be  p-operator-valued frames in  $ \mathcal{B}(\mathcal{X}, \mathcal{X}_0).$ We say that  $( \{B_j\}_{j \in \mathbb{J}}, \{\widehat{\Phi}_j \}_ {j \in \mathbb{J}})$ is   
\begin{enumerate}[\upshape(i)]
\item right-similar  to $ (\{A_j\}_{j \in \mathbb{J}}, \{\widehat{\Psi}_j \}_ {j \in \mathbb{J}})$ if there exist invertible $R_{A,B} \in \mathcal{B}(\mathcal{X}), R_{\Psi, \Phi} \in   \mathcal{B}(\mathcal{X}_0)$ such that $B_j=A_jR_{A,B} , \widehat{\Phi}_j=\widehat{\Psi}_jR_{\Psi, \Phi}, \forall j \in \mathbb{J}.$
\item left-similar to $ ( \{A_j\}_{j \in \mathbb{J}},\{\widehat{\Psi}_j \}_ {j \in \mathbb{J}})$ if there exist invertible $L_{A,B}  \in \mathcal{B}(\mathcal{X}_0), L_{\Psi, \Phi} \in \mathcal{B}(\mathcal{X}) $ such that $B_j=L_{A,B}A_j , \widehat{\Phi}_j=L_{\Psi, \Phi}\widehat{\Psi}_j , \forall j \in \mathbb{J}.$
\end{enumerate}
\end{definition}
\begin{proposition}
Let $ \{A_j\}_{j\in \mathbb{J}}\in \widehat{\mathscr{F}}_{\Psi,p}$, $ \{B_j\}_{j\in \mathbb{J}}\in \widehat{\mathscr{F}}_{\Phi,p}$,  $R_{A,B}  \in \mathcal{B}(\mathcal{X})$, $R_{\Psi, \Phi} \in   \mathcal{B}(\mathcal{X}_0) $, both $R_{A,B},R_{\Psi, \Phi} $ be  invertible and  $B_j=A_jR_{A,B} ,  \widehat{\Phi}_j=\widehat{\Psi}_jR_{\Psi, \Phi}, \forall j \in \mathbb{J}.$ Then $ \theta_B=\theta_A R_{A,B}, \widehat{\theta}_\Phi= \widehat{\theta}_\Psi(I_{\ell^p(\mathbb{J})}\otimes R_{\Psi, \Phi}), \widehat{S}_{B,\Phi}=\widehat{\theta}_\Phi\theta_B=\widehat{\theta}_\Psi(I_{\ell^p(\mathbb{J})}\otimes R_{\Psi, \Phi})\theta_A R_{A,B},  \widehat{P}_{B,\Phi}=\theta_A (\widehat{\theta}_\Psi(I_{\ell^p(\mathbb{J})}\otimes R_{\Psi, \Phi})\theta_A )^{-1}\widehat{\theta}_\Psi(I_{\ell^p(\mathbb{J})}\otimes R_{\Psi, \Phi}).$	
\end{proposition}
\begin{proof}
$ \theta_B=\sum_{j\in \mathbb{J}}L_jB_j=\sum_{j\in \mathbb{J}}L_jA_jR_{A,B}=\theta_A R_{A,B}$, $\widehat{\theta}_\Phi=\sum_{j\in \mathbb{J}}\widehat{\Phi}_j\widehat{L}_j=\sum_{j\in \mathbb{J}}\widehat{\Psi}_jR_{\Psi, \Phi}\widehat{L}_j=\sum_{j\in \mathbb{J}}\widehat{\Psi}_j\widehat{L}_j(I_{\ell^p(\mathbb{J})}\otimes R_{\Psi, \Phi})= \widehat{\theta}_\Psi(I_{\ell^p(\mathbb{J})}\otimes R_{\Psi, \Phi}) $, $\widehat{S}_{B,\Phi}=\widehat{\theta}_\Phi\theta_B=\widehat{\theta}_\Psi(I_{\ell^p(\mathbb{J})}\otimes R_{\Psi, \Phi})\theta_A R_{A,B} $, $\widehat{P}_{B,\Phi}=\theta_B\widehat{S}_{B,\Phi}^{-1}\widehat{\theta}_\Phi=\theta_A R_{A,B}(\widehat{\theta}_\Psi(I_{\ell^p(\mathbb{J})}\otimes R_{\Psi, \Phi})\theta_A R_{A,B})^{-1}\widehat{\theta}_\Psi(I_{\ell^p(\mathbb{J})}\otimes R_{\Psi, \Phi})=\theta_A (\widehat{\theta}_\Psi(I_{\ell^p(\mathbb{J})}\otimes R_{\Psi, \Phi})\theta_A )^{-1}\widehat{\theta}_\Psi(I_{\ell^p(\mathbb{J})}\otimes R_{\Psi, \Phi}).$
\end{proof}
\begin{proposition}
Let $ \{A_j\}_{j\in \mathbb{J}}\in \widehat{\mathscr{F}}_{\Psi,p}$, $ \{B_j\}_{j\in \mathbb{J}}\in \widehat{\mathscr{F}}_{\Phi,p}$,  $L_{A,B}  \in \mathcal{B}(\mathcal{X}_0)$, $L_{\Psi, \Phi} \in   \mathcal{B}(\mathcal{X}) $, both $L_{A,B},L_{\Psi, \Phi} $ be  invertible and  $B_j=L_{A,B}A_j ,  \widehat{\Phi}_j=L_{\Psi, \Phi}\widehat{\Psi}_j, \forall j \in \mathbb{J}.$ Then $ \theta_B=(I_{\ell^p(\mathbb{J})}\otimes L_{A,B})\theta_A, \widehat{\theta}_\Phi=L_{\Psi, \Phi}\widehat{\theta}_\Psi , \widehat{S}_{B,\Phi}=L_{\Psi, \Phi}\widehat{\theta}_\Psi(I_{\ell^p(\mathbb{J})}\otimes L_{A,B})\theta_A,  \widehat{P}_{B,\Phi}=(I_{\ell^p(\mathbb{J})}\otimes L_{A,B})\theta_A(\widehat{\theta}_\Psi(I_{\ell^p(\mathbb{J})}\otimes L_{A,B})\theta_A)^{-1}\widehat{\theta}_\Psi.$	
\end{proposition}
\begin{proof}
$\theta_B=\sum_{j\in \mathbb{J}}L_jB_j=\sum_{j\in \mathbb{J}}L_jL_{A,B}A_j=(I_{\ell^p(\mathbb{J})}\otimes L_{A,B})\sum_{j\in \mathbb{J}}L_jA_j=(I_{\ell^p(\mathbb{J})}\otimes L_{A,B})\theta_A $, $\widehat{\theta}_\Phi=\sum_{j\in \mathbb{J}}\widehat{\Phi}_j\widehat{L}_j=\sum_{j\in \mathbb{J}}L_{\Psi, \Phi}\widehat{\Psi}_j\widehat{L}_j=L_{\Psi, \Phi}\sum_{j\in \mathbb{J}}\widehat{\Psi}_j\widehat{L}_j=L_{\Psi, \Phi}\widehat{\theta}_\Psi $, $\widehat{S}_{B,\Phi}=\widehat{\theta}_\Phi\theta_B=L_{\Psi, \Phi}\widehat{\theta}_\Psi(I_{\ell^p(\mathbb{J})}\otimes L_{A,B})\theta_A $, $\widehat{P}_{B,\Phi}=\theta_B\widehat{S}_{B,\Phi}^{-1}\widehat{\theta}_\Phi=(I_{\ell^p(\mathbb{J})}\otimes L_{A,B})\theta_A(L_{\Psi, \Phi}\widehat{\theta}_\Psi(I_{\ell^p(\mathbb{J})}\otimes L_{A,B})\theta_A)^{-1}L_{\Psi, \Phi}\widehat{\theta}_\Psi=(I_{\ell^p(\mathbb{J})}\otimes L_{A,B})\theta_A(\widehat{\theta}_\Psi(I_{\ell^p(\mathbb{J})}\otimes L_{A,B})\theta_A)^{-1}\widehat{\theta}_\Psi $.
\end{proof}

\begin{definition}
Let $(\{A_j\}_{j \in \mathbb{J}}, \{\widehat{\Psi}_j \}_ {j \in \mathbb{J}})$ and $(\{B_j\}_{j \in \mathbb{J}}, \{\widehat{\Phi}_j \}_ {j \in \mathbb{J}})$ be  p-operator-valued frames in  $  \mathcal{B}(\mathcal{X}, \mathcal{X}_0)$. We say that  $( \{B_j\}_{j \in \mathbb{J}}, \{\widehat{\Phi}_j \}_ {j \in \mathbb{J}})$ is   
\begin{enumerate}[\upshape(i)]
\item RL-similar (right-left-similar) to $ (\{A_j\}_{j \in \mathbb{J}}, \{\widehat{\Psi}_j \}_ {j \in \mathbb{J}})$ if there exist invertible $R_{A,B}, L_{\Psi, \Phi}  \in \mathcal{B}(\mathcal{X})  $ such that $B_j=A_jR_{A,B} , \widehat{\Phi}_j=L_{\Psi, \Phi}\widehat{\Psi}_j, \forall j \in \mathbb{J}.$
\item LR-similar (left-right-similar) to $ ( \{A_j\}_{j \in \mathbb{J}},\{\widehat{\Psi}_j \}_ {j \in \mathbb{J}})$ if there exist invertible $L_{A,B}, R_{\Psi, \Phi}  \in \mathcal{B}(\mathcal{X}_0)$ such that $B_j=L_{A,B}A_j , \widehat{\Phi}_j=\widehat{\Psi}_jR_{\Psi, \Phi} , \forall j \in \mathbb{J}.$
\end{enumerate}
\end{definition}

\begin{lemma}\label{P-OVFLEMMA}
Let $ \{A_j\}_{j\in \mathbb{J}}\in \widehat{\mathscr{F}}_{\Psi,p}$, $ \{B_j\}_{j\in \mathbb{J}}\in \widehat{\mathscr{F}}_{\Phi,p}$,  $R_{A,B}, L_{\Psi, \Phi} \in \mathcal{B}(\mathcal{X})$ be  invertible and  $B_j=A_jR_{A,B} , \widehat{\Phi}_j=L_{\Psi, \Phi}\widehat{\Psi}_j,  \forall j \in \mathbb{J}.$ Then $ \theta_B=\theta_A R_{A,B}, \widehat{\theta}_\Phi= L_{\Psi,\Phi}\widehat{\theta}_\Psi, \widehat{S}_{B,\Phi}=L_{\Psi,\Phi}\widehat{S}_{A, \Psi}R_{A,B},  \widehat{P}_{B,\Phi}=\widehat{P}_{A, \Psi}.$ Assuming that $ (\{A_j\}_{j\in \mathbb{J}},\{\widehat{\Psi}_j\}_{j\in \mathbb{J}})$ is a Parseval p-(ovf), then $ (\{B_j\}_{j\in \mathbb{J}},\{\widehat{\Phi}_j\}_{j\in \mathbb{J}})$ is a Parseval  p-(ovf) if and only if   $ L_{\Psi, \Phi}R_{A,B}=I_\mathcal{X}.$ 
\end{lemma}
\begin{proof}
$ \theta_B=\sum_{j\in\mathbb{J}}L_jB_j=\sum_{j\in\mathbb{J}}L_jA_jR_{A,B}=\theta_AR_{A,B}$, $\widehat{\theta}_\Phi=\sum_{j\in\mathbb{J}}\widehat{\Phi}_j\widehat{L}_j=\sum_{j\in\mathbb{J}}L_{\Psi,\Phi}\widehat{\Psi}_j\widehat{L}_j= L_{\Psi,\Phi}\widehat{\theta}_\Psi $, $\widehat{S}_{B,\Phi}=\widehat{\theta}_\Phi\theta_B=L_{\Psi,\Phi}\widehat{\theta}_\Psi\theta_AR_{A,B} =L_{\Psi,\Phi}\widehat{S}_{A, \Psi}R_{A,B} $, $\widehat{P}_{B,\Phi}=\theta_B\widehat{S}_{B, \Phi}^{-1}\widehat{\theta}_\Phi=(\theta_AR_{A,B})(L_{\Psi,\Phi}\widehat{S}_{A, \Psi}R_{A,B})^{-1}(L_{\Psi,\Phi}\widehat{\theta}_\Psi)=\theta_A\widehat{S}_{A, \Psi}^{-1}\widehat{\theta}_\Psi=\widehat{P}_{A, \Psi} $.
\end{proof}
\begin{theorem}\label{RIGHTLEFTSIMILARITYPOVF}
Let $ \{A_j\}_{j\in \mathbb{J}}\in \widehat{\mathscr{F}}_{\Psi,p},$ $ \{B_j\}_{j\in \mathbb{J}}\in \widehat{\mathscr{F}}_{\Phi,p}.$ The following are equivalent.
\begin{enumerate}[\upshape(i)]
\item $B_j=A_jR_{A,B} , \widehat{\Phi}_j=L_{\Psi, \Phi}\widehat{\Psi}_j ,  \forall j \in \mathbb{J},$ for some invertible  $R_{A,B}, L_{\Psi, \Phi} \in \mathcal{B}(\mathcal{X}).$
\item $\theta_B=\theta_AR_{A,B}' , \widehat{\theta}_\Phi= L_{\Psi, \Phi}'\widehat{\theta}_\Psi  $ for some invertible  $ R_{A,B}' ,L_{\Psi, \Phi}' \in \mathcal{B}(\mathcal{X}). $
\item $\widehat{P}_{B,\Phi}=\widehat{P}_{A,\Psi}.$
\end{enumerate}
If one of the above conditions is satisfied, then  invertible operators in  $ \operatorname{(i)}$ and  $ \operatorname{(ii)}$ are unique and are given by $R_{A,B}=\widehat{S}_{A,\Psi}^{-1}\widehat{\theta}_\Psi\theta_B, L_{\Psi, \Phi}=\widehat{\theta}_\Phi\theta_A\widehat{S}_{A,\Psi}^{-1}.$
In the case that $(\{A_j\}_{j\in \mathbb{J}},  \{\widehat{\Psi}_j\}_{j\in \mathbb{J}})$ is a Parseval p-(ovf), then $(\{B_j\}_{j\in \mathbb{J}},  \{\widehat{\Phi}_j\}_{j\in \mathbb{J}})$ is  a Parseval p-(ovf) if and only if $L_{\Psi, \Phi}R_{A,B}=I_\mathcal{X} $  if and only if $R_{A,B}L_{\Psi, \Phi}=I_\mathcal{X}.$  
\end{theorem}
\begin{proof}
(i) $\Rightarrow $ (ii) $\Rightarrow $ (iii) are from Lemma \ref{P-OVFLEMMA}. Assume (ii). From Proposition \ref{BANACHSPACEOVFFUNDAMENTALLEMMA}, $ B_j=\widehat{L}_j\theta_B=\widehat{L}_j\theta_AR_{A,B}'=A_jR_{A,B}',\forall j \in \mathbb{J}.$ Again Proposition \ref{BANACHSPACEOVFFUNDAMENTALLEMMA} gives $ \widehat{\Phi}_j=\widehat{\theta}_{\Phi}L_j=L_{\Psi, \Phi}'\widehat{\theta}_\Psi L_j=L_{\Psi, \Phi}\widehat{\Psi}_j, \forall j\in \mathbb{J}.$ Hence (i) holds. For (iii) $\Rightarrow $ (ii), $\theta_B=\widehat{P}_{B,\Phi}\theta_B=\widehat{P}_{A,\Psi}\theta_B=\theta_A(\widehat{S}_{A,\Psi}^{-1}\widehat{\theta}_{\Psi}\theta_B) $, $ \widehat{\theta}_\Phi=\widehat{\theta}_\Phi\widehat{P}_{B,\Phi}=\widehat{\theta}_\Phi\widehat{P}_{A,\Psi}=(\widehat{\theta}_\Phi\theta_A\widehat{S}_{A,\Psi}^{-1})\widehat{\theta}_\Psi$. To show $\widehat{S}_{A,\Psi}^{-1}\widehat{\theta}_{\Psi}\theta_B $, and $\widehat{\theta}_\Phi\theta_A\widehat{S}_{A,\Psi}^{-1} $ are invertible: $(\widehat{S}_{A,\Psi}^{-1}\widehat{\theta}_{\Psi}\theta_B)(\widehat{S}_{B,\Phi}^{-1}\widehat{\theta}_{\Phi}\theta_A)=\widehat{S}_{A,\Psi}^{-1}\widehat{\theta}_{\Psi}\widehat{P}_{B,\Phi}\theta_A=\widehat{S}_{A,\Psi}^{-1}\widehat{\theta}_{\Psi}\widehat{P}_{A,\Psi}\theta_A=I_\mathcal{X} $, $(\widehat{S}_{B,\Phi}^{-1}\widehat{\theta}_{\Phi}\theta_A)(\widehat{S}_{A,\Psi}^{-1}\widehat{\theta}_{\Psi}\theta_B)=\widehat{S}_{B,\Phi}^{-1}\widehat{\theta}_{\Phi}\widehat{P}_{A,\Psi}\theta_B=\widehat{S}_{B,\Phi}^{-1}\widehat{\theta}_{\Phi}\widehat{P}_{B,\Phi}\theta_B=I_\mathcal{X} $, and $(\widehat{\theta}_\Phi\theta_A\widehat{S}_{A,\Psi}^{-1})(\widehat{\theta}_\Psi\theta_B\widehat{S}_{B,\Phi}^{-1})=\widehat{\theta}_\Phi\widehat{P}_{A,\Psi}\theta_B\widehat{S}_{B,\Phi}^{-1}=\widehat{\theta}_\Phi\widehat{P}_{B,\Phi}\theta_B\widehat{S}_{B,\Phi}^{-1}=I_\mathcal{X} $, $(\widehat{\theta}_\Psi\theta_B\widehat{S}_{B,\Phi}^{-1})(\widehat{\theta}_\Phi\theta_A\widehat{S}_{A,\Psi}^{-1})=\widehat{\theta}_\Psi\widehat{P}_{B,\Phi}\theta_A\widehat{S}_{A,\Psi}^{-1}=\widehat{\theta}_\Psi\widehat{P}_{A,\Psi}\theta_A\widehat{S}_{A,\Psi}^{-1}=I_\mathcal{X} $.

For the formula, let invertible  $R_{A,B}, L_{\Psi, \Phi} \in \mathcal{B}(\mathcal{X})$ satisfy $\theta_B=\theta_AR_{A,B},\widehat{\theta}_\Phi= L_{\Psi, \Phi}\widehat{\theta}_\Psi $. Hence $\widehat{\theta}_\Psi\theta_B=\widehat{\theta}_\Psi\theta_AR_{A,B}=\widehat{S}_{A,\Psi}R_{A,B},\widehat{\theta}_\Phi\theta_A= L_{\Psi, \Phi}\widehat{\theta}_\Psi\theta_A=L_{\Psi, \Phi}\widehat{S}_{A,\Psi} $, from which the formula follows.
\end{proof}
\begin{corollary}
For a given p-(ovf) $ (\{A_j\}_{j \in \mathbb{J}} , \{\widehat{\Psi}_j\}_{j \in \mathbb{J}})$, the canonical dual of $ (\{A_j\}_{j \in \mathbb{J}} , \{\widehat{\Psi}_j\}_{j \in \mathbb{J}}  )$ is the only dual p-(ovf) that is RL-similar to $ (\{A_j\}_{j \in \mathbb{J}} , \{\widehat{\Psi}_j\}_{j \in \mathbb{J}} )$.
\end{corollary}
\begin{proof}
Whenever $ (\{B_j\}_{j \in \mathbb{J}} , \{\widehat{\Phi}_j\}_{j \in \mathbb{J}} )$ is dual of $ (\{A_j\}_{j \in \mathbb{J}} , \{\widehat{\Psi}_j\}_{j \in \mathbb{J}})$ as well as RL-similar to $ (\{A_j\}_{j \in \mathbb{J}} , \{\widehat{\Psi}_j\}_{j \in \mathbb{J}} )$, we have $ \widehat{\theta}_\Psi\theta_B=I_\mathcal{X}=\widehat{\theta}_\Phi\theta_A$ and  there exist invertible $ R_{A,B},L_{\Psi,\Phi}\in \mathcal{B}(\mathcal{X})$ such that  $B_j=A_jR_{A,B} , \widehat{\Phi}_j=L_{\Psi, \Phi}\widehat{\Psi}_j ,  \forall j \in \mathbb{J} $. Theorem \ref{RIGHTLEFTSIMILARITYPOVF} tells  $R_{A,B}=\widehat{S}_{A,\Psi}^{-1}\widehat{\theta}_\Psi\theta_B$, $L_{\Psi, \Phi}=\widehat{\theta}_\Phi\theta_A\widehat{S}_{A,\Psi}^{-1}.$ But then $R_{A,B}=\widehat{S}_{A,\Psi}^{-1}I_\mathcal{X}=\widehat{S}_{A,\Psi}^{-1}, L_{\Psi, \Phi}=I_\mathcal{X}\widehat{S}_{A,\Psi}^{-1}=\widehat{S}_{A,\Psi}^{-1}.$ So $ (\{B_j=A_j\widehat{S}_{A,\Psi}^{-1}\}_{j \in \mathbb{J}} , \{\widehat{\Phi}_j=\widehat{S}_{A,\Psi}^{-1}\widehat{\Psi}_j\}_{j \in \mathbb{J}} )$ is canonical dual of $ (\{A_j\}_{j \in \mathbb{J}} , \{\widehat{\Psi}_j\}_{j \in \mathbb{J}})$.
\end{proof}
\begin{corollary}
Two RL-similar p-operator-valued frames cannot be orthogonal.
\end{corollary}
\begin{proof}
Let $ (\{B_j\}_{j \in \mathbb{J}} , \{\widehat{\Phi}_j\}_{j \in \mathbb{J}} )$ be RL-similar to $ (\{A_j\}_{j \in \mathbb{J}} , \{\widehat{\Psi}_j\}_{j \in \mathbb{J}})$. Choose invertible $ R_{A,B},L_{\Psi,\Phi}\in \mathcal{B}(\mathcal{X})$ such that  $B_j=A_jR_{A,B} , \widehat{\Phi}_j=L_{\Psi, \Phi}\widehat{\Psi}_j ,  \forall j \in \mathbb{J} $. Using 	Theorem \ref{RIGHTLEFTSIMILARITYPOVF} we get $\widehat{\theta}_\Psi\theta_B=\widehat{\theta}_\Psi\theta_AR_{A,B}=\widehat{S}_{A,\Psi}R_{A,B} \neq 0.$ 
\end{proof}

\begin{remark}
For every p-(ovf) $(\{A_j\}_{j \in \mathbb{J}}, \{\widehat{\Psi}_j\}_{j \in \mathbb{J}}),$ each  of `p-operator-valued frames'  $( \{A_j\widehat{S}_{A, \Psi}^{-1}\}_{j \in \mathbb{J}}, \{\widehat{\Psi}_j\}_{j \in \mathbb{J}}),$   $( \{A_j\widehat{S}_{A, \Psi}^{-1/2}\}_{j \in \mathbb{J}}, \{\widehat{S}_{A,\Psi}^{-1/2}\widehat{\Psi}_j\}_{j \in \mathbb{J}}),$ and  $ (\{A_j \}_{j \in \mathbb{J}}, \{\widehat{S}_{A,\Psi}^{-1}\widehat{\Psi}_j\}_{j \in \mathbb{J}})$ is a Parseval p-(ovf) which is RL-similar to  $ (\{A_j\}_{j \in \mathbb{J}} , \{\widehat{\Psi}_j\}_{j \in \mathbb{J}}).$  Thus every p-(ovf) is RL-similar to  Parseval  p-operator-valued frames.
\end{remark}
\begin{proposition}
Let $ \{A_j\}_{j\in \mathbb{J}}\in \widehat{\mathscr{F}}_{\Psi,p}$, $ \{B_j\}_{j\in \mathbb{J}}\in \widehat{\mathscr{F}}_{\Phi,p}$ and   $B_j=L_{A,B}A_j , \widehat{\Phi}_j=\widehat{\Psi}_jR_{\Psi, \Phi},  \forall j \in \mathbb{J}$, for some invertible $ L_{A,B} ,R_{\Psi, \Phi} \in \mathcal{B}(\mathcal{X}_0).$ Then 
\begin{enumerate}[\upshape(i)]
\item $ \theta_B=(I_{\ell^p(\mathbb{J})}\otimes L_{A,B})\theta_A , \widehat{\theta}_\Phi=\widehat{\theta}_\Psi(I_{\ell^p(\mathbb{J})}\otimes R_{\Psi,\Phi}), \widehat{S}_{B,\Phi}=\widehat{\theta}_\Psi(I_{\ell^p(\mathbb{J})}\otimes R_{\Psi,\Phi}L_{A,B})\theta_A,  \widehat{P}_{B,\Phi}=(I_{\ell^p(\mathbb{J})}\otimes L_{A,B})\theta_A(\widehat{\theta}_\Psi(I_{\ell^p(\mathbb{J})}\otimes R_{\Psi,\Phi}L_{A,B})\theta_A)^{-1}\widehat{\theta}_\Psi(I_{\ell^p(\mathbb{J})}\otimes R_{\Psi,\Phi}).$ 
\item Assuming $( \{A_j\}_{j \in \mathbb{J}}, \{\widehat{\Psi}_j\}_{j \in \mathbb{J}})$ is a Parseval p-(ovf), then $( \{B_j\}_{j \in \mathbb{J}}, \{\widehat{\Phi}_j\}_{j \in \mathbb{J}})$ is a Parseval p-(ovf) if and only if $ \widehat{P}_{A, \Psi}(I_{\ell^p(\mathbb{J})}\otimes R_{\Psi,\Phi}L_{A,B})\widehat{P}_{A,\Psi}=\widehat{P}_{A,\Psi}$ if and only if $ \widehat{P}_{B,\Phi}=(I_{\ell^p(\mathbb{J})}\otimes L_{A,B})\widehat{P}_{A,\Psi}(I_{\ell^p(\mathbb{J})}\otimes R_{\Psi, \Phi}).$
\end{enumerate}	
\end{proposition}
\begin{proof}
Similar to the proof of Proposition \ref{3.3}. Please note that rather considering ``stars" we have to consider ``hats".
\end{proof}

\section{Sequential version of p-operator-valued frames and p-bases}\label{SVSECTION}

Here we develop abstractly the notion of sequential version of p-frames in Banach spaces. Our definition is different than that  of frames for Banach spaces by Grochenig \cite{GROCHENIG1} as well as by Aldroubi, Sun, and Tang \cite{ALDROUBISUNTANG}.

\begin{definition}
Let  $ \mathcal{X}$ be a Banach space, $ p \in[1, \infty)$. A sequence $\{f_j\}_{j\in \mathbb{J}} $ in $\mathcal{X}^*$ is said to be a p-frame w.r.t. a sequence $\{\tau_j\}_{j\in \mathbb{J}} $ in $\mathcal{X}$  if 
\begin{enumerate}[\upshape(i)]
\item $\theta_f: \mathcal{X} \ni x \mapsto \{f_j(x)\}_{j\in \mathbb{J}} \in \ell^p(\mathbb{J})  $ (analysis operator) and $ \widehat{\theta}_\tau:\ell^p(\mathbb{J}) \ni \{a_j\}_{j\in \mathbb{J}} \mapsto  \sum_{j\in \mathbb{J}}a_j\tau_j\in \mathcal{X} $ (synthesis operator) are well-defined bounded  operators, 
\item  $  \widehat{S}_{f,\tau}: \mathcal{X} \ni x \mapsto \sum_{j\in \mathbb{J}}f_j(x)\tau_j \in \mathcal{X}$ (p-frame operator) is a bounded operator whose resolvent contains $(-\infty, 0] $. 
\end{enumerate}
In this situation,  we write $ (\{f_j\}_{j\in \mathbb{J}}, \{\tau_j\}_{j\in \mathbb{J}})$ is a p-frame for $\mathcal{X}.$ 

Constants $ a, b >0$ satisfying 
$$a^{1/p}\|x \|_\mathcal{X}\leq  \|\widehat{S}_{f,\tau}^{1/p}x\|_\mathcal{X} \leq b^{1/p} \|x\|_\mathcal{X} , ~\forall x \in \mathcal{X}$$ are called as lower and upper p-frame bounds, in order. Supremum of the set of all lower p-frame bounds is called the optimal lower frame  bound and similarly for the optimal upper frame bound. Whenever $\widehat{S}_{f,\tau}=\alpha I_\mathcal{X}$, for some $ \alpha \in \mathbb{K}$, we say p-frame is tight and whenever $\widehat{S}_{f,\tau}=I_\mathcal{X}$ we call p-frame as Parseval p-frame.
\end{definition}

\begin{definition}\label{ORTHOGONALITYINBANACH}
A sequence 	$ \{x_j\}_{j \in \mathbb{J}}$ in $ \mathcal{X}$ is said to be  p-orthogonal for  $ \mathcal{X}$ if   
 $ \mathbb{L}\subseteq \mathbb{J}$ and  $ \{c_j\}_{j \in \mathbb{L}}$ is a sequence of scalars such that $ \sum_{j\in \mathbb{L}}c_jx_j\in \mathcal{X}$, then  $ \|\sum_{j\in \mathbb{L}}c_jx_j\|^p=\sum_{j\in \mathbb{L}}\|c_jx_j\|^p$.
\end{definition}

\begin{definition}\label{ORTHONORMALITYINBANACH}
A sequence 	$ \{x_j\}_{j \in \mathbb{J}}$ in $ \mathcal{X}$ is said to be  p-orthonormal for  $ \mathcal{X}$ if the following conditions hold.
\begin{enumerate}[\upshape(i)]
\item If $ \mathbb{L}\subseteq \mathbb{J}$ and  $ \{c_j\}_{j \in \mathbb{L}}$ is a sequence of scalars such that $ \sum_{j\in \mathbb{L}}c_jx_j\in \mathcal{X}$, then  $ \|\sum_{j\in \mathbb{L}}c_jx_j\|^p=\sum_{j\in \mathbb{L}}|c_j|^p$.
\item If $ \{a_j\}_{j \in \mathbb{J}} \in \ell^p(\mathbb{J})$, then $ \sum_{j\in \mathbb{J}}a_jx_j\in \mathcal{X}$ and $\|\sum_{j\in \mathbb{J}}a_jx_j\|^p=\sum_{j\in \mathbb{J}}|a_j|^p$.
\end{enumerate}
\begin{example}\label{PORTHOGONALASWELLASPORTHONORMALEXAMPLE}
For $1\leq p<\infty$, the standard Schauder basis $ \{e_j\}_{j \in \mathbb{J}}$ for $ \ell^p(\mathbb{J})$ is p-orthogonal as well as p-orthonormal for $ \ell^p(\mathbb{J})$, for $1\leq p<\infty$.
\end{example}
\end{definition}
 Note that every  orthogonal set (resp. orthonormal set)  in  every  Hilbert space satisfies Definition \ref{ORTHOGONALITYINBANACH} (resp. Definition \ref{ORTHONORMALITYINBANACH}).
\begin{theorem}\label{2-OGSIMPLIESOGS}
A 2-orthogonal sequence  for a Hilbert space $\mathcal{H}$ is an orthogonal sequence for $\mathcal{H}$.
\end{theorem}
\begin{proof}
Let $ \{x_j\}_{j \in \mathbb{J}}$ be  2-orthogonal for  $\mathcal{H}$.  We have to show  $ \langle x_j, x_k \rangle =0 ,\forall j\neq k,\forall j, k \in \mathbb{J}$. Let $j\neq k$. Case (i): $\mathcal{H}$ is over $ \mathbb{C}$. We use  polarization identity to get, 
\begin{align*}
\langle x_j, x_k \rangle &=\frac{\|x_j+x_k\|^2-\|x_j-x_k\|^2+i\|x_j+ix_k\|^2-i\|x_j-ix_k\|^2}{4}\\	&=\frac{(\|x_j\|^2+\|x_k\|^2)-(\|x_j\|^2+\|-x_k\|^2)+i(\|x_j\|^2+\|ix_k\|^2)-i(\|x_j\|^2+\|-ix_k\|^2)}{4}=0.
\end{align*}
 Case (ii): $\mathcal{H}$ is over $ \mathbb{R}$. We again use  polarization identity to get, 
\begin{align*}
\langle x_j, x_k \rangle =\frac{\|x_j+x_k\|^2-\|x_j-x_k\|^2}{4}	=\frac{(\|x_j\|^2+\|x_k\|^2)-(\|x_j\|^2+\|-x_k\|^2)}{4}=0.
\end{align*}
\end{proof}

\begin{proposition}
If $ \{x_j\}_{j \in \mathbb{J}} $ is  p-orthogonal   for  $ \mathcal{X}$ such that $ x_j \neq 0, \forall j \in \mathbb{J} $, then $ \left\{\frac{x_j}{\|x_j\|}\right\}_{j \in \mathbb{J}} $ is  p-orthonormal   for  $ \mathcal{X}$.	
\end{proposition}
\begin{proof}
If  $ \mathbb{L}\subseteq \mathbb{J}$ and  $ \{c_j\}_{j \in \mathbb{L}}$ is a sequence of scalars such that $ \sum_{j\in \mathbb{L}}c_j\frac{x_j}{\|x_j\|}\in \mathcal{X}$, then we see $ \|\sum_{j\in \mathbb{L}}c_j\frac{x_j}{\|x_j\|}\|^p=\sum_{j\in \mathbb{L}}\|c_j\frac{x_j}{\|x_j\|}\|^p=\sum_{j\in \mathbb{L}}|c_j|^p$, and if  $ \{a_j\}_{j \in \mathbb{J}} \in \ell^p(\mathbb{J})$, then  
$$\sum_{j\in \mathbb{S}}|a_j|^p=\sum_{j\in \mathbb{S}}\left\|a_j\frac{x_j}{\|x_j\|}\right\|^p=\sum_{j\in \mathbb{S}}\left\|\frac{a_j}{\|x_j\|}x_j\right\|^p= \left\|\sum_{j\in \mathbb{S}}\frac{a_j}{\|x_j\|}x_j\right\|^p=\left\|\sum_{j\in \mathbb{S}}a_j\frac{x_j}{\|x_j\|}\right\|^p$$
for every finite subset $\mathbb{S} $ of $\mathbb{J}$. From the existence of   $\sum_{j\in \mathbb{J}}|a_j|^p $ we get the existence of  $\sum_{j\in \mathbb{J}}a_j\frac{x_j}{\|x_j\|} $  and $\sum_{j\in \mathbb{J}}|a_j|^p=\|\sum_{j\in \mathbb{J}}a_j\frac{x_j}{\|x_j\|}\|^p$.

\end{proof}
\begin{theorem}
\begin{enumerate}[\upshape(i)]
\item If $ \{x_n\}_{n=1}^m $ is  p-orthogonal  for  $ \mathcal{X}$, then $\|\sum_{n=1}^{m}x_n\|^p=\sum_{n=1}^{m}\| x_n\|^p $.
\item If $ \{x_j\}_{j \in \mathbb{J}} $ is  p-orthogonal   for  $ \mathcal{X}$ such that $ x_j \neq 0, \forall j \in \mathbb{J} $, then $ \{x_j\}_{j \in \mathbb{J}} $ is linearly independent. In particular, if $ \{x_j\}_{j \in \mathbb{J}} $ is  p-orthonormal, then it is linearly independent.
\end{enumerate}
\end{theorem}
\begin{proof}
Only (ii) is nontrivial, we prove it. Case (i): $ \mathcal{X}$ is over $ \mathbb{C}$. Let $ c_{j_1},..., c_{j_n} \in \mathbb{C}$ be such that $\sum_{k=1}^nc_{j_k}x_{j_k}=0 $, where $ j_1,...,j_n \in \mathbb{J}$. Fix $ 1\leq l\leq n$. Now 
\begin{align*}
&\|0+x_{j_l}\|^p-\|0-x_{j_l}\|^p+i\|0+ix_{j_l}\|^p-i\|0-ix_{j_l}\|^p=0\\
&\Rightarrow \left\|\sum_{k=1}^nc_{j_k}x_{j_k}+x_{j_l}\right\|^p-\left\|\sum_{k=1}^nc_{j_k}x_{j_k}-x_{j_l}\right\|^p+i\left\|\sum_{k=1}^nc_{j_k}x_{j_k}+ix_{j_l}\right\|^p-i\left\|\sum_{k=1}^nc_{j_k}x_{j_k}-ix_{j_l}\right\|^p=0\\
&\Rightarrow  \left\|\sum_{k=1,k\neq l}^nc_{j_k}x_{j_k}+(c_{j_l}+1)x_{j_l}\right\|^p-\left\|\sum_{k=1,k\neq l}^nc_{j_k}x_{j_k}+(c_{j_l}-1)x_{j_l}\right\|^p\\
&\quad+i\left\|\sum_{k=1,k\neq l}^nc_{j_k}x_{j_k}+(c_{j_l}+i)x_{j_l}\right\|^p-i\left\|\sum_{k=1,k\neq l}^nc_{j_k}x_{j_k}+(c_{j_l}-i)x_{j_l}\right\|^p=0\\
&\Rightarrow \sum_{k=1,k\neq l}^n\|c_{j_k}x_{j_k}\|^p+\|(c_{j_l}+1)x_{j_l}\|^p-\sum_{k=1,k\neq l}^n\|c_{j_k}x_{j_k}\|^p-\|(c_{j_l}-1)x_{j_l}\|^p\\
&\quad+i\sum_{k=1,k\neq l}^n\|c_{j_k}x_{j_k}\|^p+i\|(c_{j_l}+i)x_{j_l}\|^p-i\sum_{k=1,k\neq l}^n\|c_{j_k}x_{j_k}\|^p-i\|(c_{j_l}-i)x_{j_l}\|^p=0
\end{align*}
$\Rightarrow |c_{j_l}+1|=|c_{j_l}-1| ~\text{and}~|c_{j_l}+i|=|c_{j_l}-i|\Rightarrow c_{j_l}+\overline{c}_{j_l}=0 ~\text{and}~c_{j_l}-\overline{c}_{j_l}=0$. Hence $ c_{j_l}=0$. Since $ l$ was arbitrary, we must have $ c_{j_1}=\cdots=c_{j_n}=0$. Case (ii): Consider $\|0+x_{j_l}\|^p-\|0-x_{j_l}\|^p=0 $ and proceed.
\end{proof}
\begin{proposition}\label{ellp}
Let $ \{x_j\}_{j \in \mathbb{J}}$  be  p-orthonormal  for $ \mathcal{X}$. Then 
\begin{enumerate}[\upshape(i)]
\item $ \|x_j\|=1, \forall j \in \mathbb{J}$. In particular, if $ \{x_n\}_{n=1}^m $ is  p-orthonormal  in  $ \mathcal{X}$, then   $\|\sum_{n=1}^{m}x_n\|^p=m$. 
\item If $ x \in  \mathcal{X}$ has an expansion $ x=\sum_{j\in \mathbb{J}}c_jx_j$ for some   scalar sequence $ \{c_j\}_{j \in \mathbb{J}}$, then $ \|x\|^p=\sum_{j\in \mathbb{J}}|c_j|^p$. In particular, $ \{c_j\}_{j \in \mathbb{J}} \in \ell^p(\mathbb{J})$.
\item If $\mathbb{S} $ is any finite subset of $\mathbb{J}$, then $\|\sum_{j\in \mathbb{S}}a_jx_j\|^p=\sum_{j\in \mathbb{S}}|a_j|^p$.
\end{enumerate}
\end{proposition}

\begin{proof}
\begin{enumerate}[\upshape(i)]
\item $ x_j=1\cdot x_j \Rightarrow \|x_j\|^p=|1|^p,  \forall j \in \mathbb{J}$.
\item For every finite subset $\mathbb{S}\subseteq\mathbb{J}$, $\|\sum_{j\in \mathbb{S}}c_jx_j\|^p=\sum_{j\in \mathbb{S}}|c_j|^p $. From the   convergence of  $\sum_{j\in \mathbb{J}}c_jx_j $, we get the convergence of $\sum_{j\in \mathbb{J}}|c_j|^p $ and $ \|x\|^p=\sum_{j\in \mathbb{J}}|c_j|^p$.
\item Define $d_j\coloneqq a_j $ if $ j\in \mathbb{S}$ and $ 0$ if $ j\in \mathbb{J}\setminus \mathbb{S}$. Now apply (i) of Definition \ref{ORTHONORMALITYINBANACH}.
\end{enumerate}	
\end{proof}
\begin{proposition}
If $ \{x_j\}_{j \in \mathbb{J}} $ is p-orthonormal for $ \mathcal{X}$, then it  is closed.
\end{proposition}
\begin{proof}
Let $\{ x_{j_n}\}_{n=1}^\infty$	be in  $ \{x_j\}_{j \in \mathbb{J}} $ converging to an element $ x \in \mathcal{X}$. Then $ \|x_{j_l}-x_{j_m}\|^p=|1|^p+|-1|^p=2, \forall  x_{j_l}\neq x_{j_m}$. Therefore $\{ x_{j_n}\}_{n=1}^\infty$ is  eventually constant.
\end{proof}

\begin{proposition}
Let   $  \mathcal{X}$ has  a dense subset indexed with $ \mathbb{J}$. If    $ \{y_l\}_{l \in \mathbb{L}} $ is  p-orthonormal  for  $ \mathcal{X}$, then $ \operatorname{Card}(\mathbb{L}) \leq  \operatorname{Card}(\mathbb{J})$.
\end{proposition}
\begin{proof}
We just need to observe  $ \|y_l-y_m\|^p=|1|^p+|-1|^p=2, \forall  y_l\neq y_m$. Other arguments are similar with the  corresponding result in  Hilbert C*-module. 
\end{proof}
\begin{corollary}
If $\mathcal{X}$ is separable, then every p-orthonormal set  for  $\mathcal{X}$ is also separable.
\end{corollary}

\begin{definition}\label{orthonormality in Banach}
A sequence 	$ \{x_j\}_{j \in \mathbb{J}}$ in $ \mathcal{X}$ is said to be a p-orthonormal basis for $ \mathcal{X}$ if the following conditions hold.
\begin{enumerate}[\upshape(i)]
\item $ \{x_j\}_{j \in \mathbb{J}}$ is a Schauder basis for  $ \mathcal{X}$.
\item If $ \mathbb{L}\subseteq \mathbb{J}$ and  $ \{c_j\}_{j \in \mathbb{L}}$ is a sequence of scalars such that $ \sum_{j\in \mathbb{L}}c_jx_j\in \mathcal{X}$, then  $ \|\sum_{j\in \mathbb{L}}c_jx_j\|^p=\sum_{j\in \mathbb{L}}|c_j|^p$.
\item If $ \{a_j\}_{j \in \mathbb{J}} \in \ell^p(\mathbb{J})$, then $ \sum_{j\in \mathbb{J}}a_jx_j\in \mathcal{X}$ and $\|\sum_{j\in \mathbb{J}}a_jx_j\|^p=\sum_{j\in \mathbb{J}}|a_j|^p$.
\end{enumerate}
\end{definition}
\begin{example}
We refer Example \ref{PORTHOGONALASWELLASPORTHONORMALEXAMPLE}.
\end{example}
We note that every  orthonormal basis for every  Hilbert space satisfies Definition \ref{orthonormality in Banach}.
\begin{theorem}\label{2-ONBIMPLIESONB}
A 2-orthonormal set in a  Hilbert space $\mathcal{H}$ is an orthonormal set in  $\mathcal{H}$.
\end{theorem}
\begin{proof}
Let $ \{x_j\}_{j \in \mathbb{J}}$ be a 2-orthonormal set  in  $\mathcal{H}$.  From Proposition \ref{ellp}, $ \|x_j\|=1, \forall j \in \mathbb{J}$.  Let $j\neq k,  j, k \in \mathbb{J}$. When  $\mathcal{H}$ is  over $ \mathbb{C}$, using polarization identity, 
\begin{align*}
\langle x_j, x_k \rangle &=\frac{\|x_j+x_k\|^2-\|x_j-x_k\|^2+i\|x_j+ix_k\|^2-i\|x_j-ix_k\|^2}{4}\\
&=\frac{(|1|^2+|1|^2)-(|1|^2+|-1|^2)+i(|1|^2+|i|^2)-i(|1|^2+|-i|^2)}{4}=0.
\end{align*}
By using polarization identity for the real case we get $\langle x_j, x_k \rangle=0$ when $\mathcal{H}$ is over $ \mathbb{R}$.
\end{proof}
\begin{corollary}\label{PIMPLIESORTHONORMALFOR HILBERT}
A 2-orthonormal basis for a Hilbert space $\mathcal{H}$ is an orthonormal basis for $\mathcal{H}$.
\end{corollary}
\begin{lemma}\label{CFLEMMA}
If $ \{x_j\}_{j \in \mathbb{J}}$ is p-orthonormal for $ \mathcal{X}$, then $ f_j: \operatorname{span}\{x_j\}_{j \in \mathbb{J}}  \ni \sum_{\operatorname{finite}}a_jx_j\mapsto a_j \in \mathbb{K} , \forall j \in \mathbb{J}$ are well-defined bounded linear functionals of norm one. In particular,  each $f_j $  has unique extension to $ \overline{\operatorname{span}}\{x_j\}_{j \in \mathbb{J}}$ by preserving the norm.
\end{lemma}
\begin{proof}
Since $ \{x_j\}_{j \in \mathbb{J}}$ is linearly independent, all the $ f_j$'s are well-defined. For $x=\sum_{j \in \mathbb{S}} a_jx_j \in  \operatorname{span}\{x_j\}_{j \in \mathbb{J}}$ we get $ |f_j(x)|^p=|a_j|^p\leq\sum_{j\in \mathbb{S}}|a_j|^p=\|\sum_{j\in \mathbb{S}}a_jx_j\|^p=\|x\|^p $  for all  $j \in \mathbb{J} $. We also have $ f_j(x_j)=1$ for all  $j \in \mathbb{J} $. Hence $ \|f_j\|=1, \forall j \in \mathbb{J}$.
\end{proof}
\begin{theorem}\label{P-ORTHONORMALSEQUENCECHARACTERIZATION}
Let  $ \{x_j\}_{j \in \mathbb{J}}$ be a  p-orthonormal sequence for  $\mathcal{X}$ and $ \{a_j\}_{j \in \mathbb{J}}$ be a  sequence of scalars. Then 
$$ \sum_{j\in \mathbb{J}}a_jx_j ~ \text{converges in}~ \mathcal{X}~ \text{if and only if }~\sum_{j\in \mathbb{J}}|a_j|^p ~\text{converges}.$$
In this case, if $x=\sum_{j\in \mathbb{J}}a_jx_j\in \overline{\operatorname{span}}\{x_j\}_{j \in \mathbb{J}}$, then  $ a_j=f_j(x), \forall j \in \mathbb{J}$, where $ f_j$ is the  unique bounded linear extension of $f_j :\operatorname{span}\{x_j\}_{j \in \mathbb{J}}  \ni \sum_{\operatorname{finite}}a_jx_j\mapsto a_j \in \mathbb{K} ,\forall j \in \mathbb{J}$.
\end{theorem}
\begin{proof}
Proof comes from (iii) in Proposition \ref{ellp}. For `in this case', we act by  $f_k$ to the sum $ \sum_{j\in \mathbb{J}}a_jx_j$, for each $k \in \mathbb{J}$.
\end{proof}
\begin{theorem}\label{CHARACTERIZATIONPBASISBANACHSEQUENCE}
Let $ \{x_j\}_{j \in \mathbb{J}}$ be  a p-orthonormal sequence  for $ \mathcal{X}$ such that $ \{f_j(x)\}_{j \in \mathbb{J}} \in \ell^p(\mathbb{J}), \forall x \in \mathcal{X} $, where $ f_j $ is the unique bounded linear extension of   $f_j: \operatorname{span}\{x_j\}_{j \in \mathbb{J}}  \ni \sum_{\operatorname{finite}}a_jx_j\mapsto a_j \in \mathbb{K}  $ for each $j \in \mathbb{J}$ (Lemma \ref{CFLEMMA} gives the existence of $ f_j$'s). If $\overline{\operatorname{span}}\{x_j\}_{j \in \mathbb{J}} = \mathcal{X}$, then the following are equivalent.
\begin{enumerate}[\upshape(i)]
\item $ \{x_j\}_{j \in \mathbb{J}}$ is  a p-orthonormal basis for $ \mathcal{X}$.
\item $ x= \sum_{j \in \mathbb{J}}f_j(x)x_j, \forall x \in \mathcal{X}$, and $ \phi=\sum_{j \in \mathbb{J}}\phi(x_j)f_j, \forall \phi \in  \mathcal{X}^*$, i.e., $ \phi(x)=\sum_{j \in \mathbb{J}}\phi(x_j)f_j(x), \forall x \in \mathcal{X}, \forall \phi \in \mathcal{X}^*$. 
\item $\mathcal{X}^*=\overline{\operatorname{span}}\{f_j\}_{j \in \mathbb{J}}$, in pointwise limit, i.e., for each $ \phi \in \mathcal{X}^*$, there exists a sequence  $\{\phi_n\}_{n=1}^\infty $ in $\operatorname{span}\{f_j\}_{j \in \mathbb{J}}$ such that  $ \phi_n(x) \rightarrow \phi(x)$ as  $n \rightarrow \infty $ for all $ x \in \mathcal{X}$.
\item If $ x \in \mathcal{X}$ is such that $f_j(x)=0,\forall j \in \mathbb{J}$, then $x=0.$
\end{enumerate}	
\end{theorem}
\begin{proof}
(i) $\Rightarrow $ (ii) is clear.  For (ii) $\Rightarrow $ (iii), we observe  that for each $x $ in $ \mathcal{X}$ and for each $ \phi $ in $ \mathcal{X}^*$, the net  $\{\sum_{j\in \mathbb{S}}\phi(x_j)f_j(x):\mathbb{S} \subseteq \mathbb{J}, \mathbb{S} ~\text{is finite}\} $ converges to $\phi(x)$. For (iii) $\Rightarrow $ (iv), let $ x \in \mathcal{X}$ be such that  $f_j(x)=0, \forall j \in \mathbb{J}$. Let $ \phi \in \mathcal{X}^* $. Then there exists a sequence  $\{\phi_n\}_{n=1}^\infty $ in $\operatorname{span}\{f_j\}_{j \in \mathbb{J}}$ such that  $ \phi_n(x) \rightarrow \phi(x)$ as  $n \rightarrow \infty $. Since $\phi_n\in \operatorname{span}\{f_j\}_{j \in \mathbb{J}}, \forall n \in \mathbb{N},$ we have  $\phi_n(x)=0,\forall n \in \mathbb{N}$. Therefore $\phi(x)=\lim_{n\rightarrow \infty}\phi_n(x)=0.$ Since $ \phi \in \mathcal{X}^* $ was arbitrary,  $x=0$.
Now, for (iv) $\Rightarrow $ (i),  since $ \{f_j(x)\}_{j \in \mathbb{J}} \in \ell^p(\mathbb{J}), \forall x \in \mathcal{X} $, from Theorem  \ref{P-ORTHONORMALSEQUENCECHARACTERIZATION},  $\sum_{j \in \mathbb{J}}f_j(x) x_j $ converges in $ \mathcal{X}, \forall x \in \mathcal{X}.$ Let $x \in \mathcal{X} $ be fixed.  Define $y\coloneqq\sum_{j \in \mathbb{J}}f_j(x) x_j.$ Then $ f_j (y-x)=\sum_{k \in \mathbb{J}}f_k(x)f_j(x_k)-f_j(x) =f_j(x)-f_j(x)=0, \forall j \in \mathbb{J}.$ Therefore (iv) gives $ y=x.$ If $ x$ also has representation $ \sum_{j \in \mathbb{J}}a_j x_j,$ for some $ a_j \in \mathbb{K}, j\in \mathbb{J}$, then $ f_k(x)= f_k(\sum_{j \in \mathbb{J}}a_j x_j)= a_k, \forall k \in \mathbb{J}.$  Thus $ \{x_j\}_{j \in \mathbb{J}} $ is a p-orthonormal basis for $\mathcal{X}.$	
\end{proof}
\begin{remark}
In Theorem \ref{CHARACTERIZATIONPBASISBANACHSEQUENCE},  only in proving $\operatorname{(iv)}$ $ \Rightarrow $ $\operatorname{(i)}$ we used $ \{f_j(x)\}_{j \in \mathbb{J}} \in \ell^p(\mathbb{J}), \forall x \in \mathcal{X} $. Thus even without assumption $ \{f_j(x)\}_{j \in \mathbb{J}} \in \ell^p(\mathbb{J}), \forall x \in \mathcal{X} $  in Theorem \ref{CHARACTERIZATIONPBASISBANACHSEQUENCE} we get $\operatorname{(i)}$  $ \Rightarrow $ $\operatorname{(ii)}$ $ \Rightarrow $ $\operatorname{(iii)}$ $ \Rightarrow $ $\operatorname{(iv)}$.
\end{remark}

\begin{theorem}
Let $ \{x_j\}_{j \in \mathbb{J}}$ be a p-orthonormal basis for $\mathcal{X}$ and $ f_j : \mathcal{X} \ni \sum_{k\in\mathbb{J}}a_kx_k \mapsto a_j \in  \mathbb{K}, \forall j \in \mathbb{J}$. Then for each $ x \in \mathcal{X}$, the set $ Y_x\coloneqq \{f_j: f_j(x)\neq 0, j \in \mathbb{J}\}$ is either  finite or countable.
\end{theorem}
\begin{proof}
For $ n \in \mathbb{N}$, define 
$$ Y_{n,x}\coloneqq \left\{f_j: |f_j(x)|^p> \frac{\|x\|^p}{n},j\in \mathbb{J}\right\}.$$

Suppose, for some $n$, $ Y_{n,x}$ has more than $n-1$ elements, say $f_1,...,f_n$. Then $ \sum_{j=1 }^n|f_j(x)|^p>n\frac{\|x\|^p}{n}=\|x\|^p$. We already have $\sum_{j=1}^n|f_j(x)|^p \leq \|x\|^p .$ This gives $ \|x\|^p< \|x\|^p $ which is impossible. Hence $ \operatorname{Card}(Y_{n,x})\leq n-1$. We next note 
 $Y_x=\cup_{n=1}^\infty Y_{n,x}$, being a countable union of finite sets is finite or countable.
\end{proof}
\begin{theorem}
\begin{enumerate}[\upshape(i)]	
\item If $\mathcal{X}$ has a p-orthonormal basis which is also a Hamel basis for $\mathcal{X}$,  then $\mathcal{X}$ is finite dimensional.
\item An infinite p-orthonormal  basis for $\mathcal{X}$ can not be a Hamel basis  for $\mathcal{X}$.
\end{enumerate}
\end{theorem}
\begin{proof}
\begin{enumerate}[\upshape(i)]	
\item Let $ \{x_j\}_{j \in \mathbb{J}}$ be a p-orthonormal basis as well as a Hamel basis  for $\mathcal{X}$,  $f_j:\mathcal{X}  \ni  \sum_{k \in \mathbb{J}}a_kx_k \mapsto a_j \in \mathbb{K}$, $\forall j \in \mathbb{J}$.  Since  $\sum_{j\in \mathbb{J}} \frac{1}{j^{p^2+p}}$ is a convergent series, from Theorem \ref{P-ORTHONORMALSEQUENCECHARACTERIZATION}, $\sum_{j\in \mathbb{J}} \frac{x_j}{j^{p+1}}$ converges in  $\mathcal{X}$, say to $x$. Since  $ \{x_j\}_{j\in \mathbb{J}}$ is a Hamel basis for $\mathcal{X}$, there exist $x_{j_1},...,x_{j_m} $ in $ \{x_j\}_{j \in \mathbb{J}}$ such that  $x=a_{j_1} x_{j_1}+\cdots +a_{j_m}x_{j_m}, a_{j_1},...,a_{j_m} \in \mathbb{K} $, uniquely. 
Suppose dimension of  $\mathcal{X}$ is infinite.  Then $ \{x_j\}_{j \in \mathbb{J}}$ is infinite. So there exists a $k\in \mathbb{J}$ such that $x_k\neq x_{j_l},1\leq l\leq m $.
 Then 
$$0=f_k(a_{j_1} x_{j_1}+\cdots +a_{j_m}x_{j_m}) =f_k(x)=f_k\left(\sum_{j\in \mathbb{J}} \frac{x_j}{j^{p+1}}\right)=f_k\left(\frac{x_k}{k^{p+1}}\right)=\frac{1}{k^{p+1}},$$
 which is a contradiction. Hence $\mathcal{X}$ is finite dimensional.
\item Similar to the proof of (i).
\end{enumerate}	
\end{proof}

\begin{theorem}
Let $ \{x_j\}_{j \in \mathbb{J}}$ be a p-orthonormal basis for $\mathcal{X}$.
\begin{enumerate}[\upshape(i)]
\item If $ T:\mathcal{X}\rightarrow \mathcal{X}_0 $ is  an isometric isomorphism, then  $ \{Tx_j\}_{j \in \mathbb{J}}$ is a p-orthonormal basis for $\mathcal{X}_0$.
\item If $ \{y_j\}_{j \in \mathbb{J}}$ is a p-orthonormal basis for $\mathcal{X}_0$, then the map $T:\mathcal{X} \ni \sum_{j\in \mathbb{J}}a_jx_j\mapsto \sum_{j\in \mathbb{J}}a_jy_j \in \mathcal{X}_0 $ is a well-defined isometric isomorphism.
\end{enumerate}
\end{theorem}
\begin{proof}
\begin{enumerate}[\upshape(i)]
\item Since $ T$ is bounded invertible, $ \{Tx_j\}_{j \in \mathbb{J}}$  is a Scahauder basis for $\mathcal{X}_0$. If $ \mathbb{L}\subseteq \mathbb{J}$ and  $ \{c_j\}_{j \in \mathbb{L}}$ is a sequence of scalars such that $ \sum_{j\in \mathbb{L}}c_jTx_j\in \mathcal{X}$, then  $ \|\sum_{j\in \mathbb{L}}c_jTx_j\|^p=\|T(\sum_{j\in \mathbb{L}}c_jx_j)\|^p =\|\sum_{j\in \mathbb{L}}c_jx_j\|^p =\sum_{j\in \mathbb{L}}|c_j|^p$. Also, if  $ \{a_j\}_{j \in \mathbb{J}} \in \ell^p(\mathbb{J})$, then $ \sum_{j\in \mathbb{J}}a_jTx_j\in \mathcal{X}_0$ and $\|\sum_{j\in \mathbb{J}}a_jTx_j\|^p=\sum_{j\in \mathbb{J}}|a_j|^p$.
\item  From Theorem \ref{P-ORTHONORMALSEQUENCECHARACTERIZATION}, $T$ is well-defined. Again, Theorem \ref{P-ORTHONORMALSEQUENCECHARACTERIZATION} says that $T$ is surjective. Orthonormality of $ \{y_j\}_{j \in \mathbb{J}}$ gives the injectiveness of $T$. For $x=\sum_{j\in \mathbb{J}}a_jx_j \in\mathcal{X},$  $ \|Tx\|^p=\|\sum_{j\in \mathbb{J}}a_jy_j\|^p=\sum_{j\in \mathbb{J}}|a_j|^p=\|x\|^p$. Thus $T$ is isometry.
\end{enumerate}
\end{proof}
One can see the  following inequality, we call 4-inequality, which  looks like Cauchy-Schwarz inequality (reason: if  $\mathcal{H}$ is a real Hilbert space, then  $ (\|x+y\|^2- \|x-y\|^2)/4=\langle x, y\rangle \leq \|x\|\|y\|, \forall x, y \in \mathcal{H}$).
\begin{theorem}\label{4INEQUALITY}
(4-inequality) If $\mathcal{X}$ has a 4-orthonormal basis, then 
$$\frac{\|x+y\|^4- \|x-y\|^4}{8}\leq (\|x\|^2+\|y\|^2)\|x\|\|y\| , ~\forall x ,y \in \mathcal{X}.$$
\end{theorem}
\begin{proof}
Starting with a 4-orthonormal basis $ \{x_j\}_{j \in \mathbb{J}}$ for  $\mathcal{X}$, and $x ,y \in \mathcal{X}$, we can write  $ x=\sum_{j\in \mathbb{J}}a_jx_j, y=\sum_{j\in \mathbb{J}}b_jx_j $ for unique $  \{a_j\}_{j \in \mathbb{J}},  \{b_j\}_{j \in \mathbb{J}} \in \ell^4(\mathbb{J})$. Then
 
\begin{align*}
\|x+y\|^4- \|x-y\|^4&=\left\|\sum_{j\in \mathbb{J}}(a_j+b_j)x_j\right\|^4-\left\|\sum_{j\in \mathbb{J}}(a_j-b_j)x_j\right\|^4
 =\sum_{j\in \mathbb{J}}|a_j+b_j|^4-\sum_{j\in \mathbb{J}}|a_j-b_j|^4\\
& =4\sum_{j\in \mathbb{J}}(|a_j|^2+|b_j|^2)(a_j\bar{b_j}+\bar{a_j}b_j)
 \leq 8\sum_{j\in \mathbb{J}}(|a_j|^2+|b_j|^2)|a_jb_j|\\
 &\leq 8\left(\sum_{j\in \mathbb{J}}(|a_j|^2+|b_j|^2)^2\right)^\frac{1}{2}\left(\sum_{j\in \mathbb{J}}|a_jb_j|^2\right)^\frac{1}{2}\\
& \leq 8\left(\sum_{j\in\mathbb{J}}(|a_j|^2+|b_j|^2)^2\right)^\frac{1}{2}\left(\sum_{j\in \mathbb{J}}|a_j|^4\right)^\frac{1}{4} \left(\sum_{j\in \mathbb{J}}|b_j|^4\right)^\frac{1}{4}\\
 &=8\left(\sum_{j\in \mathbb{J}}|a_j|^4+\sum_{j\in \mathbb{J}}|b_j|^4+2\sum_{j\in \mathbb{J}}|a_jb_j|^2\right)^\frac{1}{2}\|x\|\|y\|\\
 &=8\left(\|x\|^4+\|y\|^4+2\sum_{j\in \mathbb{J}}|a_jb_j|^2\right)^\frac{1}{2}\|x\|\|y\|\\
 &\leq 8\left(\|x\|^4+\|y\|^4+2\left(\sum_{j\in \mathbb{J}}|a_j|^4\right)^\frac{1}{2}\left(\sum_{j\in \mathbb{J}}|b_j|^4\right)^\frac{1}{2}\right)^\frac{1}{2}\|x\|\|y\|\\
 &=8(\|x\|^4+\|y\|^4+2\|x\|^2\|y\|^2)^\frac{1}{2}\|x\|\|y\|=8(\|x\|^2+\|y\|^2)\|x\|\|y\|.
\end{align*}
\end{proof}
\begin{theorem}\label{SPECIALPL}
(4-parallelogram law)
If $\mathcal{X}$ has a 4-orthonormal basis, then
\begin{align*} 
 \|x+y\|^4+ \|x-y\|^4&\leq2(\|x\|^4+\|y\|^4)+12\|x\|^2y\|^2\\
 &=2((\|x\|^2+\|y\|^2)^2+4\|x\|^2\|y\|^2) ,~\forall x ,y \in \mathcal{X}.
\end{align*}
\end{theorem}
\begin{proof}
Let $ \{x_j\}_{j \in \mathbb{J}}$ be a 4-orthonormal  basis for $\mathcal{X}$, and $x ,y \in \mathcal{X}.$ Then $ x=\sum_{j\in \mathbb{J}}a_jx_j, y=\sum_{j\in \mathbb{J}}b_jx_j $ for unique $  \{a_j\}_{j \in \mathbb{J}},  \{b_j\}_{j \in \mathbb{J}} \in \ell^4(\mathbb{J})$. Consider

\begin{align*}
 \|x+y\|^4+ \|x-y\|^4&=\left\|\sum_{j\in \mathbb{J}}(a_j+b_j)x_j\right\|^4+\left\|\sum_{j\in \mathbb{J}}(a_j-b_j)x_j\right\|^4
 =\sum_{j\in \mathbb{J}}|a_j+b_j|^4+\sum_{j\in \mathbb{J}}|a_j-b_j|^4\\
 &=2\sum_{j\in \mathbb{J}}|a_j|^4+2\sum_{j\in \mathbb{J}}|b_j|^4+2\sum_{j\in \mathbb{J}}((\bar{a_j}b_j)^2+(a_j\bar{b_j})^2)+8\sum_{j\in \mathbb{J}}|a_jb_j|^2\\
 &=2\|x\|^4+2\|y\|^4+2\sum_{j\in \mathbb{J}}((\bar{a_j}b_j)^2+(a_j\bar{b_j})^2)+8\sum_{j\in \mathbb{J}}|a_jb_j|^2\\
 &\leq 2(\|x\|^4+\|y\|^4)+4\sum_{j\in \mathbb{J}}|\bar{a_j}b_j|^2+8\sum_{j\in \mathbb{J}}|a_jb_j|^2
=2(\|x\|^4+\|y\|^4)+12\sum_{j\in \mathbb{J}}|a_jb_j|^2\\
&\leq 2(\|x\|^4+\|y\|^4)+12\left(\sum_{j\in \mathbb{J}}|a_j|^4\right)^\frac{1}{2}\left(\sum_{j\in \mathbb{J}}|b_j|^4\right)^\frac{1}{2}
=2(\|x\|^4+\|y\|^4)+12\|x\|^2y\|^2.
\end{align*}
\end{proof}
\begin{theorem}\label{4PROJECTIONTHEOREM}
(4-projection theorem) If $ Y$ is a closed convex subset of   $\mathcal{X}$ which has a  4-orthonormal basis, then for every $ x \in \mathcal{X},$ there exists unique $ y \in Y$  such that 
$$\|x-y\|=\inf_{z\in Y} \|x-z\|.$$
\end{theorem}
\begin{proof}
Let $ \{x_n\}_{n=1}^\infty$ be a sequence in $ Y$ such that 
$\lim_{n\rightarrow\infty}\|x_n-x\| =\inf_{z\in Y} \|x-z\|.$
Let $ d=\inf_{z\in Y} \|x-z\|$. Using the convexity of $Y$, we see that 
$$\left\|\frac{x_n+x_m}{2}-x\right\| \geq d , \forall n , m \in \mathbb{N},~\text{i.e.,}~ \left\|x_n-x_m-2x\right\| \geq 2d,~ \forall n , m \in \mathbb{N}.$$
By using 4-parallelogram law (Theorem \ref{SPECIALPL}) we get 
\begin{align*}
\|x_n-x_m\|^4&=\|(x_n-x)-(x_m-x)\|^4\\
&\leq 2\|x_n-x\|^4+2\|x_m-x\|^4+12\|x_n-x\|^2\|x_m-x\|^2-\|(x_n-x)+(x_m-x)\|^4\\
&\leq 2\|x_n-x\|^4+2\|x_m-x\|^4+12\|x_n-x\|^2\|x_m-x\|^2-16d^4, ~\forall n , m \in \mathbb{N}.
\end{align*}
Therefore $\|x_n-x_m\|^4 \rightarrow2d^4+2d^4+12d^2d^2-16d^4 =0 $ as $ n,m \rightarrow \infty$. Let $ y=\lim_{n\rightarrow \infty}x_n$. Since $ Y$ is closed, $ y \in Y$. Now $ \|x-y\|=\|x-\lim_{n\rightarrow \infty}x_n\|=\lim_{n\rightarrow \infty}\|x-x_n\|=d$. Let  $y^* \in Y$ be such that $\|x-y^*\|=\inf_{z\in Y} \|x-z\|.$ We again use the convexity of $ Y$ to get $(y+y^*)/2 \in Y $ and hence $ \|x-(y+y^*)/2\|\geq d $.  Then 
\begin{align*}
\|y-y^*\|^4&=\|(y-x)-(y^*-x)\|^4\\
&\leq 2\|y-x\|^4+2\|y^*-x\|^4+12\|y-x\|^2\|y^*-x\|^2-\|(y-x)+(y^*-x)\|^4\\
&= 2d^4+2d^4+12d^2d^2-\|y+y^*-2x\|^4\leq 16d^4-16d^4=0.
\end{align*} 
Therefore $ y=y^*$.
\end{proof}
\begin{corollary}
If $ Y$ is a closed  subspace  of   $\mathcal{X}$ which has a  4-orthonormal basis, then for every $ x \in \mathcal{X},$ there exists unique $ y \in Y$  such that $\|x-y\|=\inf_{z\in Y} \|x-z\|.$
\end{corollary}
\begin{remark}
Since $ \ell^4(\mathbb{J})$ has a 4-orthonormal basis (the standard Schauder basis $ \{e_j\}_{j \in \mathbb{J}}$  is a 4-orthonormal basis for $ \ell^4(\mathbb{J})$), 4-inequality, 4-parallelogram law, and 4-projection theorem hold in $ \ell^4(\mathbb{J})$.
\end{remark}
\begin{proposition}
If $\mathcal{X}$ has a 4-orthonormal basis, then $\mathcal{X}$ is uniformly convex.
\end{proposition}
\begin{proof}
Let  $0< \epsilon\leq 2$. Let $ x, y \in\mathcal{X} $ such that $ \|x\|= 1, \|y\|= 1$ and $ \|x-y\|\geq \epsilon $. Take  $ 0<\delta\leq (1-\epsilon^4/16)^{1/4} $. Then by  the 4-parallelogram law,
\begin{align*}
 \left\|\frac{x+y}{2}\right\|^4&\leq 2\left\|\frac{x}{2}\right\|^4+2\left\|\frac{x}{2}\right\|^4+12\left\|\frac{x}{2}\right\|^2 \left\|\frac{y}{2}\right\|^2-\left\|\frac{x-y}{2}\right\|^4\leq \frac{1}{8}+\frac{1}{8}+\frac{3}{4}-\frac{\epsilon^4}{16}\\
 &=1-\frac{\epsilon^4}{16}\leq (1-\delta)^4.
\end{align*} 
\end{proof}
\begin{theorem}
\begin{enumerate}[\upshape(i)]
\item Every p-orthonormal set $Y$ for $\mathcal{X}$ is contained in a maximal p-orthonormal set.
\item Every $\mathcal{X}$ has a maximal p-orthonormal set.
\end{enumerate}
\end{theorem}
\begin{proof}
(i) comes from the use of Zorn's lemma and (ii) comes from (i) by considering a nonzero $x \in \mathcal{X}$ and $Y=\{\|x\|^{-1}x\}$.
\end{proof}
\begin{theorem}
If $\mathcal{X}$ has a p-orthonormal basis $ \{x_j\}_{j \in \mathbb{J}}$,  then $\mathcal{X}$ is isometrically isomorphic to $\ell^p(\mathbb{J})$.
\end{theorem}
\begin{proof}
 Let $ \{e_j\}_{j \in \mathbb{J}}$	 be the standard Schauder basis for $\ell^p(\mathbb{J})$. Define $ U : \mathcal{X} \ni \sum_{j \in \mathbb{J}}a_jx_j\mapsto \sum_{j \in \mathbb{J}}a_je_j \in \ell^p(\mathbb{J})$. From Proposition \ref{ellp}, $ U$ is well-defined. Clearly it is linear. For $ x=\sum_{j \in \mathbb{J}}a_jx_j \in \mathcal{X}$, $ \|Ux\|^p=\|\sum_{j \in \mathbb{J}}a_je_j\|^p=\sum_{j \in \mathbb{J}}|a_j|^p=\|x\|^p$. Thus $ U$ is isometry. Let $ \sum_{j \in \mathbb{J}}a_je_j \in\ell^p(\mathbb{J}).$ From condition (iii) of Definition \ref{orthonormality in Banach}, we get $\sum_{j \in \mathbb{J}}a_jx_j \in \mathcal{X}$. Now it is transparent that $ U(\sum_{j \in \mathbb{J}}a_jx_j)=\sum_{j \in \mathbb{J}}a_je_j$.
\end{proof}
\begin{theorem}\label{P-ORTHONORMALBASISCHARACTERIZATION}
Let  $ \{x_j\}_{j \in \mathbb{J}}$ be a p-orthonormal basis for  $\mathcal{X}$. Then the p-orthonormal bases for $\mathcal{X}$ are precisely  $ \{Ux_j\}_{j \in \mathbb{J}}$, where $ U: \mathcal{X}\rightarrow \mathcal{X}$ is invertible isometry. 
\end{theorem}
\begin{proof}
$(\Leftarrow)$ Since $ U$ is invertible, clearly  $ \{Ux_j\}_{j \in \mathbb{J}}$ is a Schauder basis for $\mathcal{X}$. Let $ \{c_j\}_{j \in \mathbb{L}}$ be any sequence of scalars such that $ \sum_{j\in \mathbb{L}}c_jUx_j\in \mathcal{X}$. Then, linearity and continuity of $U ^{-1}$ gives $ \sum_{j\in \mathbb{L}}c_jx_j\in \mathcal{X}$. But then  $ \|\sum_{j\in \mathbb{L}}c_jUx_j\|^p = \|U(\sum_{j\in \mathbb{L}}c_jx_j)\|^p=\|\sum_{j\in \mathbb{L}}c_jx_j\|^p=\sum_{j\in \mathbb{J}}|c_j|^p$. Now,  let 
$ \{a_j\}_{j \in \mathbb{J}} \in \ell^p(\mathbb{J})$. Then we have $ \sum_{j\in \mathbb{J}}a_jx_j\in \mathcal{X}$ and $\|\sum_{j\in \mathbb{J}}a_jx_j\|^p=\sum_{j\in \mathbb{J}}|a_j|^p$ $ \Rightarrow $  $ \sum_{j\in \mathbb{J}}a_jUx_j=U(\sum_{j\in \mathbb{J}}a_jx_j)\in \mathcal{X}$ and $\|\sum_{j\in \mathbb{J}}a_jUx_j\|^p=\|U(\sum_{j\in \mathbb{J}}a_jx_j)\|^p=\|\sum_{j\in \mathbb{J}}a_jx_j\|^p=\sum_{j\in \mathbb{J}}|a_j|^p$. 

$(\Rightarrow)$	Let  $ \{y_j\}_{j \in \mathbb{J}}$ be an arbitrary  p-orthonormal basis for $\mathcal{X}$. Define $V: \mathcal{X} \ni \sum_{j \in \mathbb{J}}a_jx_j\mapsto \sum_{j \in \mathbb{J}}a_jy_j\in \mathcal{X}$. From Proposition \ref{ellp} and condition (iii) of Definiton \ref{orthonormality in Banach} we get the well-definedness of $V$. Then $ y_j=Vx_j, \forall j \in \mathbb{J}$ and  for $ x=\sum_{j \in \mathbb{J}}a_jx_j \in \mathcal{X}$, $ \|Vx\|^p=\|\sum_{j \in \mathbb{J}}a_jy_j\|^p=\sum_{j \in \mathbb{J}}|a_j|^p=\|x\|^p$, so $ U$ is an isometry. One can easily check that the inverse of $V$ is $ V^{-1}: \mathcal{X} \ni \sum_{j \in \mathbb{J}}b_jy_j\mapsto \sum_{j \in \mathbb{J}}b_jx_j\in \mathcal{X}$.
\end{proof}
\begin{definition}
A sequence $\{g_j\}_{j\in \mathbb{J}} $ in $\mathcal{X}^*$ is said to be a Riesz p-basis  w.r.t. a sequence $\{\omega_j\}_{j\in \mathbb{J}} $ in $\mathcal{X}$  if   there exist a p-orthonormal basis $ \{\tau_j\}_{j \in \mathbb{J}}$ for  $ \mathcal{X}$ and bounded invertible operators $ U,V : \mathcal{X} \rightarrow \mathcal{X}$ with resolvent of $ VU$ contains $(-\infty, 0] $ such that   $ \omega_j=V\tau_j, \forall j \in \mathbb{J}$ and $g_j=f_jU, \forall j \in \mathbb{J} $, where $ f_j :\mathcal{X} \ni \sum_{k \in \mathbb{J}}a_k\tau_k \mapsto a_j \in \mathbb{K} , \forall j \in \mathbb{J}$. We write $(\{g_j\}_{j\in \mathbb{J}},\{\omega_j\}_{j\in \mathbb{J}} )$ is a Riesz p-basis for $\mathcal{X}$.
\end{definition}
\begin{remark}
Definition of Riesz p-basis demands existence  of a  p-orthonormal basis, whereas that of p-frame doesn't.
\end{remark}

\begin{theorem}
If  $(\{g_j=f_jU\}_{j\in \mathbb{J}},\{\omega_j=V\tau_j\}_{j \in \mathbb{J}})$ is a Riesz p-basis for $ \mathcal{X}$, then there exist a unique sequence $\{h_j\}_{j\in \mathbb{J}} $ in $\mathcal{X}^*$, and a unique sequence $\{\rho_j\}_{j\in \mathbb{J}} $ in $\mathcal{X}$ such that 
$$ x=\sum_{j\in\mathbb{J}}h_j(x)\omega_j=\sum_{j\in\mathbb{J}}g_j(x)\rho_j, \forall x \in \mathcal{X}, \text{and}~ \{h_j(x)\}_{j\in \mathbb{J}}  \in \ell^p(\mathbb{J}),~\forall x \in \mathcal{X}.$$
Moreover, $(\{h_j\}_{j\in \mathbb{J}},\{\rho_j\}_{j \in \mathbb{J}})$ is a Riesz p-basis. 
\end{theorem}
\begin{proof}
Define $ h_j\coloneqq f_jV^{-1}$, $ \rho_j \coloneqq U^{-1}\tau_j,\forall j \in \mathbb{J}$. Then $ \sum_{j\in\mathbb{J}}h_j(x)\omega_j=\sum_{j\in\mathbb{J}}f_j(V^{-1}x)V\tau_j=V(\sum_{j\in\mathbb{J}}f_j(V^{-1}x)\tau_j)=V(V^{-1}x)=x$, $\sum_{j\in\mathbb{J}}g_j(x)\rho_j=\sum_{j\in\mathbb{J}}f_j(Ux)U^{-1}\tau_j=U^{-1}(\sum_{j\in\mathbb{J}}f_j(Ux)\tau_j)=U^{-1}(Ux)=x,\forall x \in \mathcal{X}$. Now resolvent of $ U^{-1}V^{-1}=(VU)^{-1}$ contains $(-\infty, 0] $. Hence $(\{h_j\}_{j\in \mathbb{J}},\{\rho_j\}_{j \in \mathbb{J}})$ is a Riesz p-basis. Clearly $\{h_j(x)\}_{j\in \mathbb{J}}  \in \ell^p(\mathbb{J}),\forall x \in \mathcal{X}$. Suppose there are  $\{\phi_j\}_{j\in \mathbb{J}} $ in $\mathcal{X}^*$, and  $\{x_j\}_{j\in \mathbb{J}} $ in $\mathcal{X}$ such that $x=\sum_{j\in\mathbb{J}}\phi_j(x)\omega_j=\sum_{j\in\mathbb{J}}g_j(x)x_j, \forall x \in \mathcal{X}.$ Then $x=\sum_{j\in\mathbb{J}}\phi_j(x)V\tau_j=\sum_{j\in\mathbb{J}}f_j(Ux)x_j, \forall x \in \mathcal{X}.$ First equality gives $ V^{-1}x=\sum_{j\in\mathbb{J}}\phi_j(x)\tau_j,\forall x \in \mathcal{X}$. By applying  $f_k$ on this  gives $f_k(V^{-1}x)=\phi_k(x),\forall x \in \mathcal{X},\forall k \in \mathbb{J}$. Therefore $ \phi_j=h_j, \forall j \in \mathbb{J}$. Since $U$ is invertible, from  $ x=\sum_{j\in\mathbb{J}}f_j(Ux)x_j, \forall x \in \mathcal{X}$ we get $ U^{-1}y=\sum_{j\in\mathbb{J}}f_j(y)x_j, \forall y \in \mathcal{X}$. We evaluate this at $\tau_k$ to get $U^{-1}\tau_k=x_k, \forall k \in \mathbb{J}$. Hence $ x_j=\rho_j, \forall j \in \mathbb{J}$.
\end{proof} 
\begin{proposition}
If $(\{g_j\}_{j\in \mathbb{J}},\{\omega_j\}_{j \in \mathbb{J}})$ is a Riesz p-basis for $ \mathcal{X}$, then
\begin{enumerate}[\upshape(i)]
\item $\{\omega_j\}_{j \in \mathbb{J}}$ is complete in  $ \mathcal{X}$, and there exist $a, b>0$ such that for every finite subset $ \mathbb{S}\subseteq\mathbb{J}$, 
$$ a\sum_{j\in \mathbb{S}}|c_j|^p\leq \left\|\sum_{j\in \mathbb{S}}c_j\omega_j \right\|^p\leq b\sum_{j\in \mathbb{S}}|c_j|^p,~\forall  c_j \in \mathbb{K}, \forall j \in \mathbb{S}, \text{and}$$
\item the collection of functionals $ \{g_j\}_{j\in \mathbb{J}}$ is uniformly bounded (in the operator-norm) in $ \mathcal{X}^*$.
\end{enumerate}
\end{proposition}
\begin{proof}
Let  $ \{\tau_j\}_{j \in \mathbb{J}}$ be a p-orthonormal basis for   $ \mathcal{X}$ and  $ U,V : \mathcal{X} \rightarrow \mathcal{X}$ be  bounded invertible with resolvent of $ VU$ contains $(-\infty, 0] $ such that   $ \omega_j=V\tau_j, \forall j \in \mathbb{J}$ and $g_j=f_jU, \forall j \in \mathbb{J} $, where $ f_j :\mathcal{X} \ni \sum_{k \in \mathbb{J}}a_k\tau_k \mapsto a_j \in \mathbb{K} , \forall j \in \mathbb{J}$.
\begin{enumerate}[\upshape(i)]
\item Since $V$ is invertible, $\{\omega_j\}_{j \in \mathbb{J}}$ is complete in  $ \mathcal{X}$. We find
\begin{align*}
\frac{1}{\|V^{-1}\|^p}\sum_{j\in \mathbb{S}}|c_j|^p&=\frac{1}{\|V^{-1}\|^p}\left\| \sum_{j\in \mathbb{S}}c_j\tau_j\right\|^p\leq\left\| \sum_{j\in \mathbb{S}}c_jV\tau_j\right\|^p=\left\| \sum_{j\in \mathbb{S}}c_j\omega_j\right\|^p\\
&\leq \|V\|^p\left\| \sum_{j\in \mathbb{S}}c_j\tau_j\right\|^p= \|V\|^p\sum_{j\in \mathbb{S}}|c_j|^p, ~\forall c_j \in \mathbb{K}, \forall j \in \mathbb{S}.
\end{align*}
\item We first see $ |f_j(x)|^p\leq \sum_{k\in \mathbb{J}}|f_k(x)|^p=\|\sum_{k\in \mathbb{J}}f_k(x)\tau_k\|^p=\|x\|^p ,\forall x \in \mathcal{X},  f_j(\tau_j) =1, \forall j \in \mathbb{J}$. Therefore $\|f_j\|=1, \forall j \in \mathbb{J}$ and hence  $\sup\{\|g_j\|\}_{j \in \mathbb{J}}\leq\sup\{\|f_j\|\|U\|\}_{j \in \mathbb{J}}=\|U\|$.
\end{enumerate}	
\end{proof}

\begin{theorem}
\begin{enumerate}[\upshape(i)]
\item If $ \{\tau_j\}_{j \in \mathbb{J}}$  is a p-orthonormal basis for  $ \mathcal{X}$, then $(\{f_j\}_{j\in \mathbb{J}},\{\tau_j\}_{j \in \mathbb{J}})$ is p-frame for  $ \mathcal{X}$, where $ f_j :\mathcal{X} \ni \sum_{k \in \mathbb{J}}a_k\tau_k \mapsto a_j \in \mathbb{K} , \forall j \in \mathbb{J}$. 
\item A Riesz p-basis $(\{g_j=f_jU\}_{j\in \mathbb{J}},\{\omega_j=V\tau_j\}_{j \in \mathbb{J}})$ for $ \mathcal{X}$ is a p-frame for $ \mathcal{X}$ with optimal bounds $\|(VU)^{-1/p}\|^{-p} $ and $ \|(VU)^{1/p}\|^p$.
\end{enumerate}
\end{theorem}
\begin{proof}
\begin{enumerate}[\upshape(i)]
\item Second condition in Definition \ref{orthonormality in Banach} says $ \theta_f$ exists and is an isometry, whereas  third condition says $ \widehat{\theta}_\tau$ exists and is an isometry. Now the p-frame operator  $ \widehat{S}_{f,\tau}x=\sum_{j \in \mathbb{J}}f_j(x)\tau_j=I_\mathcal{X}x, \forall x \in \mathcal{X}$.
\item For $ x=\sum_{j \in \mathbb{J}}f_j(x)\tau_j \in \mathcal{X}$ we see $ \theta_g(x)=\{(f_jU)(x)\}_{j\in \mathbb{J}}=\theta_f(Ux)\in \ell^p(\mathbb{J})$, and for $ \{a_j\}_{j\in \mathbb{J}}\in \ell^p(\mathbb{J})$ we see $ \widehat{\theta}_\omega(\{a_j\}_{j\in \mathbb{J}})=\sum_{j \in \mathbb{J}}a_j\omega_j=\sum_{j \in \mathbb{J}}a_jV\tau_j=V\widehat{\theta}\tau(\{a_j\}_{j\in \mathbb{J}})$. Also $ \widehat{S}_{g,\omega}x=\sum_{j \in \mathbb{J}}g_j(x)\omega_j=\sum_{j \in \mathbb{J}}f_j(Ux)V\tau_j=V(\sum_{j \in \mathbb{J}}f_j(Ux)\tau_j)=V(Ux)$. Therefore $ (\{g_j\}_{j\in \mathbb{J}},\{\omega_j\}_{j\in \mathbb{J}})$ is a p-frame; for optimal bounds we note that $ VU$  invertible and $(VU)^{-1/p}$, $ (VU)^{1/p} $ exist (from Definition \ref{KOMATSU}).
\end{enumerate}
\end{proof}

\begin{proposition}
For every $ \{f_j\}_{j \in \mathbb{J}}  \in \widehat{\mathscr{F}}_{\tau,p}$,
\begin{enumerate}[\upshape (i)]
\item $ \widehat{S}_{f, \tau} = \widehat{\theta}_\tau\theta_f .$  
\item $ (\{f_j\}_{j \in \mathbb{J}}, \{\tau_j \}_ {j \in \mathbb{J}})$ is Parseval if and only if $  \widehat{\theta}_\tau\theta_f =I_\mathcal{X}.$ 
\item $ (\{f_j\}_{j \in \mathbb{J}}, \{\tau_j \}_ {j \in \mathbb{J}})$ is Parseval  if and only if $ \theta_f\widehat{\theta}_\tau $ is idempotent.
\item $\theta_f\widehat{S}_{f,\tau}^{-1}\widehat{\theta}_\tau$ is idempotent.
\item $\theta_f $ is  injective whose range is closed.
\item  $\widehat{\theta}_\tau $ is  surjective.
\end{enumerate}
\end{proposition}
\begin{proof}
(i) is direct verification,  and   (ii) comes from that. Arguments for the remainings are similar to the proof of Proposition \ref{2.2}.
\end{proof}
The idempotent operator $\widehat{P}_{f, \tau}\coloneqq \theta_f\widehat{S}_{f,\tau}^{-1}\widehat{\theta}_\tau$ is called as the \textit{frame idempotent} for $(\{f_j\}_{j \in \mathbb{J}}, \{\tau_j\}_ {j \in \mathbb{J}}).$
\begin{definition}
A p-frame $ (\{f_j\}_{j\in \mathbb{J}}, \{\tau_j\}_{j\in \mathbb{J}}) $ for  $ \mathcal{X}$ is said to be a Riesz  p-frame   if $ \widehat{P}_{f,\tau}= I_{\ell^p(\mathbb{J})}$. A Parseval and  Riesz p-frame (i.e., $\widehat{\theta}_\tau\theta_x=I_\mathcal{X} $ and  $\theta_f\widehat{\theta}_\tau=I_{\ell^p(\mathbb{J})}  $) is called as an orthonormal p-frame.
\end{definition}
\begin{proposition}
A p-frame $ (\{f_j\}_{j\in \mathbb{J}}, \{\tau_j\}_{j\in \mathbb{J}}) $ for  $ \mathcal{X}$ is a Riesz p-frame  if and only if  $\theta_f(\mathcal{X})=\ell^p(\mathbb{J}) .$ 
\end{proposition}
\begin{definition}
A p-frame $ (\{g_j\}_{j\in \mathbb{J}}, \{\omega_j\}_{j\in \mathbb{J}})$  for $\mathcal{X}$ is said to be a dual of a p-frame $ (\{f_j\}_{j\in \mathbb{J}}, \{\tau_j\}_{j\in \mathbb{J}})$  for $\mathcal{X}$ if  $ \widehat{\theta}_\omega\theta_f=\widehat{\theta}_\tau\theta_g=I_\mathcal{X}$. The `p-frame' $ (\{\widetilde{f}_j\coloneqq f_j \widehat{S}_{f,\tau}^{-1}\}_{j\in \mathbb{J}}, \{\widetilde{\tau}_j\coloneqq \widehat{S}_{f,\tau}^{-1}\tau_j\}_{j\in \mathbb{J}})$ for $\mathcal{X}$ which is a `dual' of  $ (\{f_j\}_{j\in \mathbb{J}}, \{\tau_j\}_{j\in \mathbb{J}})$ is called as the  canonical dual of   $ (\{f_j\}_{j\in \mathbb{J}}, \{\tau_j\}_{j\in \mathbb{J}})$.
\end{definition}
We see whenever $ (\{g_j\}_{j\in \mathbb{J}}, \{\omega_j\}_{j\in \mathbb{J}})$  is a dual of $ (\{f_j\}_{j\in \mathbb{J}}, \{\tau_j\}_{j\in \mathbb{J}})$, then $ (\{f_j\}_{j\in \mathbb{J}}, \{\tau_j\}_{j\in \mathbb{J}})$ is a dual of $ (\{g_j\}_{j\in \mathbb{J}}, \{\omega_j\}_{j\in \mathbb{J}})$.
\begin{theorem}
Let $( \{f_j\}_{j\in \mathbb{J}},\{\tau_j\}_{j\in \mathbb{J}} )$ be a  p-frame for $ \mathcal{X}$ with frame bounds $ a$ and $ b.$ Then
\begin{enumerate}[\upshape(i)]
\item The canonical dual p-frame of the canonical dual p-frame  of $ (\{f_j\}_{j\in \mathbb{J}} ,\{\tau_j\}_{j\in \mathbb{J}} )$ is itself.
\item$ \frac{1}{b}, \frac{1}{a}$ are frame bounds for the canonical dual of $ (\{f_j\}_{j\in \mathbb{J}},\{\tau_j\}_{j\in \mathbb{J}}).$
\item If $ a, b $ are optimal frame bounds for $( \{f_j\}_{j\in \mathbb{J}} , \{\tau_j\}_{j\in \mathbb{J}}),$ then $ \frac{1}{b}, \frac{1}{a}$ are optimal  frame bounds for its canonical dual.
\end{enumerate} 
\end{theorem} 
\begin{proof}
For $ x \in \mathcal{X},$ 
$$ \sum\limits_{j\in \mathbb{J}}\widetilde{f}_j(x)\widetilde{\tau}_j= \sum\limits_{j\in \mathbb{J}}f_j(\widehat{S}_{f,\tau}^{-1}x)\widehat{S}_{f,\tau}^{-1}\tau_j = \widehat{S}_{f,\tau}^{-1}\left(\sum\limits_{j\in \mathbb{J}}f_j(\widehat{S}_{f,\tau}^{-1}x)\tau_j\right)=\widehat{S}_{f,\tau}^{-1}\widehat{S}_{f,\tau}(\widehat{S}_{f,\tau}^{-1}x)=\widehat{S}_{f,\tau}^{-1}x.$$
Thus p-frame operator for the canonical dual $(\{\widetilde{f}_j\}_{j \in \mathbb{J}},   \{\widetilde{\tau}_j\}_{j\in \mathbb{J}} )$ is $ \widehat{S}_{f,\tau}^{-1}.$
Therefore, its canonical dual is $(\{ f_j\widehat{S}_{f,\tau}^{-1}\widehat{S}_{f,\tau}\}_{j \in \mathbb{J}}, \{\widehat{S}_{f,\tau}\widehat{S}_{f,\tau}^{-1}\tau_j\}_{j\in \mathbb{J}}).$ Other facts are easy.
\end{proof}
\begin{proposition}
Let $ (\{f_j\}_{j\in \mathbb{J}}, \{\tau_j\}_{j\in \mathbb{J}})$  and $ (\{g_j\}_{j\in \mathbb{J}}, \{\omega_j\}_{j\in \mathbb{J}})$  be  p-frames for   $\mathcal{X}$. Then the following are equivalent.
\begin{enumerate}[\upshape(i)]
\item  $(\{g_j\}_{j\in \mathbb{J}}, \{\omega_j\}_{j\in \mathbb{J}}) $ is dual of $ (\{f_j\}_{j\in \mathbb{J}}, \{\tau_j\}_{j\in \mathbb{J}})$.   
\item $ \sum_{j\in \mathbb{J}}f_j(x)\omega_j=\sum_{j\in \mathbb{J}}g_j(x)\tau_j=x, \forall x \in \mathcal{X}.$
\end{enumerate}
\end{proposition}
\begin{proof}
$\widehat{\theta}_\omega\theta_f x= \sum_{j\in \mathbb{J}}f_j(x)\omega_j, \widehat{\theta}_\tau\theta_g x= \sum_{j\in \mathbb{J}}g_j(x)\tau_j, \forall x \in \mathcal{X}$.
\end{proof}
\begin{proposition}
Let $ (\{f_j\}_{j\in \mathbb{J}}, \{\tau_j\}_{j\in \mathbb{J}})$ be a p-frame for   $\mathcal{X}$. If $\{\tau_j\}_{j\in \mathbb{J}}$ is a Schauder basis for   $\mathcal{X}$ and $ f_j(\tau_k)=\delta_{j,k},\forall j, k \in \mathbb{J}$, then $ (\{f_j\}_{j\in \mathbb{J}}, \{\tau_j\}_{j\in \mathbb{J}}) $ has unique dual.
\end{proposition}
\begin{proof}
Let $(\{g_j\}_{j\in \mathbb{J}}, \{\omega_j\}_{j\in \mathbb{J}})$  and $ (\{h_j\}_{j\in \mathbb{J}}, \{\rho_j\}_{j\in \mathbb{J}})$ be two dual p-frames of $ (\{f_j\}_{j\in \mathbb{J}}, \{\tau_j\}_{j\in \mathbb{J}})$. Then $ \sum_{j\in \mathbb{J}}(g_j(x)-h_j(x))\tau_j=0, \forall x \in \mathcal{X}$ and $ \sum_{j\in \mathbb{J}}f_j(x)(\omega_j-\rho_j)=0, \forall x \in \mathcal{X}$. First equality gives $ g_j=h_j, \forall j \in \mathbb{J}$ and by evaluating second equality at a fixed $ x_k$ gives $ \omega_k=\rho_k$, and $ k$ is free.
\end{proof}
\begin{lemma}\label{DUALCHARATERIZATIONLEMMABANACH1}
Let  $ (\{f_j\}_{j\in \mathbb{J}}, \{\tau_j\}_{j\in \mathbb{J}}) $  be a  p-frame for   $\mathcal{X}$ and $ \{e_j\}_{j\in \mathbb{J}}$ be the standard  Schauder basis for $ \ell^p(\mathbb{J})$. Let $ h_j : \sum_{k\in\mathbb{J}}a_ke_k\mapsto a_j \in \mathbb{K}, \forall j\in \mathbb{J}$.  Then the dual p-frames  of $ (\{f_j\}_{j\in \mathbb{J}}, \{\tau_j\}_{j\in \mathbb{J}})$ are precisely $ (\{g_j=h_jU\}_{j\in \mathbb{J}}, \{\omega_j=Ve_j\}_{j\in \mathbb{J}})$, where $ U:\mathcal{X} \rightarrow\ell^p(\mathbb{J})$ is  bounded right-inverse of $ \widehat{\theta}_\tau$, and  $V: \ell^p(\mathbb{J}) \rightarrow \mathcal{X}$ is  bounded left-inverse of $ \theta_f$ such that the resolvent of $ VU$ contains $(-\infty,0]$.
\end{lemma}
\begin{proof}
$(\Leftarrow)$ We see $ \theta_gx=\{g_j(x) \}_{{j\in \mathbb{J}}}=\{f_j(Ux) \}_{{j\in \mathbb{J}}}=\sum_{j\in \mathbb{J}}f_j(Ux) e_j=Ux, \forall x \in \mathcal{X}$, and 
$\widehat{\theta}_\omega(\{a_j \}_{{j\in \mathbb{J}}})=\sum_{j\in \mathbb{J}}a_j\omega_j=\sum_{j\in \mathbb{J}}a_jVe_j=V(\{a_j \}_{{j\in \mathbb{J}}}) , \forall \{a_j \}_{{j\in \mathbb{J}}} \in \ell^p(\mathbb{J})$. Then  $ \widehat{S}_{g,\omega}=\widehat{\theta}_\omega\theta_g=VU$ whose resolvent  contains $(-\infty,0]$. Hence $ (\{g_j\}_{j\in \mathbb{J}}, \{\omega_j\}_{j\in \mathbb{J}})$ is a  p-frame. Now we check duality:  $\widehat{\theta}_\tau\theta_g=\widehat{\theta}_\tau U=I_\mathcal{X} $, $ \widehat{\theta}_\omega\theta_f=V\theta_f =I_\mathcal{X}$.

$(\Rightarrow)$ Let $ (\{g_j\}_{j\in \mathbb{J}}, \{\omega_j\}_{j\in \mathbb{J}})$ be a dual frame of $ (\{f_j\}_{j\in \mathbb{J}}, \{\tau_j\}_{j\in \mathbb{J}})$.  Then $\widehat{\theta}_\tau\theta_g =I_\mathcal{X} $, $ \widehat{\theta}_\omega\theta_f =I_\mathcal{X}$. Define $ U\coloneqq\theta_g, V\coloneqq\widehat{\theta}_\omega.$ Then $ U:\mathcal{X} \rightarrow\ell^p(\mathbb{J})$ is a bounded right-inverse of $ \widehat{\theta}_\tau$, and  $V: \ell^p(\mathbb{J}) \rightarrow \mathcal{X}$ is  a bounded left-inverse of $ \theta_f$ such that the resolvent of $ VU=\widehat{\theta}_\omega\theta_g=\widehat{S}_{g,\omega}$ contains $(-\infty,0]$. Moreover, $ h_jU(x)=h_j(\sum_{k\in \mathbb{J}}g_k(x)e_k)=\sum_{k\in \mathbb{J}}g_k(x)h_j(e_k)=g_j(x), \forall x \in \mathcal{X}, \forall j \in \mathbb{J}$, and $Ve_j=\widehat{\theta}_\omega e_j=\omega_j, \forall j \in \mathbb{J} $.
\end{proof}
\begin{lemma}\label{DUALCHARATERIZATIONLEMMABANACH2}
Let $ (\{f_j\}_{j\in \mathbb{J}}, \{\tau_j\}_{j\in \mathbb{J}}) $  be a  p-frame for   $\mathcal{X}$. Then the bounded 
\begin{enumerate}[\upshape(i)]
\item  right-inverses of $ \widehat{\theta}_\tau$ are precisely  $\theta_f\widehat{S}_{f,\tau}^{-1}+(I_{\ell^p(\mathbb{J})}-\theta_f\widehat{S}_{f,\tau}^{-1}\widehat{\theta}_\tau)U,$ where $U\in \mathcal{B}(\mathcal{X}, \ell^p(\mathbb{J}))$.
\item  left-inverses of $ \theta_f$ are precisely  $\widehat{S}_{f,\tau}^{-1}\widehat{\theta}_\tau+V(I_{\ell^p(\mathbb{J})}-\theta_f\widehat{S}_{f,\tau}^{-1}\widehat{\theta}_\tau)$, where $V\in \mathcal{B}(\ell^p(\mathbb{J}), \mathcal{X})$. 
\end{enumerate}	
\end{lemma} 
\begin{proof}
\begin{enumerate}[\upshape(i)]
\item  $(\Leftarrow)$ Let $U:\mathcal{X} \rightarrow \ell^p(\mathbb{J})$ be  a bounded operator. Then $\widehat{\theta}_\tau(\theta_f\widehat{S}_{f,\tau}^{-1}+(I_{\ell^p(\mathbb{J})}-\theta_f\widehat{S}_{f,\tau}^{-1}\widehat{\theta}_\tau)U)=I_\mathcal{X}+\widehat{\theta}_\tau U-I_\mathcal{X}\widehat{\theta}_\tau U=I_\mathcal{X}$. Therefore $\theta_f\widehat{S}_{f,\tau}^{-1}+(I_{\ell^p(\mathbb{J})}-\theta_f\widehat{S}_{f,\tau}^{-1}\widehat{\theta}_\tau)U$ is a bounded right-inverse of $ \widehat{\theta}_\tau$.  

$(\Rightarrow)$ Let $ R :\mathcal{X} \rightarrow \ell^p(\mathbb{J})$ be a bounded right-inverse of $ \widehat{\theta}_\tau$. Define $U\coloneqq R $. Then $\theta_f\widehat{S}_{f,\tau}^{-1}+(I_{\ell^p(\mathbb{J})}-\theta_f\widehat{S}_{f,\tau}^{-1}\widehat{\theta}_\tau)U=\theta_f\widehat{S}_{f,\tau}^{-1}+(I_{\ell^p(\mathbb{J})}-\theta_f\widehat{S}_{f,\tau}^{-1}\widehat{\theta}_\tau)R=\theta_f\widehat{S}_{f,\tau}^{-1}+R-\theta_f\widehat{S}_{f,\tau}^{-1}=R$.
\item
$(\Leftarrow)$ Let $V: \ell^p(\mathbb{J})\rightarrow \mathcal{X}$ be a bounded operator. Then $(\widehat{S}_{f,\tau}^{-1}\widehat{\theta}_\tau+V(I_{\ell^p(\mathbb{J})}-\theta_f\widehat{S}_{f,\tau}^{-1}\widehat{\theta}_\tau))\theta_f=I_\mathcal{X}+V\theta_f-V\theta_fI_\mathcal{X}=I_\mathcal{X}$. Therefore  $\widehat{S}_{f,\tau}^{-1}\widehat{\theta}_\tau+V(I_{\ell^p(\mathbb{J})}-\theta_f\widehat{S}_{f,\tau}^{-1}\widehat{\theta}_\tau)$ is a bounded left-inverse of $\theta_f$.

$(\Rightarrow)$ Let $ L:\ell^p(\mathbb{J})\rightarrow \mathcal{X}$ be a bounded left-inverse of $ \theta_f$. Define $V\coloneqq L$. Then $\widehat{S}_{f,\tau}^{-1}\widehat{\theta}_\tau+V(I_{\ell^p(\mathbb{J})}-\theta_f\widehat{S}_{f,\tau}^{-1}\widehat{\theta}_\tau) =\widehat{S}_{f,\tau}^{-1}\widehat{\theta}_\tau+L(I_{\ell^p(\mathbb{J})}-\theta_f\widehat{S}_{f,\tau}^{-1}\widehat{\theta}_\tau)=\widehat{S}_{f,\tau}^{-1}\widehat{\theta}_\tau+L-\widehat{S}_{f,\tau}^{-1}\widehat{\theta}_\tau= L$.
\end{enumerate}	
\end{proof} 
\begin{theorem}
Let  $ (\{f_j\}_{j\in \mathbb{J}}, \{\tau_j\}_{j\in \mathbb{J}}) $  be a  p-frame for   $\mathcal{X}$. The dual p-frames 	 $ (\{g_j\}_{j\in \mathbb{J}}, \{\omega_j\}_{j\in \mathbb{J}}) $ of $ (\{f_j\}_{j\in \mathbb{J}}, \{\tau_j\}_{j\in \mathbb{J}}) $ are precisely 
\begin{align*}
(\{g_j=f_j\widehat{S}_{f,\tau}^{-1}+h_jU-f_j\widehat{S}_{f,\tau}^{-1}\widehat{\theta}_\tau U\}_{j\in \mathbb{J}},
\{\omega_j=\widehat{S}_{f,\tau}^{-1}\tau_j+Ve_j-V\theta_f\widehat{S}_{f,\tau}^{-1}\tau_j\}_{j\in \mathbb{J}} )
\end{align*}
such that the resolvent of 
$$\widehat{S}_{f,\tau}^{-1}+VU-V\theta_f\widehat{S}_{f,\tau}^{-1}\widehat{\theta}_\tau U$$
contains $(-\infty,0]$, where  $ \{e_j\}_{j\in \mathbb{J}}$ is  the standard  Schauder basis for $ \ell^p(\mathbb{J})$, $ h_j : \sum_{k\in\mathbb{J}}a_ke_k\mapsto a_j \in \mathbb{K}, \forall j\in \mathbb{J}$, and  $U\in \mathcal{B}(\mathcal{X}, \ell^p(\mathbb{J})), V\in \mathcal{B} (\ell^p(\mathbb{J}), \mathcal{X})$.
\end{theorem} 
\begin{proof}
 From Lemma \ref{DUALCHARATERIZATIONLEMMABANACH1} and Lemma \ref{DUALCHARATERIZATIONLEMMABANACH2} we can  characterize   the dual frames of  $ (\{f_j\}_{j\in \mathbb{J}}, \{\tau_j\}_{j\in \mathbb{J}}) $ as the families 
 \begin{align*}
 &\left\{g_j=h_j\theta_f\widehat{S}_{f,\tau}^{-1}+h_jU-h_j\theta_f\widehat{S}_{f,\tau}^{-1}\widehat{\theta}_\tau U=f_j\widehat{S}_{f,\tau}^{-1}+h_jU-f_j\widehat{S}_{f,\tau}^{-1}\widehat{\theta}_\tau U\right\}_{j\in \mathbb{J}},\\
 &\left\{\omega_j=\widehat{S}_{f,\tau}^{-1}\widehat{\theta}_\tau e_j+Ve_j-V\theta_f\widehat{S}_{f,\tau}^{-1}\widehat{\theta}_\tau e_j=\widehat{S}_{f,\tau}^{-1}\tau_j+Ve_j-V\theta_f\widehat{S}_{f,\tau}^{-1}\tau_j\right\}_{j\in \mathbb{J}}
 \end{align*}
 such that the resolvent of 
 $$(\widehat{S}_{f,\tau}^{-1}\widehat{\theta}_\tau+V(I_{\ell^p(\mathbb{J})}-\theta_f\widehat{S}_{f,\tau}^{-1}\widehat{\theta}_\tau))(\theta_f\widehat{S}_{f,\tau}^{-1}+(I_{\ell^p(\mathbb{J})}-\theta_f\widehat{S}_{f,\tau}^{-1}\widehat{\theta}_\tau)U) $$
 contains $(-\infty,0]$, where  $ \{e_j\}_{j\in \mathbb{J}}$ is  the standard Schauder  basis for $ \ell^p(\mathbb{J})$, $ h_j : \sum_{k\in\mathbb{J}}a_ke_k\mapsto a_j \in \mathbb{K}, \forall j\in \mathbb{J}$, and  $U\in \mathcal{B}(\mathcal{X},\ell^p(\mathbb{J})), V\in \mathcal{B}( \ell^p(\mathbb{J}),\mathcal{X})$. We expand and see 
 \begin{align*}
 &(\widehat{S}_{f,\tau}^{-1}\widehat{\theta}_\tau+V(I_{\ell^p(\mathbb{J})}-\theta_f\widehat{S}_{f,\tau}^{-1}\widehat{\theta}_\tau))(\theta_f\widehat{S}_{f,\tau}^{-1}+(I_{\ell^p(\mathbb{J})}-\theta_f\widehat{S}_{f,\tau}^{-1}\widehat{\theta}_\tau)U)\\
&=\widehat{S}_{f,\tau}^{-1}+VU-V\theta_f\widehat{S}_{f,\tau}^{-1}\widehat{\theta}_\tau U.
 \end{align*}
 \end{proof} 
 
\begin{definition}
A p-frame $ (\{g_j\}_{j\in \mathbb{J}}, \{\omega_j\}_{j\in \mathbb{J}})$  for $\mathcal{X}$ is said to be orthogonal to    a p-frame $ (\{f_j\}_{j\in \mathbb{J}}, \{\tau_j\}_{j\in \mathbb{J}})$  for $\mathcal{X}$ if  $ \widehat{\theta}_\omega\theta_f=\widehat{\theta}_\tau\theta_g=0.$
\end{definition}
Orthogonality is symmetric. Dual p-frames cannot be orthogonal to each other and orthogonal p-frames cannot be dual to each other. Also,  if $ (\{g_j\}_{j\in \mathbb{J}}, \{\omega_j\}_{j\in \mathbb{J}})$ is orthogonal to $ (\{f_j\}_{j\in \mathbb{J}}, \{\tau_j\}_{j\in \mathbb{J}})$, then  both $ (\{f_j\}_{j\in \mathbb{J}}, \{\omega_j\}_{j\in \mathbb{J}})$ and $ (\{g_j\}_{j\in \mathbb{J}}, \{\tau_j\}_{j\in \mathbb{J}})$ are not p-frames.
\begin{proposition}
Let   $ (\{f_j\}_{j\in \mathbb{J}}, \{\tau_j\}_{j\in \mathbb{J}})$ and $(\{g_j\}_{j\in \mathbb{J}}, \{\omega_j\}_{j\in \mathbb{J}}) $ be  p-frames for  $\mathcal{X}$. Then the following are equivalent.
\begin{enumerate}[\upshape(i)]
\item  $(\{g_j\}_{j\in \mathbb{J}}, \{\omega_j\}_{j\in \mathbb{J}}) $ is orthogonal to  $ (\{f_j\}_{j\in \mathbb{J}}, \{\tau_j\}_{j\in \mathbb{J}})$.   
\item $ \sum_{j\in \mathbb{J}}f_j(x)\omega_j=\sum_{j\in \mathbb{J}}g_j(x)\tau_j=0, \forall x \in \mathcal{X}.$
\end{enumerate}
\end{proposition}
\begin{proposition}
Let $ (\{f_j\}_{j\in \mathbb{J}}, \{\tau_j\}_{j\in \mathbb{J}}) $ and $ (\{g_j\}_{j\in \mathbb{J}}, \{\omega_j\}_{j\in \mathbb{J}}) $ be  two Parseval p-frames for  $\mathcal{X}$ which are  orthogonal. If $A,B,C,D \in \mathcal{B}(\mathcal{X})$ are such that $ CA+DB=I_\mathcal{X}$, then  $ (\{f_jA+g_jB\}_{j\in \mathbb{J}}, \{C\tau_j+D\omega_j\}_{j\in \mathbb{J}}) $ is a  Parseval p-frame for  $\mathcal{X}$. In particular,  if scalars $ a,b,c,d$ satisfy $ca+db =1$, then $ (\{af_j+bg_j\}_{j\in \mathbb{J}}, \{c\tau_j+d\omega_j\}_{j\in \mathbb{J}}) $ is a  Parseval p-frame for  $\mathcal{X}$.
\end{proposition} 
\begin{proof}
We observe $ \theta_{fA+gB} x = \{(f_jA+g_jB)(x) \}_{j\in \mathbb{J}}=\{f_j(Ax) \}_{j\in \mathbb{J}}+\{g_j(Bx) \}_{j\in \mathbb{J}}=\theta_f(Ax)+\theta_g(Bx),\forall x \in \mathcal{X}$ and  $ \widehat{\theta}_{C\tau+D\omega}(\{a_j \}_{j\in \mathbb{J}})=\sum_{j\in \mathbb{J}}a_j(C\tau_j+D\omega_j)=C\widehat{\theta}_\tau(\{a_j \}_{j\in \mathbb{J}})+D\widehat{\theta}_\omega(\{a_j \}_{j\in \mathbb{J}}) $. Therefore 	$\widehat{S}_{fA+gB,C\tau+D\omega} =\widehat{\theta}_{C\tau+D\omega} \theta_{fA+gB}= ( C\widehat{\theta}_\tau+ D\widehat{\theta}_\omega)(\theta_fA+\theta_gB)=C\widehat{\theta}_\tau\theta_fA+C\widehat{\theta}_\tau\theta_gB+D\widehat{\theta}_\omega\theta_fA+D\widehat{\theta}_\omega\theta_gB=C\widehat{S}_{f,\tau}A+0+0+D\widehat{S}_{g,\omega}B=CI_\mathcal{X}A+DI_\mathcal{X}B=I_\mathcal{X}.$
\end{proof} 
\textbf{Characterizations}
\begin{theorem}\label{PSEQUENTIALCHARACTERIZATIONBANACH}
Let $ \{\tau_j\}_{j \in \mathbb{J}}$  be  a p-orthonormal basis for  $ \mathcal{X}$ and  $ f_j :\mathcal{X} \ni \sum_{k \in \mathbb{J}}a_k\tau_k \mapsto a_j \in \mathbb{K} , \forall j \in \mathbb{J}$.  Then 
\begin{enumerate}[\upshape(i)]
\item The Riesz p-bases $(\{g_j\}_{j\in \mathbb{J}},\{\omega_j\}_{j \in \mathbb{J}})$ for $ \mathcal{X}$ are precisely  $(\{f_jU\}_{j\in \mathbb{J}},\{V\tau_j\}_{j \in \mathbb{J}})$, where $ U,V \in \mathcal{B}(\mathcal{X})$ are invertible such that the  resolvent of $ VU$ contains $(-\infty, 0]$. 
\item The p-frames $(\{g_j\}_{j\in \mathbb{J}},\{\omega_j\}_{j \in \mathbb{J}})$ for $ \mathcal{X}$ are precisely  $(\{f_jU\}_{j\in \mathbb{J}},\{V\tau_j\}_{j \in \mathbb{J}})$, where $ U,V : \mathcal{X} \rightarrow \mathcal{X}$ are such that the resolvent of $ VU$ contains $(-\infty, 0]$. 
\item The Riesz p-frames  $(\{g_j\}_{j\in \mathbb{J}},\{\omega_j\}_{j \in \mathbb{J}})$ for $ \mathcal{X}$ are precisely  $(\{f_jU\}_{j\in \mathbb{J}},\{V\tau_j\}_{j \in \mathbb{J}})$, where $ U,V \in \mathcal{B}(\mathcal{X})$ are such that the  resolvent of $ VU$ contains $(-\infty, 0]$ and $ U(VU)^{-1}V =I_\mathcal{X}$.
\item The orthonormal  p-frames  $(\{g_j\}_{j\in \mathbb{J}},\{\omega_j\}_{j \in \mathbb{J}})$ for $ \mathcal{X}$ are precisely  $(\{f_jU\}_{j\in \mathbb{J}},\{V\tau_j\}_{j \in \mathbb{J}})$, where $ U,V \in \mathcal{B}(\mathcal{X})$ are such that   the  resolvent of $ VU$ contains $(-\infty, 0]$ and $UV=I_\mathcal{X} =VU$.
\end{enumerate}	
\end{theorem}
\begin{proof}
\begin{enumerate}[\upshape(i)]
\item $(\Leftarrow)$ This is the definition of Riesz p-basis.
		
$(\Rightarrow)$ Let $ \{\rho_j\}_{j \in \mathbb{J}}$  be  a p-orthonormal basis for  $ \mathcal{X}$, $ R,S :\mathcal{X} \rightarrow \mathcal{X}$ be bounded invertible with resolvent of $ SR$ contains $(-\infty, 0]$ such that  $g_j=h_jR, \omega_j=S\rho_j,\forall j \in \mathbb{J}  $, where $ h_j :\mathcal{X} \ni \sum_{k \in \mathbb{J}}a_k\rho_k \mapsto a_j \in \mathbb{K} , \forall j \in \mathbb{J}$. Define $T :\mathcal{X} \ni \sum_{j\in \mathbb{J}}a_j\tau_j \mapsto \sum_{j\in \mathbb{J}}a_j\rho_j \in \mathcal{X} $. Then $ T$ is bounded invertible with $T ^{-1}:\mathcal{X} \ni \sum_{j\in \mathbb{J}}b_j\rho_j \mapsto \sum_{j\in \mathbb{J}}b_j\tau_j \in \mathcal{X}$. Define $ U\coloneqq T^{-1}R$ and $ V\coloneqq ST$. Then  $ U,V$ are bounded invertible and the resolvent of $ VU=STT^{-1}R=SR$ contains $(-\infty, 0]$. Let $j\in \mathbb{J} $. Now $ V\tau_j=ST\tau_j=S\rho_j=\omega_j$ and for $ x=\sum_{j\in \mathbb{J}}f_j(x)\tau_j \in \mathcal{X}$, we get $ (h_jT)x=h_j(\sum_{k\in \mathbb{J}}f_k(x)T\tau_k)=h_j(\sum_{k\in \mathbb{J}}f_k(x)\rho_k)= \sum_{k\in \mathbb{J}}f_k(x)h_j(\rho_k)=\sum_{k\in \mathbb{J}}f_k(x)\delta_{j,k}=f_j(x)$. Thus $h_jT=f_j \Rightarrow h_j=f_jT^{-1}\Rightarrow h_jR=f_jT^{-1}R=f_jU$. But $ h_jR=g_j$ and $j$ was arbitrary. This gives $ f_jU=g_j, \forall j \in \mathbb{J}$.

\item $(\Leftarrow)$ $ \theta_{g}x=\{f_jUx\}_{j\in \mathbb{J}}=\theta_f(Ux),\forall x \in \mathcal{X},$ $\widehat{\theta}_{\omega}(\{a_j\}_{j\in \mathbb{J}})=\sum_{j\in\mathbb{J}}a_jV\tau_j=V(\sum_{j\in\mathbb{J}}a_j\tau_j)=V\widehat{\theta}_\tau(\{a_j\}_{j\in \mathbb{J}})$, $\forall \{a_j\}_{j\in \mathbb{J}}\in \ell^p(\mathbb{J}).$ So $ \theta_{g}$ and $ \widehat{\theta}_{\omega}$ exist and are bounded linear. Now $ \widehat{S}_{g, \omega}x= \sum_{j\in\mathbb{J}}(f_jU)(x)V\tau_j=V(\sum_{j\in\mathbb{J}}f_j(Ux)\tau_j)=V(Ux), \forall x \in \mathcal{X}$. Therefore $\widehat{S}_{g, \omega}=VU $ whose resolvent contains $(-\infty, 0]$.
		
$(\Rightarrow)$ Let $ \{e_j\}_{j\in \mathbb{J}}$ be the standard Schauder basis for $\ell^p(\mathbb{J})$. Since $ \{\tau_j\}_{j\in \mathbb{J}}$ is an orthonormal basis for $ \mathcal{X}$, the map defined by $ T:\mathcal{X}\ni \sum_{j\in\mathbb{J}}a_j\tau_j \mapsto   \sum_{j\in\mathbb{J}}a_je_j \in \ell^p(\mathbb{J})$ is an isometric isomorphism with inverse $ T^{-1} :\ell^p(\mathbb{J}) \ni \sum_{j\in\mathbb{J}}b_je_j \mapsto \sum_{j\in\mathbb{J}}b_j\tau_j \in \mathcal{X}$. Define $ U\coloneqq T^{-1}\theta_g$ and $V\coloneqq\widehat{\theta}_\omega T$. Then $ U,V$ are bounded  with resolvent of $ VU=(\widehat{\theta}_\omega T)(T^{-1}\theta_g)=\widehat{\theta}_\omega\theta_g=\widehat{S}_{g,\omega}$ contains $(-\infty, 0]$ and  for $ x=\sum_{j\in \mathbb{J}}f_j(x)\tau_j \in \mathcal{X}$ we have $ (f_jU)x= f_j(T^{-1}\theta_gx)=f_j(T^{-1}(\{g_k(x)\}_{k\in\mathbb{J}})) = f_j(\sum_{k\in \mathbb{J}}g_k(x)T^{-1}e_k)=f_j(\sum_{k\in \mathbb{J}}g_k(x)\tau_k)=\sum_{k\in \mathbb{J}}g_k(x)\delta_{j,k}=g_j(x) , V\tau_j=\widehat{\theta}_\omega T\tau_j=\widehat{\theta}_\omega e_j=\omega_j, \forall x \in \mathcal{X}, \forall j \in \mathbb{J}$.
\item From (ii). $(\Leftarrow)$
We see $\theta_f\widehat{\theta}_\tau(\{a_j\}_{j \in \mathbb{J}})=\sum_{j\in \mathbb{J}}a_j\theta_f(\tau_j)=\sum_{j\in \mathbb{J}}a_j\sum_{k\in \mathbb{J}}f_k(\tau_j)e_k=\sum_{j\in \mathbb{J}}a_je_j=\{a_j\}_{j \in \mathbb{J}}, \forall \{a_j\}_{j \in \mathbb{J}} \in \ell^p(\mathbb{J}) $. Hence  $\widehat{P}_{g,\omega}=\theta_g\widehat{S}_{g,\omega}^{-1}\widehat{\theta}_\omega=\theta_fU(VU)^{-1}V\widehat{\theta}_\tau =\theta_fI_{\ell^p(\mathbb{J})}\widehat{\theta}_\tau=I_{\ell^p(\mathbb{J})}$.

$(\Rightarrow)$  $U(VU)^{-1}V=(T^{-1}\theta_g)\widehat{S}_{g,\omega}^{-1}(\widehat{\theta}_\omega T)=T^{-1}\widehat{P}_{g,\omega} T =T^{-1}I_{\ell^p(\mathbb{J})} T =I_{\ell^p(\mathbb{J})}$.
\item From (iii).  $(\Leftarrow)$ $\widehat{S}_{g,\omega}=VU=I_{\mathcal{X}},  \widehat{P}_{g,\omega}=\theta_g\widehat{S}_{g,\omega}^{-1}\widehat{\theta}_\omega=\theta_gI_{\mathcal{X}}\widehat{\theta}_\omega =\theta_fUV\widehat{\theta}_\tau =\theta_fI_{\mathcal{X}}\widehat{\theta}_\tau=\theta_f\widehat{\theta}_\tau=I_{\ell^p(\mathbb{J})}$.

 $(\Rightarrow)$  $ VU=\widehat{\theta}_\omega TT^{-1}\theta_g=\widehat{\theta}_\omega \theta_g=\widehat{S}_{g,\omega}=I_\mathcal{X}, UV=T^{-1}\theta_g\widehat{\theta}_\omega T=T^{-1}\widehat{P}_{g,\omega}T=T^{-1}I_{\ell^p(\mathbb{J})}T =I_\mathcal{X}.$
\end{enumerate}
\end{proof}
\begin{theorem}(cf.\cite{OLE1})\label{COORDINATEFUNCTIONAL}
Let $\{x_j\}_{j \in \mathbb{J}} $ be a Schauder basis for $ \mathcal{X}$. Then the functionals $ f_j: \mathcal{X}\ni \sum_{k\in \mathbb{J}}a_kx_k \mapsto a_j \in  \mathbb{K} ,\forall  j \in \mathbb{J}$ are bounded. If there exists $ r>0$ such that $\|x_j\|\geq r , \forall j \in \mathbb{J}$, then the collection $\{f_j\}_{j \in \mathbb{J}} $ is uniformly bounded.
\end{theorem}
\begin{corollary}
\begin{enumerate}[\upshape(i)]
\item If  $(\{g_j=f_jU\}_{j\in \mathbb{J}},\{\omega_j=V\tau_j\}_{j \in \mathbb{J}})$ is a Riesz p-basis for $ \mathcal{X}$, then  $\{g_j\}_{j \in \mathbb{J}} $ is uniformly bounded, and $\|V^{-1}\|^{-1} \leq \|\omega_j\|\leq \|V\|,\forall j \in \mathbb{J}.$
\item If  $(\{g_j\}_{j\in \mathbb{J}},\{\omega_j\}_{j \in \mathbb{J}})$ is a p-frame for $ \mathcal{X}$, then  $\{g_j\}_{j \in \mathbb{J}} $ is uniformly bounded, and $ \|\omega_j\|\leq \|V\|,\forall j \in \mathbb{J}.$
\end{enumerate}
\end{corollary}
\begin{proof}
We recall $\|\tau_j\|=1,\forall j \in \mathbb{J}$. Remainings follow from $\|g_j\|\leq \|f_j\|\|U\|,\forall j \in \mathbb{J} $ and Theorem \ref{COORDINATEFUNCTIONAL}.
\end{proof}

\textbf{Similarity }
\begin{definition}
Let  $(\{f_j\}_{j\in \mathbb{J}}, \{\tau_j\}_{j\in \mathbb{J}})$ and $(\{g_j\}_{j\in \mathbb{J}}, \{\omega_j\}_{j\in \mathbb{J}}) $ be  p-frames for  $\mathcal{X}$. We say that $(\{g_j\}_{j\in \mathbb{J}}, \{\omega_j\}_{j\in \mathbb{J}}) $ is similar to $(\{f_j\}_{j\in \mathbb{J}}, \{\tau_j\}_{j\in \mathbb{J}})$ if there exist invertible $T_{f,g}, T_{\tau,\omega}\in \mathcal{B}(\mathcal{X}) $ such that $ g_j=f_jT_{f,g}, \omega_j= T_{\tau,\omega}\tau_j, \forall j \in \mathbb{J}$.
\end{definition}
 \begin{lemma}\label{P-FRAMELEMMA}
 Let $ \{f_j\}_{j\in \mathbb{J}}\in \widehat{\mathscr{F}}_\tau,$ $ \{g_j\}_{j\in \mathbb{J}}\in \widehat{\mathscr{F}}_\omega$ and   $g_j=f_jT_{f, g} , \omega_j=T_{\tau,\omega}\tau_j,  \forall j \in \mathbb{J}$, for some invertible $T_{f,g}, T_{\tau,\omega}\in \mathcal{B}(\mathcal{X}) .$ Then 
 $ \theta_g=\theta_f T_{f,g}, \widehat{\theta}_\omega=T_{\tau,\omega}\widehat{\theta}_\tau ,\widehat{S}_{g,\omega}=T_{\tau,\omega}\widehat{S}_{f, \tau}T_{f,g},  \widehat{P}_{g,\omega}=\widehat{P}_{f, \tau}.$ Assuming that  $(\{f_j\}_{j\in \mathbb{J}}, \{\tau_j\}_{j\in \mathbb{J}})$ is a Parseval p-frame, then $(\{g_j\}_{j\in \mathbb{J}}, \{\omega_j\}_{j\in \mathbb{J}})$  is a Parseval p-frame if and only if $T_{\tau,\omega}T_{f,g}=I_\mathcal{X} .$
 \end{lemma} 
\begin{proof}
$ \theta_g(x)=\{g_j(x)\}_{j\in \mathbb{J}}=\{f_j(T_{f,g}x)\}_{j\in \mathbb{J}}=\theta_f(T_{f,g}x), \forall x \in \mathcal{X}$, 	 $\widehat{\theta}_\omega(\{a_j\}_{j\in \mathbb{J}})=\sum_{j\in \mathbb{J}}a_j\omega_j=\sum_{j\in \mathbb{J}}a_jT_{\tau,\omega}\tau_j=T_{\tau,\omega}(\widehat{\theta}_\tau(\{a_j\}_{j\in \mathbb{J}})) , \forall \{a_j\}_{j\in \mathbb{J}} \in \ell^p(\mathbb{J})$, $ \widehat{S}_{g,\omega}=\widehat{\theta}_\omega\theta_g=T_{\tau,\omega}\widehat{\theta}_\tau\theta_f T_{f,g} =T_{\tau,\omega}\widehat{S}_{f, \tau}T_{f,g},$  $\widehat{P}_{g,\omega}=\theta_g\widehat{S}_{g,\omega}^{-1}\widehat{\theta}_\omega=(\theta_f T_{f,g})(T_{\tau,\omega}\widehat{S}_{f, \tau}T_{f,g})^{-1}(T_{\tau,\omega}\widehat{\theta}_\tau)=\widehat{P}_{f, \tau}.$
\end{proof}
\begin{theorem}\label{SEQUENTIALSIMILARITYCHARACTERIZATIONPFRAME}
Let $ \{f_j\}_{j\in \mathbb{J}}\in \widehat{\mathscr{F}}_\tau,$ $ \{g_j\}_{j\in \mathbb{J}}\in \widehat{\mathscr{F}}_\omega.$ The following are equivalent.
\begin{enumerate}[\upshape(i)]
\item $ \{g_j\}_{j\in \mathbb{J}}\in \widehat{\mathscr{F}}_\omega$ and   $g_j=f_jT_{f, g} , \omega_j=T_{\tau,\omega}\tau_j,  \forall j \in \mathbb{J}$, for some invertible $T_{f,g}, T_{\tau,\omega}\in \mathcal{B}(\mathcal{X}) .$ 
\item $\theta_g=\theta_f T'_{f,g}, \widehat{\theta}_\omega=T'_{\tau,\omega}\widehat{\theta}_\tau  $ for some invertible $T'_{f,g}, T'_{\tau,\omega}\in \mathcal{B}(\mathcal{X}) .$
\item $\widehat{P}_{g,\omega}=\widehat{P}_{f, \tau}.$
\end{enumerate}
If one of the above conditions is satisfied, then  invertible operators in  $\operatorname{(i)}$ and  $\operatorname{(ii)}$ are unique and are given by  $T_{f,g}= \widehat{S}_{f,\tau}^{-1}\widehat{\theta}_\tau\theta_g, T_{\tau, \omega}=\widehat{\theta}_\omega\theta_f\widehat{S}_{f,\tau}^{-1}.$
In the case that $(\{f_j\}_{j\in \mathbb{J}},  \{\tau_j\}_{j\in \mathbb{J}})$ is a Parseval p-frame, then $(\{g_j\}_{j\in \mathbb{J}},  \{\omega_j\}_{j\in \mathbb{J}})$ is  a Parseval p-frame if and only if $T_{\tau, \omega}T_{f,g} =I_\mathcal{X}$   if and only if $ T_{f,g}T_{\tau, \omega} =I_\mathcal{X}$. 
\end{theorem}
\begin{proof}
Lemma \ref{P-FRAMELEMMA} gives (i) $\Rightarrow $ (ii) $\Rightarrow $ (iii). (ii) $\Rightarrow $ (i) Let $\{e_j\}_{j\in \mathbb{J}}$ be standard Schauder basis for $\ell^p(\mathbb{J}).$ Then $ \sum_{j\in\mathbb{J}}g_j(x)e_j=\theta_g(x)=\theta_f(T_{f,g}x)=\sum_{j\in\mathbb{J}}f_j(T_{f,g}x)e_j, \forall x \in \mathcal{X}.$ (iii) $\Rightarrow $ (ii) $\theta_g=\widehat{P}_{g,\omega}\theta_g=\widehat{P}_{f,\tau}\theta_g=\theta_f(\widehat{S}_{f,\tau}^{-1}\widehat{\theta}_{\tau}\theta_g)$, and $\widehat{\theta}_\omega=\widehat{\theta}_\omega\widehat{P}_{g,\omega}=\widehat{\theta}_\omega\widehat{P}_{f,\tau}=(\widehat{\theta}_\omega\theta_f\widehat{S}_{f,\tau}^{-1})\widehat{\theta}_\tau $. We  show  $\widehat{S}_{f,\tau}^{-1}\widehat{\theta}_{\tau}\theta_g$ and $\widehat{\theta}_\omega\theta_f\widehat{S}_{f,\tau}^{-1} $ are invertible. In fact, $(\widehat{S}_{f,\tau}^{-1}\widehat{\theta}_{\tau}\theta_g)(\widehat{S}_{g,\omega}^{-1}\widehat{\theta}_{\omega}\theta_f)=\widehat{S}_{f,\tau}^{-1}\widehat{\theta}_{\tau}\widehat{P}_{g,\omega}\theta_f=\widehat{S}_{f,\tau}^{-1}\widehat{\theta}_{\tau}\widehat{P}_{f,\tau}\theta_f=I_\mathcal{X} $,   $(\widehat{S}_{g,\omega}^{-1}\widehat{\theta}_{\omega}\theta_f)(\widehat{S}_{f,\tau}^{-1}\widehat{\theta}_{\tau}\theta_g)=\widehat{S}_{g,\omega}^{-1}\widehat{\theta}_{\omega}\widehat{P}_{f,\tau}\theta_g=\widehat{S}_{g,\omega}^{-1}\widehat{\theta}_{\omega}\widehat{P}_{g,\omega}\theta_g=I_\mathcal{X} $, and $ (\widehat{\theta}_\omega\theta_f\widehat{S}_{f,\tau}^{-1})(\widehat{\theta}_\tau\theta_g\widehat{S}_{g,\omega}^{-1})=\widehat{\theta}_\omega\widehat{P}_{f,\tau}\theta_g\widehat{S}_{g,\omega}^{-1}=\widehat{\theta}_\omega\widehat{P}_{g,\omega}\theta_g\widehat{S}_{g,\omega}^{-1}=I_\mathcal{X}$, $(\widehat{\theta}_\tau\theta_g\widehat{S}_{g,\omega}^{-1})(\widehat{\theta}_\omega\theta_f\widehat{S}_{f,\tau}^{-1})=\widehat{\theta}_\tau\widehat{P}_{g,\omega}\theta_f\widehat{S}_{f,\tau}^{-1}=\widehat{\theta}_\tau\widehat{P}_{f,\tau}\theta_f\widehat{S}_{f,\tau}^{-1}=I_\mathcal{X} $.

Let $T_{f,g}, T_{\tau,\omega}\in \mathcal{B}(\mathcal{X})$ be invertible and $g_j=f_jT_{f, g}, \omega_j=T_{\tau,\omega}\tau_j,  \forall j \in \mathbb{J}$. Then $\theta_g=\theta_fT_{f, g} $ implies $\widehat{\theta}_\tau\theta_g=\widehat{\theta}_\tau\theta_fT_{f, g}=\widehat{S}_{f,\tau}T_{f, g}  $ implies $ T_{f, g} =\widehat{S}_{f,\tau}^{-1}\widehat{\theta}_\tau\theta_g$, and $\widehat{\theta}_\omega=T_{\tau,\omega}\widehat{\theta}_\tau $ implies $\widehat{\theta}_\omega\theta_f=T_{\tau,\omega}\widehat{\theta}_\tau\theta_f=T_{\tau,\omega}\widehat{S}_{f,\tau} $, hence  $T_{\tau,\omega}=\widehat{\theta}_\omega\theta_f\widehat{S}_{f,\tau}^{-1} $. 
\end{proof}

\begin{corollary}
For a given p-frame $ (\{f_j\}_{j \in \mathbb{J}} , \{\tau_j\}_{j \in \mathbb{J}})$, the canonical dual of $ (\{f_j\}_{j \in \mathbb{J}} , \{\tau_j\}_{j \in \mathbb{J}}  )$ is the only dual p-frame that is similar to $ (\{f_j\}_{j \in \mathbb{J}} , \{\tau_j\}_{j \in \mathbb{J}} )$.
\end{corollary}
\begin{proof}
If $ (\{f_j\}_{j \in \mathbb{J}} , \{\tau_j\}_{j \in \mathbb{J}})$ and $ (\{g_j\}_{j \in \mathbb{J}} , \{\omega_j\}_{j \in \mathbb{J}})$  are similar and dual to each other, then there exist invertible  $T_{f,g}, T_{\tau,\omega} \in \mathcal{B}(\mathcal{X})$  such that $ g_j=f_jT_{f,g},\omega_j=T_{\tau,\omega}\tau_j ,\forall j \in \mathbb{J}$. Theorem \ref{SEQUENTIALSIMILARITYCHARACTERIZATIONPFRAME} gives $T_{f,g}=\widehat{S}_{f,\tau}^{-1}\widehat{\theta}_\tau\theta_g=\widehat{S}_{f,\tau}^{-1}I_\mathcal{X}=\widehat{S}_{f,\tau}^{-1}$, $T_{\tau, \omega}=\widehat{\theta}_\omega\theta_f\widehat{S}_{f,\tau}^{-1}=I_\mathcal{X}\widehat{S}_{f,\tau}^{-1}=\widehat{S}_{f,\tau}^{-1}.$ Hence $ (\{g_j\}_{j \in \mathbb{J}} , \{\omega_j\}_{j \in \mathbb{J}})$ is the canonical dual of $ (\{f_j\}_{j \in \mathbb{J}} , \{\tau_j\}_{j \in \mathbb{J}})$.	
\end{proof}
\begin{corollary}
 Two similar  p-frames cannot be orthogonal.
\end{corollary}
\begin{proof}
Let $ (\{f_j\}_{j \in \mathbb{J}} , \{\tau_j\}_{j \in \mathbb{J}})$ and $ (\{g_j\}_{j \in \mathbb{J}} , \{\omega_j\}_{j \in \mathbb{J}})$  be similar.  Then there exist invertible  $T_{f,g}, T_{\tau,\omega} \in \mathcal{B}(\mathcal{X})$  such that $ g_j=f_jT_{f,g},\omega_j=T_{\tau,\omega}\tau_j ,\forall j \in \mathbb{J}$. From Theorem \ref{SEQUENTIALSIMILARITYCHARACTERIZATIONPFRAME},  $\theta_g=\theta_f T_{f,g}, \widehat{\theta}_\omega=T_{\tau,\omega}\widehat{\theta}_\tau  $.   Therefore $ \widehat{\theta}_\tau \theta_g=\widehat{\theta}_\tau\theta_f T_{f,g}=\widehat{S}_{f,\tau}T_{f,g}\neq0. $ 
\end{proof}
\begin{remark}
For every p-frame  $(\{f_j\}_{j \in \mathbb{J}}, \{\tau_j\}_{j \in \mathbb{J}}),$ each  of `p-frames'  $( \{f_j\widehat{S}_{x, \tau}^{-1}\}_{j \in \mathbb{J}}, \{\tau_j\}_{j \in \mathbb{J}})$,    $( \{f_j\widehat{S}_{x, \tau}^{-1/2}\}_{j \in \mathbb{J}}, \{\widehat{S}_{x,\tau}^{-1/2}\tau_j\}_{j \in \mathbb{J}}),$ and  $ (\{f_j \}_{j \in \mathbb{J}}, \{\widehat{S}_{x,\tau}^{-1}\tau_j\}_{j \in \mathbb{J}})$ is a  Parseval p-frame which is similar to  $ (\{f_j\}_{j \in \mathbb{J}} , \{\tau_j\}_{j \in \mathbb{J}}).$  Hence each p-frame is similar to  Parseval p-frames.
\end{remark} 
\textbf{A finite dimensional result}

We refer \cite{PIETSCH1}, for the definition of trace of operators in Banach spaces.
\begin{theorem}\label{EXTENDEDDIMENSINANDTRACEFORMULABANACH}
If $(\{f_j\}_{j=1}^n,\{\tau_j\}_{j=1}^n)$ is a Parseval p-frame for  a finite dimensional Banach space $\mathcal{X}$, then
\begin{enumerate}[\upshape(i)]
\item (Extended dimension formula) $ \dim\mathcal{X}=\sum_{j=1}^nf_j(\tau_j).$
\item (Extended trace formula) For any $ T \in \mathcal{B}(\mathcal{X})$, $ \operatorname{Trace}(T)=\sum_{j=1}^nf_j(T\tau_j).$
\end{enumerate}	
\end{theorem}
\begin{proof}
\begin{enumerate}[\upshape(i)]
\item We may assume $\mathcal{X}=\mathbb{K}^m$ and let $\{e_j\}_{j=1}^m$ be the standard basis for $\mathbb{K}^m$. Define $ p_j: \mathbb{K}^m \ni (x_1,...,x_m) \mapsto x_j \in  \mathbb{K}, j=1,...,m$. Then
\begin{align*}
m=\sum_{j=1}^mp_j(e_j)=\sum_{j=1 }^mp_j\left(\sum_{k=1 }^nf_k(e_j)\tau_k\right)=\sum_{k=1}^n\sum_{j=1}^mf_k(e_j)p_j(\tau_k)
=\sum_{k=1}^nf_k\left(\sum_{j=1}^mp_j(\tau_k)e_j\right)=\sum_{k=1}^nf_k(\tau_k).
\end{align*}
\item Since $\mathcal{X} $ is finite dimensional, $ T$ becomes a  finite rank operator. Then there exist $\{g_j\}_{j=1}^m$ in $\mathcal{X}^*$ and $\{\omega_j\}_{j=1}^m$ in $\mathcal{X}$ such that $Tx=\sum_{j=1}^mg_j(x)\omega_j, \forall x \in \mathcal{X}$, for some $m$. Starting from the definition of trace, we get 
\begin{align*}
\operatorname{Trace}(T)&=\sum_{j=1}^mg_j(\omega_j)=\sum_{j=1 }^mg_j\left(\sum_{k=1 }^nf_k(\omega_j)\tau_k\right)=\sum_{k=1}^n\sum_{j=1}^mf_k(\omega_j)g_j(\tau_k)\\
&=\sum_{k=1}^nf_k\left(\sum_{j=1}^mg_j(\tau_k)\omega_j\right)=\sum_{k=1}^nf_k(T\tau_k).
\end{align*}
\end{enumerate}
\end{proof}

\section{Appendix}\label{APPENDIX}
With the analogy of Hilbert spaces, and since  we have defined the notion of p-orthonormal basis for Banach spaces (Section \ref{SVSECTION}) we define
\begin{definition}\label{SINGLEPRIESZ}
A collection $ \{x_j\}_{j\in \mathbb{J}}$  in  $ \mathcal{X}$ is said to be a Riesz p-basis for $ \mathcal{X}$ if  there exists a p-orthonormal basis $ \{\tau_j\}_{j\in \mathbb{J}}$ for $ \mathcal{X}$ and a bounded invertible operator $ T :  \mathcal{X}  \rightarrow \mathcal{X}$ such that $ x_j=T\tau_j, \forall j \in \mathbb{J}$.
\end{definition}
\begin{remark}\label{PRIESZTOUSUALRIESZ}
In Corollary \ref{PIMPLIESORTHONORMALFOR HILBERT} we proved that 2-orthonormal basis for a Hilbert space $\mathcal{H}$ is an orthonormal basis for $\mathcal{H}$. Hence Definition  \ref{SINGLEPRIESZ} gives Riesz basis definition  when considered on Hilbert spaces.
\end{remark}
We  can get Riesz p-bases by perturbing p-orthonormal bases, in precise  we can say the following theorem (which is an  extension of Paley-Wiener theorem).
\begin{theorem}\label{GENERALIZEDPALEYWEINER}
Let $\{x_j\}_{j\in \mathbb{J}} $  be a  p-orthonormal basis  for  $ \mathcal{X}$. If $ \{y_j\}_{j\in \mathbb{J}} $  is a sequence in $\mathcal{X}$ such that there exists $ 0\leq\lambda<1$ and 
$$ \left\|\sum_{j\in \mathbb{S}}c_j(x_j-y_j)\right\|\leq \lambda \left(\sum_{j\in \mathbb{S}}|c_j|^p\right)^\frac{1}{p} ,~ \forall c_j \in \mathbb{K}, \forall j \in \mathbb{S}$$
for every finite subset $\mathbb{S}\subseteq\mathbb{J}$, then $\{y_j\}_{j\in \mathbb{J}} $ is a  Riesz p-basis  for  $ \mathcal{X}$.
\end{theorem}
\begin{proof}
Given inequality shows that the map $ T:  \mathcal{X} \ni  \sum_{j\in \mathbb{J}}c_jx_j \mapsto \sum_{j\in \mathbb{J}}c_j(x_j-y_j) \in \mathcal{X}$ is a bounded linear operator of norm atmost $ \lambda$. Since $\lambda<1 $, $ I_ \mathcal{X}-T$ is bounded invertible. Now  $ ( I_ \mathcal{X}-T)x_j=y_j, \forall j \in \mathbb{J}$. Hence $\{y_j\}_{j\in \mathbb{J}} $ is a  Riesz p-basis  for  $ \mathcal{X}$.
\end{proof}
\begin{remark}
From Corollary \ref{PIMPLIESORTHONORMALFOR HILBERT} we see that Theorem \ref{GENERALIZEDPALEYWEINER} gives Paley-Wiener theorem (in Hilbert space) when $p=2$.
\end{remark}
\begin{corollary}
 Let $\{x_j\}_{j\in \mathbb{J}} $  be a  p-orthonormal basis  for  $ \mathcal{X}$, $q$ be the conjugate index of $p$. If $ \{y_j\}_{j\in \mathbb{J}} $  is a sequence in $\mathcal{X}$ such that $ \sum_{j\in \mathbb{J}}\|x_j-y_j\|^q<1$, then $\{y_j\}_{j\in \mathbb{J}} $ is a  Riesz p-basis  for  $ \mathcal{X}$.
\end{corollary}
\begin{proof}
Define $ \lambda \coloneqq \sum_{j\in \mathbb{J}}\|x_j-y_j\|^q$, and let   $\mathbb{S}$ be a finite subset of $\mathbb{J}$. By using Holder's inequality, we get 
$$\left\|\sum_{j\in \mathbb{S}}c_j(x_j-y_j)\right\|\leq \left(\sum_{j\in \mathbb{S}}|c_j|^p\right)^\frac{1}{p}\left(\sum_{j\in \mathbb{S}}\|x_j-y_j\|^q\right)^\frac{1}{q}\leq \lambda\left(\sum_{j\in \mathbb{S}}|c_j|^p\right)^\frac{1}{p} .$$
 Now apply Theorem \ref{GENERALIZEDPALEYWEINER}.
\end{proof}
\begin{theorem}
If  $\{x_j\}_{j\in \mathbb{J}}$ is a Riesz p-basis for $ \mathcal{X}$, then there exists a unique sequence $\{g_j\}_{j\in \mathbb{J}} $ in $\mathcal{X}^*$ such that 
$$ x=\sum_{j\in\mathbb{J}}g_j(x)x_j, \forall x \in \mathcal{X}, \text{and}~ \{g_j(x)\}_{j\in \mathbb{J}}  \in \ell^p(\mathbb{J}),~\forall x \in \mathcal{X}.$$
\end{theorem}
\begin{proof}
There exist a bounded invertible operator $T:\mathcal{X}\rightarrow \mathcal{X}  $ and a p-orthonormal basis $\{\tau_j\}_{j \in \mathbb{J}} $ for $\mathcal{X} $ such that $x_j=T\tau_j , \forall j \in \mathbb{J}$.   Let $ f_j : \mathcal{X} \ni\sum_{k\in\mathbb{J}}a_k\tau_k  \mapsto a_j \in \mathbb{K}, \forall j \in \mathbb{J}$. Define $ g_j\coloneqq f_jT^{-1},\forall j \in \mathbb{J}$. Then   $ x=TT^{-1}x=T(\sum_{j\in\mathbb{J}}f_j(T^{-1}x)\tau_j)=\sum_{j\in\mathbb{J}}f_j(T^{-1}x)T\tau_j=\sum_{j\in\mathbb{J}}g_j(x)x_j, \forall x \in \mathcal{X} $. Let $\{h_j\}_{j\in \mathbb{J}} $ in $\mathcal{X}^*$ be such that $ x=\sum_{j\in\mathbb{J}}h_j(x)x_j=\sum_{j\in\mathbb{J}}h_j(x)T\tau_j, \forall x \in \mathcal{X} \Rightarrow T^{-1}x=\sum_{j\in\mathbb{J}}h_j(x)\tau_j,\forall x \in \mathcal{X}.$ An action of $f_k$  on this sum gives $ g_k(x)=f_k(T^{-1}x)=h_k(x), \forall x \in \mathcal{X}, \forall k \in \mathbb{J}.$ Moreover, $ \sum_{j\in\mathbb{J}}|g_j(x)|^p=\sum_{j\in\mathbb{J}}|f_j(T^{-1}x)|^p=\|T^{-1}x\|^p\leq \|T^{-1}\|^p\|x\|^p <\infty,\forall x \in \mathcal{X} .$
\end{proof}
\begin{lemma}\label{RIESZCHARACTERIZATIONBANACH}
Let $ \{x_j\}_{j \in \mathbb{J}}$ be complete in $\mathcal{X}$, $ \{y_j\}_{j \in \mathbb{J}}$ be  in $\mathcal{X}_0$. Suppose $a, b>0$ are such that for all finite subsets $ \mathbb{S}_\mathcal{X},\mathbb{S}_{\mathcal{X}_0}\subseteq\mathbb{J}$, 
\begin{align*}
 a\sum_{j\in \mathbb{S}_\mathcal{X}}|c_j|^p&\leq \left\|\sum_{j\in \mathbb{S}_\mathcal{X}}c_jx_j \right\|^p, ~\forall c_j \in \mathbb{K}, \forall j \in \mathbb{S}_\mathcal{X} ,\\
 \left\|\sum_{j\in \mathbb{S}_{\mathcal{X}_0}}d_jy_j \right\|^p&\leq b\sum_{j\in \mathbb{S}_{\mathcal{X}_0}}|d_j|^p, ~\forall d_j \in \mathbb{K}, \forall j \in \mathbb{S}_{\mathcal{X}_0} .
 \end{align*}
Then $ T : \operatorname{span}\{x_j\}_{ j \in \mathbb{J}}\ni\sum_{\operatorname{finite}}a_jx_j \mapsto\sum_{\operatorname{finite}}a_jy_j  \in \operatorname{span}\{y_j\}_{j \in \mathbb{J}}$ defines a bounded linear operator  and extends uniquely as a bounded linear  operator from  $\mathcal{X}$ into $\mathcal{X}_0$ and the norm of the extended operator is less than or equal to $ (\frac{b}{a})^{1/p}$.
\end{lemma}
\begin{proof}
First inequality in the statement says that $ \{x_j\}_{j \in \mathbb{J}}$ is linearly independent in $\mathcal{X}$. Clearly $T$ is linear. Now 
\begin{align*}
\left\|T\left(\sum_{j\in \mathbb{S}_{\mathcal{X}}}c_jx_j \right)\right\|^p=\left\|\sum_{j\in \mathbb{S}_{\mathcal{X}}}c_jy_j \right\|^p\leq b\sum_{j\in \mathbb{S}_{\mathcal{X}}}|c_j|^p\leq \frac{b}{a}\left\|\sum_{j\in \mathbb{S}_\mathcal{X}}c_jx_j \right\|^p.
\end{align*}
Therefore $\|T\|\leq (\frac{b}{a})^{1/p}$. Since $ \overline{\operatorname{span}}\{x_j\}_{j \in \mathbb{J}}=\mathcal{X}$, lemma follows.	
\end{proof}
\begin{theorem}\label{CHARACTERIZATIONOFP-RIESZBASIS}
Let $ \{x_j\}_{j \in \mathbb{J}}$ be in $\mathcal{X}$. Then the following are equivalent.
\begin{enumerate}[\upshape(i)]
\item $ \{x_j\}_{j \in \mathbb{J}}$ is a Riesz p-basis for $\mathcal{X}$.
\item $\overline{\operatorname{span}}\{x_j\}_{j \in \mathbb{J}}=\mathcal{X}$, and there exist $a, b>0$ such that for every finite subset $ \mathbb{S}\subseteq\mathbb{J}$, 
$$ a\sum_{j\in \mathbb{S}}|c_j|^p\leq \left\|\sum_{j\in \mathbb{S}}c_jx_j \right\|^p\leq b\sum_{j\in \mathbb{S}}|c_j|^p, ~ \forall c_j \in \mathbb{K}, \forall j \in \mathbb{S} .$$
\end{enumerate}
\end{theorem}
\begin{proof}
(i) $\Rightarrow$	(ii) Let   $ \{\tau_j\}_{j\in \mathbb{J}}$ be  a p-orthonormal basis for $ \mathcal{X}$ and  $ T :  \mathcal{X}  \rightarrow \mathcal{X}$ be  bounded invertible such that $ x_j=T\tau_j, \forall j \in \mathbb{J}$. Since $\overline{\operatorname{span}}\{\tau_j\}_{j \in \mathbb{J}}=\mathcal{X}$ and $T$ is invertible, we clearly have $\overline{\operatorname{span}}\{x_j\}_{j \in \mathbb{J}}=\mathcal{X}$. If $ \mathbb{S} \subset \mathbb{J}$ is finite, then 
\begin{align*}
\frac{1}{\|T^{-1}\|^p}\sum_{j\in \mathbb{S}}|c_j|^p&=\frac{1}{\|T^{-1}\|^p}\left\| \sum_{j\in \mathbb{S}}c_j\tau_j\right\|^p\leq\left\| \sum_{j\in \mathbb{S}}c_jT\tau_j\right\|^p=\left\| \sum_{j\in \mathbb{S}}c_jx_j\right\|^p\\
&\leq \|T\|^p\left\| \sum_{j\in \mathbb{S}}c_j\tau_j\right\|^p= \|T\|^p\sum_{j\in \mathbb{S}}|c_j|^p, ~\forall c_j \in \mathbb{K}, \forall j \in \mathbb{S}.
\end{align*}
(ii) $\Rightarrow$ (i) 
 Let   $ \{e_j\}_{j\in \mathbb{J}}$ be  a p-orthonormal basis for $ \mathcal{X}$.	From Lemma \ref{RIESZCHARACTERIZATIONBANACH}, the operators $ U : \operatorname{span}\{e_j\}_{ j \in \mathbb{J}}\ni\sum_{\operatorname{finite}}a_je_j \mapsto\sum_{\operatorname{finite}}a_jx_j  \in \operatorname{span}\{x_j\}_{j \in \mathbb{J}}$, $V : \operatorname{span}\{x_j\}_{ j \in \mathbb{J}}\ni\sum_{\operatorname{finite}}a_jx_j \mapsto\sum_{\operatorname{finite}}a_je_j  \in \operatorname{span}\{e_j\}_{j \in \mathbb{J}}$ extend uniquely as bounded operators on  $ \mathcal{X}$, which we again denote by $U,V$, respectively. Then $UV=VU=I_\mathcal{X} $ and hence $ U$ is invertible. Therefore $ \{x_j=Ue_j\}_{j\in \mathbb{J}}$ is a Riesz p-basis for  $ \mathcal{X}$.
\end{proof}
\begin{corollary}
Let $ \overline{\operatorname{span}}\{x_j\}_{j \in \mathbb{J}}=\mathcal{X}$, and for each finite subset $ \mathbb{S} \subseteq \mathbb{J}$,
\begin{align}\label{FINITESETP-EQUATION}
\left\| \sum_{j\in \mathbb{S}}c_jx_j\right\|^p= \sum_{j\in \mathbb{S}}|c_j|^p , ~\forall c_j \in \mathbb{K}, \forall j \in \mathbb{S}.
\end{align}
Then $\{x_j\}_{j \in \mathbb{J}}$ is a p-orthonormal basis for $\mathcal{X}.$
\end{corollary}
\begin{proof}
From Theorem \ref{CHARACTERIZATIONOFP-RIESZBASIS} we see that $\{x_j\}_{ j \in \mathbb{J}}$ is a Riesz p-basis for $\mathcal{X}$. So, there exist a p-orthonormal basis $\{\tau_j\}_{j \in \mathbb{J}}$  for $\mathcal{X}$ and an invertible operator $ T:\mathcal{X} \rightarrow \mathcal{X} $ such that $ x_j=T\tau_j , \forall j \in \mathbb{J}$. Then for each $ x=\sum_{j\in \mathbb{J}}c_j\tau_j\in \mathcal{X}$ we see that $\{c_j\}_{j \in \mathbb{J}}\in \ell^p(\mathbb{J}) $ (from Proposition \ref{ellp}) and hence (from   Equation (\ref{FINITESETP-EQUATION})) $\sum_{j\in \mathbb{J}}c_jx_j $ exists in $ \mathcal{X}$. Then
\begin{align*}
\|Tx\|^p=\left\|\sum_{j\in \mathbb{J}}c_jx_j \right\|^p=\sum_{j\in \mathbb{J}}|c_j|^p =\left\|\sum_{j\in \mathbb{J}}c_j\tau_j \right\|^p=\|x\|^p.
\end{align*}
 Hence $ T$ is an isometry. Theorem \ref{P-ORTHONORMALBASISCHARACTERIZATION} now says that  $\{x_j\}_{j \in \mathbb{J}}$ is a p-orthonormal basis for $\mathcal{X}$.
\end{proof}
\begin{definition}
A collection of vectors $ \{y_j\}_{j \in \mathbb{J}}$ in $\mathcal{X}$ is said to be a Riesz  p-basis  w.r.t. $\{\omega_j\}_{j \in \mathbb{J}} $ in $\mathcal{X}$ if there  are  bounded invertible operators $ U,V:\mathcal{X}\rightarrow \mathcal{X} $ and a p-orthonormal basis $\{\tau_j\}_{j\in \mathbb{J}} $ for $\mathcal{X}$ such that $x_j=U\tau_j, \omega_j=V\tau_j, \forall j \in \mathbb{J}$ and the resolvent of $ VU^{-1}$ contains $(-\infty,0]$. We write $ (\{y_j\}_{j\in \mathbb{J}}, \{\omega_j\}_{j\in \mathbb{J}})$ is Riesz  p-basis.
\end{definition}
\begin{theorem}\label{RIESZPBASISSEQUENTIALCHARACTERIZATIONAPPENDIX}
Let  $ \{\tau_j\}_{j \in \mathbb{J}}$ be an arbitrary p-orthonormal basis for $\mathcal{X}$. Then the Riesz p-bases  $ (\{y_j\}_{j \in \mathbb{J}},\{\omega_j\}_{j \in \mathbb{J}})$ for  $\mathcal{X}$  are precisely $( \{U\tau_j\}_{j \in \mathbb{J}},\{V\tau_j\}_{j \in \mathbb{J}}) $, where $ U,V :  \mathcal{X} \rightarrow   \mathcal{X} $ are bounded invertible  with the resolvent of $ VU^{-1}$ contains $(-\infty,0]$.
\end{theorem}
\begin{proof}
 We have to justify only the direct part.  Let $ \{\rho_j\}_{j \in \mathbb{J}}$ be a p-orthonormal basis for $  \mathcal{X}$ and $ R,S: \mathcal{X}\rightarrow \mathcal{X} $  be bounded invertible such that $y_j=R\rho_j, \omega_j=S\rho_j, \forall j \in \mathbb{J}$ and the resolvent of $ SR^{-1}$ contains $(-\infty,0]$.  Define  $T:\mathcal{X}\ni \sum_{j\in \mathbb{J}}a_j\tau_j\mapsto \sum_{j\in \mathbb{J}}a_j\rho_j \in\mathcal{X} $. Clearly $T$ is a bounded invertible operator (also, $ T$ is an isometry).  Define $ U\coloneqq RT, V\coloneqq ST.$ Then $ U,V$ are invertible, $ U\tau_j=RT\tau_j=R\rho_j=y_j, V\tau_j=ST\tau_j=S\rho_j=\omega_j, \forall j \in \mathbb{J}$ and the resolvent of $ VU^{-1}=STT^{-1}R^{-1}=SR^{-1}$ contains $(-\infty,0]$.
\end{proof}
\begin{corollary}
If $ (\{x_j\}_{j \in \mathbb{J}},\{\tau_j\}_{j \in \mathbb{J}})$ is a Riesz p-basis  for  $\mathcal{X}$, then 
$\|U^{-1}\|^{-1} \leq \|y_j\|\leq \|U\|, \forall j \in \mathbb{J},  \|V^{-1}\|^{-1} \leq \|\omega_j\|\leq \|V\| ,\forall j \in \mathbb{J}.$
\end{corollary}
\begin{proposition}
If $ (\{y_j=U\tau_j\}_{j\in \mathbb{J}}, \{\omega_j=V\tau_j\}_{j\in \mathbb{J}})$ is a Riesz  p-basis for $\mathcal{X}$, then $ (\{g_j\coloneqq f_jU^{-1}\}_{j\in \mathbb{J}}, \{\omega_j\}_{j\in \mathbb{J}})$ is a p-frame for $\mathcal{X}$, where $ f_j: \mathcal{X} \ni \sum_{k\in \mathbb{J}}a_k\tau_k\mapsto a_j\in  \mathbb{K}, \forall j \in \mathbb{J}$.
\end{proposition}
\begin{proof}
$\widehat{S}_{g,\omega}x=\sum_{j\in \mathbb{J}}g_j(x)V\omega_j =\sum_{j\in \mathbb{J}}f_j(U^{-1}x)V\tau_j=V(\sum_{j\in \mathbb{J}}f_j(U^{-1}x)\tau_j)=VU^{-1}x, \forall x \in \mathcal{X}$. Therefore $\widehat{S}_{g,\omega}=VU^{-1} $  whose resolvent contains $(-\infty,0]$.
\end{proof}

\section{Conjectures}\label{CONJECTURE}
We end this paper, the Part I, with the following conjectures.
\begin{statement}
Let $\pi$  be a unitary representation of a discrete group $G$
 on $ \mathcal{H}$, and let $\Psi \in  \mathcal{B}(\mathcal{H},\mathcal{H}_0)$. Suppose 
\begin{align*}
\emptyset \neq \mathscr{F}_{G,\Psi}\coloneqq\{A \in \mathcal{B}(\mathcal{H},\mathcal{H}_0): (\{A\pi_{g^{-1}}\}_{g \in G},\{\Psi\pi_{g^{-1}}\}_{g \in G}) ~\text{is an  operator-valued  frame  in}~ \mathcal{B}(\mathcal{H},\mathcal{H}_0)\}.
\end{align*}
\begin{enumerate}[\upshape{CONJECTURE} (i)]
\item  If $ \dim \mathcal{H}_0<\infty$, then $ \mathscr{F}_{G,\Psi}$  need not be path-connected  in the operator-norm topology on $\mathcal{B}(\mathcal{H},\mathcal{H}_0)$.
\item If $ \dim \mathcal{H}_0=\infty$, then the following statement is not true:

$ \mathscr{F}_{G,\Psi}$ is path-connected  in the operator-norm topology on $\mathcal{B}(\mathcal{H},\mathcal{H}_0)$ if and only if the von Neumann algebra $\mathscr{R}(G)$ generated by the right regular representations of $G$ is diffuse (i.e., $\mathscr{R}(G)$ has no nonzero minimal projections).
\end{enumerate}
\end{statement}
\begin{statement}
Let $\pi$  be a unitary representation of a group-like unitary  system  $\mathcal{U}$  on $ \mathcal{H}$, and let $\Psi \in  \mathcal{B}(\mathcal{H},\mathcal{H}_0)$. Suppose
\begin{align*}
\emptyset \neq \mathscr{F}_{\mathcal{U},\Psi}\coloneqq\{A \in \mathcal{B}(\mathcal{H},\mathcal{H}_0): (\{A\pi(U)^{-1}\}_{U \in \mathcal{U}},\{\Psi\pi(U)^{-1}\}_{U \in \mathcal{U}}) ~\text{is an operator-valued  frame  in}~ \mathcal{B}(\mathcal{H},\mathcal{H}_0)\}.
\end{align*}
\begin{enumerate}[\upshape{CONJECTURE} (i)]
\item  If $ \dim \mathcal{H}_0<\infty$, then $ \mathscr{F}_{\mathcal{U},\Psi}$  need not be path-connected  in the operator-norm topology on $\mathcal{B}(\mathcal{H},\mathcal{H}_0)$.
\item If $ \dim \mathcal{H}_0=\infty$, then the following statement is not true:
		
$ \mathscr{F}_{\mathcal{U},\Psi}$ is path-connected  in the operator-norm topology on $\mathcal{B}(\mathcal{H},\mathcal{H}_0)$ if and only if the von Neumann algebra $\mathscr{R}(\mathcal{U})$ generated by the right regular representations of $\mathcal{U}$ is diffuse. 
\end{enumerate}
\end{statement}
\begin{statement}
Let $\pi$  be a unitary representation of a discrete group $G$
on $ \mathcal{H}$, and let $\tau \in  \mathcal{H}$. Suppose 
\begin{align*}
\emptyset \neq \mathscr{F}_{G,\tau}\coloneqq\{x \in \mathcal{H}: (\{\pi_{g}x\}_{g \in G},\{\pi_{g}\tau\}_{g \in G}) ~\text{is   a frame  for}~ \mathcal{H}\}.
\end{align*}
CONJECTURE:  $ \mathscr{F}_{G,\tau}$  need not be path-connected  in the norm topology on $\mathcal{H}$.
\end{statement}
\begin{statement}
Let $\pi$  be a unitary representation of a group-like unitary  system  $\mathcal{U}$  on $ \mathcal{H}$, and let $\tau \in  \mathcal{H}$. Suppose
\begin{align*}
\emptyset \neq \mathscr{F}_{\mathcal{U},\tau}\coloneqq\{x \in \mathcal{H}: (\{\pi(U)x\}_{U \in \mathcal{U}},\{\pi(U)\tau\}_{U \in \mathcal{U}}) ~\text{is a  frame  for}~ \mathcal{H}\}.
\end{align*}
CONJECTURE:  $ \mathscr{F}_{\mathcal{U},\tau}$  need not be path-connected  in  the norm topology on $\mathcal{H}$.
\end{statement}

\section{Acknowledgments}
We thank  Prof. Victor Kaftal, University of Cincinnati, Ohio for   giving reasons of some of the arguments in the  paper ``Operator-valued frames"  \cite{KAFTALLARSONZHANG1} coauthored by him. The first author thanks the National Institute of Technology Karnataka (NITK), Surathkal for giving financial support and the present work of the second author was partially supported by Science and Engineering Research Council (SERC), DST, Government of India, through the Fast Track  Scheme for Young Scientists (D.O. No. SR/FTP/MS-050/2011).

 \bibliographystyle{plain}
 \bibliography{reference.bib}
 
 \end{document}